\documentclass[a4paper, twoside, 10pt]{scrbook}
\usepackage[headheight=15pt, height=23cm, foot=1.5cm, left=2.75cm,right=2.75cm]{geometry}
\usepackage{amsmath,amssymb,amsfonts,mathrsfs}
\usepackage {amsthm, latexsym}
\usepackage {graphicx}
\usepackage {stmaryrd}   
\usepackage[latin1]{inputenc}
\usepackage [mathscr] {eucal}
\usepackage [all] {xy}
\usepackage{ae,enumerate}
\usepackage{setspace}\spacing{1.1} 
\usepackage {pstricks}
\usepackage {pst-node, pst-plot}
\usepackage {pst-bezier}
\usepackage {pstricks-add}
\usepackage{stackengine,  scalerel} 

\def\letterskip{\hskip\interletterskip}

\def\finishgesperrt{\afterassignment\testletter\let\next= }
\def\testletter{\expandafter\ifx\space\next\relax\else\nobreak\fi
   \letterskip\next}
\newif\iffirstword
\newskip\interletterskip\interletterskip=.15em plus.044em minus.044em
\def\gesperrtend{End}

\def\gesperrt{\letterskip\igesperrt}
\def\igesperrt#1{\firstwordtrue
   \dogesperrt#1\gesperrtend\finishgesperrt}
\def\dogesperrt{\afterassignment\dodogesperrt\let\next= }
\def\dodogesperrt{\ifx\next\gesperrtend\let\next\relax\else
  \iffirstword\firstwordfalse\else
    \if\space\next\else\nobreak\fi\letterskip\fi
  \ifx\-\next\firstwordtrue\hbox{}\fi
  \ifx\bgroup\next\def\next{\nobreak
         \bgroup\aftergroup\dogesperrt}\else
    \next\let\next=\dogesperrt\fi\fi\next}

\usepackage{makeidx}
\usepackage[refpage, intoc]{nomencl}
\makenomenclature
\makeindex

\usepackage{scrpage2}
\clearscrheadfoot
\rehead{\headmark}
\lehead{\pagemark}
\lohead{\headmark}
\rohead{\pagemark}

\newtheorem{Theorem}{Theorem}[section]
\newtheorem{Lemma}[Theorem]{Lemma}
\newtheorem{Corollary}[Theorem]{Corollary}
\newtheorem{Proposition}[Theorem]{Proposition}
\theoremstyle{definition}
\newtheorem{Remark}[Theorem]{Remark}
\newtheorem{Definition}[Theorem]{Definition}
\newtheorem{Example}[Theorem]{Example}
\newtheorem*{Convention}{Convention}
\newtheorem*{DefinitionOhne}{Definition} 

\newcounter{Klammerzahl}

\newenvironment {ListeTheorem}{\begin{list}{(\arabic{Klammerzahl})}{\usecounter{Klammerzahl} \topsep1,5mm \parsep0,5mm \itemsep0,5mm \labelwidth6mm \leftmargin8,5mm \labelsep2,5mm }}{\end{list}}

\newcommand {\A} {\mathbb A}
\newcommand {\C} {\mathbb C}
\newcommand {\F} {\mathbb F}
\newcommand {\N} {\mathbb N}
\newcommand {\Proj} {\mathbb P}
\newcommand {\Q} {\mathbb Q}
\newcommand {\R} {\mathbb R}
\newcommand {\Z} {\mathbb Z}

\newcommand {\Bsf} {\mathsf{B}}
\newcommand {\bBsf} {\mathsf{boolB}}
\newcommand {\cBsf} {\mathsf{comB}}
\newcommand {\sBsf} {\mathsf{simplB}}
\newcommand {\Top} {\mathsf{Top}}
\newcommand {\SCsf} {\mathsf{SC}}

\newcommand {\Acal} {\mathcal A}
\newcommand {\Bcal} {\mathcal B}

\newcommand {\Fcal} {\mathcal F} 
\newcommand {\Ical} {\mathcal I} 
\newcommand {\Jcal} {\mathcal J} 
\newcommand {\Pcal} {\mathcal P}
\newcommand {\Qcal} {\mathcal Q}
\newcommand {\Rcal} {\mathcal R}
\newcommand {\Scal} {\mathcal S}
\newcommand {\Tcal} {\mathcal T}
\newcommand {\Vcal} {\mathcal V}

\newcommand{\Bscr}{\mathscr B}

\newcommand{\Jscr}{\mathscr J}
\newcommand{\Qscr}{\mathscr Q}

\newcommand {\aideal} {\mathfrak a}
\newcommand {\bideal} {\mathfrak b}
\newcommand {\Iideal} {\mathfrak I}
\newcommand {\mideal} {\mathfrak m}
\newcommand {\pideal} {\mathfrak p}

\newcommand {\opD} {\operatorname{D}}
\newcommand {\opDE} {\opD^{\emptyset}}
\newcommand {\opFilt} {\operatorname{Filt}}
\newcommand{\Pset}{\operatorname P}
\newcommand {\opH} {\operatorname{H}}
\newcommand {\opT} {\operatorname{T}}
\newcommand {\opV} {\operatorname{V}}

\newcommand {\opalg} {\operatorname{alg}}
\newcommand {\opcan} {\operatorname{can}}
\newcommand {\End} {\operatorname{End}}
\newcommand {\opfe} {\operatorname{fe}}
\newcommand {\opint} {\operatorname{int}}
\newcommand {\oppos} {\operatorname{pos}}
\newcommand {\opred} {\operatorname{red}}
\newcommand {\optf} {\operatorname{tf}}
\newcommand {\opae} {\operatorname{ae}}
\newcommand {\opu} {\operatorname{u}}
\newcommand {\opc} {\operatorname{c}}

\newcommand {\sepdim} {\operatorname{\dim_{\operatorname{sep}}}}
\newcommand {\bidual} {^{^{_{\operatorname{vv}}}}}
\newcommand {\bool} {\operatorname{bool}}
\newcommand {\opchar} {\operatorname{char}}
\newcommand {\const} {\operatorname{const}}
\newcommand {\diff} {\operatorname{diff}}
\newcommand {\Dim} {\operatorname{Dim}}
\newcommand {\dual} {^{^{_{\operatorname{v}}}}}
\newcommand {\free} {\mathsf{FC}}
\newcommand {\freenc} {\mathsf{F}}
\newcommand {\opgr} {\operatorname{gr}}
\newcommand {\Hom} {\operatorname{Hom}}
\newcommand {\height} {\operatorname{ht}}
\newcommand {\id} {\operatorname{id}}
\newcommand {\im} {\operatorname{im}}

\newcommand {\lcm} {\operatorname{lcm}}
\newcommand {\nonint} {\opint^{\opc}}
\newcommand {\map} {\operatorname{map}}
\newcommand {\mult} {\operatorname{mult}}
\newcommand {\nil} {\operatorname{nil}}
\newcommand {\ord} {\operatorname{ord}}
\newcommand {\slim} {\operatorname{s}\mina\varprojlim}
\newcommand {\spec} {\operatorname{spec}}
\newcommand {\Spec} {\operatorname{Spec}}
\newcommand {\Spm} {\operatorname{Spm}}
\newcommand {\specE} {\operatorname{spec^{\emptyset}}}
\newcommand {\minspec} {\text{-}\spec}
\newcommand {\minSpec} {\text{-}\Spec}
\newcommand {\supp} {\operatorname{supp}}

\newcommand {\trivial} {\{0,\infty\}}
\newcommand {\zero} {\{\infty\}}

\newcommand {\supsetu} {\mathrel{\reflectbox{\rotatebox[origin=c]{90}{$\subset$}}}}
\newcommand {\subsetu} {\mathrel{\reflectbox{\rotatebox[origin=c]{270}{$\supset$}}}}
\newcommand {\supsetul} {\mathrel{\reflectbox{\rotatebox[origin=c]{45}{$\subset$}}}}
\newcommand {\subsetur} {\mathrel{\reflectbox{\rotatebox[origin=c]{315}{$\supset$}}}}

\newcommand {\lrto} {\leftrightarrow}
\newcommand {\rto} {\rightarrow}
\newcommand {\Rto} {\longrightarrow}
\newcommand {\lto} {\longmapsto}
\newcommand {\embto} {\hookrightarrow}
\newcommand {\mto} {\mapsto}
\newcommand {\Rarrow}{\Rightarrow}

\newcommand {\eq}{\Leftrightarrow}

\newcommand {\pkt} {\hspace{0,05cm}.}
\newcommand {\komma} {\hspace{0,05cm},}

\newcommand {\pluspkt} {+\cdots+}
\newcommand {\wedgepkt} {\wedge\cdots\wedge}

\newcommand {\timespkt} {\times\cdots\times}

\newcommand {\kpkt} {, \ldots, }

\newcommand {\onull} {^{\mathbf{0}}}
\newcommand {\opkt} {^{^{_{\bullet}}}}
\newcommand {\okreuz} {^{\times}}
\newcommand {\Uplus} {_{+}}
\newcommand {\UDelta} {_{^{_{\Delta}}}}

\newcommand {\minus} {\text{-}}
\newcommand {\mina} {\,\text{-}\,}

\newcommand {\azl}{\textquoteleft}
\newcommand {\azr}{\textquoteright\,}

\newcommand{\bigcupdot}{\mathop{\ThisStyle{\ensurestackMath{\stackinset{c}{}{c}{+.25ex}{\cdot}{\SavedStyle\bigcup}}}}}

\newcommand{\bigcupbidot}{\mathop{\ThisStyle{\ensurestackMath{\stackinset{c}{-1.4\LMpt}{c}{+.25ex}{:}{\SavedStyle\bigcup}}}}}
\newcommand{\cupbidot}{\mathop{\ensurestackMath{\stackinset{l}{-.19ex}{c}{+0.1ex}{:}{\cup}}}}

\makeatother

\begin{document}
\pagenumbering{gobble}
\titlehead{\begin {center}
Dissertation\\
zur Erlangung des Doktorgrades (Dr.\ rer.\ nat.)\\
des Fachbereichs Mathematik/Informatik\\
der Universit\"at Osnabr\"uck           
\end {center}}
 
\title{\begin {center}
Monoids with absorbing elements and their associated algebras
\end {center}}

\subtitle{
vorgelegt\\
von}

\author{Simone B\"ottger}
 
\date{Osnabr\"uck, 2015}
 
\publishers{\begin {center}
Betreuer: Prof.\ Dr.\ Holger Brenner
\end {center}}

\maketitle



\newpage

\cleardoublepage

\pagenumbering {arabic}


\markright{Introduction}
\chapter*{Introduction}
\addcontentsline{toc}{chapter}{Introduction}

In what follows, $K$ denotes a field, $P$ the polynomial algebra $K[X_{v}\mid v\in V]$, $V:=\{1\kpkt n\}$, in $n\ge 1$ commuting variables $X_{1}\kpkt X_{n}$, and $X^{\nu}$ will be shorthand for the monomial $X_{1}^{\nu_{1}}\cdots X_{n}^{\nu_{n}}$.

\bigskip

One of the main objects of concern in combinatorics, commutative algebra, and algebraic geometry are algebras $P/\Iideal$ defined by an ideal $\Iideal\subseteq P$, or more generally, varieties and schemes associated to $\Iideal$. Among the large amount of examples that appear in these areas some of the most important are the following.

\begin{ListeTheorem}
\item [(1)] The coordinate ring $K[\Vcal]=P/\Iideal_{\Vcal}$ of an affine (or projective) variety $\Vcal\subseteq \A^{n}(K)$ ($\subseteq\Proj^{n-1}(K)$), where $\Iideal_{\Vcal}=\{F\in P$ (homogeneous)$\mid F(v)=0$ for all $v\in \Vcal\}$.
\item [(2)] The Stanley-Reisner algebra $K[\Delta]:=P/\Iideal_{\Delta}$ associated to a simplicial complex $\Delta$ on $V$, where $\Iideal_{\Delta}=(\prod_{v\in F}X_{v}\mid F\not\in\Delta)$ is the so-called Stanley-Reisner ideal of $\Delta$.
\item [(3)] The toric ring $K[C]\subseteq K[Y_{1}\kpkt Y_{r}]$ of a finite set of monomials $C=\{Y^{\nu_{1}}\kpkt Y^{\nu_{n}}\}\subseteq K[Y_{1}\kpkt Y_{r}]$, which is given by  $P/\Iideal_{C}$ via $\pi:P\rto K[C]$, $X_{i}\mto Y^{\nu_{i}}$, where $\Iideal_{C}=\ker\pi$ is the so-called toric ideal associated to $C$.
\end{ListeTheorem}

As a non-commutative example, we give the following from representation theory.

\begin{ListeTheorem}
\item [(4)] The path algebra $K[Q]=P/\Iideal_{D}$ associated to a quiver $Q:=(V,A)$, where $D$ is the set of paths in $Q$ and $\Iideal_{D}=(X^{p}\mid p\not\in D)$ is the so-called path ideal of $Q$, where $P$ denotes in this case the polynomial algebra with non-commuting variables $X_{v}$, $v\in V$.
\end{ListeTheorem}

The algebras $K[\Delta]$ and $K[Q]$ are related to combinatorial objects, namely $\Delta$ and $Q$, and so is $K[C]$ since it is naturally isomorphic to the monoid algebra $KM$ of the affine monoid $M=\N\nu_{1}\pluspkt\N\nu_{n}$ generated by $\nu_{1}\kpkt\nu_{n}\in\N^{r}$. While $K[\Delta]$ and $K[Q]$ are monomial algebras (i.e.\ defined by an ideal generated by monomials), the defining ideal $\Iideal_{C}$ of $K[C]$ is a graded prime ideal generated by finitely many pure difference binomials (\cite[Proposition 7.1.2]{Villarreal}), which are polynomials of the form $X^{\nu}-X^{\mu}$. In fact, every monoid algebra is the homomorphic image of a polynomial algebra $K[X_{v}\mid v\in W]$ for a not necessarily finite set $W$ by an ideal $\Iideal$ generated by pure difference binomials (\cite[Theorem 7.11]{Gilmer}). This also shows that certain varieties are related to combinatorial objects, for instance, toric varieties (\cite{CoxLittleSchenck},\cite{Fulton}). 

\bigskip

From a conceptual point of view it seems desirable to have a theory that treats those algebras (or varieties or schemes) related to combinatorial objects within the same framework. Toric face rings are one attempt. Introduced by Stanley (\cite{StanleyToricFR}), they unify Stanley-Reisner algebras and affine monoid algebras by mimicking the connection of a monoid to its algebra with a modified operation described below (\ref{modified}). The construction of Stanley has been generalized to monoideal complexes (\cite{IchimRoemer}, \cite[Chapter 7.B]{BrunsGubeladze}) and has proven to be relevant to such a unification.

Another very powerful framework is the theory of binomial ideals. A binomial is a polynomial of the form $aX^{\nu}-bX^{\mu}$ with $a,b\in K$, and an ideal generated by binomials is a binomial ideal. By definition, a pure difference binomial is a binomial with $a=b=1$, and a monomial is a binomial with $a=1$ and $b=0$. In particular, every Stanley-Reisner algebra and monoid algebra can be realized as an algebra defined by a binomial ideal. Algebras, varieties, and schemes related to a binomial ideal constitute an immense quantity of classical and important examples in combinatorics, commuatitve algebra, and algebraic geometry, which justifies the strong interest as well as the large body of literature and active work on binomial ideals (\cite{BrunsGubeladze}, \cite{EisenbudSturmfels}, \cite{HerzogHibi}, \cite{MillerSturmfels}, \cite{Villarreal}, \cite{KahleMiller}). Moreover, they have notable applications beyond pure mathematics like coding theory, algebraic statistics, game theory, linear programming, and chemical kinetics (\cite{Sturmfels}, \cite{Miller2011}, and the references therein). Due to their combinatorial nature, a lot of algorithms and feasible computations are possible, which makes them even more attractive.

\bigskip

However, the somewhat smaller class of binomial ideals generated merely by pure difference binomials and/or monomials (i.e.\ by binomials with coefficients $a,b\in\{0,1\}$), which are independent of the base field (or ring), still cover a great part of the significant cases (e.g.\ toric face rings). This class outlines precisely the scope of this thesis. The main reason for restricting ourselves exactly to this class is that it is the largest class within which the powerful interplay of monoids and their algebras can be extended in a rather simple way.

\begin{DefinitionOhne}
Let $M$ be an additively written monoid. An \emph{absorbing element} $a\in M$ satisfies $a+b=a$ for all $b\in M$. Such an element is always unique and will be denoted by $\infty$ if it exists. A monoid with an absorbing element is called a \emph{binoid}. Its associated algebra, called the \emph{binoid algebra} of $M$, is defined to be the quotient algebra
$$K[M]\,:=\,KM/(X^{\infty})\pkt$$
\end{DefinitionOhne}

By definition, every binoid is a monoid. On the other hand, adjoining an absorbing element to an arbitrary monoid $M$ subject to its defining property yields a binoid $M\cup\zero=:M^{\infty}$ such that $K[M^{\infty}]=KM$. Thus, binoid algebras generalize monoid algebras. They are precisely the homomorphic images of polynomial algebras by ideals generated by pure difference binomials and/or monomials. Since adjoining an absorbing element has no impact on the structure of a monoid there is no loss of generality when turning from a monoid to a binoid.

As a matter of fact, dealing with binoids is more natural for several reasons. As in ring theory, one of the most fundamental tools in monoid theory is the Rees quotient 
$$M/\Ical\,=\,\{[a]\mid a\in M\setminus\Ical\}\cup\{[\Ical]\}\,\cong\,M\setminus\Ical\cup\{[\Ical]\}$$
of a monoid $M$ by an ideal $\Ical\subseteq M$ with addition given by 
$$[a]+[b]\,=\,\begin {cases}
[a+b]&\text{, if }a+b\not\in\Ical\komma\\
[\Ical]&\text{, otherwise.}
\end {cases}$$
On the one hand, the class $[\Ical]$ is an absorbing element in $M/\Ical$, which shows that one inevitably encounters binoids even when dealing solely with monoids that admit no absorbing element in first place. On the other hand, considering $M^{\infty}$ yields a canonical representative for the class $[\Ical]$, namely $[\Ical]=[\infty]$, because it reflects its characteristic feature of being absorbing. By proceeding in this way, one obtains an elegant description of $M^{\infty}/\Ical$, which is also appreciated in monoid theory but has to be enforced by the standard practice of defining $[\Ical]=:\infty$ (\cite[Chapter 1.3]{GrilletCS}).

\bigskip

Considering binoids accomplishes more than just focussing on a particular class of monoids that appear naturally in many situations, and yielding better notation when the absorbing element is emphasized. The main motivation for binoid theory was the observation that many algebras that arise in combinatorics can be realized as binoid algebras but not as monoid algebras, because of zero- divisors that can be encoded on the combinatorial level with binoids but not with monoids.

Let us outline how such algebras are typically constructed: given a collection $C$ of objects with a partial operation, that is, to some pairs $c_{1},c_{2}\in C$ there exists a unique object in $C$, say $c_{1}\ast c_{2}$, and whenever $c_{1}\ast c_{2}$ and $c_{2}\ast c_{3}$ are defined for $c_{1},c_{2},c_{3}\in C$, then $(c_{1}\ast c_{2})\ast c_{3}$ is defined if and only if $c_{1}\ast(c_{2}\ast c_{3})$ is as well, and in this case they coincide. The algebra associated to $C$ is now defined to be the $K\mina$module
$$K[C]\,:=\,\bigoplus_{c\in C}KX^{c}$$
with multiplication defined by
\[X^{c_{1}}\cdot X^{c_{2}}\,:=\,\begin {cases}
X^{c_{1}\ast c_{2}}&\text{, if }c_{1}\ast c_{2}\text{ is defined,}\\
0&\text{, otherwise.}
\end {cases}\tag{$\star$}\label{modified}\]
Obviously, if the partial operation is no operation, this algebra is not an integral domain and cannot be realized as a monoid algebra. However, it is straightforward to see that $C$ defines a binoid, namely $C\cup\zero=:M_{C}$ with addition given by
$$c_{1}+c_{2}\,:=\,\begin {cases}
c_{1}\ast c_{2}&\text{, if this is defined,}\\
\infty&\text{, otherwise,}
\end {cases}$$
and that $K[M_{C}]=K[C]$; that is, $K[C]$ is a binoid algebra. Hence, shifting the defining operation from the module to the combinatorial level creates a situation similar to the correspondence of monoids to their algebras. This is to say, with a theory of binoids and their algebras at hand, the strucure of $K[C]$ could be revealed by the combinatorics of $M_{C}$ and vice versa. 

The advantages of such a theory are obvious. Like toric face rings, binoid algebras unify different algebras that arise in combinatorial commutative algebra by studying them within the same framework. This is a very powerful and comprehensive framework since binoids being monoids allow us to adapt the well-established and elaborate theory of monoids and their algebras, given for instance in \cite{CliffordPreston}, \cite{Gilmer}, \cite{GrilletS}, \cite{GrilletCS}, \cite{Kobsa} or \cite{OkninskiSA} to mention just a very few (of which some also deal with binoids and their algebras like the books \cite{CliffordPreston} from the 1960s). 

Of course, considering the absorbing element causes new properties to arise, known concepts need to be modified, and pathologies appear compared to common structures in (commutative) algebra. However, for several reasons and motivations, working with binoids rather than arbitrary monoids seems preferable in certain situations, as can be verified by the frequency with which they appear in the literature and their profound applications in different branches of algebra and geometry (\cite{AndersonIT}, \cite{CliffordPreston}, \cite{GrilletCS}, \cite{Kist}, \cite{Kozhukhov}, \cite{Munn}, \cite{Novikov}, \cite{OkninskiSA}). There, they are usually written multiplicatively and called ``monoid with zero'' or ``pointed monoid'', while the binoid algebra is also known as a ``contracted'' or ``pointed monoid algebra''. For instance, spaces with a monoidal structure became increasingly important due to their appearance in $\F_{1}\mina$geometry and logarithmic algebraic geometry within the last two decades. Study in both areas depends largely on the commutative algebra of monoids and consequently enriches monoid theory by developing a geometric theory of monoids including sheaves of monoids and monoid schemes (\cite{Ogus}, \cite{DeitmarF1Schemes}, \cite{DeitmarToricVarieties}), with which we will not be concerned. Even there, some approaches favor working with binoids from scratch for a more general theory (\cite{ChuLorscheidSanthanam}, \cite{ConnesConsani1}, \cite{ConnesConsani2}, \cite{ConnesConsani3}, \cite{CortinasWeibelCDH}, \cite{FloresWeibel}).

Thus, binoids appear in many different branches of mathematics but, unfortunately, it seems that one is not always aware of what has already been done elsewhere. This is most likely due to terminology, since in older works on binoids, they are simply addressed as semigroups.

\bigskip

We also want to bring two more works on binoids to the attention of the reader. In \cite{BayarThesis}, Hilbert-Kunz theory for binoids is developed. It is shown that for certain binoids the Hilbert-Kunz multiplicity exists and is a rational number. A geometrical approach to binoids can be found in \cite{DavideThesis}, where invariants of a binoid scheme are related to those of the corresponding ring. The main focus of this work are divisor class groups and Picard groups of binoid schemes.

\bigskip

It is one aim of this thesis to provide a general theory of binoids, their algebras, and their spectra, but we will not pay as much attention to the algebras (and modules) as to the binoids and spectra. The core of this thesis is rather based on the familiar idea in commutative algebra to go one step further and have a view on algebraic geometry by means of the seminal relation between commutative algebras and geometric spaces. In Chapter \ref{ChapTopology} and \ref{ChapSepGradings}, we show that the space of $K\mina$points of a binoid algebra can be studied and fully understood from the combinatorics provided by the underlying binoid by making use of the crucial isomorphism
$$K\minspec K[M]\,\,\cong\,\, K\minspec M\komma$$
where $K\minspec M$ is the set of all $K\mina$valued morphisms of $M$. In this regard, we investigate several connectedness properties of such spaces. In the last chapter, we apply binoid theory to simplicial complexes, determine those binoids that realize Stanley-Reisner algebras, and show how a theory of these objects could start from a binoid point of view.

\bigskip

This thesis contains six chapters, which we will sketch now. The first two chapters cover binoid theory including the ideal theory of binoids to which we return later when we investigate the spectrum of a binoid from a topological point of view. Within the scope of this thesis, it was not possible to give a full account of the whole of binoid theory developed so far and elsewhere. We try to compensate this by giving references but do not claim completeness.

More precisely, Chapter 1 is a collection of basic definitions and constructions for binoids. Here, terminology and mainly additive notation are taken from monoid theory. The main definitions and properties of binoids are given before we focus on $N\mina$spectra of binoids, where $N$ is another binoid. In view of our main concern, the $N\mina$spectra are determined for every construction dicussed in this chapter. After recalling the well-known concept of congruences for a binoid, we collect specific congruences of interest for us. We prove that the product, coproduct, and limits exist in the category of (commutative) binoids, and describe these in detail. For binoids (resp.\ pointed sets), one can also define the pointed and bipointed union, where the latter applies only to binoids with a trivial unit group. The module and algebra counterparts in binoid theory are described by pointed sets and binoids over a binoid. We close this chapter with considerations on localization and integrality properties including the normalization of a binoid. Though we are mainly interested in finitely generated commutative binoids and devote one section to them, the first chapter is kept as general as possible, whereas in the remainder, we deal almost exclusively with commutative binoids.

In Chapter 2, we investigate the ideal theory of commutative binoids with focus on the spectrum of a binoid, which corresponds one-to-one to its $\trivial\mina$spectrum and to its filtrum, which is the set of all filters. The spectra of several binoid constructions are determined by taking advantage of our knowledge of their $N\mina$spectra. The ideal theory of binoids is very similar to that of rings with the exception that the union of (prime) ideals in a binoid is again a (prime) ideal. Moreover, every binoid can be considered to be local in the sense that there is a unique maximal ideal with respect to set inclusion. 

In Chapter 3, we move on to modules and algebras associated to binoids, $N\mina$sets, and $N\mina$binoids. After recalling useful results on monoid algebras, we introduce binoid algebras, state basic facts about them, and investigate the relationship between their ideals and the ideals of the underlying binoid. Modules and (binoid) algebras over binoid algebras generalize the concept of binoid algebras and will be treated at the end of this chapter.

In Chapter 4, we take a topological approach to commutative binoids via their spectra and $K\mina$spectra, where $K$ denotes a field. The spectrum of a binoid can be endowed with the Zariski topology in the same manner as that of a ring and the theories are very similar. This also applies to the dimension theory of binoids, which we treat in the third section. We introduce the booleanization of a binoid, whose spectrum is homeomorphic to that of the initial binoid, but which is a fairly easier binoid to study. The topology on the $K\mina$spectrum of a binoid comes from that on the $K\mina$spectrum of the binoid algebra $K[M]$. We prove some basic facts before we investigate connectedness properties of the $K\mina$spectrum of $K[M]$ that can be analyzed solely on the combinatorial level. Here, we focus in particular on the case of hypersurfaces.

In Chapter 5, we study a stronger connectedness property which involves the (special) $K\mina$point that is related to the unique maximal ideal in the binoid. We give necessary and sufficient conditions under which the special point is contained in every cancellative component of the $K\mina$spectrum, and under which every $K\mina$point is $\A^{n}\mina$connected to the special point. To obtain these results, we develop the notion of separated and graded binoids.

In Chapter 6, simplicial complexes and Stanley-Reisner algebras are studied from a binoid theoretic point of view. After recalling basic facts on simplicial complexes, we consider two binoids defined by a simplicial complex (with respect to $\cup$ and $\cap$). The first three sections deal with the basic properties of these binoids, their spectra, and morphisms between them. For a binoid $N$, the $N\mina$points of a simplicial complex $\Delta$ are introduced. These correspond one-to-one to the $N\mina$points of the binoid $M\UDelta$ associated to $\Delta$, which is the main concern of the last section. We prove that the binoid algebra of $M\UDelta$ is given by the Stanley-Reisner algebra and characterize those binoids that realize Stanley-Reisner algebras.

\bigskip

\bigskip

\subsubsection*{Conventions}
Throughout this thesis we make the following agreements. The sum over the empty set is $0$, and the product over the emptyset is $1$. By a ring, we always mean a commutative ring with identity and all ring homomorphisms preserve the identity.

\bigskip

\bigskip

\bigskip

\bigskip

\bigskip

\bigskip

\bigskip

\bigskip

\bigskip

\bigskip

\bigskip

\bigskip

\bigskip

\bigskip

\bigskip

\bigskip

\subsubsection*{Acknowledgments}
It is a great pleasure to thank all of those who helped me along my way. My deepest thanks go to Holger Brenner whose expertise, understanding, and enthusiasm added considerably to my postgraduate experience. I would like to thank him for all his support during the past years and for an excellent supervision of this thesis. I always enjoyed working with him and benefitted a lot from numerous interesting discussions. I take this oppportunity to express my gratitude to my past teacher Uwe Storch who introduced me to commutative algebra. He allowed me to work with him and to share some beautiful piece of mathematics. I also owe him the chance to spend six months of graduate years at the Department of Mathematics of the Indian Insitute of Science, Bangalore, India. I would also like to thank Lukas Katth\"an, Davide Alberelli, Sean Tilson, and Katja Brunkhorst for suggestions and proof-reading parts of this thesis. 

\clearpage

\tableofcontents

\clearpage


\chapter {Basic concepts of binoids} \label{Chap1Basics}
\markright {\ref{Chap1Basics} Basic concepts of binoids}

In this chapter, the basic terminology and notation concerning monoids with an absorbing element, which we call binoids, will be introduced. All definitions come along with examples which will frequently appear in the following chapters. 

In the remainder of this thesis, we will mainly focus on the class of commutative binoids, but in this chapter we try to keep it as general as possible and consider arbitrary binoids if not otherwise stated. At the beginning of those sections that deal only with commutative binoids, we will mention this restriction. From the beginning, we fix additive notation for arbitrary binoids, though the terminology and notation of classical concepts will not always follow this convention.

\medskip

On the one hand, every binoid is a monoid so that the concepts from monoid theory apply to binoids if the special element $\infty$ (and its inherent property of being a universal absorbing element) has no effect. In this case, definitions and results of semigroup and monoid theory may be transfered directly to (semi-) binoids, otherwise we will point out the necessity to adapt well-known concepts by making it compatible with $\infty$. For the sake of completeness, definitions and results are given. Standard references for semigroup and monoid theory are the books \cite{CliffordPreston} by A.H. Clifford and G.B Preston, as well as \cite{GrilletS} and \cite{GrilletCS} by P.A. Grillet, and \cite{Gilmer} by R. Gilmer, where the latter two focus on the commutative case.

On the other hand, every binoid can also be considered as a pointed set, where the distinguished point is given by the absorbing element. This concept yields very interesting and useful constructions for binoids, cf.\ Section \ref {SecSmashProduct} and Section \ref {SecOperation}.

\bigskip

\section {Basic objects}  \label{Sec1DefProp}
\markright{\ref{Sec1DefProp} Basic objects}

In this section, we give the definitions of the main objects of this thesis, namley binoids and their algebras. We say how to consider a binoid as a pointed set and thereby define the product and the direct sum of a family of binoids in the same way as for pointed sets. 

\begin {Definition}
A \gesperrt{semigroup} \index{semigroup}$(M,\ast)$ is a set $M$ with an associative operation 
$$\ast:M\times M\Rto M\komma\quad(a,b)\lto a\ast b\pkt$$
A \gesperrt{monoid}\index{monoid}$(M,\ast,e)$ is a semigroup that admits an element $e$ which satisfies $a\ast e=e\ast a=a$ for all $a\in M$. Such an element is always unique and will be called the \gesperrt{identity} \index{element!identity--}element of $M$. An identity element is always unique. A \gesperrt{submonoid} \index{submonoid}of a monoid $M$ is a subsemigroup that contains the identity element of $M$. In additive notation, the identity element will be denoted by $0$ and in multiplicative notation by $1$.

An element $a\in M$ in a semigroup is \gesperrt{absorbing} \index{element!absorbing --}\index{absorbing}if $a\ast x=x\ast a=a$ for all $x\in M$. An absorbing element is always unique. A \gesperrt{binoid} \index{binoid}$(M,\ast,e,a)$ (resp.\ \gesperrt{semibinoid} \index{semibinoid}$(M,\ast,a)$) is a monoid (resp.\ semigroup) with an absorbing element. A  \gesperrt{subbinoid} \index{subbinoid}(resp.\ \gesperrt{subsemibinoid}\index{subsemibinoid}) of $M$ is a submonoid (resp.\ subsemigroup) of $M$ that contains the absorbing element of $M$. In additive notation, the absorbing element will be denoted by $\infty$ and in multiplicative notation by $0$.
\end {Definition}

By definition, semibinoids and monoids are never empty and, in particular, so are binoids.

\begin {Convention}
In this thesis, arbitrary binoids will be written \emph{additively} if not otherwise stated (even if the binoid is not commutative).
\end {Convention}

By abuse of notation, we will not strictly use additive notation when referring to classical concepts such as localization and so forth. Moreover, unless there is confusion, we abbreviate 
$$na:=a\pluspkt a\quad\text{and}\quad nA:=\{a_{1}\pluspkt a_{n}\mid a_{i}\in A\}$$
for $n\in\N$, $a\in M$, and $A\subseteq M$, where $0a:=0$ and $0A:=\emptyset$.\footnote{\, The latter definitions are consistent with our convention that the sum over the empty set is $0$.}

\medskip

While we are studying binoids, we always have their associated algebras in mind and will refer to them. Then we tacitly assume little knowledge about monoid algebras since the \gesperrt{binoid algebra} \index{binoid!-- algebra}is given by\label{DefBinoidAlgebra}
$$R[M]:=RM/(T^{\infty})\komma$$
where $M$ is a binoid, $R$ a ring, $RM=\bigoplus_{a\in M}RT^{a}$ the monoid algebra, and $(T^{\infty})$ the ideal in $RM$ generated by $T^{\infty}$. An elaboration of the basics of monoid and binoid algebras will be given in Section \ref{SecMonoidAlgebra} and Section \ref{SecBinoidAlgebra}. 

\begin {Example}
Let $R$ be a ring.
\begin {ListeTheorem}
\item The binoid algebra of the \gesperrt{zero binoid} \index{binoid!zero --}$\zero$, i.e.\  $0=\infty$, over any ring is the zero ring.
\item The binoid algebra of the \gesperrt{trivial binoid} \index{binoid!trivial --}$\trivial$, i.e.\ with $0\not=\infty$, over $R$ is $R$.
\item Adjoining an absorbing element $\infty$ to the commutative monoid $(\N^{n},+,(0\kpkt0))$, $n\ge1$, \nomenclature[N]{$(\N^{n})^{\infty}$}{$=((\N^{n})\cup\{\infty\},+,(0\kpkt0),\infty)$}yields a binoid, denoted by $(\N^{n})^{\infty}$, by defining $k+\infty=:\infty$ (this construction is based on a general concept which will be discussed subsequently). The binoid algebra $R[(\N^{n})^{\infty}]=R(\N^{n})$ is isomorphic to the polynomial algebra $R[X_{1}\kpkt X_{n}]$ in $n$ variables.
\item Every ring yields a (commutative) binoid by forgetting the additive structure.
 \end {ListeTheorem}
\end {Example}

The binoids of the following example play an important role in this thesis and will frequently appear.

\begin {Example}\label{ExpBinoidsPowerSet}
Let $V$ be an arbitrary set. The power set $\Pset(V)$\nomenclature[Powerset]{$\Pset(V)$}{powerset of $V$} gives rise to two different commutative binoids, namely\nomenclature[PowersetBinoids]{$\Pset(V)_{\cap},\Pset(V)_{\cup}$}{binoids defined by the discrete topology on $V$}
$$\Pset(V)_{\cap}:=(\Pset(V), \cap,V,\emptyset)\quad\text{and}\quad\Pset(V)_{\cup}:=(\Pset(V),\cup,\emptyset,V)\pkt$$
In either case, $\Pset(\emptyset)$ yields the zero binoid and $\Pset(\{1\})$ the trivial binoid. If $V$ is finite, we abbreviate $\Pset(\{1,\kpkt n\})=:\Pset_{n}$ and write $\Pset_{n,\cap}$ and $\Pset_{n,\cup}$\nomenclature[PowersetBinoidsFinite]{$\Pset_{n,\cap},\Pset_{n,\cup}$}{binoids defined by the discrete topology on $\{1\kpkt n\}$} for the corresponding binoids, $n\ge 1$.

The subbinoids of $\Pset(V)_{\cap}$ and $\Pset(V)_{\cup}$ are given by the subsets $M\subseteq\Pset(V)$ that are closed with respect to the operation of $\cup$ and $\cap$, respectively, and contain $\emptyset$ and $V$. If $M$ is a subbinoid of $\Pset(V)_{\cup}$ (resp.\ $\Pset(V)_{\cap}$), then $$M^{\opc}:=\{U^{\opc}\mid U\in M\}\komma$$ 
where $U^{\opc}:=V\setminus U$ denotes the complement of $U$, is a subbinoid of $\Pset(V)_{\cap}$ (resp.\ $\Pset(V)_{\cup}$) since $U^{\opc}\cup W^{\opc}=(U\cap W)^{\opc}$ (resp.\ $U^{\opc}\cap W^{\opc}=(U\cup W)^{\opc}$) for all $U,W\subseteq\Pset(V)$. 

In particular, every topology $\Tcal=\{U\mid U\subseteq V$ open$\}$ on a nonempty set $V$ defines commutative binoids with respect to the join and meet operation, namely $(\Tcal,\cap,V,\emptyset)=\Tcal_{\cap}$ and $(\Tcal,\cup,\emptyset,V)=\Tcal_{\cup}$\nomenclature[Topology]{$\Tcal_{\cap},\Tcal_{\cup}$}{binoids defined by a topology $\Tcal$}, as well as the set of all closed sets $\Tcal^{\opc}=\{U^{\opc}\mid U\in\Tcal\}$. The binoids $\Pset(V)_{\cup}$ and $\Pset(V)_{\cap}$ come from the discrete topology on $V$, and the binoid due to the trivial topology on $V$ is nothing else than the trivial binoid $\{V,\emptyset\}$.
\end {Example}

By definition, every semibinoid is a semigroup and every binoid is a monoid. On the other hand, it is always possible to adjoin an identity or an absorbing element subject to their defining conditions even if such an element already exists. For instance, if $(M,+,0,\infty)$ is a binoid, then $M\cup\{\mathbf{0}\}$ subject to $a+\mathbf{0}=\mathbf{0}+a=a$ for all $a\in M$, in particular,
$$0+\mathbf{0}=\mathbf{0}+0=0\komma$$
is a binoid with identity element $\mathbf{0}$. Thus, every semigroup (or monoid or semibinoid) is embeddable in a binoid so that considering binoids amounts to no loss of generality.

\begin {Definition}
The semibinoid and monoid that arises from adjoining an absorbing and an identity element to a semigroup $M$ will be denoted by $M^{\infty}$\nomenclature[m]{$M^{\infty}$}{$=M\cup\zero$} and $M\onull$\nomenclature[m]{$M\onull$}{$=M\cup\{0\}$}, respectively. If $M$ possesses an absorbing element $\infty$, we write $M\opkt$ for the set $M\setminus\zero$\nomenclature[m]{$M\opkt$}{$=M\setminus\zero$}.\end {Definition}

Objects like monoids and semibinoids originate from a more general concept.

\begin{Definition}
A \gesperrt{pointed set} \index{pointed set}$(S,p)$ is a set $S$ with a distinguished element $p\in S$ and a map $(S,p)\rto(T,q)$ of pointed sets with $p\mto q$ is called a  \gesperrt{pointed map}\index{map!pointed}. The set of all pointed maps $S\rto T$ will be denoted by $\map_{p\mto q}(S,T)$\nomenclature[map1]{$\map_{p\mto q
}(S,T)$}{set of all pointed maps $(S,p)\rto(T,q)$}. In case $T=S$, we simply write $\map_{p}S$.\nomenclature[map2]{$\map_{p
}S$}{set of all pointed maps $(S,p)\rto(S,p)$}
\end{Definition}

The set $\map_{p}S$ is a binoid with respect to the composition of maps $S\rto S$. The identity element is given by $\id_{S}$ and the absorbing element by the constant map $c_{p}:s\mto p$, $s\in S$. In the next section, cf.\ Lemma \ref{LemMappingRealization}, we will show that every binoid can be realized as the subbinoid of such a mapping binoid
$$(\map_{p}S, \circ,\id_{S},c_{p})\pkt$$
Note that there are no conditions required from the distinguished point unless it comes to morphisms of pointed sets. Thus, a set $M$ is a monoid if and only if it is a pointed semigroup $(M,0)$ with the additional property of the identity element $0$. Similarly, a semibinoid $M$ is a pointed semigroup $(M,\infty)$ subject to the defining property of $\infty$. Having observed this, a binoid $M$ can be considered as a pointed set $(M,p)$ in two different ways: as a pointed set with $p=0$ or as a pointed set with $p=\infty$. It turns out that the latter approach is the one that makes more sense in our context, see for instance the definition of the product below or Section \ref{SecSmashProduct} and Section \ref{SecOperation}.

\medskip

Since binoids are monoids, the product and the direct sum of a family of binoids is well-defined. It remains the question if the outcome has a binoid structure as well. For the sake of completeness, we recall the definitions for pointed sets and state the main results about monoids (without giving proofs) before we consider binoids.

\begin {Definition} \label{DefProdPointedSets}
The \gesperrt{product} \index{product!-- of pointed sets}of an arbitrary family $(S_{i},p_{i})_{i\in I}$ of pointed sets is given by the cartesian product $\prod_{i\in I}S_{i}$\nomenclature[AProductProduct]{$\prod_{i\in I}S_{i}$}{product of a family of pointed sets}. The subset consisting of all tuples $(s_{i})_{i\in I}$ with $s_{i}=p_{i}$ for almost all $i\in I$ is called the \gesperrt {direct sum} \index{direct sum!-- of pointed sets}and will be denoted by $\bigoplus_{i\in I}S_{i}$.\nomenclature[AProductSum]{$\bigoplus_{i\in I}S_{i}$}{direct sum of a family of pointed sets}
\end {Definition}

The product and the direct sum of a family $(S_{i},p_{i})_{i\in I}$ of pointed sets are again pointed sets with distinguished element $(p_{i})_{i\in I}$ and they coincide if and only if $I$ is finite; in this case, we prefer the notation $\prod_{i\in I}S_{i}$ instead of $\bigoplus_{i\in I}S_{i}$. Considering a monoid as a pointed set with distinguished element $0$, we get the following result.

\begin {Lemma} \label{LemProduct}
Let $(M_{i},+,0_{i})_{i\in I}$ be a family of monoids. The product $\prod_{i\in I}(M_{i})_{i\in I}$ \index{product!-- of monoids}\nomenclature[AProductProduct]{$\prod_{i\in I}M_{i}$}{product of a family of monoids}is a monoid with the componentwise addition $(a_{i})_{i\in I}+(b_{i})_{i\in I}=(a_{i}+b_{i})_{i\in I}$ and identity element $(0_{i})_{i\in I}$. The direct sum $\bigoplus_{i\in I}M_{i}$, which consists of all tuples $(a_{i})_{i\in I}$ with $a_{i}=0_{i}$ for almost all $i\in I$, is a submonoid of $\prod_{i\in I}M_{i}$.\index{direct sum!-- of monoids}\nomenclature[AProductSum]{$\bigoplus_{i\in I}M_{i}$}{direct sum of a family monoids}
\end {Lemma}
\begin {proof}
This is easy to check.
\end {proof}

\begin {Proposition} \label {PropUEproduct}
The product in the category of monoids is given by the product monoid; that is, if $(M_{i})_{i\in I}$ is a family of monoids such that for every $i\in I$ there is a monoid homomorphism $q_{i}:Q\rto M_{i}$, then there exists a unique monoid homomorphism $q:Q\rto\prod_{i\in I}M_{i}$ with $\pi_{k}q=q_{k}$ for all $k\in I$, where $\pi_{k}$ denotes the canonical projection $\prod_{i\in I}M_{i}\rto M_{k}$ on the $k$th component.
\end {Proposition}
\begin {proof}
This is standard.
\end {proof}

One can also consider the subset of the product $\prod_{i\in I}M_{i}$ of a family $(M_{i})_{i\in I}$ of binoids that contains all tuples $(a_{i})_{i\in I}$ with $a_{i}=0_{i}$ for almost all $i\in I$. Clearly, this is a subsemigroup of $\prod_{i\in I}M_{i}$.

However, the next result gives reason for our convention to consider a binoid as a pointed set with distinguished element $\infty$ because, similar to rings, the direct sum of an infinite family $(M_{i},\infty_{i})_{i\in I}$ of nonzero binoids as defined in Definition \ref{DefProdPointedSets} is only a semibinoid.

\begin {Corollary} \label{CorProduct} 
Let $(M_{i})_{i\in I}$ be a family of binoids. With respect to the componentwise addition, the product $\prod_{i\in I}M_{i}$ \index{product!-- of binoids}\nomenclature[AProductProduct]{$\prod_{i\in I}M_{i}$}{product of a family of binoids}of the family $(M_{i},\infty_{i})_{i\in I}$ of pointed sets is a binoid with identity element $(0_{i})_{i\in I}=:0_{\Pi}$ and absorbing element $(\infty_{i})_{i\in I}=:\infty_{\Pi}$. The direct sum $\bigoplus_{i\in I}M_{i}$ \index{direct sum!-- of binoids}is a semibinoid that coincides with the product if and only if $M_{i}=\{\infty_{i}\}$ for almost all $i\in I$.\nomenclature[AProductSum]{$\bigoplus_{i\in I}M_{i}$}{direct sum of a family binoids}
\end {Corollary}
\begin {proof}
This is similar to Lemma \ref {LemProduct}.
\end {proof}

\begin {Lemma} \label{LemProductCom}
The product $\prod_{i\in I}M_{i}$ of a family $(M_{i})_{i\in I}$ of nonzero binoids is commutative if and only if all $M_{i}$ are commutative.
\end {Lemma}
\begin {proof}
This is easily verified.
\end {proof}

\bigskip

\section {Homomorphisms} \label{SecHom}
\markright{\ref{SecHom} Homomorphisms}

\begin {Definition} \label{DefKer}
Let $M$ and $N$ be binoids (or semibinoids). A map $\varphi:M\rto N$ is a (\gesperrt{semi}-) \gesperrt{binoid homomorphism} \index{binoid!-- homomorphism}\index{binoid homomorphism}\index{semibinoid!-- homomorphism}\index{homomorphism!binoid --}\index{homomorphism!semibinoid --}if it is a monoid (resp.\ semigroup) homomorphism which sends $\infty_{M}$ to $\infty_{N}$. Moreover, we call $\varphi$ a \gesperrt{monomorphism} \index{monomorphism}or \gesperrt{embedding} \index{embedding}if its is injective, an \gesperrt{epimorphism} \index{epimorphism}if it is surjective, and an \gesperrt{isomorphism} \index{isomorphism}if it is bijective. The set  $\im\varphi:=\varphi(M)$\nomenclature[Image]{$\im\varphi$}{image of $\varphi$} is the \gesperrt{image} of $\varphi$, \index{image}\index{homomorphism!image of a --}and the set\nomenclature[Kernel]{$\ker\varphi$}{kernel of $\varphi$}
$$\ker\varphi\,:=\,\{a\in M\mid\varphi(a)=\infty_{N}\}$$
is the \gesperrt{kernel} \index{homomorphism!kernel of a --}of $\varphi$. The set of all binoid homomorphisms from $M$ to $N$ is denoted by $\hom(M,N)$\nomenclature[Hom]{$\hom(M,N)$}{set of all binoid homomorphisms $M\rto N$}.
\end {Definition}

\begin {Remark}
A binoid homomorphism that satisfies $\ker=\zero$ need not be injective. For example, the binoid homomorphism $\varphi:\N^{\infty}\rto\trivial$ with $x\mto 0$ if $x\not=\infty$ and $\infty$ otherwise, fulfills $\ker\varphi=\zero$, but is not injective.
\end {Remark}

\begin {Example}
Let $M$ be a binoid.
\begin {ListeTheorem}
\item There are canonical binoid homomorphisms
$$\trivial\Rto M\Rto\zero\pkt$$
\item If $M$ is commutative, the map $M\rto M$ with $a\rto ka$ is a binoid homomorphism for every $k\in\N$.
\item Every element $a\in M$ defines a binoid homomorphism $\varphi_{a}:\N^{\infty}\rto M$ with $1\mto a$. Conversely, every binoid homomorphism $\N^{\infty}\rto M$ is determined by the image of $1$. Hence, $\hom(\N^{\infty},M)\cong M$. The image $\im\varphi_{a}$ is given by the subbinoid $\{\infty,na\mid n\in\N\}$, and $\ker\varphi_{a}\not=\zero$ if and only if $a$ is nilpotent.
\item If the operation on $M$ is given by $\cap$ or $\cup$, then $M\rto M^{\opc}$ with $A\mto A^{\opc}$ is in either case a binoid isomorphism. This holds particularly for the binoid $\Scal(V)$ and for the binoid defined by a topology on $V$, for instance, for the discrete topology $\Pset(V)\cong\Pset(V)^{\opc}$.
\end {ListeTheorem}
\end {Example}

\begin {Example} \label{ExPowersetIsom}
Let $I=\{1\kpkt n\}$, $n\ge1$.
\begin {ListeTheorem}
\item The canonical bijections $\Pset_{n,\cap}\rto\trivial^{n}$,  $A\mto(\delta_{i}(A))_{i\in I}$, and $\Pset_{n,\cup}\rto\trivial^{n}$, $A\mto(\bar{\delta}_{i}(A))_{i\in I}$, where
$$\delta_{i}(A)\,=\,\begin {cases}
0&\text{, if }i\in A\komma\\
\infty&\text{, otherwise,}
\end {cases}\quad\text{and}\quad\bar{\delta}_{i}(A)\,=\,\begin {cases}
\infty&\text{, if }i\in A\komma\\
0&\text{, otherwise,}
\end {cases}$$
for $i\in I$, are binoid isomorphisms.
\item Let $(M_{i})_{i\in I}$ be a family of binoids and $k\in I$. The projection on the $k$th component $$\prod_{i\in I}M_{i}\Rto M_{k}\komma\quad (a_{i})_{i\in I}\lto a_{k}\komma$$
is a binoid epimorphism, whereas the inclusion $M_{k}\embto\prod_{i\in I}M_{i}$ given by
$$a\lto(\infty\kpkt\infty,a,\infty\kpkt\infty)$$
is only a semibinoid homomorphism, and the inclusion given by
$$a\lto(0\kpkt0,a,0\kpkt0)$$
is only a monoid homomorphism (in either case $a$ is the $k$th component).
\end {ListeTheorem}
\end {Example}

The category of binoids $\Bsf$\nomenclature[B]{$\Bsf$}{category of binoids} possesses an initial object, namely the trivial binoid $\trivial$. The terminal object of $\Bsf$ is given by the zero binoid $\zero$. Note that there is no binoid homomorphism from the zero binoid into a nonzero binoid. The category of commutative binoids is a full subcategory of $\Bsf$ which we denote by $\cBsf$.\nomenclature[CB]{$\cBsf$}{category of commutative binoids}

\begin {Proposition}
The product in the category of binoids $\Bsf$ is given by the product.
\end {Proposition}
\begin {proof}
This follows immediately from Proposition \ref {PropUEproduct} and Corollary \ref{CorProduct}.
\end {proof}

The (finite) coproduct of commutative binoids will be described in Section \ref{SecSmashProduct}. The following statement is a Cayley-type embedding theorem for binoids.

\begin {Lemma} \label{LemMappingRealization}
Given a binoid $M$, there exists a binoid embedding 
$$M\Rto\map_{\infty}M\komma\quad x\lto(t_{x}:y\mto x+y)\pkt$$
\end {Lemma}
\begin {proof}
This map is a binoid homomorphism because $(t_{x}\circ t_{x^{\prime}})(y)=x+x^{\prime}+y=t_{x+x^{\prime}}(y)$ for all $x,y\in M$, $0\mto t_{0}=\id_{M}$, and $\infty\mto (t_{\infty}:y\mto\infty)$. The injectivity follows from the simple fact that $t_{x}=t_{x^{\prime}}$ is equivalent to $x+y=x^{\prime}+y$ for all $y\in M$, which implies $x=x^{\prime}$ for $y=0$.
\end {proof}

\bigskip

\section {Generators} \label{SecGenerators}
\markright{\ref{SecGenerators} Generators}

Now we consider generators of binoids and introduce free, free commutative, and semifree binoids as well as binoids defined by generators and relations.  Later, cf.\ Section \ref{SecFGbinoids}, finitely generated commutative binoids will be treated in more detail.

\begin {Definition} \label{DefFG}
Let $M$ be a binoid and $A\subset M$ a subset. Since the intersection of a family of subbinoids of $M$ is a subbinoid, there exists a smallest subbinoid of $M$ containing the set $A$ which will be called the binoid \gesperrt{generated} by $A$ and denoted by\nomenclature[AGenerate1]{$\langle A\rangle$}{subbinoid generated by the subset $A\subseteq M$} 
$$\langle A\rangle\pkt$$
If $M=\langle A\rangle$, we say $M$ is \gesperrt{generated} \index{binoid!-- generated by}by $A$ and call $A$ a \gesperrt{generating set} \index{binoid!generating set of a --} and its elements \gesperrt{generators} \index{binoid!generators of a --}of $M$. In this case, $A$ is called a \gesperrt {minimal generating set} \index{binoid!minimal generating set of a --}of $M$ if no proper subset of $A$ generates $M$. A binoid is \gesperrt{finitely generated} \index{binoid!finitely generated --}if it is generated by a finite subset. A finitely generated binoid that admits a (minimal) generating subset with $n$ elements and all other generating subsets consist of $\ge n$ elements is called \gesperrt{$n$-generated}. A \gesperrt{finite} \index{binoid!finite --}binoid consists of finitely many elements only. These definitions apply to a semibinoid $S$ and a subset $A\subseteq S$ if they hold for the binoid $S\onull$ and $A$.\index{semibinoid!-- generated by}\index{semibinoid!generators of a --}\index{semibinoid!generating set of a --}\index{semibinoid!finitely generated --}\index{semibinoid!finite --}
\end {Definition}

A generating set $A$ of a binoid $M$ generates $M$ as a monoid if and only if there are elements $a,b\in M\opkt$ with $a+b=\infty$ (this property will be called non-integral later). Otherwise, $A\cup\zero$ generates $M$ as a monoid. In particular, a binoid is finitely generated if and only if it is finitely generated as a monoid.

\begin {Example}\label {ExBinMonGenerators}
 \begin {ListeTheorem}
\item[]
\item The zero binoid $\zero$ is finitely generated by $\emptyset$ as a monoid and as a binoid. The trivial binoid $\trivial$ is generated as a binoid by $\emptyset$, but its monoid generator is $\infty$.
\item $\N^{\infty}=\langle1\rangle$ as a binoid but the monoid generators of $\N^{\infty}$ are $1$ and $\infty$. In general, $(\N^{n})^{\infty}$, $n\ge 1$, is finitely generated as a binoid by the elements\nomenclature[EA1]{$e_{i}$}{$=(0\kpkt 0,1,0\kpkt 0)$, where $1$ is the $i$th entry} 
$$e_{i}:=(0\kpkt 0,1,0\kpkt0)\komma$$
$i\in\{1\kpkt n\}$, where $1$ is the $i$th entry, and as a monoid by $(\infty\kpkt\infty)$ and $e_{i}$, $i\in\{1\kpkt n\}$.
\item If $(M_{i})_{i\in I}$ is a finite family of binoids and $A_{i}\subseteq M_{i}$ a generating set of $M_{i}$, $i\in I$, then $\prod_{i\in I}M_{i}$ is generated by\nomenclature[EA3]{$e_{i,\infty}$}{$=(0\kpkt0,\infty,0\kpkt 0)$, where $\infty$ is the $i$th entry} 
\nomenclature[EA2]{$ae_{i}$}{$=(0\kpkt0,a,0\kpkt 0)$, where $a$ is the $i$th entry}  
$$e_{i,\infty}:=(0\kpkt 0,\infty, 0 \kpkt0)\quad\text{and}\quad ae_{i}:=(0\kpkt 0,a,0\kpkt0)\komma$$ 
$a\in A_{i}$, $i\in I$, where $\infty$ and $a$ are the $i$th component of $e_{i,\infty}$ and $ae_{i}$, respectively. For instance, $(\N^{\infty})^{n}$, $n\ge 1$, is generated as a monoid and as a binoid by the elements $e_{i}$ and $e_{i,\infty}$, $i\in I$.
\item Subbinoids of a finitely generated binoid need not be finitely generated. For instance, the subbinoids $((\N_{\ge1}^{\infty})^{n})\onull\subseteq(\N^{\infty})^{n}$ and $((\N_{\ge1})^{n})^{\mathbf{0},\infty}\subseteq(\N^{n})^{\infty}$ are not finitely generated for $n\ge2$ since every generating set has to contain the elements $(n,1\kpkt 1)$, $n\ge1$, for instance.\nomenclature[N]{$\N_{\ge d}$}{$=\{n\in\N\mid n\ge d\}$}
\item Let $n,m\in\N$. The operation $n\diamond m:=\lcm(n,m)$ on $\N$ yields a commutative binoid $(\N,\diamond,1,0)$ that is not finitely generated since every generating set has to contain all powers of prime numbers.
\item If a topology $\Tcal$ on $V$ admits a basis $\Bscr$, then $\Bscr$ generates the binoid $(\Tcal,\cup,V,\emptyset)$ if and only if $\Tcal$ is a quasi-compact topology on $V$; that is, when every open covering admits a finite one. For instance, the binoid $\Pset(V)_{\cap}$ defined by the discrete topology on $V$ is finitely generated if and only if $V$ is finite.
\end {ListeTheorem}
\end {Example}

\begin {Lemma}
A commutative binoid is finite if and only if it is finitely generated and every one-generated subbinoid is finite.
\end {Lemma}
\begin {proof}
Only the if part of the statement is not trivial. Every finitely generated commutative binoid $M=\langle x_{1}\kpkt x_{r}\rangle$, $r\in\N$, gives rise to a canonical binoid epimorphism
$$\prod_{i=1}^{r}\langle x_{i}\rangle\Rto M\quad\text{with}\quad\text (n_{1}x_{1}\kpkt n_{r}x_{r})\lto\sum_{i=1}^{r}n_{i}x_{i}\pkt$$
The product is finite since all its components are finite by assumption. Now the finiteness of $M$ follows from the surjectivity.
\end {proof}

If $M$ is a binoid generated by the (not necessarily finite) set $A$, every element $f\in M\opkt$ can be written as a finite sum of the generators. In case $M$ is commutative, we have the following notation
$$f=\sum_{a\in A}n_{a}a$$
with $n_{a}\in\N$ and $n_{a}=0$ for almost all $a\in A$. Of course, this expression need not be unique.

\begin {Definition}
Let $V$ be an arbitrary set of elements. Let $M(V)$ denote the free monoid consisting of all finite sums of elements in $V$ with addition given by
$$(x_{1}\pluspkt x_{n})+(y_{1}\pluspkt y_{m})\,:=\, x_{1}\pluspkt x_{n}+y_{1}\pluspkt y_{m}\komma$$
$x_{i},y_{j}\in V$, $i\in\{1\kpkt n\}$, $j\in\{1\kpkt m\}$, and the sum over the empty set ($:=0$) as identity element. The binoid $M(V)^{\infty}=:\freenc(V)$\nomenclature[Free]{$\freenc(V)$}{free binoid on $V$} is called the \gesperrt{free binoid} \index{binoid!free --}on $V$.
\end {Definition}

\begin {Lemma}
Every element $\not=\infty$ of $\freenc(V)$ can be written uniquely as a sum of elements of $V$.
\end {Lemma}
\begin {proof}
This is clear since $x_{1}\pluspkt x_{n}=(x_{1})\pluspkt (x_{n})$.
\end {proof}

Obviously, $\freenc(V)$ is commutative if and only if $V=\{x\}$ is a singleton. In a commutative binoid with more than one generator, one cannot expect unique expression of elements.

\begin {Lemma} \label {LemFfrakEpsilon}
Let $M$ be a binoid. Every subset $A=\{a_{i}\mid i\in I\}\subseteq M$ gives rise to a unique binoid homomorphism 
$$\varepsilon:\freenc(I)\Rto M\komma\quad i\lto a_{i}\komma$$
$i\in I$, which is surjective if and only if $A$ generates $M$.
\end {Lemma}
\begin {proof}
This is easily verified.
\end {proof}

The preceding lemma implies, in particular, that every generating set $A=\{a_{i}\mid i\in I\}$ of a binoid $M$ yields a canonical binoid epimorphism $\varepsilon:\freenc(I)\rto M$ with $i\mto a_{i}$, $i\in I$.

\begin {Definition}
Let $M$ be a commutative binoid. We say $M$ is a \gesperrt{free commutative} \index{binoid!free commutative --}binoid if there exists a binoid isomorphism $\varepsilon:(\N^{(I)})^{\infty}\rto M$ for some (possibly infinite) set $I$. The family $(\varepsilon(e_{i}))_{i\in I}$ of elements in $M$ is called a \gesperrt{basis} \index{basis}\index{binoid!basis of a --}of $M$. The free commutative binoid with basis $V$ will be denoted by $\free(V)$ or by $\free_{n}$ if $V=\{1\kpkt n\}$\nomenclature[FreeC]{$\free(V),\free_{n}$}{free commutative binoid on $V$ and $\{1\kpkt n\}$}.
\end {Definition}

\begin {Lemma}
Every element $f\not=\infty$ of a finitely generated free commutative binoid with basis $(x_{i})_{i\in I}$ admits a unique expression as $f=\sum_{i\in I}n_{i}x_{i}$ with $n_{i}=0$ for almost all $i\in I$.
\end {Lemma}
\begin {proof}
This follows from the fact that every element $a\in(\N^{(I)})^{\infty}$ has this property for $I$ finite.
\end {proof}

\begin {Example}
The commutative binoid $\Z^{\infty}$ is not free commutative because there are non-trivial invertible elements, which yield non-unique expressions of $0\in\Z^{\infty}$. The canonical binoid epimorphism $\varphi:(\N^{2})^{\infty}\rto\Z^{\infty}$ with $(1,0)\mto 1$ and $(0,1)\mto -1$ is not injective because  $\varphi^{-1}(0)=\{(n,n)\mid n\in\Z\}$ for instance. See also Example \ref{ExpSemifree}(2).
\end {Example}

Lemma \ref{LemFfrakEpsilon} carries over to free commutative binoids.

\begin {Lemma} \label {LemHomFreeCom}
Given a family $(a_{i})_{i\in I}$ of commuting elements in a binoid $M$, i.e.\ $a_{i}+a_{j}=a_{j}+a_{i}$ for all $i,j\in I$, there exists a unique binoid homomorphism 
$$\varepsilon:\free(I)\Rto M\komma\quad i\lto a_{i}\komma$$
$i\in I$, which is surjective if and only if $\{a_{i}\mid i\in I\}$ generates $M$.
\end {Lemma}
\begin {proof}
All statements are easily verified.
\end {proof}

In a commutative binoid $M$ every set $\{a_{i}\mid i\in I\}$ of elements in $M$ gives rise to a binoid homomorphism
$$(\N^{(I)})^{\infty}\Rto M\komma\quad e_{i}\lto a_{i}\komma$$
$i\in I$, that is surjective if and only if $\{a_{i}\mid i\in I\}$ generates  $M$. For instance, 
$$\varphi:(\N^{n})^{\infty}\Rto(\N^{\infty})^{n}\komma\quad e_{i}\lto e_{i}\komma$$
$i\in\{1\kpkt n\}$, is not surjective because $\{e_{1}\kpkt e_{n}\}$ is only half of a generating set of $(\N^{\infty})^{n}$, cf.\ Example \ref{ExBinMonGenerators}(3). Note that 
$$\pi:(\N^{\infty})^{n}\Rto(\N^{n})^{\infty}\quad\text{with}\quad e_{i}\lto e_{i}\quad\text{and}\quad e_{i,\infty}\lto\infty\komma$$
$i\in\{1\kpkt n\}$, is an epimorphism such that $\pi\varphi=\id_{(\N^{n})^{\infty}}$.

\begin {Definition}
Let $V$ be an arbitrary set. The binoids that arise from $\freenc(V)$ and $\free(V)$ by taking additional relations $\Rcal_{i}$, $i\in I$, among the elements of $V$ into account will be denoted by
$$\freenc(V)/(\Rcal_{i})_{i\in I}\quad\text{and}\quad\free(V)/(\Rcal_{i})_{i\in I}\pkt$$
\end {Definition}

Here we tacitly assume the reader to be familiar with the characterization of a monoid defined by generators and relations given in the above definition in a little sketchy way. A precise justification of the notation (in the commutative case) can be found in Example \ref{ExGen}.

\begin{Definition} \label{DefRepFree}
Let $M$ be a nonzero commutative binoid. We say $M$ is a \gesperrt{semifree} \index{binoid!semifree --}\index{semifree}binoid with \gesperrt{semibasis} \index{semibasis}$(a_{i})_{i\in I}$ if $M$ is generated by $\{a_{i}\mid i\in I\}$ such that every element $f\in M\opkt$ can be written uniquely as $f=\sum_{i\in I}n_{i}a_{i}$ with $n_{i}=0$ for almost all $i\in I$. Then the set $\{a_{i}\mid n_{i}\not=0\}=:\supp(f)$ is called the \gesperrt{support} \index{element!support of an --}\index{support}of $f$.\nomenclature[Supp]{$\supp(f)$}{support of $f$} A commutative semibinoid $S$ is semifree if the binoid $S\onull$ is semifree.\index{semibinoid!semifree --}
\end{Definition}

Obviously, every finitely generated free commutative binoid is semifree, cf.\ also Corollary \ref {CorFree}, and a semibasis is always a minimal generating set. Semifree binoids represent an important class of commutative binoids, for this see the characterization in Corollary \ref{CorQuotientRepFree}.

\begin {Example}\label{ExpSemifree}
\begin {ListeTheorem}
\item []
\item The binoid $(\N^{(I)})^{\infty}$ is semifree with semibasis $e_{i}$, $i\in I$.
\item The canonical generators $1$ and $-1$ of $\Z^{\infty}$ form no semibasis since
$0=n\cdot 1+n\cdot(-1)$ for all $n\ge1$.
In fact, $\Z^{\infty}$ is not semifree. Every minimal generating set of $\Z^{\infty}$ is given by two integers $n,m\in\Z$ with $n>0$, $m<0$, and $\gcd(n,-m)=1$. Hence, $kn+lm=1$ for some $k,l>0$, which yields $\tilde{k}n+\tilde{l}m=-1$ for some $\tilde{k},\tilde{l}>0$ by adding $mn-nm=0$ sufficiently often to $-kn-lm=-1$. By applying this to the equations above, we obtain a non-unique expression of $0$ in terms of $n$ and $m$. See also Lemma \ref{LemPropSemifree}.
\item The semibasis of the free commutative binoid $(\N,\cdot,1,0)$ is given by $\{p\mid p\in\N$ prime$\}$ by the fundamental theorem of arithmetic, cf.\ \cite[Hauptsatz 10.1]{SchejaStorch}.
\item A commutative binoid with a non-trivial idempotent (see below, Section  \ref{SecAdditivityProp}) is never semifree. For instance, $\Pset(V)_{\cup}$ is not semifree since $A=A\cup A=A\cup A\cup A=\cdots$ for every $A\subseteq V$. However, the singletons $\{v\}$, $v\in V$, $V$ finite, generate $\Pset(V)_{\cup}$ and every set $A\subseteq V$ is uniquely given as $A=\cup_{v\in A}\{v\}$. Thus, $\Pset(V)_{\cup}$ can be considered as semifree \emph{up to idempotence}. In Section  \ref{SecSimplCompl}, we will encounter these kind of binoids.

\item The generating set $A=\{1/p_{1}^{n_{1}}\cdots p_{r}^{n_{r}}\mid p_{1}\kpkt p_{r}\in\N$ prime, $n_{1}\kpkt n_{r}\ge1,r\in\N\}$ of $(\Q_{\ge0}^{\infty},+,0,\infty)$ is no semibasis because, for instance, $\frac{1}{2}=\frac{1}{4}+\frac{1}{4}=\frac{1}{8}+\frac{1}{8}+\frac{1}{8}+\frac{1}{8}=\ldots$. Since this sequence is infinite, one cannot deduce a semibasis from $A$ by omitting elements.
\item The binoid $\free(x_{1},x_{2})/(x_{1}+x_{2}=\infty)$ is semifree because every element $\not=\infty$ can be written uniquely as $nx_{1}$ or $mx_{2}$ for certain $n,m\ge0$. This example shows that Lemma \ref{LemHomFreeCom} need not be true for semifree binoids since there exists, for instance, no binoid homomorphism $$\free(x_{1},x_{2})/(x_{1}+x_{2}=\infty)\Rto\free(y_{1},y_{2})$$ 
with $x_{i}\mto y_{i}$, $i\in\{1,2\}$.
\item Another interesting class of semifree binoids are simplicial binoids, which will be introduced in Section \ref{SecSimplBinos}.
\end {ListeTheorem}
\end {Example}

\bigskip

\section {Additivity properties}  \label{SecAdditivityProp}
\markright{\ref{SecAdditivityProp} Additivity properties}

This section deals with additivity properties of binoids. Before we recall classical concepts that originate from monoid theory, we point out those properties of a binoid which have no counterpart when there is no absorbing element, namely non-integrity and nilpotence. Concerning those properties that translate directly from monoids to binoids we follow essentially \cite{Gilmer}. Then the additively closed subobjects of a binoid, its filters, are introduced. These are very useful when we are missing the theory of (prime) ideals of a binoid. The connection will be shown in Chapter \ref {ChapIdealTheory}. We close this section with the behavior of additive properties under homomorphisms. 

\medskip

The phenomena that may occur if an absorbing element exists are described in the following definitions.

\begin {Definition}
Let $M$ be a binoid (or semibinoid). An element $a\in M$ is called \gesperrt{nilpotent} \index{element!nilpotent --}\index{nilpotent}if $na=a\pluspkt a=\infty$ for some $n\ge 1$. The set of all nilpotent elements will be denoted by $\nil(M)$. We say $M$ is \gesperrt{reduced} \index{semibinoid!reduced --}\index{binoid!reduced --}\index{reduced}if $\nil(M)=\zero$\nomenclature[Nilpotent]{$\nil(M)$}{set of all nilpotent elements in $M$}, and $M$ is \gesperrt{strongly reduced}\index{semibinoid!strongly reduced --}\index{binoid!strongly reduced --}\index{reduced!strongly --} if $a+a+b=\infty$ for $a,b\in M$ implies $a+b=\infty$.
\end {Definition}

\begin {Lemma}\label{LemStrongRed=Red}
\begin {ListeTheorem}
\item[]
\item A strongly reduced binoid is reduced.
\item A commutative binoid is strongly reduced if and only if it is reduced.
\end {ListeTheorem}
\end {Lemma}
\begin {proof}
(1) If $na=\infty$ for some $a$ and $n\ge2$ in a strongly reduced binoid, then by applying successively
$$\infty=na=a+a+(n-2)a=a+(n-2)a=(n-1)a\komma$$
we obtain $\infty=2a=a+a+0$, which yields $\infty=a+0=a$.
(2) Any equation $a+a+b=\infty$ in a commtative binoid yields $2(a+b)=\infty$ by adding $b$. Hence, $a+b=\infty$ if the binoid is reduced.
\end {proof}

\begin {Definition}
Let $M\not=\zero$ be a binoid (or semibinoid). An \gesperrt{integral} \index{element!integral --}\index{integral}element $a\in M\opkt$ satisfies the property that $a+b=\infty$ or $b+a=\infty$ implies $b=\infty$. The set of all integral elements will be denoted by $\opint(M)$ \nomenclature[Integral]{$\opint(M)$}{set of all integral elements in $M$}and the complement $M\setminus\opint(M)$ by $\nonint(M)$\nomenclature[Integral]{$\nonint(M)$}{set of all non-integral elements in $M$}. We say $M$ is \gesperrt{integral} \index{binoid!integral --}\index{semibinoid!integral --}\index{integral}if $M\opkt$ is a submonoid (resp.\ subsemigroup) of $M$; that is, $M\opkt$ consists only of integral elements.
\end {Definition}

Every binoid that emerges from a monoid by adjoining an absorbing element is integral, and all integral binoids are of this type.

\begin {Example} \label{ExpIntegrity}
\begin {ListeTheorem}
\item[]
\item By definition, $(\N^{n})^{\infty}$ is an integral binoid and $(\Z^{n})^{\infty}$ is a binoid group (see below). On the other hand, the binoids $(\N^{\infty})^{n}$ and $(\Z^{\infty})^{n}$, $n\ge 2$, are obviously not integral.
\item The notion of integrity for a ring $R$ and for its underlying binoid $(R,\cdot,1,0)$ coincides; that is, the non-integral elements are precisely the zero-divisors of $R$. Therefore, $\nonint(R)=\{0\}$ if and only if $R$ is an (integral) domain.
\item A binoid $\Tcal_{\cap}$ defined by a topology on $V$ is integral if and only if $V$ is irreducible with respect to $\Tcal$; \index{irreducible}\index{space!irreducible --}that is, $V\not=\emptyset$ and $U\cap V\not=\emptyset$ for arbitrary nonempty open subsets $U,V\subseteq \Tcal$. 

For instance, the binoid defined by the Zariski topology on $\A_{K}^{n}$ and $\A^{n}(K)$ for $n\ge 1$ and $K$ an algebraically closed field, is integral as well as the binoids given by the subspace topology on an (affine) variety $V\subseteq\A_{K}^{n}$, cf.\ \cite[Chapter I.1]{Hartshorne}.
\item The binoid $\Pset(V)_{\cap}$ is never integral for $\#V\ge 2$ since $A\cap A^{\opc}=\emptyset$ for every subset $A\subset V$. Similarly, $\Pset(V)_{\cup}$ with $\#V\ge 2$ is never integral. From the criterion given in (3) we can easily derive non-trivial integral subbinoids of 
$\Pset(V)_{\cap}$ for $\#V\ge 2$, like $\langle \{Y,x_{1}\}\kpkt \{Y,x_{k}\}\rangle\subseteq\Pset(V)_{\cap}$, where $\emptyset\not=Y\subsetneq V$ such that $V\setminus Y=\{x_{1}\kpkt x_{k}\}$ with $k\ge 1$.
\item The set of integral and non-integral elements in the product of a family $(M_{i})_{i\in I}$ of binoids are given by
$$\opint\Big(\prod_{i\in I}M_{i}\Big)\,\,=\,\,\prod_{i\in I}\opint(M_{i})$$
and
$$\nonint\Big(\prod_{i\in I}M_{i}\Big)\,\,=\,\,\{(a_{i})_{i\in I}\!\mid\! a_{k}\!\in\!\nonint(M_{k})\text{ for at least one }k\in I\}\pkt$$
\end {ListeTheorem}
\end {Example}

\begin {Lemma}
Nilpotent elements are not integral; that is, $\nil(M)\subseteq\nonint(M)$ for every nonzero binoid $M$. In particular, an integral binoid is reduced.
\end {Lemma}
\begin {proof}
This follows directly from the definitions.
\end {proof}

\begin {Lemma} \label {LemINT}
The subset $M_{\opint}:=\opint(M)\cup\zero\subseteq M$\nomenclature[MACongruence]{$M_{\opint}$}{$=\opint(M)\cup\zero$} is an integral subbinoid for every nonzero binoid $M$.
\end {Lemma}
\begin {proof}
This is easily verified.
\end {proof}

\begin {Example}
The binoid $M:=\free(x,y)/(x+y=\infty)$ has no non-trivial integral elements (i.e.\ $M_{\opint}=\trivial$), but two non-trivial integral subbinoids, namely $\langle x\rangle$ and $\langle y\rangle$. 
\end {Example}

Now we translate common properties of monoids (resp.\ semigroups) and their elements to binoids. We start with the notions of units and idempotents which are independent of the existence of an absorbing element.

\begin {Definition}
Let $M$ be a monoid (or nonzero binoid). An element $u$ in $M$ is a \gesperrt{unit} \index{unit}if there exists an element $a\in M$ such that $a+u=u+a=0$. The element $a$ is the unique (additive) \gesperrt {inverse} \index{element!inverse --}\index{inverse}of $u$ and will be denoted by $\minus u$. The set of all units $M\okreuz$\nomenclature[M]{$M\okreuz$}{set of all units in $M$} is a submonoid of $M$ which is a group, the \gesperrt {unit group} \index{binoid!unit group of a --}\index{monoid!unit group of a --}\index{unit!-- group}of $M$. The set of all nonunits $M\setminus M\okreuz$ will be denoted by $M\Uplus$\nomenclature[M]{$M\Uplus$}{set of all nonunits in $M$}. We say $M$ is \gesperrt {positive} \index{binoid!positive --}\index{monoid!positive --}\index{positive}if it has a trivial unit group (i.e.\ $M\setminus\{0\}=M\Uplus$). A \gesperrt{binoid group} \index{binoid!-- group}\index{binoid group}is a binoid $G$ such that $G\opkt=G\okreuz$; that is, $G\opkt$ is a group. 
\end {Definition}

Among many other names for the property of a monoid or binoid to have a trivial unit group \emph{pointed} in combinatorics (\cite{MillerSturmfels}) and \emph{sharp} in geometry (\cite{Ogus}) are also very common, while we follow \cite{BrunsGubeladze}.

We fix the following notation: if $R$ is a ring, then $R^{\infty}$ denotes the (commutative) binoid group that arises by adjoining an absorbing element to $R$ with respect to its \emph{additive} structure. For instance,
$$(\Z^{n})^{\infty}\,=\,(\Z^{n}\cup\{\infty\},+,(0\kpkt0),\infty)\quad\text{and}\quad(\Z/m\Z)^{\infty}\,=\,(\Z/m\Z\cup\{\infty\},+,[0],\infty)\komma$$
where $n\ge1$, and $m\ge2$.\nomenclature[Z1]{$(\Z^{n})^{\infty}$}{$=(\Z^{n}\cup\{\infty\},+,(0\kpkt0),\infty)$, $n\ge1$}\nomenclature[Z2]{$(\Z/n\Z)^{\infty}$}{$=(\Z/n\Z\cup\{\infty\},+,[0],\infty)$, $n\ge2$}

\begin {Example} \label{ExpProdUnits}
If $(M_{i})_{i\in I}$ is a family of nonzero binoids, then $\left(\prod_{i\in I}M_{i}\right)\okreuz=\prod_{i\in I}M_{i}\okreuz$.
\end {Example}

\begin {Lemma}
Every unit is integral; that is, $M\okreuz\subseteq\opint(M)$ for a nonzero binoid $M$.
\end {Lemma}
\begin {proof}
Obvious.
\end {proof}

\begin {Definition} \label{DefIdempotent}
Let $M$ be a semigroup. An element $f\in M$ is called \gesperrt{idempotent} \index{element!idempotent --}\index{idempotent}if $f+f=f$. In a \gesperrt{boolean} \index{semigroup!boolean --}\index{monoid!boolean --}\index{semibinoid!boolean --}\index{binoid!boolean --}\index{boolean}semigroup every element is idempotent. The set of all idempotent elements will be denoted by $\bool(M)$\nomenclature[Boolean]{$\bool(M)$}{set of all idempotent elements in $M$}. A commutative semigroup that is boolean is called a \gesperrt{semilattice}\index{semilattice}.
\end {Definition}

The operation of a semilattice $L$ is usually denoted by $\sqcup$ or $\sqcap$ and $L$ is called a \gesperrt{join-} or \gesperrt{meet-} semilattice, respectively. In either case, the operation gives rise to a partial order $\le$ on $L$ by setting $x\le y:\eq x\sqcup y=y$ or $x\le y:\eq x\sqcap y=x$.

\begin {Example}
\begin {ListeTheorem}
\item []
\item The identity element and the absorbing element are always idempotent elements. In particular, the set of idempotents of a commutative binoid $M$ is a subbinoid since $2(a+b)=2a+2b=a+b$ for idempotent elements $a, b\in M$. More precisely, $\bool(M)$ is the largest boolean subbinoid of a commutative binoid $M$.
\item The set $\Pset(V)$ is a meet- and join-semilattice with respect to $\cap$ and $\cup$, respectively.
\item The commutative binoid defined by a topology is boolean such that $\subseteq$ ($\supseteq$) is the induced partial order with respect to the operation $\cap$ ($\cup$). Moreover, Example \ref {ExpIntegrity}(3) and (4) show that a boolean binoid can be integral in contrast to a boolean algebra $\not=\{0\}$, which has characteristic $2$ always.
\end {ListeTheorem}
\end {Example}

\begin {Lemma} \label{LemBool=>PosRed}
Every boolean binoid is positive and reduced.
\end {Lemma}
\begin {proof}
Only the positivity of a boolean binoid need to be shown. For this let $a+b=0$ for two idempotent elements $a$ and $b$. Adding $a$ from the left and $b$ from the right yields $a=a+b=b$, hence $a=b=0$.
\end {proof}

The well-known notion of torsion and cancellativity in monoid theory need to be modified for binoids.

\begin {Definition}
Let $M$ be a binoid (or semibinoid). An element $a$ in $M$ is a \gesperrt{torsion} \index{element!torsion --}element if $a=\infty$ or $na=nb$ for some $b\in M$ with $b\not=a$ and $n\ge 2$. We say $M$ is \gesperrt{torsion-free} \index{binoid!torsion-free --}\index{semibinoid!torsion-free --}\index{torsion-free}if there are no other torsion elements in $M$ besides $\infty$; that is, $na=nb$ implies $a=b$ for every $a,b\in M$ and $n\ge 1$. If $na=nb\not=\infty$ implies $a=b$ for every $a,b\in M$ and $n\ge 1$, then $M$ is called \gesperrt{torsion-free up to nilpotence}.\index{binoid!torsion-free up to nilpotence --}\index{semibinoid!torsion-free up to nilpotence --}\index{torsion-free!-- up to nilpotence}
\end {Definition}

By definition, a binoid is torsion-free if and only if it is reduced and torsion-free up to nilpotence. With this notation, a monoid with no absorbing element is torsion-free if $M^{\infty}$ is a torsion-free binoid. A group $G$ is a torsion group if and only if all elements of $G^{\infty}$ are torsion elements. In particular, the unit group $M\okreuz$ need not be torsion-free. 

An important example of binoids that are torsion-free up to nilpotence but (in general) not reduced is given in Corollary \ref{CorQuotientRepFree}. The set of all torsion elements in a binoid that is torsion-free up to nilpotence is precisely $\nil M$. In general, the set of all torsion elements in $M$ has no structure as the following example shows.

\begin {Example}
Consider the binoid $M=\free(x,y)/(10x+2y=\infty)$. The elements $x$ and $y$ are no torsion elements but 
every element $nx+my$ with $n, m\ge1$ is a torsion element.
\end {Example}

\begin {Lemma} \label{LemTorsion}
\begin {ListeTheorem}
\item[]
\item Nilpotent elements are torsion elements.
\item Boolean binoids are torsion-free.
\end {ListeTheorem}
\end {Lemma}
\begin {proof}
Both statements follow directly from the definitions.
\end {proof}

The second statement of the preceding lemma shows that there are non-trivial finite binoids which are torsion-free.

\begin {Definition}
Let $M$ be a commutative monoid (or binoid). Two elements $a,b\in M$ are called \gesperrt{asymptotically equivalent} \index{asymptotically equivalent}\index{element!asymptotically equivalent --s}if there exists an $n_{0}\in\N$ with $na=nb$ for all $n\ge n_{0}$. We say that $M$ is \gesperrt{free of asymptotic torsion} \index{free of asymptotic torsion}\index{binoid!free of asymptotic torsion --}\index{asymptotic torsion}if any two distinct elements of $M$ are not asymptotic equivalent. 
\end {Definition}

Of course, asymptotically equivalent elements are torsion elements and nilpotent elements are asymptotically equivalent to $\infty$.

\begin {Lemma} \label{LemAsymEquiv}
Let $M$ be a  commutative binoid and $a,b\in M$. The following conditions on $a$ and $b$ are equivalent:
\begin {ListeTheorem}
\item $a$ and $b$ are asymptotically equivalent.
\item $na=nb$ and $(n+1)a=(n+1)b$ for an $n\in\N$.
\item $na=nb$ and $ma=mb$ for some $n,m\in\N$ with $\gcd(n,m)=1$.
\end {ListeTheorem}
\end {Lemma}
\begin {proof}
The implications $(1)\Rarrow(2)\Rarrow(3)$ are trivial. For $(3)\Rarrow(1)$ consider the set $N:=\{l\in\N^{\infty}\mid la=lb\}$. Obviously, $N\subseteq\N^{\infty}$ is a subbinoid because all possible combinations of elements in $N$ lie in $N$. Now take $n,m\in N$ with $\gcd(n,m)=1$. By \cite[Theorem 2.2]{Gilmer}, there are for every $k\in\N$ with $k\ge(n-1)(m-1)$ non-negative integers $r$ and $s$ such that $k=rn+sm\in N$. Hence, $ka=kb$ for all $k\ge(n-1)(m-1)$.
\end {proof}

A weaker property than being asymptotically equivalent, called functionally equivalent, will come into play at the end of Section \ref{SecKpoints}.

\begin {Definition}
Let $M\not=\zero$ be a binoid (or semibinoid). An element $a\in M\opkt$ is called \gesperrt{cancellative} \index{element!cancellative --}\index{cancellative}if from $b+a=c+a\not=\infty$ or $a+b=a+c\not=\infty$, for $b,c\in M$ it follows $b=c$. We say $M$ is \gesperrt{cancellative} \index{binoid!cancellative --}\index{semibinoid!cancellative --}if every element $\not=\infty$ is cancellative. The set of all cancellative elements in $M$ will be denoted by $\opcan(M)$\nomenclature[Cancellative]{$\opcan(M)$}{set of all cancellative elements in $M$}. $M$ is called \gesperrt{regular} \index{binoid!regular --}\index{semibinoid!regular --}\index{regular}if it is integral and cancellative element.
\end {Definition}

A cancellative binoid has no non-trivial idempotent elements. The trivial binoid is regular, whereas the zero binoid was excluded from the definitions of an integral and cancellative binoid. 

Note that a cancellative element in a binoid need not be integral, it may even be nilpotent. For instance, the binoid 
$$\free(x)/(nx=\infty)$$
with $n\ge2$ is cancellative. Thus, for binoids, cancellativity is far from being such a strong property as it is for monoids. The equivalent property for binoids is regularity because a binoid $M$ is regular if and only if $M\opkt$ is a \emph{cancellative monoid}; \index{monoid!cancellative --}that is, $M\opkt$ is a monoid in which $b+a=c+a$ or $a+b=a+c$ implies $b=c$ for all $a,b,c\in M$. See also Lemma \ref{LemmaCanc}(2) and Proposition \ref{PropCancellation} with the subsequent Remark \ref{RemCancellation}. Similar to a cancellative commutative monoid, which can always be embedded in a commutative group, cf.\ \cite[Theorem 1.2]{Gilmer}, any regular commutative binoid $M$ admits an embedding $M\embto G$ into a commutative binoid group $G$. For this result see Proposition \ref{PropGroupEmbedding}. 

\begin {Example}
The binoid $\Pset(V)_{\cap}$ is cancellative for $\#V<3$ and not cancellative for $\#V\ge 3$ since $\{1,2\}\cap\{1,3\}=\{1,2\}\cap\{1\}$. By a similar argument, a subbinoid $N\subseteq\Pset(V)_{\cap}$ is not cancellative if there are $A,B\in N$ with $A\cap B\not=\emptyset$. Analogous statements hold for $\Pset(V)_{\cup}$.
\end {Example}

\begin {Lemma} \label{LemmaCanc}
Let $M$ be a binoid.
\begin {ListeTheorem}
\item $\opcan(M)^{\infty}$ is a cancellative subbinoid of $M$.
\item Units are cancellative. Conversely, if $M$ is regular and finite, then $M$ is a binoid group.
\end {ListeTheorem}
\end {Lemma}
\begin {proof}
Both statements are easily verified except the converse of (2), which follows from the same argument that proves the analogous statement in group theory.
\end {proof}

\begin {Example}
The binoid $\free(x,y)/(x+y=y,2x=x)$ has no non-trivial cancellative elements (i.e.\ $\opcan(M)^{\infty}=\trivial$), but a non-trivial cancellative subbinoid, namely $\langle y\rangle$.
\end {Example}

\begin {Lemma} \label{LemPropSemifree}
Semifree binoids are positive, cancellative, and torsion-free up to nilpotence.
\end {Lemma}
\begin {proof}
This is easily verified.
\end {proof}

The converse of the preceding lemma is false. For instance, the binoid
$$M:=\free(x,y)/(2x=3y)$$
is positive, torsion-free, and cancellative but not semifree. Note that $\{x,y\}$ is a minimal generating set of $M$ (which is unique by Proposition \ref{PropUniqueMinSyst} below) and that every element $f\in M\opkt$ can be written uniquely as
$$f=my\quad\text{or}\quad f=x+my$$
for some $m\in\N$. The positivity is clear. To show that $M$ is torsion-free, one only needs to check that an equation like $k(x+my)=kny$ for some $k\ge2$ is not possible. This is true since otherwise it would imply $k=2k^{\prime}$ and therefore yield the impossible equation $3+2m=2n$. Similarly, one can show that the generators $x$ and $y$ are cancellative elements, which implies the cancellativity of $M$ by the following lemma.

\begin {Lemma} \label{LemIntCanGenerators}
A binoid is integral or cancellative if and only if a generating set is so.
\end {Lemma}
\begin {proof}
The only if part of the assertion is obvious. So let $A\subseteq M$ be a generating set consisting of integral elements and suppose that $a+b=\infty$ for some $a,b\in M$. Let $b\not=\infty$ so that $b=\tilde{b}+x$ for some $\tilde{b}\in M\opkt$ and $x\in A$. Hence, $\infty=a+b=a+\tilde{b}+x$, which implies that $a+\tilde{b}=\infty$ by the integrality of $x$. Proceeding the same way yields $a+y=\infty$ for some $y\in A$, hence $a=\infty$. The assertion for the cancellativity of $M$ follows similarly.
\end {proof}

Even if the generators of a binoid are no torsion elements, the binoid may not be torsion-free. Consider, for instance, the binoid $\free(a,b,c,d)/(n(a+b)=n(c+d))$ with $n\ge2$. Also, a binoid may have non-trivial cancellative elements though all generators are not cancellative as the following example shows.

\begin {Example}
The binoid $M:=\free(x,y)/(3x=3y=\infty,2x+y=x+2y)$ is positive, not reduced (hence not integral), and not cancellative since both generators $x$ and $y$ are not cancellative for instance. However, $a:=2x+y$ is a non-trivial cancellative element because
$$(2x+y)+x=\infty\quad\text{and}\quad(2x+y)+y=(x+2y)+y=x+3y=\infty\komma$$
and so there are no equations like $a+b=a+c\not=\infty$ with $b,c\in M$.
\end {Example}

\begin {Lemma} \label{LemCanM=0}
If $M$ is an integral commutative binoid that admits a generating set of non-can\-cel\-lative elements, then $M$ contains no non-trivial cancellative element.
\end {Lemma}
\begin {proof}
Let $A\subseteq\opcan M$ be a generating set of $M$. By the assumption on $A$, there is for every $f\in M\opkt$ an $a\in A$ with $f=\tilde{f}+a$ and $a+b=a+c\not=\infty$ for some $c,b\in M$ with $c\not=b$. Since $M$ is integral the latter equation implies that $\tilde{f}+a+b=\tilde{f}+a+c\not=\infty$ by adding $\tilde{f}$. Hence, $f+b=f+c\not=\infty$ with $c\not=b$.
\end {proof}

The preceding result may fail to be true for an integral binoid that is not commutative, like the binoid $\freenc(x,y,z)/(z+x=z+y=x+y=x+z)$ in which the element $x+y$ is cancellative.

\begin {Proposition} \label{PropUniqueMinSyst}
Every finitely generated commutative binoid $M$ that is positive and cancellative admits a unique minimal generating set given by $M\Uplus\setminus2 M\Uplus$.
\end {Proposition}
\begin {proof}
By the positivity, every generating set is contained in $M\Uplus=M\setminus\{0\}$. First we show that the set $M\Uplus\setminus 2M\Uplus$ contains any minimal generating set, which implies that it is a generating set as well. For this take an arbitrary minimal generating set $\{x_{1}\kpkt x_{r}\}$ and assume that $x_{1}\not\in M\Uplus\setminus 2M\Uplus$. Then $x_{1}=y+z$ for some $y,z\in M\Uplus$. Hence, $x_{1}=n_{1}x_{1}\pluspkt n_{r}x_{r}$ with at least two $n_{i}\ge1$ or one $n_{i}\ge2$, $i\in I$. If $n_{1}\not=0$,  the cancellativity of $M$ yields a non-trivial equation $0=(n_{1}-1)x_{1}+n_{2}x_{2}\pluspkt n_{r}x_{r}$ which is a contradiction to $M\okreuz=\{0\}$. Consequently, $n_{1}=0$ and $x_{1}$ can be dropped, contrary to the minimality of the generating set. Thus, $\{x_{1}\kpkt x_{r}\}\subseteq M\Uplus\setminus 2M\Uplus$. 

The minimality of $M\Uplus\setminus 2M\Uplus$ follows immediately because if $x\in M\Uplus\setminus 2M\Uplus$ could be omitted, there were an expression $x=n_{1}y_{1}\pluspkt n_{s}y_{s}$ with $y_{i}\in M\Uplus\setminus 2M\Uplus$ and at least one $n_{i}\not=0$, $i\in\{1\kpkt s\}$. This means $x=y_{i}$ for some $i\in\{1\kpkt s\}$ since $x\not\in 2M\Uplus$, hence $x$ cannot be omitted. The same argument shows that $M\Uplus\setminus2 M\Uplus$ must be contained in every generating set of $M$. 
\end {proof}

All conditions on a binoid assumed in Proposition \ref{PropUniqueMinSyst} are required. Counterexamples of binoids that fulfill all except one of the three conditions are given by $(\Q_{\ge0}^{\infty},+)$, $(\Z/n\Z)^{\infty}$ with $n\ge 2$, and $\free(x,y)/(x+y=x)$.

\begin {Corollary} \label {CorFree}
A finitely generated binoid is free commutative if and only if it is commutative, integral and semifree.
\end {Corollary}
\begin {proof}
All properties of a free commutative binoid follow from the fact that  $(\N^{n})^{\infty}$ has these properties. For the converse note that a binoid $M$ with the given properties admits a unique minimal generating set by Proposition \ref{PropUniqueMinSyst}, say $M\Uplus\setminus 2M\Uplus=\{x_{1}\kpkt x_{n}\}$. Since $M$ is semifree and integral, the canonical binoid epimorphism $(\N^{n})^{\infty}\rto M$ with $e_{i}\mto x_{i}$, $i\in\{1\kpkt n\}$, is injective.
\end {proof}

The proof of Proposition \ref{PropUniqueMinSyst} shows that the set $M\Uplus\setminus2 M\Uplus$ is finite. This can be generalized.

\begin{Lemma}
Let $M$ be a positive finitely generated commutative binoid. Then $M\Uplus\setminus nM\Uplus$ is a finite set for all $n\ge 1$.
\end{Lemma}
\begin {proof}
If $\{x_{1}\kpkt x_{r}\}$ is a minimal generating set of $M$, then $x_{1}\kpkt x_{r}\in M\Uplus$ by the positivity of $M$, and hence $nM\Uplus=\{n_{1}x_{1}\pluspkt n_{r}x_{r}\mid \sum_{k=1}^{r}n_{i}\ge n\}$. This gives 
$$\#(M\Uplus\setminus nM\Uplus)\,\,\le\,\,\#\{(n_{1}\kpkt n_{r})\mid\ n_{1}\pluspkt n_{r}< n\}\,\,=\,\,\sum_{k=1}^{n-1}\sum_{n_{1}\pluspkt n_{r}=k}\binom{k}{n_{1}\kpkt n_{r}}\pkt$$
In particular, $M\Uplus\setminus nM\Uplus$ is finite for all $n\ge 1$.
\end {proof}

\begin {Definition}
Let $M$ be a positive finitely generated commutative binoid. The map 
$$\opH(-,M):\N\Rto\N$$
with $\opH(n,M):=\#(M\setminus nM\Uplus)$ for $n\ge1$ and $\opH(0,M):=0$, \nomenclature[HilbertSamuel]{$\opH(-,M)$}{Hilbert-Samuel function of $M$}is called the \gesperrt{Hilbert-Samuel function} \index{Hilbert-Samuel function}of $M$ and $\opH(M):=\sum_{n\in\N}\opH(n,M)T^{n}$ \nomenclature[HilbertSamuel]{$\opH(M)$}{Hilbert-Samuel series of $M$}the \gesperrt{Hilbert-Samuel series} \index{Hilbert-Samuel series}of $M$.
\end {Definition}

$\opH(1,M)=-1$ always holds.

\begin {Remark}
As we will see later in Example \ref{ExCompositionsIdeal}(2), the definition of the Hilbert-Samuel function is similar to that of a semilocal ring, cf.\ \cite[Chapter II.B \S4]{SerreLA}.
\end {Remark}

\begin {Definition}
Let $\Acal\subseteq\Pset(V)$ be a collection of subsets of an arbitrary set $V$. The collection $\Acal$ is called  \gesperrt{subset-closed} (resp.\ superset-closed) \index{subset-closed}\index{superset-closed}if $B\subseteq A$ (resp.\ $A\subseteq B$) for some $B\in\Pset(V)$ and $A\in\Acal$, implies $B\in\Acal$; that is, with every set all its subsets (resp.\ supersets) lie in $\Acal$. The set of all subset-closed subsets of $\Pset(V)$ will be denoted by $\Scal(V)$\nomenclature[Subsetclosed]{$\Scal(V)$}{set of all subset-closed subsets of $\Pset(V)$}.
\end {Definition}

\begin {Lemma} \label {LemSubsetClosed}
Let $V\not=\emptyset$ be an arbitrary set.
\begin {ListeTheorem}
\item $\Scal(V)$ defines a topology on $\Pset(V)$. In particular, $\Scal(V)$ defines two boolean commutative binoids with respect to $\cup$ and $\cap$, namely\nomenclature[Subsetclosed]{$\Scal(V)_{\cup},\Scal(V)_{\cap}$}{binoids associated to $\Scal(V)$}
$$(\Scal(V),\cup,\{\emptyset\},\Pset(V))\,=:\,\Scal(V)_{\cup}\quad\text{and}\quad(\Scal(V),\cap,\Pset(V),\{\emptyset\})\,=:\,\Scal(V)_{\cap}\komma$$
which are integral if and only if $\#V=1$.
\item A collection $\Acal\subseteq\Pset(V)$ is subset-closed if and only if $\Pset(V)\setminus\Acal$ is superset-closed. In particular, the set $\Scal^{\opc}(V):=\{\Pset(V)\setminus\Acal\mid\Acal\in\Scal(V)\}$ of all superset-closed subsets of $\Pset(V)$ defines two boolean commutative binoids with respect to $\cup$ and $\cap$, namely 
$$(\Scal^{\opc}(V),\cup,V,\Pset(V))\,=:\,\Scal^{\opc}(V)_{\cup}\quad\text{and}\quad(\Scal^{\opc}(V),\cap,\Pset(V),V)\,=:\,\Scal^{\opc}(V)_{\cap}\pkt$$
\end {ListeTheorem}
\end {Lemma}
\begin {proof}
The first statement of (1) is easily verified. For the supplement note that if there are $i,j\in V$ with $i\not=j$, then $\{\emptyset,\{i\}\}\cap\{\emptyset,\{j\}\}=\{\emptyset\}$ and $\Pset(V\setminus\{i\})\cup\Pset(V\setminus\{j\})=\Pset(V)$. (2) Note that if $\Acal\in\Scal(V)$ is subset-closed, then $A\in\Pset(V)\setminus\Acal$ means $A\not\in\Acal$. Hence, $B\not\in\Acal$ for all $B\in\Pset(V)$ with $A\subseteq B$ if $\Acal$ is subset-closed and $A\not\in\Acal$. This is equivalent to $B\in\Pset(V)\setminus\Acal$ for all $B\in\Pset(V)$ with $A\subseteq B$. Hence, $\Pset(V)\setminus\Acal$ is superset-closed. The supplement of (2) follows from (1).
\end {proof}

The cardinality of the set of all subset-closed sets significantly increases as $V$ becomes larger. A detailed computation shows, for instance, $\#\Scal(\{1\})=2$, $\#\Scal(\{1,2\})=4$, $\#\Scal(\{1,2,3\})=19$, and $\#\Scal(\{1,2,3,4\})=170$

\medskip

Now we introduce the common concept of filters to the theory binoids. 

\begin {Definition}
A subset $F$ of a binoid $M$ is a \gesperrt{filter} \index{filter}if the following conditions are satisfied.
\begin {ListeTheorem}
\item $0\in F$.
\item If $f,g\in F$, then $f+g\in F$.
\item If $f+g\in F$, then $f,g\in F$.
\end {ListeTheorem}
The \gesperrt {filtrum} \index{filtrum}of $M$ is the set of all filters in $M$ and will be denoted by $\Fcal(M)$\nomenclature[Filter]{$\Fcal(M)$}{set of all filters in $M$}. 
\end {Definition}

Loosely speaking, $F\subseteq M$ is a filter if it is a submonoid with the property  $f+g\in F\Rarrow f,g\in F$. The additional condition on a submonoid to be a filter may be called \emph{summand-closed}. Since $0$ is a summand of every element of $M$, the first condition can be replaced by $F\not=\emptyset$.

\begin {Remark} \label{RemPrimeUnionConsequence1}
If $R$ is a ring, then every proper filter of $(R,\cdot,1,0)$ is the complement of an arbitrary union of (ring) prime ideals of $R$, and conversely, every complement of a filter is the union of (ring) prime ideals. Of course, this union does not need to be unique. For instance, the set of invertible elements $R\okreuz$ is the complement of the union of all prime ideals as well as the complement of the union of all maximal ideals. After establishing the necessary ideal theory for binoids in Chapter \ref{ChapIdealTheory}, we will show that similarly, every filter $\not=M$ in a binoid $M$ is the complement of \emph{one} unique prime ideal, cf.\ Corollary \ref {CorHomFiltSpec}. Therefore, considering proper filters or prime ideals in binoids amounts to the same theory.
\end {Remark}

If $(F_{i})_{i\in I}$ is a family of filters in $M$, then $\bigcap_{i\in I}F_{i}$ is again a filter in $M$, and since every unit is a summand of $0$ the filter $M^{\times}$ is contained in any other filter of $M$. Hence, the commutative binoid\nomenclature[Filter]{$\Fcal(M)_{\cap}$}{$=(\Fcal(M),\cap,M,M\okreuz)$}
$$\Fcal(M)_{\cap}\,:=\,(\Fcal(M),\cap,M, M^{\times})$$
is a (meet-) semilattice with largest element $M$ and smallest element $M^{\times}$ with respect to set inclusion. We will encounter this binoid frequently.

\begin {Definition}
Let $A$ be a subset of $M$. The filter $\opFilt(A)$\nomenclature[Filter]{$\opFilt(A)$}{filter generated by $A$} \gesperrt{generated} \index{filter!-- generated by a subset}by $A$ is the set of all summands of sums that lie in $A$; in other words,
$$\opFilt(A)\,\,=\!\!\bigcap_{A\subseteq F\in\Fcal(M)}F\pkt$$ 
By definition, it is the smallest filter of $M$ containing $A$. The filter generated by a singleton $A=\{f\}$, $f\in M$, will be denoted by $\opFilt(f)$.
\end {Definition}

\begin {Remark} \label{RemFilterOfF}
Since $M$ is the only filter containing $\infty$, we have $\opFilt(f)=M$ if and only if $f$ is nilpotent. Moreover, $g\in\opFilt(f)$ means $g=n_{1}f_{1}\pluspkt n_{r}f_{r}$ for some summands $f_{1}\kpkt f_{r}$ of $f$. Therefore, $g\in\opFilt(f)$ is equivalent to $g+x=nf$ for some suitable $x\in M$ and $n\in\N$. 
\end {Remark}

The operation $F\star G:=\opFilt(F+G)$, $F,G\in\Fcal(M)$, on $\Fcal(M)$ turns the set $\Fcal(M)$ into a commutative binoid, namely\nomenclature[Filter]{$\Fcal(M)_{\star}$}{$=(\Fcal(M),\star,M\okreuz,M)$, where $F\star G=\opFilt(F+G)$}
$$\Fcal(M)_{\star}\,:=\,(\Fcal(M),\star,M\okreuz, M)\pkt$$
This binoid will appear later in Proposition \ref{PropLocalizationSmash}.

\begin {Example} \label{ExpFilterPowerset}
Let $V$ be a finite set. A subset $F\subseteq\Pset(V)$ is a filter in $\Pset(V)_{\cap}$ if and only if $F$ is superset-closed (which implies that $V\in F$) and contains only one minimal element $J\in\Pset(V)$ with respect to $\subseteq$. Thus, for every filter $F$, one has $F=\opFilt(J)=\{A\in\Pset(V)\mid J\subseteq A\}$ for some $J\subseteq V$. Similarly, a subset $F\subseteq\Pset(V)$ is a filter in $\Pset(V)_{\cup}$ if and only if $F$ is subset-closed (which implies that $\emptyset\in F$) and contains only one maximal element $J\in\Pset(V)$ with respect to $\subseteq$. Thus, for every filter $F$, one has $F=\opFilt(J)=\{A\in\Pset(V)\mid A\subseteq J\}=\Pset(J)$ for some $J\subseteq V$. In particular, the one-to-one correspondences

$$\begin {array} {rcccl}
\Fcal(\Pset(V)_{\cap})\!\!\!&\longleftrightarrow&\!\!\!\Pset(V)\!\!\!&\longleftrightarrow&\!\!\!\Fcal(\Pset(V)_{\cup})\\
F\!\!\!&\lto&\!\!\!\min_{\subseteq}F\!\!\!&&\\
\{A\in\Pset(V)\mid J\subseteq A\}\!\!\!&\longmapsfrom&\!\!\! J\!\!\!&\lto&\!\!\!\{A\in\Pset(V)\mid A\subseteq J\}=\Pset(J)\\
&&\!\!\!\max_{\subseteq}F\!\!\!&\longmapsfrom&\!\!\! F
\end {array}$$
yield the binoid isomorphisms $\Pset(V)_{\cup}\cong\Fcal(\Pset(V)_{\cap})_{\cap}$  and $\Pset(V)_{\cap}\cong\Fcal(\Pset(V)_{\cup})_{\cap}$.
\end {Example}

Finally we take a look on the behavior of certain properties under a binoid homomorphism.

\begin {Lemma} \label {LemHomProperties}
Let $\varphi:N\rto M$ be a binoid homomorphism.
\begin {ListeTheorem}
\item $\varphi(\nil(N))\subseteq\nil(M)$.
\item $\varphi(\bool(N))\subseteq\bool(M)$.
\item $\varphi(N\okreuz)\subseteq M\okreuz$. Moreover, if $M$ is commutative, then $\varphi^{-1}(M\okreuz)\in\Fcal(N)$.
\end {ListeTheorem}
\end {Lemma}
\begin {proof}
We only show the second part of the last assertion because all other statements are easily verified. Set $F:=\varphi^{-1}(M\okreuz)$. Clearly, $0_{N}\in F$ because $\varphi$ is a binoid homomorphism. If $a,b\in F$, then $\varphi(a),\varphi(b)\in M\okreuz$. Since $M\okreuz$ is a group, we obtain $\varphi(a+b)=\varphi(a)+\varphi(b)\in M\okreuz$, which means $a+b\in F$. Similarly, $a+b\in F$ means $\varphi(a+b)\in M\okreuz$. So there is an element $u\in M\okreuz$ with $\varphi(a+b)=$ $\varphi(a)+\varphi(b)=u$. This yields $\varphi(a)+(\varphi(b)+(\minus u))=0_{M}=\varphi(b)+(\varphi(a)+(\minus u))$, which shows that $\varphi(a),\varphi(b)\in M\okreuz$.
\end {proof}

\begin {Remark} \label {RemImageNonIntegral}
Note that the image of a non-integral element need not be non-integral under a binoid homomorphism. If, for instance, $a\in N\setminus\nil(N)$ is a non-integral element such that the set $A:=\{b\in N\mid a+b=\infty\}\subseteq\opint^{\opc}(N)$ is contained in the kernel of $\varphi:N\rto M$, then $\varphi(a)$ may even be a unit, cf.\ Example \ref {ExpBHomBoolean}(2). A trivial example is given by the binoid homomorphism $$\free(x,y)/(x+y=\infty)\Rto(\Z/n\Z)^{\infty}\komma$$
$n\ge2$, with $x\mto 1$ and $y\mto\infty$.
\end {Remark}

\begin {Proposition} \label{PropHomFilter}
Let $M$ be a binoid. A map $\varphi:M\rto\trivial$ is a binoid homomorphism if and only if $M\setminus\ker\varphi\in\Fcal(M)$.
\end {Proposition}
\begin {proof}
Let $F:=M\setminus\ker\varphi$. We have $\varphi(a+b)=0$ if and only if $\varphi(a)=0=\varphi(b)$ since $\varphi$ is a binoid homomorphism and, in particular, $\varphi(0_{M})=0_{N}$. Hence, $0_{M}\in F$, and $a+b\in F$ is equivalent to $a,b\in F$.
\end {proof}

In Section \ref {SecPrime}, we will rephrase this result in terms of ideal theory, cf.\ Lemma \ref {LemCharacterizationPrime}.

\bigskip

\section {$N\mina$spectra}  \label{SecNspectra}
\markright{\ref{SecNspectra} $N\mina$spectra}

The notion of homomorphisms between binoids yields the dual and bidual of a binoid as well as its $N\mina$spectrum for another binoid $N$. The latter admits a geometric interpretation when $N=K$ is a field, namely the $K\mina$spectrum of the associated binoid algebra, which will be treated in Section \ref{SecKpoints} in more detail. All these objects are illustrated by several examples.

\medskip

By definition, every binoid homomorphism $\varphi$ between nonzero binoids fulfills $\trivial\subseteq\im\varphi$. Those homomorphisms $\varphi$ with $\im\varphi=\trivial$ are uniquely determined by $\ker\varphi$ and closely related to the prime ideals of the binoid, cf.\ Section \ref{SecPrime}.

\begin {Definition}
Let $M$ be a binoid and $N$ an arbitrary subset of $M$. The map $\chi_{N}:M\rto\trivial$ \nomenclature[AMap]{$\chi_{N}$}{indicator function with respect to the subset $N$}with $a\mto0$ if $a\in N$ and $\infty$ otherwise, is called the \gesperrt{indicator function} \index{indicator function}of $N$. The \gesperrt {anti-indicator function} \index{indicator function!anti- --}$\chi_{M\setminus N}$ of $N$ will be denoted by $\alpha_{N}$\nomenclature[AMap]{$\alpha_{N}$}{$=\chi_{M\setminus N}$, anti-indicator function with respect to the subset $N\subseteq M$}. If an (anti-) indicator function is a binoid homomorphism we call it an \gesperrt{(anti-) indicator homomorphism}\index{homomorphism!indicator --}\index{homomorphism!anti-indicator --}.
\end {Definition}

The indicator function is also known as the \emph{characteristic function} and the codomain is usually the trivial binoid with respect to the multiplication; that is, $\chi_{N}$, $N\subseteq M$, is usually defined to be the map $M\rto\{1,0\}$ with $a\mto 1$ if $a\in N$ and $0$ otherwise.

The set $\map(M,\trivial)$ can be naturally embedded into the set $\map(M,M^{\prime})$ for any two nonzero binoids $M$ and $M^{\prime}$; that is, every function $M\rto\trivial$ can be considered as a map $M\rto M^{\prime}$. Then we will use the same notation $\chi_{N}$ and $\alpha_{N}$ for the indicator and anti-indicator function of $N\subseteq M$ in $\map(M,M^{\prime})$, respectively.

\begin {Remark} \label {RemBHom}
Let $M$ and $N$ be nonzero binoids. The set $\hom(M,N)$ is a semigroup with respect to the operation
$$(\phi+\psi)(a)\,=\,\phi(a)+\psi(a)\komma$$
$\phi,\psi\in\hom(M,N)$, $a\in M$. The semigroup $\hom(M,N)$ is commutative if $M=\trivial$ or if $N$ is commutative. If $N$ is boolean, then so is $\hom(M,N)$. In particular, the set of all indicator homomorphisms on $M$ is a boolean subsemigroup of $\hom(M,N)$.

We have $\chi_{M\opkt}+\psi=\psi+\chi_{M\opkt}=\psi$ and $\chi_{\{0\}}+\psi=\psi+\chi_{\{0\}}=\chi_{\{0\}}$ for all $\psi:M\rto N$. In general, the functions $\chi_{M\opkt}, \chi_{\{0\}}:M\rto N$ are no homomorphisms. More precisely, $\chi_{M\opkt}$ is a binoid homomorphism if and only if $M\opkt$ is a monoid, and $\chi_{\{0\}}$ is a binoid homomorphism if and only if $M$ is positive. Thus, $(\hom(M,N),+,\chi_{M\opkt},$ $\chi_{\{0\}})$ is a binoid if and only $M\opkt$ is a positive monoid. However, one can always adjoin an absorbing and an identity element to get a binoid. For instance, the constant maps $\alpha_{M}:a\mto\infty$ and $\chi_{M}:a\mto 0$, $a\in M$, can serve as an absorbing and as an identity element, respectively.
\end {Remark}

\begin {Definition}
Let $M$ and $N$ be binoids. The semigroup $(\hom(M,N),+)=:N\minspec M$ \nomenclature[N]{$N\minspec M$}{set of all $N\mina$points of $M$ ($N\mina$spectrum of $M$)}is called the \gesperrt{$N\mina$spectrum} \index{spectrum!N@$N\mina$-- of a binoid}\index{N@$N\mina$spectrum!-- of a binoid}\index{binoid!N@$N\mina$spectrum of a --}of $M$ and we refer to an element of $N\minspec M$ as an \gesperrt{$N\mina$point} \index{binoid!N@$N\mina$point of a --}\index{N@$N\mina$point!-- of a binoid}of $M$\index{point!$N\mina$--}. 

With this notation, the semilattice of all indicator homomorphisms on $M$, which is contained in $N\minspec M$, is denoted by $\trivial\minspec M$. By Proposition \ref{PropHomFilter},
$$\trivial\minspec M=\{\chi_{F}\mid F\in\Fcal(M)\}\pkt$$

The $N\mina$point $\chi_{M\okreuz}:M\rto N$ is called the \gesperrt{special point} \index{point!special --}of $M$. In general, we refer to an $N\mina$point $\chi_{F}:M\rto N$ given by a filter $F\in\Fcal(M)$ as a \gesperrt{characteristic} (or \gesperrt{boolean}) \gesperrt{point} of $M$. \index{point!characteristic --}\index{point!boolean --}\index{boolean!-- point}We use this terminology, in particular, if $N$ is the trivial binoid.
\end {Definition}

\begin {Remark}
The terminology introduced here becomes clear later when we will see that the filters in $M$ are precisely the complements of the prime ideals in $M$, cf.\ Corollary \ref{CorHomFiltSpec}. Hence, the special point $\chi_{M\okreuz}=\alpha_{M\Uplus}$ corresponds to the \emph{maximal} ideal $M\Uplus$ (maximal with respect to $\subseteq$ since $M\okreuz$ is the smallest filter and contained in any other filter), and the characteristic points $\chi_{F}=\alpha_{M\setminus F}$, $F\in\Fcal(M)$, correspond to the the prime ideals $M\setminus F$.

For a field $K$, the $K\mina$spectrum of a binoid $M$ can be identified with the $K\mina$spectrum $K\minspec K[M]:=\Hom_{K\minus\opalg}(K[M],K)$ of $K[M]$, cf.\ Propostition \ref{PropUnivPropBinoidA}, which is a well-studied topological space if $M$ is commutative. Section \ref{SecKpoints} and in most parts Chapter \ref{ChapSepGradings} are devoted to a detailed description of the set of $K\mina$points of $M$ by applying the knowledge of $K\minspec K[M]$, where $K$ is a field and $M$ a commutative binoid. Once in a while, we anticipate this description of the $K\mina$spectrum of a binoid and draw a picture of the $\R\mina$spectrum.
\end {Remark}

\begin {Example} \label {ExpBHomBoolean}
Let $N$ be a commutative binoid.
\begin{ListeTheorem}
\item By Lemma \ref{LemHomProperties}(2), $N\minspec\free(x)/(2x=x)\cong\bool N$.
\item Every binoid homomorphism $\trivial^{n}\rto N$ is uniquely determined by the images of the generators $e_{1,\infty}\kpkt e_{n,\infty}$. Since the elements $e_{i,\infty}$ are idempotent their images lie in $\bool N$, and because there are no further relations than $e_{1,\infty}\pluspkt e_{n,\infty}=\infty_{\Pi}$, we get
$$N\minspec\trivial^{n}=\{(a_{1}\kpkt a_{n})\in(\bool N)^{n}\mid a_{1}\pluspkt a_{n}=\infty\}\pkt$$
In particular, $N\minspec\trivial\cong\zero$ and  $\trivial\minspec\trivial^{n}\cong\trivial^{n}\setminus\{0_{\Pi}\}$, which one may also obtain from Proposition \ref{PropNspecProd}.
\end{ListeTheorem}
\end {Example}

We have to postpone the determination of $N\minspec\prod_{i\in I}M_{i}$ for an arbitrary family $(M_{i})_{i\in I}$ of binoids because a little more theory is needed for this, cf.\ Proposition \ref {PropNspecProd}.

\begin {Proposition} \label {PropIndHomNspec}
Let $M$, $N$, and $L$ be binoids. Every binoid homomorphism $\varphi:M\rto N$ induces a semigroup homomorphism $\varphi^{\ast}:L\minspec N\Rto L\minspec M$ with $\psi\mto\psi\varphi$. Moreover, $\varphi^{\ast}$ is injective if $\varphi$ is surjective.
\end {Proposition}
\begin {proof}
The first assertion is trivial. For the supplement, we need to show that $\psi=\psi^{\prime}$ if $\psi\varphi=\psi^{\prime}\varphi$, but this follows from the surjectivity of $\varphi$.
\end {proof}

\begin {Lemma} \label {LemNspec}
If $N$ is a nonzero binoid, then $N\minspec (\N^{(I)})^{\infty}\cong N^{(I)}$ as binoids. In particular, for every finitely generated free commutative binoid $M$ there is a unique $n\ge 1$ such that $N\minspec M\cong N^{n}$.
\end {Lemma}
\begin {proof}
Since $(\N^{(I)})^{\infty}$ is integral, positive, and $\bool((\N^{(I)})^{\infty})=\trivial$, the images of the elements of the basis $(e_{i})_{i\in I}$, which determine every homomorphism on $(\N^{(I)})^{\infty}$, can be choosen at will in $N$. Moreover, $(N\minspec(\N^{(I)})^{\infty},+,\chi_{M\opkt},$ $\chi_{\{0\}})$ is a binoid by Remark \ref {RemBHom}. The supplement is clear by the definition of a free commutative binoid.
\end {proof}

\begin {Lemma} \label {LemPosHom}
Let $M$ and $N$ be binoids. If $N$ is positive, the homomorphism $\chi_{M\okreuz}:M\rto N$ is an absorbing element in $N\minspec M$. In particular, $N\minspec M$ is a boolean semibinoid if $N$ is a boolean binoid.
\end {Lemma}
\begin {proof}
By Proposition \ref {PropHomFilter}, the characteristic function $\chi_{M\okreuz}$ is a binoid homomorphism. Since $N$ is positive, $\varphi(f)=0$ for all $f\in M\okreuz$ and all $\varphi\in N\minspec M$. It follows 
$$(\chi_{M\okreuz}+\varphi)(f)\,=\,\chi_{M\okreuz}(f)+\varphi(f)\,=\,\begin {cases}
0+0&\text{, if }f\in M^{\times}\komma\\
\infty+\varphi(f)\,=\,\infty&\text{, otherwise.}\end {cases}\,\,\,=\,\,\,\chi_{M\okreuz}(f)\komma$$
and $(\varphi+\chi_{M\okreuz})(f)=\chi_{M\okreuz}$ in the same manner. The supplement is clear by Lemma \ref{LemBool=>PosRed} and  Remark \ref{RemBHom}.
\end {proof}

\begin {Corollary}
$N\minspec(-)$ is a contravariant functor from the category of binoids $\Bsf$ into the category of (boolean) semibinoids for every positive (boolean) binoid $N$.
\end {Corollary}
\begin {proof}
This follows from Lemma \ref {LemPosHom}. If $\varphi:M\rto M^{\prime}$ is a morphism in the category of binoids, the morphism $\hom(\varphi,N):\hom(M^{\prime},N)\rto\hom(M,N)$ is given by $\psi\mto \psi\varphi$.
\end {proof}

\begin {Corollary} \label{CorIsomFcalMinspec}
The canonical maps $\chi\mto M\setminus\ker\chi$ and $\chi_{F}\mapsfrom F$ are semibinoid isomorphisms between $\trivial\minspec M$ and $(\Fcal(M)\setminus\{M\},\cap,M\okreuz)$ which are inverse to each other.
\end {Corollary}
\begin {proof}
The bijectivity is just a restatement of Proposition \ref{PropHomFilter}. The easy computations
\begin {align*}
\ker(\chi+\chi^{\prime})&=\{f\in M\mid\chi(f)+\chi^{\prime}(f)=\infty\}\\
&=\{f\in M\mid\chi(f)=\infty\}\cup\{f\in M\mid\chi^{\prime}(f)=\infty\}=\ker\chi\cup\ker\chi^{\prime}
\end {align*}
and 
$$\chi_{F\cap F^{\prime}}(f)=\begin {cases}
0&\text{, if }f\in F\text{ and }f\in F^{\prime}\\
\infty&\text{, otherwise,}
\end {cases}\,\,\,=\,\,\,\chi_{F}(f)+\chi_{F^{\prime}}(f)\,=\,(\chi_{F}+\chi_{F^{\prime}})(f)$$
prove the homomorphism property for both assignments.
\end {proof}

\begin {Definition}
The boolean binoid\nomenclature[M]{$M\dual$}{dual of $M$}
$$M\dual:=\,(\hom(M,\trivial)\cup\{\chi_{M}\},+,\chi_{M},\chi_{M\okreuz})$$
is the \gesperrt{dual} \index{binoid!dual of a --}\index{dual}and $M\bidual$ \nomenclature[M]{$M\bidual$}{bidual of $M$}is the \gesperrt{bidual} \index{bidual}\index{binoid!bidual of a --}of $M$.
\end {Definition}

The bidual realizes the booleanization of finitely generated commutative binoids, cf.\ Corollary \ref{CorBooleanizationFG}.

\begin {Example} \label {ExpAdjointDual}
Let $M$ be a binoid.
If $\mathbf{0}$ is an adjoint identity element, then
$(M\onull)\dual=\{\psi_{\mathbf{0}}\mid\psi\in M\dual\}$,
where $\psi_{\mathbf{0}}$ is the extension of $\psi$ to $M\onull$. In particular, $(M\onull)\dual\cong M\dual$.

If $M$ is a finitely generated free commutative binoid, then $M\dual=(\trivial^{n})\onull$ for some $n\ge 1$ by Lemma \ref {LemNspec}, and hence $M\bidual=\trivial^{n}$.
\end {Example}

\begin {Remark} \label {RemBidual}
There is a canonical binoid homomorphism $\delta:M\rto M\bidual$ with $f\mto \delta_{f}$, where $\delta_{f}:M\dual\rto\trivial$ is the evaluation  $\delta_{f}(\chi):=\chi(f)$ at $f$ for every $\chi\in M\dual$. Note that $\ker\delta=\nil(M)$, hence $\delta_{f}$ coincides with the absorbing element $\chi_{(M\dual)\okreuz}:\varphi\mto\infty$ of $M\bidual$ if and only if $f\in\nil(M)$. Moreover, $\delta_{f}$ coincides with the identity element $\chi_{M\dual}:\varphi\mto 0$ of $M\bidual$ if and only if $f\in M\okreuz$.
\end {Remark}

The Example \ref {ExpAdjointDual} above shows that $M\bidual\not=M$ happens. Later we will show that the homomorphism $\delta$ of the preceding remark is an isomorphism (i.e.\ $M\bidual\cong M$) if and only if $M$ is finitely generated and boolean, cf.\ Corollary \ref{CorBooleanizationFG}. Hence, $(M\dual)\bidual=M\dual$ for all finitely generated binoids $M$, cf.\ Corollary \ref{CorBidualDual}. In particular, $\Pset(V)\bidual_{\cap}\not=\Pset(V)_{\cap}$ if and only if $V$ is infinite. So far, we have the following examples.

\begin {Example}
By Example \ref {ExpBHomBoolean}(2), $(\trivial^{n})\dual=\trivial^{n}$, and hence $(\trivial^{n})\bidual=\trivial^{n}$. In particular, $\Pset\bidual_{n,\cap}=\Pset_{n,\cap}\dual=\Pset_{n,\cap}$ and $\Pset\bidual_{n,\cup}=\Pset_{n,\cup}\dual=\Pset_{n,\cup}$ by Example \ref{ExPowersetIsom}(1).
\end {Example}

$M\dual=M$ need not be true for every finitely generated boolean binoid as the following example shows.

\begin {Example}
For the boolean binoid $B=\free(b_{1},b_{2},b_{3})/(b_{1}+b_{2}=\infty, 2b_{i}=b_{i}, i\in\{1,2,3\})$, one has 
$B\dual=\langle \chi_{1},\chi_{2},\chi_{1,3},\chi_{2,3}\rangle$,
where $\chi_{I}$, $I\subseteq\{1,2,3\}$, is the indicator homomorphism $B\rto\trivial$ defined by $b_{i}\mto 0$ for $i\in I$. Hence, $B\dual\not=B$. If one denotes the generators of $B\dual$ by $d_{1}, d_{2},d_{3},$ and $d_{4}$, then (with the same notation as above) $B\bidual\cong\langle\chi_{1,2},\chi_{1,3},\chi_{2,4}\rangle\cong B$.
\end {Example}

\bigskip

\section {Congruences} \label{SecCong}
\markright{\ref{SecCong} Congruences}

Homomorphisms on a monoid $M$ and its homomorphic images are closely related to the notion of congruences on $M$, see \cite[Chapter I.4]{Gilmer}, \cite[Chapter I.3 and I.4]{GrilletS}, or \cite[Chapter I.2]{GrilletCS}. This relation transfers to binoids since a congruence on a binoid is nothing else than a congruence on the underlying monoid, cf.\ Proposition \ref{PropHomCong}. At the end of this section we prove R\'edei's Theorem for finitely generated commutative binoids. In the next section, we will give attention to specific congruences with which we will be concerned later.

\begin {Definition}
A \gesperrt{congruence} \index{congruence}on a binoid $M$ is an equivalence relation $\sim$ on $M$ that is compatible with the addition; that is, if $a,b\in M$ with $a\sim b$, then $a+c\sim b+c$ and $c+a\sim c+b$ for all $c\in M$. We denote the congruence class of $a\in M$ by $[a]$, sometimes also by $\bar{a}$, and the set of all congruence classes by $M/\sim$. Let $\sim_{1}$ and $\sim_{2}$ be two congruences on a binoid. The \gesperrt{intersection} \index{congruence!intersection of --s}$\sim_{1}\cap\sim_{2}$ is the congruence $\sim$ on $M$ defined by $a\sim b$ for $a,b\in M$ if $a\sim_{1}b$ and $a\sim_{2}b$. We write $\sim_{1}\,\le\,\sim_{2}$ if $a\sim_{1}b$ implies $a\sim_{2}b$ for all $a,b\in M$ .
\end {Definition}

On the set of congruences on a binoid, $\le$ defines a partial order under which there always exist a largest and a smallest congruence, which means there are congruences $\sim_{\opu}$ and $\sim_{\id}$ such that $\sim_{\id}\,\le\,\sim\,\le\,\sim_{\opu}$ for all congruences $\sim$ on the binoid. These are given by the \gesperrt{universal} \index{congruence!universal --}congruence $\sim_{\opu}$\nomenclature[ACongruenceU]{$\sim_{\opu}$}{universal congruence}, which relates any two elements with each other, and the \gesperrt{identity} \index{congruence!identity --}congruence $\sim_{\id}$, under which two elements are related only when they coincide.\nomenclature[ACongruenceId]{$\sim_{\id}$}{identity congruence}

\medskip

Before studying congruences on free commutative and finitely generated binoids, the equivalent notion of homomorphisms and congruences on a binoid is described in the following proposition.

\begin {Proposition} \label{PropHomCong}
Let $M$ be a binoid. Given a congruence $\sim$ on $M$, the quotient $M/\!\!\sim$ is a binoid with respect to the operation $[a]+[b]:=[a+b]$ such that the canonical map $M\rto M/\!\!\sim$, $a\mto[a]$, is an epimorphism of binoids. Conversely, if $\varphi:M\rto N$ is a binoid epimorphism, then the relation $\sim_{\varphi}$ \nomenclature[ACongruenceZPhi]{$\sim_{\varphi}$}{congruence on a binoid with respect to $\varphi$} on $M$ defined by $a\sim_{\varphi}\!b$ if $\varphi(a)=\varphi(b)$, is a congruence on $M$ such that $M/\!\sim_{\varphi}\,\rto N$, $[a]\mto\varphi(a)$, is a binoid isomorphism.
\end {Proposition}
\begin {proof}
It is easily checked that $M/\!\!\sim$ is a binoid with respect to the given operation and that $\sim_{\varphi}$ defines a congruence on $M$. The converse follows from the subsequent lemma.
\end {proof}

\begin {Lemma} \label {LemIndCong}
If $\varphi:M\rto N$ is a binoid homomorphism and $\sim$ a congruence on $M$ with $\sim\,\le\,\sim_{\varphi}$, then there is a unique binoid homomorphism $\tilde{\varphi}$ such that the diagram
$$\xymatrix{
M\ar[r]^{\varphi}\ar[d]_{\pi}&N\\
M/\sim\ar[ur]_{\tilde{\varphi}}&}$$
commutes. In particular, given two congruences $\sim_{1}$ and $\sim_{2}$ on $M$, the relation $\sim_{1}\,\le\,\sim_{2}$ is equivalent to the fact that the map $M/\!\sim_{1}\,\rto M/\!\sim_{2}$, $[a]_{1}\mto[a]_{2}$, is a well-defined binoid epimomorphism.
\end {Lemma}
\begin {proof}
Define $\tilde{\varphi}([a]):=\varphi(a)$. To show that this is well-defined assume that $[a]=[b]$ for $a,b\in M$. By assumption, $a\sim_{\varphi} b$, and hence $\tilde{\varphi}([a])=\varphi(a)=\varphi(b)=\tilde{\varphi}([b])$. By definition, $\tilde{\varphi}$ is a binoid homomorphism with $\tilde{\varphi}\pi=\varphi$, and hence unique. The situation of the supplement is expressed in the following diagram 
$$\xymatrix{
M\ar[r]^{\pi_{2}}\ar[d]_{\pi_{1}}&M/\sim_{2}\\
M/\sim_{1}\ar[ur]_{\tilde{\varphi}}&}$$
with $\pi_{2}=\tilde{\varphi}\pi_{1}$ surjective. Hence, $\tilde{\varphi}$ is surjective.
\end {proof}

\begin {Remark} \label{RemKer}
In standard literature, the congruence $\sim_{\varphi}$ defined by a monoid homomorphism $\varphi:M\rto N$ as in Proposition \ref{PropHomCong} is usually denoted by $\ker\varphi$, see for example \cite{GrilletS}, \cite{GrilletCS}, or \cite{GarciaRosalesNS}. Most probably this is due to the fact that the famous isomorphism theorem 
$$M/\!\sim_{\varphi}\,\,\cong\,\,\im\varphi$$
can be stated with the common notation, cf.\ for instance \cite[Proposition 8.2]{GarciaRosalesNS}, and maybe because there is no absorbing element taken into account that justifies our definition of $\ker$. Note that (continuing with our notation) the binoid homomorphism induced by $\ker\varphi$ need \emph{not} be an isomorphism, see Remark \ref{RemIndIsom} below. On the other hand, if $\varphi$ is a binoid homomorphism such that $\varphi_{|M\setminus\ker\varphi}$ is injective, which means that $\ker\varphi$ is defined by the congruence $\sim_{\varphi}$, then the induced homomorphism is an isomorphism, cf.\ Remark \ref{RemIndIsom}. For instance, the congruences given in Lemma \ref{LemIdealCong} below reflect this situation.
\end {Remark}

It is very common to identify a congruence $\sim$ with the set $R(\sim):=\{(a,b)\mid a\sim b\}\subseteq M\times M$.\nomenclature[R]{$R(\sim)$}{$:=\{(a,b)\mid a\sim b\}$} For instance, $R(\sim_{\opu})=M\times M$ and $R(\sim_{\id})$ is the diagonal $D:=\{(a,a)\mid a\in M\}\subseteq M\times M$. On the other hand, every subset $R\subseteq M\times M$ generates a congruence on $M$. For this set 
$$R^{\prime}:=\{(a+c,b+c)\mid(a,b)\in R\cup R^{-1}\cup D,c\in M\}\komma$$
where $R^{-1}=\{(b,a)\mid (a,b)\in R\}$. Now the relation $\sim_{R}$ on $M$ given by 
$$a\sim_{R}b\quad:\Leftrightarrow\quad\exists\, a=a_{1},a_{2}\kpkt a_{n}=b\quad\text{such that}\quad(a_{i},a_{i+1})\in R^{\prime}\komma\quad i\in\{1\kpkt n-1\}\komma$$
is a congruence on $M$. \nomenclature[ACongruenceRel]{$\sim_{R}$}{congruence on a binoid $M$ generated by relations $R\subseteq M\times M$}

\begin {Definition}
Let $M$ be a binoid and $R\subseteq M\times M$ a subset. The congruence $\sim_{R}$ defined as above is the congruence \gesperrt{generated} by the \gesperrt{relations} $R$.\index{congruence!-- generated by relations}\index{relation} A congruence $\sim$ on a binoid $M$ is called \gesperrt{finitely generated} \index{congruence!finitely generated --}if there exists a finite subset $R\subseteq M\times M$ such that $\sim$ is the congruence generated by $R$.
\end {Definition}

\begin {Remark}\label{RemNoetherian=FG} (Noetherian binoids)
The set of all cogruences on a binoid $M$ is a partially ordered set with respect to $\le$, where $\sim_{1}\,\le\,\sim_{2}$ is equivalent to $R(\sim_{1})\subseteq R(\sim_{2})$.

In \cite[Chapter 1.4]{Gilmer}, Gilmer defined a commutative monoid $M$ to be \gesperrt {noetherian} \index{binoid!noetherian --}if this order on the set of congruences on $M$ is noetherian; that is to say, $M$ satisfies the ascending chain condition \index{ascending chain condition (a.c.c.)!-- on congruences}\index{congruence!a.c.c.\ on --s}(a.c.c.) on congruences, or in other words, for every chain $\,\sim_{1}\,\,\le\,\,\sim_{2}\,\,\le\,\,\sim_{3}\,\,\le\,\cdots$ of congruences there is a $k\in\N$ such that $\,\sim_{n}\,\,=\,\,\sim_{k}\,$ for all $n\ge k$. Equivalently, every congruence on $M$ is finitely generated. Moreover, Gilmer showed that a commutative monoid is noetherian if and only if it is finitely generated, cf.\ \cite[Theorem 5.10 and 7.8]{Gilmer}. Since a binoid is finitely generated if and only if it is finitely generated as a monoid, and a congruence on a binoid is a congruence on the underlying monoid, we have the same statement for commutative binoids. Thus, considering noetherian or finitely generated commutative binoids amounts to the same thing. 

There is another (weaker) definition of noetherian in use by several authors in terms of ideals, cf.\ Remark \ref{RemNakayamaACC}.
\end {Remark}

\begin {Example} \label{ExGen}
Let $M$ be a commutative binoid generated by $\{a_{i}\mid i\in I\}$. If $\varepsilon:\free(I)\rto M$ is the canonical epimorphism $i\mto a_{i}$, then $M\cong\free(I)/\!\sim_{\varepsilon}$ by Proposition \ref{PropHomCong}. Thus, $M$ is given by the generating set $\{a_{i}\mid i\in I\}$ and the family of relations $\varepsilon(y_{j})=\varepsilon(z_{j})$ (i.e.\ $\Rcal_{j}:y_{j}\sim_{\varepsilon}z_{j}$, $j\in J$), which generate $\sim_{\varepsilon}$. This justifies the notation given in Definition \ref{DefFG}, 
$$M=\free(I)/(\Rcal_{j})_{j\in J}\pkt$$
If $\#I=r$, then $M\cong(\N^{r})^{\infty}/\!\sim_{\varepsilon}$, where $\varepsilon:e_{i}\mto a_{i}$, $i\in\{1\kpkt r\}$. In Section \ref {SecFGbinoids}, finitely generated commutative binoids are studied in detail.
\end {Example}

Similar to the situation for semigroups and monoids the index set $J$ in the preceding example can be replaced by a finite subset if $M$ is a finitely generated \emph{commutative} binoid. This result is known as the Theorem of R\'edei, and the proof for binoids is identical.

\begin {Lemma}\label{LemPreston} \emph{(Preston)}
Let $\sim$ be a congruence on the free commutative binoid $(\N^{n})^{\infty}$ and $\Iideal(\sim)$ the ideal in the polynomial ring $\Z[X_{1}\kpkt X_{n}]$, which is the binoid algebra of $(\N^{n})^{\infty}$ over $\Z$, generated by the binomials $X^{a}-X^{b}$, where $a,b\in(\N^{n})^{\infty}$ such that $a\sim b$ and $X^{c}=X_{1}^{c_{1}}\cdots X_{n}^{c_{n}}$ for $c=(c_{1}\kpkt c_{n})\in\N^{n}$ and $X^{\infty}=0$. Then $a\sim b$ if and only if $X^{a}-X^{b}\in\Iideal(\sim)$.
\end {Lemma}
\begin {proof}
The proof of the statement for the monoid $\N^{n}$ can be found in \cite[VI, Lemma 1.1]{GrilletCS} for instance. The argumentation for $(\N^{n})^{\infty}$ is exactly the same.
\end {proof}

\begin {Theorem} \label{ThRedei} \emph{(R\'edei)}
Every congruence on the free commutative binoid $(\N^{n})^{\infty}$ is finitely generated.
\end {Theorem}
\begin {proof}
Let $\sim$ be a congruence on $(\N^{n})^{\infty}$. If $\sim$ is not finitely generated, there exists an ascending chain $R_{1}\subset R_{2}\subset\cdots\subset R_{k}\subset\cdots$ of finite subsets of $R(\sim)=\{(a,b)\mid a\sim b\}$, which induces an ascending chain of congruences
$$\sim_{R_{1}}\,\,\le\,\,\sim_{R_{2}}\,\,\le\,\,\cdots\,\,\le\,\,\sim_{R_{k}}\,\,\le\,\,\cdots$$
on $(\N^{n})^{\infty}$. By assumption on $\sim$, this chain cannot become stationary, and hence the same holds for the ascending chain of ideals 
$$\Iideal(\sim_{R_{1}})\,\,\subseteq\,\,\Iideal(\sim_{R_{2}})\,\,\subseteq\,\,\cdots\,\,\subseteq\,\,\Iideal(\sim_{R_{k}})\,\,\subseteq\,\,\cdots$$
in $\Z[X_{1}\kpkt X_{n}]$ by Lemma \ref{LemPreston}, contrary to Hilbert's Basis Theorem, cf.\ \cite[Theorem 1.C.4]{PatilStorch}.
\end {proof}

Similar to the theory of monoids, the following consequence of Theorem \ref{ThRedei} may be called \emph{R\'edei's Theorem for finitely generated commutative binoids}. 

\begin {Corollary} \label{CorRedei}
Every finitely generated commutative binoid $M$ with generators $x_{1}\kpkt x_{r}$ admits finitely many relations $\Rcal_{1}\kpkt\Rcal_{s}$ such that $M\cong\free(x_{1}\kpkt x_{r})/(\Rcal_{j}$, $j\in\{1\kpkt s\})$.
\end {Corollary}
\begin {proof}
This follows immediately from Theorem \ref{ThRedei} since $M\cong(\N^{r})^{\infty}/\sim_{\varepsilon}$, where $\varepsilon:(\N^{r})^{\infty}\rto M$ is the canonical binoid homomorphism $e_{i}\mto x_{i}$, $i\in\{1\kpkt r\}$.
\end {proof}

\begin {Example}
Though finite, the number of relations on a finitely generated commutative binoid might be arbitrary large as the following example of binoids with only two generators shows:
$$M_{n}\,:=\,\free(x,y)/(\Rcal_{k}, k\in\{n,n+1\kpkt 2n-1\})\quad\text{with}\quad\Rcal_{k}:kx=ky\komma$$
$n\ge2$. Every such binoid $M_{n}$ is obviously positive, integral, and not torsion-free. Moreover, $M_{n}$ is not cancellative since $x$ is not cancellative as
$$x+((n-1)x+y)\,=\,nx+y\,=\,ny+y\,=\,(n+1)y\,=\,(n+1)x\,=\,x+nx$$
but $(n-1)x+y\not=nx$. Similarly, $y$ is not cancellative, which already implies that $M_{n}$ contains no non-trivial cancellative elements by Lemma \ref{LemCanM=0}.
\end {Example}

\begin {Corollary}
If $M$ is a finitely generated or commutative binoid, then so is $M/\!\sim$.
\end {Corollary}
\begin {proof}
This is clear. 
\end {proof}

Of course, if $\sim$ is a congruence on a binoid $M$, not all properties of $M$ transfer to $M/\!\sim$. In fact, the binoid $M/\!\sim$ might be nicer than the initial binoid $M$ if $M\not=(\N^{r})^{\infty}$, $r\ge1$.

\bigskip

\section {Important congruences} \label{SecImpCong}
\markright{\ref{SecImpCong} Important congruences}

In this section, we give a list of some congruences with which we will be concerned in this thesis. Most of them have a well-known counterpart in monoid theory, see for instance \cite[Chapter I.4]{Gilmer}.

\begin {Definition}
A congruence $\sim$ on $M$ with the property that $\pi:M\rto M/\!\sim$ is injective on 
$M\setminus\ker\pi$ is called an \gesperrt{ideal congruence}\index{congruence!ideal --}.
\end {Definition}

\begin {Example}\label{ExpKernKongruenz}
Every binoid homomorphism $\varphi:M\rto N$ defines a (very important) ideal congruence on $M$, namely $a\sim_{\ker\varphi}b$ if $a=b$ or $a,b\in\ker\varphi$. \nomenclature[ACongruenceKer]{$\sim_{\ker\varphi}$}{congruence on a binoid}The elements of $M/\!\sim_{\ker\varphi}$ are given by $[a]=\{a\}$ if $a\not\in\ker\varphi$ and $[a]=[\infty]$ otherwise. In the following, we will make use of the notation 
$$M/\ker\varphi$$
for $M/\!\sim_{\ker\varphi}$. This notation will be justified in Section \ref{SecIdeals}, where we will see that the congruence defined by an ideal is an ideal congruence and every ideal congruence is up to isomorphism given by an ideal, cf.\ Remark \ref{RemIndIsom}.
\end {Example}

\begin {Lemma} \label{LemFactoriazationKer}
Every binoid homomorphism $\varphi:M\rto N$ factors uniquely through $M/\ker\varphi$ such that $M/\ker\varphi$ is integral if $N$ is so.
\end {Lemma}
\begin {proof}
The factorization $\tilde{\varphi}:M/\ker\varphi\rto N$ follows from Lemma \ref{LemIndCong} and the supplement from the fact that $\tilde{\varphi}(a)=\infty$ is equivalent to $a=\infty$.
\end {proof}

\begin {Lemma}
If $M$ is a positive binoid, then so is $M/\sim$ for every ideal congruence $\sim$.
\end {Lemma}
\begin {proof}
This is clear. 
\end {proof}

\begin {Proposition} \label{PropNSpecIdealCong}
Let $M$ be a binoid and $\sim$ an ideal congruence on $M$. If $\pi:M\rto M/\!\!\sim$ denotes  the canonical projection, then
$$N\minspec(M/\!\sim)\,\,\cong\,\,\{\varphi\in N\minspec M\mid\ker\pi\subseteq\ker\varphi\}$$
as semigroups.
\end {Proposition}
\begin {proof}
The canonical projection $\pi:M\rto M/\!\sim$ induces by Proposition \ref{PropIndHomNspec} a semigroup homomorphism $N\minspec M/\!\sim\,\rto N\minspec M$ with $\psi\mto\psi\pi$ such that $\psi\pi(a)=\infty$ for all $a\in\pi^{-1}([\infty])$. On the other hand, if $\varphi\in N\minspec M$ with $\ker\pi\subseteq\ker\varphi$, then $\sim\,\le\,\sim_{\varphi}$, and hence $\varphi$ factors through a binoid homomorphism $\tilde{\varphi}:M/\!\sim\,\rto N$ with $\varphi=\tilde{\varphi}\pi$ by Lemma \ref{LemIndCong}. This shows that $N\minspec(M/\!\!\sim)\rto\{\psi\in N\minspec M\mid\ker\pi\subseteq\ker\varphi\}$ with $\psi\mto\psi\pi$ is surjective, and it is injective because $\phi\pi=\psi\pi$ implies $\phi=\psi$ by the surjectivity of $\pi$.
\end {proof}

\begin {Lemma} \label {LemIdealCong}
Let $M$ be a commutative binoid. 
\begin {ListeTheorem}
\item The relation $\sim_{\opint}\!$\nomenclature[ACongruenceInt]{$\sim_{\opint}$}{congruence on a binoid} on $M$ given by $a\sim_{\opint}\!b$ if $a=b$ or $a,b\in\nonint(M)$, is an ideal congruence such that $M/\!\sim_{\opint}\,\cong\, M_{\opint}$.
\item The relation $\sim_{\opred}\!$\nomenclature[ACongruenceRed]{$\sim_{\opred}$}{congruence on a binoid} on $M$ given by $a\sim_{\opred}\!b$ if $a=b$ or $a,b\in\nil(M)$, is an ideal congruence such that $M/\!\sim_{\opred}=:M_{\opred}$.\nomenclature[MACongruence]{$M_{\opred}$}{$=M/\sim_{\opred}$}
\end {ListeTheorem}
\end {Lemma}
\begin {proof}
Clearly, $\sim_{\opint}$ and $\sim_{\opred}$ are equivalence relations. To show that $\sim_{\opint}$ is a congruence let $a,b\in M$ with $a\sim_{\opint}\! b$. Clearly, $a=b$ implies $a+c=b+c$ for all $c\in M$. So consider the case $a,b\in\nonint(M)$. Then $a+x=\infty=b+y$ for some $x,y\in M\opkt$. By the commutativity of $M$, we obtain $(a+c)+x=\infty=(b+c)+y$ for all $c\in M$, which shows $a+c\sim_{\opint}\!b+c$. The last assertion of (1) follows from the observation that $[a]=[\infty]$ if and only if $a\in\nonint(M)$ and $[a]=\{a\}$ otherwise. The poof of (2) is similar to that of (1). 
\end {proof}

\begin {Corollary} \label{CorNspecIsoMred}
If $M$ is a commutative and $N$ a reduced binoid, then $N\minspec M\cong N\minspec M_{\opred}$ as semigroups.
\end {Corollary}
\begin {proof}
By Proposition \ref{PropIndHomNspec}, the canonical projection $\pi:M\rto M_{\opred}$ induces an injective semigroup homomorphism $N\minspec M_{\opred}\rto N\minspec M$, $\psi\mto\psi\pi$. For the surjectivity, we need to show that every binoid homomorphism $\varphi:M\rto N$ factors through $\pi:M\rto M_{\opred}$; that is, $\varphi=\tilde{\varphi}\pi$ for a binoid homomorphism $\tilde{\varphi}:M\rto N$. Since $N$ is reduced, $\varphi(a)=\infty_{N}$ for every $a\in\nil(M)$ and $\varphi\in N\minspec M$. Hence $\sim_{\opred}\,\le\,\sim_{\varphi}$. Now the surjectivity follows from Lemma \ref {LemIndCong}.
\end {proof}

\begin {Definition}
Let $M$ be a commutative binoid. Two elements $a,b\in M$ are called \gesperrt{associated} \index{associated}\index{element!associated --s}if $a=b+u$ for some unit $u$ of $M$.
\end {Definition}

\begin {Lemma} \label{LemPosCong}
Let $M$ be a commutative binoid. The relation $\sim_{\oppos}\!$ \nomenclature[ACongruencePos]{$\sim_{\oppos}$}{congruence on a binoid} on $M$ given by $a\sim_{\oppos}\!b$ if $a,b$ are associated, is a congruence such that $M/\!\sim_{\oppos}\,=:M_{\oppos}$ \nomenclature[MACongruence]{$M_{\oppos}$}{$=M/\sim_{\oppos}$}is a positive binoid. Moreover, $M_{\oppos}\,\cong\,M$ if and only if $M$ is positive.
\end {Lemma}
\begin {proof}
Clearly, $\sim_{\oppos}\!$ is an equivalence relation. To show that it is a congruence let $a\sim_{\oppos}\! b$. Then $a=b+u$ for some $u\in M\okreuz$, which implies $a+c=b+c+u$ for all $c\in M$ by the commutativity of $M$. Hence, $a+c\sim_{\oppos}\! b+c$. The last assertion follows from the observation that $[u]=[0]$ for every unit $u$. The supplement is clear.
\end {proof}

The composition of the canonical embedding $\iota$ and the canonical projection $\pi_{\oppos}$
$$(M\okreuz)^{\infty}\stackrel{\iota}{\Rto}M\stackrel{\pi_{\oppos}}{\Rto}M_{\oppos}$$
yields $\pi_{\oppos}\iota(u)=[0]$ for all $u\in M\okreuz$. In general, $M$ is not the product of $(M\okreuz)^{\infty}$ and $M_{\oppos}$, i.e.\ $M\not\cong(M\okreuz)^{\infty}\times M_{\oppos}$. This already follows for $M$ finite often from a counting argument. If $\#M\okreuz=k$ and $a\not=a+u$ for $a\not=\infty$ and $u\in M\okreuz\setminus\{0\}$ (which holds for example if $M$ is separated, cf.\ Section \ref{SecSepBinoids}), then the equivalence classes with respect to $\sim_{\oppos}$ have all $k$ elements with the exception of $[\infty]=\zero$. If $\#M_{\oppos}=n+1$, then $M$ has $kn+1$ elements and not $(k+1)(n+1)$ as the product. See also Remark \ref{RemSmashPos} and Example \ref{ExpSmashPos}.

\medskip

There are also congruences related to torsion-freeness, asymptotically equivalence, and cancellativity.

\begin {Lemma}
Let $M$ be a commutative binoid.
\begin {ListeTheorem}
\item The relation $\sim_{\optf}\!$ \nomenclature[ACongruenceTF]{$\sim_{\optf}$}{congruence on a binoid} on $M$ given by $a\sim_{\optf}\!b$ if $na=nb$ for some $n\ge 1$, is a congruence such that $M/\!\sim_{\optf}\,=:M_{\optf}$ \nomenclature[MACongruence]{$M_{\optf}$}{$=M/\sim_{\optf}$} is a torsion-free binoid. In particular, $M_{\optf}\,\cong\,M$ if and only if $M$ is torsion-free.
\item The relation $\sim_{\opae}\!$ \nomenclature[ACongruenceF]{$\sim_{\opae}$}{congruence on a binoid} on $M$ given by $a\sim_{\opae}\!b$ if $a,b$ are asymptotically equivalent, is a congruence such that $M/\!\sim_{\opae}$ is free of asymptotic torsion.
\end {ListeTheorem}
\end {Lemma}
\begin {proof}
(1) Obviously, $\sim_{\optf}$ is reflexive and symmetric. If $a\sim_{\optf}\! b$ and $b\sim_{\optf}\! c$, say $na=nb$ and $mb=mc$ for $n,m\ge 1$, then $(nm)a=(nm)c$, and hence $\sim_{\optf}$ is transitive. By the commutativity of $M$, the relation $\sim_{\optf}$ is compatible with the operation on $M$ since $na=nb$ implies $n(a+c)=na+nc=nb+nc=n(b+c)$ for every $c\in M$. Finally, $M/\!\sim_{\optf}$ is torsion-free because if $n[a]=n[b]$, then $m(na)=m(nb)$ for some $m\ge 1$, hence $[a]=[b]$. The poof of (2) is similar to that of (1).
\end {proof}

\begin {Lemma}
Let $M$ be a commutative binoid.
\begin {ListeTheorem}
\item  If $M$ is positive, then so is $M_{\optf}$.
\item If $N$ is a torsion-free binoid, then $N\minspec M\cong N\minspec M_{\optf}$ as semigroups.
\end {ListeTheorem}
\end {Lemma}
\begin {proof}
The first statement is obvious and the second follows similar to Corollary \ref{CorNspecIsoMred}
\end {proof}

\begin {Lemma} \label{LemCancellation}
Let $M$ be a commutative binoid and $S$ a submonoid of $M$. The relation $\sim_{\opcan,S}\!$ \nomenclature[ACongruenceCan]{$\sim_{\opcan,S}$}{congruence on a binoid $M$, $S\subseteq M$ a submonoid} on $M$ given by $a\sim_{\opcan,S}\!b$ if $a+s=b+s$ for some $s\in S$, is a congruence such that $[s]$ is cancellative in $M/\!\sim_{\opcan,S}=:M_{\opcan,S}$\nomenclature[MACongruence]{$M_{\opcan,S}$}{$=M/\sim_{\opcan,S}$ for some submonoid $S\subseteq M$} for all $s\in S$.
\end {Lemma}
\begin {proof}
Only the transitivity is not trivial. So assume that $a\sim_{\opcan,S}b$ and $b\sim_{\opcan,S}c$ for $a, b, c\in M$. Then $a+s=b+s$ and $b+t=c+t$ for some $s,t\in S$. Since $M$ is commutative, this gives $a+(s+t)=b+s+t=c+(s+t)$ with $s+t\in S$. Hence, $a\sim_{\opcan,S}c$.
\end {proof}

If $S$ is a submonoid of $M$ such that $\infty\in S$, then all elements are congruent to each other, in which case $M_{\opcan,S}=\{[\infty]\}$. Note that the binoid $M_{\opcan,S}$ need not be positive if $M$ is so. However, we will prove later, cf.\ Lemma \ref{LemSeparatedProp2}, that this is true if $M$ contains no elements $f\in M\opkt$ with $f=f+g$ for some $g\not=0$.

\begin {Definition}
Let $M$ be a commutative binoid. The binoid $M_{\opcan}:=M_{\opcan,\opint M}$ \nomenclature[MACongruence]{$M_{\opcan}$}{$=M_{\opcan,\opint M}$}is called the \gesperrt{can\-cel\-la\-tion} (binoid) \index{binoid!cancellation -- of}of $M$.
\end {Definition}

\begin {Proposition}\label{PropCancellation}
Let $M$ be a commutative and $C$ a regular binoid.
\begin {ListeTheorem}
\item Every binoid homomorphism $\varphi:M\rto C$ with $\ker\varphi\cap\opint M=\emptyset$ factors uniquely through $M_{\opcan}$.
\item Every binoid homomorphism $\varphi:M\rto C$ has a unique factorization
$$\xymatrix{
M\ar[r]^{\varphi}\ar[d]_{\pi}&C\\
M/\ker\varphi\ar[d]_{\pi_{\opcan}}\\
(M/\ker\varphi)_{\opcan}\ar[uur]_{\tilde{\varphi}}&}\pkt$$
\end {ListeTheorem}
\end {Proposition}
\begin {proof}
The second statement follows from the first since the factorization through $M/\ker\varphi$ satisfies the assumption of (1) by Lemma \ref{LemFactoriazationKer}. By Lemma \ref{LemIndCong}, it suffices to show that under the assumptions of (1) on $\varphi:M\rto C$ one has $\sim_{\opcan,\opint M}\,\le\,\sim_{\varphi}$. So let $\ker\varphi\cap\opint M=\emptyset$ and assume that $a\sim_{\opcan,\opint M}b$. This is equivalent to $a+s=b+s$ for $a,b\in M$ and $s\in\opint M$, which gives
$$\varphi(a)+\varphi(s)=\varphi(a+s)=\varphi(b+s)=\varphi(b)+\varphi(s)\pkt$$
If this is $\infty$, then $\varphi(a)=\varphi(b)=\infty$ because of the integrality of $C$ and $\varphi(s)\not=\infty$ by assumption. In the other case, when $\varphi(a)+\varphi(s)=\varphi(b)+\varphi(s)\not=\infty$, the cancellativity of $C$ yields $\varphi(a)=\varphi(b)$. Hence, $a\sim_{\varphi}b$.
\end {proof}

\begin {Remark}\label{RemCancellation}
The preceding result may fail when the binoid $C$ is cancellative but not integral. As an example consider $M:=\free(x,y)/(x+y=2x+y,2y=\infty)$ and the binoid homomorphism
$$\varphi:M\Rto\N^{\infty}\times(\N^{\infty}/(2=\infty))\komma$$
defined by $x\mto(1,1)$ and $y\mto (\infty,1)$. The two sets
$$\ker\varphi=\{nx+my\mid n,m\ge1\}\quad\text{and}\quad\opint M=\{nx\mid n\ge1\}$$
are disjoint, but $\varphi$ does not factor through $M_{\opcan}=M/(y=x+y)=\free(x,y)/(y=x+y,2y=\infty)$.
\end {Remark}

Finally we apply these congruences to the product of a family of binoids.

\begin{Lemma} \label {LemProdCong}
Let $(M_{i})_{i\in I}$ be a family of nonzero binoids and $S_{i}\subseteq M_{i}$, $i\in I$, a family of submonoids.
\begin {ListeTheorem}
\item $\left(\prod_{i\in I}M_{i}\right)_{\opint}=\Big(\prod_{i\in I}\opint(M_{i})\Big)\cup\{\infty_{\Pi}\}$.
\item $\left(\prod_{i\in I}M_{i}\right)_{\oppos}=\prod_{i\in I}(M_{i})_{\oppos}$.
\item  If $I$ is finite, then $\left(\prod_{i\in I}M_{i}\right)_{\optf}=\prod_{i\in I}(M_{i})_{\optf}$.
\item If $S:=\prod_{i\in I}S_{i}$, then $\left(\prod_{i\in I}M_{i}\right)_{\opcan,S}=\prod_{i\in I}(M_{i})_{\opcan,S_{i}}$.
\end {ListeTheorem}
\end {Lemma}
\begin {proof}
The first assertion follows from Example \ref{ExpIntegrity}(5). (2) By definition, $(a_{i})_{i\in I}\sim_{\oppos}(b_{i})_{i\in I}$ for $(a_{i})_{i\in I}$, $(b_{i})_{i\in I}\in\prod_{i\in I}M_{i}$ is equivalent to 
$(a_{i}+u_{i})_{i\in I}=(a_{i})_{i\in I}+(u_{i})_{i\in I}=(b_{i})_{i\in I}+(u_{i})_{i\in I}=(b_{i}+u_{i})_{i\in I}$, where $(u_{i})_{i\in I}\in(\prod_{i\in I}M_{i})\okreuz=\prod_{i\in I}M_{i}\okreuz$, cf.\ Example \ref{ExpProdUnits}, which is equivalent to $a_{i}\sim_{\oppos}b_{i}$ on $M_{i}$ for all $i\in I$. Similarly, one proves (4). (3) It is clear that $(a_{i})_{i\in I}\sim_{\optf}(b_{i})_{i\in I}$ implies $a_{i}\sim_{\optf}b_{i}$ on $M_{i}$ for all $i\in I$. Conversely, if for all $i\in I$ there is an $n_{i}\ge1$ such that $n_{i}a_{i}=n_{i}b_{i}$, then $m(a_{i})_{i\in I}=m(b_{i})_{i\in I}$ with $m:=\prod_{i\in I}n_{i}$. Hence, $(a_{i})_{i\in I}\sim_{\optf}(b_{i})_{i\in I}$ in $\prod_{i\in I}M_{i}$.
\end {proof}

\begin {Remark}\label{RemProdCongruences}
\begin {ListeTheorem}
\item[]
\item There is no similar result for $\sim_{\opred}$. For instance, if $N$ and $M$ are binoids, $a,a^{\prime}\in\nil M$, and $b\not\in\nil N$, then $(a,b)\not\sim_{\opred}(a^{\prime},b)$ on $M\times N$, but $a\sim_{\opred}a^{\prime}$ on $M$ and $b\sim_{\opred}b$ on $N$.
\item Lemma \ref{LemProdCong}(3) need not be true if $I$ is infinite. As a counterexample consider a family $(M_{n})_{n\in\N}$ of binoids with $a_{n},b_{n}\in M_{n}$ such that $na_{n}=nb_{n}$ for all $n\in\N$. Then $a_{n}\sim_{\optf}b_{n}$ on $M_{n}$ but $(a_{n})_{n\in\N}\not\sim_{\optf}(b_{n})_{n\in\N}$ on $\prod_{i\in I}M_{i}$.
\end {ListeTheorem}
\end {Remark}

In Section \ref{SecIdeals} and Section \ref{SecBooleanization}, we will encounter two more important congruences, namely the Rees congruence $\sim_{I}$ of an ideal, and the congruence $\sim_{\bool}$ that yields the booleanization of a binoid. The next sections are devoted to important equivalences and congruences on the disjoint union and the product of binoids.

\bigskip

\section{Smash product}\label {SecSmashProduct}

\markright{\ref{SecSmashProduct} Smash product}

Similar to monoid theory, we have shown in Section \ref{Sec1DefProp} that the product of a family of binoids is again a binoid and that the direct sum coincides with the product if and only if the family is finite. This and the next section is devoted to other constructions for binoids which have no counterpart in the theory of monoids, but which yield particularly interesting binoid algebras.

\medskip

As is well-known, the coproduct in the category of commutative monoids is the direct sum, whereas the (finite) coproduct in the category of arbitrary monoids is given by the free monoid on the disjoint union of the monoids modulo a particular equivalence relation that ensures the universal property, cf.\ \cite[Example 3.9]{Awodey}. The same phenomenon occurs when dealing with binoids. For this reason, we only focus on the  category of commutative binoids when determining the coproduct. Nonetheless, the construction of the coproduct makes sense for arbitrary binoids (as the direct sum for monoids) and for pointed sets in general, which we will call the smash product.\footnote{\, The terminology stems from algebraic topology.} 

Moreover, several examples show that the binoid algebra (over a ring $R$) of the smash product of binoids yields the tensor product (over $R$) of the binoid algebras, which will be proved later in Corollary \ref{CorSmash=Tensor}. In the next section, when we introduce the analogue of modules and algebras for binoids, namely $N\mina$sets and $N\mina$binoids, we will extend the definition of the smash product. For this, we need to start here with pointed sets in general.

\begin {Definition}
Let $(S_{i},p_{i})_{i\in I}$ be a family of pointed sets and denote by $\sim_{\wedge}$ \nomenclature[ACongruenceZZZSmash]{$\sim_{\wedge}$}{relation on the product of pointed sets}the relation on $\prod_{i\in I}S_{i}$ given by 
$$(s_{i})_{i\in I}\sim_{\wedge}(t_{i})_{i\in I}\quad:\eq\quad  s_{i}=t_{i}, \forall i\in I\komma\quad\text{or}\quad s_{k}=p_{k},t_{\ell}=p_{\ell}\quad\text{for some}\quad k,\ell\in I\pkt$$
Then the pointed set \nomenclature[AProductSmash]{$\bigwedge_{i\in I}S_{i}$}{smash product of a family of pointed sets}
$$\bigwedge_{i\in I}S_{i}:=\Big(\prod_{i\in I}S_{i}\Big)\Big/\!\sim_{\wedge}$$
with distinguished point $[(p_{i})_{i\in I}]=:p_{\wedge}$ is called the \gesperrt{smash product} \index{smash product}\index{product!smash --}of the family $(S_{i},p_{i})_{i\in I}$. The class $[(s_{i})_{i\in I}]\in\bigwedge_{i\in I}S_{i}$ for some $(s_{i})_{i\in I}\in\prod_{i\in I}S_{i}$ will be denoted by $\wedge_{i\in I}s_{i}$.
\end {Definition}

Let $(M_{i})_{i\in I}$ be a family of binoids. Considering every binoid as a pointed set with distinguished element $\infty_{i}$, the relation $\sim_{\wedge}$ on $\prod_{i\in I}M_{i}$ identifies all tuples with at least one entry equal to $\infty_{i}$ with the absorbing element $(\infty_{i})_{i\in I}=\infty_{\Pi}$ and leaves the rest untouched.

\begin {Example} \label{ExpSmashTensor}
The binoid $M=\free(x,y)/(3x=x,2y=0)$ is the smash product of $(M\okreuz)^{\infty}$ and $M_{\oppos}$, i.e.\ $M\cong (M\okreuz)^{\infty}\wedge M_{\oppos}$. Moreover, for a ring $K$ we get 
\begin {align*}
K[M]&=K[X,Y]/(X^{3}-X,Y^{2}-1)\\
&=K[X]/(X^{3}-X)\otimes_{K}K[Y]/(Y^{2}-1)\\
&=K[M_{\oppos}]\otimes_{K}K[(M\okreuz)^{\infty}]\pkt
\end {align*}
\end {Example}


\begin {Lemma}
Given a family $(M_{i})_{i\in I}$ of binoids, the relation $\sim_{\wedge}$ on $\prod_{i\in I}M_{i}$ 
is an ideal congruence such that $\bigwedge_{i\in I}M_{i}$ is a binoid with identity element $\wedge_{i\in I}0_{i}=:0_{\wedge}$ and absorbing element $\wedge_{i\in I}\infty_{i}=:\infty_{\wedge}$. Moreover, the canonical inclusions\nomenclature[AProductSmash]{$\bigwedge_{i\in I}M_{i}$}{smash product of a family of binoids}
$$\iota_{k}:M_{k}\Rto\bigwedge_{i\in I}M_{i}\komma\quad a\lto \wedge_{i\in I}a_{i}\komma$$
where $a_{k}=a$ and $a_{i}=0$ for $i\not=k$, are binoid homomorphisms.
\end {Lemma}
\begin {proof}
All assertions are easily verified.
\end {proof}

Note that the smash product is the zero binoid if one binoid of the family is so, and $\bigwedge_{i\in I}\trivial\cong\trivial$. In general, one has $M\wedge\trivial=M$ and $M\wedge\zero=\{\infty_{\wedge}\}$ for every binoid $M$. Thus, $\trivial$ is an \azl identity object\azr and $\zero$ an \azl absorbing object\azr in the category of binoids $\Bsf$ with respect to the smash product.

\begin {Example} 
If $(M_{i})_{i\in I}$ is a finite family of nonzero binoids and $A_{i}\subseteq M_{i}$ a generating set of $M_{i}$, $i\in I$, then $\bigwedge_{i\in I}M_{i}$ is generated by\nomenclature[EB]{$a\widehat{e_{i}}$}{$=0\wedgepkt 0\wedge a\wedge 0\wedgepkt
0\in\bigwedge_{k\in I}M_{k}$, where $a\in M_{i}$ is the $i$th entry} 
$$a\widehat{e_{i}}:=0\wedgepkt 0\wedge a\wedge 0\wedgepkt
0\komma$$
$a\in A_{i}$, $i\in I$, where $a$ is the $i$th entry of $a\widehat{e_{i}}$. In particular, the generators of $\bigwedge_{i=1}^{r}\N^{\infty}$ are $$\widehat{e_{i}}:=0\wedgepkt0\wedge 1\wedge0\wedgepkt 0\komma$$
\nomenclature[EB]{$\widehat{e_{i}}$}{$=0\wedgepkt0\wedge 1\wedge0\wedgepkt 0$, where $1$ is the $i$th entry}$i\in\{1\kpkt n\}$, where $1$ is the $i$th entry of $\widehat{e_{i}}$. Note that $\bigwedge_{i=1}^{r}\N^{\infty}\cong(\N^{r})^{\infty}$.
\end {Example}

\begin {Remark} \label{RemSmashPos}
Let $G$ be a binoid group and $N$ a positive binoid. Then $G\wedge N$ is a binoid whose unit group is $G\okreuz$ and such that $(G\wedge N)_{\oppos}$ is isomorphic to $N$. In particular, there is an injective and a surjective binoid homomorphism
$$(G\okreuz)^{\infty}\Rto G\wedge N\Rto N$$
such that their composition is just the characterisic point $\chi_{G\okreuz}$. 
\end {Remark}

\begin {Example} \label{ExpSmashPos}
Consider the binoid group $G=(\Z/2\Z)^{\infty}=\free(y)/(2y=0)$ and the binoid $N=\free(x)/(3x=x)$. Their smash product is
$$G\wedge N=\free(x,y)/(3x=x,2y=0)\pkt$$
Thus, $G\wedge N$ is the binoid from Example \ref{ExpSmashTensor}. Now consider the binoid $$\tilde{M}=\free(x,y)/(3x=x+y,2y=0)\pkt$$
Both binoids, $G\wedge N$ and $\tilde{M}$, have the same unit group $\{0,y\}\cong\Z/2\Z$ and coincide \emph{as sets}, more precisely, one has $G\wedge N=\{\infty,0,x,y,x+y,2x,2x+y\}=\tilde{M}$. Hence, $(G\wedge N)_{\oppos}=\{[\infty],[0],[x],2[x]\}=N=\tilde{M}_{\oppos}$ as binoids. Thus, 
$$((G\wedge N)\okreuz)^{\infty}\times (G\wedge N)_{\oppos}\,=\,(\tilde{M}\okreuz)^{\infty}\times\tilde{M}_{\oppos}\pkt$$
However, $G\wedge N\not\cong\tilde{M}$ since $3a=a$ for all $a\in G\wedge N$, but $3(x+y)=x$ in $\tilde{M}$. The corresponding binoid algebras over a ring $K$
$$K[G\wedge N]=K[X,Y]/(X^{3}-X,Y^{2}-1)$$
and 
$$K[\tilde{M}]=K[X,Y]/(X^{3}-YX, Y^{2}-1)$$
are not isomorphic either. 
\end {Example}

\begin {Lemma} \label{LemProdFactorsWedge}
Let  $(M_{i})_{i\in I}$ be a family of nonzero binoids and $N$ a binoid. Every binoid homomorphism $\varphi:\prod_{i\in I}M_{i}\rto N$ with $\varphi((a_{i})_{i\in I})=\infty$ if $a_{i}=\infty$ for at least one $i\in I$, factors through $\bigwedge_{i\in I}M_{i}$.
\end {Lemma}
\begin {proof}
This is an immediate consequence of Lemma \ref{LemIndCong} since $\sim_{\varphi}\le\sim_{\pi}$, where $\pi$ denotes the canonical epimorphism $\prod_{i\in I}M_{i}\rto\bigwedge_{i\in I}M_{i}$.
\end {proof}

\begin {Lemma}  \label{LemSmashRules}
Let $(M_{i})_{i\in I}$ be a family of nonzero binoids.
\begin {ListeTheorem}
\item $\bigwedge_{i\in I}M_{i}$ is commutative if and only if all $M_{i}$ are commutative.
\item $\left(\bigwedge_{i\in I}M_{i}\right)\okreuz\cong\prod_{i\in I}M_{i}\okreuz$.
\item $\opint\left(\bigwedge_{i\in I}M_{i}\right)\cong\prod_{i\in I}\opint(M_{i})$ and $\nonint\left(\bigwedge_{i\in I}M_{i}\right)=\{\wedge_{i\in I} a_{i}\mid a_{k}\in \nonint(M_{k})$ for at least one $k\in I\}$. In particular, $\left(\bigwedge_{i\in I}M_{i}\right)_{\opint}=\bigwedge_{i\in I}(M_{i})_{\opint}$; that is, $\bigwedge_{i\in I}M_{i}$ is integral if and only if all $M_{i}$ are integral. In this case, $\bigwedge_{i\in I}M_{i}=(\prod_{i\in I}M_{i}\opkt)^{\infty}$.
\item $\nil\left(\bigwedge_{i\in I}M_{i}\right)=\{\wedge_{i\in I} a_{i}\mid a_{k}\in \nil(M_{k})$ for at least one $k\in I\}$.
\item $\left(\bigwedge_{i\in I}M_{i}\right)_{\oppos}=\bigwedge_{i\in I}(M_{i})_{\oppos}$.
\item If $I$ is finite, then $\left(\bigwedge_{i\in I}M_{i}\right)_{\optf}=\bigwedge_{i\in I}(M_{i})_{\optf}$.
\item Let $S_{i}\subseteq M_{i}$, $i\in I$, be a family of submonoids. If $S:=\{\wedge_{i\in I}s_{i}\mid s_{i}\in S_{i}\}\subseteq\bigwedge_{i\in I}M_{i}$, then $\left(\bigwedge_{i\in I}M_{i}\right)_{\opcan,S}=\bigwedge_{i\in I}(M_{i})_{\opcan,S_{i}}$.
\end {ListeTheorem}
\end {Lemma}
\begin {proof}
By definition, the class of an element in $\prod_{i\in I}M_{i}$ under $\sim_{\wedge}$ is a singleton if (and only if) there is no entry equal to $\infty_{i}$ and $\infty_{\wedge}$ otherwise. In particular, $\sim_{\wedge}$ is an ideal congruence. Having observed this, most statements follow by comparsion with similar statements on the product $\prod_{i\in I}M_{i}$. See, for instance, Lemma \ref {LemProdCong} for (5)-(7). The first statement is obvious by Lemma \ref{LemProductCom}. The statements (2)-(4) are easily verified. See also Example \ref {ExpIntegrity}(5). 
\end {proof}

\begin {Remark}
As for the product, cf.\ Remark \ref{RemProdCongruences}(1), there are no such results like (5)-(7) for the congruence $\sim_{\opred}$ on the smash product. For instance, since $[e_{i,\infty}]=[e_{j,\infty}]=\infty_{\wedge}$ in $\wedge_{i\in I}M_{i}$, one has $[e_{i,\infty}]\sim_{\opred}[e_{j,\infty}]$ but $0\not\sim_{\opred}\infty$ on every nonzero binoid.
\end {Remark}

The smash product is a universal object, namely, it is the (finite) coproduct in the category of commutative binoids $\cBsf$. This is contained in the following result.

\begin {Proposition} \label{PropCoproductComBinoid}
Let $(M_{i})_{i\in I}$ be a finite family of binoids and $N$ a commutative binoid. Every family $\varphi_{i}:M_{i}\rto N$, $i\in I$, of binoid homomorphisms gives rise to a unique binoid homomorphism $\varphi:\bigwedge_{i\in I}M_{i}\rto N$ such that $\varphi\iota_{k}=\varphi_{k}$, where $\iota_{k}$ denotes the canonical embedding $M_{k}\embto\bigwedge_{i\in I}M_{i}$, $a\mto a\widehat{e_{k}}$, $k\in I$.
\end {Proposition}
\begin {proof}
Define $\varphi(\wedge_{i\in I}a_{i}):=\sum_{i\in I}\varphi_{i}(a_{i})$ for $\wedge_{i\in I}a_{i}\in\bigwedge_{i\in I}M_{i}$. It is easily checked that this is a well-defined binoid homomorphism with $\varphi_{k}=\varphi\iota_{k}$ for all $k\in I$.
\end {proof}

Similar to the situation in the category of monoids, the induced map $\varphi$ in the proof of Corollary \ref{PropCoproductComBinoid} is a homomorphism because $N$ is commutative.

\begin {Example}
Since $\bigwedge_{i=1}^{n}\N^{\infty}=(\N^{n})^{\infty}$, one has $N\minspec\bigwedge_{i=1}^{n}\N^{\infty}=N^{n}$ by Lemma \ref{LemNspec}, which yields by Example \ref{ExpAdjointDual}
$$\Big(\bigwedge_{i=1}^{n}\N^{\infty}\Big)\dual=(\trivial^{n})\onull\quad\text{and}\quad\Big(\bigwedge_{i=1}^{n}\N^{\infty}\Big)\bidual=\trivial^{n}\pkt$$
\end {Example}

\begin {Proposition} \label {PropNspecSmash}
Let $(M_{i})_{i\in I}$ be a finite family of nonzero binoids. If $N$ is a commutative binoid, then
$$N\minspec\bigwedge_{i\in I}M_{i}\,\,\cong\,\,\prod_{i\in I}N\minspec M_{i}$$
as semigroups. The isomorphism is given by $\varphi\mto(\varphi\iota_{i})_{i\in I}$, where $\iota_{k}$ denotes the canonical embedding $M_{k}\embto\bigwedge_{i\in I}M_{i}$, $a\mto a\widehat{e_{i}}$, with inverse $\varphi\mapsfrom(\varphi_{i})_{i\in I}$, where $\varphi(\wedge_{i\in I}a_{i}):=\sum_{\in I}\varphi_{i}(a_{i})$.
\end {Proposition}
\begin {proof}
The canonical embeddings $\iota_{k}:M_{k}\rto\bigwedge_{i\in I}M_{i}$, $k\in I$, with $a\mto a\widehat{e_{i}}$ induce by Proposition \ref{PropIndHomNspec} a semigroup homomorphism $N\minspec\bigwedge_{i\in I}M_{i}\rto N\minspec M_{k}$ with $\varphi\mto\varphi\iota_{k}$. By the universal property of the product in the category of semigroups, this gives rise to a semigroup homomorphism
$$\psi:N\minspec\bigwedge_{i\in I}M_{i}\Rto\prod_{i\in I}N\minspec M_{i}$$
with $\varphi\mto(\varphi\iota_{i})_{i\in I}$, which is surjective by Proposition \ref{PropCoproductComBinoid}. For the injectivity suppose that $\varphi$, $\varphi^{\prime}\in N\minspec\bigwedge_{i\in I}M_{i}$ with $\varphi\iota_{k}=\varphi^{\prime}\iota_{k}$ for all $k\in I$. Since every element $\wedge_{i\in I}a_{i}\in\bigwedge_{i\in I}M_{i}$ can be written as $\sum_{i\in I}a_{i}\widehat{e_{i}}=\sum_{i\in I}\iota_{i}(a_{i})$, we get
$$\varphi(\wedge_{i\in I}a_{i})=\sum_{i\in I}\varphi\iota_{i}(a_{i})=\sum_{i\in I}\varphi^{\prime}\iota_{i}(a_{i})=\varphi^{\prime}(\wedge_{i\in I}a_{i})\pkt$$
Hence, $\varphi=\varphi^{\prime}$. For the supplement, one easily checks that $\varphi$ as given in the statement is the preimage of $(\varphi_{i})_{i\in I}$ under $\psi$.
\end {proof}

To determine the $N\mina$spectrum of the product of a family of nonzero binoids, we need the following definition.

\begin{Definition}
The \gesperrt{disjoint union} \index{disjoint union}of a family $(S_{i})_{i\in I}$ of arbitrary sets is given by the set\nomenclature[AProductDisjointUnion]{$\biguplus_{i\in I}S_{i}$}{disjoint union of a family of sets}
$$\biguplus_{i\in I}S_{i}:=\{(s;i)\mid s\in S_{i}, i\in I\}\pkt$$
\end{Definition}

\begin {Proposition} \label{PropNspecProd}
Let $(M_{i})_{i\in I}$ be a finite family of nonzero binoids. If $N$ is a commutative binoid with no non-trivial idempotents, then 
$$N\minspec\prod_{i\in I}M_{i}\,\,\cong\,\,\biguplus_{\emptyset\not=J\subseteq I}\!\!N\minspec\bigwedge_{i\in J}M_{i}\,\,\cong\,\,\biguplus_{\emptyset\not=J\subseteq I}\prod_{i\in J}N\minspec M_{i}$$
as semigroups, where the semigroup structure in the middle and on the right-hand side are given by 
$$(\psi;J)+(\psi^{\prime};J^{\prime}):=(\psi+\psi^{\prime};J\cup J^{\prime})\komma$$
where $(\psi+\psi^{\prime})((a_{i})_{i\in J\cup J^{\prime}}):=\psi((a_{i})_{i\in J})+\psi^{\prime}((a_{i})_{i\in J^{\prime}})$, and
$$((\psi_{i})_{i\in J};J)+((\psi_{i}^{\prime})_{i\in J};J^{\prime}):=((\psi_{i})_{i\in J}+(\psi_{i})_{i\in J^{\prime}}^{\prime});J\cup J^{\prime})\komma$$
where $((\psi_{i})_{i\in J}+(\psi_{i})_{i\in J^{\prime}}^{\prime})((a_{i})_{i\in J\cup J^{\prime}}):=(\psi_{i}(a_{i}))_{i\in J}+(\psi_{i}^{\prime}(a_{i}))_{i\in J^{\prime}}$, respectively.
\end {Proposition}
\begin{proof}
The latter semigroup isomorphism is an immediate consequence of Proposition \ref{PropNspecSmash} and given by $(\psi;J)\mto((\psi\iota_{i})_{i\in J};J)$, where $\iota_{i}$ denotes the binoid embedding $M_{i}\embto\bigwedge_{j\in J}M_{j}$. To obtain the first isomorphism observe that for every $J \subseteq I$, $J \neq \emptyset$, the natural binoid epimorphism $\pi\pi_{J}$, where
$$\prod_{i \in I} M_{i} \stackrel{\pi_J}{\Rto}\prod_{i \in J} M_{i} \stackrel{\pi}{\Rto}\bigwedge_{i\in J} M_{i}\komma$$
induces a semigroup embedding
$$\varphi_{J}:N\minspec\bigwedge_{i \in J}M_{i}\Rto N\minspec\prod_{i \in I} M_{i}$$
with $\psi\mto\psi\pi\pi_{J}$. By the universal property of the disjoined union, these embeddings give rise to a map
$$\varphi:\biguplus_{\emptyset\not=J\subseteq I}N\minspec\bigwedge_{i\in J}M_{i}\Rto N\minspec\prod_{i\in I}M_{i}$$
with $(\psi;J)\mto\psi\pi\pi_{J}$. It is easily checked that this is a semigroup homomorphism, where the disjoint union is a semigroup as described in the proposition. We show that $\varphi$ is a bijection. For this take an arbitrary $\phi\in N\minspec\prod_{i \in I} M_{i}$ and set 
$$J:=\{i\in I\mid\phi(e_{i,\infty})=\infty\}\pkt$$
By Lemma \ref{LemHomProperties}(2) and the assumption on $N$, we obtain $\phi(e_{i,\infty})\in\bool(N)=\trivial$. Since $\sum_{i \in I}e_{i,\infty} = \infty$, we have $\sum_{i \in I} \phi(e_{i,\infty})=\infty$ in $N$, which therefore implies that $\phi(e_{i,\infty})= \infty$ for at least one $i\in I$. In particular, $J\not=\emptyset$. We claim that $\phi$ factors through $\prod_{i \in J} M_i$. For this, we may assume that $I=\{1, \ldots ,n\}$  and that $J=\{1, \ldots , k\}$. Since $0=\phi(\sum_{i\not\in J}e_{i,\infty})$ ($=\phi(0\kpkt 0,\infty\kpkt\infty)$), we obtain
$$\phi(a_{1}\kpkt a_{n})\,=\,\phi\Big((a_{1}\kpkt a_{n})+\sum_{i\not\in J}e_{i,\infty}\Big)=\phi(a_{1}\kpkt a_{k},\infty\kpkt\infty)\komma$$
which shows that $\phi$ depends only on the $J\mina$components. Therefore, $\phi$ factors through the binoid homomorphism
$$\tilde{\phi}:\prod_{i\in J}M_{i}\Rto N\komma\quad(a_{1}\kpkt a_{k})\lto\phi(a_{1}\kpkt a_{k},\infty\kpkt\infty)\pkt$$ 
Hence, $\tilde{\phi}$ factors through $\wedge_{i\in J}M_{i}$ by Lemma \ref{LemProdFactorsWedge}. This shows the surjectivity of $\varphi$. The injectivity is clear since $J$ is uniquely determined by $\phi$. 
\end{proof}

The condition on $N$ containing only trivial idempotent elements is necessary in the preceding proposition. Consider, for instance, the binoid $M=\free(x_{1},x_{2})/(x_{1}+x_{2}= \infty, 2x_{1}=x_{1}, 2x_{2}=x_{2})$. The binoid homomorphism
$$\N^{\infty}\times\N^{\infty}\Rto M$$
with $e_{i}\mto 0$ and $e_{i,\infty}\mto x_{i}$, $i\in\{1,2\}$, does not factor through one of the factors.

\begin {Example} 
We apply Proposition \ref{PropNspecProd} to determine the $N\mina$spectra of the products $\trivial^{n}$ and $(\N^{\infty})^{n}$, where $N$ denotes a binoid with $\bool(N)=\trivial$. For this, recall that $\bigwedge_{i\in I}M_{i}=(\prod_{i\in I}M_{i})^{\infty}$ if all $M_{i}$ are integral, cf.\ Lemma \ref{LemSmashRules}(3), and let $I=\{1\kpkt n\}$. Since $\bigwedge_{i\in I}\trivial\cong\trivial$ and $N\minspec\trivial=\{\trivial\rto N\}\cong\zero$, we get
$$N\minspec\trivial^{n}\,\,\,\cong\,\,\,\biguplus_{\emptyset\not=  J\subseteq I}\zero\komma$$
as semigroups. Similarly, since $\bigwedge_{i=1}^{r}\N^{\infty}=(\N^{r})^{\infty}$ for every $r\ge1$, we obtain
$$N\minspec(\N^{\infty})^{n}\,\,\,\cong\,\,\,\biguplus_{\emptyset\not=J\subseteq I}N\minspec(\N^{\#J})^{\infty}\,\,\,\cong\,\,\,\biguplus_{\emptyset\not=J\subseteq I}N^{\#J}$$
as semigroups, where the latter isomorphism is due to Lemma \ref{LemNspec}. In particular, for the dual of $\trivial^{n}$ and $(\N^{\infty})^{n}$, we get
$$(\trivial^{n})\dual\,\,\cong\,\,\biguplus_{J\subseteq I}\zero\quad\text{and}\quad((\N^{\infty})^{n})\dual\,\,\cong\,\,\biguplus_{J\subseteq I}\trivial^{\#J}$$
See also Example \ref{ExpBHomBoolean}(2).
\end {Example}

\bigskip

\section {Pointed unions}  \label{SecPointedUnion}
\markright{\ref{SecPointedUnion} Pointed unions}

In this section, we construct a binoid, called the bipointed union, from the disjoint union of a family of (positive) binoids, which was introduced at the end of the last section. As a first step, we define the pointed union of a family of pointed sets.

\begin {Lemma} \label{LemGlueingPoints}
Given a family $(S_{i},p_{i})_{i\in I}$ of pointed sets, the relation $\sim_{\infty}$ on the disjoint union $\biguplus_{i\in I}S_{i}$ that glues all points $(p_{i};i)$, $i\in I$, together and leaves the rest untouched, i.e.\ \nomenclature[ACongruenceZZGlued]{$\sim_{\infty}$}{equivalence relation on the disjoint union of pointed sets}
$$(s;i)\sim_{\infty}(t;j)\quad:\eq\quad i=j\text{ and }s=t\quad\text{or}\quad a=p_{i}\text{ and }b=p_{j}\komma$$
defines an equivalence relation such that\nomenclature[AProductPointeda]{$\bigcupdot_{i\in I}S_{i}$}{pointed union of a family of pointed sets} 
$$\Big(\biguplus_{i\in I}S_{i}\Big)\Big/\!\sim_{\infty}\,\,\,=:\,\bigcupdot_{i\in I}S_{i}$$
is again a pointed set $(\bigcupdot_{i\in I}S_{i},p)$, where $p=[(p_{i};i)]$, $i\in I$.\footnote{\, Apart from the class of $p$, all classes are singletons and will therefore be written as $(a;i)$ for $a\in S_{i}$, $i\in I$, by abuse of notation.} In particular, if $(M_{i})_{i\in I}$ is a family of binoids (or semibinoids), then the addition
$$(a;i)+(b;j):=\begin {cases}
(a+b;i)&\text{, if }i=j,\\
\infty&\text{, otherwise,}
\end {cases}$$
defines a semibinoid structure on $\bigcupdot_{i\in I}M_{i}$.
\end {Lemma}
\begin {proof}
This is easily verified.
\end {proof}

\begin {Definition}
Let $(S_{i})_{i\in I}$ be a family of pointed sets. The pointed set $\bigcupdot_{i\in I}S_{i}$ is called the \gesperrt{pointed union}. \index{pointed union!-- of pointed sets}By the pointed union of a family $(M_{i})_{i\in I}$ of binoids, we always mean the semibinoid $\bigcupdot_{i\in I}M_{i}$  \index{pointed union!-- of binoids}with addition given as in Lemma \ref{LemGlueingPoints}.
\end {Definition}

Note that $\bigcupdot_{i\in I}M_{i}$ contains no integral elements. There are examples where the pointed union of pointed sets turns into a binoid with respect to a certain addition different from the one of the pointed union of binoids, cf.\ Remark \ref{RemBlowupBinoid}.

\begin {Remark}
The disjoint union of a family $(M_{i})_{i\in I}$, $\#I\ge 2$, of binoids turns \emph{not} into a monoid by glueing together the identity elements using the same construction as in the lemma above with $0$ instead of $\infty$ everywhere. In fact, the operation $(a;i)+(b;j):=(a+b;i)$ if $i=j$ and $0:=[{(0;i)}]$ otherwise, is not associative since for $i\not= j$ and $a\not=0_{i}$, one has $((a;i)+(b;i))+(c;j)=0 $ and $(a;i)+((b;i)+(c;j))=(a;i)\not=0$.

For the same reason, the semibinoid $\bigcupdot_{i\in I}M_{i}$ turns not into a binoid by glueing the identity elements $(0;i)$, $i\in I$, together (so that $0:=[{(0;i)}]$ becomes the desired identity element) \emph{if at least one $M_{i}$ is not positive}. To see this let $0_{i}\not=a\in M_{i}\okreuz$ and $\infty_{j}\not=b\in M_{j}$, $j\not=i$. Then $((a;i)+(\minus a;i))+(b;j)=(b;j)$ and $(a;i)+((\minus a;i)+(b;j))=\infty$; that is, the operation on $\bigcupdot_{i\in I}M_{i}$ is not associative anymore.

There are other relations on the disjoint union of not necessarily positive binoids, where among other identifications all identity elements are glued together as well as all absorbing elements, such that the quotient is a binoid with respect to a certain addition different from those considered in this section, cf.\ Lemma \ref{LemDirectLimit}.
\end {Remark}

\begin {Lemma} \label {LemPointedUnion}
Let $(M_{i})_{i\in I}$ be a family of positive binoids. The relation $\sim_{_{\bullet}}$\nomenclature[ACongruenceZZPointed]{$\sim_{_{\bullet}}$}{equivalence relation on the disjoint union of positive binoids} on the disjoint union $\biguplus_{i\in I}M_{i}$ of the underlying sets of the binoids given by 
$$(a;i)\sim_{_{\bullet}}(b;j)\quad:\eq\quad i=j\text{ and }a=b\quad\text{or}\quad a=\infty_{i}\text{ and }b=\infty_{j}\quad\text{or}\quad a=0_{i}\text{ and }b=0_{j}\komma$$
defines an equivalence relation such that\nomenclature[AProductPointedb]{$\bigcupbidot_{i\in I}M_{i}$}{bipointed union of a family of positive binoids} 
$$\Big(\biguplus_{i\in I}M_{i}\Big)\Big/\!\sim_{_{\bullet}}\,\,\,=:\,\bigcupbidot_{i\in I}M_{i}$$
equipped with the addition
$$(a;i)+(b;j):=\begin {cases}
(a+b;i)&\text{, if }i=j,\\
(b;j)&\text{, if }a=0,\\
(a;i)&\text{, if }b=0,\\
\infty&\text{, otherwise,}
\end {cases}$$
is a binoid with absorbing element $\infty:=[(\infty_{i};i)]$ and identity element $0:=[(0_{i};i)]$. Moreover, the canonical inclusions 
$$\iota_{k}:M_{k}\Rto\bigcupbidot_{i\in I}M_{i}\quad\text{with}\quad a\lto(a;k)$$
and projections 
$$\pi_{k}:\bigcupbidot_{i\in I}M_{i}\Rto M_{k}\quad\text{with}\quad(a;i)\lto\begin {cases}
a&\text{, if }i=k\komma\\
\infty_{k}&\text{, otherwise,}
\end {cases}$$
$k\in I$, are binoid homomorphisms.
\end {Lemma}
\begin {proof}
This is easy to check.
\end {proof}

By definition, we have $\bigcupbidot_{i\in I}M_{i}=\big(\bigcupdot_{i\in I}M_{i}\big)\big/\!\sim$, where $\sim$ is the equivalence relation on $\bigcupdot_{i\in I}M_{i}$ that glues the identity elements $(0;i)$, $i\in I$, together.

\begin {Definition}
Let $(M_{i})_{i\in I}$ be a family of positive binoids. With the notation of Lemma \ref {LemPointedUnion}, the binoid $\bigcupbidot_{i\in I}M_{i}$ is called the \gesperrt{bipointed union} \index{bipointed union}of the family $(M_{i})_{i\in I}$.
\end {Definition}

The bipointed union of a family of positive binoids can be visualized as a family of lines, where each line represents a binoid $M_{i}$, $i\in I$, which are glued together at the ends.

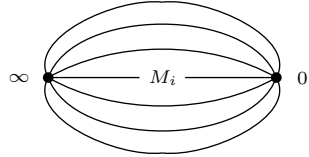
\begin {center}
\begin {pspicture}(-2,-1.5)(2,1.5)
\qdisk (1.5,0){2pt}\qdisk (-1.5,0){2pt}
\uput [0] (-2.2,0){\scriptsize{$\infty$}}
\uput [0] (-0.35,0){\scriptsize{\text{$M_{i}$}}}
\uput [0] (1.6,0){\scriptsize{$0$}}
\pcarc [arcangleA=30, arcangleB=30, linewidth=.5pt] (1.5,0)(-1.5,0)
\pcarc [arcangleA=330, arcangleB=330, linewidth=.5pt] (1.5,0)(-1.5,0)
\pcarc [arcangleA=70, arcangleB=70, linewidth=.5pt] (1.5,0)(-1.5,0)
\pcarc [arcangleA=85, arcangleB=40, linewidth=.5pt] (1.5,0)(0,-1)
\pcarc [arcangleA=40, arcangleB=85, linewidth=.5pt] (0,-1)(-1.5,0)
\pcarc [arcangleA=0, arcangleB=0, linewidth=.5pt] (1.5,0)(0.3,0)
\pcarc [arcangleA=0, arcangleB=0, linewidth=.5pt] (-1.5,0)(-0.3,0)
\pcarc [arcangleA=290, arcangleB=290, linewidth=.5pt] (1.5,0)(-1.5,0)
\pcarc [arcangleA=275, arcangleB=320, linewidth=.5pt] (1.5,0)(0,1)
\pcarc [arcangleA=320, arcangleB=275, linewidth=.5pt] (0,1)(-1.5,0)
\end {pspicture}
\end {center}

Note that a trivial binoid does not contribute to the bipointed union since $M\cupbidot\trivial\cong M$.

\begin {Remark}\label{RemProdEpisPos}
If $(M_{i})_{i\in I}$ is a family of positive binoids, then 
$$\bigcupbidot_{i\in I}M_{i}\,\,\cong\,\,\Big(\bigwedge_{i\in I}M_{i}\Big)\Big/\!\sim\komma$$
where $\sim$ denotes the ideal congruence on $\bigwedge_{i\in I}M_{i}$ given by
$\wedge_{i\in I} a_{i}\sim\infty_{\wedge}$ if $a_{i}\not=0$ for at least two $i\in I$. Thus, there are canonical binoid epimorphisms
$$\prod_{i\in I}M_{i}\stackrel{\pi_{\wedge}}{\Rto}\bigwedge_{i\in I}M_{i}\stackrel{\!\!\pi_{\cupbidot}}{\Rto}\bigcupbidot_{i\in I}M_{i}\pkt$$
\end {Remark}

\begin {Lemma} \label {LemPointedComposition}
Let $(M_{i})_{i\in I}$ be a family of positive binoids.
\begin {ListeTheorem}
\item The binoid $\bigcupbidot_{i\in I}M_{i}$ is commutative if and only if all $M_{i}$ are commutative.
\item If at least two of the binoids are non-trivial, then $\opint\big(\bigcupbidot_{i\in I}M_{i}\big)=\{0\}$. In particular, $\bigcupbidot_{i\in I}M_{i}$ is a positive binoid in which all elements $\not=0$ are non-integral.
\item $\nil\big(\bigcupbidot_{i\in I}M_{i}\big)=\,\,\bigcupbidot_{i\in I}\nil(M_{i})$ and $\big(\bigcupbidot_{i\in I}M_{i}\big)_{\opred}=\,\,\bigcupbidot_{i\in I}(M_{i})_{\opred}\pkt$
\item $\big(\bigcupbidot_{i\in I}M_{i}\big)_{\oppos}=\,\,\bigcupbidot_{i\in I}M_{i}$ and  $\big(\bigcupbidot_{i\in I}M_{i}\big)_{\optf}=\,\,\bigcupbidot_{i\in I}(M_{i})_{\optf}\pkt$
\end {ListeTheorem}
\end {Lemma}
\begin {proof}
(1)-(3) are immediate. The first assertion of (4) follows from (2) and Lemma \ref{LemPosCong}. The latter is due to the observation that for $a\in M_{i}$ and $b\in M_{j}$ the equality $n(a;i)=n(b;j)$ for some $n\ge 1$ is equivalent to $i=j$ and $na=nb$ for some $n\ge 1$ or $a=b\in\trivial$.
\end {proof}

\begin {Lemma}\label{LemGenSystPointedUnion}
Let $(M_{i})_{i\in I}$ be a finite family of positive binoids. If $A_{i}\subseteq M_{i}$ is a generating set of $M_{i}$, $i\in I$, then $\bigcupbidot_{i\in I}M_{i}$ is generated by $(a;i)$, $a\in A_{i}$, $i\in I$.
\end {Lemma}
\begin {proof}
This is obvious.
\end {proof}

The bipointed union admits a universal property.

\begin {Proposition} \label{PropUnivPropPUnion}
Let $(M_{i})_{i\in I}$ be a finite family of positive binoids and $N$ another binoid. Given a family of binoid homomorphisms $\varphi_{i}:M_{i}\rto N$, $i\in I$, such that $\varphi_{i}(a)+\varphi_{j}(b)=\infty$ for all $a\not=0_{i}$ and $b\not=0_{j}$, $i\not=j$, there is a unique binoid homomorphism
$$\varphi:\,\bigcupbidot_{i\in I}M_{i}\Rto N$$
with $\varphi\iota_{i}=\varphi_{i}$, where $\iota_{i}:M_{i}\rto\,\,\bigcupbidot_{i\in I}M_{i}$ denotes the canonical embedding $a\mto(a;i)$, $i\in I$.
\end {Proposition}
\begin {proof}
The unique binoid homomorphism is obviously given by $\varphi(a;i):=\varphi_{i}(a)$ for $a\in M_{i}$, $i\in I$.
\end {proof}

\begin {Corollary} \label{CorNSpecPUnion}
Let $(M_{i})_{i\in I}$ be a finite family of positive binoids and $N$ an integral binoid. Then we have a semibinoid isomorphism:
$$N\minspec\bigcupbidot_{i\in I}M_{i}\,\,\cong\,\,\bigcupdot_{i\in I}N\minspec M_{i}\pkt$$
\end {Corollary}
\begin {proof}
Set $M\!:=\,\,\bigcupbidot_{i\in I}M_{i}$. Note that the characteristic functions $\chi_{M_{i}\okreuz}\!:\!M_{i}\rto N$, $i\in I$, and $\chi_{M\okreuz}\!:\!M\rto N$ are binoid homomorphisms by Proposition \ref{PropHomFilter}. In particular, $N\minspec M_{i}$, $i\in I$, and $N\minspec M$ are semibinoids such that the pointed union is well-defined, cf.\ Lemma \ref{LemGlueingPoints}. The canonical projections $\pi_{k}:M\rto M_{k}$, $k\in I$, induce semigroup homomorphisms 
$$N\minspec M_{k}\Rto N\minspec M\komma$$
$k\in I$, which by the universal property of the disjoint union give rise to a map 
$$\varphi:\biguplus_{i\in I}N\minspec M_{i}\Rto N\minspec M\pkt$$
This map factors through
$$\phi:\bigcupdot_{i\in I}N\minspec M_{i}\Rto N\minspec M\komma$$
$(\psi;i)\mto\psi\pi_{i}$, because $\varphi((\chi_{M_{i}\okreuz};i))=\chi_{M_{i}\okreuz}\pi_{i}=\chi_{M\okreuz}$ for all $i\in I$. It is easily checked that $\phi$ is a semibinoid homomorphism, which is obviously injective. For the surjectivity let $\psi\in N\minspec M$ (i.e.\ $\psi:\bigcupbidot_{i\in I}M_{i}\rto N$). If $\psi=\chi_{M\okreuz}$, then $\psi$ is the image of the absorbing element under $\phi$. So let $\psi\not=\chi_{M\okreuz}$. Then there is a $(a;k)\in M\setminus\{0\}$ with $\psi((a;k))\not=\infty$. In this case, we have for all $b\in M_{i}\setminus\{0\}$, where $i\not=k$,
$$\infty=\psi(\infty)=\psi((a;k)+(b;i))=\psi(a;k)+\psi(b;i)\komma$$ 
which implies that $\psi(b;i)=\infty$ by the integrity of $N$. Hence, $\phi:\psi\iota_{k}\mto\psi$.
\end {proof}

\begin {Example}
We will double-check Corollary \ref{CorNSpecPUnion} by determining $N\minspec\bigcupbidot_{i\in I}\N^{\infty}$ step by step for an integral binoid $N$ and $I=\{1\kpkt n\}$. A generating set of the bipointed union $\bigcupbidot_{i\in I}\N^{\infty}$ is given by $\{(1;i)\mid i\in I\}$, where the generators satisfy $(1;i)+(1;j)=\infty$ for all $i,j\in I$ when $i\not=j$. Since $N$ is integral, at most one generator lies not in the kernel of any binoid homomorphism $\bigcupbidot_{i\in I}\N^{\infty}\rto N$. If for $a\in N$,
$$(\varphi_{a};k):\bigcupbidot_{i\in I}\N^{\infty}\Rto N$$
denotes the binoid homomorphism $(1;i)\mto a$ if $i=k$ and $\infty$ otherwise, then $(\varphi_{\infty};i)=(\varphi_{\infty};j)=:\alpha$ for all $i,j\in I$ and 
$$(\varphi_{a};i)+(\varphi_{b};j)\,=\,\begin {cases}
(\varphi_{a+b};i)&\text{, if }i=j\komma\\
\alpha&\text{, otherwise,}
\end {cases}$$
From this we obtain
$$N\minspec\bigcupbidot_{i\in I}\N^{\infty}\,\,=\,\,\big(\{(\varphi_{a};k)\mid a\in N\opkt,k\in I\}\cup\{\alpha\},+,\alpha\big)\,\,\cong\,\,\bigcupdot_{i\in I}N\pkt$$
In particular, the dual of $\bigcupbidot_{i\in I}\N^{\infty}$ is given by
$$\Big(\bigcupdot_{i\in I}\trivial\Big)\onull\,\,\cong\,\,\bigcupbidot_{i\in I}\free(x)/(2x=x)\,=:\,B\pkt$$
Now the same argumentation yields $\big(\bigcupbidot_{i\in I}\N^{\infty}\big)\bidual\cong B$.
\end {Example}

\begin {Example} \label{ExCompositions}
To illustrate the differences of the constructions we have encountered so far consider the following binoids:
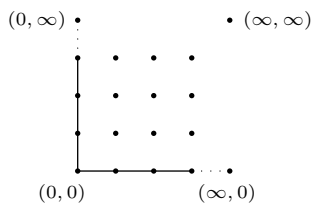
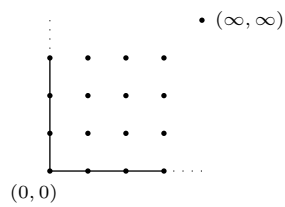
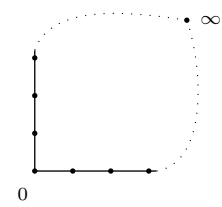
\begin {center}
\begin {pspicture} (-1.3,-0.5)(3,3)
\qdisk (0,0){1pt}\qdisk (0,0.5){1pt}\qdisk (0,1){1pt}\qdisk (0,1.5){1pt}\qdisk (0,2){1pt}
\qdisk (0.5,0){1pt}\qdisk (0.5,0.5){1pt}\qdisk (0.5,1){1pt}\qdisk (0.5,1.5){1pt}
\qdisk (1,0){1pt}\qdisk (1,0.5){1pt}\qdisk (1,1){1pt}\qdisk (1,1.5){1pt}
\qdisk (1.5,0){1pt}\qdisk (1.5,0.5){1pt}\qdisk (1.5,1){1pt}\qdisk (1.5,1.5){1pt}
\qdisk (2,0){1 pt}\qdisk (2,2){1pt}
\psline [linewidth=0.5 pt] (0,0)(0,1.5)
\psline [linewidth=0.5 pt] (0,0)(1.5,0)
\psline [linewidth=0.5 pt, linestyle=dotted] (0,1.5)(0,2)
\psline [linewidth=0.5 pt, linestyle=dotted] (1.5,0)(2,0)
\uput [0] (-0.7,-0.3){\scriptsize{$(0,0)$}}
\uput [0](-1.1,2){\scriptsize{$(0,\infty)$}}
\uput [0](1.4,-0.3){\scriptsize{$(\infty,0)$}}
\uput [0](2,2){\scriptsize{$(\infty,\infty)$}}
\uput [0] (-1.1,2.9){$\N^{\infty}\times\N^{\infty}=\N^{\infty}\oplus\N^{\infty}$:}
\end {pspicture}
\quad\quad\quad
\begin {pspicture}(-1.3,-0.5)(3,3)
\qdisk (0,0){1pt}\qdisk (0,0.5){1pt}\qdisk (0,1){1pt}\qdisk (0,1.5){1pt} 
\qdisk (0.5,0){1pt}\qdisk (0.5,0.5){1pt}\qdisk (0.5,1){1pt}\qdisk (0.5,1.5){1pt}\
\qdisk (1,0){1pt}\qdisk (1,0.5){1pt}\qdisk (1,1){1pt}\qdisk (1,1.5){1pt}
\qdisk (1.5,0){1pt}\qdisk (1.5,0.5){1pt}\qdisk (1.5,1){1pt}\qdisk (1.5,1.5){1pt}
\qdisk (2,2){1pt}
\psline [linewidth=0.5 pt] (0,0)(0,1.5)
\psline [linewidth=0.5 pt] (0,0)(1.5,0)
\psline [linewidth=0.5 pt, linestyle=dotted] (0,1.5)(0,2)
\psline [linewidth=0.5 pt, linestyle=dotted] (1.5,0)(2,0)
\uput [0] (-0.7,-0.3){\scriptsize{$(0,0)$}}
\uput [0](2,2){\scriptsize{$(\infty,\infty)$}}
\uput [0] (-0.65,2.9){$\N^{\infty}\wedge\N^{\infty}=(\N\times\N)^{\infty}:$}
\end {pspicture}
\quad\quad\quad
\begin {pspicture}(-1.3,-0.5)(3,3)
\qdisk (0,0){1pt}\qdisk (0,0.5){1pt}\qdisk (0,1){1pt}\qdisk (0,1.5){1pt}
\qdisk (0.5,0){1pt}\qdisk (1,0){1pt}\qdisk (1.5,0){1pt}\qdisk (2,2){1pt}
\psline [linewidth=0.5 pt] (0,0)(0,1.6)
\psline [linewidth=0.5 pt] (0,0)(1.6,0)
\uput [0] (-0.4,-0.3){\scriptsize{$0$}}
\uput [0](2,2){\scriptsize{$\infty$}}
\uput [0] (-0.5,2.9){$\N^{\infty}\cupbidot\,\N^{\infty}:$}
\pcarc [arcangleA=60, arcangleB=20, linewidth=.5pt, linestyle=dotted] (0,1.6)(2,2)
\pcarc [arcangleA=290, arcangleB=330, linewidth=.5pt, linestyle=dotted] (1.6,0)(2,2)
\end {pspicture}
\end {center}
In terms of generators, these binoids are given by
$$\free(x,y,u,v)/(u+v=\infty, x+u=u,y+v=v)\komma\quad\quad\free(x,y)\komma\quad\quad\free(x,y)/(x+y=\infty)\komma$$
with binoid algebras\footnote{\, Here $K[X]\times K[Y]\times K[X,Y]$ is the product in the category of rings.}
\begin {align*}
K[\N^{\infty}\!\times\N^{\infty}]&\,\,\cong\,\, K[X]\times K[Y]\times K[X,Y]\komma\\
K[\N^{\infty}\!\wedge\N^{\infty}]&\,\,\cong\,\, K[X,Y]\komma\\
K[\N^{\infty}\cupbidot\N^{\infty}]&\,\,\cong\,\, K[X,Y]/(XY)\komma
\end {align*}
over the ring $K$. The first isomorphism is given by $T^{(1,0)}\mto X$ and $T^{(0,1)}\mto Y$, and the last by $T^{(1;1)}\mto X$ and $T^{(1;2)}\mto Y$. The result for the smash product follows from Lemma \ref {LemSmashRules} and the theory of monoid rings since $K[\N^{\infty}\wedge\N^{\infty}]=K[(\N\times\N)^{\infty}]=K(\N\times\N)=K\N\otimes_{K}K\N\cong K[X,Y]$.

\end {Example}

\begin {Example}
We have the following $\R\mina$spectra:
\begin {center}
\begin {pspicture} (-2,-2.5)(2,2)
\qdisk (0,0){1.75pt}\qdisk (0,0.7){1.75pt}\qdisk (0.7,0){1.75pt}
\psline [linewidth=0.5 pt] (0,-1.5)(0,1.5)
\psline [linewidth=0.5 pt] (-1.5,0)(1.5,0)
\uput [0] (-1.45,-2.2){\small{$\R\minspec(\N^{\infty}\cupbidot\,\N^{\infty})$}}
\end {pspicture}
\quad\quad\quad\quad\quad
\begin {pspicture} (-2,-2.5)(2,2)
\qdisk (0,0){1.75pt}\qdisk (0,0.7){1.75pt}\qdisk (0.7,0){1.75pt}\qdisk (-0.5,-0.30
){1.75pt}
\psline [linewidth=0.5 pt] (0,-1.5)(0,1.5)
\psline [linewidth=0.5 pt] (-1.5,0)(1.5,0)
\psline [linewidth=0.5 pt] (1.3,0.75)(-1.3,-0.75)
\uput [0] (-1.95,-2.2){\small{$\R\minspec(\N^{\infty}\cupbidot\,\N^{\infty}\cupbidot\,\N^{\infty})$}}
\end {pspicture}
\end {center}
\end {Example}

\bigskip

\section {Operations} \label{SecOperation}

\markright{\ref{SecOperation} Operations}

This section deals with pointed sets and binoids on which a binoid $N$ operates. These so-called $N\mina$sets and $N\mina$binoids represent the binoid theoretic counterpart of modules and algebras over a ring. In this spirit, we extend the definition of the smash product to $N\mina$sets and $N\mina$binoids. In Section \ref{SecBinoidModules} and Section \ref{SecNBinoidAlgebra}, we will discuss the associated modules and algebras of $N\mina$sets and $N\mina$binoids, respectively. 

A thorough investigation of $N\mina$sets and their modules can be found in \cite[Section 2.2]{ChuLorscheidSanthanam}, where the notion of finitely generated, noetherian, and projective $N\mina$sets has also been introduced and studied in detail, while here we omit a treatment of the latter two.

\begin {Convention}
In this section, $N$ always denotes an arbitrary binoid.
\end {Convention}

\begin {Definition}
A (left) \gesperrt{operation} \index{operation}\index{pointed set!operation on a --}of $N$ on a pointed set $(S,p)$ is a map
$$+:N\times S\Rto S\komma\quad(a,s)\lto a+s\komma$$
such that the following conditions are fulfilled.
\begin {ListeTheorem}
\item $0+s=s$ for all $s\in S$.
\item $\infty+s=p$ for all $s\in S$.
\item $a+p=p$ for all $a\in N$.
\item $(a+b)+s=a+(b+s)$ for all $a,b\in N$ and $s\in S$.
\end {ListeTheorem}
Then $S$ is called an \gesperrt{$N\mina$set}\index{N@$N\mina$set}. An \gesperrt{$N\mina$map} \index{N@$N\mina$map}\index{map!N@$N\mina$--}is a pointed map $\varphi:S\rto T$ of $N\mina$sets such that the diagram
$$\xymatrix{
N\times S\ar[r]\ar[d]_{\id\times\varphi}&S\ar[d]^{\varphi}\\
N\times T\ar[r]&T}$$
commutes; that is, $\varphi(a+s)=a+\varphi(s)$ for all $a\in N$ and $s\in S$. We say $S$ is a \gesperrt{finitely generated} \index{N@$N\mina$set!finitely generated --}$N\mina$set if there exists a finite subset $T\subseteq S$ such that every $s\in S$ can be written as $s=a+t$ for some $a\in N$ and $t\in T$, i.e.\ $S=\bigcup_{t\in T}(N+t)$. Then $(S,p)$ is \gesperrt{generated} as an $N\mina$set by $T$.
\end {Definition}

The addition $M\times M\rto M$, $(x,y)\mto x+y$, on a binoid $M$ defines for every subbinoid $N\subseteq M$ an operation on $M$ by restricting the first component to $N$. In particular, $M$ is an $M\mina$binoid. In general, if $S$ is an $N\mina$set and $N^{\prime}\subseteq N$ a subbinoid, then $S$ is also an $N^{\prime}\mina$set. The zero binoid operates only on $S=\{p\}$, whereas the trivial binoid operates on every pointed set in the trivial way.

\begin {Remark}
Property (3) of the definition of an $N\mina$set says that $p$ is an invariant element under the operation of $N$. Of course, $p$ need not be unique with this property.

Arbitrary sets with an operation (also called action) of a semigroup or monoid $S$ are usually called $S\mina$acts. A thorough investigation of $S\mina$acts can be found in \cite{Kilp}. See also \cite{Ogus} for a treatment of the associated modules. However, when it comes to study properties and results of ring and module theory for $S\mina$acts, it is not uncommon to focus on $S\mina$acts that admit a unique invariant element and where $S$ contains an absorbing element, which resembles our setting of $N\mina$sets. In \cite{Ahmadi}, for instance, $S\mina$acts that satisfy some versions of Nakayama's lemma and Krull's (intersection) theorem are studied, which do not translate directly from rings to binoids either, cf.\ Remark \ref{RemNakayamaACC}.

In general, there is a close connection between operations of an object as defined above and its representations, see the subsequent lemma. Taking up this point of view, $S\mina$acts are studied by Clifford and Preston in \cite[Chapter 11]{CliffordPreston} and called operands over $S$. Right from the beginning they frequently point out that there is no loss of generality in dealing with acts that admit a unique invariant element (centered operands) over a semigroup $S$ with an absorbing element, which they do ``almost exclusively''.
\end {Remark}

\begin {Lemma}\label{LemNsetHom}
A pointed set $(S,p)$ is an $N\mina$set if and only if there is a binoid homomorphism
$$N\Rto(\map_{p}S,\circ,\id,\varphi_{\infty})\komma\quad a\lto\varphi_{a}\komma$$
where $\varphi_{a}:S\rto S$ denotes the translation $s\mto a+s$ by $a\in N$.
\end {Lemma}
\begin {proof}
This is easy to check.
\end {proof}

An $\N^{\infty}\mina$set is by Lemma \ref{LemNsetHom} the same as $S$ together with a fixed pointed map $\varphi:S\rto S$, the operation being given by $n+s=\varphi^{n}(s)$.

\begin {Lemma}\label{LemProductsNsets}
The product, \index{product!-- of $N\mina$sets}the direct sum, \index{direct sum!-- of $N\mina$sets}the smash product, \index{smash product!-- of $N\mina$sets}and the pointed union \index{pointed union!-- of $N\mina$sets}of a family of $N\mina$sets are again $N\mina$sets.
\end {Lemma}
\begin {proof}
Let $(S_{i},p_{i})_{i\in I}$ be a family of $N\mina$sets. The operation of $N$ on $\prod_{i\in I}S_{i}$ is given by $a+(s_{i})_{i\in I}=(a+s_{i})_{i\in I}$, which immediately yields the operation of $N$ on $\bigoplus_{i\in I}S_{i}$ and $\bigwedge_{i\in I}S_{i}$. The operation of $N$ on $\bigcupdot_{i\in I}S_{i}$ is given by $a+(s;i)=(a+s;i)$.
\end {proof}

\begin {Definition} \label{DefSmashOver}
Let $(S_{i},p_{i})_{i\in I}$ be a finite family of $N\mina$sets. If $\sim_{\wedge_{_{\!N}}}$\nomenclature[ACongruenceZZZSmashN]{$\sim_{\wedge_{_{N}}}$}{congruence on the smash product of a finite family of $N\mina$sets} denotes the equivalence relation on $\bigwedge_{i\in I}S_{i}$ generated by
$$\cdots\wedge(a+s_{i})\wedge\cdots\wedge s_{j}\wedge\cdots\quad\sim_{\wedge_{_{\!N}}}\quad\cdots\wedge s_{i}\wedge\cdots\wedge(a+s_{j})\wedge\cdots\komma$$
where $a\in N$ and $s_{k}\in S_{k}$, $k\in I$, then
$$\bigwedge_{i\in I} \!\!{}_{_{N}}S_{i}:=\Big(\bigwedge_{i\in I}S_{i}\Big)\Big/\!\sim_{\wedge_{_{\!N}}}$$
\nomenclature[AProductSmashN]{$\bigwedge_{N,i\in I}S_{i}$}{smash product over $N$ of a finite family of $N\mina$sets}is called the \gesperrt{smash product} of the family $(S_{i})_{i\in I}$ \gesperrt{over} \index{smash product!-- of $N\mina$binoids}$N$. A class $[\wedge_{i\in I}s_{i}]\in\bigwedge_{N,i\in I}S_{i}$ will be denoted by $\wedge_{N,i\in I}s_{i}$. Unless there is confusion, we will sometimes omit the index set and simply write $\bigwedge_{N}S_{i}$ and $\wedge_{N}s_{i}$.
\end {Definition}

\begin {Lemma} \label{LemSmaschPointedSets}
The smash product $\bigwedge_{N}S_{i}$ of a finite family $(S_{i},p_{i})_{i\in I}$ of $N\mina$sets is again an $N\mina$set with distinguished point $\wedge_{N}p_{i}=:p_{\wedge_{N}}$. If one component of $\wedge_{N}s_{i}$ equals $p_{i}$, then $\wedge_{N}s_{i}=p_{\wedge_{N}}$.
\end {Lemma}
\begin {proof}
Let $I=\{1\kpkt n\}$. It is clear that $(\bigwedge_{N}S_{i},p_{\wedge_{N}})$ is a pointed set and that the map $N\times\bigwedge_{N}S_{i}\rto\bigwedge_{N}S_{i}$ with $(a,\wedge_{N}s_{i})\mto(a+ s_{1})\wedge_{N}s_{2}\wedge_{N}\cdots\wedge_{N }s_{n}$ is a well-defined operation of $N$ on $\bigwedge_{N}S_{i}$. The supplement follows from 
\begin {align*}
p_{1}\wedge_{N}s_{2}\wedge_{N}\cdots\wedge_{N}s_{n}&=(p_{1}+\infty)\wedge_{N}s_{2}\wedge_{N}\cdots\wedge_{N}s_{n}\\
&=p_{1}\wedge_{N}(s_{2}+\infty)\wedge_{N}s_{3}\wedge_{N}\cdots\wedge_{N}s_{n}\\
&=p_{1}\wedge_{N}p_{2}\wedge_{N}s_{3}\wedge_{N}\cdots\wedge_{N}s_{n}
\end {align*}
and similar arguments.
\end {proof}

\begin {Proposition}\label{PropSmashPointed}
Let $(S_{i})_{i\in I}$ be a finite family of $N\mina$sets and $T$ another $N\mina$set. Then we have an isomorphism of $N\mina$sets:
$$T\wedge_{N}\Big(\bigcupdot_{i\in I}S_{i}\Big)\,\,\cong\,\,\bigcupdot_{i\in I}(T\wedge_{N}S_{i})\pkt$$
\end {Proposition}
\begin {proof}
The bijection is obviously given by $t\wedge_{N}(s;i)\leftrightarrow(t\wedge_{N}s;i)$ with $t\in T$ and $s\in S_{i}$, $i\in I$.
\end {proof}

\begin {Definition}
Let $(S_{i})_{i\in I}$ be a family of $N\mina$sets and $T$ another $N\mina$set. A pointed map $\psi:\prod_{i\in I}S_{i}\rto T$ is called an \emph{$N\mina$multi map} \index{N@$N\mina$multi!-- map}if
$$\psi(\ldots,s_{i-1},a+s_{i},s_{i+1},\dots)\,=\,a+\psi(\ldots,s_{i-1},s_{i},s_{i+1},\dots)$$
for all $a\in N$ and $s_{k}\in S_{k}$, $k\in I$.
\end {Definition}

\begin {Remark}\label{RemMultiProp}
Note that $\psi$ being an $N\mina$multi map implies the following two properties: for all $a\in N$ and $s_{k}\in S_{k}$, $k\in I$, one has
$$\psi(\ldots,a+s_{i}\kpkt s_{j},\dots)=\psi(\ldots,s_{i}\kpkt a+s_{j},\dots)\komma$$
and if one component of $(s_{i})_{i\in I}$ equals the distinguished point, say $s_{l}=p_{l}$, then $\psi((s_{i})_{i\in I})=\infty$ since
\begin {align*}
\psi((s_{i})_{i\in I})&=\psi(\ldots,s_{l-1},p_{l},s_{l+1},\ldots)\\
&=\psi(\ldots,s_{l-1},p_{l}+\infty,s_{l+1},\ldots)\\
&=\infty+\psi((s_{i})_{i\in I})\\
&=\infty\pkt
\end {align*}
\end{Remark}

\begin {Proposition} \label{PropUnivPropSmashN}
Let $(S_{i})_{i\in I}$ be a finite family of $N\mina$sets and $T$ another $N\mina$set. Every $N\mina$multi map $\psi:\prod_{i\in I}S_{i}\rto T$ gives rise to a unique $N\mina$map $\tilde{\psi}:\bigwedge_{N}S_{i}\rto T$ such that $\tilde{\psi}\pi=\psi$, where $\pi$ denotes the canonical projecion $\prod_{i\in I}S_{i}\rto\bigwedge_{N}S_{i}$.
\end {Proposition}
\begin {proof}
Let $I=\{1\kpkt n\}$. The map $\tilde{\psi}$ is a well-defined binoid homomorphism by Remark \ref{RemMultiProp}. Its uniqueness follows from $\tilde{\psi}\pi=\psi$ because $\pi$ is surjective. Moreover, an easy computation
\begin {align*}
\tilde{\psi}(a+\wedge_{N}s_{i})&=\tilde{\psi}((a+s_{1})\wedge_{N}s_{2}\wedge_{N}\cdots\wedge_{N}s_{n})\\
&=\psi(a+s_{1},s_{2},\ldots, s_{n})\\
&=a+\psi((s_{i})_{i\in I})\\
&=a+\tilde{\psi}(\wedge_{N}s_{i})
\end {align*}
proves that $\tilde{\psi}$ is an $N\mina$map.
\end {proof}

The preceding proposition yields the following bijection of sets:
$$\Big\{N\mina\text{maps }\,\bigwedge_{i\in I} \!\!{}_{_{N}}S_{i}\rto T\Big\}\Rto\Big\{N\mina\text{multi maps }\,\prod_{i\in I}S_{i}\rto T\Big\}\komma\quad\psi\lto\psi\pi\pkt$$

\begin {Definition}
Let $M$ be a binoid. If $\varphi:N\rto M$ is a binoid homomorphism with $\varphi(a)+x=x+\varphi(a)$ for all $a\in N$ and $x\in M$, then $M$ is called an \gesperrt{$N\mina$binoid} \index{N@$N\mina$binoid}with respect to the \gesperrt{structure homomorphism} \index{N@$N\mina$binoid!structure homomorphism of an --}\index{structure homomorphism!N@-- of an $N\mina$binoid}$\varphi$. An \gesperrt{$N\mina$binoid homomorphism} \index{N@$N\mina$binoid!-- homomorphism}\index{homomorphism!N@$N\mina$binoid --}is a binoid homomorphism $\psi:M\rto M^{\prime}$ of $N\mina$binoids such that $\psi\varphi=\varphi^{\prime}$, where $\varphi$ and $\varphi^{\prime}$ are the structure homomorphisms of $M$ and $M^{\prime}$, respectively. The set of all $N\mina$binoid homomorphisms $M\rto M^{\prime}$ will be denoted by $\hom_{N}(M,M^{\prime})$\nomenclature[HomN]{$\hom_{N}(M,M^{\prime})$}{set of all $N\mina$binoid homomorphisms $M\rto L$}. $N\mina$binoids together with $N\mina$binoid homomorphisms form a category $\Bsf_{N}$.\nomenclature[B]{$\Bsf_{N}$}{category of $N\mina$binoids}

Let $M$ be a commutative $N\mina$binoid via $\varphi:N\rto M$. We say that $M$ is \gesperrt{finitely generated} \index{N@$N\mina$binoid!finitely generated --} as $N\mina$binoid by $x_{1}\kpkt x_{r}$ if every element $f\in M$ can be written as $f=\varphi(a)+\sum_{i=1}^{r}n_{i}x_{i}$ for some $a\in N$ and $n_{i}\in\N$, $i\in\{1\kpkt r\}$.
\end {Definition}

Every $N\mina$binoid $M$ can be considered as an $N\mina$set in a natural way with respect to the operation given by 
$$+:N\times M\Rto M\komma\quad(a,x)\lto a+x:=\varphi(a)+x\pkt$$
Thus, an $N\mina$binoid is a binoid that is an $N\mina$set such that the operation of $N$ is compatible with the addition of the binoid. 
If $\sim$ is a congruence on an $N\mina$binoid $M$, then $M/\sim$ is again an $N\mina$binoid with respect to the structure homomorphism $\pi\varphi:N\rto M\rto M/\sim$.

\begin {Remark}
A commutative $N\mina$binoid that is finitely generated as $N\mina$set is also finitely generated as $N\mina$binoid, but the converse need not be true. For instance, $\N^{\infty}$ is finitely generated (by the element $1$) as $\trivial\mina$binoid but not as $\trivial\mina$set. In general, a binoid $M$ is finitely generated as $\trivial\mina$binoid if and only if it is finitely generated over every binoid $N$ that admits a binoid homomorphism $\varphi:N\rto M$. Equivalently, $M$ is a finitely generated binoid. 

Every commuative binoid $M$ that is finitely generated as $N\mina$binoid by $x_{1}\kpkt x_{r}$ gives rise to a canonical binoid epimorphism
$$N\wedge(\N^{r})^{\infty}\Rto M\komma\quad a\wedge(n_{1}\kpkt n_{r})\lto\varphi(a)+\sum_{i=1}^{r}n_{i}x_{i}\komma$$
where $\varphi:N\rto M$ denotes the structure homomorphism.
\end {Remark}

\begin {Corollary}
Let $(M_{i})_{i\in I}$ be a family of (positive) $N\mina$binoids. The product, the direct sum, the smash product, and the bipointed union of a family of (positive) $N\mina$binoids are $N\mina$binoids.
\end {Corollary}
\begin {proof}
This is an immediate consequence of Lemma \ref{LemProductsNsets} and the preceding observation.
\end {proof}

\begin {Lemma}
Let $M,M^{\prime}$, and $L$ be $N\mina$binoids. Every $N\mina$binoid homomorphism $\varphi:M\rto M^{\prime}$ induces a canonical map of sets
$$\hom_{N}(M^{\prime},L)\Rto\hom_{N}(M,L)\komma\quad\psi\lto\psi\varphi\pkt$$ 
\end {Lemma}
\begin {proof}
This is clear from the diagram
$$\xymatrix{
&N\ar[dl]\ar[d]\ar[rd]&\\
M\ar[r]^{\varphi}&M^{\prime}\ar[r]^{\psi}&L\pkt}$$
\end {proof}

\begin {Corollary}\nomenclature[AProductSmashN]{$\bigwedge_{N,i\in I}M_{i}$}{smash product over $N$ of a finite family of $N\mina$binoids}
Given a finite family $(M_{i})_{i\in I}$ of $N\mina$binoids, the equivalence relation $\sim_{\wedge_{_{\!N}}}$ on $\bigwedge_{i\in I}M_{i}$ is a congruence. In particular, the smash product of $(M_{i})_{i\in I}$ over $N$ is again an $N\mina$binoid with identity element  $0\wedge_{N}\cdots\wedge_{N}0=:0_{\wedge_{N}}$ and absorbing element $\infty\wedge_{N}\cdots\wedge_{N}\infty=:\infty_{\wedge_{N}}$. The canonical inclusions
$$\iota_{k}:M_{k}\Rto\bigwedge_{i\in I} \!\!{}_{_{N}}M_{i}\komma\quad x\lto 0\wedge_{N}\cdots\wedge_{N}0\wedge_{N} x\wedge_{N}0\wedge_{N}\cdots\wedge_{N}0\komma$$
where $x$ is the $k$th component, $k\in I$, are $N\mina$binoid homomorphisms.
\end {Corollary}
\begin {proof}
This is easily verified.
\end {proof}

Every binoid can be considered as a $\trivial\mina$binoid since $\trivial$ is an initial object in the category of binoids $\Bsf$. In this spirit, the smash product of a family of binoids from Section \ref{SecSmashProduct} is precisely the smash product over $\trivial$. In particular, one cannot expect better or more results for the smash product over an arbitrary binoid $N$ than for the smash product (over $\trivial$) as listed in Lemma \ref{LemSmashRules}. In fact, all these results are difficult to generalize to arbitrary $N\mina$binoids because the elements of a class $\wedge_{N}x_{i}$ depend on the structure homomorphisms $\varphi_{i}:N\rto M_{i}$, $i\in I$, see for instance Lemma \ref{LemSmashNoverN} and Example \ref{ExpNsmashNcancellative} below. More precisely, $\wedge_{N} t_{i}=\wedge_{N}x_{i}$ if $t_{i}=\sum_{j\in J}\pm\varphi_{i}(a_{j})+x_{i}$ with $a_{j}\in N$, $j\in J\subseteq I$, where the summands $\pm\varphi_{i}(a_{i})$ are distributed in a way such that they vanish after a \emph{suitable} not necessarily unique shifting. Nevertheless, generators of the smash product yield generators of the smash product over $N$ via the canonical projection $\bigwedge_{i\in I}M_{i}\rto\bigwedge_{N}M_{i}$.

\begin {Example}
Let $I$ be finite. If $(M_{i})_{i\in I}$ is a family of $N\mina$binoids with generating sets $A_{i}\subseteq M_{i}$, $i\in I$, then $\bigwedge_{N}M_{i}$ is generated by\nomenclature[EB]{$a\widehat{e}_{i,N}$}{$=0\wedge_{N}\cdots\wedge_{N}0\wedge_{N}a\wedge_{N}0\wedge_{N}\cdots\wedge_{N}0$, where $a$ is the $i$th entry}
$$a\widehat{e}_{i,N}:=0\wedge_{N}\cdots\wedge_{N}0\wedge_{N}a\wedge_{N}0\wedge_{N}\cdots\wedge_{N}0\komma$$
where $a\in A_{i}$ is the $i$th entry, $i\in I$. In particular, the smash product $\bigwedge_{N}\N^{\infty}$ of the family $(\N^{\infty})_{i\in I}$ is generated by\nomenclature[EB]{$\widehat{e}_{i,N}$}{$=0\wedge_{N}\cdots\wedge_{N}0\wedge_{N}1\wedge_{N}0\wedge_{N}\cdots\wedge_{N}0$, where $1$ is the $i$th entry}
$$\widehat{e}_{i,N}:=0\wedge_{N}\cdots\wedge_{N}0\wedge_{N}1\wedge_{N}0\wedge_{N}\cdots\wedge_{N}0\komma$$
where $1$ is the $i$th entry, $i\in I$.
\end {Example}

\begin {Lemma}\label{LemSmashNoverN}
If $M$ is an $N\mina$binoid, then $N\wedge_{N}M\cong M$.
\end {Lemma}
\begin {proof}
The natural operation $N\times M\rto M$ with $(a,x)\mto \varphi(a)+x$, where $\varphi:N\rto M$ denotes the structure homomorphism of $M$, induces an $N\mina$binoid homomorphism $N\wedge_{N}M\rto M$ with $a\wedge_{N}x\mto\varphi(a)+x$. The inverse is obviously the $N\mina$binoid homomorphism $M\rto N\wedge_{N}M$, $x\mto 0\wedge_{N}x$.
\end {proof}

The preceding result also shows that the smash product of $N\mina$binoids may have nice properties which not all components need to fulfill. An easy example is given by 
$$\Z^{\infty}\wedge_{\N^{\infty}}\N^{\infty}\,\,\cong\,\,\Z^{\infty}\komma$$
where all $a\wedge_{\N^{\infty}}b\not=\infty_{\wedge}$ are units though both components need not be so. 

\begin {Example} \label{ExpNsmashNcancellative}
Consider the $\N^{\infty}\mina$binoid $\N^{\infty}\wedge_{\N^{\infty}}\N^{\infty}$, where $\N^{\infty}$ is considered as an $\N^{\infty}\mina$binoid once (say on the left) via $\varphi_{k}:1\mto k$ and once via $\varphi_{l}:1\mto l$ for some $k,l\ge 1$, which means $$(a+k)\wedge_{\N^{\infty}}b\,=\,a\wedge_{\N^{\infty}}(b+l)$$
for all $a,b\in\N^{\infty}$. In particular, $nk\wedge_{\N^{\infty}} 0=0\wedge_{\N^{\infty}} nl$ for $n\ge 1$. For instance, $\N^{\infty}\wedge_{\N^{\infty}}\N^{\infty}\cong\N^{\infty}$ if $k=0$ or $l=0$. Clearly, $\N^{\infty}\wedge_{\N^{\infty}}\N^{\infty}$ is a positive $\N^{\infty}\mina$binoid. We claim that it is also cancellative. For this note that each element $a\wedge_{\N^{\infty}}b\in\N^{\infty}\wedge_{\N^{\infty}}\N^{\infty}$ can be written uniquely as $a^{\prime}\wedge_{\N^{\infty}}b^{\prime}$, where $b^{\prime}$ is determined by $b=b^{\prime}+ql$ with $b^{\prime}<l$, $q\in\N$, and $a^{\prime}=a+qk$. In virtue of Lemma \ref{LemIntCanGenerators}, it suffices to show that the generators $1\wedge_{\N^{\infty}} 0$ and $0\wedge_{\N^{\infty}} 1$ are cancellative elements in $\N^{\infty}\wedge_{\N^{\infty}}\N^{\infty}$. So suppose that
$$(a\wedge_{\N^{\infty}} b)+(1\wedge_{\N^{\infty}} 0)\,=\,(c\wedge_{\N^{\infty}} d)+(1\wedge_{\N^{\infty}} 0)\,\not=\,\infty_{\wedge_{\N^{\infty}}}$$
in $\N^{\infty}\wedge_{\N^{\infty}}\N^{\infty}$, where $a,b,c,d\in\N$. By the observation above, we may assume that $b,d<l$. Now the equation $(a+1)\wedge_{\N^{\infty}}b=(c+1)\wedge_{\N^{\infty}}d$ implies $a+1=c+1$ and $b=d$, which shows that $a\wedge_{\N^{\infty}}b=c\wedge_{\N^{\infty}}d$. The cancellativity of $0\wedge_{\N^{\infty}}1$ follows similarly. In particular, $1\wedge_{\N^{\infty}}0$ and $0\wedge_{\N^{\infty}}1$ is the unique minimal generating set of $\N^{\infty}\wedge_{\N^{\infty}}\N^{\infty}$ by Proposition \ref{PropUniqueMinSyst}. 

In Section \ref{SecFGbinoids}, we will show that $\N^{\infty}\wedge_{\N^{\infty}}\N^{\infty}$ is torsion-free if and only if $k$ and $l$ are coprime, cf.\ Lemma \ref{LemNwedgeNtorsionfree}.
\end {Example}

\begin {Definition}
Let $(M_{i})_{i\in I}$ be a finite family of $N\mina$binoids and $L$ another $N\mina$binoid. An $N\mina$multi map $\psi:\prod_{i\in I}M_{i}\rto L$ that is also a binoid homomorphism will be called an \gesperrt{$N\mina$multi binoid homomorphism}\index{N@$N\mina$multi!-- binoid homomorphism}\index{homomorphism!N@$N\mina$multi binoid --}. The set of all $N\mina$multi binoid homomorphisms will be denoted by $\mult_{N}(\prod_{i\in I}M_{i},L)$.\nomenclature[Mult]{$\mult_{N}(\prod_{i\in I}M_{i},L)$}{set of all $N\mina$multi binoid homomorphisms $\prod_{i\in I}M_{i}\rto L$}
\end {Definition}

Let $(M_{i})_{i\in I}$ be a finite family of $N\mina$binoids and $L$ another $N\mina$binoid with  structure homomorphisms denoted by $(\varphi_{i})_{i\in I}$ and $\varphi$, respectively. An $N\mina$multi binoid homomorphism $\psi:\prod_{i\in I}M_{i}\rto L$ is a binoid homomorphism that satisfies
$$\psi(\ldots,x_{i-1},\varphi_{i}(a)+x_{i},x_{i+1},\dots)=\varphi(a)+\psi(\ldots,x_{i-1},x_{i},x_{i+1},\dots)$$
for all $a\in N$ and $x_{j}\in M_{j}$, $j\in I$. 

\begin {Corollary} \label{CorUnivPropSmashN}
Let $(M_{i})_{i\in I}$ be a finite family of $N\mina$binoids and $L$ another $N\mina$binoid. Every $N\mina$multi binoid homomorphism $\psi:\prod_{i\in I}M_{i}\rto L$ gives rise to a unique $N\mina$binoid homomorphism $\tilde{\psi}:\bigwedge_{N}M_{i}\rto L$ with $\tilde{\psi}\pi=\psi$, where $\pi:\prod_{i\in I}M_{i}\rto\bigwedge_{N}M_{i}$ denotes the canonical projection.
\end {Corollary}
\begin {proof}
The existence of a unique $N\mina$multi map $\tilde{\psi}:\bigwedge_{N}M_{i}\rto L$ with $\tilde{\psi}(\wedge_{N}x_{i})=\psi((x_{i})_{i\in I})$ follows from Proposition \ref{PropUnivPropSmashN}. Since 
\begin {align*}
\tilde{\psi}(\wedge_{N}x_{i}+\wedge_{N}y_{i})&=\tilde{\psi}(\pi((x_{i}+y_{i})_{i\in I}))\\
&=\psi((x_{i}+y_{i})_{i\in I})\\
&=\psi((x_{i})_{i\in I})+\psi((y_{i})_{i\in I})\\
&=\tilde{\psi}(\wedge_{N}x_{i})+\tilde{\psi}(\wedge_{N}y_{i})
\end {align*}
for all $\wedge_{N}x_{i}, \wedge_{N}y_{i}\in\bigwedge_{N}M_{i}$, the map $\tilde{\psi}$ is an $N\mina$binoid homomorphism.
\end {proof}

\begin {Corollary}
Let $(M_{i})_{i\in I}$ be a finite family of $N\mina$binoids and $L$ another $N\mina$binoid. Then
$$\hom_{N}\Big(\bigwedge_{i\in I} \!\!{}_{_{N}}M_{i},L\Big)\,\,\cong\,\,\mult_{N}\Big(\prod_{i\in I}M_{i},L\Big)\pkt$$
as sets.
\end {Corollary}
\begin {proof}
As a consequence of Corollary \ref{CorUnivPropSmashN}, the map $\hom_{N}(\bigwedge_{N}M_{i},L)\rto\mult_{N}(\prod_{i\in I}M_{i},L)$ with $\varphi\mto\varphi\pi$, where $\pi:\prod_{i\in I}M_{i}\rto\bigwedge_{N}M_{i}$, is bijective.
\end {proof}

The smash product over $N$ is a universal object, namely, it is the (finite) coproduct in the category of $N\mina$binoids $\Bsf_{N}$. This is contained in the following statement.

\begin {Corollary} \label{CorCoProdSmashN}
Let $(M_{i})_{i\in I}$ be a finite family of $N\mina$binoids and $L$ a commutative $N\mina$binoid. Every family $\psi_{i}:M_{i}\rto L$, $i\in I$, of $N\mina$binoid homomorphisms gives rise to a unique $N\mina$binoid homomorphism
$$\psi:\bigwedge_{i\in I} \!\!{}_{_{N}}M_{i}\Rto L\quad\text{with}\quad\wedge_{N}x_{i}\longmapsto\sum_{i\in I}\psi_{i}(x_{i})\pkt$$
In particular, if $M$ is another binoid and $\varphi:M\rto N$ a binoid homomorphism, then the induced $M\mina$binoid homomorphism 
$$\bigwedge_{i\in I} \!\!{}_{_{M}}M_{i}\Rto\bigwedge_{i\in I} \!\!{}_{_{N}}M_{i}\quad\text{with}\quad\wedge_{M}x_{i}\lto\wedge_{N}\;x_{i}$$
is surjective.
\end {Corollary}
\begin {proof}
The map $\tilde{\psi}:\prod_{i\in I}M_{i}\rto L$ with $(x_{i})_{i\in I}\mto\sum_{i\in I}\psi_{i}(x_{i})$ is an $N\mina$binoid homomorphism because all $\psi_{i}$, $i\in I$, are so and $L$ is commutative. If the structure homomorphisms are given by $\varphi:N\rto L$ and $\varphi_{i}:N\rto M_{i}$, $i\in I$, then $\psi_{k}\varphi_{k}(a)=\varphi(a)$ for all $a\in N$ and $k\in I$, which yields
\begin {align*}
\tilde{\psi}(\ldots,x_{k-1},\varphi_{k}(a)+x_{k}, x_{k+1},\ldots)&=\psi_{k}(\varphi_{k}(a)+x_{k})+\sum_{i\not=k}\psi_{i}(x_{i})\\
&=\varphi(a)+\sum_{i\in I}\psi_{i}(x_{i})\\
&=\varphi(a)+\tilde{\psi}((x_{i})_{i\in I})\pkt
\end {align*}
Hence, $\tilde{\psi}$ is an $N\mina$multi binoid homomorphism so that we can apply Corollary \ref{CorUnivPropSmashN} to obtain the unique $N\mina$binoid homomorphism $\psi:\bigwedge_{N}M_{i}\rto L$ with $\psi(\wedge_{N}x_{i})=\tilde{\psi}((x_{i})_{i\in I})=\sum_{i\in I}\psi_{i}(x_{i})$.

The supplement is just a special case of the first part with $(M_{i})_{i\in I}$ considered as a family of $M\mina$binoids via $\tilde{\varphi}_{i}:=\varphi_{i}\varphi:M\rto M_{i}$ for all $i\in I$, $L=\bigwedge_{N}M_{i}$, and $\psi_{i}$ the natural embeddings $M_{i}\embto\bigwedge_{N}M_{i}$ such that the statement follows from  Lemma \ref{LemIndCong} since $\sim_{\wedge_{_{\!M}}}\le\,\sim_{\wedge_{_{\!N}}}$.
\end {proof}

\begin {Example}\label{ExpNSmashNMap}
Consider again the binoid $\N^{\infty}\wedge_{\infty}\N^{\infty}$ discussed in Example \ref{ExpNsmashNcancellative} and the binoid homomorphisms $\psi_{i}:\N^{\infty}\rto\N^{\infty}$ with $1\mto i/\gcd(k,l)$, $i\in\{k,l\}$. In this situation, there is a commutative diagram
$$\xymatrix{
\N^{\infty}\ar[r]^{\varphi_{k}}\ar[d]_{\varphi_{l}}&\N^{\infty}\ar[d]^{\iota_{\text{r}}}\ar@/^/[ddr]^{\psi_{k}}\\
\N^{\infty}\ar[r]^{\!\!\!\!\!\!\!\!\!\!\!\!\!\iota_{\text{l}}}\ar@/_/[rrd]_{\psi_{l}}&\N^{\infty}\wedge_{\N^{\infty}}\N^{\infty}\ar[dr]^{\psi}\\
&&\N^{\infty}}$$
with $\psi(n\wedge_{\N^{\infty}}m)=\frac{nl+mk}{\gcd(k,l)}$.
\end {Example}

\begin {Proposition} \label{PropLSpecNSmash}
Let $(M_{i})_{i\in I}$ be a finite family of $N\mina$binoids with structure homomorphisms $(\varphi_{i})_{i\in I}$. If $L$ is a commutative binoid, then
$$L\minspec\bigwedge_{i\in I} \!\!{}_{_{N}}M_{i}\,\,\cong\,\,\Big\{(\psi_{i})_{i\in I}\in\prod_{i\in I}L\minspec M_{i}\,\,\Big|\,\,\psi_{i}\varphi_{i}=\psi_{j}\varphi_{j}, i,j\in I\Big\}$$
as semigroups
\end {Proposition}
\begin {proof}
This is just a restatement of Corollary \ref{CorCoProdSmashN}.

\end {proof}

\begin {Corollary} \label{CorSmashNHoms}
Let $M_{i}$ and $L_{i}$ be $N\mina$binoids, $i\in I$. Every family $\psi_{i}:M_{i}\rto L_{i}$, $i\in I$, of $N\mina$binoid homomorphisms gives rise to a unique $N\mina$binoid homomorphism
$$\psi:=\wedge_{N}\psi_{i}:\bigwedge_{i\in I} \!\!{}_{_{N}}M_{i}\Rto \bigwedge_{i\in I} \!\!{}_{_{N}}L_{i}\quad\text{with}\quad \wedge_{N} a_{i}\lto\wedge_{N}\psi_{i}(a_{i})\pkt$$ 
\end {Corollary}
\begin {proof}
Apply Corollary \ref{CorCoProdSmashN} to the $N\mina$binoid homomorphisms $M_{i}\stackrel{\psi_{i}}{\rto}L_{i}\stackrel{\!\!\iota_{i}}{\rto}\bigwedge_{N}L_{i}$, $i\in I$.
\end {proof}

\bigskip

\section {Direct and projective limits} \label{SecLimits}
\markright{\ref{SecLimits} Direct and projective limits}

In this section, the existence of the limit and colimit in the category of commutative $N\mina$binoids will be proved. These can be described explicitly, but due to a lack of theory so far interesting examples for so-called direct limits (limits) and projective limits (colimits) will be discussed in subsequent chapters. Furthermore, we introduce the strong projective limit for which we also give an example.

\begin {Convention}
In this section, arbitrary binoids are assumed to be \emph{commutative}.
\end {Convention}

\begin {Definition}
Let $(M_{i})_{i\in I}$ be a family of $N\mina$binoids and $(I,\ge)$ a \gesperrt{directed set}\index{directed!-- set}; that is, $I$ is partially ordered with respect to $\ge$ and for all $i,j\in I$ there is a $k\in I$ with $k\ge i,j$. Assume that for every pair $i,j\in I$ with $i\ge j$, there is an $N\mina$binoid homomorphism $\varphi_{ij}:M_{j}\rto M_{i}$ such that

\begin {ListeTheorem}
\item $\varphi_{ii}=\id_{M_{i}}$ for all $i\in I$.
\item $\varphi_{kj}=\varphi_{ki}\varphi_{ij}$ for all $k\ge i\ge j$.
\end {ListeTheorem}

Then $(M_{i},(\varphi_{ij})_{i\ge j})_{i,j\in I}$ is called a \gesperrt{directed system} \index{system!directed --}\index{directed!-- system}of $N\mina$binoids.
\end {Definition}

\begin {Lemma}\label{LemDirectLimit}
If $(M_{i},(\varphi_{ij})_{i\ge j})_{i,j\in I}$ is a directed system of $N\mina$binoids, then
$$D_{I}\,:=\,\Big(\biguplus_{i\in I}M_{i}\Big)\Big/\!\sim\komma$$
where $\sim$ denotes the equivalence relation on the disjoint union given by
$$(a;i)\,\sim\,(b;j)\quad:\Leftrightarrow\quad\varphi_{ki}(a)=\varphi_{kj}(b)\text{ for some }k\ge i,j\komma$$ 
is an $N\mina$binoid with respect to the addition defined by
$$[(a;i)]+[(b;j)]\,:=\,[(\varphi_{ki}(a)+\varphi_{kj}(b);k)]\komma$$
for some $k\ge i,j$. 
Moreover, there is a canonical $N\mina$binoid homomorphism $\varphi_{i}:M_{i}\rto D_{I}$, $ a\mto[(a;i)]$, for every $i\in I$ such that $\varphi_{i}=\varphi_{j}\varphi_{ji}$ for all $j\ge i$; in other words, the diagram
$$\xymatrix{
M_{i}\ar[dd]_{\varphi_{ji}}\ar[dr]^{\varphi_{i}}&\\
&D_{I}\\
M_{j}\ar[ur]_{\varphi_{j}}&}$$
commutes.
\end {Lemma}
\begin {proof}
Only the transivity of $\sim$ is not trivial. If $(a;i)\,\sim\,(b;j)$ and $(b;j)\,\sim\,(c;r)$, there are $k,l\in I$ with $k\ge i,j$ and $l\ge j,r$ such that $\varphi_{ki}(a)=\varphi_{kj}(b)$ and $\varphi_{lj}(b)=\varphi_{lr}(c)$.
Since $I$ is a directed set, we have an $s\in I$ with $s\ge k,l$, which in particular satisfies $s\ge i,j,r$. So we obtain 
$$\varphi_{si}(a)=\varphi_{sk}(\varphi_{ki}(a))=\varphi_{sk}(\varphi_{kj}(b))=\varphi_{sj}(b)=\varphi_{sl}(\varphi_{lj}(b))=\varphi_{sl}(\varphi_{lr}(c))=\varphi_{sr}(c)\pkt$$
Hence, $(a;i)\,\sim\,(c;r)$. To show that the addition on $D_{I}$ is well-defined, we have to verify that 
$$(\varphi_{ki}(a)+\varphi_{kj}(b);k)\,\sim\,(\varphi_{li}(a)+\varphi_{lj}(b);l)$$
whenever $k\ge i,j$ and $l\ge i,j$. Again, we have an $s\in I$ with $s\ge k,l$, which in particular satisfies $s\ge i,j$. Therefore,
\begin {align*}
\varphi_{sk}(\varphi_{ki}(a)+\varphi_{kj}(b))&=\varphi_{sk}(\varphi_{ki}(a))+\varphi_{sk}(\varphi_{kj}(b))\\
&=\varphi_{si}(a)+\varphi_{sj}(b)\\
&=\varphi_{sl}(\varphi_{li}(a))+\varphi_{sl}(\varphi_{lj}(b))\\
&=\varphi_{sl}(\varphi_{li}(a)+\varphi_{lj}(b))\pkt
\end {align*}
Since all $\varphi_{ij}$, $i\ge j$, are $N\mina$binoid homomorphisms, the identity elements $(0;i)$ of the $M_{i}$ are glued together under $\sim$ and so are their absorbing elements $(\infty;i)$, $i\in I$. Their equivalence classes $[(\infty;i)]$ and $[(0;i)]$ serve as an absorbing and as an identity element of $D_{I}$, respectively. The supplement is clear.
\end {proof}

The $N\mina$binoid $D_{I}$ together with the $N\mina$binoid homomorphisms $\varphi_{i}:M_{i}\rto D_{I}$, $i\in I$, has the following universal property.

\begin {Lemma}\label{LemUnivPropColimit}
Let $(M_{i},(\varphi_{ij})_{i\ge j})_{i,j\in I}$ be a directed system of $N\mina$binoids and $(D_{I}, (\varphi_{i})_{i\in I})$ as in Lemma \ref{LemDirectLimit}. Given another $N\mina$binoid $M^{\prime}$ and a family $\psi_{i}:M_{i}\rto M^{\prime}$ of $N\mina$binoid homomorphisms such that $\psi_{i}=\psi_{j}\varphi_{ji}$ for all $j\ge i$, there exists a unique $N\mina$binoid homomorphism $\psi:D_{I}\rto M^{\prime}$ such that the diagram
$$\xymatrix{
M_{i}\ar[dd]_{\varphi_{ji}}\ar[dr]_{\varphi_{i}}\ar@/^/[drr]^{\psi_{i}}&&\\
&D_{I}\ar[r]^{\!\!\!\psi}&M^{\prime}\\
M_{j}\ar[ur]^{\varphi_{j}}\ar@/_/[urr]_{\psi_{j}}&&}$$
commutes for all $j\ge i$.
\end {Lemma}
\begin {proof}
It is immediately checked that the map $\psi:D_{I}\rto M^{\prime}$ with $\psi([(a;k)]):=\psi_{k}(a)$ is the unique $N\mina$binoid homomorphism. 
\end {proof}

This universal property shows that $D_{I}$ is the \gesperrt{colimit} \index{colimit}in the category of commutative $N\mina$binoids, which one also calls the \gesperrt{direct} \index{direct limit}or \gesperrt{inductive limit} \index{inductive limit}\index{limit!direct/inductive --}and denotes it by\nomenclature[Limes]{$\varinjlim M_{i}$}{direct/inductive limit} $\varinjlim M_{i}$,
cf.\ \cite[Chapter 2.6]{Borceux} or \cite[Chapter 5.6]{Awodey}.

\begin {Remark}
Though differently defined our construction of $\varinjlim M_{i}$ coincides (of course) in case $N=\trivial$ with the one given in \cite[Proposition 2.1]{ChuLorscheidSanthanam}. Note, for instance, that $(a;i)\,\sim\,(\varphi_{ki}(a);k)$ for all $k\ge i$ because $\varphi_{li}(a)=\varphi_{lk}(\varphi_{ki}(a))$ whenever $l\ge k$.
\end {Remark}

An example of a directed system of $N\mina$binoids and its direct limit is given in Example \ref{ExpDirectLimitLocalizations}.

\begin {Definition}
Let $(M_{i})_{i\in I}$ be a family of $N\mina$binoids and $(I,\ge)$ a partially ordered set with respect to $\ge$. Assume that for every pair $i,j\in I$ with $i\ge j$, there is an $N\mina$binoid homomorphisms $\varphi_{ji}:M_{i}\rto M_{j}$ such that

\begin {ListeTheorem}
\item $\varphi_{ii}=\id_{M_{i}}$ for all $i\in I$.
\item $\varphi_{jk}=\varphi_{ji}\varphi_{ik}$ for all $k\ge i\ge j$.
\end {ListeTheorem}

Then $(M_{i},(\varphi_{ij})_{i\ge j})_{i,j\in I}$ is called an \gesperrt{inverse system} \index{system!inverse --}of $N\mina$binoids.
\end {Definition}

\begin {Lemma}\label{LemProjectiveLimit}
If $(M_{i},(\varphi_{ji})_{i\ge j})_{i,j\in I}$ is an inverse system of $N\mina$binoids, then
$$P_{I}\,:=\,\Big\{(a_{i})_{i\in I}\in\prod_{i\in I}M_{i}\,\,\Big|\,\,\varphi_{ji}(a_{i})=a_{j}\text{ for all }i\ge j\Big\}$$
is an $N\mina$subbinoid of $\prod_{i\in I}M_{i}$. Moreover, for every $i\in I$ there is a canonical $N\mina$binoid homomorphism 
$$\varphi_{i}:P_{I}\Rto M_{i}\komma\quad (a_{j})_{j\in I}\lto a_{i}\komma$$
such that $\varphi_{i}=\varphi_{ij}\varphi_{j}$ for all $j\ge i$; in other words, the diagram
$$
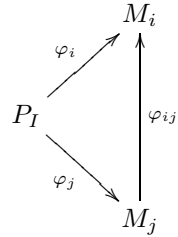
\xymatrix{
&M_{i}\\
P_{I}\ar[dr]_{\varphi_{j}}\ar[ur]^{\varphi_{i}}&\\
&M_{j}\ar[uu]_{\varphi_{ij}}}$$
commutes.
\end {Lemma}
\begin {proof}
Both assertions follow immediately from the fact that the $\varphi_{ij}$, $j\ge i$, are $N\mina$binoid homomorphisms.
\end {proof}

The binoid $P_{I}$ together with the $N\mina$binoid homomorphisms $\varphi_{i}:P_{I}\rto M_{i}$, $i\in I$, has the following universal property.

\begin {Lemma}
Let $(M_{i},(\varphi_{ji})_{i\ge j})_{i,j\in I}$ be an inverse system of $N\mina$binoids and $(P_{I}, (\varphi_{i})_{i\in I})$ as in Lemma \ref{LemProjectiveLimit}. Given another $N\mina$binoid $M^{\prime}$ and a family $\psi_{i}:M^{\prime}\rto M_{i}$ of $N\mina$binoid homomorphisms such that $\psi_{i}=\varphi_{ij}\psi_{j}$ for all $j\ge i$, there exists a unique binoid homomorphism $\psi:M^{\prime}\rto P_{I}$ such that the diagram
$$
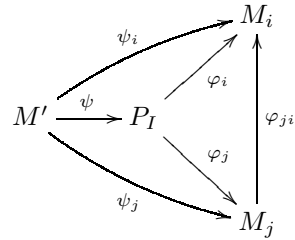
\xymatrix{
&&M_{i}\\
M^{\prime}\ar@/_/[drr]_{\,\psi_{j}}\ar@/^/[urr]^{\psi_{i}}\ar[r]^{\psi}&P_{I}\ar[dr]^{\varphi_{j}}\ar[ur]_{\varphi_{i}}&\\
&&M_{j}\ar[uu]_{\varphi_{ji}}}$$
commutes for all $j\ge i$.
\end {Lemma}
\begin {proof}
It is immediately checked that the map $\psi:M^{\prime}\rto P_{I}$ with $\psi(a):=(\psi_{i}(a))_{i\in I}$ is the unique $N\mina$binoid homomorphism.
\end {proof}

This universal property shows that $P_{I}$ is the \gesperrt{limit} \index{limit}in the category of commutative binoids, which one also calls the \gesperrt{projective} \index{projective limit}\index{limit!inverse/projective --}or \gesperrt{inverse limit} \index{inverse!-- limit}and denotes it by\nomenclature[Limes]{$\varprojlim M_{i}$}{projective/inverse limit} $\varprojlim M_{i}$,
cf.\ \cite[Chapter 2.6]{Borceux} or \cite[Chapter 5.4]{Awodey}.

\begin {Example} \label{ExpProjLimTrivial}
Let $I=\{1\kpkt n\}$ and $M_{0}:=\trivial$. Every family $(M_{i})_{i\in I}$ of binoids yields an inverse system of $\trivial\mina$binoids 
$$\big((M_{0},M_{i}),\, (\chi_{M_{i}\okreuz}, \id_{M_{i}}, \id_{\trivial})\big)_{i\in I}$$
with
$$\varprojlim M_{i}\,=\,\Big(\prod_{i=1}^{n}(M_{i})\Uplus\Big)\cup\Big(\prod_{i=1}^{n}M_{i}\okreuz\Big)\pkt$$
If for instance $M_{i}=\N^{\infty}$ for all $i\in I\setminus\{0\}$, then $\varprojlim M_{i}\cong(\N_{\ge1}^{\infty})^{n}\cup\{(0)_{i\in I}\}$, which is not a finitely generated binoid for $n\ge 2$.
\end {Example}

More examples of inverse systems of ($\trivial\mina$) binoids and their projective limits are given in Example \ref {ExpProjLim} and Remark \ref{RemCompletion}.

\begin {Definition}
Let $(M_{i},(\varphi_{ji})_{i\ge j})_{i,j\in I}$ be an inverse system of $N\mina$binoids. An element $(a_{i})_{i\in I}\in\varprojlim M_{i}$ is called \gesperrt{strongly compatible} \index{stronlgy compatible}\index{element!strongly compatible --}if for any two entries $a_{i}$ and $a_{j}$ which are $\not=\infty$ there exists a $k\in I$ with $k\le i,j$ such that $\varphi_{ki}(a_{i})=\varphi_{kj}(a_{j})\not=\infty$. Denote by $\slim M_{i}$ the subbinoid of $\varprojlim M_{i}$ generated by all stronlgy compatible elements. We call $\slim M_{i}$ the \gesperrt{strong projective limit} \index{projective limit!strong --}\index{limit!strong projective --}\index{strong projective limit}of the inverse system $(M_{i},(\varphi_{ji})_{i\ge j})_{i,j\in I}$.\nomenclature[slimes]{$\slim M_{i}$}{strong projective limit}
\end {Definition}

\begin {Example}
The strongly compatible elements of the projective limit considered in Example \ref{ExpProjLimTrivial} above are given by the elements of $\prod_{i=0}^{n}M_{i}\okreuz$ and those elements of $\prod_{i=0}^{n}(M_{i})\Uplus$ that are of the form $(\infty,a_{i})_{i\in I}$ with at most one $a_{i}\not=\infty$. Since the addition of such elements is again an element of this kind, we obtain
$$\slim M_{i}\,\cong\,\Big(\prod_{i\in I}M_{i}\okreuz\Big)\cup\Big\{(a_{i})_{i\in I}\in\prod_{i\in I}(M_{i})\Uplus\,\Big|\, a_{i}\not=\infty\text{ for at most one }i\in I\Big\}\pkt$$
The special case of all $M_{i}$ positive yields
$$\slim M_{i}\,\cong\,\,\,\bigcupbidot_{i\in I}M_{i}\pkt$$
\end {Example}

In Lemma \ref{CorStrongProjLimReduzierung}, we are able to describe $M_{\opred}$ of a finitely generated binoid $M$ by means of the strong projective limit of the inverse system defined by its integral quotients.

\bigskip

\section {Finitely generated commutative binoids} \label {SecFGbinoids}
\markright{\ref{SecFGbinoids} Finitely generated commutative binoids}

Finitely generated binoids were introduced in Section \ref {Sec1DefProp}. There, cf.\ Proposition \ref{PropUniqueMinSyst}, we could show that a finitely generated commutative binoid $M$ always admits a unique minimal generating set, namely $M\Uplus\setminus 2M\Uplus$, if it is positive and cancellative. Moreover, properties of a finitely generated binoid usually only depend on the properties of the generators and the relations among them, which is a finite data by R\'edei's Theorem if $M$ is commutative, cf.\ Corollary \ref{CorRedei}. This makes finitely generated commutative binoids comparatively easy to describe and study. They will provide us with interesting examples especially when we study their associated algebras later. For this purpose, all one- and, to a certain extent, two-generated commutative binoids with their binoid algebras will be classified in this section.

\begin {Convention}
In this section, arbitrary binoids are assumed to be \emph{commutative} and, if not otherwise stated, $K$ denotes a ring.
\end {Convention}

From now on, generator will always mean binoid generator. For instance, $\Z^{\infty}$ is a two-generated binoid (group) with generators $1$ and $-1$, none can be dropped and there is no generating set with less than two elements. In general, $M$ is an $r$-generated \emph{integral} binoid if and only if $M\opkt$ is an $r$-generated monoid. Basics on finitely generated commutative monoids can be found in \cite[Chapter VI]{GrilletCS}, \cite[\S9.3]{CliffordPreston}, and in the monograph \cite{GarciaRosales}.

\medskip

Recall, cf.\ Example \ref {ExGen}, that a finitely generated binoid $M$ with generators $x_{1}\kpkt x_{r}$ is isomorphic to $(\N^{r})^{\infty}/\!\sim_{\varepsilon}$, where 
$$\varepsilon:(\N^{r})^{\infty}\Rto M\quad\text{with}\quad\varepsilon(e_{i})=x_{i}\komma$$
$i\in\{1\kpkt r\}$, is the canonical binoid epimorphism.
Now by R\'edei's Theorem, cf.\ Corollary \ref{CorRedei}, the congruence $\sim_{\varepsilon}$ can be generated by finitely many relations $\Rcal:\varepsilon(a)=\varepsilon(b)$, $a,b\in(\N^{r})^{\infty}$. Thus, $M$ is given by finitely many generators and relations, which we denote by 
$$\free(x_{1}\kpkt x_{r})/(\Rcal_{1}\kpkt \Rcal_{n})\quad\text{or}\quad(\N^{r})^{\infty}/(\Rcal_{1}\kpkt \Rcal_{n})\pkt$$
Of course, what kind of relations occur for a specific binoid depends on the chosen generating set. In general, $M/(\Rcal_{1}\kpkt \Rcal_{n})$ is the binoid that arises from the binoid $M$ by taking the additional relations $\Rcal_{1}\kpkt \Rcal_{n}$ into account.

The binoid algebra of a finitely generated binoid $\free(x_{1}\kpkt x_{r})/(\Rcal_{1}\kpkt \Rcal_{n})$ is $K[X_{1}\kpkt X_{r}]/\Iideal$, where $\Iideal$ is an ideal generated by monomials and/or binomials depending on the relations $\Rcal_{1}\kpkt \Rcal_{n}$. This observation suggests the following definition.

\begin {Definition}\label{DefMonomialBinomial}
With the notation from above, a relation on an $r$-generated binoid $M$ of the form $\Rcal:\varepsilon(x)=\infty$, $x\in\free_{r}$, is called a \gesperrt{monomial relation}\index{relation!monomial --}. A \gesperrt{binomial relation} \index{relation!binomial --}is a relation of the form $\Rcal:\varepsilon(x)=\varepsilon(y)\not=\infty$, $x,y\in\free_{r}$. In addition, we say $\Rcal$ is an \gesperrt{unmixed relation} \index{relation!unmixed --}if it is a monomial relation $\varepsilon(x)=\infty$ with $\#\supp(x)=1$, $x\in\free_{r}$, or a binomial relation $\varepsilon(x)=\varepsilon(y)\not=\infty$ such that $\supp(x)=\supp(y)$ is a singleton, $x,y\in\free_{r}$. Otherwise, $\Rcal$ is a  \gesperrt{mixed relation}.
\end {Definition}

The relation $\varepsilon(x)=\varepsilon(y)=0$ is a binomial relation. A relation like $\varepsilon(x)=\varepsilon(y)=\infty$ will be considered as two monomial relations, namely $\varepsilon(x)=\infty$ and $\varepsilon(y)=\infty$.

\begin {Remark} \label{RemICongMonRel}
A congruence on a finitely generated binoid is an ideal congruence if and only if it is generated by monomial relations.
\end {Remark}

Loosely speaking, a mixed relation is a relation where more than one generator is involved. We sometimes indicate this by writing $\Rcal^{(s)}$ when $s$ is the number of generators that appear.\nomenclature[R]{$\Rcal^{(s)}$}{relation on an $s$-generated binoid} On one-generated binoids there are no mixed relations. Precisely, one has three different kinds of (unmixed) relations which may appear, namely\nomenclature[R]{$\Rcal^{(1)}_{i}$}{, $i=1,2,3$, relations on a one-generated binoid}

\begin {ListeTheorem}
\item [$\Rcal^{(1)}_{1}:$] $nx=0$ for some $n\ge2$,
\item [$\Rcal^{(1)}_{2}:$] $nx=\infty$ for some $n\ge2$,
\item [$\Rcal^{(1)}_{3}:$] $rx=sx\not\in\trivial$ for some $1\le r<s$,
\end {ListeTheorem}

where $x$ denotes the generator. Note that $\Rcal^{(1)}_{1}$ and $\Rcal^{(1)}_{2}$ are mutually exclusive (else $\langle x\rangle$ is the zero binoid and not one-generated).

In this spirit, we say that there are no mixed \gesperrt{relations among} a finite family of elements $y_{i}\in M$, $i\in\{1\kpkt s\}$, if the relations defining the subbinoid generated by this family are unmixed. For another description of this property consider the induced binoid homomorphism, cf.\ Proposition \ref{PropCoproductComBinoid} (with $M_{i}=\langle y_{i}\rangle\embto M$),
$$\bigwedge_{i=1}^{r}\langle y_{i}\rangle\Rto M\komma\quad n_{1}y_{1}\wedge\cdots\wedge n_{r}y_{r}\lto\sum_{i=1}^{r}n_{i}y_{i}\pkt$$

\begin {Lemma} \label{LemNoRelations}
Let $M$ be a finitely generated binoid and $y_{1}\kpkt y_{r}\in M$. The binoid homomorphism
$\varphi:\bigwedge_{i=1}^{r}\langle y_{i}\rangle\rto M$, $\varphi(\wedge_{i=1}^{r}n_{i}y_{i})=\sum_{i=1}^{r}n_{i}y_{i}$, is surjective if and only if $y_{1}\kpkt y_{r}$ generate $M$, and injective if and only if there are no mixed relations among the elements $y_{i}$, $i\in I$. In particular, if $M$ is finitely generated with no mixed relations among its generators $y_{1}\kpkt y_{r}$, then $M\cong\langle y_{1}\rangle\wedge\cdots\wedge\langle y_{r}\rangle$.
\end {Lemma}
\begin {proof}
The statement concerning the injectivity follows from the definition of a mixed relation. For the surjectivity consider the following composition of binoid homomorphisms
$$\varphi\psi:(\N^{r})^{\infty}\Rto\bigwedge_{i=1}^{r}\langle y_{i}\rangle\Rto M\komma$$
where 
$$\psi:=\wedge_{i=1}^{r}\psi_{i}:(\N^{r})^{\infty}\Rto\bigwedge_{i=1}^{r}\langle y_{i}\rangle$$
with $\psi_{i}:\N^{\infty}\rto\langle y_{i}\rangle$, $1\mto y_{i}$, cf.\ Corollary \ref{CorSmashNHoms} (with $N=\trivial$). Thus, $\varphi\psi=\varepsilon:(\N^{r})^{\infty}\rto M$ with $e_{i}\mto y_{i}$, $i\in\{1\kpkt r\}$, and $\varphi$ is surjective if and only if $\varepsilon$ is surjective. The supplement is clear.
\end {proof}

Since the smash product realizes the tensor product for the binoid algebras, cf.\ Corollary \ref{CorSmash=Tensor}, we have $K[\langle x_{1}\rangle\wedge\cdots\wedge\langle x_{r}\rangle]=K[\langle x_{1}\rangle]\otimes_{K}\cdots\otimes_{K} K[\langle x_{\ell}\rangle]$.

\medskip

In order to describe all finitely generated binoids that admit a generating set with no mixed relations among the generators, it suffices by Lemma \ref{LemNoRelations} to determine all one-generated binoids up to isomorphism, what we want to do now.

\begin {Definition}
Let $M$ be a binoid. An element  $x\in M$ has a \gesperrt{loop} \index{loop}if $nx=mx$ for some $n\not=m$. Otherwise $x$ is \gesperrt{loopfree}\index{loop!--free}\index{element!loopfree --}. In case $x$ has a loop, there is a smallest integer $s\ge 1$ such that $rx=sx$ for unique $0\le r<s$. Then $r<s$ is called the \gesperrt{initial pair} \index{initial pair}and $s-r$ the \gesperrt{length} \index{loop!length of a --}of the loop.
\end {Definition}

If $x\in M$ has a loop with inital pair $r<s$, then $\langle x\rangle\cong\N^{\infty}/(r=s)$.

\begin {Example} \label{ExLoop}
\begin {ListeTheorem}
\item[]
\item Nilpotent and idempotent elements, in particular absorbing and identity elements, have a loop of length $1$. The initial pair of an identity element $0$ is $0<1$ and that of an absorbing element $\infty$ (recall that $0\infty=0$ as this is the empty sum) and of all other idempotent elements $\not=0$ is $1<2$.
\item A non-trivial unit $u$ with $nu=0$ for an $n\ge 2$ that is minimal with respect to this relation has a loop with initial pair $0<n$ and length $n$. Moreover, the binoid group $\free(u)/(nu=0)\cong(\Z/n\Z)^{\infty}$ \emph{is} a loop, cf.\ the picture after Corollary \ref {CorClassificationOnegenerated} below. With the terminology of group theory we may call $u$ an element of \emph{order} $n$.
\item A cancellative nonunit is loopfree.
\end {ListeTheorem}
\end {Example}

Recall that all boolean binoids are torsion-free, cf.\ Lemma \ref{LemTorsion}(2). In particular, $\free(x)/(2x=x)$ ($=\{0,x,\infty\}$) is torsion-free.

\begin {Lemma} \label{LemLoop}
Let $M$ be a binoid and $x\in M$. If $x$ has a loop of length $\ge2$ or a loop of length $1$ with initial pair $r+1>r\ge 2$, then $x$ is a torsion element and, in particular, $M$ is not torsion-free.
\end {Lemma}
\begin {proof}
If $r<s$ is the initial pair of the loop of $x$ and $k=s-r\ge 2$ its length, then $nx=nx+kx$ for all $n\ge r$. So $krx=krx+kx=k(r+1)x$ but $rx\not=(r+1)x$ by assumption. Hence, $x$ is a torsion element. In the other case, if $rx=(r+1)x=\infty$, there is nothing to show. So let $rx=(r+1)x\not=\infty$ for $r\ge 2$. We claim that $(r-1)x$ is a torsion element. By assumption, we have inductively $(r+n)x=rx$ for all $n\ge 1$. Hence, $2(r-1)x=rx=2rx$ but $(r-1)x\not=rx$. 
\end {proof}

\begin {Proposition} \label{PropOnegenerated} Let $M$ be a one-generated binoid with generator $x$.
\begin {ListeTheorem}
\item $M$ is integral if and only if it is reduced (i.e.\ if $x$ is not nilpotent).
\item $x$ is nilpotent if and only if $M\cong\N^{\infty}/(n=\infty)$ for some $n\ge 2$.
\item $x$ is a unit if and only if $M\cong(\Z/n\Z)^{\infty}$ for some $n\ge2$.
\item The following statements are equivalent if $x$ is neither nilpotent nor a unit.
\begin {ListeTheorem}
\item [(a)] $x$ is a cancellative element.
\item [(b)] $M$ is infinite.
\item [(c)] $x$ is loopfree.
\item [] If $x\not=2x$, then this is also equivalent to
\item [(d)] $M$ is torsion-free.
\item [(e)] $M\cong\N^{\infty}$.
\end {ListeTheorem}
\end {ListeTheorem}
\end {Proposition}
\begin {proof}
(1)-(3) are clear. (4) The implications $(c)\eq(b)\Rarrow(a)$ are obvious and $(a)\Rarrow(c)$ follows from Example \ref{ExLoop}(3). So let $2x\not=x$. By Lemma \ref {LemLoop}, $(d)$ implies $(c)$. In particular, (d) implies (a) and (b), which shows that $M$ is regular by (1) and (2). The implication $(d)\Rarrow(e)$ therefore follows from Corollary \ref{CorFree}. Condition $(e)$ implies everything.
\end {proof}

\begin{Corollary}\label {CorClassificationOnegenerated}
Up to isomorphism there are four different types of one-generated binoids, namely

\begin{tabular}[c]{ccccclllll}
&&&&&$\N^{\infty}\komma$&$(\Z/n\Z)^{\infty}\komma$&$\N^{\infty}/(r=s)\komma$&and&$\quad\N^{\infty}/(m=\infty)$\\
\multicolumn{10}{l}{\text{with $n,m\ge 2$ and $1\le r< s$. Their binoid algebras over $K$ are}}\\
&&&&&$K[X]\komma$&$ K[X]/(X^{n}-1)\komma$&$K[X]/(X^{s}-X^{r})\komma$&and&$\quad K[X]/(X^{m})\pkt$\\
\end{tabular}

\end{Corollary}
\begin {proof}
This is clear by Proposition \ref{PropOnegenerated}.
\end {proof}

These one-generated binoids can be illustrated as follows (the arrows indicate addition with the generator $1$):

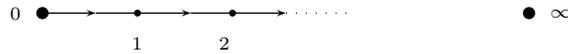
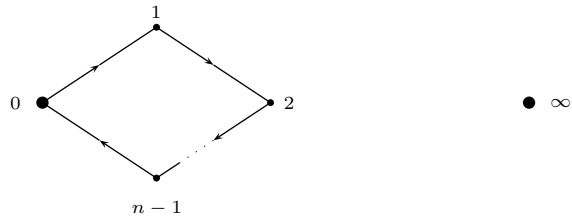
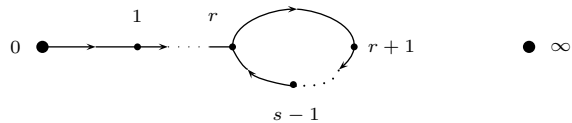
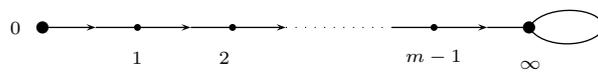
\begin {ListeTheorem}
\item []$\N^{\infty}$:
\begin {center}
\begin {pspicture} (-2,-0.5)(4.5,0.25)
\qdisk (-1.5,0){2.3pt}\qdisk (-0.25,0){1.3pt}\qdisk (1,0){1.3pt}\qdisk (4.9,0){2.3pt}
\psline [arrows={->},linewidth=.5 pt] (-1.5,0)(-0.8,0)
\psline [linewidth=.5 pt](-0.8,0)(-0.25,0)
\psline [arrows={->},linewidth=.5 pt] (-0.25,0)(0.45,0)
\psline [linewidth=.5 pt] (0.45,0)(1,0)
\psline [arrows={->},linewidth=.5 pt] (1,0)(1.7,0)
\psline [linewidth=.5 pt, linestyle=dotted] (1.7,0)(2.5,0)
\uput [0] (-2.1,-0){\scriptsize{$0$}}
\uput [0] (-0.5,-0.4){\scriptsize{$1$}}
\uput [0] (0.65,-0.4) {\scriptsize{$2$}}
\uput [0] (5,0){\scriptsize{$\infty$}}
\end {pspicture}
\end {center}
\item [] $(\Z/n\Z)^{\infty}$:
\begin {center}
\begin {pspicture} (-2,-1.5)(4.5,1.25)
\qdisk (-1.5,0){2.3pt}\qdisk (0,1){1.3pt}\qdisk (1.5,0){1.3pt}\qdisk (0,-1){1.3pt}\qdisk (4.9,0){2.3pt}
\psline [arrows=->,linewidth=.5 pt] (-1.5,0)(-0.75,0.5)
\psline [linewidth=.5 pt](-0.75,0.5)(0,1)
\psline [arrows=->,linewidth=.5 pt] (0,1)(0.75,0.5)
\psline [linewidth=.5 pt] (0.75,0.5)(1.5,0)
\psline [arrows=->,linewidth=.5 pt] (1.5,0)(0.75,-0.5)
\psline [linewidth=.5 pt, linestyle=dotted](0.75,-0.5) (0.3,-0.8)
\psline [linewidth=.5 pt](0.3,-0.8) (0,-1)
\psline [arrows=->,linewidth=.5 pt] (0,-1)(-0.75,-0.5)
\psline [linewidth=.5 pt] (-0.75,-0.5)(-1.5,0)
\uput [0] (-2.1,0){\scriptsize{$0$}}
\uput [0] (-0.25,1.2){\scriptsize{$1$}}
\uput [0] (1.5,0){\scriptsize{$2$}}
\uput [0] (-0.5,-1.4) {\scriptsize{$n-1$}}
\uput [0] (5,0) {\scriptsize{$\infty$}}
\end {pspicture}
\end {center}
\item [] $\N^{\infty}/(r=s)$:
\begin {center}
\begin {pspicture} (-2,-1)(4.5,0.75)
\qdisk (-1.5,0){2.3pt}\qdisk (-0.25,0){1.3pt}\qdisk (1,0){1.3pt}\qdisk (1.8,-0.5){1.3pt}\qdisk (2.6,0){1.3pt}\qdisk (4.9,0){2.3pt}
\psline [arrows=->,linewidth=.5 pt] (-1.5,0)(-0.8,0)
\psline [linewidth=.5 pt](-0.8,0)(-0.25,0)
\psline [arrows=->,linewidth=.5 pt] (-0.25,0)(0.15,0)
\psline [linewidth=.5 pt, linestyle=dotted] (0.15,0) (0.7,0)
\psline [linewidth=.5 pt] (0.7,0)(1,0)
\pcarc [arcangleA=60, arcangleB=30, linewidth=.5pt]{->}(1,0)(1.85,0.5)
\pcarc [arcangleA=30, arcangleB=60, linewidth=.5pt](1.85,0.5)(2.6,0)
\pcarc [arcangleA=30, arcangleB=15, linewidth=.5pt]{->}(2.6,0)(2.4,-0.3)
\pcarc [arcangleA=60, arcangleB=30, linestyle=dotted](2.4,-0.3)(1.8,-0.5)
\pcarc [arcangleA=180, arcangleB=30,linewidth=.5 pt]{->} (1.8,-0.5)(1.2,-0.3)
\pcarc [arcangleA=30, arcangleB=180,linewidth=.5 pt] (1.2,-0.3)(1,0)
\uput [0] (-2.1,0){\scriptsize{$0$}}
\uput [0] (-0.5,0.4){\scriptsize{$1$}}
\uput [0] (0.5,0.4){\scriptsize{$r$}}
\uput [0] (2.6,0) {\scriptsize{$r+1$}}
\uput [0] (1.35,-0.9) {\scriptsize{$s-1$}}
\uput [0] (5,0) {\scriptsize{$\infty$}}
\end {pspicture}
\end {center}
\item [] $\N^{\infty}/(m=\infty)$:
\begin {center}
\begin {pspicture} (-2,-1)(4.5,0.25)
\qdisk (-1.5,0){2.3pt}\qdisk (-0.25,0){1.3pt}\qdisk (1,0){1.3pt}\qdisk (3.65,0){1.3pt}\qdisk (4.9,0){2.3pt}
\psline [arrows=->,linewidth=.5 pt] (-1.5,0)(-0.8,0)
\psline [linewidth=.5 pt](-0.8,0)(-0.25,0)
\psline [arrows=->,linewidth=.5 pt] (-0.25,0)(0.45,0)
\psline [linewidth=.5 pt] (0.45,0)(1,0)
\psline [arrows=->,linewidth=.5 pt] (1,0)(1.7,0)
\psline [linewidth=.5 pt, linestyle=dotted] (1.7,0)(3.1,0)
\psline [linewidth=.5 pt] (3.1,0)(3.65,0)
\psline [arrows=->,linewidth=.5 pt] (3.65,0)(4.35,0)
\psline [linewidth=.5 pt] (4.35,0)(4.9,0)
\pcarc [arcangleA=60, arcangleB=60, linewidth=.5pt]{->}(4.9,0)(5.9,0)
\pcarc [arcangleA=300, arcangleB=300, linewidth=.5pt](4.9,0)(5.9,0)\uput [0] (-2.1,0){\scriptsize{$0$}}
\uput [0] (-0.5,-0.4){\scriptsize{$1$}}
\uput [0] (0.65,-0.4){\scriptsize{$2$}}
\uput [0] (3.1,-0.4) {\scriptsize{$m-1$}}
\uput [0] (4.6,-0.5) {\scriptsize{$\infty$}}
\end {pspicture}
\end {center}
\end {ListeTheorem}

\bigskip

If $M$ is a finitely generated binoid with generators $x_{1}\kpkt x_{r}$, then every subbinoid $\langle x_{k}\rangle$ is isomorphic to one of the above listed one-generated binoids, and if there are no mixed relations among the $x_{i}$, then $M\cong\langle x_{1}\rangle\wedge\cdots\wedge\langle x_{r}\rangle$ and $K[M]=K[\langle x_{1}\rangle]\otimes_{K}\cdots\otimes_{K}K[\langle x_{r}\rangle]$ by Lemma \ref{LemNoRelations} and the subsequent remark. For $r=2$, we obtain up to isomorphism  $4+\binom{4}{2}=10$ possibilities for $M$ of this unmixed type given by
$$\begin {array}{rll}
1.& \N^{\infty}\,\wedge\,\N^{\infty}\cong(\N\,\times\,\N)^{\infty}
     &\leadsto\quad K[X,Y]\\
2.& \N^{\infty}\,\wedge\,(\Z/m\Z)^{\infty}\cong(\N\,\times\,\Z/m\Z)^{\infty}
   &\leadsto\quad K[X,Y]/(Y^{m}-1)\\
3.& \N^{\infty}\,\wedge\,\N^{\infty}/(l=m) 
   &\leadsto\quad K[X,Y]/(Y^{l}-Y^{m})\\
4.& \N^{\infty}\,\wedge\,\N^{\infty}/(m=\infty) 
   &\leadsto\quad K[X,Y]/(Y^{m})\\
5.&(\Z/n\Z)^{\infty}\,\wedge\,(\Z/n\Z)^{\infty}\cong(\Z/n\Z\,\times\,\Z/m\Z)^{\infty}
   &\leadsto\quad K[X,Y]/(X^{n}-1,Y^{m}-1)\\
6.&(\Z/n\Z)^{\infty}\,\wedge\,\N^{\infty}/(l=m)
   &\leadsto\quad K[X,Y]/(X^{n}-1,Y^{l}-Y^{m})\\
7.&(\Z/n\Z)^{\infty}\,\wedge\,\N^{\infty}/(m=\infty) 
   &\leadsto\quad K[X,Y]/(X^{n}-1,Y^{m})\\
8.&\N^{\infty}/(k=n)\,\wedge\,\N^{\infty}/(l=m)
   &\leadsto\quad K[X,Y]/(X^{k}-X^{n},Y^{l}-Y^{m})\,\,\quad\quad\quad\quad\\
9.&\N^{\infty}/(k=n)\,\wedge\,\N^{\infty}/(m=\infty) 
   &\leadsto\quad K[X,Y,]/(X^{k}-X^{n},Y^{m})\\
10.&\N^{\infty}/(n=\infty)\,\wedge\,\N^{\infty}/(l=\infty)
   &\leadsto\quad K[X,Y]/(X^{n},Y^{m})\\
\end {array}$$

with $k,l\ge 1$ and $n,m\ge 2$, and where the identifications are given by $T^{1\wedge0}\leftrightarrow X$ and $T^{0\wedge1}\leftrightarrow Y$. Note that $\Z/n\Z\,\times\,\Z/m\Z\cong\Z/nm\Z$ if $n$ and $m$ are coprime so that the binoid group $(\Z/nm\Z)^{\infty}$ is a \emph{one}-generated binoid, cf.\ also Lemma \ref{LemBinGroup1} below.

\medskip

The mixed relations on a two-generated binoid are\nomenclature[R]{$\Rcal^{(2)}_{i}$}{, $i=1,2,3$, relations on a two-generated binoid}
\begin {ListeTheorem}
\item [$\Rcal^{(2)}_{1}:$] $kx+ly=0$ for some $l,k\ge 1$,
\item [$\Rcal^{(2)}_{2}:$] $kx+ly=\infty$ for some $l,k\ge 1$,
\item [$\Rcal^{(2)}_{3}:$] $kx+ly=nx+my\not\in\trivial$, where $(k,l)\not=(n,m)$ are not both $=(0,0)$, and if $(k,m)=(0,0)$, then $l,n\ge2$ and vice versa.
\end {ListeTheorem}

We have the following binoid algebras over $K$:

$$\begin {array}{ll}
\free(x,y)/(\Rcal^{(2)}_{1})&\leadsto\quad K[X,Y]/(X^{k}Y^{l}-1) \\
\free(x,y)/(\Rcal^{(2)}_{2})&\leadsto\quad K[X,Y]/(X^{k}Y^{l}) \\
\free(x,y)/(\Rcal^{(2)}_{3})&\leadsto\quad K[X,Y]/(X^{k}Y^{l}-X^{n}Y^{m})\pkt
\end {array}$$

Of course, more than one mixed relation may occur and also combinations, except the relations $\Rcal^{(2)}_{1}$ and $\Rcal^{(2)}_{2}$ which are mutually exclusive.

\medskip

In the following, two-generated binoids with respect to the mixed relations above are studied in detail. We start with the binoid groups.

\begin {Lemma} \label {LemBinGroup}
Let $M$ be a two-generated binoid with generators $x$ and $y$. 
\begin {ListeTheorem}
\item $M$ is a finite binoid group if and only if $kx=0$ and $ly=0$ for some $k,l\ge2$. In this case, $k$ and $l$ can be chosen to be minimal with this condition.
\item  Let $M$ be an infinite binoid group.
\begin {ListeTheorem}
\item [(a)] $\langle x\rangle\cap\langle y\rangle=\trivial$.
\item [(b)] There are $n,m\ge1$ minimal with respect to $nx+my=0$; that is, for any other pair $r,s\ge1$ with $rx+sy=0$, one has $r=ln$ and $s=lm$ for some $l\ge 1$. 
\item [(c)]  $M$ is torsion-free if and only if $\gcd(n,m)=1$, where $n,m\ge 1$ are minimal with respect to $nx+my=0$ (see (b)).
\end{ListeTheorem}
\end {ListeTheorem}
\end {Lemma}
\begin {proof}
(1) $M$ is a finite binoid group if and only if $\langle x\rangle$ and $\langle y\rangle$ are finite, which is in case $x$ and $y$ are units, hence cancellative, equivalent to $kx=0$ and $ly=0$ for some $k,l\ge2$. The statement on the minimality is obvious. (2) Let $M$ be an infinite binoid group with $nx+my=0$ for some $n,m\ge 0$ (a) Suppose that $rx=sy$ for $r,s\ge2$. Then
$$(sn+rm)x\,=\,snx+smy\,=\,s(nx+my)\,=\,0$$
and similarly $(sn+rm)y=0$, which obviously contradicts the assumption on $M$ being infinite. (b) First observe that $kx,ky\not=0$ for all $k\ge2$ by (a) because if, for instance, $ky=0$ for some $k\ge2$, then $nx+my=0$ would imply $nx=sy$ for some $s\ge1$ by adding $y$ sufficiently often to it. In particular, $n,m\ge 1$. Let $n_{0}:=\min_{\le}\{n\in\N\mid nx+my=0$ for some $m\ge1\}$ and let $m_{0}\ge1$ with $n_{0}x+m_{0}y=0$. Suppose that
$$nx+my\,=\,0\quad(=n_{0}x+m_{0}y)$$
for another pair $n,m\ge1$. By the choice of $n_{0}$, we have $n\ge n_{0}$. Since $x$ is cancellative and $kx,ky\not=0$ for all $k\ge1$, $n=n_{0}$ is equivalent to $m=m_{0}$. In particular, $m_{0}$ is unique with $n_{0}x+m_{0}y=0$. The assumption $m<m_{0}$ (and $n\ge n_{0}$) yields again a contradiction to $\langle x\rangle\cap\langle y\rangle=\trivial$. Therefore, $n>n_{0}$ and $m>m_{0}$. So one finds an $l\ge1$ such that 
$$n=ln_{0}+r\quad\text{and}\quad m=lm_{0}+s\komma$$ 
with $0\le r<n_{0}$ or $0\le s<m_{0}$, which yields
$$0\,=\,nx+my\,=\,l(n_{0}x+m_{0}y)+rx+sy\,=\,rx+sy\pkt$$
By the choice of $n_{0}$, the case $0\le r<n_{0}$ yields $r=0$, and hence $s=0$. The other case, $r\ge n_{0}$ and $s<m_{0}$, is not possible as observed above. Hence, $n=ln_{0}$ and $m=lm_{0}$ for some $l\ge1$. (c) Let $n,m\ge1$ be minimal with respect to $nx+my=0$ as in (b). If $\gcd(n,m)=d>1$, then
$$d(n^{\prime}x+m^{\prime}y)\,=\,0\,=\,d\cdot0\komma$$
where $n^{\prime}=n/d$ and  $m^{\prime}=m/d$, but $n^{\prime}x+m^{\prime}y\not=0$ by the minimality of $n$ and $m$. This shows that $M$ is not torsion-free. For the converse, we need to prove that $M\opkt$ is a torsion-free group if $\gcd(n,m)=1$. So assume that $z=rx+sy\in M\opkt$ with 
$$0\,=\,kz\,=\,krx+ksy$$
for some $k\ge2$. By (b), we get $kr=ln$ and $ks=lm$ for some $l\in\N$. Since $\gcd(n,m)=1$, $k$ divides $l$, so we can write $r=n(l/k)$ and $s=m(l/k)$ with $l/k\in\N$. Hence, $z=rx+sy=(l/k)(nx+my)=0$.
\end{proof}

\begin{Lemma} \label{LemBinGroup1}
Let $M$ be a two-generated binoid that admits a generating set $\{x,y\}$ with no mixed relation of the form $\Rcal_{3}^{(2)}$.
\begin {ListeTheorem}
\item If $M$ is a finite binoid group, then $M$ is isomorphic to
$$(\Z/k\Z\times\Z/l\Z)^{\infty}\komma$$
where $k,l\ge 2$ are minimal with respect to $kx=0=ly$ and $\gcd(k,l)\ge 2$.
\item If $M$ is an infinite binoid group and $n,m\ge1$ minimal with respect to $nx+my=0$ (as in Lemma \ref{LemBinGroup}(2b)), then $M$ is isomorphic to
$$\Z^{\infty}\quad\text{or}\quad(\Z\times\Z/d\Z)^{\infty}\pkt$$
In the first case, $M$ is torsion-free with $\gcd(n,m)=1$, and in the latter $M$ is not torsion-free with $\gcd(n,m)=d\ge2$.
\end {ListeTheorem}
\end{Lemma}
\begin{proof}
(1) By Lemma \ref {LemBinGroup}, there are $k,l\ge 2$ with $kx=0=ly$. If $k$ and $l$ are minimal with this property, we obtain from Lemma \ref{LemNoRelations} 
$$M\,\cong\,\langle x\rangle\wedge\langle y\rangle\,\cong\,(\Z/k\Z)^{\infty}\wedge(\Z/l\Z)^{\infty}\,\cong\,(\Z/k\Z\times\Z/l\Z)^{\infty}\pkt$$
If $\gcd(k,l)=1$, then $M\cong(\Z/kl\Z)^{\infty}$ is one-generated. Hence, $\gcd(k,l)\ge2$. 

(2) First consider the torsion-free case, which is the case when $\gcd(n,m)=1$ by Lemma \ref {LemBinGroup}. Thus, $lm+kn=1$ for some $k,l\in\Z$. We may assume that $k<0$ and $l>0$ because the other case $k>0$ and $l<0$ follows by symmetry. By adding $nm-mn=0$ to $|k|n-lm=-1$ as often as necessary, one finds a $k^{\prime}<0$ and an $l^{\prime}>0$ with $l^{\prime}m+k^{\prime}n=-1$. With this, the homomorphism 
$$\varphi:M\opkt\Rto\Z\komma\quad x\lto m\komma\quad y\lto -n\komma$$
is a well-defined group epimorphism since 
\begin {align*}
0=nx+my&\lto 0\\
 lx+|k|y&\lto lm+|k|y=1\\
l^{\prime}x+|k^{\prime}|y&\lto l^{\prime}m+|k^{\prime}|y=-1
\end {align*}
The injectivity of $\varphi$ follows from the fact that $0=\varphi(ax+by)=am-bn$ is equivalent to $a=sn$ and $b=sm$ for some $s\in\N$, which gives $ax+by=s(nx+my)=0$. Extending $\varphi$ by $\infty\mto\infty$ now yields a binoid isomorphism $M\cong\Z^{\infty}$. 

Finally, assume that $M$ is not torsion-free. By Lemma \ref{LemBinGroup}(2c), $\gcd(n,m)=d\ge 2$ so we can write $n=\tilde{n}d$ and $m=\tilde{m}d$ for $\tilde{n},\tilde{m}\ge1$ with $\gcd(\tilde{n},\tilde{m})=1$.
The latter means $1=s\tilde{m}+r\tilde{n}$ for some $r,s\in\Z$ with $r<0$ and $s>0$ or vice versa. By symmetry, we may assume that $r<0$ and $s>0$. The map 
$$\psi:M\opkt\Rto\Z\times(\Z/d\Z)\komma\quad\psi(ix+jy)\,=\,(i\tilde{m}-j\tilde{n},ir+js\mod d)\komma$$
is a well-defined group isomorphism. The surjectivity holds since 
$$\psi(sx+(-r)y)=(1,0)\quad\text{and}\quad\psi(\tilde{n}x+\tilde{m}y)=(0,1)\pkt$$
For the injectivity, assume that $\psi(ix+jy)=(0,0)$. Then 
$$i\,=\,k\tilde{n}\komma\quad j\,=\,k\tilde{m}\komma\quad\text{and}\quad ir+js\,=\,k(\tilde{n}r+\tilde{m}s)\,=\,k\,\equiv\,0\mod d$$
where the latter implies $ix+jy=0$. Therefore, we have shown that if $M$ is a two-generated infinite binoid group that is not torsion-free, then $M\cong(\Z\times\Z/d\Z)^{\infty}$ for some $d\ge 2$. 
\end {proof}

\begin {Remark}
The two-generated binoid group $\Z^{\infty}$ is isomorphic to $\free(x,y)/(x+y=0)$. In general, we have
$$\free(x_{1}\kpkt x_{r})/(x_{1}\pluspkt x_{r}=0)\,\,\cong\,\,(\Z^{r-1})^{\infty}\komma$$
where the isomorphism is given by $x_{i}\mto e_{i}$ when $i\not=r$, and $x_{r}\mto-(e_{1}\pluspkt e_{r})$ otherwise.
\end {Remark}

\begin{Proposition} \label{PropBinGroup}
A two-generated binoid is a binoid group if and only if it is isomorphic to 
$$\Z^{\infty}\quad\text{or}\quad(\Z\times\Z/d\Z)^{\infty}$$
for $d\ge2$ or
$$(\Z/k\Z\times\Z/l\Z)^{\infty}$$
for $k,l\ge2$ with $\gcd(k,l)\ge 2$.
\end{Proposition}
\begin{proof}
$M$ is a two-generated binoid group if and only if $M\opkt$ is a two-generated (abelian) group. By the structure theorem for finitely generated abelian groups, we know that
$$M\opkt\,\,\cong\,\,\Z^{d}\times\Z/n_{1}\Z\timespkt\Z/n_{r}\Z$$
for some $d\ge0$ and $n_{1}\kpkt n_{r}\ge 2$. By Lemma \ref{LemBinGroup1}, the binoid groups listed in the proposition are the only two-generated binoid groups of this kind, all other have strictly more or less generators.
\end{proof}

The associated binoid algebras of the two-generated binoid groups are given by 

$$\begin {array}{rl}
K[\Z^{\infty}]\!\!&\cong\,\, K[X,Y]/(XY-1)\komma\\
K[(\Z\times\Z/d\Z)^{\infty}]\!&\cong\,\, K[X,Y]/(X^{d}Y^{d}-1)\komma\\
K[((\Z/k\Z)\times(\Z/l\Z)^{\infty}]&\cong\,\, K[X,Y]/(X^{k}-1,Y^{l}-1)\komma
\end {array}$$

where $d\ge 2$ and $\gcd(k,l)\ge 2$. The second result is deduced from the fact that $(1,1)$ and $(-1,0)$ generate the binoid $(\Z\times\Z/d\Z)^{\infty}$, $d\ge 2$, since
$$d(1,1)+(d-1)(-1,0)=(1,0)\quad\text{and}\quad(1,1)+(-1,0)=(0,1)\pkt$$
Hence, $K[X,Y]\rto K(\Z\times\Z/d\Z)$ with $X\mto T^{(1,1)}$ and $Y\mto T^{(-1,0)}$ is an epimorphism of $K\mina$algebras with kernel generated by $X^{d}Y^{d}-1$. The other identifications are obvious.

The two-generated finite binoid groups with one mixed relation of the form $\Rcal_{3}^{(2)}$ have a very nice description in the terminology of binoids.

\begin {Proposition}  \label{PropBinGroupWith}
If $M$ is a two-generated finite binoid group that admits generators with a relation of the from $\Rcal_{3}^{(2)}$, then $M$ is isomophic to
$$(\Z/k\Z)^{\infty}\wedge_{\N^{\infty}}(\Z/l\Z)^{\infty}$$
for some $k,l\in\N$ with $\gcd(k,l)\ge2$, where $(\Z/k\Z)^{\infty}$ is a $\Z^{\infty}\mina$binoid via $\varphi_{r}:1\mto r$ and $(\Z/l\Z)^{\infty}$ is a $\Z^{\infty}\mina$binoid via $\varphi_{s}:1\mto s$ for some  $2\le r<k$ and  $2\le s<l$.
\end {Proposition}
\begin {proof}
First note that on a finite two-generated binoid group with generators $x,y\in M$ a relation of the from $$\Rcal_{3}^{(2)}:\,nx+my\,=\,px+qy\,\not\in\,\trivial$$
is always equivalent to $rx=sy\not\in\trivial$ for certain $r,s\ge2$, which is a specialization of $\Rcal_{3}^{(2)}$. More precisely, if $k,l\ge2$ are minimal with respect to $kx=0$ and $ly=0$, cf.\ Lemma \ref{LemBinGroup1}(1), we may assume that $0\le n,p\le k$ and $0\le m,q\le l$. Then $\Rcal_{3}^{(2)}$ is equivalent to $rx=sy\not\in\trivial$, where $r\equiv n+k-p\mod k$ and $s\equiv q+l-m\mod l$.

In particular, instead of an arbitrary finite two-generated binoid group with one mixed relation of the form $\Rcal_{3}^{(2)}$, it suffices to consider
$$(\Z/k\Z)^{\infty}\wedge(\Z/l\Z)^{\infty}/(r\wedge0=0\wedge s)=:M$$
with $2\le r<k$ and  $2\le s<l$. In this case, we have a commutative diagram
$$\xymatrix{
&(\Z/l\Z)^{\infty}\ar[d]^{\iota_{l}}\ar[rd]^{\psi_{l}}&&\\
\N^{\infty}\ar[ru]^{\varphi_{s}}\ar[rd]_{\varphi_{r}}&(\Z/k\Z)^{\infty}\wedge_{\N^{\infty}}(\Z/l\Z)^{\infty}&M\komma\\
&(\Z/k\Z)^{\infty}\ar[u]_{\iota_{k}}\ar[ru]_{\psi_{k}}&&\\
}$$
where $\varphi_{i}$, $i\in\{r,s\}$, are as in the proposition and $\psi_{k}$ and $\psi_{l}$ are given by the compositions 
$$\psi_{j}:(\Z/j\Z)^{\infty}\Rto(\Z/k\Z)^{\infty}\wedge(\Z/l\Z)^{\infty}\stackrel{\!\!\pi}{\Rto}M\komma$$ $j\in\{k,l\}$, of the canonical embedding (on the respective component) and the canonical projection. Hence, there is by Corollary \ref{CorCoProdSmashN} a unique binoid homomorphism
$$\psi:(\Z/k\Z)^{\infty}\wedge_{\N^{\infty}}(\Z/l\Z)^{\infty}\Rto M$$
with $\iota_{j}\psi=\psi_{j}$, $j\in\{k,l\}$, which is surjective since $\psi(1\wedge_{\N^{\infty}}0)=x$ and $\psi(0\wedge_{\N^{\infty}}1)=y$. On the other hand, there is a $\trivial\mina$binoid homomorphism $$\varphi:(\Z/k\Z)^{\infty}\wedge(\Z/l\Z)^{\infty}\Rto(\Z/k\Z)^{\infty}\wedge_{\N^{\infty}}(\Z/l\Z)^{\infty}$$
by Corollary \ref{CorCoProdSmashN}, which factors through $M$ because $\varphi(r\wedge0)=r\wedge_{\N^{\infty}}0=0\wedge_{\N^{\infty}}s=\varphi(0\wedge s)$. This factorization and $\psi$ are inverse to each other. Hence, $M\cong(\Z/k\Z)^{\infty}\wedge_{\N^{\infty}}(\Z/l\Z)^{\infty}$.
\end {proof}

A two-generated (finite) binoid group as in Proposition \ref{PropBinGroupWith} with $kx=ly=0$, $\gcd(k,l)\ge 2$, and only one more generating relation of the form $rx=sy$ with $r<k$ and $s<l$, can be displayed in the following way (here $r$ divides $k$ and $s$ divides $l$):

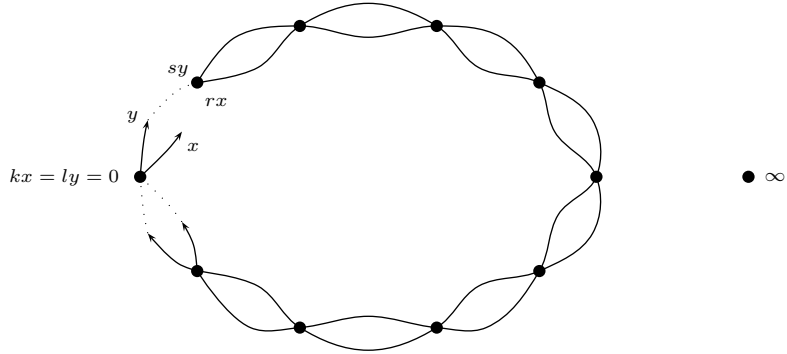
\begin {center}
\begin{pspicture}(-4.5,-3)(6,3)
\qdisk(-3,0){2.3pt}\qdisk(-2.25,1.25){2.3pt}\qdisk(-0.9,2){2.3pt}\qdisk(0.9,2){2.3pt}\qdisk(2.25,1.25){2.3pt}\qdisk(3,0){2.3pt}
\qdisk(2.25,-1.25){2.3pt}\qdisk(0.9,-2){2.3pt}\qdisk(-0.9,-2){2.3pt}\qdisk(-2.25,-1.25){2.3pt}
\qdisk(5,0){2.3pt}
\rput(-4,0){\scriptsize{$kx=ly=0$}}
\rput(-2.3,0.4){\scriptsize{$x$}}
\rput(-3.1,0.8){\scriptsize{$y$}}
\rput(-2,1){\scriptsize{$rx$}}
\rput(-2.5,1.4){\scriptsize{$sy$}}
\rput(5.35,0){\scriptsize{$\infty$}}
\pscurve[linewidth=0.5pt]{->}(-3,0)(-2.6,0.4)(-2.45,0.6)
\pscurve[linewidth=0.5pt](-2.25,1.25)(-1.75,1.85)(-0.9,2)(0,1.85)(0.9,2)(1.75,1.85)(2.25,1.25)(2.45,0.6)(3,0)(2.95,-0.75)(2.25,-1.25)(1.45,-1.45)(0.9,-2)(0,-2.3)(-0.9,-2)(-1.45,-1.45)(-2.25,-1.25)
\pscurve[linewidth=0.5pt]{->}(-2.25,-1.25)(-2.5,-1.15)(-2.9,-0.75)
\pscurve[linewidth=0.5pt, linestyle=dotted](-2.9,-0.75)(-2.95,-0.5)(-3,0)
\pscurve[linewidth=0.5pt]{->}(-3,0)(-2.95,0.5)(-2.9,0.75)
\pscurve[linewidth=0.5pt, linestyle=dotted](-2.9,0.75)(-2.5,1.15)(-2.25,1.25)
\pscurve[linewidth=0.5pt](-2.25,1.25)(-1.45,1.45)(-0.9,2)(0,2.3)(0.9,2)(1.45,1.45)(2.25,1.25)(2.95,0.75)(3,0)(2.5,-0.5)(2.25,-1.25)(1.55,-2)(0.9,-2)(0,-1.85)(-0.9,-2)(-1.55,-2)(-2.25,-1.25)
\pscurve[linewidth=0.5pt]{->}(-2.25,-1.25)(-2.3,-0.9)(-2.45,-0.6)
\pscurve[linewidth=0.5pt, linestyle=dotted](-2.45,-0.6)(-2.6,-0.4)(-3,0)
\end{pspicture}
\end {center}

Its associated algebra over $K$ is given by 
$$K[(\Z/k\Z)^{\infty}\wedge_{\N^{\infty}}(\Z/l\Z)^{\infty}]\,\,\,\cong\,\,\, K[X,Y]/(X^{k}-1,Y^{l}-1, X^{r}-Y^{s})\pkt$$


The situation, where $kx=ly$ (possibly $\in\trivial$) with $l,k\ge2$ is the only generating relation, can be visualized in the following way:

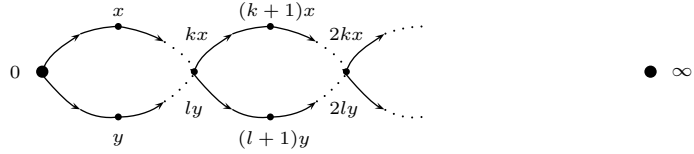
\begin {center}
\begin {pspicture} (-2,-1)(4.5,1)
\qdisk (-1.5,0){2.3pt}\qdisk(-0.5,0.6){1.3pt}\qdisk(0.5,0){1.3pt}\qdisk(-0.5,-0.6){1.3pt}
\qdisk(1.5,0.6){1.3pt}\qdisk(1.5,-0.6){1.3pt}\qdisk(2.5,0){1.3pt}\qdisk(6.5,0){2.3pt}
\pcarc [arcangleA=30, arcangleB=10, arrows=->, linewidth=.5pt](-1.5,0)(-1,0.5)
\pcarc [arcangleA=10, arcangleB=10, linewidth=.5pt](-1,0.5)(-0.5,0.6)
\pcarc [arcangleA=10, arcangleB=10, arrows=->, linewidth=.5pt](-0.5,0.6)(0.1,0.4)
\pcarc [arcangleA=10, arcangleB=30, linestyle=dotted](0.1,0.4)(0.5,0)
\pcarc [arcangleA=185, arcangleB=340, arrows=->, linewidth=.5pt](-1.5,0)(-1,-0.5)
\pcarc [arcangleA=340, arcangleB=340, linewidth=.5pt](-1,-0.5)(-0.5,-0.6)
\pcarc [arcangleA=340, arcangleB=340, arrows=->, linewidth=.5pt](-0.5,-0.6)(0.1,-0.4)
\pcarc [arcangleA=340, arcangleB=185, linestyle=dotted](0.1,-0.4)(0.5,0)
\pcarc [arcangleA=30, arcangleB=10, arrows=->, linewidth=.5pt](0.5,0)(1,0.5)
\pcarc [arcangleA=10, arcangleB=10, linewidth=.5pt](1,0.5)(1.5,0.6)\pcarc [arcangleA=10, arcangleB=10, arrows=->, linewidth=.5pt](1.5,0.6)(2.1,0.4)
\pcarc [arcangleA=10, arcangleB=30, linestyle=dotted](2.1,0.4)(2.5,0)
\pcarc [arcangleA=185, arcangleB=340, arrows=->, linewidth=.5pt](0.5,0)(1,-0.5)
\pcarc [arcangleA=340, arcangleB=340, linewidth=.5pt](1,-0.5)(1.5,-0.6)
\pcarc [arcangleA=340, arcangleB=340, arrows=->, linewidth=.5pt](1.5,-0.6)(2.1,-0.4)
\pcarc [arcangleA=340, arcangleB=185, linestyle=dotted](2.1,-0.4)(2.5,0)
\pcarc [arcangleA=30, arcangleB=10, arrows=->, linewidth=.5pt](2.5,0)(3,0.5)
\pcarc [arcangleA=10, arcangleB=10, linestyle=dotted](3,0.5)(3.5,0.6)
\pcarc [arcangleA=185, arcangleB=340, arrows=->, linewidth=.5pt](2.5,0)(3,-0.5)
\pcarc [arcangleA=340, arcangleB=340, linestyle=dotted](3,-0.5)(3.5,-0.6)
\uput [0] (-2.1,0){\scriptsize{$0$}}
\uput [0] (-0.75,0.8){\scriptsize{$x$}}
\uput [0] (-0.75,-0.9){\scriptsize{$y$}}
\uput [0] (0.2,0.5){\scriptsize{$kx$}}
\uput [0] (0.2,-0.5){\scriptsize{$ly$}}
\uput [0] (0.9,0.8){\scriptsize{$(k+1)x$}}
\uput [0] (0.9,-0.9){\scriptsize{$(l+1)y$}}
\uput [0] (2.1,0.5){\scriptsize{$2kx$}}
\uput [0] (2.1,-0.5){\scriptsize{$2ly$}}
\uput [0] (6.6,0){\scriptsize{$\infty$}}
\end {pspicture}
\end {center}

\bigskip

Now we consider the case where a mixed relation of this type is given.

\begin {Lemma} \label{LemLoopSmash}
Let $M$ be a two-generated binoid that admits generators $x$ and $y$ such that $kx=ly$ for some $l,k\ge2$.
\begin {ListeTheorem}
\item $x$ is loopfree if and only if $y$ is so.
\item The following statements are equivalent.
\begin {ListeTheorem}
\item [(a)] $M$ is integral
\item [(b)] $M$ is reduced
\item [(c)] $x$ and $y$ are not nilpotent.
\end {ListeTheorem}
\item [(3)]There is a surjective binoid homomorphism $\phi:\langle x\rangle\wedge_{\N^{\infty}}\langle y\rangle\rto M$, where $\langle x\rangle$ and $\langle y\rangle$ are considered as $\N^{\infty}\mina$binoids via $\varphi_{k}:1\mto kx$ and $\varphi_{l}:1\mto ly$, respectively, which is an isomorphism if $kx=ly$ is the only mixed generating relation.
\end {ListeTheorem}
\end {Lemma}
\begin {proof}
(1) and (2) follow from easy computations. (3) By assumption, we have a commutative diagram
$$
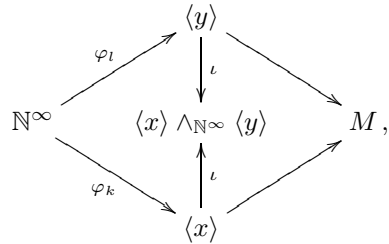
\xymatrix{
&\langle y\rangle\ar[d]^{\iota}\ar[rd]&&\\
\N^{\infty}\ar[ru]^{\varphi_{l}}\ar[rd]_{\varphi_{k}}&\langle x\rangle\wedge_{\N^{\infty}}\langle y\rangle&M\komma\\
&\langle x\rangle\ar[u]_{\iota}\ar[ru]&&\\
}$$
which gives rise to a binoid homomorphism $\phi:\langle x\rangle\wedge_{\N^{\infty}}\langle y\rangle\rto M$ by Corollary \ref{CorCoProdSmashN}. $\phi$ is surjective because $x$ and $y$ generate $M$. If $kx=ly$ is the only mixed generating relation, then
$$M\cong(\N^{2})^{\infty}/(\Rcal, ke_{1}=le_{2})\komma$$
where $\Rcal$ is a (maybe empty) set of unmixed relations, i.e.\ relations that involve only one of the generators, $e_{1}$ or $e_{2}$. By Lemma \ref{LemHomFreeCom}, there is a unique binoid homomorphism
$$\psi:(\N^{2})^{\infty}\Rto\langle x\rangle\wedge_{\N^{\infty}}\langle y\rangle\komma\quad e_{1}\lto x\wedge_{\N^{\infty}}0\komma\quad e_{2}\lto 0\wedge_{\N^{\infty}}y\komma$$
which respects the defining relations of $M$ as listed above since 
$$\psi(ke_{1})=(kx)\wedge_{\N^{\infty}}0=0\wedge_{\N^{\infty}}(ly)=\psi(le_{2})\pkt$$
Hence, there is a binoid homomorphism $\tilde{\psi}:M\rto\langle x\rangle\wedge_{\N^{\infty}}\langle y\rangle$ by Lemma \ref{LemIndCong}. The homomorphisms $\phi$ and $\tilde{\psi}$ are inverse to each other.
\end {proof}

By the preceding lemma, a two-generated binoid with generating set $\{x,y\}$ and $\Rcal:kx=ly$, $l,k\ge2$ being the only mixed generating relation, is given by the $\N^{\infty}\mina$binoid $\langle x\rangle\wedge_{\N^{\infty}}\langle y\rangle$ described as in the lemma. In this situation, the subbinoids $\langle x\rangle$ and $\langle y\rangle$ are isomorphic to $\N^{\infty}$, $\N^{\infty}/(n=\infty)$, $(\Z/n)^{\infty}$, or $\N^{\infty}/(n=m)$ with $n\ge2$ and $m\ge1$. The possible combinations are 
$$(\Z/n\Z)^{\infty}\wedge_{\N^{\infty}}(\Z/m\Z)^{\infty}$$
with $\gcd(n,m)\ge2$ as in Proposition \ref{PropBinGroupWith} and
$$\begin {array}{ll}
\N^{\infty}\wedge_{\N^{\infty}}\N^{\infty}&\leadsto\quad K[X,Y]/(X^{k}-Y^{l})\\
\N^{\infty}/(p=\infty)\wedge_{\N^{\infty}}\N^{\infty}/(q=\infty)&\leadsto\quad K[X,Y]/(X^{k}-Y^{l}, X^{r}, Y^{m})\\
\N^{\infty}/(r=s)\wedge_{\N^{\infty}}\N^{\infty}/(m=n)&\leadsto\quad K[X,Y]/(X^{k}-Y^{l}, X^{s}-X^{r}, Y^{n}-Y^{m})
\end {array}$$
where $1\le p<k$ and $1\le q<l$ and $1\le r<s<k$ and $1\le m<n<l$. The first identiy follows with $K[\N^{\infty}]\otimes_{K[\N^{\infty}]}K[\N^{\infty}]\cong K[X]\otimes_{K[Z]}K[Y]$, where $X^{k}=Z=Y^{l}$, and the latter two from the first.

\begin {Remark}
The $K\mina$spectrum of $\N^{\infty}\wedge_{\N^{\infty}}\N^{\infty}$, $K$ a field, is given by $K\minSpec K[X,Y]/(X^{k}-Y^{l})$, which is for $k=3$, $l=2$ and $K=\R$ precisely, 

\begin {center}
\begin{pspicture}(-2,-2.5)(2,2)
\qdisk (0.6,0.6){1.75pt}\qdisk (0,0){1.75pt}
\uput [0] (-1.25,-2.2){\small{Neil's parabola.}}
\psline [linewidth=0.5 pt, linestyle=dotted] (0,-1.5)(0,1.5)
\psline [linewidth=0.5 pt, linestyle=dotted] (-1.5,0)(1.5,0)
\pscurve[linewidth=0.5pt](0,0)(0.3,0.2)(0.6,0.6)(0.9,1.5)
\pscurve[linewidth=0.5pt](0,0)(0.3,-0.2)(0.6,-0.6)(0.9,-1.5)
\end{pspicture}
\end {center}
\end {Remark}

In the last section, cf. Example \ref{ExpNsmashNcancellative}, we have already shown that $\N^{\infty}\wedge_{\N^{\infty}}\N^{\infty}$ is a positive and cancellative $\N^{\infty}\mina$binoid that admits a unique minimal generating set consisting of the elements $1\wedge_{\N^{\infty}}0$ and $0\wedge_{\N^{\infty}}1$. 

The same holds true for the latter two $\N^{\infty}\mina$binoids in the list above since they are isomorphic to $(\N^{\infty}\wedge_{\N^{\infty}}\N^{\infty})/(\Rcal_{i}, i\in I)$, where the respective additional relations $\Rcal_{i}$ do not affect these properties. Obviously, $\N^{\infty}/(p=\infty)\wedge_{\N^{\infty}}\N^{\infty}/(q=\infty)$ is not reduced  as $p(1\wedge_{\N^{\infty}}0)=\infty_{\wedge}$, hence not torsion-free. By Lemma \ref{LemLoop}, the binoid $\N^{\infty}/(r=s)\wedge_{\N^{\infty}}\N^{\infty}/(m=n)$ is not torsion-free if $s-r\ge2$ or $n-m\ge2$. The remaining case $r=m=1$ and $s=n=2$ (the boolean case) is excluded as the mixed relation would enforce $x=y$. As announced in the last section, we will 
now determine when $\N^{\infty}\wedge_{\N^{\infty}}\N^{\infty}$ is torsion-free.

\begin{Lemma} \label{LemNwedgeNtorsionfree}
Consider the $\N^{\infty}\mina$binoid $\N^{\infty}\wedge_{\N^{\infty}}\N^{\infty}$ from above (see also Example \ref{ExpNsmashNcancellative}). The $\N^{\infty}\mina$binoid homomorphism 
$$\psi:\N^{\infty}\wedge_{\N^{\infty}}\N^{\infty}\Rto\N^{\infty}\quad\text{with}\quad n\wedge_{\N^{\infty}} m\lto\frac{nl+mk}{\gcd(k,l)}$$
is injective if and only if $\gcd(k,l)=1$. In particular, $\N^{\infty}\wedge_{\N^{\infty}}\N^{\infty}$ is torsion-free (and a subbinoid of $\N^{\infty}$) if and only if $k$ and $l$ are coprime.
\end {Lemma}
\begin {proof}
Recall that $\psi$ is the well-defined $\N^{\infty}\mina$binoid homomorphism induced by the $\N^{\infty}\mina$binoid homomorphisms $\N^{\infty}\rto\N^{\infty}$, $1\mto l/\gcd(k,l)$, when the codomain is the $\N^{\infty}\mina$binoid via $1\mto k$, and $1\mto k/\gcd(k,l)$ if the codomain is the $\N^{\infty}\mina$binoid via $1\mto l$, cf.\ Example \ref{ExpNSmashNMap}.

Suppose that $\gcd(k,l)=d>1$. Then $k=k^{\prime}d$ and $l=l^{\prime}d$ for some $1\le k^{\prime}<k$ and $1\le l^{\prime}<l$. Therefore, we obtain in $\N^{\infty}\wedge_{\N^{\infty}}\N^{\infty}$,
$$d(k^{\prime}\wedge_{\N^{\infty}}0)\,=\,k\wedge_{\N^{\infty}}0\,=\,0\wedge_{\N^{\infty}}l\,=\,d(0\wedge_{\N^{\infty}}l^{\prime})$$
but $k^{\prime}\wedge_{\N^{\infty}}0\not=0\wedge_{\N^{\infty}}l^{\prime}$. This shows that $\N^{\infty}\wedge_{\N^{\infty}}\N^{\infty}$ is not torsion-free, in particular, $\psi$ is not injective. For the converse let $k$ and $l$ be coprime and
$$\psi(n\wedge_{\N^{\infty}}m)\,=\,nl+mk\,=\,rl+sk\,=\,\psi(r\wedge_{\N^{\infty}}s)$$
with, say, $n\ge r$. Then $0\le(n-r)l=(s-m)k$ which implies $(n-r)=ck$ and $(s-m)=cl$ for some $c>0$ because $\gcd(k,l)=1$. So we have $n=r+ck$ and $s=m+cl$, which yields
$$n\wedge_{\N^{\infty}}m\,=\,(r+ck)\wedge_{\N^{\infty}}m\,=\,r\wedge_{\N^{\infty}}(m+cl)\,=\,r\wedge_{\N^{\infty}}s$$
in $\N^{\infty}\wedge_{\N^{\infty}}\N^{\infty}$. Hence, $\psi$ is injective if $\gcd(k,l)=1$, which means that  $\N^{\infty}\wedge_{\N^{\infty}}\N^{\infty}$ is a subbinoid of $\N^{\infty}$ in this case and, in particular, torsion-free.
\end {proof}

\begin {Proposition} \label{PropOrbitintersection}
If $M$ is a two-generated binoid with generating set $\{x,y\}$, then $M$ is isomorphic to $\N^{\infty}\wedge_{\N^{\infty}}\N^{\infty}$, where $\N^{\infty}$ is considered as an $\N^{\infty}\mina$binoid once via $1\mto k$ and once via $1\mto l$ for some $k,l\ge2$, if and only if $M$ is a cancellative and positive binoid with $\langle x\rangle\cap\langle y\rangle\not=\trivial$.
\end {Proposition}
\begin {proof}
By Example \ref{ExpNsmashNcancellative}, the binoid $\N^{\infty}\wedge_{\N^{\infty}}\N^{\infty}$ described in the proposition is positive and cancellative with unique minimal generating set $1\wedge_{\N^{\infty}}0$ and $0\wedge_{\N^{\infty}}1$. Hence, if $M\cong\N^{\infty}\wedge_{\N^{\infty}}\N^{\infty}$, the isormophism is given by $x\lrto 1\wedge_{\N^{\infty}}0$ and $y\lrto 0\wedge_{\N^{\infty}}1$ or vice versa, which implies $\langle x\rangle\cap\langle y\rangle\not=\trivial$ since $k\wedge_{\N^{\infty}}0=0\wedge_{\N^{\infty}}l$ if the structure homomorphism of $\N^{\infty}$ is given by $1\mto k$ in the left entry and by $1\mto l$ in the right entry.

For the converse note that $x$ and $y$ is the unique minimal generating set of $M$ because $M$ is positive and cancellative, cf.\ Proposition \ref{PropUniqueMinSyst}. Furthermore, $M$ is cancellative and positive if and only if $x$ and $y$ are cancellative nonunits, cf.\ Lemma \ref{LemIntCanGenerators}. By assumption, there exsits a $k\ge2$ minimal with respect to $kx=ly\not=\infty$ for some $l\ge2$. Then $l$ is unique since the subbinoid $\langle y\rangle$ generated by the cancellative nonunit $y$ is loopfree by Proposition \ref {PropOnegenerated}. To see that $l$ is also minimal with respect to $ly\in\langle x\rangle$ assume that  $nx=my$ with $m<l$. Then $mkx=mly=lnx$, where $ln>km\ge2$ because $n>k$ and $l>m$, which yields a non-trivial equation $0=(ln-km)x$ by the cancellativity of $x$ and therefore a contradiction to $x\not\in M\okreuz$. In particular, $\langle x\rangle\cap\langle y\rangle=\langle lx\rangle=\langle ky\rangle$ and $M\rto\N^{\infty}\wedge_{\N^{\infty}}\N^{\infty}$ with $x\mto 1\wedge_{\N^{\infty}}0$ and
 $y\mto0\wedge_{\N^{\infty}}1$ is a well-defined binoid isomorphism because there are by assumption no other generating relations on $M$ than $\Rcal:kx=ly$
\end {proof}

Now we consider the special case of $\Rcal_{2}^{(2)}:kx+ly=\infty$ when $k=l=1$ (i.e.\ $x+y=\infty$).

\begin {Proposition} 
Let $M$ be a finitely generated binoid with minimal generating set $\{x_{i}\mid i\in I\}$. If $x_{i}+x_{j}=\infty$ for all $i\not=j$, then $M$ is isomorphic to
$$\bigcupbidot_{k=1}^{m}M_{k}\quad\text{with}\quad M_{k}:=\langle x_{i}\mid i\in I_{k}\rangle\komma$$
where $I_{1}\kpkt I_{m}$ are the equivalence classes with respect to the relation on $I$ generated by 
$$i\sim j\quad:\eq\quad\langle x_{i}\rangle\cap\langle x_{j}\rangle\not=\trivial\pkt$$
If $\# I_{k}=1$, then $M_{k}$ is isomorphic to $\N^{\infty}$, $\N^{\infty}/(r=s)$ with $1\le r<s$, or $\N^{\infty}/(n=\infty)$ with $n\ge2$. If $\# I_{k}\ge 2$, then
$$M_{k}\,\,\cong\,\,\bigcupbidot_{i\in I_{k}}\N^{\infty}\Big/\big((n_{i};i)=(n_{j};j)\not=\infty\text{ for all }i,j\in I_{k}\big)$$
for some $n_{i}\ge 2$, $i\in I_{k}$, which may be displayed in the following way:
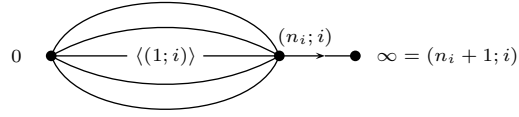
\begin {center}
\begin {pspicture}(-2,-1)(2,1)
\qdisk (1.5,0){2pt}\qdisk (-1.5,0){2pt}\qdisk (2.5,0){2pt}
\uput [0] (-2.2,0){\scriptsize{$0$}}
\uput [0] (-0.575,0){\scriptsize{\text{$\langle(1;i)\rangle$}}}
\uput [0] (1.3,0.25){\scriptsize{$(n_{i};i)$}}
\uput [0] (2.6,0){\scriptsize{$\infty=(n_{i}+1;i)$}}
\pcarc [arcangleA=30, arcangleB=30, linewidth=.5pt] (1.5,0)(-1.5,0)
\pcarc [arcangleA=330, arcangleB=330, linewidth=.5pt] (1.5,0)(-1.5,0)
\pcarc [arcangleA=70, arcangleB=70, linewidth=.5pt] (1.5,0)(-1.5,0)
\pcarc [arcangleA=290, arcangleB=290, linewidth=.5pt] (1.5,0)(-1.5,0)
\psline [linewidth=.5 pt](-1.5,0)(-0.5,0)
\psline [linewidth=.5 pt](0.5,0)(1.5,0)
\psline [arrows=->,linewidth=.5 pt] (1.5,0)(2.1,0)
\psline [linewidth=.5 pt](2.1,0)(2.5,0)
\end {pspicture}
\end {center}
In particular, if $\langle x_{i}\rangle\cap\langle x_{j}\rangle=\trivial$ for all $i\not=j$, then $M$ is isomorphic to the bipointed union of one-generated positive binoids.
\end {Proposition}
\begin {proof}
The property $x_{i}+x_{j}=\infty$ for all $i\not=j$ implies the positivity of $M$. Let $M_{k}\subseteq M$ be as defined in the proposition. If $y\in M_{k}$ and $z\in M_{k^{\prime}}$, $k\not=k^{\prime}$, with $y,z\not=0$, then $y+z=\infty$ by the construction of $M_{k}$. Hence, there exists by Proposition \ref {PropUnivPropPUnion} a canonical epimorphism
$$\bigcupbidot_{k=1}^{m}M_{k}\Rto M\pkt$$
The injectivity follows from the injectivity of the embeddings $M_{k}\embto M$ and the fact that $M_{k}\cap M_{k^{\prime}}=\trivial$ for $k\not=k^{\prime}$. The description for $\#I_{k}=1$ is clear from Corollary \ref {CorClassificationOnegenerated}.
 
For the description of $M_{k}$ when $\#I_{k}\ge2$, we may treat the $M_{k}$ separately. So assume that $\langle x_{i}\rangle\cap\langle x_{j}\rangle\not=\trivial$ for $i,j\in I_{k}$ and $M_{k}=M$. Again by Proposition \ref{PropUnivPropPUnion}, there exists a binoid epimorphism
 $$\bigcupbidot_{i\in I}\N^{\infty} \Rto M\quad\text{with}\quad (m_{i};i)\lto m_{i}x_{i}\pkt$$
Let $i\not=j$. By assumption, there are $n_{i},n_{j}\ge2$ such that $n_{i}x_{i}=n_{j}x_{j}\not=\infty$. It follows $(n_{i}+1)x_{i}=n_{j}x_{j}+x_{i}=\infty$ and similarly $(n_{j}+1)x_{j}=\infty$. In particular, $n_{i}$  is minimal with respect to the property $n_{i}x_{i}\not=\infty$ and $(n_{i}+1)x_{i}=\infty$, hence it is in particular independent of $j$. Therefore, we get by Lemma \ref{LemIndCong} a surjective factorization
$$\Big(\bigcupbidot_{i\in I}\N^{\infty}\Big)\Big/\big((n_{i};i)=(n_{j};j)\not=\infty\text{ for all }i,j\in I_{k}\big)\Rto M\komma$$
which is obviously injective on the components. If $rx_{i}=sx_{j}$ in $M$ for $i\not=j$, then $r\ge n_{i}$ and $s\ge n_{j}$ by the definition of the $n_{i}$s, so the homomorphism is injective. The supplement is clear.
\end {proof}

By the preceding proposition, those two-generated binoids whose generators $x$ and $y$ fulfill $x+y=\infty$ are up to isomorphism given by the following seven binoids. For $n,m\ge 1$ and $l,k\ge 2$:
$$\begin {array}{ll}
\N^{\infty}\cupbidot\,\N^{\infty}&\leadsto\quad K[X,Y]/(XY)\\
\N^{\infty}\cupbidot\,\N^{\infty}/(k=\infty)&\leadsto\quad K[X,Y]/(XY,Y^{k})\\
\N^{\infty}\cupbidot\,\N^{\infty}/(m=l)&\leadsto\quad K[X,Y]/(XY,Y^{m}-Y^{l})\\
\N^{\infty}/(k=\infty)\cupbidot\,\N^{\infty}/(l=\infty)&\leadsto\quad K[X,Y]/(XY,X^{k},Y^{l})\\
\N^{\infty}/(k=\infty)\cupbidot\,\N^{\infty}/(m=l)&\leadsto\quad K[X,Y]/(XY,X^{k},Y^{m}-Y^{l})\\
\N^{\infty}/(n=k)\cupbidot\,\N^{\infty}/(m=l)&\leadsto\quad K[X,Y]/(XY,X^{n}-X^{k},Y^{m}-Y^{l})\\
(\N^{\infty}\cupbidot\N^{\infty})/((k;1)=(l;2)\not=\infty)&\leadsto\quad K[X,Y]/(XY,X^{k}-Y^{l})
\end {array}$$

We close this section with a general study of properties of two-generated binoids.

\begin {Proposition}\label{PropTwoGenerated}
Let $M$ be a two-generated binoid with generating set $\{x,y\}$.
\begin {ListeTheorem}
\item If $M$ is not cancellative but contains a non-trivial cancellative element, then $M$ is integral if and only if one generator, say $x$, is regular and the other generator, $y$, is not nilpotent. In this case,
\begin {ListeTheorem}
\item [(a)]  $\opcan(M)^{\infty}=\langle x\rangle$.
\item [(b)] $M$ is torsion-free if and only if $ky$ is no torsion element for all $k\ge1$.
\item [(c)] $M$ is finite if and only if $M^{\times}=\langle x\rangle\opkt$ and $ny=my\not=\infty$ for some $n,m\ge1$. In this situation, $\langle x\rangle\cap\langle y\rangle=\trivial$.
\end {ListeTheorem}
\item [(2)] If $M$ is finite with no non-trivial cancellative elements, then $x$ and $y$ have a loop. If the initial pairs are $1\le n<m$ and  $1\le r<s$, then
\begin {ListeTheorem}
\item [(a)] $M$ is reduced if and only if $M$ is isomorphic to
$$\N^{\infty}/(r=s)\,\bigcupbidot\,\N^{\infty}/(n=m)$$
or $nx+ry\not=\infty$.
\item [(b)] $M$ is integral if and only if $nx+ry\not=\infty$.
\item [(c)] $M$ is torsion-free if and only if $M$ is boolean.
\end {ListeTheorem}
\end{ListeTheorem}
\end {Proposition}
\begin {proof}
(1) By assumption, there is a cancellative element $rx+sy\not=0$, $r,s\in\N$. We may assume that $r\not=0$, the case $s\not=0$ follows similarly. The implication $\Leftarrow$ is clear because every equation $rx+sy=\infty$, $r,s\in\N$, implies $sy=\infty$ since $x$ is integral, a contradiction to $y$ not nilpotent. So let $M$ be integral. Then every equation $x+a=x+b\not=\infty$, $a,b\in M$, implies $rx+sy+a=rx+sy+b\not=\infty$, and hence $a=b$ by the cancellativity of $rx+sy$. Therefore, $x$ is cancellative and $s=0$ because $M$ is not cancellative, cf.\ Lemma \ref{LemIntCanGenerators}. In particular, $\opcan(M)^{\infty}=\langle x\rangle$, which proves (a). Since $M$ is integral, $x$ is a regular element and $y$ not nilpotent.

(b) The implication $\Rightarrow$ is trivial. So let $ky$ be no torsion element for all $k\ge 1$ and suppose that $n(rx+sy)=n(lx+ky)$ for some $r,s,l,k\in\N$ and $n\ge2$. We may assume that $r\ge l$. By the cancellativity of $x$, we deduce from this equation that $n((r-l)x+sy)=nky$. Hence, $(r-l)x+sy=ky$ since $ky$ is no torsion element. Adding $lx$ gives the desired equality $rx+sy=lx+ky$. (c) is clear since $M$ is finite if and only if the subbinoids $\langle x\rangle$ and $\langle y\rangle$ are finite. For the supplement note that $sy=rx$ implies $y\in M\okreuz$, which yields a contradiction to $M$ being not cancellative. Hence, $\langle x\rangle\cap\langle y\rangle=\trivial$.

(2) Clearly, $x$ and $y$ have a loop if $M$ is finite. If $1\le n<m$ and $1\le r<s$ are the initial pairs, then every $a\in\ M\opkt$ can be written as $a=ix+jy$ with $0\le i<m$ and $0\le j<s$.

(a) First let $M$ be reduced. Obviously, $x$ and $y$ are not nilpotent. In case $x+y=\infty$, there is no other generating relation on $x$ and/or $y$ possible, so we obtain $M\cong\N^{\infty}/(r=s)\bigcupbidot\N^{\infty}/(n=m)$.
If $x+y\not=\infty$ and $k:=\max(r,n)$, then $\infty\not=k(x+y)=nx+ry+b$ for some $b\in M$. In particular,  $nx+ry\not=\infty$. Conversely, if $M$ is  isomorphic to the given bipointed union, then $M$ is reduced. So assume that $nx+ry\not=\infty$. Note that this implies that $x$ and $y$ are not nilpotent (otherwise the loops were given by $nx=\infty$ and $ry=\infty$). If $ka=\infty$ for some $k\ge 2$ and $\infty\not=a=ix+jy\in M$  with $1\le i<m$ and $1\le j<s$, then $\infty=ka=kix+kjy=i^{\prime}x+j^{\prime}y$, where $1\le i^{\prime}\le m$ and $1\le j^{\prime}\le s$. Hence, 
$$\infty=ka+(m-i^{\prime})x+(s-j^{\prime})y=mx+sy=nx+ry\komma$$
which contradicts our assumption.
(b) The implication $\Rarrow$ is trivial. So assume that $nx+ry\not=\infty$. If $M$ is not integral, there is an equation $ix+jy=\infty$, and by assumption $n\le i<m$ and $r\le j<s$ such that $(i,j)\not=(n,r)$. Again we get a contradiction by $\infty=ix+jy+(m-i)x+(s-j)y=nx+ry$. (c) The implication $\Leftarrow$ follows from Lemma \ref{LemTorsion}(2). Conversely, if $M$ is not boolean, then $s-r\ge2$ or $m-n\ge2$, which implies that $x$ or $y$ is a torsion element by Lemma \ref {LemLoop}.
\end {proof}

\begin {Example}
Let $l\ge2$. The binoid 
$$M:=\free(x,y)/(y=y+lx)$$
is obviously positive and integral, but not cancellative because $y$ is a non-cancellative element. On the other hand, $x$ is a cancellative element. For this note that the relation $y=y+lx$ implies $sy=sy+mlx$ for all $s,m\ge1$. In particular, every element $f=rx+sy\in M$ with $s\ge1$ can be written uniquely as $r^{\prime}x+s^{\prime}y$ with $r^{\prime}<l$. Since $\langle x\rangle\cong\N^{\infty}\cong\langle y\rangle$, we only need to consider an equation like
$$nx+my=rx\quad\text{with}\quad r,n,m\ge1\quad\text{and}\quad 1\le n<l\pkt$$
Write $r=r^{\prime}+kl$ with $r^{\prime}<l$ and $k\ge0$. Adding $y$ to the equation yields
$$nx+(m+1)y=rx+y=r^{\prime}x+klx+y=r^{\prime}x+y\komma$$
which implies $n=r^{\prime}$ and $m=0$. Hence, $n=r$ since $\langle x\rangle\cong\N^{\infty}$. $M$ is also not torsion-free because $l(x+y)=lx+ly=ly$ but $y+x\not=y$. The element $y$ is a so-called unseparated element in $M$ and its existence turns $M$ (in this case) in an unseparated binoid. Unseparatedness will be discussed in Section \ref{SecSepBinoids} in more detail.  
\end {Example}

\begin {Example}
The binoid
$$M:=\free(x,y)/(2x=\infty,x+y=x,2y=y)$$
is positive, not reduced (hence not integral), and contains no non-trivial cancellative element because both generators $x$ and $y$ are not cancellative and $M=\{0,\infty, x,y\}$.
\end {Example}

At the end of Section \ref {SecIntegrality}, we are able to describe finitely generated binoids that are regular or regular and positive in detail, cf. Proposition \ref{PropClassFgRegPos}.

\bigskip

\section {Localization} \label{SecLocalization}
\markright{\ref{SecLocalization} Localization }

We want to introduce the concept of localization to binoids since this is a very powerful tool and frequently used in commutative algebra with which we will be concerned while passing to the algebra of a binoid. Due to the fact that localization itself is a problem in non-commutative algebra we make the following agreement.

\begin {Convention}
In this section, arbitrary binoids are assumed to be \emph{commutative}.
\end {Convention}

\begin {Definition}
Let $M$ be a binoid and $S$ a submonoid of $M$. We define an equivalence relation on $M\times S$ as follows: for all $a,a^{\prime}\in M$ and $s,s^{\prime}\in S$ define
$$(a,s)\sim (a^{\prime},s^{\prime})\quad:\eq\quad\exists c\in S:a+s^{\prime}+c=a^{\prime}+s+c\pkt$$
The equivalence class of $(a,s)$ will be denoted by $a\minus s$. The set of all equivalence classes \nomenclature[M5]{$M_{S}$}{localization of $M$ at $S$}
$$M_{S}:=\{a\minus s\mid a\in M, s\in S\}$$
is the \gesperrt {localization} \index{localization}of $M$ at $S$. The localization at a submonoid that is generated by a single element $f\in M$ will be abreviated by $M_{f}$.
\end {Definition}

\begin {Lemma} \label{LemLocalization}
Let $M$ be a binoid and $S$ a submonoid of $M$. The addition 
$$(a\minus s)+(a^{\prime}\minus s^{\prime}):=(a+a^{\prime})\minus (s+s^{\prime})$$
defines a binoid structure on $M_{S}$ such that $\iota_{S}:M\Rto M_{S}$, $a\lto a\minus 0$, is a binoid homomorphism, where the elements of $S$ become units and $\ker\iota_{S}=\{a\in M\mid\exists b\in S: a+b=\infty\}\subseteq\nonint(M)$, which turns $M_{S}$ in an $M\mina$binoid.
\end {Lemma}
\begin {proof}
It is immediate to check that the operation is well-defined and that the map is a homomorphism. If $s\in S$, then $(s\minus0)+(0\minus s)=s\minus s\sim0\minus0$. Hence, $s\minus0\in M_{S}\okreuz$.
\end {proof}

By the preceding lemma, one has $M_{S}=M_{S+M\okreuz}$. Moreover, $M_{S}=M$ if and only if $S\subseteq M^{\times}$, and $M_{S}=\zero$ if and only if $\infty\in S$.

\begin {Proposition} \label{PropUnivPropLocalization}
Given a binoid homomorphism $\varphi:M\rto N$ and a submonoid $S\subseteq M$ with $\varphi(S)\subseteq N\okreuz$, there is a unique $M\minus\,$binoid homomorphism $\varphi_{S}:M_{S}\rto N$ such that the diagram
$$\xymatrix{
M\ar[r]^{\varphi}\ar[d]_{\iota_{S}}&N\\
M_{S}\ar[ur]_{\varphi_{S}}&}$$
commutes.
\end {Proposition}
\begin {proof}
Define $\varphi_{S}(a\minus s):=\varphi(a)+\minus\varphi(s)$ for $a\in M$ and $s\in S$, which is possible since $\varphi(s)\in N\okreuz$ by assumption. To show that this is well-defined let $a\minus s=b\minus t$ for $a,b\in M$ and $s,t\in S$. Then $a+t+c=b+s+c$ for some $c\in S$, which implies that $\varphi(a)+\varphi(t)+\varphi(c)=\varphi(b)+\varphi(s)+\varphi(c)$. Since $\varphi(S)\subseteq N\okreuz$, this is equivalent to $\varphi(a)+\minus\varphi(s)\,=\,\varphi(b)+\minus\varphi(t)$. Hence, $\varphi_{S}$ does not depend on the chosen representatives. It is an $M\mina$binoid homomorphism since $\varphi$ is a binoid homomorphism and $\varphi_{S}(a\minus 0)=\varphi(a)+\minus\varphi(0)=\varphi(a)$. The uniqueness follows since the commutativity of the diagram requires $\varphi_{S}(a\minus 0)=\varphi(a)$ and $\varphi_{S}(0\minus a)=\minus\varphi(a)$ for $s\in S$.
\end {proof}

The universal property of the localization has the following consequences.

\begin {Corollary}\label{CorNSpecLoc}
Let $M$ and $N$ be binoids. If $S\subseteq M$ is a submonoid, then 
$$N\minspec M_{S}\,\,\cong\,\,\{\varphi\in N\minspec M\mid\varphi(S)\subseteq N\okreuz\}$$
as semigroups.
\end {Corollary}
\begin {proof}
By Proposition \ref{PropIndHomNspec}, the canonical binoid homomorphism $\iota_{S}:M\rto M_{S}$ induces a semigroup homomorphism $\iota_{S}^{\ast}:N\minspec M_{S}\rto N\minspec M$ with $\varphi\mto\varphi\iota_{S}$ such that $\varphi\iota_{S}(S)\subseteq N\okreuz$. By Proposition \ref{PropUnivPropLocalization}, $\iota_{S}^{\ast}$ is an isomorphism.
\end {proof}

\begin {Corollary} \label {CorLocal}
If $S$ is a submonoid of $M$, then $M_{S}\cong M_{\opFilt(S)}$ as $M\minus\,$binoids. In particular, the localizations at the submonoids $S$ and $T$ of $M$ are isomorphic as $M\minus\,$binoids if $\opFilt(S)=\opFilt(T)$.
\end {Corollary}
\begin {proof}
Consider the canonical binoid homomorphisms $\iota_{S}:M\rto M_{S}$ and $\iota:M\rto M_{\opFilt(S)}$. Since $S\subseteq\opFilt(S)$, one has $\iota(S)\subseteq M_{\opFilt (S)}\okreuz$. On the other hand,  if $t\in\opFilt(S)$, say $t+u=s\in S$ for $u\in M$, then 
$$(t\minus 0)+(u\minus s)\,=\,(t+u)\minus s\,=\,0\minus0$$
with $u\minus s\in M_{S}$. Hence, $\iota_{S}(t)=t\minus 0\in M_{S}\okreuz$ for all $t\in\opFilt(S)$. Thus, $\iota_{S}$ and $\iota$ satisfy the assumptions of Proposition \ref{PropUnivPropLocalization}. So we obtain unique $M\mina$binoid homomorphisms $M_{S}\rto M_{\opFilt(S)}$, $a\minus s\mto \iota(a)+\minus\iota(s)=a\minus s$, and  $M_{\opFilt(S)}\rto M_{S}$, $a\minus s\mto \iota_{S}(a)+\minus\iota_{S}(s)=a\minus s$, which are obviously inverse to each other. The supplement follows from the first part.
\end {proof}

By the preceding result, the set of all localizations of a binoid $M$ is given by $\{M_{F}\mid F\in\Fcal(M)\}=:M_{\Fcal(M)}$\nomenclature[M1]{$M_{\Fcal(M)}$}{$=\{M_{F}\mid F\in\Fcal(M)\}$}. The operation
$$M_{F}\circ M_{G}:=M_{F\cap G}\komma$$
for $F,G\in\Fcal(M)$, turns $M_{\Fcal(M)}$ into a binoid \nomenclature[M2]{$M_{\Fcal(M),\circ}$}{$=(M_{\Fcal(M)}, \circ,\zero,M)$, where $M_{F}\circ M_{G}=M_{F\cap G}$}
$$M_{\Fcal(M),\circ}\,:=\,(M_{\Fcal(M)},\circ,M_{M},M_{M\okreuz})\komma$$
where $M_{M}=\zero$ and $M_{M\okreuz}=M$.

\begin {Corollary}
The canonical maps $F\leftrightarrow M_{F}$ are order preserving binoid isomorphisms between $\Fcal(M)_{\cap}$ and $M_{\Fcal(M),\circ}$.
\end {Corollary}
\begin {proof}
Only the injectivity of $F\mto M_{F}$ is not obvious, but this follows from Corollary \ref {CorLocal}.
\end {proof}

\begin {Corollary} \label{CorFGlocalization}
If $M$ is finitely generated, then so is $M_{S}$ for every submonoid $S\subseteq M$.
\end {Corollary}
\begin {proof}
Let $I$ be finite and $\{x_{i}\mid i\in I\}$ a generating set of $M$. By Corollary \ref{CorLocal}, we may assume that $S$ is a filter. Therefore, $S$ is the submonoid generated by $(x_{j})_{j \in J}$ for some subset $J\subseteq I$. If $a\minus s\in M_{S}$, where $a=\sum_{i\in I}n_{i}x_{i}$ and $s=\sum_{j\in J}m_{j}x_{j}$ with $n_{i}$, $m_{j}\in\N$, $i\in I$, $j\in J$, then 
$$a\minus s\,=\,\Big(\sum_{i\in I}n_{i}(x_{i}\minus 0)\Big)+\Big(\sum_{j\in J}m_{j}(0\minus x_{j})\Big)\pkt$$
This shows that $(x_{i}\minus 0)_{i\in I}$ and $(0\minus x_{j})_{j\in J}$ generate $M_{S}$.
\end {proof}

\begin {Example}\label{ExpDirectLimitLocalizations}
Let $S$ be a submonoid of the binoid $M$ and consider $S^{\prime}:=S/\sim$, where $f\sim g$ if $\opFilt(f)=\opFilt(g)$. By abuse of notation, we omit the brackets and write $f$ instead of $[f]$ for an element of $S^{\prime}$. Define a partial order on $S^{\prime}$ by 
$$f\ge g\,\text{ for }\,f,g\in S^{\prime}\quad:\eq\quad\opFilt(g)\subseteq\opFilt(f)$$
and denote by 
$$\varphi_{fg}:M_{g}\rto M_{f}\komma\quad a\minus s\lto a\minus s\komma$$
the canonical $M\mina$binoid homomorphism  for $f\ge g$. Note that $\varphi_{f0}:M\rto M_{f}$ (i.e.\ $f\ge0$) for all $f\in S^{\prime}$, and since $f+g\in S^{\prime}$ for every $f,g\in S^{\prime}$, we have $f+g\ge f,g$. Thus, $(M_{f},\varphi_{fg})_{f\ge g}$ is a directed system of $M\mina$binoids and
$$\varinjlim M_{f}\,=\,\biguplus_{f\in S^{\prime}}M_{f}\Big/\!\sim\komma$$
where $(a\minus s;f)\sim(b\minus t;g):\Leftrightarrow\varphi_{hf}(a\minus s)=\varphi_{hg}(b\minus t)$ for some $h\in S^{\prime}$ with $h\ge f,g$, together with the canonical $M\mina$binoid homomorphisms $$\varphi_{g}:M_{g}\rto\varinjlim M_{f}\komma\quad a\minus s\mto[(a\minus s;g)]\komma$$ 
$g\in S^{\prime}$, is the direct limit of this system, cf. Section \ref{SecLimits}. Since $\psi_{f}:M_{f}\rto M_{S}$ with $a\minus s\mto a\minus s$ is an $M\mina$binoid homomorphism for every $f\in S^{\prime}$ such that $\psi_{g}=\psi_{f}\varphi_{fg}$ if $f\ge g$, there is a unique $M\mina$binoid homomorphism $\psi:\varinjlim M_{f}\rto M_{S}$ with $\psi\varphi_{g}=\psi_{g}$, $g\in S^{\prime}$, cf.\ Lemma \ref{LemUnivPropColimit}. On the other hand, for $g\in S$ one has $\varphi_{0}(g)=[(g\minus0;g)]\in(\varinjlim M_{f})\okreuz$ because $$[(g\minus0;g)]+[(0\minus g;g)]\,=\,[(\varphi_{gg}(g\minus0)+\varphi_{gg}(0\minus g));g]\,=\,[(0\minus0;g)]\,=\,0\pkt$$ 
Hence, $S\subseteq(\varinjlim M_{f})\okreuz$. By Proposition \ref{PropUnivPropLocalization}, we therefore have a unique $M\mina$binoid homomorphism $\varphi:M_{S}\rto\varinjlim M_{f}$ with $\varphi\iota_{S}=\varphi_{0}$. Thus, for every $g\in S^{\prime}$ there is a commutative diagram
$$\xymatrix{
&M_{g}\ar[dr]^{\psi_{g}}\ar[dl]_{\varphi_{g}}&\\
\varinjlim M_{f}\ar@<1ex>[rr]^{\psi}&&M_{S}\pkt\ar@<1ex>[ll]^{\varphi}}$$
From this we deduce for $a\minus s\in M_{g}$,
$$(\varphi\psi)([(a\minus s;g)])=\varphi(\psi(\varphi_{g}(a\minus s)))=\varphi(\psi_{g}(a\minus s))=\varphi_{g}(a\minus s)=[(a\minus s;g)]$$
and 
$$(\psi\varphi)(a\minus s)=\psi(\varphi(\psi_{g}(a\minus s))=\psi(\varphi_{g}(a\minus s))=\psi_{g}(a\minus s)=a\minus s\pkt$$
This shows that $\psi$ and $\varphi$ are inverse to each other. Hence, $\varinjlim M_{f}\cong M_{S}$ as $M\mina$binoids.
\end{Example}

To describe the smash product (over $M$) of two localizations of $M$ recall the binoid structure on $\Fcal(M)$ given by $F\star G:=\opFilt(F+G)$ for $F,G\in\Fcal(M)$, in which $M\okreuz$ serves as the identity element and $M$ as the absorbing element.

\begin {Proposition}\label{PropLocalizationSmash}
Let $F,G\in\Fcal(M)$. Then
$$M_{F}\wedge_{M} M_{G}\,=\,M_{F\star G}\komma$$
where the localizations are $M\mina$binoids via the canonical binoid homomorphisms $\iota_{F}:M\rto M_{F}$ and $\iota_{G}:M\rto M_{G}$, both given by $x\rto x\minus0$. In particular, the canonical maps $F\leftrightarrow M_{F}$ are order preserving binoid isomorphisms between $\Fcal(M)_{\star}=(\Fcal(M),\star, M\okreuz,M)$ and $(M_{\Fcal(M)},\wedge,M,\zero)$.
\end {Proposition}
\begin {proof}
By Corollary \ref{CorCoProdSmashN}, the $M\mina$binoid homomorphisms $\psi_{F}:M_{F}\rto M_{F\star G}$ with $x\minus f\mto x\minus f$ and $\psi_{G}:M_{G}\rto M_{F\star G}$ with $y\minus g\mto y\minus g$ induce an $M\mina$binoid homomorphism 
$$\psi:M_{F}\wedge_{M} M_{G}\Rto M_{F\star G}\komma\quad(x\minus f)\wedge_{M} (y\minus g)\lto\psi_{F}(x\minus f)+\psi_{G}(y\minus g)=(x+y)\minus(f+g)\komma$$
where $x,y\in M$, $f\in F$, and $g\in G$. On the other hand, we have by the universal property of localization, cf.\ Proposition \ref {PropUnivPropLocalization}, an $M\mina$binoid homomorphism 
$$\phi:M_{F\star G}\Rto M_{F}\wedge_{M} M_{G}\pkt$$ 
To see this consider the binoid homomorphism $\varphi:M\rto M_{F}\wedge_{M} M_{G}$ with 
\begin{equation}
\varphi(x)=(x\minus 0)\wedge_{M}(0\minus0)=\iota_{F}(x)\wedge_{M}(0\minus 0)=(0\minus 0)\wedge_{M}\iota_{G}(x)=(0\minus 0)\wedge_{M}(x\minus0) \tag{$\ast$}\label{DefVarphi}
\end{equation}
for $x\in M$. We want to show that $\varphi$ factors through $\phi$ by applying Proposition \ref {PropUnivPropLocalization}. For this we need to verify that $\varphi(F\star G)\subseteq(M_{F}\wedge_{M}M_{G})\okreuz$. So take an arbitrary $x\in F\star G$. Then $x+u=f+g$ for some $u\in M$, $f\in F$, and $g\in G$. Hence,
\begin {align*}
((x\minus 0)\wedge_{M}(0\minus 0))+((u\minus f)\wedge_{M} 0\minus g)&=((x+u)\minus f)\wedge_{M} (0\minus g)\\
&=((f+g)\minus f)\wedge_{M}(0\minus g)\\
&=((g\minus0)+(0\minus0))\wedge_{M}(0\minus g)\\
&=(0\minus0)\wedge_{M}((g\minus0)+(0\minus g))\\
&=(0\minus0)\wedge_{M}(0\minus0)\\
&=0_{\wedge_{M}}\pkt
\end {align*}
Hence, $\varphi(x)\in(M_{F}\wedge_{M}M_{G})\okreuz$ with inverse $-\varphi(x)=(u\minus f)\wedge_{M}(0\minus g)$. The  $M\mina$binoid homomorphism $\phi$ is therefore given by 
$$\phi(z\minus x)=\varphi(z)+(-\varphi(x))\pkt$$ 
We claim that $\phi$ and $\psi$ are inverse to each other. First consider an element $z\minus x\in M_{F\star G}$, where $x+u=f+g$ with $u\in M$, $f\in F$, and $g\in G$. Then
\begin {align*}
\psi\phi(z\minus x)&=\psi(\varphi(z))+\psi(-\varphi(x))\\
&=\psi(z\wedge_{M}0)+\psi((u\minus f)\wedge_{M}(0\minus g))\\
&=(z\minus 0)+(u\minus(f+g))\\
&=(z\minus 0)+(0\minus x)\\
&=z\minus x\pkt
\end {align*}
By making use of \eqref{DefVarphi}, we finally obtain for $x,y\in M$, $f\in F$, and $g\in G$,
\begin {align*}
 \phi\psi((x\minus f)\wedge_{M}(y\minus g))&=\phi((x+y)\minus(f+g))\\
&=\varphi(x+y)+(-\varphi(f+g))\\
&=((x\minus 0)\wedge_{M}(y\minus 0))+((0\minus f)\wedge_{M}(0\minus g))\\
&=(x\minus f)\wedge_{M}(y\minus g)\komma
\end {align*}
where $-\varphi(f+g)=(0\minus f)\wedge_{M}(0\minus g)$ because $\varphi(f+g)=((f+g)\minus0)\wedge_{M}(0\minus0)=(f\minus0)\wedge_{M}(g\minus0)$
in $M_{F}\wedge_{M}M_{G}$.
\end {proof}

\begin {Definition}
Let $M$ be a binoid and $N$ the submonoid of $M$ that is generated by $M\opkt$. The binoid $M_{N}=:\diff(M)$\nomenclature[Diff]{$\diff(M)$}{difference binoid of the binoid $M$} will be called the \gesperrt{difference binoid} \index{difference binoid}\index{binoid!difference --}of $M$.
\end {Definition}

In other words, if $M$ is integral, then $\diff(M)=M_{M\opkt}$ is a binoid group and $\diff(M)=\zero$ otherwise.

\begin {Proposition} \label{PropGroupEmbedding}
Let $S$ be a submonoid of $M$. The binoid homomorphism $\iota_{S}:M\rto M_{S}$ is injective if and only if $S$ consists only of regular elements. In particular, every regular binoid $M$ is embedded into a binoid group, namely $\diff(M)$.
\end {Proposition}
\begin {proof}
Clear, because $\iota_{S}$ is injective if and only if $a+s=b+s$ for $a,b\in M$, $s\in S$, implies $a=b$.
\end {proof}

In the regular situation, we identify $a\in M$ with $a\minus0\in M_{S}$.

\begin {Corollary} \label{CorUnivPropDiff}
Let $M$ be a binoid and $G$ a binoid group. Every binoid homomorphism $\varphi:M\rto G$ with $\varphi(a)\not=\infty$ for $a\in M\opkt$ gives rise to an $M\mina$binoid homomorphism $\tilde{\varphi}:\diff(M)\rto G$ with $\tilde{\varphi}\iota=\varphi$, where $\iota:M\rto\diff(M)$ denotes the canonical binoid homomorphism. In particular, for every binoid homomorphism $\varphi:M\rto G$ there exists a unique injective factorization 
$$\xymatrix{
M\ar[r]^{\varphi}\ar[d]_{\pi}&G\\
N\ar[d]_{\pi_{\opcan}}\\
N_{\opcan}\ar[d]_{\iota}\\
\diff(N_{\opcan})\ar[uuur]_{\tilde{\varphi}}&}\komma$$
where $N:=M/\ker\varphi$.
\end {Corollary}
\begin {proof}
Since every binoid group is regular, there exists no such binoid homomorphism $\varphi$ for a non-integral binoid $M$. Therefore, we only need to consider the case of $M$ being integral; that is, when $\diff(M)=M_{M\opkt}$. If $M$ is integral, the assumption $\varphi(a)\not=\infty$ for $a\in M\opkt$ is equivalent to $\varphi(M\opkt)\subseteq G\okreuz$. So we can apply Proposition \ref{PropUnivPropLocalization}, which provides the unique $M\mina$binoid homomorphism $\tilde{\varphi}$ with $\tilde{\varphi}\iota=\varphi$. This proves the first statement. By Proposition \ref{PropCancellation}(2), there is a factorization $\psi:N_{\opcan}\rto G$, which fulfills the assumption of the first statement since the integrality of $G$ implies that of $N$ by Lemma \ref{LemFactoriazationKer}, and hence that of $N_{\opcan}$. The injectivity of the factorization $\diff(N_{\opcan})\rto G$ follows from Proposition \ref{PropGroupEmbedding}.
\end {proof}

\begin {Lemma} \label{LemDGroupTorsionfree}
Let $M$ be a regular binoid. The binoid group $\diff(M)$ is torsion-free if and only if $M$ is torsion-free.
\end {Lemma}
\begin {proof}
Any equation $nx=ny$ for some $x,y\in M$ and $n\ge1$ in $M$ is equivalent to $n(x\minus 0)=n(y\minus0)$ in $\diff(M)$, which implies that $x=y$ because $M$ is cancellative and $\diff(M)$ is torsion-free. For the converse assume that $n(x\minus s)=0\minus 0$ in $\diff(M)\opkt$. Since $M$ is cancellative, this is equivalent to $nx=ns$, hence $x=s$ because $M$ is torsion-free, and so $x\minus s=0\minus 0=0$.
\end {proof}

\bigskip

\section {Integrality} \label{SecIntegrality}
\markright{\ref{SecIntegrality} Integrality }

In this section, we introduce integral binoid homomorphisms and related objects such as the integral closure and the normalization of a (regular) binoid.

\begin {Definition}
A binoid homomorphism $\varphi:M\rto N$ is called \gesperrt{integral} \index{integral!-- homomorphism}\index{homomorphism!integral --}if for every $a\in N$ there is an integer $k\ge1$ such that $ka\in\varphi(M)$. In this case, we say $N$ is \gesperrt{integral over} \index{integral!-- over}\index{binoid!-- integral over another binoid}$M$ (via $\varphi$).
\end {Definition}

\begin {Lemma}\label{LemIntHomUnits}
Let $M$ be a subbinoid of the commutative binoid $N$. If $N$ is integral over $M$ with respect to the inclusion $\iota:M\embto N$, then $N\okreuz\cap M=M\okreuz$.
\end {Lemma}
\begin {proof}
Only the inclusion $\subseteq$ is not obvious. So assume that $a\in N\okreuz\cap M$ and denote by $\minus a\in N$ the inverse of $a$. By assumption, there is a $k\ge 1$ with $k(\minus a)\in M$. Then $c:=k(\minus a)+(k-1)a\in M$ satisfies $a+c=0$, hence $a\in M\okreuz$.
\end {proof}

\begin {Definition}
Let $M$ be a subbinoid of the binoid $N$. The \gesperrt{integral closure} \index{integral!-- closure}\index{binoid!integral closure of a --}of $M$ in $N$ is the binoid $\overline{M}^{N}:=\{a\in N\mid na\in M$ \nomenclature[M]{$\overline{M}^{N}$}{integral closure of $M$ in $N$}for some integer $n\ge1\}$. $M$ is \gesperrt{integrally closed} \index{integrally closed}\index{binoid!integrally closed --}in $N$ if $\overline{M}^{N}=M$. If $M$ is regular, we call $\overline{M}^{\diff(M)}$ the \gesperrt{normalization} \index{normalization}\index{binoid!normalization of a --}of $M$, which will usually be abbreviated with $\overline{M}$\nomenclature[M]{$\overline{M}$}{normalization of $M$}. We say $M$ is \gesperrt{normal} \index{normal}\index{binoid!normal --}if $M=\overline{M}$.
\end {Definition}

The integral closure is a subbinoid of $N$ because if $a,b\in\overline{M}^{N}$ with $na,mb\in M$ for some $n,m\ge1$, then $mn(a+b)\in M$ and so $a+b\in\overline{M}^{N}$.

\begin {Proposition} \label{PropLSpecIntegral}
Let $M$ be a subbinoid of the binoid $N$ such that $N$ is integral over $M$ with respect to the inclusion $\iota:M\embto N$ and let $L$ denote an arbitrary binoid. The induced semigroup homomorphism 
$$\iota^{\ast}:L\minspec N\Rto L\minspec M\komma\quad\psi\lto\psi\iota\komma$$
is a semigroup embedding if $L$ is torsion-free. If $N$ is commutative and $L$ boolean, then $\iota^{\ast}$ is a semibinoid isomorphism and, in particular,
$$L\minspec\overline{M}^{N}\,\,\cong\,\, L\minspec M$$
as semibinoids.
\end {Proposition}
\begin {proof}
The induced semigroup homomorphism $\iota^{\ast}$ comes from Proposition \ref{PropIndHomNspec}. To prove the injectivity for $L$ torsion-free suppose that $\psi,\psi^{\prime}\in L\minspec N$ with $\psi_{|M}=\psi^{\prime}_{|M}$. By assumption, there is for every $a\in N\setminus M$ an integer $k\ge 2$ such that $ka=b\in M$. Then $k\psi(a)=\psi(b)=\psi^{\prime}(b)=k\psi^{\prime}(a)$, and hence $\psi(a)=\psi^{\prime}(a)$ since $L$ is torsion-free.

Now let $N$ be commutative and $L$ boolean. For the surjectivity of $\iota^{\ast}$ let $\varphi\in L\minspec M$ be given. Since the inclusion $\iota:M\embto N$ is integral and injective, there is for every $a\in N$ a unique element $b\in M$ such that $ka=b$, where we can choose the integer $k\ge1$ to be minimal with respect to the property $na\in M$, $n\ge1$. This yields a well-defined map $\tilde{\varphi}:N\rto L$ with $\tilde{\varphi}(a):=\varphi(b)$, where $b\in M$ is this unique element for $a\in N$. By definition, $\tilde{\varphi}\iota=\varphi$. To see that $\tilde{\varphi}$ is a binoid homomorphism let $a,a^{\prime}\in N$ with $ka=b$, $k^{\prime}a^{\prime}=b^{\prime}$, and $l(a+a^{\prime})=c$ with $b,b^{\prime},c\in M$, and $k,k^{\prime},l\ge1$ minimal so that $\tilde{\varphi}(a)=\varphi(b)$, $\tilde{\varphi}(a^{\prime})=\varphi(b^{\prime})$, and $\tilde{\varphi}(a+a^{\prime})=\varphi(c)$. Since $b,b^{\prime},c$, and all their multiples lie in $M$ and $L$ is boolean, we obtain
\begin {align*}
\tilde{\varphi}(a+a^{\prime})=\varphi(c)=kk^{\prime}\varphi(c)=\varphi(kk^{\prime}c)=\tilde{\varphi}(kk^{\prime}c)&=\tilde{\varphi}(kk^{\prime}la+kk^{\prime}la^{\prime})\\ 
&=\tilde{\varphi}(k^{\prime}lb+klb^{\prime})\\
&=\varphi(k^{\prime}lb+klb^{\prime})\\
&=k^{\prime}l\varphi(b)+kl\varphi(b^{\prime})\\
&=\varphi(b)+\varphi(b^{\prime})\\
&=\tilde{\varphi}(a)+\tilde{\varphi}(a^{\prime})\pkt
\end {align*}
By Lemma \ref{LemPosHom}, $L\minspec M$ and $L\minspec N$ are boolean semibinoids with absorbing elements $\chi_{M\okreuz}:M\rto L$ and $\chi_{N\okreuz}:N\rto L$, respectively. By Lemma \ref{LemIntHomUnits}, $\chi_{N\okreuz}\iota=\chi_{N\okreuz\cap M}=\chi_{M\okreuz}$, which shows that $\iota^{\ast}$ is a semibinoid isomorphism. The supplement is obvious.
\end {proof}

\begin {Example}
The binoid homomorphism $\N^{\infty}\rto\N^{\infty}$ with $n\mto kn$ is integral. If $L=\C$ and $k\ge2$, the corresponding semigroup homomorphism
$$\C\minspec\N^{\infty}\Rto\C\minspec\N^{\infty}$$
is surjective but not injective. In case $L=\R$ and $k\ge2$, the corresponding semigroup homomorphism is neither surjective nor injective.

For an algebraic closed field $L$, every integral binoid embedding $M\embto N$ yields an integral $L\mina$algebra homomorphism $L[M]\embto L[N]$, and therefore a semigroup epimorphism between the $L\mina$spectra. See the subsequent discussion of \cite[Example 2.B.16]{PatilStorch}.
\end {Example}

Now we will state well-known structure theorems on finitely generated commutative monoids translated to binoids, cf.\ \cite{GarciaRosales}. Recall that an \gesperrt{affine monoid} \index{monoid!affine --}is a monoid $M$ which is finitely generated, cancellative, and torsion-free. Equivalently, $M$ is a submonoid of $\Z^{d}$ for some $d\ge0$, cf.\ \cite[Chapter 2.A]{BrunsGubeladze}.

\begin {Proposition}\label{PropClassFgRegPos}
Let $M$ be a finitely generated binoid.
\begin {ListeTheorem}
\item $M$ is regular if and only if it is isomorphic to a subbinoid of $(\Z^{s}\times\prod_{k=1}^{r}\Z/n_{k}\Z)^{\infty}$ for some $s\ge0$ and $n_{1}\kpkt n_{r}\ge 2$.
\item $M$ is regular and torsion-free if and only if it is isomorphic to a subbinoid of $(\Z^{s})^{\infty}$ for some $s\ge0$.
\item $M$ is positive, regular, and torsion-free if and only if it is isomorphic to a subbinoid of $(\N^{d})^{\infty}$ for some $d\ge0$.
\item If $M$ is regular and positive, then it is isomorphic to a subbinoid of 
$(\N^{d}\times\prod_{k=1}^{r}\Z/n_{k}\Z)^{\infty}$ for some $d\ge1$ and $n_{1}\kpkt n_{r}\ge 2$.
\end {ListeTheorem}
\end {Proposition}
\begin {proof}
(1) By Proposition \ref{PropGroupEmbedding} and Corollary \ref{CorFGlocalization}, a finitely generated binoid is regular if and only if it can be embedded into a finitely generated abelian binoid group, namely $\diff(M)$, which is given as displayed by the structure theorem for finitely generated abelian groups. (2) follows immediately from Lemma \ref{LemDGroupTorsionfree} and (1). 

(3) By (2), $M$ is positive, regular, and torsion-free if and only if $M\opkt$ is a positive affine monoid which is equivalent to $M\opkt$ is a submonoid of $\N^{d}$ for some $d\ge0$ by  \cite[Proposition 2.17]{BrunsGubeladze}.

(4) If $M$ is regular and positive, we may identify $M$ with its image under the canonical embedding $\iota:M\embto\diff(M)=\Z^{s}\times T$, where $T$ is a torsion group, cf.\ (1). If
$$F:=\{f\in\Z^{s}\mid (f,t)\in M\text{ for some }t\in T\}\komma$$
then $M\opkt=F\times T$. Clearly, $F$ is a submonoid of $\Z^{s}$, that is to say, $F$ is an affine monoid. Furthermore, $F$ is positive since $M=(F\times T)^{\infty}$ is so. Indeed, if $x,\minus x\in F$, then $(x,t)$, $(\minus x,t^{\prime})\in M$ for some $t,t^{\prime}\in T$. With $n:=\ord T$ this yields
$$n((x,t)+(\minus x,t^{\prime}))=n(0,t+t^{\prime})=(0,0)\komma$$
which is a contradiction to $M$ positive. Thus, $F$ is isomorphic to a submonoid of $\N^{d}$ for some $d\ge0$ by (3).
\end {proof}

\bigskip

\chapter {Ideal theory in commutative binoids} \label {ChapIdealTheory}
\markright{\ref{ChapIdealTheory} Ideal theory in commutative binoids}

In this chapter, we develop the ideal theory of commutative binoids in a manner parallel to that of rings. The main differences are that in the ideal theory of binoids the union of (prime) ideals is again a (prime) ideal and that every binoid is local in the sense that it admits a unique maximal ideal (with respect to $\subseteq$), which is prime. These two characteristics of binoids account for the main deviations from the ideal theory of rings, cf.\ Remark \ref{RemPrimeUnionConsequence1} and Remark \ref{RemPrimeUnionConsequence2}. Introducing the (Rees) quotient by an ideal as an important tool emphasizes again the benefit of an absorbing element. There are one-to-one corrrespondences between the prime ideals, filters, and $\trivial\mina$valued binoid homomorphisms of a binoid from which the spectrum of particular binoids can easily be deduced by taking advantage of the more general descriptions of $N\mina$spec given in Chapter \ref{Chap1Basics}. A treatment of the spectrum as a topological space will follow in Section \ref{SecTopology}. We close this section with the study on minimal prime ideals.

\medskip

The similarities and differences between the ideal theory of binoids and the that of rings have been intensively studied in \cite{AndersonIT}. An elaboration on the ideal theory of monoids can be found in \cite {Kobsa}. In both publications, the theory of binoids and monoids that satisfy the ascending chain condition on ideals and the theory of primary ideals and primary decompositions have been established to a great part.

\begin {Convention}
In this chapter, arbitrary binoids are assumed to be \emph{commutative}.
\end {Convention}

\bigskip

\section {Ideals} \label {SecIdeals}
\markright {\ref{SecIdeals} Ideals}

\begin {Definition} \label{DefIdeal}
Let $M$ be a binoid. An \gesperrt{ideal} \index{ideal}in $M$ is a subset $\Ical\subseteq M$ with $\infty\in\Ical$ and $x+M\subseteq\Ical$ for all $x\in\Ical$. 
\end {Definition}

Clearly, $M$ itself is an ideal, the \gesperrt{unit ideal}\index{ideal!unit --}, and an ideal coincides with $M$ if and only if it contains a unit. The condition $\infty\in\Ical$ is equivalent to $\Ical\not=\emptyset$. By definition, a nonempty subset $\Ical$ of $M$ is an ideal if and only if $\Ical$ is an $M\mina$set with respect to the addition on $M$ restricted to $\Ical$:
$$M\times\Ical\Rto\Ical\komma\quad (x,a)\lto x+a\pkt$$

\begin {Example}  \label {ExSpIdeals}
Some sets considered before turn out to be ideals.
\begin {ListeTheorem}
\item The subset $\zero$ is an ideal in every binoid $M$.
\item The set of all nonunits $M\Uplus=M\setminus M\okreuz$ of a nonzero binoid $M$ is an ideal in $M$ since $a+b\not\in M\okreuz$ if $a$ or $b\in M\Uplus$. 
\item The kernel of a binoid homomorphism $\varphi:M\rto N$ is an ideal in $M$ because $a\in\ker\varphi$ and $b\in M$ implies $\varphi(a+b)=\varphi(a)+\varphi(b)=\infty$. This is a special case of Lemma \ref{LemIdealCorrespondence}(1) below (with $\Jcal=\zero$). For instance, if $(S,p)$ is an $N\mina$set, the set $\{a\in N\mid a+s=p$ for all $s\in S\}\subseteq N$ is an ideal in $N$ since it is the kernel of the binoid homomorphism $N\rto(\map_{p}(S,S),\circ,\id,\varphi_{\infty})$, cf.\ Lemma \ref{LemNsetHom}. Also, the kernel of the canonical projection $\pi:M\rto M/\!\sim$ is an ideal for every ideal congruence $\sim$. In particular, for a nonzero binoid $M$ the set of all non-integral elements $\nonint(M)$ and the set of all nilpotent elements $\nil(M)$ are ideals in $M$, cf.\ Lemma \ref{LemIdealCong}. Algorithms for computing $\nonint(M)$ and $\nil(M)$ (in terms of generators of an ideal, see below) for a finitely generated binoid $M$ are given in \cite[Algorithm 9 and Algorithm 12]{Rosales}.
\item The set $\Ical+\Jcal=\{a+b\mid a\in\Ical, b\in\Jcal\}$ for two ideals $\Ical,\Jcal\subseteq M$ is an ideal in $M$. In particular, for any $n\ge0$,
$$n\Ical=\Ical\pluspkt\Ical=\{a_{1}\pluspkt a_{n}\mid a_{i}\in\Ical\}$$ is an ideal (note that $0\Ical=M$ according to our convention that the empty sum is $0$).
\item Clearly, every ideal in a ring $R$ is an ideal in $(R,\cdot,1,0)$. The converse is false. For instance, the set $\{a\mid a=0$ or $|a|\ge10\}$ is an ideal in the binoid $(\Z,\cdot,1,0)$ but not an ideal in the ring $\Z$.
\end {ListeTheorem}
\end {Example}

\begin {Definition}
We refer to $M\Uplus$ as the \gesperrt{maximal ideal} \index{ideal!maximal --}\index{maximal ideal}of $M$ because it is the largest ideal $\not=M$ in $M$ with respect to $\subseteq$. A binoid homomorphism $\varphi:N\rto M$ is \gesperrt{local} \index{homomorphism!local binoid --}\index{binoid homomorphism!local --}if $\varphi(N\Uplus)\subseteq M\Uplus$.
\end {Definition}

\begin {Lemma} \label {LemIdealCorrespondence}
Let $\varphi:M\rto N$ be a homomorphism of binoids.
\begin {ListeTheorem}
\item If $\Jcal$ is an ideal in $N$, then $\varphi^{-1}(\Jcal)$ is an ideal in $M$.
\item If $\Ical$ is an ideal in $M$, then $\varphi(\Ical)+N=\{a+b\mid a\in\varphi(\Ical), b\in N\}$ is an ideal in $N$.
\end {ListeTheorem}
\end {Lemma}
\begin {proof}
Both assertions are easily verified.
\end {proof}

\begin {Definition}
The ideal $\varphi(\Ical)+N$ \nomenclature[APhi0]{$\varphi(\Ical)+N$}{extended ideal of $\Ical\subseteq M$ with respect to $\varphi:M\rto N$}from Lemma \ref {LemIdealCorrespondence} is called the \gesperrt{extended ideal} \index{ideal!extended --}of $\Ical$ by $\varphi$. 
\end {Definition}

\begin {Corollary}
Let $M$ be a binoid and $S$ a submonoid of $M$.
\begin {ListeTheorem}
\item If $\Ical$ is an ideal in $M_{S}$, then $\iota_{S}^{-1}(\Ical)$ is an ideal in $M$.
\item If $\Ical$ is an ideal in $M$, then $\Ical_{S}:=\{a\minus s\mid a\in\Ical, s\in S\}$ is an ideal in $M_{S}$ such that  $\Ical_{S}=M_{S}$ if and only if $S\cap\Ical\not=\emptyset$.
\end {ListeTheorem}
\end {Corollary}
\begin {proof}
These are immediate consequences of Lemma \ref {LemIdealCorrespondence}. For (2) note that $\Ical_{S}=\iota_{S}(\Ical)+M_{S}$.
\end {proof}

To describe the structure of the set of ideals of a binoid, we need the following definition.

\begin {Definition}
A \gesperrt{lattice} \index{lattice}$L$ is a set that is both a join- and meet-semilattice; that is, $(L,\cup)$ and $(L,\cap)$ are boolean commutative semigroups, which are related by the absorption laws $a\cup(a\cap b)=a$ and $a\cap(a\cup b)=a$ for all $a,b\in L$. A lattice is \gesperrt{complete} \index{lattice!complete --}if any subset has a least upper and a greatest lower bound with respect to the partial order $\subseteq$ on $L$ defined by $a\subseteq b:\eq a\cup b=b$ (or $a\subseteq b:\eq a\cap b=a$). In particular, the lattice itself is bounded by a largest and a smallest element.
\end {Definition}

\begin {Lemma} \label {LemUnionIdeals}
The union and the intersection of ideals in a binoid $M$ is again an ideal in $M$; that is, the set of all ideals is a complete lattice, partially ordered by set inclusion with largest element $M$ and smallest element $\zero$.
\end {Lemma}
\begin {proof}
Let $(\Ical_{k})_{k\in J}$ be a family of ideals in $M$. If $x\in\bigcup_{k\in J}\Ical_{k}$, then $x\in\Ical_{l}$ for some $l\in J$. Hence,  $x+M\subseteq\Ical_{l}\subseteq\bigcup_{k\in J}\Ical_{k}$. If $x\in\bigcap_{k\in J}\Ical_{k}$, then $x\in\Ical_{k}$ for all $k\in J$. Hence, $x+M\subseteq\Ical_{k}$ for all $k\in J$, which means $x+M\subseteq\bigcap_{k\in J}\Ical_{k}$.
\end {proof}

The subset of all proper ideals in a binoid $M$ also defines a complete lattice, partially ordered by set inclusion, because $M\Uplus$ is a largest element within this set. If $R$ is a ring, the lattice of ideals of $(R,\cdot,1,0)$ is the set of all unions of ideals in the ring $R$. For instance, the ideal $\{a\mid a=0$ or $|a|\ge10\}$ in the binoid $(\Z,\cdot,1,0)$ is the union of all ideals $n\Z$, $n\ge 10$, in the ring $\Z$.

\begin {Remark}
Those lattices that arise up to isomorphism as the lattice of ideals of a binoid have been characterized by D. D. Anderson and E. W. Johnson in \cite[Theorem 2.4]{AndersonIT}.
\end {Remark}

\begin {Example}  \label {ExpPsetIdeals}
Let $V$ be an arbitrary set. The ideal lattices of the positive binoids $\Pset(V)_{\cap}$ and $\Pset(V)_{\cup}$ can be described as follows: a subset $\Acal\subseteq\Pset(V)$ is an ideal in $\Pset(V)_{\cap}$ if $A\in\Acal$ and $B\in\Pset(V)$ implies $A\cap B\in\Acal$. Since $A\cap B\subseteq A$ for all $B\in\Pset(V)$, all subsets of $A$ need to be contained in $\Acal$. Hence, $\Acal\subseteq\Pset(V)$ is an ideal in $\Pset(V)_{\cap}$ if and only if it is subset-closed; that is, the lattice of ideals is given by $(\Scal(V),\cap,\cup)$. Similarly, $\Acal\subseteq\Pset(V)$ is an ideal in $\Pset(V)_{\cup}$ if and only if $\Acal$ is superset-closed; that is, the lattice of ideals is given by $(\Scal^{\opc}(V),\cap,\cup)$.
\end {Example}

Note that if $V$ is finite, every ideal $\Acal$ in $\Pset(V)_{\cap}$ (resp.\ in $\Pset(V)_{\cup}$) is uniquely determined by the maximal (resp.\ minimal) sets in $\Acal$ with respect to $\subseteq$.

\begin {Definition}
By Lemma \ref{LemUnionIdeals}, there exists for every set of elements $V\subseteq M$ a smallest ideal in $M$ containing $V$, which will be denoted by $_{_{M}}\langle V\rangle$\nomenclature[AGenerate2]{$_{_{M}}\langle V\rangle$}{ideal in $M$ generated by $V\subseteq M$} or simply by $\langle V\rangle$ unless the context requires clarification. We say $\langle V\rangle$ is the ideal \gesperrt{generated} \index{ideal!-- generated by}by $V$ and call the elements of $V$ \gesperrt{generators} \index{ideal!generators of an --}of $\langle V\rangle$. A set $V\subseteq M$ is a \gesperrt{minimal} \index{ideal!generators of an --!minimal set of --}set of generators of an ideal $\Ical\subseteq M$ if $\langle V\rangle=\Ical$ and no proper subset of $V$ generates $\Ical$. An ideal $\Ical$ is \gesperrt{finitely generated} \index{ideal!finitely generated --}if $\Ical=\langle V\rangle$ for a finite subset $V\subseteq M$. A \gesperrt{principal ideal} \index{ideal!principal --}is an ideal generated by a 
singleton.
\end {Definition}

\begin {Example}
The extended ideal of $\Ical\subseteq M$ by $\varphi:M\rto N$ is the ideal in $N$ that is generated by $\im\Ical$. 
\end {Example}

\begin {Example}\label {ExpPsetIdealsGen}
Let $V$ be a finite set. A minimal set of generators of an ideal $\Acal\subseteq\Pset(V)_{\cap}$ is given by the maximal sets contained in $\Acal$ with respect to set inclusion, i.e. $\Acal=\langle \max_{\subseteq}\Acal\rangle$. Similarly, $\min_{\subseteq}\Acal$ is a minimal set of generators for every ideal $\Acal\subseteq\Pset(V)_{\cup}$. In particular, $(\Pset(V)_{\cap})\Uplus=\Pset(V)\setminus\{V\}=\langle V\setminus\{v\}\mid v\in V\rangle$ and $(\Pset(V)_{\cup})\Uplus=\Pset(V)\setminus\{\emptyset\}=\langle v\mid v\in V\rangle$.  For the principal ideals there are one-to-one correspondences, namely
$$\begin {array} {rcccl}
\{\text{Principal ideals of }\Pset(V)_{\cap}\}\!\!\!&\longleftrightarrow&\!\!\!\Pset_{n}\!\!\!&\longleftrightarrow&\!\!\!\{\text{Principal ideals of }\Pset(V)_{\cup}\}\komma\\
\langle J\rangle=\{A\mid A\subseteq J\}\!\!\!&\longleftrightarrow&\!\!\! J\!\!\!&\longleftrightarrow&\!\!\!\{A\mid J\subseteq A\}=\langle J\rangle
\end {array}$$
where the first is order preserving and the latter reversing.
\end {Example}

For the next result recall that the \gesperrt{product order} \index{product order}$\le$ on $\N^{n}$ is defined by $(a_{1}\kpkt a_{n})\le(b_{1}\kpkt b_{n})$ if $a_{i}\le b_{i}$ for all $i\in\{1\kpkt n\}$.

\begin {Proposition}
Every ideal in $(\N^{n})^{\infty}$, $n\ge 0$, is finitely generated. In particular, the ideals of a finitely generated binoid are finitely generated.
\end {Proposition}
\begin {proof}
It is clear that every ideal $\Ical$ in $(\N^{n})^{\infty}$ is generated by the minimal elements of $\Ical$ with respect to the product order. This is a finite set by Dickson's Lemma\index{Dickson's Lemma}, cf.\ \cite[Theorem 5.1]{GarciaRosales} or \cite[Lemma 8.6]{GarciaRosalesNS}, which proves the first statement. Let $\Jcal$ be an ideal in an $n$-generated binoid $M$. The preimage $\Ical:=\varepsilon^{-1}(\Jcal)$
under the canonical binoid epimorphism $\varepsilon:(\N^{n})^{\infty}\rto M$ is an ideal in $(\N^{n})^{\infty}$ by Lemma \ref{LemIdealCorrespondence}(1). By the first part, $\Ical$ is generated by finitely many elements $a_{1}\kpkt a_{r}$. Now the surjectivity of $\varepsilon$ implies that $\Jcal$ is generated by $\varepsilon(a_{i})$, $i\in\{1\kpkt r\}$.
\end {proof}

\begin {Example}  
By (the proof of) the preceding proposition, every ideal $\Acal\not=\zero$ in $(\N^{n})^{\infty}$ is of the form $\langle m_{1}\kpkt m_{r}\rangle$, where $\{ m_{1}\kpkt m_{r}\}=\min_{\le}\Acal$, $m_{i}\in\N^{n}$, and $\le$ denotes the product order on $\N^{n}$. In particular, the principal ideals are given by
$$\Ical_{k}\,:=\,\langle k\rangle\,=\,(\N_{\ge k_{1}}\timespkt\N_{\ge k_{n}})^{\infty}\,\subseteq\,(\N^{n})^{\infty}$$
where $k=(k_{1}\kpkt k_{n})\in\N^{n}$. If $k,l\in\N^{n}$, then $\Ical_{k}\cap \Ical_{l}=\Ical_{(\max(k_{i},l_{i}))_{i\in I}}$. Similarly, every ideal in $(\N^{\infty})^{n}$ is of the form $\langle m_{1}\kpkt m_{r}\rangle$, where $\{ m_{1}\kpkt m_{r}\}=\min_{\le}\Acal$, $m_{i}\in(\N^{\infty})^{n}$, and $\le$ is the product order on $\N^{n}$ extended to $(\N^{\infty})^{n}$ by $m\le\infty$ for all $m\in\N$. In particular, the principal ideals of $(\N^{\infty})^{n}$ are given by
$$\tilde{\Ical}_{k}\,:=\,\langle k\rangle\,=\,\N^{\infty}_{\ge k_{1}}\timespkt\N^{\infty}_{\ge k_{n}}\,\subseteq\,(\N^{\infty})^{n}$$
for some $k=(k_{1}\kpkt k_{n})\in(\N^{\infty})^{n}$, where  $\N^{\infty}_{\ge\infty}:=\zero$. Similarly, one has $\tilde{\Ical}_{k}\cap\tilde{\Ical}_{l}=\tilde{\Ical}_{(\max(k_{i},l_{i}))_{i\in I}}$ for $k,l\in(\N^{\infty})^{n}$. Note that $\varphi:(\N^{\infty})^{n}\rto(\N^{n})^{\infty}$ defined by $e_{i}\mto e_{i}$ and $e_{i,\infty}\mto\infty$, $i\in\{1\kpkt n\}$, is a binoid epimorphism with $\varphi(\tilde{\Ical}_{k})+(\N^{n})^{\infty}=\Ical_{k}$ if $k\in\N^{n}$ and $\{\infty\}$ otherwise.
\end {Example}

\begin {Definition} \label{DefReesCong}
Let $\Ical$ be an ideal in $M$. The congruence $\sim_{\Ical}$ on $M$ defined by 
$$a\sim_{\Ical} b\quad:\Leftrightarrow\quad a=b\quad\text{or}\quad a,b\in\Ical$$
is called the \gesperrt{Rees congruence} \index{congruence!Rees --}of $\Ical$ and $M/\!\sim_{\Ical}\,\;=:M/\Ical$\nomenclature[MACongruenceRees]{$M/\Ical$}{Rees quotient of $M$ with respect to $\Ical$} the \gesperrt{(Rees) quotient} of $M$ by \index{ideal!Rees quotient by an --}$\Ical$. The extended ideal of an ideal $\Jcal\subseteq M$ by the canonical projection $M\rto M/\Ical$ will be denoted by $\Jcal/\Ical$.\nomenclature[J]{$\Jcal/\Ical$}{extended ideal of $\Jcal$ with respect to $M\rto M/\Ical$}
\end {Definition}

\begin {Remark} \label{RemIndIsom}
The quotient $M/\Ical$ may be described as the result of collapsing $\Ical$ into a single element, namely $\infty$, while the elements outside $\Ical$ remain unchanged. In particular, $\sim_{\Ical}$ is an ideal congruence and every ideal congruence is of this form since the kernel of a binoid homomorphism is an ideal, cf.\ Example \ref{ExSpIdeals}(3). In this spirit, $M/\Ical$ can be identified with the binoid $(M\setminus\Ical)\cup\zero$ with addition given by
$$a+b\,\,=\,\,\begin{cases}
a+b&\text{, if }a+b\not\in\Ical\komma\\
\infty&\text{, otherwise,}
\end{cases}$$
for $a,b\in (M\setminus\Ical)\cup\zero$. Similarly, the extended ideal of $\Jcal$ under $M\rto M/\Ical$ can be identified with the subset $(\Jcal\setminus(\Jcal\cap\Ical))\cup\zero$.

Note that the congruence $\sim_{\pi_{\Ical}}$ induced by $\pi_{\Ical}:M\rto M/\Ical$ coincides with the Rees congruence of $\Ical$; that is, $M/\sim_{\pi_{\Ical}}=M/\Ical$ and $\Ical=\ker\pi_{\Ical}$. In general, if $\varphi:M\rto N$ is a binoid epimorphism, the induced homomorphism $M/\ker\varphi\rto N$ fails to be an isomorphism when $\sim_{\pi_{\ker\varphi}}\,\not=\,\sim_{\varphi}$. For instance, the kernel of the binoid homomorphism $\varphi:\N^{\infty}\!\times\!\N^{\infty}\rto\N^{\infty}$, $(a,b)\mto b$, is given by the ideal $\{(a,\infty)\mid a\in \N^{\infty}\}$ and the quotient by
$$(\N^{\infty}\times\N^{\infty})/\ker\varphi\,\,\cong\,\,(\N^{\infty}\times\N)\cup\{(\infty,\infty)\}\pkt$$
On the other hand, we have $(\N^{\infty}\!\times\!\N^{\infty})/\!\sim_{\varphi}\,\cong\N^{\infty}$, cf.\ Remark \ref{RemKer}. See also Example \ref{ExpKernKongruenz}.
\end {Remark}

\begin {Example}  \label{ExCompositionsIdeal}
\begin {ListeTheorem}
\item []
\item According to Remark \ref{RemIndIsom}, we obtain from the observations we made in Example \ref{ExSpIdeals}(3):
$$M_{\opint}=M/\nonint(M)\quad\text{and}\quad M_{\opred}=M/\nil(M)\pkt$$
\item Similar to the theory of semilocal rings, we have a description of the Hilbert-Samuel function
$\opH(-,M):\N\rto\N$ \index{Hilbert-Samuel function}in terms of the quotient by a multiple of the maximal ideal
$$\opH(n,M)=\#(M/nM\Uplus)-1$$
for a positive finitely generated binoid $M$.
\item Let $(M_{i})_{i\in I}$ be a finite family of binoids. Then  
$$\bigwedge_{i\in I}M_{i}\,\,\cong\,\,\Big(\prod_{i\in I}M_{i}\Big)\Big/\Ical\quad\text{ and }\quad\bigcupbidot_{i\in I}M_{i}\,\,\cong\,\,\Big(\bigwedge_{i\in I}M_{i}\Big)\Big/\Jcal\,\,\cong\,\,\Big(\prod_{i\in I}M_{i}\Big)\Big/\Ical\cup\Jcal\komma$$
where 
$$\Ical=\{(a_{i})_{i\in I}\mid a_{i}=\infty_{i}\text{ for some }i\in I\}$$
and 
$$\Jcal=\{\wedge_{i\in I}a_{i}\mid a_{i}\not=0_{i}\text{ for at least two different }i\in I\}\pkt$$
Note that $\Ical=\ker\pi_{\wedge}$ and $\Jcal=\ker\pi_{\cupbidot}\komma$ where $\pi_{\wedge}:\prod_{i\in I}M_{i}\rto\bigwedge_{i\in I}M_{i}$ and $\pi_{\cupbidot}:\bigwedge_{i\in I}M_{i}\rto\bigcupbidot_{i\in I}M_{i}$ are the canonical binoid epimorphisms, cf.\ Remark \ref{RemProdEpisPos}.
\end {ListeTheorem}
\end {Example}

\begin {Lemma} \label {LemQuotientCancPosRepF}
Let $\Ical\not=M$ be an ideal in $M$. If $M$ is finitely generated, positive, semifree, or cancellative, then so is $M/\Ical$.
\end {Lemma}
\begin{proof}
Only for the latter property there is something to show. So let $M$ be cancellative. If $[a+b]=[a+c]\not=[\infty]$ in $M/\Ical$ for some $a,b,c\in M$, then $a+b=a+c$ in $M$, which implies $b=c$ since $M$ is cancellative. Hence, $[b]=[c]$.
\end{proof}

If $M$ is reduced and $\Ical\subseteq M$ an ideal, then $M/\Ical$ need not be reduced anymore. A sufficient condition on $\Ical$ such that $M/\Ical$ is torsion-free if $M$ is so, is given in Lemma \ref{LemModuloRadical}.

\begin {Lemma}  \label {LemQuotientTorF}
If $M$ is torsion-free, then $M/\Ical$ is torsion-free up to nilpotence for every $\Ical$ in $M$.
\end {Lemma}
\begin{proof}
Any equality $[na]=[nb]\not=[\infty]$ in $M/\Ical$ means $na=nb$ in $M$. Hence, $a=b$ in $M$ by assumption, in particular, $[a]=[b]$ in $M/\Ical$.
\end{proof}

\begin {Corollary} \label{CorQuotientRepFree}
Let $V$ be a (finite) set. If $\Ical\subseteq\free(V)$ is an ideal, then $\free(V)/\Ical$ is (finitely) generated by $\{v\in V\mid v\not\in\Ical\}$ and semifree. In particular, $\free(V)/\Ical$ is positive, cancellative, and torsion-free up to nilpotence. Conversely, every (finitely generated) semifree binoid is isomorphic to $\free(V)/\Ical$ for some (finite) set $V$ and an ideal $\Ical\subseteq\free(V)$.
\end {Corollary}
\begin {proof}
The first statement follows from Corollary \ref{CorFree} with Lemma \ref{LemQuotientCancPosRepF} and Lemma \ref{LemQuotientTorF}, and its supplement from Lemma \ref{LemPropSemifree}. For the converse take a minimal generating set of the semifree binoid $M$, say $V\subseteq M$. Then $M\cong\free(V)/(\Rcal_{j})_{j\in J}$, where $(\Rcal_{j})_{j\in J}$ is a family of relations of the form $\Rcal:f=g$, $f,g\in\free(V)$, cf.\ Example \ref{ExGen}. Since $M$ is semifree, all relations $\Rcal_{j}$ are monomial relations (i.e.\ $\Rcal_{j}:f_{j}=\infty$ for all $j\in J$). This shows that the congruence defined by $(\Rcal_{j})_{j\in J}$ is the Rees congruence of the ideal $\Ical=\langle f_{j}\mid j\in J\rangle$. Hence, $M\cong\free(V)/\Ical$
by Proposition \ref{PropHomCong}.
\end {proof}

\begin {Remark}\label{RemMonSemifree}
By Corollary \ref{CorQuotientRepFree}, semifree binoids are precisely those binoids that arise from free commutative binoids modulo monomial relations. Thus, semifree binoids could also be called \emph{monomial binoids} \index{binoid!monomial --}since their binoid algebras are monomial algebras and every (commutative) monomial algebra can be realized as the binoid algebra of such a binoid. A special class of monomial algebras are Stanley-Reisner algebras (or face rings), which we will encounter in Section \ref{SecSimplBinos}.
\end {Remark}

\begin {Proposition} \label{PropNspecModIdeal}
Let $M$ and $N$ be binoids and $\Ical$ an ideal in $M$. Every binoid homomorphism $\varphi:M\rto N$ with $\Ical\subseteq\ker\varphi$ factors uniquely through $M/\Ical$. In particular, we have an isomorphism of semigroups,
$$N\minspec M/\Ical\,\,\cong\,\,\{\varphi\in N\minspec M\mid \Ical\subseteq\ker\varphi\}\pkt$$
\end {Proposition}
\begin {proof}
This is a special case of Proposition \ref {PropNSpecIdealCong}.
\end {proof}

\begin {Proposition}
Let $\Ical$ and $\Jcal$ be ideals of $M$. Then
$$(M/\Ical)\wedge_{M}(M/\Jcal)\,\,\cong\,\, M/(\Ical\cup\Jcal)\komma$$
where $M/\Acal$ is an $M\mina$binoid via the canonical projection $\pi_{\Acal}:M\rto M/\Acal$, $\Acal\in\{\Ical,\Jcal\}$.
\end {Proposition}
\begin {proof}
The canonical $M\mina$binoid epimorphisms $\psi_{\Acal}:M/\Acal\rto M/(\Ical\cup\Jcal)$ with $\bar{a}\mto[a]$, $\Acal\in\{\Ical,\Jcal\}$, induce by Corollary \ref{CorCoProdSmashN} an $M\mina$binoid homomorphism
$$\psi:(M/\Ical)\wedge_{M}(M/\Jcal)\Rto M/(\Ical\cup\Jcal)\quad\text{with}\quad \bar{a}\wedge_{M}\bar{b}\lto[\bar{a}+\bar{b}]\komma$$
which is surjective since $\psi_{\Ical}$ and $\psi_{\Jcal}$ are so. For the injectivity note that $\bar{a}\wedge_{M}\bar{b}=\infty$ if and only if $a+b\in\Ical\cup\Jcal$, and then $[\bar{a}+\bar{b}]=\infty$. If $a\in\Ical$ or $b\in\Jcal$ this is clear. Otherwise, $a\not\in\Ical$ and $b\not\in\Jcal$ is equivalent to the fact that $\bar{a}=\{a\}$ and $\bar{b}=\{b\}$ are singletons, so we have
$$\bar{a}\wedge_{M}\bar{b}\,=\, a\wedge_{M}b\,=\, 0\wedge_{M}(a+b)\,=\,(a+b)\wedge_{M}0$$
and this is $\infty$ if and only if $a+b\in\Ical\cup\Jcal$. In particular, $[\bar{a}+\bar{b}]=[a+b]=\infty$. Hence, $\infty\not=\bar{a}\wedge_{M}\bar{b}=a\wedge_{M}b\mto [a+b]=a+b\not=\infty$ if $a+b\not\in\Ical\cup\Jcal$, which implies the injectivity.
\end {proof}

\begin {Lemma}
Let $\varphi:M\rto N$ be a binoid homomorphism and $\Ical\subseteq M$ an ideal. Then
$$(M/\Ical)\wedge_{M}N\,\,\,\cong\,\,\, N/(\varphi(\Ical)+N)$$
as binoids.
\end {Lemma}
\begin {proof}
Since $\ker\pi_{\Ical}\subseteq\ker(\pi\varphi)$, and hence $\sim_{\pi_{\Ical}}\,\le\,\sim_{\pi\varphi}$, we have by Lemma \ref{LemIndCong} a commutative diagram of binoid homomorphisms
$$\xymatrix{
M\ar[d]_{\pi_{\Ical}}\ar[r]^{\varphi}&N\ar[r]^{\!\!\!\!\!\!\!\!\!\!\!\!\!\!\!\!\!\!\!\!\!\!\pi}&N/(\varphi(\Ical)+N)\\
M/\Ical\ar[urr]_{\tilde{\varphi}}&&}$$
with $\tilde{\varphi}([a])=[\varphi(a)]$. In fact, the induced homomorphism $\tilde{\varphi}$ is an $M\mina$binoid homomorphism because all others are so. Hence, $\tilde{\varphi}$ and $\pi\varphi$ induce by Corollary \ref{CorCoProdSmashN} a unique $M\mina$binoid homomorphism 
$$\psi:(M/\Ical)\wedge_{M}N\Rto N/(\varphi(\Ical)+N)\quad\text{with}\quad[a]\wedge_{M}x\lto[\varphi(a)+x]\pkt$$
On the other hand, the canonical $M\mina$binoid homomorphism $\iota:N\rto M\wedge_{M}N\rto (M/\Ical)\wedge_{M}N$ with $\iota(x)=[0]\wedge_{M}x$ factors through $N/(\varphi(\Ical)+N)$ because $\varphi(\Ical)+N\subseteq\ker\iota$. In other words, there is an $M\mina$binoid homomorphism
$$\Phi:N/(\varphi(\Ical)+N)\Rto (M/\Ical)\wedge_{M}N\quad\text{with}\quad[x]\mto[0]\wedge_{M}x\pkt$$
An easy computation shows that $\psi$ and $\Phi$ are mutually inverse.
\end {proof}

\begin{Remark} \label {RemNakayamaACC}
In \cite{AndersonIT}, Anderson and Johnson raised the question of which binoids admit an ideal theory most similar to that of arbitrary or noetherian rings. For the latter, they considered binoids satisfying the \gesperrt{ascending chain condition} \index{ideal!a.c.c.\ on --s}\index{ascending chain condition (a.c.c.)!-- on ideals}(a.c.c.) on ideals, which means that every chain of ideals 
$\Ical_{1}\subseteq\Ical_{2}\subseteq\Ical_{3}\cdots$ in such a binoid $M$ becomes stationary; that is, there is a $k\in\N$ such that $\Ical_{n}=\Ical_{k}$ for all $n\ge k$. Equivalently, every ideal in $M$ is finitely generated, which is a weaker property than being noetherian as defined in this thesis, cf.\ Remark \ref{RemNoetherian=FG} (Anderson and Johnson call binoids that satisfy the a.c.c.\ on ideals noetherian).\index{binoid!noetherian -- (ideals)}

They could prove that Krull's intersection theorem \index{Krull's intersection theorem}(\cite[Folgerung 5.6]{BruskeIschebeckVogel}) and the Artin-Rees lemma \index{Artin-Rees lemma}(\cite[Satz 5.5]{BruskeIschebeckVogel}) hold true for binoids that satisfy the a.c.c.\ on ideals, namley, for any two ideals $\Ical$ and $\Jcal$ of such a binoid, one has
$$\Ical+N\,=\,N\quad\text{with}\quad N:=\bigcap_{n\ge1}(n\Ical+\Jcal)\komma$$
cf.\ \cite[Theorem 3.10]{AndersonIT}, and there is an $m\ge0$ such that 
$$\Ical\cap k\Jcal\,=\,(\Ical\cap m\Jcal)+(k-m)\Jcal\komma\quad\forall k\ge m\komma$$
cf.\ \cite[Theorem 3.11]{AndersonIT}. An application of the Artin-Rees lemma in combination with Nakayama's lemma is Krull's theorem, \index{Krull's theorem}which says that in a noetherian ring $R$ one has $\bigcap_{n\ge1}\aideal^{n}=\{0\}$ for every ideal $\aideal$ contained in the Jacobson ideal of $R$ (\cite[Folgerung 5.8]{BruskeIschebeckVogel}).\footnote{\, Often, this is called Krull's intersection theorem.} This need not be true for binoids. For instance, the maximal ideal $M\Uplus$ of the finitely generated binoid
$$M=\free(x)/(2x=x)$$
is $\langle x\rangle$, and obviously $M\Uplus=2M\Uplus=3M\Uplus=\cdots$ so that $\bigcap_{n\ge0}nM\Uplus=M\Uplus\not=\zero$. Binoids that satisfy $\bigcap_{n\ge0}nM\Uplus=\zero$ are called separated and are the main subject of Chapter \ref{ChapSepGradings}. 

Nakayama's lemma (\cite[Satz 3.19]{BruskeIschebeckVogel}) \index{Nakayama's lemma}does not translate directly to binoids either. The canonical projection $\pi:\N^{\infty}\rto(\Z/n\Z)^{\infty}$ for $n\ge2$ turns $((\Z/n\Z)^{\infty},\infty)$ into a finitely generated $\N^{\infty}\mina$set, but for every principal ideal $\Ical_{k}:=\langle k\rangle=\N_{\ge k}^{\infty}$, $k\ge1$, in $\N^{\infty}$, one has $$\pi(\Ical_{k})+(\Z/n\Z)^{\infty}\,=\,(\Z/n\Z)^{\infty}\,\not=\,\trivial\pkt$$

For monoids that satisfy the a.c.c.\ on ideals, The Artin-Rees lemma \cite[Lemma 3.3.8]{Kobsa}, Krull's  theorem \cite[Satz 3.3.9]{Kobsa}, and some version of Nakayama's lemma \cite[Lemma 3.3.1]{Kobsa} have been proved by Kobsa, where the latter two results require in addition that the monoid $M$ is cancellative in the sense of monoid theory (i.e.\ $M^{\infty}$ has to be cancellative in our terms). Kobsa also studied monoids that satisfy the a.c.c.\ on submonoids, cf.\ \cite[Chapter 7.4]{Kobsa}. For this see also \cite[Chapter VI.7]{GrilletCS}.
\end{Remark}

\begin {Lemma}
Let $\Ical\subsetneq N$ be an ideal. If $\varphi:N\rto M$ is a local binoid homomorphism, then $\varphi(\Ical)+M\not= M$.
\end {Lemma}
\begin {proof}
If $\varphi(\Ical)+M=M$, then $\varphi(a)+b=0$ for some $a\in\Ical$ and $b\in M$, and so $\varphi(a)\in M\okreuz$, which contradicts $\varphi$ being local since $a\in\Ical\subseteq N\Uplus$.
\end {proof}

In \cite{Ahmadi}, certain versions of Nakayama's lemma satisfied by $S\mina$acts are studied.

\bigskip

\section {Prime ideals} \label{SecPrime}
\markright {\ref{SecPrime} Prime ideals}

\begin {Definition}
An ideal $\Pcal\subsetneq M$ is called \gesperrt{prime} \index{ideal!prime --}\index{prime ideal}if $M\setminus\Pcal$ is a submonoid of $M$. The \gesperrt{spectrum} \index{spectrum}of $M$, denoted by $\spec M$\nomenclature[Spec1]{$\spec M$}{set of all prime ideals in the binoid $M$ (spectrum of $M$)}, is the set of all prime ideals of $M$.
\end {Definition}

\begin {Example}\label {ExMaxIdeal}
Let $M\not=\zero$.
\begin {ListeTheorem}
\item The maximal ideal $M\Uplus$ is prime since $M\setminus M\Uplus=M\okreuz$ is a subgroup of $M$ (i.e.\ a submonoid that is also a group). Therefore, $\spec M\not=\emptyset$ if and only if $M\not=\zero$.
\item The ideal $\nonint(M)$ of all non-integral elements is prime since $M\setminus\nonint(M)=\opint(M)$ is a submonoid of $M$, cf. Lemma \ref{LemINT}.
\end {ListeTheorem}
\end {Example}

The set of all ideals is a complete lattice with respect to $\cap$ and $\cup$. This translates not to the full extent to the subset $\spec M$.

\begin {Proposition} \label {PropPrimeUnion}
The union of a set of prime ideals is again a prime ideal. In particular, the spectrum of a nonzero binoid $M$ is a (join-) semilattice with largest element $M\Uplus$.
\end {Proposition}
\begin {proof}
By Lemma \ref {LemUnionIdeals}, the union $\bigcup_{i\in I}\Pcal_{i}$ of a family of prime ideals $(\Pcal_{i})_{i\in I}$ is an ideal and since the intersection of submonoids of $M$ is again a submonoid, we obtain $M\setminus\bigcup_{i\in I}\Pcal_{i}=\bigcap_{i\in I}(M\setminus\Pcal_{i})$.
\end {proof}

The preceding result is the most significant difference between the ideal theories of binoids and rings because it contradicts the prime avoidance lemma (which states that in a ring an ideal contained in a union of $n$ ideals with at least $n-2$ prime lies in one of the $n$ ideals, cf.\ \cite[II \S1.2, Proposition 2]{BourbakiCA} or \cite[Lemma 3.3]{Eisenbud}). As this lemma has many consequences in ring theory, Proposition \ref{PropPrimeUnion} accounts for many of the differences between the theories of binoids and rings. See for instance Proposition \ref{PropExistencePrimes} and the subsequent remark.

\begin {Proposition}
The intersection of a finite family of different prime ideals is prime if and only if the family admits a unique minimal element with respect to set inclusion (i.e.\ a prime ideal which is contained in any other prime ideal of the family). 
\end {Proposition}
\begin {proof}
The if part is trivial. For the converse let $M$ be a binoid and $(\Pcal_{i})_{i\in I}$ a finite family of different prime ideals in $M$ that admits strictly more than one minimal element with respect to set inclusion. Say $\Pcal_{1}\kpkt\Pcal_{n}$, $n\ge2$, are those minimal elements of the family. By the minimality, there are elements $x_{ij}\in\Pcal_{i}\setminus(\Pcal_{i}\cap\Pcal_{j})$ for all $i,j\in\{1\kpkt n\}$ with $i\not=j$. Thus, $\sum_{i\not=j}x_{ij}\in\bigcap_{i=1}^{n}\Pcal_{i}=\bigcap_{i\in I}\Pcal_{i}$, but by construction none of the $x_{ij}$ is contained in the intersection. Hence, the ideal $\bigcap_{i\in I}\Pcal_{i}$ is not prime.
\end {proof}

Here is an example showing that the only if part of the preceding result may fail for an infinite familiy of prime ideals. Let $I$ be infinite and let $\Pcal_{i}$ denote the ideal $\langle e_{i}, e_{i+1}, e_{i+2}\ldots\rangle$ in $(\N^{I})^{\infty}$, where the $i$th component of $e_{i}$ is $1$ and all others are $0$, $i\in I$. It is clear that all $\Pcal_{i}$ are prime, and since $\Pcal_{i}\supset\Pcal_{i+1}\supset\Pcal_{i+2}\supset\cdots$ the family $(\Pcal_{i})_{i\in I}$ admits no minimal element. However, $\bigcap_{n\ge 1}\Pcal_{n}=\zero$ is a prime ideal in $(\N^{I})^{\infty}$. There are other very useful characterizations of prime ideals which will be used in the following without further reference.

\begin {Lemma} \label{LemCharacterizationPrime}
Let $\Pcal\not=M$ be an ideal in $M$. The following statements are equivalent:
\begin {ListeTheorem}
\item $\Pcal$ is prime.
\item $M/\Pcal$ is integral.
\item $M\setminus\Pcal\in\Fcal(M)$.
\item $\Pcal$ is the kernel of a binoid homomorphism $M\rto\trivial$.
\item $a+b\in\Pcal$ for some $a,b\in M$ implies that $a\in\Pcal$ or $b\in\Pcal$.
\end {ListeTheorem}
\end {Lemma}
\begin {proof}
$(1)\Rarrow(2)$ follows from the definition. $(2)\Rarrow(3)$ If $M/\Pcal$ is integral, then $M\setminus\Pcal=:F$ is a submonoid of $M$, hence $0\in F$ and $f+g\in F$ if $f,g\in F$. On the other hand, if $f+g\in F$ (i.e.\ $f+g\not\in\Pcal$), then $f,g\not\in\Pcal$ because $\Pcal$ is a prime ideal. Hence, $f,g\in F$. The equivalence $(3)\eq(4)$ is just a restatement of  Proposition \ref {PropHomFilter} and the implications $(4)\Rarrow(5)\Rarrow(1)$ are obvious.
\end {proof}

The equivalence $(1)\eq(4)$ of Lemma \ref{LemCharacterizationPrime} together with Proposition \ref {PropHomFilter} shows, in particular, that $\Pcal$ is a prime ideal if and only if $\alpha_{\Pcal}=\chi_{M\setminus\Pcal}$ is a binoid homomorphism.

\begin {Corollary} \label{CorHomFiltSpec}
The semibinoids $\trivial\minspec M$, $\spec M$, and $\Fcal(M)_{\cap}\setminus\{M\}$ are isomorphic, namely
$$\begin {array} {rcccl}
\trivial\minspec M\!\!\!&\stackrel{\sim}{\longleftrightarrow}&\!\!\!\spec M\!\!\!&\stackrel{\sim}{\longleftrightarrow}&\!\!\!\Fcal(M)\setminus\{M\}\komma\\
\alpha\!\!\!&\lto&\!\!\!\ker\alpha&&\\
(\chi_{M\setminus\Pcal}=)\,\,\alpha_{_{\Pcal}}\!\!\!&\longmapsfrom&\!\!\!\Pcal\!\!\!&\lto&\!\!\! M\setminus\Pcal\\
&&M\setminus F\!\!\!&\longmapsfrom&\!\!\! F
\end {array}$$
where the latter correspondence is inclusion reversing.
\end {Corollary}
\begin {proof}
The assignments on the left-hand side are well-defined and bijective by Lemma \ref{LemCharacterizationPrime}. The homomorphism property for both directions can be shown similarly as in Corollary \ref {CorIsomFcalMinspec}. This also yields the statement for the assignments on the right-hand side.
\end {proof}

The isomorphism $\trivial\minspec\cong\spec$ is due to Schwarz \cite{Schwarz}, see also \cite[Lemma 5.54]{CliffordPreston}.

\begin {Example} \label {ExSpecPowerset}
The ideal lattices of the binoids $\Pset(V)_{\cap}$ and $\Pset(V)_{\cup}$ have been described in Example \ref {ExpPsetIdeals} for arbitrary $V$ and in Example \ref{ExpPsetIdealsGen} for $V$ finite. Now we determine the subsemilattices of prime ideals for these binoids if $V$ is finite. From the description of the filters in $\Pset(V)_{\cap}$ and $\Pset(V)_{\cup}$, cf.\ Example \ref {ExpFilterPowerset}, we obtain the following description of the prime ideals and spectra:
$$\spec\Pset(V)_{\cap}=\{\Pcal_{J,\cap}\mid\emptyset\not=J\subseteq V\}\quad\text{and}\quad\spec\Pset(V)_{\cup}=\{\Pcal_{J,\cup}\mid J\subsetneq V\}\komma$$
where $\Pcal_{J,\cap}$ is the subset-closed subset $\{A\in\Pset(V)\mid J\not\subseteq A\}$. Thus,
$$\Pcal_{J,\cap}=\bigcup_{j\in J}\Pset(V\setminus\{j\})=\langle V\setminus\{j\}\mid j\in J\rangle\quad\text{with}\quad\Pcal_{J,\cap}\cup\Pcal_{I,\cap}=\Pcal_{J\cup I,\cap}$$
and $\Pcal_{J,\cup}$ is the superset-closed subset $\{A\in\Pset(V)\mid A\not\subseteq J\}$. Thus,
$$\Pcal_{J,\cup}=\Pset(V)\setminus\Pset(J)=\langle\{j\}\mid j\not\in J\rangle\quad\text{with}\quad\Pcal_{J,\cup}\cup\Pcal_{I,\cup}=\Pcal_{J\cap I,\cup}\pkt$$
In particular, 
$$\spec\Pset(V)_{\cap}\,\,\cong\,\,(\Pset(V)\setminus\{\emptyset\},\cup,V)\quad\text{and}\quad\spec\Pset(V)_{\cup}\,\,\cong\,\,(\Pset(V)\setminus\{V\},\cap,\emptyset)$$  as semibinoids, where the order-preserving correspondences are given by $\Pcal_{J,\cap}\lrto J$ and $\Pcal_{J,\cup}\lrto J$. This also shows that $\#\spec\Pset_{n,\cap}=\#\spec\Pset_{n,\cup}=2^{n}-1$, $n\ge0$. The prime ideals of $(\N^{n})^{\infty}$ and $(\N^{\infty})^{n}$, $n\ge1$, will be described in Example \ref{ExpPrimeLatticeNinftyN}.
\end {Example}

The following corollaries are consequences of the fact that $\trivial\minspec\cong\spec$ and of results we deduced from the induced homomorphism given in Proposition \ref{PropIndHomNspec}, which reads in terms of prime ideals as follows.

\begin {Corollary} \label{CorPrimeIdealHom}
Let $M$ and $N$ be binoids. Every binoid homomorphism $\varphi:M\rto N$ induces a semigroup homomorphism $\varphi^{\ast}:\spec N\rto \spec M$ with $\Pcal\mto\varphi^{-1}(\Pcal)$ such that $\varphi^{\ast}$ is injective if $\varphi$ is surjective.
\end {Corollary}
\begin {proof}
By Proposition \ref{PropIndHomNspec} (with $L=\trivial$), there is semigroup homomorphism $\varphi^{\ast}:\spec N\rto \spec M$ with $\alpha_{\Pcal}\mto\alpha_{P}\varphi$, where $\alpha_{\Pcal}$ is the anti-indicator homomorphism of the prime ideal $\Pcal\in\spec N$ and $\ker(\alpha_{\Pcal}\varphi)=\varphi^{-1}(\Pcal)$.
\end {proof}

\begin {Example}
We continue with Example \ref{ExSpecPowerset} from above. The canonical binoid isomorphism $\varphi:\Pset(V)_{\cup}\stackrel{\sim}{\rto}\Pset(V)_{\cap}$ with $J\mto V\setminus J=J^{\opc}$
induces a semibinoid isomorphism $\varphi^{\ast}:\spec\Pset(V)_{\cap}\rto\spec\Pset(V)_{\cup}$ with $\Pcal_{J}\mto\Pcal_{J^{\opc}}$.
\end {Example}

\begin {Corollary} \label{CorExtIdealPrime}
If $\Ical\subseteq M$ is an ideal, then $\spec M/\Ical\cong\{\Pcal\in\spec M\mid\Ical\subseteq\Pcal\}$. More precisely, there is an order preserving semibinoid isomorphism given by
$$\begin {array}{rcl}
\spec (M/\Ical)&\stackrel{\sim}{\longleftrightarrow}&(\{\Pcal\in\spec M\mid\Ical\subseteq\Pcal\},\cup,M\Uplus)\komma\\
\Qcal&\longmapsto& \pi^{-1}(\Qcal)\\
\Pcal/\Ical&\longmapsfrom&\Pcal
\end {array}$$
where $\pi:M\rto M/\Ical$ denotes the canonical projection. In particular, the extended ideal of $\Pcal\in\spec M$ by $\pi$ is a prime ideal in $M/\Ical$ if and only if $\Ical\subseteq\Pcal$.
\end {Corollary}
\begin {proof}
Note that $\{\Pcal\in\spec M\mid\Ical\subseteq\Pcal\}$ is a subsemibinoid of $\spec M$. The isomorphism as semigroups is just a restatement of Proposition \ref{PropNspecModIdeal} (with $N=\trivial$) using the identification $\alpha_{\Pcal}\leftrightarrow\Pcal$ of Corollary \ref{CorHomFiltSpec}. Since $(M/\Ical)\Uplus=M\Uplus/\Ical$ this is an isomorphism of semibinoids.
\end {proof}

\begin {Corollary} \label{CorSpecMred}
$\spec M\cong \spec M_{\opred}$ as semibinoids for every binoid $M$.
\end {Corollary}
\begin {proof}
If $M=\zero$, the statement is trivial, and if $M\not=\zero$ the statement follows from Corollary \ref{CorExtIdealPrime} with $\Ical=\nil M$ and the obvious fact that $\nil M\subseteq\Pcal$ for every $\Pcal\in\spec M$. (The statement also follows from Corollary \ref{CorNspecIsoMred} with $N=\trivial$).
\end {proof}

\begin{Corollary} \label{CorSpecSmash}
Given a finite family $(M_{i})_{i\in I}$ of binoids, there is a semibinoid isomorphism
$$\spec\bigwedge_{i\in I} M_{i}\stackrel{\sim}{\longleftrightarrow}\prod_{i\in I}\spec M_{i}\quad\text{with}\quad\bigcup_{i\in I}\widehat{\Pcal}_{i}\longleftrightarrow(\Pcal_{i})_{i\in I}\komma$$
where $\widehat{\Pcal}_{i}=\wedge_{k\in I}\Acal_{k}$ with $\Acal_{k}=M_{k}$, $k\not=i$, and $\Acal_{i}=\Pcal_{i}$.
\end{Corollary}
\begin {proof}
By Proposition \ref {PropNspecSmash} (with $N=\trivial$) and using the identification $\trivial\minspec=\spec$, there is a semigroup isomorphism 
$$\psi:\spec\bigwedge_{i\in I} M_{i}\Rto\prod_{i\in I}\spec M_{i}\quad\text{with}\quad\alpha_{\Qcal}\lto(\alpha_{\Qcal}\iota_{i})_{i\in I}\komma$$
where $\Qcal\in\spec\bigwedge_{i\in I}M_{i}$ and $\iota_{k}:M_{k}\embto\bigwedge_{i\in I}M_{i}$, $k\in I$, are the canonical embeddings. We have $\ker(\alpha_{\Qcal}\iota_{i})=\iota_{i}^{-1}(\Qcal)\in\spec M_{i}$, $i\in I$. Thus, $\psi(\Qcal)=(\iota_{i}^{-1}(\Qcal))_{i\in I}$. Furthermore, if $\Pcal_{i}\in\spec M_{i}$, $i\in I$, then 
$$\alpha:\bigwedge_{i\in I}M_{i}\Rto\trivial\komma\quad\wedge_{i\in I} a_{i}\lto\sum_{i\in I}\alpha_{\Pcal_{i}}(a_{i})\komma$$
where $\alpha_{\Pcal_{i}}:M_{i}\rto\trivial$, $i\in I$, is the preimage of $(\Pcal_{i})_{i\in I}$ under $\psi$ with $\ker\alpha=\bigcup_{i\in I}\widehat{\Pcal}_{i}$ and $\widehat{\Pcal}_{i}$ as described in the statement. Hence, every prime ideal $\Qcal\in\spec\bigwedge_{i\in I} M_{i}$ is of this form for unique $\Pcal_{i}\in\spec M_{i}$, $i\in I$. In particular, $\bigcup_{i\in I}(\widehat{M}_{i})\Uplus$ is the maximal ideal of $\bigwedge_{i\in I}M_{i}$, which shows that $\psi$ is an isomorphism of semibinoids.
\end {proof}

\begin {Corollary} 
Given a finite family $(M_{i})_{i\in I}$ of nonzero binoids, there is a semibinoid isomorphism
$$\spec\prod_{i\in I}M_{i}\stackrel{\sim}{\longleftrightarrow}\biguplus_{\emptyset\not=J\subseteq I}\prod_{i\in J}\spec M_{i}\komma\quad\bigcup_{i\in J}\tilde{\Pcal}_{i}\longleftrightarrow((\Pcal_{i})_{i\in J};J)\komma$$
where $\tilde{\Pcal}_{i}=(\Acal_{k})_{k\in I}$ with $\Acal_{k}=M_{k}$, $k\not=i$, and $\Acal_{i}=\Pcal_{i}$.
\end {Corollary}
\begin {proof}
Consider the semigroup isomorphisms
$$\biguplus_{\emptyset\not=J\subseteq I}\prod_{i\in J}\spec M_{i}\Rto\biguplus_{\emptyset\not=J\subseteq I}\bigwedge_{i\in J}\spec M_{i}\Rto\spec\prod_{i\in I}M_{i}$$
given in Proposition \ref{PropNspecProd} (with $N=\trivial$). We use the identification $\Pcal\leftrightarrow\alpha_{\Pcal}$ of Corollary \ref{CorHomFiltSpec} and the description of the prime ideals of the smash product, cf.\ Corollary \ref{CorSpecSmash}, for the map between the disjoint unions on the left-hand side. From these results, we deduce that the semigroup structures of the disjoint unions (as decribed in Proposition \ref{PropNspecProd} for $\alpha_{\Pcal}$) admit absorbing elements, namely $(((M_{i})\Uplus)_{i\in I};I)$ and $(\bigcup_{i\in I}(\widehat{M}_{i})\Uplus;I)$. This shows that the first map is a semibinoid isomorphism with $((\Pcal_{i})_{i\in J};J)\mto(\bigcup_{i\in J}\widehat{\Pcal}_{i};J)$, where $\widehat{\Pcal}_{i}=\wedge_{k\in I}\Acal_{k}$ with $\Acal_{k}=M_{k}$ for $k\not=i$ and $\Acal_{i}=\Pcal_{i}$. By the proof of Proposition \ref{PropNspecProd}, the semigroup isomorphism on the right-hand side is given by $(\alpha_{\Pcal};J)\mto\alpha_{\Pcal}\pi\pi_{J}$, where
$$\prod_{i \in I} M_{i} \stackrel{\pi_J}{\Rto}\prod_{i \in J} M_{i} \stackrel{\pi}{\Rto}\bigwedge_{i\in J} M_{i} \stackrel{\alpha_{\Pcal}}{\Rto}\trivial\pkt$$
Hence, if $\Pcal=(\bigcup_{i\in J}\widehat{\Pcal}_{i};J)$ as above, then $\ker\alpha_{\Pcal}\pi\pi_{J}=\bigcup_{i\in J}\tilde{\Pcal}_{i}$, where the $\tilde{\Pcal}_{i}$ are as in the statement. Since $(\bigcup_{i\in I}(\widehat{M}_{i})\Uplus;I)\mto((M_{i})\Uplus)_{i\in I}$ this is also a semibinoid isomorphism
\end {proof}

\begin {Corollary}
Let $(M_{i})_{i\in I}$ be a finite family of nonzero $N\mina$binoids with structure homomorphisms $\varphi_{i}:N\rto M_{i}$, $i\in I$. Then 
$$\spec\bigwedge_{i\in I} \!\!{}_{_{N}}M_{i}\,\,\cong\,\,\Big\{(\Pcal_{i})_{i\in I}\in\prod_{i\in I}\spec M_{i}\,\,\Big|\,\,\varphi_{i}^{-1}(\Pcal_{i})=\varphi_{j}^{-1}(\Pcal_{j}), i,j\in I\Big\}$$
as semibinoids. More precisely, every prime ideal in $\bigwedge_{N}M_{i}$ is of the form $\bigcup_{i\in I}\tilde{\Pcal_{i}}$, where $\tilde{\Pcal}_{k}=\wedge_{N}\Acal_{k}$ with $\Acal_{k}=M_{k}$ if $k\not=i$ and $\Acal_{i}=\Pcal_{i}$ otherwise, and $(\Pcal_{i})_{i\in I}$ is an element of the semibinoid on the right-hand side.
\end {Corollary}
\begin {proof}
The isomorphism as semigroups is just a restatement of Proposition \ref{PropLSpecNSmash} (with $L=\trivial$) using the identification $\alpha_{\Pcal}\leftrightarrow\Pcal$ of Corollary \ref{CorHomFiltSpec}. The supplement follows from Corollary \ref{CorSpecSmash} since $\bigwedge_{N}M_{i}$ is a quotient of $\bigwedge_{i\in I}M_{i}$. This shows that this is a semibinoid isomophism because $\bigcup_{i\in I}(\tilde{M}_{i})\Uplus\leftrightarrow((M_{i})\Uplus)_{i\in I}$.
\end {proof}

\begin {Corollary}
Let $(M_{i})_{i\in I}$ be a finite family of positive binoids. Then
$$\spec\bigcupbidot_{i\in I}M_{i}\,\,\cong\,\,\bigcupdot_{i\in I}\spec M_{i}$$
as semigroups, where the pointed union is taken over $(\spec M_{i},(M_{i})\Uplus)_{i\in I}$. In particular, every prime ideal $\not=(\bigcupbidot_{i\in I}M_{i})\Uplus$ in $\bigcupbidot_{i\in I}M_{i}$ is of the form $\Pcal_{(j)}=\biguplus_{i\in I}\Acal_{i}$, where $\Acal_{i}=(M_{i};i)$ for $i\not=j$ and $\Acal_{j}=(\Pcal;j)$, $\Pcal\in\spec M_{j}\setminus\{(M_{j})\Uplus\}$ for some $j\in I$.
\end {Corollary}
\begin {proof}
The isomorphism and the description of the prime ideals follow from Corollary \ref{CorNSpecPUnion} (with $N=\trivial$) using the identification $\alpha_{\Pcal}\leftrightarrow\Pcal$ of Corollary \ref{CorHomFiltSpec}.
\end {proof}

\begin {Corollary} \label{CorSpecIntegralSubbinoid}
Given a subbinoid $M$ of $N$ such that $N$ is integral over $M$ with respect to the inclusion $\iota:M\embto N$, there is a semibinoid isomorphism
$$\begin {array}{rcl}
\spec N&\longleftrightarrow&\spec M\pkt\quad\quad\quad\quad\quad\quad\quad\quad\quad\\
\Qcal&\longmapsto& M\cap\Qcal\\
\{a\in N\mid ka\in\Pcal\text{ for some }k\in\N\}&\longmapsfrom&\Pcal
\end {array}$$
In particular, $\spec M\cong\spec\overline{M}^{N}$ as semibinoids for every subbinoid $M$ of $N$.
\end {Corollary}
\begin {proof}
This is just a restatement of Proposition \ref{PropLSpecIntegral} (with $L=\trivial$), using the identification $\alpha_{\Pcal}\leftrightarrow\Pcal$ of Corollary \ref{CorHomFiltSpec}.
\end {proof}

\begin {Corollary} \label{CorIndSpecLoc}
Let $M$ be a binoid and $S$ a submonoid of $M$. The canonical homomorphism $\iota_{S}:M\rto M_{S}$ induces an order preserving semigroup isomorphism given by
$$\begin {array}{rcl}
\quad\quad\quad\quad\quad\quad\quad\quad\quad\spec M_{S}&\longleftrightarrow&(\{\Pcal\in\spec M\mid\Pcal\cap S=\emptyset\},\cup)\pkt\\
\Qcal&\longmapsto& \iota_{S}^{-1}(\Qcal)\\
\iota_{S}(\Pcal)+M_{S}&\longmapsfrom&\Pcal
\end {array}$$
\end {Corollary}
\begin {proof}
This is just a restatement of Corollary \ref{CorNSpecLoc} (with $N=\trivial$), using the identification $\alpha_{\Pcal}\leftrightarrow\Pcal$ of Corollary \ref{CorHomFiltSpec}.
\end {proof}

\begin {Definition}
If $\Pcal$ is a prime ideal, we use the common notation $M_{\Pcal}$\nomenclature[M4]{$M_{\Pcal}$}{localization of $M$ at $M\setminus\Pcal$, $P\in\spec M$} for the localization of $M$ at the filter $M\setminus\Pcal$ and $\iota_{\Pcal}$ for the canonical homomorphism $M\rto M_{\Pcal}$.
\end {Definition}

With this notation Corollary \ref {CorLocal} now translates to the following result.

\begin {Corollary}
Let $S$ be a submonoid of $M$ with $\infty\not\in S$. Then $M_{S}\cong M_{\Pcal}$, where $\Pcal$ is the prime ideal $M\setminus\opFilt(S)$. Moreover, the unique maximal prime ideal of $M_{\Pcal}$ is the extended ideal $\iota_{\Pcal}(\Pcal)+M_{\Pcal}$.
\end {Corollary}
\begin {proof}
By Corollary \ref {CorLocal}, we have $M_{S}\cong M_{\opFilt(S)}=M_{\Pcal}$. Thus, we only need to verify that $(M_{\Pcal})\okreuz=M_{\Pcal}\setminus (\iota_{\Pcal}(\Pcal)+M_{\Pcal})$. For the inclusion $\supseteq$ let $a\minus s\in M_{\Pcal}$, $a\in M$, $s\not\in\Pcal$. If $a\not\in\Pcal$, then $s\minus a\in M_{\Pcal}$, which is the inverse of $s\minus a$, hence $(a\minus s)\in(M_{\Pcal})\okreuz$. Conversely, if $a\minus s$ is a unit, there are elements $b\in M$ and $t\not\in\Pcal$ such that $0=(a\minus s)+(b\minus t)=(a+b)\minus(s+t)$, which is equivalent to $a+b+c=s+t+c$ for some $c\not\in\Pcal$. Since $s+t+c\not\in\Pcal$, the element $a$ lies not in $\Pcal$. Hence, $a\minus s\not\in(\iota_{\Pcal}(\Pcal)+M_{\Pcal})$.
\end {proof}

\begin {Proposition} \label{PropFgPrimes}
If $M$ is a binoid with generating set $V\subseteq M$, then every prime ideal of $M$ is of the form $\langle A\rangle$ for some subset $A\subseteq V$. In particular, $\#\spec M\le2^{\#V}$.
\end {Proposition}
\begin {proof}
If $M=\zero$, the statement is clear, so let $M\not=\zero$. Note that every ideal of the form 
$\langle A\rangle$, $\emptyset\not=A\subsetneq V$, is prime. On the other hand, consider $\Pcal\in\spec M$. If $\Pcal=\zero$, then $M$ is integral and $\Pcal=\langle\emptyset\rangle$. So let $\Pcal\not=\zero$. Every $\infty\not=f\in\Pcal$ can be written as $f=\sum_{x\in V}n_{x}x$, where $n_{x}\not=0$ for at least one but only finitely many $x\in V$. By the prime property, we get $x\in\Pcal$ for at least one $x\in\{v\in V\mid n_{v}\not=0\}$, which proves the first statement. The supplement is clear.
\end {proof}

\begin{Example} \label {ExpPrimeLatticeNinftyN}
Let $I=\{1\kpkt n\}$, $n\ge 1$.
\begin {ListeTheorem}
\item The binoid $(\N^{n})^{\infty}$ admits a minimal generating set given by $e_{i}$, $i\in I$, cf.\ Proposition \ref {PropUniqueMinSyst}, and all ideals $\langle e_{i}\mid i\in J\rangle$, $J\subseteq I$, are prime, which implies that $\#\spec(\N^{n})^{\infty}=2^{n}$. This shows that the bound in Proposition \ref{PropFgPrimes} is sharp.
\item The binoid $(\N^{\infty})^{n}$ is generated by the $2n$ elements $e_{i}$, $e_{i,\infty}$, $i\in I$, cf.\ Example \ref {ExBinMonGenerators}(3). Thus, the prime ideals are of the form $\Pcal_{J^{\prime},J}=\langle e_{i},e_{j,\infty}\mid i\in J^{\prime},j\in J\rangle$, where $J^{\prime},J\subseteq\Pset_{n}$ with $J^{\prime}\not=\emptyset$ or $J\not=\emptyset$ because $(\N^{\infty})^{n}$ is not integral. Since $e_{i}+e_{i,\infty}=e_{i,\infty}$, we have $\Pcal_{J^{\prime},J}=\Pcal_{J^{\prime},J\cup J^{\prime}}$. Therefore, the different prime ideals of $(\N^{\infty})^{n}$ are given by 
$$\Pcal_{J^{\prime},J}=\langle e_{i},e_{j,\infty}\mid i\in J^{\prime},j\in J\rangle\komma$$
where $\emptyset\not=J\subseteq\Pcal_{n}$ and $J^{\prime}\subseteq J$. Hence, $\#\spec(\N^{\infty})^{n}=\sum_{k=1}^{n}2^{k}\binom{n}{k}$. For instance, the semilattice $\spec(\N^{\infty})^{2}$ consisting of $8$ prime ideals can be illustrated as follows
$$\xymatrix{ 
&\langle e_{1},e_{2}\rangle&\\
\langle e_{1},e_{2,\infty}\rangle\ar@{}[ur]|{\!\!\!\!\!\subsetur}&&\langle e_{2},e_{1\infty}\rangle\ar@{}[ul]|{\!\supsetul}\\
\langle e_{1}\rangle\ar@{}[u]|{\subsetu}&\langle e_{1,\infty},e_{2,\infty}\rangle\ar@{}[ur]|{\!\!\!\!\!\subsetur}\ar@{}[ul]|{\!\supsetul}&\langle e_{2}\rangle\ar@{}[u]|{\supsetu}\\
\langle e_{1,\infty}\rangle\ar@{}[u]|{\subsetu}\ar@{}[ur]|{\!\!\!\!\!\subsetur}&&\langle e_{2,\infty}\rangle\ar@{}[u]|{\supsetu}\ar@{}[ul]|{\!\supsetul}
}$$
\item The bipointed union $\bigcupbidot_{i\in I}\N^{\infty}$ is minimally generated by $(1;i)$, $i\in I$, cf.\ Proposition \ref{PropUniqueMinSyst}. If $\Pcal$ is a prime ideal with $(1;i)\not\in\Pcal$, then $(1;j)\in\Pcal$ for every $j\not=i$ because $(1;i)+(1;j)=\infty\in\Pcal$ when $j\not=i$. Hence, the prime ideals of $\bigcupbidot_{i\in I}\N^{\infty}$ are given by
$$\Pcal_{i}:=\langle(1;j)\mid j\in I, j\not=i\rangle\komma$$
$i\in I$, and $(\bigcupbidot_{i\in I}\N^{\infty})\Uplus=\langle (1;i)\mid i\in I\rangle$. In particular, $\#\spec\bigcupbidot_{i\in I}\N^{\infty}=n+1$.
\item Consider the binoid $M:=\free(x_{1}\kpkt x_{n})/(x_{i}+x_{i+1}=x_{i})_{i\in I\setminus\{n\}})$. If  $\emptyset\not=J\subseteq I$, then 
$$\Pcal_{J}:=\langle x_{i}\mid i\in J\rangle=\langle x_{s}\rangle\quad(=\langle x_{1}\kpkt x_{s}\rangle)$$
with $s:={\max}_{\le}J$. Thus, $\spec M$ is a totally ordered complete lattice given by the following sequence of principal ideals
$$\zero\subset\langle x_{1}\rangle\subset\langle x_{2}\rangle\subset\cdots\subset\langle x_{n}\rangle=M\Uplus\pkt$$
\item The prime ideal $\nonint(M)$ is generated by those generators of $M$ that are not integral.
\item The spectra of the four different types of one-generated binoids
\item []\begin{tabular}[c]{cccclllll}
&&&&$\N^{\infty}\komma$&$(\Z/n\Z)^{\infty}\komma$&$\N^{\infty}/(r=s)\komma$&and&$\quad\N^{\infty}/(n=\infty)\komma$\\
\multicolumn{5}{l}{\!\!cf.\ Corollary \ref{CorClassificationOnegenerated}, are}\\
&&&&$\{\zero,\langle 1\rangle\}\komma$&$\{\zero\}\komma$&$\{\zero,\langle 1\rangle\}\komma$&and&$\quad \{\langle 1\rangle\}\pkt$\\
\end{tabular}
\end {ListeTheorem}
\end {Example}

\begin {Example} \label{ExpProjLim}
The spectrum $\spec M=:I$ of a (finitely generated) binoid $M$ is a partially ordered (finite) set such that $$(M/\Pcal,(\varphi_{\Pcal\Qcal})_{\Qcal\subseteq\Pcal})_{\Pcal,\Qcal\in I}\komma$$
where
$$\varphi_{\Pcal\Qcal}:M/\Qcal\Rto M/\Pcal$$
is the canonical projection for $\Qcal\subseteq\Pcal$, is an inverse system of binoids. The projective limit of this system is given by 
$$\varprojlim_{\Pcal\in I} M/\Pcal\,=\,\Big\{(a_{\Pcal})_{\Pcal\in I}\in\prod_{\Pcal\in I}M/\Pcal\,\,\Big|\,\, \varphi_{\Pcal\Qcal}(a_{\Qcal})=a_{\Pcal}\text{ for all }\Qcal\subseteq \Pcal\Big\}$$
together with the canonical projections $\varphi_{\Qcal}:\varprojlim M/\Pcal\rto M/\Qcal$ with $(a_{\Pcal})_{\Pcal\in I}\lto a_{\Qcal}$, $\Qcal\in I$, cf.\ Section \ref{SecLimits} (here, the partial order $\ge$ on $I$ is given by $\Qcal\ge\Pcal$ if $\Qcal\subseteq\Pcal$). Note that if $M$ is positive, then $M/M\Uplus=\trivial$ and $\varphi_{M\Uplus\Pcal}:M/\Pcal\rto\trivial$ is the binoid homomorphism sending everything but $0$ to $\infty$ (i.e.\ $\varphi_{M\Uplus\Pcal}=\chi_{\{0\}}$).

If, for instance, $I$ contains one minimal element with respect to $\subseteq$, say $\Qcal$, then there is for every $\Pcal\in I$ a binoid homomorphism $\varphi_{\Pcal\Qcal}:M/\Qcal\rto M/\Pcal$, which is the identity on $\{a\in M/\Qcal\mid a\not\in\Pcal\}=(M/\Qcal)\setminus\ker\varphi_{\Pcal\Qcal}$. Hence,
\begin {align*}
\varprojlim_{\Pcal\in I} M/\Pcal&\,=\,\Big\{(\varphi_{\Pcal\Qcal}(a))_{\Pcal\in I}\in\prod_{\Pcal\in I}M/\Pcal\,\,\Big|\,\, a\in M\Big\}\\
&\,\cong\,\Big\{(a,\varphi_{\Pcal\Qcal}(a))_{\Pcal\not=\Qcal}\,\,\Big|\,\, a\in M/\Qcal, a\not=0\Big\}\cup\{(0\kpkt 0)\}\,\,\cong\,\,M/\Qcal\pkt
\end {align*}
As another example consider the binoid $M\,=\,\bigcupbidot_{i=1}^{n}\N^{\infty}$ from \ref{ExpPrimeLatticeNinftyN}(3) (with the notation given there). The inverse system of $M$ is given by the $n+1$ binoids $M/M\Uplus=\trivial$, $M/\Pcal_{i}=(\N^{\infty};i)\cong\N^{\infty}$, $i\in\{1\kpkt n\}$, together with the binoid homomorphisms $\varphi_{M\Uplus\Pcal_{i}}:(\N^{\infty};i)\rto\trivial$, $\varphi_{\Pcal_{i}\Pcal_{i}}=\id_{(\N^{\infty};i)}$ for $i\in\{1\kpkt n\}$, and $\varphi_{M\Uplus M\Uplus}=\id_{\trivial}$. Hence,
$$
\varprojlim_{\Pcal\in I} M/\Pcal\,\,=\,\,\{(\infty,a)\mid a\in(\N_{\ge1}^{\infty})^{n}\}\cup\{(0\kpkt 0\}
\,\,\cong\,\,(\N_{\ge1}^{\infty})^{n}\cup\{(0\kpkt 0\}\komma$$
which is not a finitely generated binoid if $n\ge2$, though $M$ is so. We will come back to these examples later in Example \ref{ExpProjLimitAlgebras}.
\end {Example}

\begin {Remark}\label{RemMonoidDecomposition}
Let $N$ and $M$ be binoids. Every $N\mina$point $\varphi\in N\minspec M$ factors uniquely through $\ker\varphi$, cf.\ Lemma \ref{LemFactoriazationKer}, which is a prime ideal. Indeed, we can write 
$$N\minspec M\,=\,\biguplus_{\Pcal\in\spec M}S_{\Pcal}\komma$$
where $S_{\Pcal}:=\{\varphi\in N\minspec M\mid\ker\varphi=\Pcal\}\cong N\minspec(M/\Pcal)$ is a subsemigroup of $N\minspec M$ in which the characteristic point $\chi_{M\setminus\Pcal}$ serves as an identity element. Moreover, $S_{\Pcal}+S_{\Qcal}\subseteq S_{\Pcal\cup\Qcal}$. See also \cite[Chapter 5]{CliffordPreston}. If $N$ is a binoid group, then $S_{\Pcal}\cong N\minspec\diff((M/\Pcal)_{\opcan})$ by Corollary \ref{CorUnivPropDiff}. Therefore,  
$$\#N\minspec M\,=\,\sum_{\Pcal\in\spec M}\#N\minspec\diff((M/\Pcal)_{\opcan})$$
when $\spec M$ and $N\minspec\diff((M/\Pcal)_{\opcan})$, $\Pcal\in\spec M$, are finite. Take for instance a finitely generated torsion-free binoid $M$. By Lemma \ref{LemDGroupTorsionfree}, $\diff((M/\Pcal)_{\opcan})$ is a finitely generated torsion-free group, which implies that $\diff((M/\Pcal)_{\opcan})=\Z^{l_{\Pcal}}$ for some $l_{\Pcal}\in\N$ by the structure theorem of finitely generated commutative groups. Hence, $\#\F_{q}^{\infty}\minspec M\,=\,\sum_{\Pcal\in\spec M}(q-1)^{l_{\Pcal}}$. See also Proposition \ref{PropUnionCanComp}.
\end {Remark}

Finally, we state a result which we will need later.

\begin {Proposition} \label {PropExistencePrimes} 
Let $M$ be a binoid and $\Ical\subseteq M$ an ideal. If $N$ is a submonoid of $M$ with $N\cap\Ical=\emptyset$, then $\Ical$ is contained in a unique ideal $\Pcal\subseteq M$ which is maximal with respect to $N\cap\Pcal=\emptyset$. Moreover, $\Pcal$ is prime.
\end {Proposition}
\begin {proof}
Let $\Pcal$ be the union of all ideals $\Jcal$ with $\Ical\subseteq\Jcal\subseteq M\setminus N$. Then $\Pcal$ is an ideal contained in $M\setminus N$ by Proposition \ref{PropPrimeUnion}. We have to show that $\Pcal$ is prime. For this suppose that $a+b\in\Pcal$ with $a,b\in M\setminus\Pcal$. By the maximality of $\Pcal$ among all ideals with $N\cap\Pcal=\emptyset$, there are elements $x\in(\Pcal\cup\langle a\rangle)\cap N$ and $y\in(\Pcal\cup\langle b\rangle)\cap N$. Hence, $x=a+m$ and $y=b+m^{\prime}$ for some $m,m^{\prime}\in N$ because $N\cap\Pcal=\emptyset$. Then $x+y=a+b+m+m^{\prime}\in N\cap\Pcal$, which is a contradiction to $N\cap\Pcal=\emptyset$.
\end {proof}

\begin {Remark} \label {RemPrimeUnionConsequence2}
As the analogous statement for rings, Proposition \ref {PropExistencePrimes} implies (take $N=\{0\}$) that every proper ideal $\Ical$ in a binoid $M$ is contained in a maximal prime ideal. This result is trivial since $M\Uplus$ is prime and the upper bound of the lattice of ideals in $M$, cf.\ Lemma \ref {LemUnionIdeals}. Note that Proposition \ref{PropExistencePrimes} is stronger than the analogous statement for rings since for rings the ideal $\Pcal$ need not be unique. 

When dealing with primary ideals and primary decompositions in binoids, cf.\ \cite{AndersonIT} for a thorough investigation, one encounters even more consequences of the fact that the union of prime ideals is again a prime ideal. For instance, similar to a result of ring theory (\cite[Proposition 4.7]{AtiyahMacDonald}), one can show that, if an ideal $\Ical\subseteq M$ admits a primary decomposition, then the set
$$\{a\in M\mid a+b\in\Ical\text{ for some }b\in M\setminus\Ical\}\,\,=\,\,\Ical\cup\nonint(M/\Ical)$$
is the union of the unique associated primes of $\Ical$, hence $\Ical\cup\nonint(M/\Ical)$ is again a prime ideal. This shows that a prime ideal in a binoid consisting of non-integral elements (\emph{zero-divisors} in terms of ring theory) need not be contained in an associated prime.
\end {Remark}

\bigskip

\section {Minimal prime ideals} \label{SecMinPrime}
\markright {\ref{SecMinPrime} Minimal prime ideals}

\begin {Definition}
Let $M$ be a binoid and $\Ical$ an ideal in $M$. A prime ideal $\Pcal$ containing $\Ical$ is a \gesperrt{minimal prime} \index{ideal!minimal prime of an --}of $\Ical$ if there is no prime ideal $\Qcal$ with $\Ical\subseteq\Qcal\subsetneq\Pcal$. The set of all minimal prime ideals of $\Ical$ is denoted by $\min_{M}\Ical$\nomenclature[Minimal]{$\min_{M}\Ical$}{set of all minimal prime ideals of $\Ical$}. The minimal prime ideals of the zero ideal $\zero$ are called the \gesperrt{minimal prime ideals} \index{prime ideal!minimal --}\index{minimal prime ideal}of $M$ and the set of all of them will be denoted by $\min M$\nomenclature[Minimal]{$\min M$}{set of all minimal prime ideals in $M$} (instead of $\min_{M}\zero$). The filter $M\setminus\Pcal$ defined by a minimal prime ideal $\Pcal\in\min M$ is called an \gesperrt{ultrafilter}\index{ultrafilter}.
\end {Definition}

If $M$ is integral, the zero ideal is the only minimal prime ideal of $M$. There are other very useful characterizations of minimal prime ideals which will be used without further references.

\begin {Lemma} \label {CharMinPrime}
Let $M$ be a binoid, $\Ical$ an ideal in $M$, and $\Pcal\in\spec M$ containing $\Ical$. The following statements are equivalent:
\begin {ListeTheorem}
\item $\Pcal$ is a minimal prime of $\Ical$.
\item $M\backslash\Pcal$ is a submonoid of $M$ which is maximal among all submonoids $S\subseteq M$ with $S\cap\Ical=\emptyset$.
\item For each $p\in\Pcal$, there is an element $a\in M\setminus\Pcal$ such that $a+np\in\Ical$ for some $n\ge1$.
\end {ListeTheorem}
In particular, if $p$ is an element of a minimal prime ideal $\Pcal$ of $M$, then $a+np=\infty$ for some $a\in M\setminus\Pcal$ and $n\ge1$.\end {Lemma}
\begin {proof}
Set $F:=M\setminus\Pcal$. For the implication $(1)\Rarrow (2)$, only the maximality of $F$ has to be verified since $F$ is a filter with $F\cap\Ical=\emptyset$ by definition. So suppose that $S$ is another submonoid with $S\cap\Ical=\emptyset$ and $F\subsetneq S$. By Proposition \ref{PropExistencePrimes}, there is a unique prime ideal $\Qcal$ which is maximal with respect to $S\cap\Qcal=\emptyset$. In particular, $F\cap\Qcal=\emptyset$. The uniqueness of $\Qcal$ yields $\Qcal\subsetneq\Pcal$, but this contradicts the minimality of $\Pcal$. To prove $(2)\Rarrow (3)$ take an arbitrary $p\in\Pcal$ and define $N(p):=\{a+ip\mid a\in F, i\in\N\}\subseteq M$. The set $N(p)$ is a submonoid of $M$ with $F\subseteq N(p)$. By the maximality of $F$, we have $N(p)\cap\Ical\not=\emptyset$, hence $a+np\in\Ical$ for some $a\in S$. (Note that $n=0$ is needed to have a submonoid containing $F$, but $a\not\in\Ical$ for all $a\in F$). $(3)\Rarrow (1)$ Finally, suppose that $(3)$ holds for $\Pcal$ and there is a prime ideal $\Qcal$ with $\Ical\subseteq\Qcal
 \subsetneq\Pcal$. Choose $p\in\Pcal$ such that $p\not\in\Qcal$. By assumption, there is an element $a\in F$ and an $n\ge1$ such that $a+np\in I\subseteq\Qcal$. Hence, $p\in\Qcal$ or $a\in\Qcal$ because $\Qcal$ is prime, but neither is true. Thus, $\Pcal$ has to be minimal over $\Ical$. 
\end {proof}

Minimal prime ideals of an ideal always exist. To show this, we need the following lemma.

\begin {Lemma} \label {LemMaxSubmonoids}
Let $\Ical$ be an ideal in $M$. If $N$ is a submonoid of $M$ with $N\cap\Ical=\emptyset$, then $N$ is contained in a submonoid which is maximal with respect to this property.
\end {Lemma}
\begin {proof}
This is an easy consequence of Zorn's Lemma.
\end {proof}

\begin {Proposition}
The set of minimal prime ideals of an ideal $\Ical\not=M$ in a binoid is not empty. In particular, $\min M\not=\emptyset$ and every prime ideal contains a minimal prime ideal.
\end {Proposition}
\begin {proof}
Since the submonoid $M\okreuz$ does not meet $\Ical$, there exsists a submonoid $N\subseteq M$ which is maximal with respect to this property by Lemma \ref{LemMaxSubmonoids}. Obviously, the prime ideal $M\setminus N$ is minimal over $\Ical$.
\end {proof}

\begin {Definition} 
The \gesperrt{radical} \index{radical!-- of an ideal}\index{ideal!radical of an --}of an ideal $\Ical$, denoted by $\sqrt{\Ical}$\nomenclature[ARadical]{$\sqrt{\Ical}$}{radical of $\Ical$}, is the set of all $a\in M$ such that $na\in\Ical$ for some $n\in\N$. An Ideal $\Ical\subseteq M$ with $\sqrt{\Ical}=\Ical$ is called a \gesperrt{radical ideal}\index{ideal!radical --}\index{radical!-- ideal}.
\end {Definition}

\begin {Example}
Obviously, the unit ideal $M$ and all prime ideals are radical ideals. The radical of an ideal $\Ical$ is a radical ideal since $\sqrt{\sqrt{\Ical}}=\sqrt{\Ical}$. In particular, the ideal $\nil(M)$ is a radical ideal because it is the radical of the zero ideal (i.e.\ $\sqrt{\zero}=\nil(M)$). Therefore, $\nil(M)$ is also called the \gesperrt{nil radical}\index{radical!nil --}\index{nil radical}.
\end {Example}

\begin {Lemma} \label{LemRadikalReduced}
An ideal $\Ical$ of $M$ is a radical ideal if and only if $M/\Ical$ is reduced.
\end {Lemma}
\begin {proof}
Let $\Ical$ be a radical ideal. If $[\infty]=k[a]=[ka]$ for some $a\in M$ and $k\ge 1$, then $ka\in\Ical$, which implies that $a\in\Ical$, so $[a]=[\infty]$. This proves that $M/\Ical$ is reduced. Conversely, if $M/\Ical$ is reduced, it is enough to verify that $\sqrt{\Ical}\subseteq\Ical$ since the other inclusion is obvious. So let $a\in M$ with $ka\in\Ical$ for some $k\ge 1$. Then $k[a]=[ka]=[\infty]$ in $M/\Ical$, and hence $[a]=[\infty]$ because $M/\Ical$ is reduced. Thus, $a\in\Ical$. 
\end {proof}

\begin {Lemma} \label{LemModuloRadical}
If $M$ is torsion-free, then so is $M/\Ical$ for every radical ideal $\Ical$ of $M$.
\end {Lemma}
\begin{proof}
By Lemma \ref {LemQuotientTorF}, we only need to show that $M/\Ical$ is reduced, which is true by Lemma \ref{LemRadikalReduced}.
\end{proof}

\begin {Proposition} \label {PropRadical}
The radical of an ideal $\Ical$ in a binoid $M$ is the intersection of all prime ideals in $M$ containing $\Ical$, namely 
$$\sqrt{\Ical}\,\,\,=\!\!\bigcap_{\Ical\subseteq\Pcal\in\spec M}\Pcal\pkt$$
In particular, the radical of an ideal $\Ical$ is the intersection of its minimal primes.
\end {Proposition}
\begin{proof}
If $\Ical=M$, the statement follows since the empty intersection is $M$. Therefore, let $\Ical\not=M$. For the inclusion $\subseteq$ let $a\in\sqrt{\Ical}$ and $\Pcal$ be a prime ideal containing $\Ical$. Hence, there is an $n\ge1$ with $na\in\Ical\subseteq\Pcal$, but this implies $a\in\Pcal$ by the prime property. Conversely, assume that $a\not\in\sqrt{\Ical}$. The submonoid $N:=\{na\mid n\in\N\}$ of $M$ satisfies $N\cap\Ical=\emptyset$. Applying Lemma \ref {PropExistencePrimes}, we find a prime ideal $\Pcal$ containing $\Ical$ with $N\cap\Pcal=\emptyset$. In particular, $a\not\in\Pcal$, which therefore lies not in the intersection of all prime ideals containing $\Ical$.
\end{proof}

\begin {Corollary} \label {CorMinPrimeIntersection}
$\bigcap_{\Pcal\in\min(M)}\Pcal=\nil(M)$ for every nonzero binoid $M$.
\end {Corollary}
\begin {proof}
This is clear by Proposition \ref{PropRadical} since $\nil(M)=\sqrt{\zero}$ is a radical ideal which is contained in every prime ideal of $M$.
\end {proof}

\begin {Corollary}
Let $M\not=\zero$ be a reduced binoid. The following statements are equivalent:
\begin {ListeTheorem}
\item $M$ is cancellative.
\item $M/\Pcal$ is cancellative for every $\Pcal\in\min M$.
\end {ListeTheorem}
\end {Corollary}
\begin {proof}
$(1)\Rarrow(2)$ follows from Lemma \ref{LemQuotientCancPosRepF}. To show $(2)\Rarrow(1)$ suppose that
$a+b=a+c\not=\infty$ but $b\not=c$ for $a,b,c\in M$. Since $M$ is reduced, we have $\bigcap_{\Pcal\in\min(M)}\Pcal=\zero$ by Corollary \ref{CorMinPrimeIntersection}. Thus, there is a minimal prime ideal $\Pcal$ of $M$ with $a+b=a+c\not\in\Pcal$. This is equivalent to $[a]+[b]=[a]+[c]\not=[\infty]$ in $M/\Pcal$, and implies $[b]=[c]$ in $M/\Pcal$ by the cancellativity. Since $\Pcal$ is prime, the elements $b$ and $c$ are not contained in $\Pcal$ by the choice of $\Pcal$. Hence, $b=c$ in $M$.
\end {proof}

The condition of being reduced is necessary in the preceding corollary. As an example consider the binoid
$$M:=\free(x,y,z)/(2x=\infty, x+y=x+z)\komma$$
which is neither reduced nor cancellative, but $M/\Pcal$ is cancellative for all $\Pcal\in\min M=\{\langle x\rangle,\langle y\rangle,\langle z\rangle\}$.

\begin {Definition} 
A \gesperrt{subdirect product} \index{product!subdirect --}of a family $(M_{k})_{k\in I}$ of binoids is a subbinoid $M$ of their product $\prod_{k\in I}M_{k}$ such that $\pi_{k}(M)=M_{k}$ for all $k\in I$, where $\pi_{k}$ denotes the canonical projection $\prod_{i\in I}M_{i}\rto M_{k}$.
\end {Definition}

\begin {Corollary} \label {CorSubdirectProd} 
Every reduced binoid $M$ is a subdirect product of $(M/\Pcal)_{\Pcal\in\min M}$.
\end {Corollary}
\begin {proof}
Clearly, if $M$ is the zero binoid, then $M$ is a subdirect product of the empty product. So let $M\not=\zero$ and $\pi_{\Pcal}:M\rto M/\Pcal$ be the canonical projection for $\Pcal\in\min M$. For the injectivity of $$\pi=(\pi_{\Pcal})_{\Pcal\in\min M}:M\Rto\prod_{\Pcal\in\min M}M/\Pcal$$
assume that $\pi(a)=\pi(b)$ for some $a,b\in M$. This means $a\sim_{\pi}\!b$ which is equivalent to $\pi_{\Pcal}(a)=\pi_{\Pcal}(b)$ for all $\Pcal\in\min M$. Therefore, $a=b$ or $a,b\in\bigcap_{\Pcal\in\min M}\Pcal$. Now the injectivity follows from Corollary \ref {CorMinPrimeIntersection}. In particular, $M\cong\im\pi$ is a subbinoid of $\prod_{\Pcal\in\min M}M/\Pcal$ with $\pi_{\Pcal}(M)=M/\Pcal$.
\end {proof}

\begin {Corollary}\label{CorStrongProjLimReduzierung}
Let $M$ be a binoid, $I=\spec M$, $\varphi_{\Pcal\Qcal}:M/\Qcal\rto M/\Pcal$ the canonical projection for $\Qcal\subseteq\Pcal$, and $(M/\Pcal, (\varphi_{\Pcal\Qcal})_{\Qcal\subseteq\Pcal})_{\Pcal,\Qcal\in I}$ the inverse system of integral quotient binoids considered in Example \ref{ExpProjLim}. Then $\slim M/\Pcal\cong M_{\opred}$.
\end {Corollary}
\begin {proof}
By Corollary \ref{CorSpecMred}, we may assume that $M$ is reduced. First observe that every element in $\varprojlim M/\Pcal$ of the form $([a]_{\Pcal})_{\Pcal\in I}$, $a\in M$, is strongly compatible; here, $[a]_{\Pcal}$ denotes the image of $a\in M$ under the canonical projection $M\rto M/\Pcal$. Indeed, if $[a]_{\Pcal}$,  $[a]_{\Qcal}\not=\infty$, then $a\not\in\Pcal$ and $a\not\in\Qcal$, which implies that  $a\not\in\Pcal\cup\Qcal\in\spec M$, and hence $\infty\not=\varphi_{\Pcal\cup\Qcal,\Pcal}([a]_{\Pcal})=\varphi_{\Pcal\cup\Qcal,\Qcal}([a]_{\Qcal})$. In particular, there is a well-defined binoid homomorphism 
$$\phi:M\Rto\operatorname{s}\mina\varprojlim_{\Pcal\in I} M/\Pcal\komma\quad a\lto([a]_{\Pcal})_{\Pcal\in I}\pkt$$
We will show that $\phi$ is bijective to prove the statement. For the surjectivity assume that $(c_{\Pcal})_{\Pcal\in I}\in\varprojlim M/\Pcal$ is strongly compatible. Thus, if $c_{\Pcal}=[a]_{\Pcal}\in (M/\Pcal)\opkt$ and $c_{\Pcal^{\prime}}=[b]_{\Pcal^{\prime}}\in(M/\Pcal^{\prime})\opkt$, then $\varphi_{\Qcal\Pcal}([a]_{\Pcal})=\varphi_{\Qcal\Pcal^{\prime}}([b]_{\Pcal^{\prime}})\not=\infty$ for some $\Qcal\supseteq\Pcal,\Pcal^{\prime}$. This equivalent to $a=b$ because all $\varphi_{\Qcal\Pcal}$, $\Pcal\subseteq\Qcal$, are the identity map on $(M/\Pcal)\setminus\ker\varphi_{\Qcal\Pcal}$. The injectivity follows now from Corollary \ref{CorSubdirectProd} since $([a]_{\Pcal})_{\Pcal\in I}=([b]_{\Pcal})_{\Pcal\in I}$ for $a,b\in M$ implies that $([a]_{\Pcal})_{\Pcal\in J}=([b]_{\Pcal})_{\Pcal\in J}$, where $J:=\min M\subseteq I$, hence $a=b$.
\end {proof}



\bigskip 

\chapter {Basic concepts of binoid algebras} \label{Chap2Basics}
\markright{\ref{Chap2Basics} Basic concepts of binoid algebras}

\begin {Convention}
Throughout this chapter, $K$ denotes a ring and whenever we refer to the monoid or binoid structure of a ring we mean the one defined by the multiplication unless otherwise stated.
\end {Convention}

In this chapter, we turn to the algebras and modules associated to binoids, $N\mina$sets, and $N\mina$binoids.  Before defining the binoid algebra, we recall needed facts on monoid algebras that translate to binoid algebras. For instance, a necessary and sufficient condition for a binoid algebra being an integral domain is given. The crucial correspondence of the sets of $K\mina$points of a binoid and its algebra, on which we will focus in the next two chapters, is stated in Proposition \ref{PropUnivPropBinoidA}. Some considerations on the connection of the ideal theories of a binoid and its algebra are made in the third section. Finally, basic notions and properties of modules and algebras associated to an $N\mina$sets and $N\mina$binoids, respectively, are assembled in the last two sections.

We want to draw the attention to two observations made in this chapter. Example \ref{ExpProjLimitAlgebras} demonstrates that projective limits need not commute with $K[-]$, and Example \ref{ExpNoMonoidAlgebra} showcases how easily one drops out of the theory of monoid algebras while still remaining in the context of binoid algebras.

\section {Monoid algebras} \label{SecMonoidAlgebra}
\markright{\ref{SecMonoidAlgebra} Monoid algebras}

\begin {Definition}
A \gesperrt{$K\mina$algebra} \index{algebra}$A$ is given by a not necessarily commutative  ring $A$ and a ring homomorphism $\varphi:K\rto A$ such that $\im\varphi$ lies in the center of $A$. The ring homomorphism $\varphi$ is called the \gesperrt{structure homomorphism} \index{algebra!structure homomorphism of an --}\index{structure homomorphism!A@-- of an algebra}of the $K\mina$algebra $A$. A ring homomorphism $A\rto A^{\prime}$ into another $K\mina$algebra $A^{\prime}$ is a \gesperrt{$K\mina$algebra homomorphism} \index{algebra!-- homomorphism}if it is compatible with the structure homomorphisms of $A$ and $A^{\prime}$. The set $\Hom_{K\minus\opalg}(A,A^{\prime})$ of all $K\mina$algebra homomorphisms $A\rto A^{\prime}$ will be denoted by $K\minSpec A$ when $A^{\prime}=K$.\nomenclature[HomK]{$\Hom_{K\minus\opalg}(A,A^{\prime})$}{set of all $K\mina$algebra homomorphisms $A\rto A^{\prime}$}\nomenclature[K6]{$K\minSpec A$}{$:=\Hom_{K\minus\opalg}(A,K)$ ($K\mina$spectrum of $A$)}
\end {Definition}

We will now recall the definition of monoid algebras and list the basic results about them without giving (detailed) proofs. As a reference, we cite \cite{SchejaStorch}, but proofs can also be found in \cite{OkninskiSA} and (for commutative monoids) in \cite{Gilmer}.

\begin {Definition}
Let $M$ be an arbitrary monoid. The \gesperrt{monoid algebra} \index{monoid!-- algebra}\index{algebra!monoid --}$KM$ \nomenclature[K1]{$KM$}{monoid algebra of $M$ over $K$}of $M$ over $K$ is the associative $K\mina$algebra with $K\mina$left module basis $T^{a}\in M$. The multiplication  for these basis elements is defined by using the operation of $M$,
$$T^{a}\cdot T^{b}:=T^{a+b}$$
for $a,b\in M$, which is then extended distributively to a multiplication on $KM$. In case $M$ is a group, $KM$ is called the \gesperrt{group algebra} \index{group algebra}\index{algebra!group --}of $M$ over $K$.
\end {Definition}

The ring $K$ is a subring of $KM$ via the ring monomorphism $K\embto KM$, $r\mto rT^{0}$. If $K\not=0$, there is a monoid embedding $M\embto KM$ with $a\mto T^{a}$ such that $M$ can be considered as a submonoid of $KM$. Every element $f\in KM$ can be written uniquely as
$$f=\sum_{a\in M}r_{a}T^{a}$$
with $r_{a}\in K$ such that $r_{a}\not=0$ for only finitely many $a\in M$. Moreover, $KM$ is an $M\mina$graded $K\mina$algebra
$$KM=\bigoplus_{a\in M}KT^{a}\komma$$
where $KT^{a}:=\{rT^{a}\mid r\in K\}$. The elements of the $K\mina$module $KT^{a}$ are called \gesperrt{homogeneous} \index{element!homogeneous --}of \gesperrt{degree} $a$. The element $0$ is of every degree by convention. Clearly, unless $K=0$ the monoid algebra $KM$ is commutative if and only if $M$ is commutative (and $K$, what we always assume).

\begin {Example}
The polynomial algebra $K[X_{i}\mid i\in I]\cong K(\N^{(I)})$ over $K$ is a monoid algebra.
\end {Example}

By the following proposition, the monoid algebra is a universal object and therefore unique up to isomorphism.

\begin {Proposition} \label {PropUnivPropMonoidA}
Given a $K\mina$algebra $A$ and a monoid homomorphism $\phi:M\rto (A,\cdot)$, there exists a unique $K\mina$algebra homomorphism $\phi:KM\rto A$ with $\tilde{\phi}(T^{a})=\phi(a)$ for all $a\in M$.
\end {Proposition}
\begin {proof}
See \cite[Korollar 52.2]{SchejaStorch}.
\end {proof}

\begin {Corollary} \label{CorPropMonoidA}
Let $M$ be a monoid.
\begin {ListeTheorem}
\item Let $\alpha:K\rto L$ be a homomorphism of rings and $\varphi:M\rto N$ a monoid homomorphism. Then there is a unique ring homomorphism $KM\rto LN$ with $ra\mto\alpha(r)\varphi(a)$.
\item If $\aideal$ is an ideal in $K$, then $KM/\aideal KM\cong(K/\aideal)M$.
\item If $S$ is a multiplicative system in $K$, then $(KM)_{S}\cong K_{S}M$.
\item Given a finite family $(M_{i})_{i\in I}$ of monoids, there is a canonical isomorphism 
$$K\Big(\prod_{i\in I}M_{i}\Big)\,\,\cong\,\,\bigotimes_{i\in I}\!\!{}_{_{K}}KM_{i}\pkt$$
\item If $A$ is a $K\mina$algebra, then $A\otimes_{K}KM\cong AM$.
\end {ListeTheorem}
\end {Corollary}
\begin {proof}
For a more precise demonstration see \cite[\S52, Beispiel 2,3 and 4]{SchejaStorch}, \cite[\S80, Beispiel 13]{SchejaStorch}, and \cite[\S81, Beispiel 11]{SchejaStorch}. Outline: the first assertion is an easy consequence of the universal property (Proposition \ref {PropUnivPropMonoidA}). The second follows from $(1)$ with $N=M$, $\varphi=\id_{M}$, and $\alpha=\pi:K\rto K/\aideal$. Finally, one may easily derive that the three remaining isomorphisms are given by the canonical $K_{S}\mina$algebra homomorphism $(\sum_{a\in M}r_{a}T^{a})/s\mto\sum_{a\in M}(r_{a}/s)T^{a}$ and  by the two canonical $K\mina$algebra homomorphisms $T^{(a_{i})_{i\in I}}\mto\otimes_{i\in I} a_{i}$ and  $b\otimes_{K}\sum_{a\in M}r_{a}T^{a}\mto\sum_{a\in M}(r_{a}b)T^{a}$.
\end {proof}

The following proposition determines under which conditions the monoid algebra contains no zero-divisors. 

\begin {Proposition} \label {PropMonoidADomain}
Let $M$ be a monoid and $K\not=0$. The monoid algebra $KM$ is a domain if and only if $K$ is an integral domain and $M^{\infty}$ is torsion-free and cancellative.
\end {Proposition}
\begin {proof}
See \cite[Theorem 8.1]{GrilletCS} or \cite[Theorem 4.18]{BrunsGubeladze}.\end {proof}

\bigskip

\section {Binoid algebras} \label{SecBinoidAlgebra}
\markright{\ref{SecBinoidAlgebra} Binoid algebras}

For the sake of completeness, we recall the definition of the binoid algebra introduced at the beginning of Section \ref{DefBinoidAlgebra} (cf.\ page \pageref{DefBinoidAlgebra}). Since we will encounter only commutative binoid algebras in the subsequent chapters, two examples of non-commutative algebras that can be realized as binoid algebras are given at the end of this section, namely matrix algebras and path algebras.

\begin {Definition}
Let $M$ be a binoid. The \gesperrt{binoid algebra} \index{binoid!-- algebra}\index{algebra!binoid --}of $M$ is defined to be the quotient algebra 
$$KM/(T^{\infty})=:K[M]\komma$$ 
\nomenclature[K3]{$K[M]$}{binoid algebra of $M$ over $K$}where $(T^{\infty})$ is the ideal in $KM$ generated by the element $T^{\infty}$. In general, if $\Ical$ is an ideal in $M$, we will denote the ideal (resp.\ $K\mina$submodule) of $KM$ generated by $T^{a}$, $a\in\Ical$, by
$$K\Ical:=(T^{a}\mid a\in\Ical)\subseteq KM$$
and by\nomenclature[K2]{$K\Ical$}{ideal in $KM$ generated by $T^{a}$, $a\in\Ical$}
$$K[\Ical]:=K\Ical/(T^{\infty})=(T^{a}\mid a\in\Ical)\subseteq K[M]$$
the associated ideal (resp.\ $K\mina$submodule) of $K[M]$.\nomenclature[K4]{$K[\Ical]$}{ideal in $K[M]$ generated by $T^{a}$, $a\in\Ical$}
\end {Definition}

By definition, $K[M]$ may be identified with the set of all formal sums $\sum_{a\in A}r_{a}T^{a}$ with $A\subseteq M\opkt$ finite and $r_{a}\in K$, where the multiplication is generated by
$$r_{a}T^{a}\cdot s_{b}T^{b}=\begin {cases}
r_{a}s_{b}T^{a+b}&\text{, if }a+b\not=\infty\komma\\
0&\text{, otherwise.}
\end {cases}$$
The $K\mina$module isomorphism
$$K[M]\,\,\cong\,\bigoplus_{a\in M\opkt}KT^{a}$$
gives rise to a graded $K\mina$algebra homomorphism 
$$KM\Rto K[M]$$
with $\ker=KT^{\infty}$. In this vein, the binoid algebra $K[M]$ emerges from the monoid algebra $KM$ by glueing together the absorbing elements of $M$ and $KM$. Thus, the notion of binoid algebras generalizes that of monoid algebras. Given an ideal $\Ical\subseteq M$, the ideals $K\Ical$ and $K[\Ical]$ are monomial ideals of $KM$ and $K[M]$, respectively. If $K\not=0$, the composition $M\embto KM\rto K[M]$ yields a binoid embedding 
$$\iota_{M}:M\Rto K[M]\komma\quad a\lto T^{a}\komma$$ 
such that $M$ can be considered as a subbinoid of $(K[M],\cdot,1,0)$.

\begin {Example}
\begin {ListeTheorem}
\item[]
\item For the zero binoid, we obtain $K[\{\infty\}]=0$ and for the trivial binoid $K[\trivial]=K$.
\item If $M$ is an integral binoid, then $K[M]\cong KM\opkt$. In particular,  $K[M/M\Uplus]$ is the group algebra $KM\okreuz$, see also Corollary \ref{CorPropBinoidA}(3) bellow.
\item The binoid algebras of the one-generated binoids, cf.\ Corollary \ref{CorClassificationOnegenerated}, as well as those of some two-generated binoids were described up to isomorphism in Section \ref{SecFGbinoids}.
\end {ListeTheorem}
\end {Example}

\begin {Proposition} \label{PropUnivPropBinoidA}
Given a binoid $M$, a $K\mina$algebra $A$, and a binoid homomorphism $\varphi:M\rto A$, there is a unique $K\mina$algebra homomorphism $\phi:K[M]\rto A$ such that the diagram
$$\xymatrix{
M\ar[d]_{\iota}\ar[r]^{\varphi}&A\\
K[M]\ar[ur]_{\phi}&}$$
commutes. 

In particular, if $A=K$, then 
$$K\minspec M\,\,\cong\,\, K\minSpec K[M]$$
as semigroups.
\end {Proposition} 
\begin {proof}
Consider the composition of the canonical maps $M\rto KM\stackrel{\pi}{\rto}K[M]$. By the universal property of the monoid algebra, cf. Proposition \ref{PropUnivPropMonoidA}, $\varphi$ induces a $K\mina$algebra homomorphism $\tilde{\varphi}:KM\rto A$ with  $rT^{a}\mto\alpha(r)\varphi(a)$, where $\alpha:K\rto A$ is the structure homomorphism, $r\in K$, and $a\in M$. Since $\ker\pi=KT^{\infty}\subseteq\ker\tilde{\varphi}$, the $K\mina$algebra homomorphism $\tilde{\varphi}$ induces a ring homomorphism $\phi:K[M]\rto A$ with $\phi\pi=\tilde{\varphi}$ such that $rT^{a}\mto\alpha(r)\varphi(a)$, $r\in K$, $a\in M$, which shows that $\phi$ is a $K\mina$algebra homo\-morphism.
\end {proof}

The binoid algebra is uniquely determined by its universal property.

\begin {Corollary} \label {CorUnivPropBinoidA}
Given a ring homomorphism $\alpha: K\rto L$ and a binoid homomorphism $\varphi: M\rto M^{\prime}$, there is a unique ring homomorphism $\phi:K[M]\rto L[M^{\prime}]$ with $\phi(rT^{a})=\alpha(r)T^{\varphi(a)}$, $r\in K$, $a\in M\opkt$.
\end {Corollary}
\begin {proof}
By Proposition \ref {PropUnivPropBinoidA}, $\varphi$ induces the $K\mina$algebra homomorphism $\tilde{\phi}$ of the following commutative diagram
$$\xymatrix{
M\ar[d]_{\iota_{M}}\ar[r]^{\varphi}&M^{\prime}\ar[d]_{\iota_{M^{\prime}}}&\\
K[M]\ar[r]^{\tilde{\phi}}&K[M^{\prime}]\ar[r]^{\tilde{\alpha}}&L[M^{\prime}]\pkt}$$
Then $\phi=\tilde{\alpha}\tilde{\phi}$ is the unique ring homomorphism, where $\tilde{\alpha}:rT^{a}\mto\alpha(r)T^{a}$.
\end {proof}

\begin {Remark}\label {RemUniPropSpecial}
There are two special cases of Corollary \ref {CorUnivPropBinoidA}.
\begin {ListeTheorem}
\item  If $\varphi=\id_{M}$, there is a unique ring homomorphism $\alpha[M]:K[M]\rto L[M]$.  It is easily verified that every subset of $L$ that generates $L$ as a $K\mina$module (or $K\mina$algebra) generates $L[M]$ as a $K[M]\mina$module (or $K[M]\mina$algebra), and that every linear independent set of elements of $L$ over $K$ is linearly independent in $L[M]$ over $K[M]$. In particular, bases retain unchanged while switching to binoid algebras. Moreover, if $\alpha$ is surjective, then so is $\alpha[M]$.
\item The other specialization is when $L=K$ . In this case there is a unique $K\mina$algebra (!) homomorphism 
$$K[\varphi]:K[M]\Rto K[N]\pkt$$
$\varphi$ is injective or surjective if and only if $K[\varphi]$ is so. In particular, $K[-]$ is a covariant functor from the category of binoids to the category of graded $K\mina$algebras.
\end {ListeTheorem}
\end {Remark}

\begin {Example}
Let $M$ be a commutative binoid. By Corollary \ref {CorSubdirectProd}, there is an embedding $M\embto\prod_{P\in\min M}M/P$ for every reduced binoid $M$, which induces an injective $K\mina$algebra homomorphism $K[M]\rto K\big[\prod_{P\in\min M}M/P\big]$.
\end{Example}

\begin {Example} \label{ExpProjLimitAlgebras}
Let $((M_{i})_{i\in I},(\varphi_{ji})_{i\ge j})_{i,j\in I}$ be an inverse system of binoids. Then $(A_{i},(f_{ji})_{i\ge j})_{i,j\in I}$ with $A_{i}:=K[M_{i}]$ and $f_{ji}:=K[\varphi_{ji}]:A_{i}\rto A_{j}$ defines an inverse system of $K\mina$algebras. The projective limit
$$\varprojlim K[M_{i}]\,=\,\Big\{(F_{i})_{i\in I}\in\prod_{i\in I}K[M_{i}]\,\,\Big|\,\, f_{ji}(F_{i})=F_{j}\text{ for all }i\ge j\Big\}$$
might but need not coincide with $K[\varprojlim M_{i}]$. However, there is always an embedding
$$K[\varprojlim M_{i}]\Rto\varprojlim K[M_{i}]\komma\quad T^{(a_{i})_{i\in I}}=\prod_{i\in I}T_{i}^{a_{i}}\lto(T_{i}^{a_{i}})_{i\in I}\komma$$
coming from the family $g_{j}:\varprojlim M_{i}\rto M_{j}\rto K[M_{j}]$, $j\in I$, of binoid homomorphisms which satisfy $f_{ij}g_{j}=g_{i}$ for all $i,j\in I$. Hence, there is a unique binoid embedding 
$$\varphi:\varprojlim M_{i}\Rto \varprojlim K[M_{i}]\quad\text{with}\quad(a_{i})_{i\in I}\lto(T_{i}^{a_{i}})_{i\in I}$$
that factors through $K[\varprojlim M_{i}]$ by Proposition \ref{PropUnivPropBinoidA}.

As an example, consider the inverse system $((M/\Pcal)_{\Pcal\in I},(\varphi_{\Pcal\Qcal})_{\Qcal\subseteq\Pcal})_{\Pcal,\Qcal\in I}$ of Example \ref{ExpProjLim}, where $M$ is a commutative binoid, $I=\spec M$, and $\varphi_{\Pcal\Qcal}:M/\Qcal\rto M/\Pcal$ are the canonical projections for $\Qcal\subseteq\Pcal$. Returning to the two particular cases discussed there, we obtain the following:

If $\min M=\{\Qcal\}$, then $\varprojlim M/\Pcal=M/\Qcal$, and similarly we get
$$\varprojlim_{\Pcal\in I} K[M/\Pcal]\,=\,\Big\{(f_{\Pcal\Qcal}(F))_{\Pcal\in I}\in\prod_{\Pcal\in I}K[M/\Pcal]\,\,\Big|\, F\in K[M] \Big\}\,\,\cong\,\, K[M/\Qcal]\komma$$
where $f_{\Pcal\Qcal}$ is the $K\mina$algebra epimorphism $K[\varphi_{\Pcal\Qcal}]:K[M/\Qcal]\rto K[M/\Pcal]$ induced by the canonical projection $\varphi_{\Pcal\Qcal}:M/\Qcal\rto M/\Pcal$.
In particular, $K[\varprojlim_{\Pcal\in I}M\Pcal]=\varprojlim K[M/\Pcal]$.

If $M=\,\,\bigcupbidot_{i=1}^{n}\N^{\infty}$, then $\varprojlim M/\Pcal=(\N_{\ge1}^{\infty})^{n}\cup\{(0\kpkt 0)\}$, cf.\ Example \ref{ExpProjLim}. Therefore,
$$K[\varprojlim M/\Pcal]\,=\,K\oplus\Big(\bigoplus_{\infty\not=a\in(\N_{\ge1}^{\infty})^{n}}KX^{a}\Big)\komma$$
where $X^{a}=X_{1}^{a_{1}}\cdots X_{n}^{a_{n}}$ for $\infty\not=a=(a_{1}\kpkt a_{n})\in(\N_{\ge1}^{\infty})^{n}$. Here the inverse system of $K\mina$algebras is given by $A_{0}:=K[M/M\Uplus]=K$ and $A_{i}:=K[\N^{\infty}]\cong K[X]$, $i\in\{1\kpkt n\}$, with $K\mina$algebra homomorphisms $f_{ii}=\id_{A_{i}}$, $i\in\{0,1\kpkt n\}$ and $f_{i0}:K[X]\rto K$, $F\rto\const(F)$, $i\not=0$, where $\const(F)$ denotes the constant term of $F\in K[M]$. \nomenclature[const]{$\const(F)$}{constant term of $F\in K[M]$}Thus, 
$$\varprojlim_{\Pcal\in I} K[M/\Pcal]\,=\,\Big\{(c,F_{1}\kpkt F_{n})\in K\times\prod_{i=1}^{n}K[M_{i}]\,\,\Big|\, F_{i}\in K[X], \const(F_{i})=c\in K,i\in\{1\kpkt n\}\Big\}\pkt$$
In particular, $\varprojlim K[M/\Pcal]\not=K[\varprojlim M/\Pcal]$ for $n\ge2$ since for instance $(0,X\kpkt 2X)\in\varprojlim K[M/\Pcal]$ lies not in the image of $K[\varprojlim M/\Pcal]\embto\varprojlim K[M/\Pcal]$.
\end {Example}

\begin {Corollary} \label{CorPropBinoidA}
Let $M$ be a binoid.
\begin {ListeTheorem}
\item If $N$ is a subbinoid of $M$, the binoid algebra $K[N]$ is a $K\mina$subalgebra of $K[M]$.
\item If $\aideal$ is an ideal in $K$, then $(K/\aideal)[M]\cong K[M]/\aideal K[M]$.
\item If $M$ is commutative and $\Ical$ an ideal in $M$, then $K[M/\Ical]\cong K[M]/K[\Ical]\cong KM/K\Ical$.
\item If $M$ is commutative and $S$ is a subbinoid of $M$, then $\widetilde{S}:=\{T^{a}\mid a\in S\}$ defines a multiplicative system in $K[M]$ and there is an isomorphism $\widetilde{S}^{-1}(K[M])\cong K[M_{S}]$.
\item If $A$ is a $K\mina$algebra, then $A\otimes_{K}K[M]\cong A[M]$.
\item If $S$ is a multiplicative system in $K$, then $K[M]_{S}\cong K_{S}[M]$.
\end {ListeTheorem}
\end {Corollary}
\begin {proof}
(1) Is clear. Let $\pi:K\rto K/\aideal$ be the canonical surjection. The kernel of the induced ring epimorphism $K[\pi]:K[M]\rto(K/\aideal)[M]$, cf.\ Corollary \ref {CorUnivPropBinoidA}, consists of all $f\in K[M]$ with coefficients in $\aideal$, which proves (2). (3) The binoid epimorphism $M\rto M/\Ical$ induces a $K\mina$algebra homomorphism $K[M]\rto K[M/\Ical]$ by Corollary \ref {CorUnivPropBinoidA}. Its kernel is given by $\bigoplus_{a\in I\opkt}KT^{a}=K[\Ical]$, hence $K[M/\Ical]\cong K[M]/K[\Ical]$. The latter isomorphism is clear. 
(4) By Corollary \ref{CorUnivPropBinoidA}, the canonical map $\iota_{S}:M\rto M_{S}$ induces a $K\mina$algebra homomorphism $\tilde{\iota}_{S}:K[M]\rto K[M_{S}]$. Let $\iota:K[M]\rto \widetilde{S}^{-1}(K[M])$ denote the canonical ring homomorphism. Since $\iota(S)\subseteq(\widetilde{S}^{-1}(K[M]))^{\times}$, there is by the universal property of localization (for rings) the following commutative diagram
$$\xymatrix{
K[M]\ar[d]_{\tilde{\iota}_{S}}\ar[r]^{\iota}&\widetilde{S}^{-1}(K[M])\ar[dl]^{\psi}\komma\\
K[M_{S}]&}$$
where the induced $K\mina$algebra homomorphism $\psi$ is an isomorphism with inverse given by rewriting elements of $K[M_{S}]$ in the following way 
$$\sum_{j=1}^{n}r_{j}T^{a_{j}\minus f_{j}}=\sum_{j=1}^{n}\frac{r_{j}T^{a_{j}\minus 0}\prod_{i\not=j}T^{f_{i}\minus 0}}{\prod_{i=1}^{n}T^{f_{i}\minus 0}}\quad\lto\quad\sum_{j=1}^{n}\frac{r_{j}T^{a_{j}}\prod_{i\not=j}^{n}T^{f_{i}}}{\prod_{i=1}^{n}T^{f_{i}}}=\frac{\sum_{j=1}^{n}r_{j}T^{a^{\prime}}}{T^{f}}$$
with $a^{\prime}=a_{j}+\sum_{i\not=j}f_{i}\in M$ and $f=\sum_{i=1}^{n}f_{i}\in S$. This is an element in $\widetilde{S}^{-1}(K[M])$. (5) We have $A\otimes_{K}K[M]\cong (A\otimes_{K}KM)/(1\otimes T^{\infty})\cong AM/(T^{\infty})\cong A[M]$, where the isomorphism in the middle is due to Corollary \ref{CorPropMonoidA}(2). (6) We have $K[M]_{S}\cong K_{S}\otimes_{K}K[M]\cong K_{S}[M]$, where the latter isomorphism is due to (5).
\end {proof}

\begin {Corollary}
Let $\Ical\subseteq M$ be an ideal and $e\in\Ical$ an idempotent element such that $K\Ical$ is a $K\mina$algebra with identity $e$. Then $\varphi:KM\rto K\Ical\times K[M/\Ical]$ defined by $\phi(x)=(ex,\pi(x))$, where $\pi:KM\rto K[M/\Ical]$ denotes the canonical homomorphism, is an isomorphism of algebras. In particular, $KM\cong K\times K[M]$ as $K\mina$algebras.
\end {Corollary}
\begin {proof}
By assumption, there is a $K\mina$algebra isomorphism $KM\cong K\Ical\times (1-e)KM$, and hence $(1-e)KM\cong KM/K\Ical\cong K[M/\Ical]$ by Corollary \ref {CorPropBinoidA}. Since $\pi((1-e)x)=\pi(x)$ for every $x\in KM$ and $\ker\pi\cap(1-e)KM=K\Ical\cap(1-e)KM=0$, the restriction of $\pi$ to $(1-e)KM$ is an isomorphism $(1-e)KM\cong K[M/\Ical]$ from which the statement follows. The case $\Ical=KT^{\infty}$ proves the supplement.
\end {proof}

\begin {Theorem} \label {ThBinoidADomain}
The binoid algebra $K[M]$ is a domain if and only if $K$ is a domain and $M$ a regular torsion-free binoid.
\end {Theorem}
\begin {proof}
If $K[M]$ is a domain, the binoid $M$ has to be integral; that is, $K[M]\cong KM\opkt$. The theorem follows now from Proposition \ref {PropMonoidADomain}.
\end {proof}

\begin {Corollary} \label{CorBinoidADomain}
If $M$ is a torsion-free binoid and $K$ a domain, then $K[M/M\Uplus]$ is a domain.
\end {Corollary}
\begin {proof}
This is an immediate consequence of Theorem \ref{ThBinoidADomain} since $M/M\Uplus=M^{\times}\cup\{\infty\}$ is regular and, by Lemma \ref{LemModuloRadical}, torsion-free.
\end {proof}

Here are two important examples of non-commutative algebras that can be realized as binoid algebras.

\begin {Example}(Matrix algebras) 
For $V=\{1\kpkt n\}$, $n\ge1$, let $M=\{e_{i,j}\mid (i,j)\in V\times V\}\cup\zero$ be the binoid with addition given by
$$e_{i,j}+e_{k,l}\,=\,\begin {cases}
e_{i,l}&\text{, if }j=k\komma\\
\infty&\text{, otherwise.}
\end {cases}$$
If $M_{n}(K)$\nomenclature[M3]{$M_{n}(K)$}{algebra of $(n\times n)\mina$matrices} denotes the $K\mina$algebra of $(n\times n)\mina$matrices, then there is a $K\mina$algebra isomorphism
$$K[M]\stackrel{\sim}{\Rto} M_{n}(K)\komma\quad T^{e_{i,j}}\lto E_{ij}$$
where $E_{ij}=(a_{kl})_{1\le k,l\le n}$ is the elementary matrix with $a_{kl}=1$ if $(k,l)=(i,j)$ and $0$ otherwise.
\end {Example}

\begin {Example} (Path algebras)
A finite \gesperrt{quiver} \index{quiver}$\Qscr$ is a finite directed graph possibly with multiple arrows and loops. In other words, $\Qscr=(V,A,s,t)$ is a quadruple consisting of the finite sets of vertices $V$ and arrows $A$ between them, and two maps $s,t:A\rto V$, which associate to each arrow $\beta\in A$ its source $s(\beta)\in V$ and its target $t(\beta)\in V$. Let $a,b\in V$ be two vertices. A path from $a$ to $b$ in $\Qscr$ is a sequence $(a|\beta_{1}\cdots\beta_{\ell}|b)$, where $\beta_{i}\in A$, $i\in\{1\kpkt\ell\}$, with
$s(\beta_{1})=a$ and $s(\beta_{i+1})=t(\beta_{i})$ for $1\le i<\ell$, and $t(\beta_{\ell})=b$, which may be  briefly denoted by $(\beta_{1}\cdots\beta_{\ell})$ or illustrated as follows
$$a=a_{0}\stackrel{\beta_{1}}{\Rto}a_{1}\stackrel{\beta_{2}}{\Rto}a_{2}\stackrel{\beta_{3}}{\Rto}\cdots\stackrel{\beta_{\ell}}{\Rto}a_{\ell}=b\pkt$$
The number $\ell$ is called the length of $(\beta_{1}\cdots\beta_{\ell})$. Furthermore, we associate to each vertice $a\in V$ a stationary path at $a$ of lenght zero, denoted by $\varepsilon_{a}=(a||a)$. If $Q_{\ell}$ denotes the set of all paths of lengh $\ell$, then $Q_{0}\cong V$ and $Q_{1}\cong A$. Let $n=|A|$. Consider the set 
$$M_{\ell}\,=\,\{(i_{1}\kpkt i_{\ell})\in \{1\kpkt n\}^{\ell} \mid(\beta_{i_{1}}\cdots\beta_{i_{\ell}})\in Q_{\ell}\}\cup\zero$$
of $\ell\mina$tuples which belong to a path in $\Qscr_{\ell}$ with $\infty$ adjoint. Tuples with no corresponding path in $\Qscr$ will be identified with $\infty$. Such tuples do not exist if and only if $A=\{\beta\}$ and $s(\beta)=t(\beta)$. Otherwise, the tuple $\infty$ has every length. It follows that 
$$M_{^{_{\Qscr}}}\,:=\,\bigoplus_{\ell\ge 0}M_{\ell}$$
is a binoid with respect to $(i_{1}\kpkt i_{\ell})\circ(j_{1}\kpkt j_{k})=(i_{1}\kpkt i_{\ell},j_{1}\kpkt j_{k})$ if this tuple belongs to a path in $Q_{\ell+k}$ and $\infty$ otherwise. The corresponding binoid algebra 
$$K[M_{^{_{\Qscr}}}]\,=\,\bigoplus_{(i_{1}\kpkt i_{\ell})\in M_{^{_{\Qscr}}}\opkt}KX_{i_{1}}\cdots X_{i_{\ell}}$$
is given by the quotient $K\langle X_{1}\kpkt X_{n}\rangle/KX^{\infty}$ of the free associative $K\mina$algebra with non-commuting variables $X_{1}\kpkt X_{n}$ and 
$KX^{\infty}={_{K}(}\:X_{i_{1}}\cdots X_{i_{\ell}}\mid (i_{1}\kpkt i_{\ell})\notin M_{\ell}, \ell\ge0)$. If $K[\Qscr]$ denotes the path algebra of $\Qscr$ with $K\mina$basis $\{(\beta_{i_{1}}\cdots\beta_{i_{\ell}})\in Q_{\ell}\mid\ell>0\}$, then the natural map given by $(\beta_{i_{1}}\cdots\beta_{i_{k}})\mto X_{i_{1}}\cdots X_{i_{k}}$ is a $K\mina$algebra isomorphism 
$$K[\Qscr]\stackrel{\sim}{\Rto} K[M_{^{_{\Qscr}}}]\pkt$$
In case $K$ is an algebraically closed field, the 
preceding definition of the path algebra associated to a quiver agrees with the classical one from representation theory (\cite[Chapter 4]{Benson}). \index{path algebra}It is a well-known fact that  $K[\Qscr]$ has a unit, namely $1_{K[\Qscr]}=\sum_{a\in V}\varepsilon_{a}$, if and only if $V$ is finite. In this case, $\{\varepsilon_{a}\mid a\in V\}$ is a complete set of primitive orthogonal idempotents for $K[\Qscr]$. 

It might be useful to have another, slightly different definition of a (finite) quiver changing this situation. For this, let $(V,A,s,t)$ as before, but instead of associating to each vertice a stationary path, we consider only the empty (or totally stationary) path $\emptyset$ subject to the rule
$$\emptyset\beta\,=\,\beta\emptyset\,=\,\beta$$
for all $\beta\in A$ (i.e.\ $Q_{0}=\{\emptyset\}$). Accordingly, we call a quiver defined in this way a \gesperrt{quiver with} $1$ ($=\emptyset$)\index{quiver!-- with $1$}, and denote it by $\Qscr^{1}$. This definition makes even more sense when passing to the binoid algebra of $M_{^{_{\Qscr^{1}}}}$ because the unit element of the quiver coincides with the unit of the $K\mina$algebra $K[\Qscr^{1}]$. If $V=\{a\}$, then $\Qscr=\Qscr^{1}$ with $1=\varepsilon_{a}$.
\end {Example}

\bigskip

\section {Ideals in binoid algebras} \label{SecIdealsBinoidAlgebra}
\markright{\ref{SecIdealsBinoidAlgebra} Ideals in binoid algebras}

\begin {Convention}
In this section, arbitrary binoids are assumed to be \emph{commutative}.
\end {Convention}

Recall that every ideal $\Ical$ in a binoid $M$ defines a monomial ideal in $K[M]$, namely\index{monomial ideal}\index{ideal!monomial --}
$$K[\Ical]\,=\,\bigoplus_{a\in\Ical}KT^{a}\pkt$$
Conversely, to each ideal $\aideal\subseteq K[M]$ there is an ideal of exponents in $M$, $$\Ical(\aideal)\,:=\,\{a\in M\mid T^{a}\in\aideal\}\pkt$$

\begin {Lemma}\label{LemMonIdealCorresp}
\begin {ListeTheorem}
\item[]
\item Let $\Ical$ and $\Jcal$ be two ideals in $M$.
\begin {ListeTheorem}
\item [(a)] $K[\Ical\cup\Jcal]\,=\,K[\Ical]+K[\Jcal]$.
\item [(b)] $K[\Ical\cap\Jcal]\,=\,K[\Ical]\cap K[\Jcal]$.
\item [(c)] $K[\Ical+\Jcal]\,=\,K[\Ical]\cdot K[\Jcal]$.
\end {ListeTheorem}
\item [(2)]$K[\Ical(\aideal)]\subseteq\aideal$ for every ideal $\aideal$ in $K[M]$. Moreover, if $\aideal$ and $\bideal$ are monomial ideals in $K[M]$, then
\begin {ListeTheorem}
\item [(a)] $K[\Ical(\aideal)]=\aideal$. In particular, $\Ical(-)$ establishes a bijection between the set of ideals of $M$ and the set of monomial ideals of $K[M]$ with inverse $K[-]$, namely $\Ical\mto K[\Ical]$ and $\Ical(\aideal)\mapsfrom\aideal$.
\item [(b)] $\bideal\subseteq\aideal$ if and only if $\Ical(\bideal)\subseteq\Ical(\aideal)$.
\item [(c)] If $\aideal$ is a radical ideal then so is $\Ical(\aideal)$.
\end {ListeTheorem}
\end {ListeTheorem}
\end {Lemma}
\begin {proof}
All assertions are easily verified.
\end {proof}

\begin {Proposition}\label{PropBinoidMax=AlgebraMax}
Let $K$ be a field. If $M$ is positive, then $K[M\Uplus]$ is a maximal ideal in $K[M]$. 
\end {Proposition}
\begin{proof} We have $K\,\cong\, K[\trivial]=K[M/M\Uplus]=K[M]/K[M\Uplus]$, where the latter identity is due to Corollary \ref{CorPropBinoidA}(3).
\end{proof}

This result fails for non-positive binoids. As an easy example consider the binoid group $M=(\Z/n\Z)^{\infty}$ with $n\ge2$. Then $K[M]/K[M\Uplus]=K[M/M\Uplus]=K[M]\cong K[X]/(X^{n}-1)$ is not a field since $X-1$ is a zero-divisor. In particular, $K[M\Uplus]=K[\zero]=0$ is not a maximal ideal in $K[M]$.

\begin {Corollary}\nomenclature[Dimension6]{$\dim_{K}V$}{dimension of the $K\mina$module $V$}
Let $K$ be a field. If $M$ is finitely generated and positive, then
$$\dim_{K}K[M]/(K[M\Uplus])^{n}=\opH(n,M)\pkt\footnote{Here, we use the convention $\aideal^{0}:=R$ for an ideal $\aideal$ in a ring $R$.}$$
\end {Corollary}
\begin {proof}
We have
$$K[M]/(K[M\Uplus])^{n}\,\,=\,\,K[M]/K[nM\Uplus]\,\,=\,\,K[M/nM\Uplus]\,\,=\,\,\bigoplus_{a\in(M/nM\Uplus)\opkt}KT^{a}$$
as $K\mina$vector spaces, hence $\dim_{K}K[M]/(K[M\Uplus])^{n}=\#(M/nM\Uplus)-1=\opH(n,M)$.
\end {proof}

\begin{Proposition} \label{PropPrimeICorrespondence}
Let $M$ be a binoid. If $\pideal\in\spec K[M]$, then $\Ical(\pideal)\in\spec(M)$. In particular, every monomial prime ideal in $K[M]$ is of the form $K[\Pcal]$ for some prime ideal $\Pcal\in\spec M$. Conversely, if $M$ is torsion-free and regular, $\Pcal\in\spec M$, and $K$ a domain, then $K[\Pcal]$ is a (monomial) prime ideal in $K[M]$.
\end{Proposition}
\begin {proof}
Let $\pideal\in\spec K[M]$ and $a,b\in M$ with $a+b\in\Ical(\pideal)$. Then $T^{a+b}=T^{a}\cdot T^{b}\in\pideal$, which implies that $T^{a}\in\pideal$ or $T^{b}\in\pideal$ by the prime property. Hence, $a\in\Ical(\pideal)$ or $b\in\Ical(\pideal)$. This shows that $\Ical(\pideal)$ is prime. The additional observation follows now from Lemma \ref{LemMonIdealCorresp}(2a). Conversely, if $\Pcal$ is a prime ideal, then $M/\Pcal$ is integral and again torsion-free and cancellative by Lemma \ref{LemQuotientCancPosRepF} and Lemma \ref{LemQuotientTorF}. Hence, $K[M/\Pcal]\cong K[M]/K[\Pcal]$ is a domain by Theorem \ref{ThBinoidADomain}, which implies that $K[\Pcal]\in\spec K[M]$.
\end {proof}

\begin {Corollary}
Let $K$ be a domain and $M$ a torsion-free cancellative binoid. If $\pideal$ is a minimal prime ideal of $K[M]$, then $\Ical(\pideal)$ is a minimal prime in $M$.
\end {Corollary}
\begin {proof}
By Proposition \ref{PropPrimeICorrespondence}, only the minimality of $\Ical(\pideal)$ needs to be verified. So suppose that $\Qcal\subseteq M$ is a prime ideal with $\Qcal\subsetneq\Ical(\pideal)$. Then $K[\Qcal]\in\spec K[M]$ by Proposition \ref{PropPrimeICorrespondence} with $K[\Qcal]\subsetneq K[\Ical(\pideal)]=\pideal$, which contradicts the minimality of $\pideal$. Hence, $\Ical(\pideal)$ is a minimal prime ideal.
\end {proof}

\bigskip

\section {$K[N]\mina$modules} \label{SecBinoidModules}
\markright{\ref{SecBinoidModules} $K[N]\mina$modules}

The concept of binoid algebras for binoids can be generalized to arbitrary $N\mina$sets $(S,p)$; that is, to every $N\mina$set $(S,p)$ one can associate a $K[N]\mina$module $K[S]$. The following definition of $K[S]$ suggests what the subsequent two results show: many of the results on binoid algebras given in Section \ref{SecBinoidAlgebra} can be generalized to $N\mina$sets and their associated $K[N]\mina$modules since every binoid is a $\trivial\mina$set. In this thesis, the results are not needed in their more general module theoretic form, but rather as they are stated in Section \ref{SecBinoidAlgebra} for binoid algebras. However, there is a purely module theoretic result, cf.\ Proposition \ref{PropNsetsDecompositionDirectSum}, even when the $N\mina$sets are given by binoids (i.e.\ $N=\trivial$). We close this section by applying this result to an important example, the blowup binoid.

\begin {Definition}
For an $N\mina$set $(S,p)$ let $K[S]$ denote the $K[N]\mina$module given by the set of all formal sums
$\sum_{s\in T}r_{s}X^{s}$ with $T\subseteq S\opkt$ finite, $r_{s}\in K$, and scalar multiplication defined by\nomenclature[K5]{$K[S]$}{$K[N]\mina$module of the $N\mina$set $(S,p)$}
$$r_{a}X^{a}\cdot r_{s}X^{s}:=\begin {cases}
r_{a}r_{s}X^{a+s}&\text{, if }a+s\not=p\komma\\
0&\text{, otherwise,}
\end{cases}$$
where $a\in N$ and $s\in S$.
\end {Definition}

\begin {Example}
\begin {ListeTheorem}
\item[]
\item The ideal $K[\Ical]$ is a $K[M]\mina$module that can be understood as coming from the $M\mina$set $(\Ical,\infty)$.
\item Let $Q\subseteq\Z^{d}$ be a submonoid and $F$ a proper filter of $Q^{\infty}$. We consider the subset
$$F-Q=\{f-q\mid f\in F,q\in Q\}$$
of $\Z^{d}$, and its translates $a+F-Q$ for $a\in\Z^{d}$. The pointed set $M_{a}:=(a+F-Q)^{\infty}$ is a $Q^{\infty}\mina$set with respect to the operation $Q^{\infty}\times M_{a}\Rto M_{a}$,
$$(q,u)\lto\begin {cases}
q+u&\text{, if }q+u\in M_{a}\komma\\
\infty&\text{, otherwise.}
\end {cases}$$
The set $F-Q$ is called the \emph{injective hull} of the filter $F$, and the $K[Q^{\infty}]\mina$module $K[M_{a}]$ an \emph{indecomposable injective} of $Q$, cf.\ \cite[Definition 11.7]{MillerSturmfels}.
\end {ListeTheorem}
\end {Example}

\begin {Lemma}
$S$ is a finitely generated $N\mina$set if and only if $K[S]$ is a finitely generated $K[N]\mina$module.
\end {Lemma}
\begin {proof}
This is clear.
\end {proof}

Every $K[N]\mina$module $V$ is also an $N\mina$set with respect to the (left) operation of $N$ on $(V,0)$ given by
$$N\times V\Rto V\komma\quad (a,v)\lto a+v:=X^{a}\cdot v\pkt$$

In particular, $K[S]$ is again an $N\mina$set for every $N\mina$set $(S,p)$, and unless $K=0$ there is a canonical injective $N\mina$map
$$\iota_{S}:S\Rto K[S]\komma\quad s\lto X^{s}\pkt$$

The associated $K[N]\mina$module of an $N\mina$set is a universal object and therefore unique up to isomorphism.

\begin {Proposition} \label{PropUnivPropKNmodule}
Let $(S,p)$ be an $N\mina$set and $V$ a $K[N]\mina$module. Every $N\mina$map $\varphi:(S,p)\rto (V,0)$ gives rise to a unique $K[N]\mina$module homomorphism $\phi:K[S]\rto V$ such that the diagram
$$\xymatrix{
S\ar[d]_{\iota_{S}}\ar[r]^{\varphi}&V&\\
K[S]\ar[ur]_{\phi}&}$$
commutes. 
\end {Proposition}
\begin {proof}
The map $\phi$ defined by $\phi(\sum_{s\in T}r_{s}X^{s}):=\sum_{s\in T}r_{s}\varphi(s)$ with $r_{s}\in K$, $s\in T$, and $T\subseteq S\opkt$ finite, is a well-defined $K\mina$module homomorphism with $\phi\iota_{S}=\varphi$. To verify that $\phi$ is also a homomorphism of $K[N]\mina$modules observe that  $\varphi(a+s)=a+\varphi(s)=X^{a}\cdot\varphi(s)$ for all $a\in N$ and $s\in S$. Therefore, if $F=\sum_{a\in A}r_{a}X^{a}\in K[N]$ and $\sum_{s\in T}r_{s}X^{s}\in K[S]$ for finite sets $A\subseteq N\opkt$ and $T\subseteq S\setminus\{p\}$, then
\begin {align*}
\phi\bigg(F\cdot\sum_{s\in T}r_{s}X^{s}\bigg)&=\phi\bigg(\sum_{a\in A,s\in T}r_{a}r_{s}X^{a+s}\bigg)\\
&=\sum_{a\in A,s\in T}r_{a}r_{s}\varphi(a+s)\\
&=\sum_{a\in A,s\in T}r_{a}r_{s}X^{a}\varphi(s)\\
&=F\cdot\sum_{s\in T}r_{s}\varphi(s)\\
&=F\cdot\phi\bigg(\sum_{s\in T}r_{s}X^{s}\bigg)\pkt\qedhere
\end {align*}
\end {proof}

\begin {Corollary} \label{CorIndKNmodulehom}
Given an $N\mina$map $\varphi:S\rto T$ of $N\mina$sets $(S,p)$ and $(T,q)$, there exists a unique $K[N]\mina$module homomorphism $\phi:K[S]\rto K[T]$ such that the diagram
$$\xymatrix{
S\ar[d]_{\iota_{S}}\ar[r]^{\varphi}&T\ar[d]^{\iota_{T}}&\\
K[S]\ar[r]^{K[\varphi]}&K[T]}$$
commutes. In particular, there is a binoid homomorphism 
$$(\map_{p}S,\circ,\id_{S},\varphi_{\infty})\Rto(\End_{K}K[S], \circ,\id_{K[S]},0_{K[S]})$$ 
with $\varphi\mto K[\varphi]$, where $\End_{K}K[S]$ denotes the ring of all $K\mina$module homomorphisms $K[S]\rto K[S]$ and $0_{K[S]}$ the zero map.\nomenclature[End]{$\End_{K}V$}{ring of all $K\mina$module endomorphisms of $V$}
\end {Corollary}
\begin {proof}
This is an immediate consequence of Proposition \ref{PropUnivPropKNmodule} applied to the $N\mina$map given by the composition $\iota_{T}\varphi:S\rto T\rto K[T]$. In particular, $K[\varphi]:K[S]\rto K[T]$ is given by $K[\varphi](F)=\sum_{s\in S^{\prime}}r_{s}X^{\varphi(s)}$ for $F=\sum_{s\in S^{\prime}}r_{s}X^{s}\in K[S]$ with $r_{s}\in K$, $s\in S^{\prime}$, and $S^{\prime}\subseteq S$ finite. This shows that $K[\varphi\circ\psi]=K[\varphi]\circ K[\psi]$, $K[\id_{S}]=\id_{K[S]}$, and $K[\varphi_{\infty}]=0_{K[S]}$, which implies the supplement.
\end {proof}

\begin {Proposition}
Let $(S,p)$ be an $N\mina$set. $K[S]$ is a $K[N]\mina$module such that there is a commutative diagram
$$\xymatrix{
N\ar[r]^{\!\!\!\!\!\!\!\!\varphi}\ar[d]_{\iota_{N}}&\map_{p}S\ar[d]^{K[\minus]}\\
K[N]\ar[r]^{\!\!\!\!\!\!\!\!\phi}&\End_{K}K[S]\komma}$$
where $\varphi$ is a binoid homomorphism and $\phi$ a ring homomorphism.
\end {Proposition}
\begin{proof}
The operation $N\times S\rto S$, $(a,s)\mto a+s$, induces a canonical operation $K[N]\times K[S]\rto K[S]$ generated by $(r_{a}T^{a},r_{s}T^{s})\mto r_{a}r_{s}T^{a+s}$. Thus, there is a ring homomorphism $\phi:K[N]\rto\End_{K}K[S]$ with $F\mto(\phi(F):G\lto G\cdot F)$, which makes the diagram commutative.
\end{proof}

\begin {Proposition} \label{PropNsetsDecompositionDirectSum}
Given a family $(S_{i},p_{i})_{i\in I}$ of $N\mina$sets, there is a $K[N]\mina$module isomorphism:
$$K\Big[\bigcupdot_{i\in I}S_{i}\Big]\,\,\cong\,\,\bigoplus_{i\in I}K[S_{i}]\pkt\\[-1mm]$$
\end {Proposition}
\begin {proof}
By Corollary \ref{CorIndKNmodulehom}, the family $(\varphi_{i})_{i\in I}$ of injective $N\mina$maps $S_{i}\embto \bigcupdot_{i\in I}S_{i}=:S$ with  $s\mto(s;i)$, $i\in I$, induces a family $\iota_{i}:K[S_{i}]\rto K[S]$, $i\in I$, of $K[N]\mina$module homomorphisms, which gives rise to a $K[N]\mina$module homomorphism
$$\phi:\bigoplus_{i\in I}K[S_{i}]\Rto K[S]$$
with $\phi((F_{i})_{i\in I})=\sum_{i\in I}\iota_{i}(F_{i})$, cf.\ \cite[Satz 39.1]{SchejaStorch}. On the other hand, the surjections $S\rto S_{i}$, $i\in I$, with $(s;k)\mto(s;i)$ if $k=i$ and $p_{i}$ otherwise, yield by Corollary \ref{CorIndKNmodulehom} $K[N]\mina$module homomorphisms $K[S]\rto K[S_{i}]$, $i\in I$, which induce a $K[N]\mina$module homomorphism
$$\psi:K[S]\Rto\bigoplus_{i\in I}K[S_{i}]$$
with $\sum_{a\in A}r_{a}X^{a}\mto\sum_{i\in I}\sum_{a\in S_{i}\cap A}r_{a}X^{a}$, where $A\subseteq\bigcupdot_{i\in I}S_{i}$ is a finite set. It is easily checked that $\phi$ and $\psi$ are inverse to each other.
\end {proof}

\begin {Remark}\label{RemBlowupBinoid}
For a binoid $M$ and an ideal $\Ical\subseteq M$ define the \gesperrt{blowup binoid} \index{blowup binoid}\index{binoid!blowup --}$R_{\Ical}$ of $\Ical$ to be the pointed union of the family $(n\Ical,\infty)_{n\ge0}$ of pointed sets, i.e. 
\begin {align*}
R_{\Ical}&:=\bigcupdot_{n\ge0}n\Ical\komma\\[-5 mm]
\intertext{with addition defined by}\\[-10 mm]
(f;n)+(g;m)&:=(f+g;n+m)\pkt
\end {align*}
The identity element is then given by $(0;0)$ and the absorbing element by the element that arises by glueing the elements $(\infty;n)$, $n\ge0$, together. Thus, $R_{\Ical}$ is generated by $M=0\cdot\Ical$ and $\Ical=1\cdot\Ical$, and can be considered as a subbinoid of $M\wedge \N^{\infty}$ with respect to $(a;n)\mto a\wedge n$. By Proposition \ref{PropNsetsDecompositionDirectSum}, it follows that the binoid algebra of the blowup binoid
$$K[R_{\Ical}]\,=\,\bigoplus_{n\ge0}K[n\Ical]\,=\,\bigoplus_{n\ge0}K[\Ical]^{n}$$
is the blowup algebra of $K[\Ical]\subseteq K[M]$, cf.\ \cite[Chapter 5.2]{Eisenbud}. Moreover, $(n\Ical,\infty)$ is an $M\mina$set for every $n\ge0$ and $R_{\Ical}$ is an $M\mina$binoid via the embedding $M\embto R_{\Ical}$, $a\mto(a;0)$. Hence, we have a commutative diagram
$$\xymatrix{
&R_{\Ical}\ar[dr]&\\
M\ar[ru]^{\iota}\ar[rd]_{\pi}&&\!\!\!\!\!\!\!\!\!\!\!\!R_{\Ical}\wedge_{M}(M/\Ical)\!\!\!\!\!\!\!\!\!\!\!\!\!\!\!\!&\cong\,\,\bigcupdot\limits_{n\ge0}n\Ical\big/(n+1)\Ical\komma\\
&M/\Ical\ar[ur]&}$$
where the isomorphism is given as $M\mina$binoids as one easily checks (where the addition on the pointed union is defined as for $R_{\Ical}$), cf.\ Proposition \ref{PropSmashPointed}. The $M\mina$binoid $\bigcupdot_{n\ge0}n\Ical\big/(n+1)\Ical=:\opgr_{M}(\Ical)$ is called the \gesperrt{associated graded binoid} \index{associated graded!-- binoid}\nomenclature[GR]{$\opgr_{M}(\Ical)$}{associated graded binoid of $M$ with respect to $\Ical\subsetneq M$}of the binoid $M$ with respect to the ideal $\Ical\subseteq M$. The terminology is justified because for the binoid algebra we obtain
\begin {align*}
K[R_{\Ical}\wedge_{M}(M/\Ical)]&\,\,\,\cong\,\,\, K[R_{\Ical}]\otimes_{K[M]}K[M/\Ical]\\
&\,\,\,\cong\,\,\,\bigg(\bigoplus_{n\ge0}K[\Ical]^{n}\bigg)\otimes_{K[M]}\big(K[M]\big/K[\Ical]\big)\\
&\,\,\,\cong\,\,\,\bigoplus_{n\ge0}\Big(K[\Ical]^{n}\otimes_{K[M]}\big(K[M]\big/K[\Ical]\big)\Big)\\
&\,\,\,\cong\,\,\,\bigoplus_{n\ge0}\big(K[\Ical]^{n}\big/K[\Ical]^{n+1}\big)\\
&\,\,\,=:\opgr_{K[M]}(K[I])\komma
\end {align*}
where $\opgr_{A}(\aideal)$ denotes the \gesperrt{associated graded ring} \index{associated graded!-- ring}\nomenclature[GR]{$\opgr_{A}(\aideal)$}{associated graded ring of $A$ with respect to $\aideal\subsetneq A$}of the ring $A$ with respect to the (ring) ideal $\aideal\subsetneq A$, cf.\ \cite[Chapter 5.1]{Eisenbud}. An important case is when $\Ical=M\Uplus$. If $M$ is positive, then
$$\opgr_{K[M]}(K[M\Uplus])\,\,=\,\, K[R_{M\Uplus}]\otimes_{K[M]}K\komma$$
where $K$ is a $K[M]\mina$algebra with respect to $K[M]\rto K$, $T^{a}\mto 1$ if $a=0$ and $0$ otherwise.
\end {Remark}

\begin {Example} \label{ExpNoMonoidAlgebra}
Consider the integral binoid $M:=\N_{\ge2}^{\infty}\cup\{0\}\subseteq\N^{\infty}$.
The blowup binoid of $M$ with respect to $M\Uplus=\N_{\ge2}^{\infty}$ is given by
$$R_{M\Uplus}\,=\,\bigcupdot_{n\ge0}nM\Uplus\,=\,\bigcupdot_{n\ge0}\N_{\ge 2n}^{\infty}\pkt$$
Since $nM/(n+1)M=\N_{\ge2n}^{\infty}/\N_{\ge2n+2}^{\infty}=\{2n,2n+1\}^{\infty}$, we have $(2n+1;n)+(2n+1;n)=(4n+2;2n)=\infty$ though $(2n+1;n)\not=\infty$ in
$$\opgr_{M}(M\Uplus)\,=\,\bigcupdot_{n\ge0}\{(2n;n),(2n+1;n)\}^{\infty}\pkt$$
Thus, the associated graded binoid $\opgr_{M}M\Uplus$ is not reduced, hence not integral. In particular, $$K[\opgr_{M}(M\Uplus)]\,=\,\opgr_{KM\opkt}(M\opkt\Uplus)$$
is a binoid algebra that is not an integral domain.
\end {Example}

\bigskip

\section {Binoid algebras of $N\mina$binoids} \label{SecNBinoidAlgebra}
\markright{\ref{SecNBinoidAlgebra} Binoid algebras of $N\mina$binoids}

\begin {Convention}
In this section, $N$ always denotes a \emph{commutative} binoid.
\end {Convention}

By Remark \ref{RemUniPropSpecial}(2), the structure homomorphism $\varphi:N\rto M$ of an $N\mina$binoid $M$ induces a ring homomorphism $K[\varphi]:K[N]\rto K[M]$. This defines a $K[N]\mina$algebra structure on $K[M]$.

\begin {Corollary}\label{CorIndNBinoidHom}
Every homomorphism $\varphi: M\rto M^{\prime}$ of $N\mina$binoids induces a unique $K[N]\mina$ algebra homomorphism $\phi:K[M]\rto K[M^{\prime}]$ with $\phi(rT^{a})=rT^{\varphi(a)}$, $r\in K$, $a\in M\opkt$.
\end {Corollary}
\begin {proof}
This is clear by Corollary \ref{CorUnivPropBinoidA} (with $K=L$ and $\alpha=\id_{K}$) and the commutative diagram
$$\xymatrix{
&N\ar[dl]_{\psi}\ar[dr]^{\psi^{\prime}}&\\
M\ar[rr]^{\varphi}&&M^{\prime}\komma}$$
where $\psi$ and $\psi^{\prime}$ are the structure homomorphisms of $M$ and $M^{\prime}$, respectively.
\end {proof}

The following corollary can be generalized to $N\mina$sets, which we will not need here.

\begin {Corollary} \label{CorSmash=Tensor}
Let $(M_{i})_{i\in I}$ be a finite family of commutative $N\mina$binoids. Then
$$K\Big[\bigwedge_{i\in I} \!\!{}_{_{N}}M_{i}\Big]\,\,\cong\,\, \bigotimes_{i\in I} \!\!{}_{_{K[N]}}K[M_{i}]$$
as $K\mina$algebras. In particular, 
$$K\Big[\bigwedge_{i\in I}M_{i}\Big]\,\,\cong\,\, \bigotimes_{i\in I} \!\!{}_{_{K}}K[M_{i}]\pkt$$
\end {Corollary}
\begin {proof}
Using an inductive argument, it suffices to prove the statement for $I=\{1,2\}$. So let $M_{1}=:M$ and $M_{2}=:M^{\prime}$ be commutative binoids. There is a commutative diagram of $K\mina$algebras,
$$\xymatrix{
&K[M]\ar[d]^{\iota}\ar[dr]^{\varphi}&\\
K[N]\ar[ur]\ar[dr]&K[M]\otimes_{K[N]}K[M^{\prime}]&K[M\wedge_{N}M^{\prime}]\\
&K[M^{\prime}]\ar[u]_{\iota^{\prime}}\ar[ur]_{\varphi^{\prime}}&}$$
where $\iota$ and $\iota^{\prime}$ are the canonical inclusions, and $\varphi$ and $\varphi^{\prime}$ are the $K[N]\mina$algebra homomorphisms induced by the inclusions $M\embto M\wedge_{N}M^{\prime}$ and $M^{\prime}\embto M\wedge_{N}M^{\prime}$, cf.\  Corollary \ref{CorIndNBinoidHom}. By the universal property of the tensor product, cf.\ \cite[Satz 80.9]{SchejaStorch}, we have a $K\mina$algebra homomorphism
$$\psi:K[M]\otimes_{K[N]}K[M^{\prime}]\Rto K[M\wedge_{N}M^{\prime}]$$
with $\psi(F\otimes_{K[N]}G)=\varphi(F)\varphi^{\prime}(G)$, which is even a $K[N]\mina$algebra homomorphism since $\varphi$ and $\varphi^{\prime}$ are so. On the other hand, the well-defined binoid homomorphism
$$M\wedge_{N}M^{\prime}\Rto K[M]\otimes_{K[N]}K[M^{\prime}]\quad\text{with}\quad a\wedge_{N}a^{\prime}\lto T^{a}\otimes_{K[N]}T^{a^{\prime}}$$
induces by Proposition \ref{PropUnivPropBinoidA} a $K\mina$algebra homomorphism $K[M\wedge_{N}M^{\prime}]\rto K[M]\otimes_{K[N]}K[M^{\prime}]$  with $r(a\wedge_{N}a^{\prime})\mto r(T^{a}\otimes_{K[N]}T^{a^{\prime}})$, which is also a $K[N]\mina$algebra homomorphism and inverse to $\psi$.
\end {proof}

\begin {Corollary}
If $M$ is a finitely generated commutative $N\mina$binoid, then $K[M]$ is a finitely generated $K[N]\mina$algebra.
\end {Corollary}
\begin {proof}
Let $\{x_{1}\kpkt x_{r}\}$ be a generating set of $M$. By Remark \ref{RemUniPropSpecial}(2) and Corollary \ref{CorSmash=Tensor}, the binoid epimorphism
$$\varphi:N\wedge(\N^{r})^{\infty}\Rto M\komma$$
$a\wedge(n_{1}\kpkt n_{r})\mto \varphi(a)+n_{1}x_{1}\pluspkt n_{r}x_{r}$, induces a $K\mina$algebra epimorphism
$$K[\varphi]:K[N][X_{1}\kpkt X_{r}]\Rto K[M]\komma$$
with $rT^{a}X^{\nu}\mto rT^{\varphi(a)}X^{\nu}$, $\nu=(n_{1}\kpkt n_{r})$. Obviously, $K[\varphi]$ is a $K[N]\mina$algebra homomorphism. This proves the statement.
\end {proof}

\begin {Remark} \label {RemCompositionsAlgebra}
\begin{ListeTheorem}
\item[]
\item The supplement of Corollary \ref {CorSmash=Tensor} can be obtained from earlier results when taking extra assumptions on the binoids into account. If, for instance, $M$ and $M^{\prime}$ are integral, then $M\wedge M^{\prime}\cong(M\opkt\!\!\times M^{\prime^{_{\bullet}}})^{\infty}$ by Lemma \ref{LemSmashRules}(3). Therefore, 
\begin {align*}
K[M\wedge M^{\prime}]&\,\,\cong\,\, K(M\opkt\!\!\times M^{\prime^{_{\bullet}}})\\
&\,\,\cong\,\,KM\opkt\!\!\otimes_{K}KM^{\prime^{_{\bullet}}}\\
&\,\,\cong\,\,K[M]\otimes_{K}K[M^{\prime}]\komma
\end {align*}
where the isomorphism in the middle is due to Corollary \ref {CorPropMonoidA}(4). However, for arbitrary (not necessarily integral) commutative binoids this also follows from Corollary \ref{CorPropBinoidA}(3) since $M\wedge M^{\prime}\cong (M\times M^{\prime})/\Ical$ with $\Ical=\{(a,b)\mid a=\infty$ or $b=\infty\}$, which yields
\begin {align*}
K[M\wedge M^{\prime}]&\,\,\cong\,\, K(M\times M^{\prime})/(T^{c}\mid c\in\Ical)\\
&\,\,\cong\,\, (KM\otimes_{K}KM^{\prime})/(KT^{\infty}\otimes_{K}KM^{\prime}+KM\otimes_{K}KT^{\infty})\\
&\,\,\cong\,\, (KM/(T^{\infty}))\otimes_{K}(KM^{\prime}/(T^{\infty}))\\
&\,\,\cong\,\, K[M]\otimes_{K}K[M^{\prime}]\pkt
\end {align*}
The isomorphism $K[(M\times M^{\prime})/\Ical]\cong K[M]\otimes_{K}K[M^{\prime}]$ can already be found in \cite[Chapter 4, Lemma 10]{OkninskiSA}.
\item Corollary \ref{CorSmash=Tensor} has shown that the tensor product of binoid algebras comes from the smash product of the involved binoids, while for the product we get
\begin {align*}
K[M\times M^{\prime}]&\,\,=\,\,K(M\times M^{\prime})/(T^{(\infty,\infty)})\\
&\,\,\cong\,\, KM\otimes_{K}KM^{\prime}/(U^{\infty}\otimes_{K}V^{\infty})\komma
\end {align*}
cf.\ Corollary \ref{CorPropMonoidA}(4). By Proposition \ref {PropUnivPropBinoidA}, there is a commutative diagram
$$\xymatrix{
M\times M^{\prime}\ar[r]^{\iota_{M\times M^{\prime}}}\ar[d]_{\iota:=(\iota_{M},\iota_{M^{\prime}})}&K[M\times M^{\prime}]\ar[dl]^{\tilde{\iota}}\\
K[M]\times K[M^{\prime}]&}$$
with $\tilde{\iota}\iota_{M\times M^{\prime}}=\iota$. Thus, $(U^{a},V^{a^{\prime}})=\iota(a,a^{\prime})=\tilde{\iota}\iota_{M\times M^{\prime}}(a,a^{\prime})=\tilde{\iota}(T^{(a,a^{\prime})})$. Clearly, $\tilde{\iota}$ is surjective, but it is not injective if $M^{\prime}$ and $M$ are nonzero, since then $\tilde{\iota}(T^{(0,0)})=(1,1)=\tilde{\iota}(T^{(0,\infty)}+T^{(\infty,0)})$ while $T^{(0,0)}\not=T^{(0,\infty)}+T^{(\infty,0)}$.
\item The natural binoid epimorphisms $\pi_{\wedge}:M\times M^{\prime}\rto M\wedge M^{\prime}$ and $\pi_{\cupbidot}:M\wedge M^{\prime}\rto M\cupbidot M^{\prime}$, cf.\ Remark \ref{RemProdEpisPos}, induce by Corollary \ref {CorUnivPropBinoidA} the $K\mina$algebra epimorphisms
$$K[M\times M^{\prime}]\stackrel{K[\pi_{\wedge}]}{\Rto}K[M\wedge M^{\prime}]\stackrel{K[\pi_{\cupbidot}]}{\Rto}K[M\cupbidot M^{\prime}]\pkt$$
If $M$ and $M^{\prime}$ are commutative, then $M\cupbidot M^{\prime}=(M\wedge M^{\prime})/J$ with $J=\{a\wedge a^{\prime}\mid a\not=0$ and $a^{\prime}\not=0\}$, cf.\ Example \ref{ExCompositionsIdeal}. By Corollary \ref{CorPropBinoidA}(3), this gives
$$K[M\cupbidot M^{\prime}]\,\,=\,\,K[M]\otimes_{K}K[M^{\prime}]/(U^{a}\otimes V^{a^{\prime}}\mid a,a^{\prime}\not=0)\pkt$$
\end{ListeTheorem}
\end {Remark}

\begin {Example}
We apply the results of the preceding remark to $M=M^{\prime}=\N^{\infty}$. This yields the following $K\mina$algebra isomorphisms:
\begin {align*}
K[\N^{\infty}\times\N^{\infty}]&\,\,=\,\, K\N^{\infty}\otimes_{K}K\N^{\infty}/(U^{\infty}\otimes_{K}V^{\infty})\vspace{1.5cm} \\
&\,\,\cong\,\, (K\times K[X])\otimes_{K}(K\times K[Y])/((1,0)\otimes (1,0))\vspace{1.5cm} \\
&\,\,\cong\,\, K[X]\times K[Y]\times K[X,Y]\vspace{1.5cm} \\
K[\N^{\infty}\wedge\N^{\infty}]&\,\,\cong\,\, K[X]\otimes_{K}K[Y]\cong K[X,Y]\vspace{1.5cm} \\
K[\N^{\infty}\cupbidot\N^{\infty}]&\,\,\cong\,\, K[X]\otimes_{K}K[Y]/(X\otimes Y)\,\,\cong\,\, K[X,Y]/(XY)\komma\vspace{1.5cm} 
\end {align*}
where $K[X]$, $K[Y]$, and $K[X,Y]$ denote the polynomial rings with indeterminates $X$ and/or $Y$.
\end {Example}

\bigskip

\chapter {Topology of commutative binoids} \label{ChapTopology}
\markright{\ref{ChapTopology} Topology og binoids}

In this chapter, we investigate the spectrum and the $K\mina$spectrum, $K$ a field, of a commutative binoid as topological spaces. The topology on the spectrum is defined analogous to the Zariski topology on the spectrum of a ring. All results  resemble those from ring theory, and the same applies to the dimension theory of binoids, with which we deal in the third section. We introduce the booleanization of a binoid $M$, which is a fairly easier binoid to study, and show that its spectrum is homeomorphic to that of $M$. Similar to a result on characters of semigroups, we give a criterion when the $K\mina$points of a binoid $M$ separate the elements of $M$, cf.\ Proposition \ref{PropFuncEquiv}. The topology on $K\minspec M$ comes from that on $K\minSpec K[M]$, and the description of its properties will be continued in the next chapter. Here we mainly focus on connectedness properties, in particular of hypersurfaces, and show that $K\minspec M$ is the union of its cancellative components, which serve as the key tool to prove our main criterion for connectedness, cf.\ Theorem \ref{ThConnectedness}.

\begin{Convention}
In this chapter, arbitrary binoids are assumed to be \emph{commutative} unless otherwise stated.
\end{Convention}

\bigskip

\section {Spectrum} \label {SecTopology}
\markright {\ref{SecTopology} Spectrum}

In this section, we define a topology on the spectrum of a binoid $M$ similar to the Zariski topology on the spectrum of a ring $R$. An excellent description of the latter can be found in \cite[Chapter 3.A]{PatilStorch}, which we will partly follow since many results translate directly to binoids and their spectra. The reader may also consult \cite[Chapter I \S1]{EGA}, \cite[Chapter II \S 1]{Mumford}, or \cite[Chapter II \S 4.1]{BourbakiCA} for spectra of rings. The spectrum of a binoid equipped with the Zariski topology and its subspace consisting of all minimal prime ideals have been studied in detail in \cite{Kist}. See also \cite{CortinasWeibelCDH}.

\medskip

Recall, cf.\ Section \ref{SecPrime}, that the semibinoid $(\spec M,\cup, M\Uplus)$ is a binoid with identity element $\zero$ if $M$ is integral. However, one can always turn $(\spec M,\cup, M\Uplus)$ into a binoid by adjoining $\{\emptyset\}$ as an identity element irrespectively of the integrality of $M$.

\begin {Definition}
Let $M$ be  binoid. The commutative binoid $(\spec M\cup\{\emptyset\},\cup, \emptyset, M\Uplus)$ is called the \gesperrt{extended spectrum} \index{spectrum!extended --}of $M$, which will be denoted by $\spec^{\emptyset}M$\nomenclature[Spec1]{$\spec^{\emptyset}M$}{extended spectrum of the binoid $M$}.\end {Definition}

\begin {Corollary} \label {CorSpecDualFilt}
For $M\not=\zero$, the boolean binoids $\specE M$, $M\dual$, and $\Fcal(M)_{\cap}$ are isomorphic.
\end {Corollary}
\begin {proof}
Extending the maps from Corollary \ref {CorHomFiltSpec} to these binoids by $(\chi_{M}=)\,\alpha_{\emptyset}\leftrightarrow\emptyset\leftrightarrow M$ gives the well-defined isomorphisms.
\end {proof}

\begin {Definition}
For a subset $A\subseteq M$ let
$$\opV(A):=\{\Pcal\in\spec M\mid A\subseteq\Pcal\}\pkt$$
\nomenclature[V]{$\opV(-)$}{closed set in the Zariski topology on $\spec M$}The complements in $\spec M$ and $\specE M$ are denoted by
$$\opD(A):=\spec M\setminus\opV(A)=\{\Pcal\in\spec M\mid A\not\subseteq\Pcal\}\quad\text{and}\quad\opDE(A):=\specE M\setminus\opV(A)=\opD(A)\cup\{\emptyset\}\pkt$$\nomenclature[D]{$\opD(-)$}{open set in the Zariski topology on $\spec M$}\nomenclature[D]{$\opDE(-)$}{open set in the Zariski topology on $\specE M$}
If $A=\{f\}$, we shall write $\opV(f)$ instead of $\opV(\{f\})$ and the same for the complements $\opD(f)$ and $\opDE(f)$.
\end {Definition}

\begin {Lemma} \label {LemBcalV}
Let $M$ be a binoid, $f,g\in M$, and $A,A_{i},B\subseteq M$, $i\in I$, subsets.
\begin {ListeTheorem}
\item $\bigcap_{i\in I}\opV(A_{i})=\opV(\bigcup_{i\in I}A_{i})$.
\item $\opV(f)\cup\opV(g)=\opV(f+g)$.
\item If $A\subseteq B$, then $\opV(B)\subseteq\opV(A)$.
\item $\opV(f)=\emptyset$ if and only if $f\in M^{\times}$.
\item $\opV(f)=\spec M$ if and only if $f$ is nilpotent.
\item $\opV(A)=\opV(_{_{M}}\!\langle A\rangle)=\opV(\sqrt{_{_{M}}\!\langle A\rangle})$.
\end {ListeTheorem}
\end {Lemma}
\begin {proof}
(1)-(4) are straightforward and (5) follows immediatley from Corollary \ref{CorMinPrimeIntersection}.
(6) Set $_{_{M}}\!\langle A\rangle=:\Acal$. We always have $A\subseteq\Acal\subseteq\sqrt{\Acal}$, and hence $\opV(\sqrt{\Acal})\subseteq\opV(\Acal)\subseteq\opV(A)$ by (3). For the other inclusions, it suffices to show that $\opV(A)\subseteq\opV(\sqrt{\Acal})$. Given $\Pcal\in\opV(A)$, we have to verify that $\sqrt{\Acal}\subseteq\Pcal$. So let $f\in\sqrt{\Acal}$. This means $nf\in\Acal$ for some $n\ge 1$, which implies that $nf=a+m$ with $a\in A\subseteq\Pcal$ and $m\in M$. In particular, $nf\in\Pcal$, and therefore $f\in\Pcal$ by the prime property.
\end {proof}

By taking complements, Lemma \ref {LemBcalV} translates to:

\begin {Lemma} \label {LemBcalD}
Let $M$ be a binoid, $f,g\in M$, and $A,A_{i},B\subseteq M$, $i\in I$, subsets.
\begin {ListeTheorem}
\item $\bigcup_{i\in I}\opD(A_{i})=\opD(\bigcup_{i\in I}A_{i})$.
\item $\opD(f)\cap\opD(g)=\opD(f+g)$.
\item If $A\subseteq B$, then $\opD(A)\subseteq\opD(B)$.
\item $\opD(f)=\spec M$ if and only if $f\in M^{\times}$.
\item $\opD(f)=\emptyset$ if and only if $f$ is nilpotent.
\item $\opD(A)=\opD(_{_{M}}\!\langle A\rangle)=\opD(\sqrt{_{_{M}}\!\langle A\rangle})$.
\end {ListeTheorem}
\end {Lemma}
\begin {proof}
All statements are immediate consequences of Lemma \ref {LemBcalV}.
\end {proof}

The preceding lemmata show that $\spec M$ and $\specE M$ are topological spaces for every binoid $M$, where the closed sets are given by $\opV(A)$, $A\subseteq M$. In either case this topology will be called the \gesperrt{Zariski topology}\index{Zariski topology!-- on $\spec M$}\index{topology!Zariski --}. The complements $\opD(A)$ and $\opDE(A)$ are the open sets, where $(\opD(f))_{f\in M}$ and $(\opDE(f))_{f\in M}$ define a basis of open sets of the Zariski topology on $\spec M$ and $\specE M$, respectively. From now on, when we consider the (extended) spectrum of a binoid as a topological space, we mean the (extended) spectrum together with the Zariski topology.

\begin {Remark}\label{RemLowUpperTop}
\begin {ListeTheorem}
\item []
\item There are two natural topologies on a finite poset $(X,\subseteq)$: the \emph{lower topology}\index{lower topology}\index{topology!lower --}, where the open sets are given by the subset-closed sets of $X$. In other words,
$$D\subseteq X\quad\text{open}\quad:\eq\quad(A\in D\text{ and }B\subseteq A\Rarrow B\in D)$$
(which means that the closed sets are given by the superset-closed sets), and the \emph{upper topology}\index{upper topology}\index{topology!upper --}, where the open sets are given by the superset-closed sets of $X$. In other words,
$$D\subseteq X\quad\text{open}\quad:\eq\quad(A\in D\text{ and }A\subseteq B\Rarrow B\in D)\pkt$$
Thus, if $\spec M$ is a finite set, for instance when $M$ is finitely generated, then the Zariski topology on $\spec M$ coincides with the lower topology on $(\spec M,\subseteq)$.
\item By Corollary \ref{CorHomFiltSpec}, $\spec M$ and $\Fcal(M)\setminus\{M\}$ are isomorphic as semigroups by taking the complements. Hence, considering a topology on $\spec M$ or on $\Fcal(M)$ is essentially the same. Those topological spaces that arise as the filtrum of a commutative monoid have been characterized in \cite[Satz 2.3.2]{BrennerFilter}.
\item When dealing with monoids the empty set is usually considered as a prime ideal to ensure that the spectrum is never empty, which were the case if $M$ is a group. By this convention, the set of all minimal prime ideals of $M$ is always a singleton. In particular, the spectrum admits always a unique generic point and is irreducible, cf.\ \cite[Chapter 1.4]{Ogus}, which still holds true if one adjoins an absorbing element to the monoid, cf.\ Corollary \ref{CorIrredCompMinPrime}.
\end{ListeTheorem}
\end {Remark}

\begin {Corollary}
Let $M$ be a binoid.
\begin {ListeTheorem}
\item If $\Ical$ is an ideal in $M$, then 
$$\spec(M/\Ical)\,\,\cong\,\,\opV(\Ical)\,\,\subseteq\spec M\pkt$$
\item If $S\subseteq M$ is a submonoid of $M$, then
$$\spec M_{S}\,\,\cong\,\,\{\Pcal\in\spec M\mid\Pcal\cap S=\emptyset\}\pkt$$
If, in addition, $S$ can be generated by a finite set $A\subseteq S$, then
$$\spec M_{S}\,\,\cong\,\,\opD(f)\,\,\subseteq\spec M\komma$$
where $f=\sum_{g\in A}g$. In particular, $\spec M_{f}=\opD(f)$ for every $f\in M$. 
\item The following statements are equivalent for $\Pcal\in\spec M$:
\begin {ListeTheorem}
\item [(a)] $M_{\Pcal}=M_{f}$ for some $f\in M$.
\item [(b)] $\spec M_{\Pcal}=\opD(f)$ for some $f\in M$.
\item [(c)] $\spec M_{\Pcal}$ is open in $\spec M$.
\end {ListeTheorem}
\end {ListeTheorem}
\end {Corollary}
\begin {proof}
(1) and the first part of (2) are just restatements of Corollary \ref{CorExtIdealPrime} and Corollary \ref{CorIndSpecLoc}. If $S$ is generated by the finite set $A\subseteq S$, we obtain from the first part and Lemma \ref{LemBcalD}(2) that
$$\spec M_{S}\,\,\cong\,\,\{\Pcal\in\spec M\mid g\not\in\Pcal\text{ for all }g\in A\}\,=\,\bigcap_{g\in A}\opD(g)\,=\,\opD(f)$$
with $f=\sum_{g\in A}g$. (3) The implication $(a)\Rightarrow(b)$ is clear by (2), and $(b)\Rightarrow(c)$ is trivial. So let $\spec M_{\Pcal}$ be open in $\spec M$. Then there is a subset $B\subseteq M$ such that $\spec M_{\Pcal}=\opD(B)=\bigcup_{f\in B}\opD(f)$. On the other hand, $\spec M_{\Pcal}=\{\Qcal\in\spec M\mid \Qcal\subseteq\Pcal\}$ by (2). Hence, $\Pcal\in\opD(f)$ for some $f\in B$, but then $\spec M_{\Pcal}\subseteq\opD(f)$, which implies that $\spec M_{\Pcal}=\opD(f)$.
\end {proof}

The equivalences of the following proposition will become very useful in Section \ref{SecBooleanization}.

\begin {Proposition} \label{PropBcal}
Let $M$ be a binoid and $f,g\in M$. The following statements are equivalent:
\begin {ListeTheorem}
\item $\opV(g)\subseteq\opV(f)$.
\item $\opD(f)\subseteq\opD(g)$.
\item $\opFilt(g)\subseteq\opFilt(f)$.
\item $nf=g+x$ for some $n\in\N$ and $x\in M$; that is, $_{_{M}}\langle f\rangle\subseteq\sqrt{_{_{M}}\langle g\rangle}$.
\end {ListeTheorem}
In particular, for a boolean binoid $M$, one has $\opD(f)=\opD(g)$ if and only if $f=g$.
\end {Proposition}
\begin {proof}
Clearly, (1) and (2) are equivalent. So assume that $\opD(f)\subseteq\opD(g)$ for $f,g\in M$. By taking complements of the prime ideals in these basic open sets, we get 
$$F:=\{H\in\Fcal(M)\mid f\in H\}\,\subseteq\,\{H\in\Fcal(M)\mid g\in H\}=:G\komma$$
which implies that $\opFilt(g)=\bigcap_{H\in G}H\subseteq\bigcap_{H\in F}H=\opFilt(f)$. Thus, (3) follows from (2). The implica\-tion $(3)\Rarrow(4)$ was observed in Remark \ref{RemFilterOfF}. $(4)\Rarrow(1)$ is obvious. The if part of the supplement is trivial (and holds for arbitrary binoids). So let $M$ be boolean and $\opD(f)=\opD(g)$ for some $f,g\in M$. By the equivalence of (2) and (4), there are elements $x,y\in M$ with $f=g+x$ and $g=f+y$. It follows 
$$f=f+f=f+g+x=f+g+g+x=f+g+f=f+g$$
and by symmetry $g=f+g$. Thus, $f=g$.
\end {proof}

\begin {Remark} \label{RemFilterPrime}
If $A$ is a subset such that $\opFilt(A)$, which is the smallest filter containing $A$, is a proper filter, then $M\setminus\opFilt(A)=:\Pcal_{A}$ is a prime ideal in $M$ with $A\not\subseteq\Pcal_{A}$. Since every prime ideal is the complement of a (proper) filter by Corollary \ref{CorHomFiltSpec}, $\Pcal_{A}$ is the largest prime ideal with $A\not\subseteq\Pcal_{A}$. Therefore, $\Pcal_{A}\in\opD(A)$ and $\Pcal\subseteq\Pcal_{A}$ for every $\Pcal\in\opD(A)$, or in other words $\Pcal_{A}=\bigcup_{\Pcal\in\opD(A)}\Pcal$.
\end {Remark}

\begin {Lemma} \label {LemOpenSetsProperties}
Let $U\subseteq\spec M$ be an open set, $\Pcal,\Pcal^{\prime}\in\spec M$, and $f\in M$.
\begin {ListeTheorem}
\item If $\Pcal\in U$, then $\Pcal^{\prime}\subseteq\Pcal$ implies $\Pcal^{\prime}\in U$.
\item $M\setminus\opFilt(f)\in U$ if and only if $\opD(f)\subseteq U$.
\item If $\opD(f)\subseteq\bigcup_{i\in I}U_{i}$ is an open cover, then $\opD(f)\subseteq U_{i}$ for one $i\in I$.
\end {ListeTheorem}
\end {Lemma}
\begin {proof}
(1) Since $U$ is open, there is an $A\subseteq M$ such that $U=\opD(A)=\{\Pcal\in\spec M\mid A\not\subseteq\Pcal\}$. Then $A\not\subseteq\Pcal^{\prime}$ if $\Pcal^{\prime}\subseteq\Pcal\in U$. The assertions (2) and (3) are clear if $f$ is nilpotent, since then $\opD(f)=\emptyset$ and $\opV(f)=M$ by Lemma \ref{LemBcalD}(5) and Lemma \ref{LemBcalV}(5). On the other hand, if $f\not\in\nil(M)$, then $\opFilt(f)\not=M$. By Remark \ref {RemFilterPrime}, $\opD(f)=\{\Pcal\in\spec M\mid\Pcal\subseteq\Qcal\}$, where $\Qcal:=M\setminus\opFilt(f)$. This proves (2) and (3).
\end {proof}

\begin {Corollary}
Let $M$ be a binoid. The spaces $\spec M$ and $\specE M$ are quasi-compact and, if $M\not=\zero$, connected. Moreover, the Zariski topology satisfies the separation axiom $\opT_{0}$.
\end {Corollary}
\begin {proof}
The first two properties are immediate consequences of Lemma \ref {LemOpenSetsProperties}(3) using the identifications $\spec M=\opD(0)$ and $\specE M=\opDE(0)$. To show the separation axiom let $\Pcal,\Qcal\in\spec M$ with $\Pcal\not=\Qcal$. Choose $f\in\Pcal$ with $f\not\in\Qcal$. Then $\opD(f)$ is an open neighborhood of $\Qcal$ with $\Pcal\not\in\opD(f)$. This also proves the statement for $\specE M$.
\end {proof}

For a subset $E$ in a topological space $X$, the \gesperrt{closure} \index{closure}$\overline{E}$ of $E$ with respect to the topology on $X$ is the smallest closed subset of $X$ containing $E$. In $\spec M$, there is an easy description of these closures.

\begin {Definition}
Given a subset $E\subseteq\spec M$, we denote the ideal $\bigcap_{\Pcal\in E}\Pcal$ by $\Jscr(E)$\nomenclature[J]{$\Jscr(E)$}{$=\bigcap_{\Pcal\in E}\Pcal$ for $E\subseteq\spec M$}.
\end {Definition}

\begin {Proposition} \label{PropClosedSets}
$\overline{E}=\opV(\Jscr(E))$ for every subset $E\subseteq\spec M$.
\end {Proposition}
\begin {proof}
For the inclusion $\subseteq$, take an arbitrary $\Pcal\in E$. By definition, we have $\Jscr(E)\subseteq\Pcal$, which implies that $\opV(\Pcal)\subseteq\opV(\Jscr(E))$ by Lemma \ref {LemBcalV}(3). Now the inclusion follows since $\Pcal\in\opV(\Pcal)$. For the other inclusion, we need to show that every closed subset $V\subseteq\spec M$ with $E\subseteq V$ contains $\opV(\Jscr(E))$. If $V\subseteq\spec M$ is a closed subset, then $V=\opV(\Ical)$ for some ideal $\Ical\subseteq M$. If, in addition, $E\subseteq V$, then  $\Ical\subseteq\Pcal$ for all $\Pcal\in E$. Hence, $\Ical\subseteq\Jscr(E)$. Now $\opV(\Jscr(E))\subseteq V$ follows from Lemma \ref {LemBcalV}(3).
\end {proof}

Since the intersection of radical ideals is again a radical ideal, $\Jscr(E)$ is a radical ideal for every $E\subseteq\spec M$. Proposition \ref {PropRadical} now translates to:

\begin {Corollary} \label{CorRadicalClosed}
$\Jscr(\opV(\Ical))=\sqrt{\Ical}$ for every ideal $\Ical\subseteq M$. In particular, the inclusion reversing assignments
$$A\longmapsto\Jscr(A)\quad\text{and}\quad\opV(\Ical)\longmapsfrom\Ical$$are inverse bijections between the closed subsets of $\spec M$ and the radical ideals of $M$.
\end {Corollary}
\begin {proof}
The equality is just a restatement of Proposition \ref {PropRadical} and implies $\Ical\mto\opV(\Ical)\mto\Jscr(\opV(\Ical))=\Ical$. Conversely we have $A\mto\Jscr(A)\mto\opV(\Jscr(A))=A$ by Proposition \ref{PropClosedSets}.
\end {proof}

Recall that a topological space $X$ is called \gesperrt {irreducible} \index{irreducible}\index{space!irreducible --}if  $X\not=\emptyset$ and any two nonempty open subsets of $X$ intersect. Equivalently, $X\not=\emptyset$ and $X$ is not the union of two proper closed subsets of $X$. If $X$ is an irreducible topological space with $X=\overline{\{x\}}$, $x\in X$, then $x$ is called a \gesperrt{generic point} \index{generic point}of $X$. A maximal irreducible subset of a topological space $X$ with respect to $\subseteq$ is called an \gesperrt{irreducible component} \index{component!irreducible --}\index{irreducible!-- component}of $X$. To study the irreducible subsets and components of $\spec M$, we need the following results.

\begin {Lemma} \label{LemIrredIFclosure}
Let $X$ be a topological space.
\begin {ListeTheorem}
\item A subset $Y\subseteq X$ is irreducible if and only if its closure $\overline{Y}$ is irreducible.
\item Every irreducible subset of $X$ is contained in an irreducible component of $X$.
\end {ListeTheorem}
\end {Lemma}
\begin {proof}
See \cite[Proposition 3.A.10 and Proposition 3.A.13]{PatilStorch}.
\end {proof}

The first statement of Lemma \ref {LemIrredIFclosure} shows that every irreducible component of $X$ is closed, and by the second, we have that $X$ is the union of its components since every point $\{x\}$, $x\in X$, of a topological space $X$ is irreducible. The description of the irreducible components of $\spec M$ follows from the next result.

\begin {Proposition} \label {PropClosedIrred}
The closed irreducible subsets of $\spec M$ are the sets $\opV(\Pcal)$, $\Pcal\in\spec M$. In particular, every closed irreducible subset of $\spec M$ has a unique generic point.
\end {Proposition}
\begin {proof}
The irreducibility of $\opV(\Pcal)$ follows from Lemma \ref {LemIrredIFclosure} since every point in a topological space is irreducible. Conversely, we have to show that every closed irreducible subset is of this form. If $E$ is such a subset of $\spec M$, then $E=\opV(\Ical)\not=\emptyset$ for a radical ideal $\Ical\not=M$ by Corollary \ref{CorRadicalClosed}. To verify the prime property of $\Ical$ assume that $f+g\in\Ical$ for some $f,g\in M$. If $\Pcal$ is a prime ideal containing $\Ical$, then $\langle f\rangle\cup\Ical\subseteq\Pcal$ or $\langle g\rangle\cup\Ical\subseteq\Pcal$. Thus, $\opV(\Ical)=\opV(\langle f\rangle\cup\Ical)\cup\opV(\langle g\rangle\cup\Ical)$, which implies that $\opV(\Ical)=\opV(\langle f\rangle\cup\Ical)$ or $\opV(\Ical)=\opV(\langle g\rangle\cup\Ical)$ by the irreducibility of $\opV(\Ical)$. Hence, $\sqrt{\Ical}=\sqrt{\langle f\rangle\cup\Ical}$ or $\sqrt{\Ical}=\sqrt{\langle g\rangle\cup\Ical}$ by Corollary \ref{CorRadicalClosed}. Thus, $f\in\Ical$ or $g\in\Ical$. This proves that $\Ical$ is a prime ideal. By Proposition \ref{PropClosedSets}, $\opV(\Pcal)=\opV(\Jscr(\{\Pcal\}))=\overline{\{\Pcal\}}$, which shows that the closed irreducible subset $\opV(\Pcal)$ has a generic point.
\end {proof}

\begin {Corollary} \label{CorIrredCompMinPrime}
The irreducible components of $\spec M$ are given by the sets $\opV(\Pcal)$ with $\Pcal\in\min M$.
\end {Corollary}
\begin {proof}
This is clear by Proposition \ref {PropClosedIrred} and Lemma \ref{LemBcalV}(3).
\end {proof}

By Corollary \ref{CorPrimeIdealHom}, every binoid homomorphism $\varphi:M\rto N$ induces a map $\varphi^{\ast}:\spec N\rto\spec M$ with $\Pcal\mto\varphi^{-1}(\Pcal)$. This significant, albeit elementary result translates to the category $\Top$ of topological spaces.\nomenclature[T]{$\Top$}{category of topological spaces}

\begin {Proposition} \label {PropInduecedspec}
Given a binoid homomorphism $\varphi:M\rto N$, the induced semigroup homomorphism $\varphi^{\ast}:\spec N\rto\spec M$, $\Pcal\mto\varphi^{-1}(\Pcal)$, is continuous, namely 
$$(\varphi^{\ast})^{-1}(\opD(A))=\opD(\varphi(A))\quad\text{and}\quad(\varphi^{\ast})^{-1}(\opV(A))=\opV(\varphi(A))$$
for all $A\subseteq M$. In particular, $\spec:\cBsf\rto\Top$ is a contravariant functor from the category of commutative binoids into the category of topological spaces.
\end {Proposition}
\begin {proof}
To prove that $\varphi^{\ast}$ is continuous, it suffices to consider the basic open sets $\opD(f)$, $f\in M$, of the topology. For such a set, we obtain $(\varphi^{\ast})^{-1}(\opD(f))=\{\Pcal\in\spec N\mid \varphi^{-1}(\Pcal)\in\opD(f)\}$. This proves the first equality because $\varphi^{-1}(\Pcal)\in\opD(f)$ is equivalent to $\varphi(f)\not\in\Pcal$ by the definition of $\opD(f)$. The second identity for the closed subsets follows by taking complements.
\end {proof}

\begin {Corollary} \label {CorInducedEmb}
If $\varphi:M\rto N$ is binoid epimorphism, the semigroup homomorphism $\varphi^{\ast}:\spec N\rto\spec M$ is a continuous embedding on $\im\varphi^{\ast}\subseteq\opV(\ker\varphi)$. If $N=M/\Ical$ for an ideal $\Ical\subseteq M$, then $\spec M/\Ical\cong\opV(\Ical)$ a semigroups and as topological spaces, where $\opV(\Ical)$ carries the induced subspace topology.
\end {Corollary}
\begin {proof}
The continuous embedding $\varphi^{\ast}$ is given by Proposition \ref{PropInduecedspec}.
Since $\ker\varphi\subseteq\varphi^{-1}(\Pcal)=\varphi^{\ast}(\Pcal)$ for all $\Pcal\in\spec N$, we have $\im\varphi^{\ast}\subseteq\opV(\ker\varphi)$. The statement for the special case $N=M/\Ical$ for an ideal  $\Ical\subseteq M$, and $\varphi=\pi:M\rto M/\Ical$ follows from Corollary \ref{CorExtIdealPrime}.
\end {proof}

\begin {Example}
Let $M$ be a binoid. 
\begin{ListeTheorem}
\item The semigroup isomorphism $\spec M\cong\spec M_{\opred}$ given in Corollary \ref {CorSpecMred} can also be deduced from Corollary \ref{CorInducedEmb} since $M_{\opred}=M/\nil(M)$, cf.\ Example \ref{ExCompositionsIdeal}, and $\opV(\nil(M))=\spec M$ by Lemma \ref{LemBcalV}. Furthermore, this is an isomorphism of topological spaces by Corollary \ref{CorInducedEmb}. Similarly, we obtain $\spec M_{\opint}\cong\opV(\nonint(M))$ as semigroups and as topological spaces from the identification  $M_{\opint}=M/\nonint(M)$.
\item Let $N$ be an integral binoid. Every binoid homomorphism $\varphi:M\rto N$ factors through $M/\Qcal$, where $\Qcal$ is the prime ideal $\ker\varphi=\varphi^{-1}(\{\infty_{N}\})$. In particular, $\varphi$ factors through every minimal prime ideal $\Pcal\subseteq\Qcal$, which implies that $\varphi^{\ast}:\spec N\rto\spec M$ factors through the irreducible component $\pi_{\Pcal}^{\ast}(\spec M/\Pcal)=\opV(\Pcal)$ of $\spec M$ since 
$$\xymatrix{
M\ar[r]^{\varphi}\ar[d]_{\pi}&N&\text{ induces }&\spec M&\spec N\pkt\ar[l]_{\,\,\varphi^{\ast}}\ar[dl]^{\bar{\varphi}^{\ast}}\\
 M/\Pcal\ar[ur]_{\bar{\varphi}}&&&\spec M/\Pcal\ar[u]^{\pi^{\ast}}&}$$
\item Let $M\not=\zero$. If $M$ is reduced, there is a binoid embedding $M\rto\prod_{\Pcal\in\min M}M/\Pcal$ by Corollary \ref{CorSubdirectProd}. In other words, two elements of a reduced binoid $M$ coincide if and only if they coincide on every irreducible component of $M$.
\end {ListeTheorem}
\end {Example}

\bigskip

\section {Booleanization} \label{SecBooleanization}
\markright {\ref{SecBooleanization} Booleanization}

In this section, we introduce the booleanization of a binoid $M$ which can be defined for all, not necessarily commutative, binoids and is in either case a universal object. However, for a commutative binoid $M$, the booleanization has an explicit realization in terms of the basic open sets $\opD(f)$, $f\in M$, of the Zariski topology on $\spec M$, which we described at the end of the last section.

For the moment we consider not necessarily commutative binoids but return very soon to the commutative situation.

\begin {Definition}
Let $M$ be an arbitrary (not necessarily commutative) binoid and $\sim_{\bool}$ \nomenclature[ACongruenceBool]{$\sim_{\bool}$}{congruence on a binoid}the congruence generated by $$f\,\sim_{\bool}\,f+f$$
for $f\in M$. The binoid $M/\!\sim_{\bool}=:M_{\bool}$ \nomenclature[MACongruence]{$M_{\bool}$}{booleanization of $M$}together with the canonical projection $\pi_{\bool}:M\rto M_{\bool}$ is called the \gesperrt{booleanization} \index{booleanization}\index{binoid!booleanization of a --}of $M$.
\end {Definition}

If $X$ is a generating set of the (not necessarily commutative) binoid $M$, then every element $f\in M$ can be written as $f=n_{1}x_{1}\pluspkt n_{r}x_{r}$
with $(x_{1},x_{2}\kpkt x_{r})\in X^{r}$, $n_{i}\ge1$, $i\in\{1\kpkt r\}$, $r\in\N$, such that $x_{i}\not=x_{i+1}$ for $i\in\{1\kpkt r-1\}$. Then $\pi_{\bool}(f)=x_{1}\pluspkt x_{r}$. 

The booleanization is a universal object as the following proposition shows.

\begin {Proposition} \label{PropMboolUniversalProp}
Let $M$ be a (not necessarily commutative) binoid. $M_{\bool}$ is a boolean binoid such that $\nil(M)\subseteq\ker\pi_{\bool}$, and whenever $\varphi:M\rto B$ is a binoid homomorphism with $B$ boolean, there exists a unique binoid homomorphism $\bar{\varphi}:M_{\bool}\rto B$ such that the diagram
$$\xymatrix{
M\ar[r]^{\varphi}\ar[d]_{\pi_{\bool}}&B\\
M_{\bool}\ar[ur]_{\bar{\varphi}}&}$$
commutes; that is, $\bar{\varphi}\pi_{\bool}=\varphi$.
\end {Proposition}
\begin {proof}
By definition, $M_{\bool}$ is a boolean binoid such that $f\sim_{\bool} nf$ for all $n\ge1$. In particular, $\nil(M)\subseteq\ker\pi_{\bool}$. The existence of the induced binoid homomorphism follows from Lemma \ref {LemIndCong} since $\sim_{\bool}\,\le\,\sim_{\varphi}$. Indeed, this need only be checked for the generating relations $f\sim_{\bool}f+f$, for which $\varphi(f)=2\varphi(f)=\varphi(f+f)$ obviously holds.
\end {proof}

The congruence $\sim_{\bool}$ can be characterized more explicitly in the commutative situation.

\begin {Lemma} \label{LemBoolCharacterization}
For $f,g\in M$, the following statements are equivalent:
\begin {ListeTheorem}
\item $f\sim_{\bool}g$.
\item There are $n,m\in\N$ and $x,y\in M$ such that $nf=g+x$ and $mg=f+y$.
\item $\opD(f)=\opD(g)$.
\item $\opV(f)=\opV(g)$.
\item $\sqrt{_{_{M}}\!\langle f\rangle}=\sqrt{_{_{M}}\!\langle g\rangle}$.
\item $\opFilt(f)=\opFilt(g)$.
\end {ListeTheorem}
\end {Lemma}
\begin {proof}
By Proposition \ref{PropBcal}, only the equivalence of (1) and (2) need to be shown. To prove $(1)\Rarrow(2)$ one easily checks that the relation $\sim$ on $M$ defined by $f\sim g$, $f,g\in M$, if (2) is satisfied is a congruence with $\sim_{\bool}\,\le\,\sim$. Conversely, assume that $nf=g+x$ and $mg=f+y$ for some $n,m\in\N$ and $x,y\in M$. Similar to the proof (of the supplement) of Proposition \ref{PropBcal}, we get $f\sim_{\bool}f+g$ and $g\sim_{\bool}f+g$, hence $f\sim_{\bool}g$.
\end {proof}

From now on we focus again on commutative binoids. Note that in this case the we the equality $\nil(M)=\ker\pi_{\bool}$ by the characterization of $\sim_{\bool}$ in (2) above.

\begin {Corollary}
Given a boolean binoid $B$, there is a canonical binoid embedding 
$$B\Rto\Pset(\spec B)_{\cap}\quad\text{with}\quad f\lto\opD(f)\pkt$$
\end {Corollary}
\begin {proof}
This is an immediate consequence of Lemma \ref{LemBoolCharacterization}.
\end {proof}

By applying the preceding result to the binoids $\Pset(V)_{\cap}$ and $\Pset(V)_{\cup}$ for a finite set $V$, we obtain the embeddings
$$\Pset(V)_{\cap}\Rto\Pset(\spec\Pset(V)_{\cap})_{\cap}\komma\quad J\lto\opD(J)=\{\Pcal_{I,\cap}\mid I\subseteq J\}$$
and
$$\Pset(V)_{\cup}\Rto\Pset(\spec\Pset(V)_{\cup})_{\cap}\komma\quad J\lto\opD(J)=\{\Pcal_{I,\cup}\mid J\subseteq I\}\komma$$
cf.\ Example \ref{ExSpecPowerset}. Now we give an explicit topological description of the booleanization.

\begin {Corollary} \label {CorBooleanization}
Let $M$ be a binoid. The set $\Bcal(M)=\{\opD(f)\mid f\in M\}$\nomenclature[Boolean]{$\Bcal(M)$}{$=\{\opD(f)\mid f\in M\}$} defines a commutative boolean binoid, namely\nomenclature[Boolean]{$\Bcal(M)_{\cap}$}{$=(\Bcal(M),\cap,\spec M, \emptyset)$}
$$(\Bcal(M),\cap,\spec M,\emptyset)\,\,=:\,\,\Bcal(M)_{\cap}\komma$$
which is isomorphic to $M_{\bool}$.
\end {Corollary}
\begin {proof}
Obviously, $(\Bcal(M),\cap,\spec M,$ $\emptyset)$ is a commutative boolean binoid. Thus, the canonical binoid epimorphism $\beta:M\rto \Bcal(M)$, $f\mto \opD(f)$, factors through $\bar{\beta}:M_{\bool}\rto\Bcal(M)$, $[f]\mto\opD(f)$, by Proposition \ref{PropMboolUniversalProp}. Clearly, $\bar{\beta}$ is surjective and the injectivity follows from Lemma \ref {LemBoolCharacterization}.
\end {proof}

\begin {Example}
By Example \ref{ExpPrimeLatticeNinftyN}, we have $\spec(\N^{n})^{\infty}=\{\langle e_{i}\rangle_{i\in I}\mid I\subseteq\{1\kpkt n\}\}$. Thus,
$$\opD(f)=\{\Pcal\in\spec(\N^{n})^{\infty}\mid f\not\in\Pcal\}=\{\langle e_{i}\rangle_{i\in I}\mid \supp f\cap I=\emptyset\}\komma$$
which yields $\Bcal((\N^{n})^{\infty})\cong\Pset_{n}$.
\end {Example}

Observe that the canonical binoid epimorphism $\beta:M\rto\Bcal(M)$ with $f\mto\opD(f)$, fulfills (like the projection $\pi_{\bool}$)
$$\ker\beta=\{f\in M\mid\opD(f)=\emptyset\}=\{f\in M\mid f\text{ nilpotent}\}=\nil(M)\komma$$
where the equality in the middle is due to Lemma \ref{LemBcalD}(5). In what follows, we will not distinguish between the booleanization $(M_{\bool},\pi_{\bool})$ and its topological realization $(\Bcal(M),\beta)$.

\begin {Corollary} \label{CorSpecBool=Spec}
The canonical binoid epimorphism $\beta:M\rto\Bcal(M)$, $f\mto\opD(f)$, induces a continuous semigroup isomorphism 
$$\beta^{\ast}:\spec\Bcal(M)\stackrel{\sim}{\Rto}\spec M\pkt$$
\end {Corollary}
\begin {proof}
By Corollay \ref {CorInducedEmb}, $\beta^{\ast}$ is a continuous embedding on $\im\beta^{\ast}$. To prove that $\im\beta^{\ast}=\spec M$ let $\Pcal\in\spec M$ and define $\Qcal:=\{\opD(f)\mid f\in\Pcal\}$. The fact that $\opD(f)\cap\opD(g)=\opD(f+g)$ shows that $\Qcal$ is an ideal in $\Bcal(M)$, which is even prime since $\opD(f)\cap\opD(g)\in\Qcal$ implies that $f\in\Pcal$ or $g\in\Pcal$, hence $\opD(f)\in\Qcal$ or $\opD(g)\in\Qcal$. This shows that $\Qcal\in\spec\Bcal(M)$ with $\beta^{\ast}(\Qcal)=\Pcal$.
\end {proof}

\begin {Remark}
The $N$-spectra of a binoid and of its booleanization are usually not isomorphic. For instance, we have $N\minspec\N^{\infty}\cong N$ by Lemma \ref{LemNspec} but 
$$N\minspec(\N^{\infty})_{\bool}\,\,=\,\,N\minspec\free(x)/(2x=x)\,\,\cong\,\,\bool(N)$$ 
by Example \ref{ExpBHomBoolean}(1). In particular, if $N=K$ is a field of characteristic $\not=2$, then $K\minspec(\N^{\infty})_{\bool}=\{\chi_{\{0\}},\chi_{M\opkt}\}\cong\trivial$ but $K\minspec\N^{\infty}\cong K$.
\end {Remark}

\begin {Corollary} \label{CorBoolFunctor}
Every binoid homomorphism $\theta:M\rto N$ induces a unique binoid homomorphism $\Bcal(\theta):\Bcal(M)\rto\Bcal(N)$ such that the diagram
$$\xymatrix{
M\ar[r]^{\theta}\ar[d]_{\beta_{M}}&N\ar[d]^{\beta_{N}}\\
\Bcal(M)\ar[r]^{\Bcal(\theta)}&\Bcal(N)}$$
commutes. In particular, $\Bcal:\cBsf\rto\bBsf$ is a covariant functor from the category of commutative binoids into the category of commutative boolean binoids $\bBsf$.\nomenclature[BB]{$\bBsf$}{category of commutative boolean binoids}
\end {Corollary}
\begin {proof}
By Proposition \ref {PropMboolUniversalProp}, the composition $\beta_{N}\theta$ induces the desired unique binoid homomorphism between the booleanizations.
\end {proof}

Now we relate the bidual of a binoid, which was studied in detail in Section \ref{SecHom}, to its booleanization.

\begin {Corollary} \label {CorBooleanizationFG}
Let $M$ be a finitely generated binoid. The map $\Bcal(M)\rto M\bidual$ with $\opD(f)\mto\delta_{f}$, where $\delta_{f}(\chi)=\chi(f)$ for $\chi\in M\dual$, is a binoid isomorphism. In particular, $M\bidual\cong M$ if and only if $M$ is boolean.
\end {Corollary}
\begin {proof}
The map $\Bcal(M)\rto M\bidual$ of the statement is the binoid homomorphism $\bar{\delta}$ induced by the canonical binoid homomorphism $\delta:M\rto M\bidual$, $f\mto\delta_{f}$, where $\delta_{f}:M\dual\rto\trivial$, $\delta_{f}(\chi)=\chi(f)$, cf.\ Proposition \ref {PropMboolUniversalProp} (see also Remark \ref{RemBidual}). In other words, there is a commutative diagram
$$\xymatrix{
M\ar[r]^{\delta}\ar[d]_{\beta}&M\bidual\\
\Bcal(M)\ar[ur]_{\bar{\delta}}&}$$
with $\bar{\delta}(\opD(f))=\delta_{f}$. To prove the bijectivity of $\bar{\delta}$, we use the identification $M\dual\cong\specE M$ via $\alpha_{\Pcal}\leftrightarrow\Pcal$, $\Pcal\in\specE M$, cf.\ Corollary \ref{CorSpecDualFilt}. For the injectivity let $f,g\in M$ and suppose that $\delta_{f}(\alpha_{\Pcal})=\delta_{g}(\alpha_{\Pcal})$ (i.e.\ $\alpha_{\Pcal}(f)=\alpha_{\Pcal}(g)$) for all $\Pcal\in\specE M$. Thus,  $f\in\Pcal$ if and only if $g\in\Pcal$ for every $\Pcal\in\spec M$, which implies that $\opD(f)=\opD(g)$. For the surjectivity let $\psi\in M\bidual$, $\psi:\specE M\rto\trivial$. Note that $\specE M$ is a finite set by Proposition \ref{PropFgPrimes} because $M$ is a finitely generated binoid. Since $\psi$ is a binoid homomorphism, $\ker\psi$ is a superset-closed subset of $\specE M$. In particular, if $A:=\{\Qcal\in\specE M\mid\Qcal\not\in\ker\psi\}$, then $\bar{\Qcal}:=\bigcup_{\Qcal\in A}\Qcal$ is a  prime ideal that does not lie in $\ker\psi$. So there is an element $f_{\Pcal}\in\Pcal$ for every $\Pcal\in\ker\psi$ such that $f_{\Pcal}\not\in\bar{\Qcal}$. Hence, $\psi=\delta_{f}=\bar{\delta}(\opD(f))$ with $f:=\sum_{\Pcal\in\ker\psi}f_{\Pcal}$.
\end {proof}

\begin {Corollary} \label {CorBidualDual}
If $M$ is a finitely generated binoid, then $(M\dual)\bidual=M\dual$.
\end {Corollary}
\begin {proof}
If $M$ is finitely generated, then so is $M\dual$ by Proposition \ref{PropFgPrimes} and Corollary \ref{CorSpecDualFilt}. Now the statement follows from Corollary \ref{CorBooleanizationFG}
\end {proof}



\bigskip

\section {Dimension} \label{SecDimension}
\markright {\ref{SecDimension} Dimension}

\begin {Definition}
The \gesperrt{dimension} \index{dimension!-- of a binoid}\index{binoid!dimension of a --}of a nonzero binoid $M$, denoted by $\dim M$,\nomenclature[Dimension1]{$\dim M$}{dimension of the binoid $M$} is defined to be the \gesperrt{combinatorial dimension} \index{dimension!combinatorial --}of the space $\spec M$ endowed with the Zariski topology, which is given by the supremum of the lengths of all chains of irreducible closed subsets in $\spec M$. The dimension of the zero binoid is $-1$ by convention.
\end {Definition}

So far, we have seen that the ideal theory of a commutative binoid $M$ resembles that of a ring, and the same is true for the topological theory of their spectra. Therefore, the following results will appear familiar again.

\begin {Lemma}
Let $M$ be a binoid. Then
$$\dim M=\sup\{\ell\mid\Pcal_{0}\subset\Pcal_{1}\subset\cdots\subset\Pcal_{\ell}, \Pcal_{i}\in\spec M\}\pkt$$
\end {Lemma}
\begin {proof}
By Lemma \ref{LemBcalV}(3) and Proposition \ref {PropClosedIrred}, every chain of prime ideals $\Pcal_{0}\subset\cdots\subset\Pcal_{\ell}$ in $M$ gives rise to a chain of irreducible closed subsets $\opV(\Pcal_{\ell})\subset\cdots\subset\opV(\Pcal_{0})$ in $\spec M$ and vice versa. This shows the equality.
\end {proof}

\begin {Example}
\begin {ListeTheorem}
\item []
\item By Example \ref {ExpPrimeLatticeNinftyN}(1)-(3), we have for $n\ge1$,
$$\dim(\N^{n})^{\infty}=n\komma\quad\dim(\N^{\infty})^{n}=2n-1\komma\quad\text{and}\quad\dim\bigcupbidot_{i=1}^{n}\N^{\infty}=1\pkt$$
\item By Example \ref  {ExSpecPowerset}, $\dim\Pset_{n,\cap}=\dim\Pset_{n,\cup}=n-1$.
\end {ListeTheorem}
\end {Example}

\begin {Corollary}
$\dim M=\dim M_{\opred}$ for every nonzero binoid $M$.
\end {Corollary}
\begin {proof}
This follows from Corollary \ref{CorSpecMred}.
\end {proof}

\begin {Corollary}
Let $M$ be a subbinoid of $N$. If $N$ is integral over $M$ with respect to the inclusion $\iota:M\embto N$, then $\dim M=\dim N$. In particular, for an arbitrary subbinoid $M$ of $N$, one has $\dim M=\dim\overline{M}^{N}$.
\end {Corollary}
\begin {proof}
This follows from Corollary \ref{CorSpecIntegralSubbinoid}.
\end {proof}

\begin {Corollary}
If a binoid $M$ admits a generating set with $n$ elements, then $\dim M\le n$.
\end {Corollary}
\begin {proof}
This follows from Proposition \ref{PropFgPrimes}.
\end {proof}

\begin {Proposition}
A binoid $M$ is a binoid group if and only if it is integral and $\dim M=0$
\end {Proposition}
\begin {proof}
The only if part is immediate. On the other hand, if $M$ is integral, then $\zero$ is a prime ideal, and since $\dim M=0$, it is the only prime ideal. Hence, $\zero=M\Uplus=M\setminus M\okreuz$.
\end {proof}

\begin {Definition}
Let $M\not=\zero$. The \gesperrt{height} \index{prime ideal!height of a --}of a prime ideal $\Pcal\in\spec M$ is the supremum of the lenghts of strictly increasing finite chains in $\spec M$ that end with $\Pcal$. In other words,
$$\height\Pcal:=\sup\{\ell\mid\Pcal_{0}\subset\cdots\subset\Pcal_{\ell}=\Pcal, \Pcal_{i}\in\spec M\}\pkt$$
\nomenclature[Height]{$\height\Pcal$}{height of $\Pcal$}The \gesperrt{dimension} \index{prime ideal!dimension of a --}\index{dimension!-- of a prime ideal}of $\Pcal$ is the supremum of the lenghts of strictly increasing finite chains in $\spec M$ that start with $\Pcal$. In other words,
$$\dim\Pcal:=\sup\{\ell\mid\Pcal=\Pcal_{0}\subset\cdots\subset\Pcal_{\ell}, \Pcal_{i}\in\spec M\}\pkt$$
\nomenclature[Dimension2]{$\dim\Pcal$}{dimension of the prime ideal $\Pcal$}For $\dim M=:d<\infty$, let $F_{i}$ denote the number of all prime ideals $\Pcal\in\spec M$ of dimension $i$, $i\in\N$. Then $F_{k}=0$ for all $k>d$ and the $(d+1)$-tuple 
$$F(M):=(F_{0}(M),F_{1}(M)\kpkt F_{d}(M))$$\nomenclature[f3]{$F_{i}(M)$}{$=\#\{\Pcal\in\spec M\mid\dim\Pcal=i\}$}\nomenclature[f4]{$F(M)$}{$=(F_{0}(M),F_{1}(M)\kpkt F_{d}(M))$, $F\mina$vector of $M$}is called the \gesperrt{$F\mina$vector} \index{binoid!F@$F\mina$vector of a --}\index{F@$F\mina$vector}of $M$.
\end {Definition}

Minimal primes have height $0$ and $\height M\Uplus=\dim M$, which implies $F_{d}\le\#\min M$ and $F_{0}(M)=1$.

\begin {Lemma}
Let $\Pcal$ be a prime ideal of the binoid $M$. Then
$$\height\Pcal=\dim M_{\Pcal}\quad\text{and}\quad\dim\Pcal=\dim M/\Pcal\pkt$$
\end {Lemma}
\begin {proof}
This follows from Corollary \ref {CorIndSpecLoc} and Corollary \ref{CorExtIdealPrime}.
\end {proof}

\begin {Remark}
As pointed out by Anderson and Johnsen in \cite{AndersonIT}, Krull's principal ideal theorem, \index{Krull's principal ideal theorem}which states that in a noetherian ring $R$ every prime ideal $\pideal$ that is minimal over the ideal $(r_{1}\kpkt r_{k})\subsetneq R$ has $\height\pideal\le k$ (\cite[Theorem 10.2]{Eisenbud}), need not be true for binoids that satisfy the ascending chain condition (a.c.c.) on ideals, or equivalently, in which all ideals are finitely generated. The binoid 
$$M=\free(x,y)/(x+y=2x)$$
is finitely generated, hence fulfills the a.c.c.\ on ideals (and on congruences, cf.\ Remark \ref{RemNoetherian=FG}). The prime ideals of $M$ are given by
$$\zero\subsetneq\langle x\rangle\subsetneq\langle x,y\rangle=M\Uplus\komma$$
which shows that $M\Uplus$ is minimal over $\langle y\rangle$, but $\height M\Uplus=2>1$. However, in \cite[Theorem 4.4]{AndersonIT} it is shown that Krull's principal ideal theorem is true for binoids that satisfy the a.c.c.\ on ideals and the following weaker cancellation property 
$$a+b=a+c\not=\infty\,\,\Rightarrow\,\, c=u+b\text{ for some }u\in M\okreuz\pkt$$

For cancellative\footnote{\, Cancellative in the sense of monoid theory, which means  $a+b=a+c$ implies $b=c$.} monoids that satisfy the a.c.c.\ on ideals, Krull's principal ideal theorem has been proved by Kobsa, cf.\ \cite[Satz 13.7]{Kobsa}.
\end {Remark}

\bigskip

\section {$K\mina$points} \label{SecKpoints}
\markright {\ref{SecKpoints} $K\mina$points}

\begin {Convention}
In this section, $K$ always denotes a field.
\end {Convention}

As the title of this section indicates, we are going to pursue a topological approach to binoids via their $K\mina$spectra, which were introduced in Section \ref{SecNspectra} (for arbitrary binoids $K$).\footnote{\, Note that arbitrary binoids are written additively, but a field $K$ is a binoid with respect to the multiplication.} For this, we tacitly assume basic knowledge of the following geometric objects in commutative algebra: the \gesperrt{spectrum} of a ring $R$,\index{spectrum!-- of a ring}\nomenclature[Spec2]{$\Spec R$}{set of all prime ideals in the ring $R$ (spectrum of $R$)}
$$\Spec R:=\{\pideal\subseteq R\mid\pideal\text{ prime (ring) ideal in }R\}\komma$$
which is a topological space equipped with the Zariski topology, and the $K\mina$\gesperrt{spectrum} of a commutative $K\mina$algebra $A$, 
$$K\minSpec A\,=\,\Hom_{K\minus\opalg}(A,K)\komma$$
\index{K@$K\mina$spectrum}\index{spectrum!K@$K\mina$--}which is a topological subspace of $\Spec A$ via
$$K\minSpec A\Rto\Spec A$$
with $\varphi\mto\ker\varphi$. For a detailed treatment of the topological spaces $\Spec A$ and $K\minSpec A$, we refer to \cite[Chapter 2 and 3]{PatilStorch}.

\medskip

By Proposition \ref{PropUnivPropBinoidA}, we have the following isomorphism
$$K\minspec M\,\,\cong\,\, K\minspec K[M]\komma$$
where a $K\mina$point\index{point!$K\mina$--|SEE{$N\mina$point}}\index{K@$K\mina$point|SEE{$N\mina$point}} $\varphi:M\rto K$ of $M$ corresponds one-to-one to the $K\mina$algebra homomorphism $\tilde{\varphi}:K[M]\rto K$ with
$$\tilde{\varphi}(F)=\sum_{a\in M}r_{a}\varphi(a)\in K$$
for $F=\sum_{a\in M}r_{a}T^{a}\in K[M]$. In particular, a characteristic point $\alpha_{\Pcal}$, $\Pcal\in\spec M$, corresponds to
$$\tilde{\alpha}_{\Pcal}:\sum_{a\in M}r_{a}T^{a}\lto\sum_{a\in M\setminus\Pcal}r_{a}\pkt$$
The above identification allows for regarding $K\minspec M$ as a topological space\index{Zariski topology!-- on $K\minSpec K[M]$}\index{topology!Zariski --}, where the closed sets are given by the affine algebraic sets\nomenclature[V]{$\opV_{K}(-)$}{closed set in the Zariski topology on $K\minspec M$}
$$\opV_{K}(F)\,\,=\,\,\{\varphi\in K\minspec M\mid\tilde{\varphi}(F)=0\}$$
for $F\in K[M]$ and
$$\opV_{K}(\{F_{i}\}_{i\in I}):=\bigcap_{i\in I}\opV_{K}(F_{i})$$
for a family $(F_{i})_{ i\in I}$ of elements in $K[M]$. The open subsets are given by the complements 
$$\opD_{K}(\{F_{i}\}_{i\in I}):=K\minspec K[M]\setminus\opV_{K}(\{F_{i}\}_{i\in I})\komma$$
for $F_{i}\in K[M]$, $i\in I$.\nomenclature[D]{$\opD_{K}(-)$}{open set in the Zariski topology on  $K\minspec M$}

\begin {Lemma} \label{LemVKProperties}
Let $M$ be a binoid.
\begin {ListeTheorem}
\item $\opV_{K}(\bigcup_{i\in I}A_{i})=\bigcap_{i\in I}\opV_{K}(A_{i})$ for a family $A_{i}\subseteq K[M]$, $i\in I$, of subsets.
\item $\opV_{K}(\sum_{j\in J}\aideal_{j})=\bigcap_{j\in J}\opV_{K}(\aideal_{j})$ for a family $\aideal_{j}\subseteq K[M]$, $j\in J$, of ideals.
\item $\opV_{K}(FG)=\opV_{K}(F)\cup\opV_{K}(G)$ for $F,G\in K[M]$.
\item $\opV_{K}((F,G))=\opV_{K}(F)\cap\opV_{K}(G)$, where $(F,G)$ is the ideal in $K[M]$ generated by $F,G\in K[M]$.
\item $\opV_{K}(\aideal\bideal)=\opV_{K}(\aideal\cap\bideal)=\opV_{K}(\aideal)\cup\opV_{K}(\bideal)$ for ideals $\aideal,\bideal\subseteq K[M]$.
\item $\opV_{K}(A)\subseteq\opV_{K}(B)$ for subsets $A,B\subseteq K[M]$ with $B\subseteq A$.
\item $\opV_{K}(1)=\emptyset$ and $\opV_{K}(0)=K\minspec K[M]$.
\end {ListeTheorem}
\end {Lemma}
\begin {proof}
All assertions are easily verified or follow from the corresponding statements on $\Spec K[M]$.
\end {proof}

\begin {Proposition}\label {PropIndHomeoKspec}
For a binoid homomorphism $\varphi:M\rto N$, the induced semigroup homomorphism $\varphi^{\ast}:K\minspec N\rto K\minspec M$, $\psi\mto\psi\varphi$, is continuous, namely, if $\aideal$ is an ideal in $K[M]$, then
$$(\varphi^{\ast})^{-1}(\opD_{K}(\aideal))\,=\,\opD_{K}(\aideal K[N])\quad\text{and}\quad(\varphi^{\ast})^{-1}(\opV_{K}(\aideal))\,=\,\opV_{K}(\aideal K[N])\komma$$
where $\aideal K[N]$ is the extended ideal $\phi(\aideal)K[N]$ under the $K\mina$algebra homomorphism $\phi:K[M]\rto K[N]$ with $r_{a}T^{a}\mto r_{a}T^{\varphi(a)}$, $a\in M\opkt$.
\end {Proposition}
\begin {proof}
By taking complements, it suffices to prove the statement for the closed sets. Moreover, we only need to consider ideals of the form $\aideal=(F)$ for one $F\in K[M]$. For these ideals, we have $\varphi^{\ast}(\opV_{K}(F))=\{\psi\in K\minspec N\mid\tilde{\psi}(\varphi(F))=0\}$, where $\tilde{\psi}(G):=\sum_{a\in M\opkt}r_{a}\psi(a)$ if $G=\sum_{a\in M\opkt}r_{a}T^{a}$, which implies the statement.
\end {proof}

\begin {Proposition} \label{PropKspecIdealLocalization}
Let $M$ be a binoid.
\begin {ListeTheorem}
\item If $\Ical$ is an ideal in $M$, then
$$K\minspec(M/\Ical)\,\,\cong\,\,\opV_{K}(K[\Ical])\pkt$$
\item For $f\in M$, we have
$$K\minspec M_{f}\,\,\cong\,\,\opD_{K}(T^{f})\,\,\cong\,\,\{\varphi\in K\minspec M\mid \varphi(f)\not=0\}\pkt$$
\end {ListeTheorem}
\end {Proposition}
\begin {proof}
Both statements follow from the general fact that for a $K\mina$algebra homomorphism $\phi:A\rto A^{\prime}$ the induced map $\phi^{\ast}:K\minSpec A^{\prime}\rto K\minSpec A$, $\psi\mto\phi\psi$, is continuous so that for an ideal $\aideal$ in $A$ one has $(\phi^{\ast})^{-1}(\opV_{K}(\aideal))=\opV_{K}(\phi(\aideal)A^{\prime})$, where $\phi(\aideal)A^{\prime}$ is the extended ideal of $\aideal$ in $A^{\prime}$ via $\phi$, cf.\ \cite[Proposition 2.B.15]{PatilStorch}. Applying this to the $K\mina$algebra homomorphisms
$$K[M]\Rto K[M/\Ical]\cong K[M]/K[\Ical]\quad\text{and}\quad K[M]\Rto K[M_{f}]\cong K[M]_{\{T^{nf}\mid n\in\N\}}\komma$$
cf.\ Corollary \ref{CorPropBinoidA}(3)\&(4), proves the proposition.
\end {proof}

\begin {Remark} \label{RemAffineEmbeddingKspec}
If $x_{1}\kpkt x_{n}$ generate the binoid $M$, the canonical binoid epimorphism $(\N^{n})^{\infty}\rto M$, $e_{i}\mto x_{i}$ with $i\in\{1\kpkt n\}$, induces a surjective $K\mina$algebra homomorphism 
$$K[X_{1}\kpkt X_{n}]\Rto K[M]$$
with $X_{i}\mto T^{x_{i}}$, $i\in\{1\kpkt n\}$. Hence, there is an embedding 
$$K\minspec M\Rto\A^{n}(K)$$
with $\varphi\mto(\varphi(x_{1})\kpkt\varphi(x_{n}))$, which shows that the topology on $K\minspec M$ can also be given as the subspace topology of $\A^{n}(K)$ (this is also true for the natural topology if $K=\R$ or $\C$). In particular,
$$K\minspec\free_{n}\,\,\cong\,\,\A^{n}(K)\pkt$$

Under the above embedding $K\minspec M\embto\A^{n}(K)$, the subset $\opV_{K}((F_{i})_{i\in I})\subseteq K\minspec M$, $F_{i}\in K[M]$, $i\in I$, can be identified with
$$\{a\in\A^{n}(K)\mid F_{i}(a)=0\text{ for all }i\in I\}\komma$$
and the characteristic points $\alpha_{\Pcal}:M\rto\{1,0\}\subseteq(K,\cdot,1,0)$ for $\Pcal\in\spec M$, which are independent of $K$, are given by the $0\,$-$1\mina$points 
$$(\alpha_{\Pcal}(x_{1})\kpkt\alpha_{\Pcal}(x_{n}))\in\A^{n}(K)\pkt$$ 
For instance, if $M$ is positive, then $\alpha_{M\Uplus}$ corresponds to $(0\kpkt 0)\in\A^{n}(K)$. Moreover, one has $\Pcal\subseteq\Pcal^{\prime}$ for two prime ideals $\Pcal$ and $\Pcal^{\prime}$ if and only if $\alpha_{\Pcal}(x_{i})\ge\alpha_{\Pcal^{\prime}}(x_{i})$ for all $i\in\{1\kpkt n\}$.
\end {Remark}

\begin {Example} \label{ExpKpoints1}
The pictures below visualize the $\R\mina$spectra of three different commutative binoids in $\A^{2}(\R)$. The marked points are the characteristic points ($\spec M$ as a set) sitting inside the $\R\mina$spectra.
\begin {center}
\begin {pspicture} (-1.5,-2.5)(2.5,2)
\qdisk (0.6,0.6){1.75pt}
\psline [linewidth=0.5 pt, linestyle=dotted] (0,-1.5)(0,1.5)
\psline [linewidth=0.5 pt, linestyle=dotted] (-1.5,0)(1.5,0)
\pcarc [arcangleA=350, arcangleB=335, linewidth=.5pt] (0.2,1.5)(0.6,0.6)
\pcarc [arcangleA=335, arcangleB=350, linewidth=.5pt] (0.6,0.6)(1.5,0.2)
\pcarc [arcangleA=10, arcangleB=25, linewidth=.5pt] (-1.5,-0.2)(-0.6,-0.6)
\pcarc [arcangleA=25, arcangleB=10, linewidth=.5pt]  (-0.6,-0.6)(-0.2,-1.5)
\uput [0] (-2.15,-2.2){\small{$\R\minspec\free(x,y)/(x+y=0)$}}
\end {pspicture}
\quad\quad\quad
\begin {pspicture} (-1.5,-2.5)(2.5,2)
\qdisk (0,0){1.75pt}\qdisk (0.6,0.6){1.75pt}
\psline [linewidth=0.5 pt, linestyle=dotted] (0,-1.5)(0,1.5)
\psline [linewidth=0.5 pt] (-1.5,0)(1.5,0)
\pcarc [arcangleA=60, arcangleB=30, linewidth=.5pt] (0,0)(0.6,0.6)
\pcarc [arcangleA=5, arcangleB=5, linewidth=.5pt] (0.6,0.6)(1.5,0.75)
\pcarc [arcangleA=300, arcangleB=330, linewidth=.5pt] (0,0)(0.6,-0.6)
\pcarc [arcangleA=355, arcangleB=355, linewidth=.5pt] (0.6,-0.6)(1.5,-0.75)
\uput [0] (-2.2,-2.2){\small{$\R\minspec\free(x,y)/(x+y=3y)$}}
\end {pspicture}
\quad\quad\quad
\begin {pspicture} (-1.5,-2.5)(2.5,2)
\qdisk (0,0){1.75pt}\qdisk (0,0.6){1.75pt}\qdisk (0.6,0){1.75pt}\qdisk (0.6,0.6){1.75pt}
\psline [linewidth=0.5 pt] (0,-1.5)(0,1.5)
\psline [linewidth=0.5 pt] (-1.5,0)(1.5,0)
\pcarc [arcangleA=350, arcangleB=335, linewidth=.5pt] (0.2,1.5)(0.6,0.6)
\pcarc [arcangleA=335, arcangleB=350, linewidth=.5pt] (0.6,0.6)(1.5,0.2)
\pcarc [arcangleA=10, arcangleB=20, linewidth=.5pt] (0.2,-1.5)(0.6,-0.6)
\pcarc [arcangleA=20, arcangleB=10, linewidth=.5pt] (0.6,-0.6)(1.5,-0.2)
\pcarc [arcangleA=10, arcangleB=20, linewidth=.5pt] (-1.5,-0.2)(-0.6,-0.6)
\pcarc [arcangleA=20, arcangleB=10, linewidth=.5pt]  (-0.6,-0.6)(-0.2,-1.5)
\pcarc [arcangleA=350, arcangleB=335, linewidth=.5pt] (-1.5,0.2)(-0.6,0.6)
\pcarc [arcangleA=335, arcangleB=350, linewidth=.5pt] (-0.6,0.6)(-0.2,1.5)
\uput [0] (-2.75,-2.2){\small{$\R\minspec\free(x,y)/(2(x+y)=4(x+y))$}}
\end {pspicture}
\end {center}
\end{Example}

Note that the $\C\mina$spectrum of the binoid $\free(x,y)/(x+y=0)$ is irreducible in the Zariski and in the complex topology, whereas the $\R\mina$spectrum is not even connected in the real topology as shown in the example above. This is a good example for the following result. 

\begin {Lemma} \label{LemFgTfRalgclosedIrred}
Let $M$ be a torsion-free regular finitely generated binoid and $K$ algebraically closed. Then $K\minspec M$ is irreducible.
\end {Lemma}
\begin{proof}
By Theorem \ref{ThBinoidADomain}, $K[M]$ is a domain, so $(0)\in\Spec K[M]$ is the only minimal prime ideal. Hence, $\Spec K[M]$ is irreducible (this is true for all fields) by a result on $\Spec K[M]$ analogous to Corollary \ref{CorIrredCompMinPrime}. See for instance \cite[Corollary 3.A.14]{PatilStorch}. Since $K$ is algebraically closed, $K\minspec M\cong K\minSpec K[M]$ coincides with the set of all maximal ideals of $K[M]$, denoted by $\Spm K[M]$\nomenclature[Spec2]{$\Spm R$}{set of all maximal ideals in the ring $R$ (maximal spectrum of $R$)}, cf.\ \cite[Theorem 2.A.2]{PatilStorch}. This implies that $\Spm K[M]$ is a dense open subset of $\Spec K[M]$, cf.\ \cite[Theorem 2.B.12 and Exercise 3.A.20]{PatilStorch}. Therefore, $\Spm K[M]=K\minspec M$ is irreducible as well.
\end{proof}

By Remark \ref{RemMonoidDecomposition}, the semigroup $N\minspec M$ decomposes into subsemigroups that are monoids. This decomposition transfers to the topological situation (if $N$ is a field).

\begin {Proposition}\label{PropUnionCanComp}
Let $M$ be a binoid. Then
$$K\minspec M\,\,=\bigcup_{\Pcal\in\spec M}K\minspec(M/\Pcal)_{\opcan}\komma$$
where $K\minspec(M/\Pcal)_{\opcan}$ are closed subsets of $K\minspec M$. In particular, if $M$ is cancellative, then
$$K\minspec M\,\,=\bigcup_{\Pcal\in\min M}\opV_{K}(K[\Pcal])\pkt$$
\end {Proposition}
\begin {proof}
On the one hand, we have for every $\Pcal\in\spec M$ a closed embedding
$$K\minspec(M/\Pcal)_{\opcan}\Rto K\minspec M$$
induced by the surjection $\pi_{\Pcal}:M\rto(M/\Pcal)_{\opcan}$, cf.\ Proposition \ref{PropIndHomNspec}. On the other hand, every $K\mina$point $\varphi:M\rto K$ of $M$ factors through $(M/\ker\varphi)_{\opcan}$ by Proposition \ref {PropCancellation}(2), which proves 
$$K\minspec M\,\,=\bigcup_{\Pcal\in\spec M}K\minspec(M/\Pcal)_{\opcan}\pkt$$
The supplement follows from Proposition \ref{PropKspecIdealLocalization}(1) and the fact that $\Pcal\subseteq\Qcal$ for $\Pcal,\Qcal\in\spec M$ implies that $K[\Pcal]\subseteq K[\Qcal]$, and hence $\opV_{K}(K[\Qcal])\subseteq\opV_{K}(K[\Pcal])$ by Lemma \ref {LemVKProperties}(6).
\end {proof}

\begin {Definition}
Let $M$ be a binoid. We call $K\minspec(M/\Pcal)_{\opcan}$, $\Pcal\in\spec M$, the \gesperrt{cancellative $K\mina$parts}, \index{cancellative!-- $K\mina$part}and those that are maximal with respect to set inclusion are called the \gesperrt{cancellative $K\mina$components}\index{component!cancellative $K\mina$--}\index{cancellative!-- $K\mina$component}.
\end {Definition}

Cancellative components need not be irreducible though the name suggests so, cf.\  Example \ref{ExpKpoints2}.

\begin {Definition}
Let $M$ be a binoid, $\Pcal\in\spec M$, and $\varphi_{\Pcal}:=\iota\pi$ the canonical binoid homomorphism 
$$M\stackrel{\pi}{\Rto} M/\Pcal\stackrel{\iota}{\Rto}\diff(M/\Pcal)\komma$$
where $\pi$ is the canonical projection onto the integral binoid $M/\Pcal$ and $\iota$ the canonical binoid homomorphism to the difference group $(M/\Pcal)_{(M/\Pcal)\opkt}\not=\zero$. With this notation, two elements $f,g\in M$ are called \gesperrt{functionally equivalent} \index{functionally equivalent}\index{element!functionally equivalent --s}if $\varphi_{\Pcal}(f)=\varphi_{\Pcal}(g)$ for every $\Pcal\in\spec M$. The relation $\sim_{\opfe}\!$\nomenclature[ACongruenceFE]{$\sim_{\opfe}$}{congruence on a binoid} on $M$ given by $f\sim_{\opfe}\!g$ if $f$ and $g$ are functionally equivalent defines a congruence on $M$.
\end {Definition}

The name stems from the characterization of functionally equivalent elements in Proposition \ref{PropFuncEquiv}(3) below.

\begin {Example}
Nilpotent elements are functionally equivalent to $\infty$ because $nf=\infty$ implies that $n\varphi_{\Pcal}(f)=\varphi_{\Pcal}(nf)=\infty$ for every $\Pcal\in\spec M$. Hence, $\varphi_{\Pcal}(f)=\infty$ since $\diff(M/\Pcal)$ is a binoid group. In particular, the class $[\infty]$ in $M/\sim_{\opfe}$ consists of all nilpotent elements by Corollary \ref{CorMinPrimeIntersection}.
\end {Example}

\begin {Proposition} \label{PropFuncEquiv}
Let $M$ be a finitely generated binoid and $f,g\in M$. The following statements are equivalent.
\begin {ListeTheorem}
\item $f$ and $g$ are functionally equivalent.
\item For every commutative binoid group $G$ and every $G\mina$point $\varphi:M\rto G$, one has $\varphi(f)=\varphi(g)$.
\item For every field $K$ and every $K\mina$point $\varphi:M\rto K$, one has $\varphi(f)=\varphi(g)$.
\item For every (some) algebraically closed field $K$ with $\opchar K=0$ and every $K\mina$point $\varphi:M\rto K$, one has $\varphi(f)=\varphi(g)$.
\end {ListeTheorem}
\end {Proposition}
\begin {proof}
The implications $(2)\Rightarrow(3)\Rightarrow(4)$ are trivial. For $(1)\Rightarrow(2)$ assume that $f$ and $g$ are functionally equivalent and that $\varphi:M\rto G$ is a binoid homomorphism to a commutative binoid group $G$. By Lemma \ref{LemFactoriazationKer}, $\varphi$ factors through $M/\ker\varphi$, where $\ker\varphi=:\Pcal$ is a prime ideal because $G$ is a binoid group. In particular, we have the following diagram
$$\xymatrix{
M\ar@/^/[rrd]^{\varphi}\ar[rd]^{\pi}\ar@/_/[ddr]_{\varphi_{\Pcal}}&&\\
&M/\Pcal\ar[r]^{\!\!\tilde{\varphi}}\ar[d]_{\iota}&G\komma\\
&\diff(M/\Pcal)&}$$
where $\ker\tilde{\varphi}=\zero$. By Corollary \ref{CorUnivPropDiff}, there is a unique binoid homomorphism $\phi:\diff(M/\Pcal)\rto G$ with $\phi\iota=\tilde{\varphi}$, hence
$$\varphi(f)\,=\,\tilde{\varphi}(\pi(f))\,=\,\phi(\iota\pi(f))\,=\,\phi(\varphi_{\Pcal}(f))\,=\,\phi(\varphi_{\Pcal}(g))\,=\,\phi(\iota\pi(g))\,=\,\tilde{\varphi}(\pi(g))\,=\,\varphi(g)\pkt$$
For $(4)\Rightarrow(1)$ take an algebraic closed field $K$ of characteristic $0$. Assuming that $f$ and $g$ are not functionally equivalent, we find a prime ideal $\Pcal\in\spec M$ with $\varphi_{\Pcal}(f)\not=\varphi_{\Pcal}(g)$ in 
$$\diff(M/\Pcal)\,\,\cong\,\,\Z^{r}\times\Z/p_{1}^{r_{1}}\Z\timespkt\Z/p_{s}^{r_{s}}\Z\komma$$ 
where $p_{i}$ is a prime number and $r_{i}\ge 1$, $i\in\{1\kpkt s\}$. Thus, for at least one $k\in\{1\kpkt r+s\}$ the $k$th entries of $\varphi_{\Pcal}(f)$ and $\varphi_{\Pcal}(g)$ in $\Z^{r}\times\Z/p_{1}^{r_{1}}\Z\timespkt\Z/p_{s}^{r_{s}}\Z$, denoted by $(\varphi_{\Pcal})_{k}(f) $ and $(\varphi_{\Pcal})_{k}(g)$, do not coincide. Since $K$ is algebraically closed of characteristic $0$, there are injective group homomorphisms 
$$\iota_{0}:\Z\Rto K\okreuz\quad\text{and}\quad\iota_{i}:\Z/p_{i}^{r_{i}}\Z\Rto K\okreuz\komma$$
$i\in\{1\kpkt s\}$. Hence, $\iota_{k}\circ(\varphi_{\Pcal})_{k}$ is a $K\mina$point which separates $f$ and $g$.
\end {proof}

The condition on the characteristic of the field $K$ in (4) of the preceding result is necessary because in $G=(\Z/p\Z)^{\infty}$, $p$ prime, any two different elements are not functionally equivalent since $\id_{G}$ is a $G\mina$point, but the only $K\mina$point of $G$ into an (algebraically closed) field $K$ of characteristic $p$ is $\chi_{G\opkt}$.

\begin {Remark} \label{RemAsymEquiv=>FuncEquiv}
In monoid theory, a monoid homomorphism $\varphi:M\rto K$ is called a \gesperrt{$K\mina$character} \index{character!$K\mina$--}of $M$. In case $K=\C$ or $K=S^{1}=\{z\in\C\mid |z|=1\}\subseteq\C$, such a homomorphism is simply called a \gesperrt{character}, \index{character}\index{monoid!c@($K$-) character of a --}\index{K@$K\mina$character|SEE{monoid}}and the set of all characters is a well-studied object of interest, see for instance \cite[Chapter IV.2]{GrilletCS} or \cite[Chapter 5.5]{CliffordPreston}. A fundamental result, cf.\ \cite[Theorem 5.58]{CliffordPreston}, says that the characters of a commutative group $G$ separates its elements; that is, for any two elements $a,b\in G$ there is a character $\varphi$ such that $\varphi(a)\not=\varphi(b)$. Similarly, cf.\ \cite[Theorem 5.59]{CliffordPreston}, the characters of a commutative monoid $M$ separates its elements if and only if $M$ is \gesperrt{separative}; \index{monoid!separative --}that is, $2a=a+b=2b$ for $a,b\in M$ implies $a=b$. By \cite[Theorem 9.13]{Gilmer}, a commutative monoid $M$ is separative if and only if $M$ is free of asymptotic torsion. In this context, the preceding proposition suits the general monoid theory since asymptotically equivalent elements of a binoid are functionally equivalent (i.e.\ $\sim_{\opfe}\,\,\le\,\,\sim_{\opae}$). Indeed, $f\sim_{\opae}g$ is equivalent to $nf=ng$ and $(n+1)f=(n+1)g$ for some $n\ge1$ by Lemma \ref{LemAsymEquiv}. If $\varphi:M\rto K$ is a $K\mina$point,
then either $\varphi(f)=\varphi(g)=0$ or $\varphi(f),\varphi(g)\in K\okreuz$ with
$$\varphi(f)\,\,=\,\,\frac{\varphi(f)^{n+1}}{\varphi(f)^{n}}\,\,=\,\,\frac{\varphi(g)^{n+1}}{\varphi(g)^{n}}\,\,=\,\,\varphi(g)\pkt$$
Hence, $f\sim_{\opfe}g$ by Proposition \ref{PropFuncEquiv}.
\end {Remark}

We close this section with another characterization of functionally equivalent elements.

\begin {Lemma}
Two elements $f,g\in M$ are functionally equivalent if and only if $X^{f}-X^{g}$ is nilpotent in $K[M]$ for every field $K$.
\end {Lemma}
\begin {proof}
Let $F:=X^{f}-X^{g}$ be nilpotent. If $\varphi:M\rto K$ is a $K\mina$point, there is by Proposition \ref{PropUnivPropBinoidA} a unique $K\mina$algebra homomorphism $\Phi:K[M]\rto K$ such that $\varphi=\Phi\iota$, where $\iota:M\rto K[M]$ is the canonical binoid homomorphism $a\mto X^{a}$, $a\in M$. Since $F$ is nilpotent and $K$ a field, $0=\Phi(F)=\Phi(X^{f})-\Phi(X^{g})$, which gives $\varphi(f)=\Phi(X^{f})=\Phi(X^{g})=\varphi(g)$. Conversely, if $F$ is not nilpotent in $K[M]$ for some field $K$, then $\{F^{n}\mid n\ge0\}$ is a multiplicative system in $K[M]$ such that $K[M]_{F}\not=0$. In particular, $\Spec K[M]_{F}\not=\emptyset$. Hence, there is a ring homomorphism $\Phi:K[M]_{F}\rto Q(K[M]_{F}/\pideal)=L$ into the quotient field of $K[M]_{F}/\pideal$, $\pideal\in\Spec K[M]_{F}$, with $\Phi(F)\not=0$. This is equivalent to $\Phi(X^{f})\not=\Phi(X^{g})$, which implies that
$$\varphi:M\stackrel{\iota}{\Rto}K[M]\stackrel{\iota_{F}}{\Rto} K[M]_{F}\stackrel{\Phi}{\Rto}L$$
is an $L\mina$point of $M$ with $\varphi(f)\not=\varphi(g)$.
\end {proof}

Any two different elements $f,g\in (\Z/p\Z)^{\infty}$, $p$ prime, are not functionally equivalent as remarked after Proposition \ref{PropFuncEquiv}. However, $X-1=X^{1}-X^{0}$ is nilpotent in $K[(\Z/p\Z)^{\infty}]\cong K[X]/(X^{p}-1)$ for every field $K$ with characteristic $p$ because $(X-1)^{p}=X^{p}-1$. This shows that $X^{f}-X^{g}$ being nilpotent needs to be checked for every field. In particular, the ring $K[(\Z/p\Z)^{\infty}$ is not reduced if $\opchar K=p$, whereas $(\Z/p\Z)^{\infty}$ is a reduced binoid. 

As another example consider the (reduced) binoid
$$M:=\free(a,b)/(2a=a+b=2b)\pkt$$
Since 
$$3a\,=\,a+2a\,=\,a+2b\,=\,(a+b)+b\,=\,2b+b\,=\,3b$$
the generators $a$ and $b$ are asymptotically equivalent, hence functionally equivalent by Remark \ref{RemAsymEquiv=>FuncEquiv}. Thus, $X^{a}-X^{b}$ is a nilpotent element  in $K[M]$  for every field $K$ by the preceding lemma.

\bigskip

\section {Connectedness properties of $K\mina$spectra} \label{SecConnectKspectra}
\markright {\ref{SecConnectKspectra} Connectedness properties of $K\mina$spectra}

\begin {Convention}
In this section, $K$ always denotes a field.
\end {Convention}

Recall that a nonempty topological space is said to be \gesperrt{connected} \index{connected}\index{space!connected --}if it is not the union of two disjoint nonempty open (resp.\ closed) sets, otherwise the space is said to be \gesperrt{disconnected}\index{disconnected}\index{space!disconnected --}.

We want to know which assumptions on a binoid $M$ ensure connectedness of $K\minspec M$. Of course, every idempotent $e$ in $M$ yields an idempotent in $K[M]$, namely $X^{e}$. Then, $K\minspec M$ is not connected since it is the disjoint union of the closed sets $\opV_{K}(X^{e})=\{\varphi\in K\minspec M\mid \varphi(e)=0\}$ and $\opV_{K}(X^{e}-1)=\{\varphi\in K\minspec M\mid \varphi(e)=1\}$, which are not empty since $K\minspec M/\langle e\rangle$ and $K\minspec M_{e}$ are not empty (this follows from the fact that both binoids, $M/\langle e\rangle$ and $M_{e}$, are not empty, and hence $\emptyset\not=\spec\subseteq K\minspec$). On the other hand, idempotents in $K[M]$ do not need to come from idempotents in $M$ as the subsequent example shows. 

At the end of this section we will show that in case of hypersurfaces, the non-existence of non-trivial combinatorial idempotents is equivalent to connectedness in the torsion-free situation, cf.\ Corollary \ref{CorIdempConnect}. First we state a criterion for connectedness and study hypersurfaces in general. 

\begin {Example}
The binoid $(\Z/2\Z)^{\infty}$ is not torsion-free and contains no idempotent elements, whereas its $K\mina$algebra $K[(\Z/2\Z)^{\infty}]= K[X]/(X^{2}-1)$ contains the idempotent $\frac{1}{2}(X+1)$ if $\opchar K\not=2$ since
$$\bigg(\frac{X+1}{2}\bigg)^{2}\,\,=\,\,\frac{X^{2}+2X+1}{4}\,\,=\,\,\frac{1+2X+1}{4}\,\,=\,\,\frac{X+1}{2}\pkt$$
In particular, if $\opchar K\not=2$, then $K\minspec(\Z/2\Z)^{\infty}$ is not connected.
\end {Example}

The following result is our main combinatorial criterion for connectedness.

\begin {Theorem}\label{ThConnectedness}
Let $M$ be a finitely generated torsion-free binoid and $K$ algebraically closed. The following statements are equivalent:
\begin {ListeTheorem}
\item $K\minspec M$ is connected.
\item For any two cancellative $K\mina$components $X$ and $Y$ of $M$, there is a sequence $X(1)\kpkt X(n)$ of cancellative $K\mina$components with $X(1)=X$ and $X(n)=Y$ such that $X(i)\cap X(i+1)$ contains a $K\mina$point, $i\in\{1\kpkt n-1\}$.
\item For any two cancellative $K\mina$components $X$ and $Y$ of $M$, there is a sequence $X(1)\kpkt X(n)$ of cancellative $K\mina$components with $X(1)=X$ and $X(n)=Y$ such that $X(i)\cap X(i+1)$ contains a characteristic point, $i\in\{1\kpkt n-1\}$.
\end {ListeTheorem}
\end {Theorem}
\begin {proof}
The equivalence $(1)\Leftrightarrow(2)$ follows from the definition of connectedness and from Proposition \ref{PropUnionCanComp} since all cancellative $K\mina$components of $M$ are  irreducible closed subsets of $K\minspec M$ by Lemma \ref{LemFgTfRalgclosedIrred}. The implication $(3)\Rightarrow(2)$ is trivial.

To prove $(2)\Rightarrow(3)$, it is enough to show that any two cancellative $K\mina$components, say $X=K\minspec(M/\Pcal)_{\opcan}$ and $Y=K\minspec(M/\Qcal)_{\opcan}$ for $\Pcal,\Qcal\in\spec M$, with a nonempty intersection contain a common characteristic point of $M$. Assuming $X\cap Y\not=\emptyset$, we have a $K\mina$point of $M$ that factors through $(M/\Pcal)_{\opcan}$ and through $(M/\Qcal)_{\opcan}$ via $M\mina$binoid homomorphisms
$$
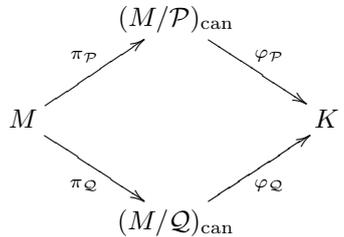
\xymatrix{
&(M/\Pcal)_{\opcan}\ar[dr]^{\varphi_{\Pcal}}&\\
M\ar[ru]^{\pi_{\Pcal}}\ar[rd]_{\pi_{\Qcal}}&&K\pkt\\
&(M/\Qcal)_{\opcan}\ar[ur]_{\varphi_{\Qcal}}&}$$
By Corollary \ref{CorCoProdSmashN}, there is an $M\mina$binoid homomorphism $$(M/\Pcal)_{\opcan}\wedge_{M}(M/\Qcal)_{\opcan}\Rto K\komma$$
in particular, $(M/\Pcal)_{\opcan}\wedge_{M}(M/\Qcal)_{\opcan}$ is not the zero binoid. Thus, there is a characteristic point $\alpha:(M/\Pcal)_{\opcan}\wedge_{M}(M/\Qcal)_{\opcan}\rto\{1,0\}\subseteq K$, and therefore 
$$
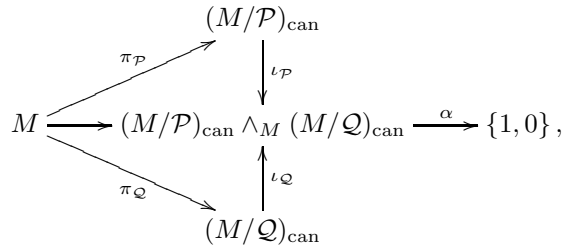
\xymatrix{
&(M/\Pcal)_{\opcan}\ar[d]^{\iota_{\Pcal}}&\\
M\ar[ru]^{\pi_{\Pcal}}\ar[rd]_{\pi_{\Qcal}}\ar[r]&(M/\Pcal)_{\opcan}\wedge_{M}(M/\Qcal)_{\opcan}\ar[r]^{\quad\quad\quad\quad\alpha}&\{1,0\}\komma\\
&(M/\Qcal)_{\opcan}\ar[u]_{\iota_{\Qcal}}&}$$
where $\alpha\iota_{\Pcal}\pi_{\Pcal}=\alpha\iota_{\Qcal}\pi_{\Qcal}:M\rto K$ is a common characteristic point of $M$ on $X$ and $Y$.
\end {proof}

\begin {Example}
The binoid $M:=\free(x,y)/(x+y=y)$ satisfies all assumptions of Theorem \ref{ThConnectedness}. Its $\R\mina$spectrum
\begin {center}
\begin {pspicture} (-2,-2.5)(2,2)
\qdisk (0,0){1.75pt}\qdisk (0.7,0){1.75pt}\qdisk (0.7,0.7){1.75pt}
\psline [linewidth=0.5 pt, linestyle=dotted] (0,-1.5)(0,1.5)
\psline [linewidth=0.5 pt] (-1.5,0)(1.5,0)
\psline [linewidth=0.5 pt] (0.7,-1.5)(0.7,1.5)
\uput [0] (-2.15,-2.2){\small{$\R\minspec\free(x,y)/(x+y=y)$}}
\end {pspicture}
\end {center}
is connected and the two cancellative $\R\mina$components meet in a characteristic point. To see this consider $\spec M$, which consists of three prime ideals, namely
$$\zero\,\subsetneq\,\langle y\rangle\,\subsetneq\,\langle x,y\rangle\,=\,\langle x\rangle\pkt$$
These yield the three cancellative parts:
$$\R\minspec M_{\opcan}\,\,\cong\,\,\R\minspec\N^{\infty}\,\,\cong\,\,\R\komma$$
which corresponds to the cancellative $\R\mina$component given by the vertical line since $x\sim_{\opcan}0$, cf.\ Lemma \ref{LemNspec},
$$\R\minspec (M/\langle y\rangle)_{\opcan}\,\,\cong\,\,\R\minspec\N^{\infty}\,\,\cong\,\,\R\komma$$
which corresponds to the cancellative $\R\mina$component given by the horizontal line, and 
$$\R\minspec (M/\langle x,y\rangle)_{\opcan}\,\,=\,\,\R\minspec\trivial\,\,\cong\,\,\zero\komma$$
which corresponds to the special point, cf.\ Example \ref{ExpBHomBoolean}(2).
\end {Example}

\begin {Example}\label{ExpKpoints2}
If $K$ contains the 8$th$ roots of unity $\zeta_{0}\kpkt\zeta_{7}$, for instance if $K=\C$, one may imagine the $K\mina$spectrum of the binoid $M:=\free(x,y)/(8x+y=y)$ in the following way:
\begin {center}
\begin {pspicture} (-2.5,-3)(2.5,2.5)
\qdisk (0,0){1.75pt}\qdisk (1.225,0){1.75pt}\qdisk (1.225,1.225){1.75pt}
\psline [linewidth=0.5 pt, linestyle=dotted] (0,-1.8)(0,1.8)
\psline [linewidth=0.5 pt] (-2,0)(2,0)
\psline [linewidth=0.5 pt] (1.225,-1.8)(1.225,1.8)
\psline [linewidth=0.5 pt] (0.875,-1.8)(0.875,1.8)
\psline [linewidth=0.5 pt] (0.525,-1.8)(0.525,1.8)
\psline [linewidth=0.5 pt] (0.175,-1.8)(0.175,1.8)
\psline [linewidth=0.5 pt] (-1.225,-1.8)(-1.225,1.8)
\psline [linewidth=0.5 pt] (-0.875,-1.8)(-0.875,1.8)
\psline [linewidth=0.5 pt] (-0.525,-1.8)(-0.525,1.8)
\psline [linewidth=0.5 pt] (-0.175,-1.8)(-0.175,1.8)
\uput [0] (-2.25,-2.2){\small{$K\minspec\free(x,y)/(8x+y=y)$}}
\end {pspicture}
\end {center}
The binoid satisfies all assumptions of Theorem \ref{ThConnectedness} except the torsion-freeness since $8y=(8x+y)+7y=8(x+y)$. However, its $K\mina$spectra is connected and all cancellative $K\mina$components are connected by characteristic points. To see this consider $\spec M$, which is given by the three prime ideals 
$$\zero\,\subsetneq\,\langle y\rangle\,\subsetneq\,\langle x,y\rangle\,=\,\langle x\rangle\pkt$$
These yield the three cancellative parts:
$$K\minspec M_{\opcan}\,\,\cong\,\,K\minspec((\Z/8\Z)^{\infty}\wedge\N^{\infty})\,\,\cong\,\,\{\zeta_{0}\kpkt\zeta_{7}\}\times K\komma$$
which corresponds to the cancellative $K\mina$component given by the $8$ vertical lines since $8x\sim_{\opcan}0$, cf.\ Lemma \ref{LemNoRelations} and Proposition \ref{PropNspecSmash}, 
$$K\minspec (M/\langle y\rangle)_{\opcan}\,\,\cong\,\,K\minspec\N^{\infty}\,\,\cong\,\,K\komma$$
which corresponds to the cancellative $K\mina$component given by the horizontal line, cf.\ Lemma \ref{LemNspec}, and 
$$K\minspec (M/\langle x\rangle)_{\opcan}\,\,\cong\,\,K\minspec\trivial\,\,\cong\,\,\zero\komma$$
which corresponds to the special point, cf.\ Example \ref{ExpBHomBoolean}(2).
\end {Example}

Now we study the case of hypersurfaces.

\begin{Proposition}\label{Prop1MonEquConnectedness}
Let $K$ be a field of characteristic zero and consider the binoid 
$$M\,\,:=\,\,\free(x_{1}\kpkt x_{n})/(f=g)$$
with $f=f_{1}x_{1}\pluspkt f_{n}x_{n}$ and $g=g_{1}x_{1}\pluspkt g_{n}x_{n}$, so that
$$K[M]\,\,=\,\,K[X]/(X^{f}-X^{g})\komma$$
where $X^{\nu}:=X_{1}^{\nu_{1}}\cdots X_{n}^{\nu_{n}}$ for $\nu=(\nu_{1}\kpkt\nu_{n})\in\N^{n}$.
\begin {ListeTheorem}
\item If $g=0$, so that 
$$K[M]\,\,=\,\,K[X]/(X^{f}-1)\komma$$
then $K\minspec M$ is connected if and only if $M$ is torsion-free. In this case, there are $r,s\ge0$ such that
$$K\minspec M\,\,=\,\,\A^{r}(K)\times (K\okreuz)^{s}\pkt$$
\item If $g\not=0$ and $f=\tilde{f}+g$, so that 
$$K[M]\,\,=\,\,K[X]/X^{g}(X^{\tilde{f}}-1)\komma$$
then $K\minspec M$ is disconnected if $\supp g\subseteq\supp\tilde{f}$ and connected otherwise.
\item If $f=\tilde{f}+h$ and $g=\tilde{g}+h$ such that $\supp\tilde{f},\supp\tilde{g}\not=\emptyset$ with $\supp\tilde{f}\cap\supp\tilde{g}=\emptyset$, i.e.\
$$K[M]\,\,=\,\,K[X]/X^{h}(X^{\tilde{f}}-X^{\tilde{g}})\komma$$
then $K\minspec M$ is connected.
\end {ListeTheorem}
\end{Proposition}
\begin {proof}
(1) In this situation, we have by the structure theorem for finitely generated commutative groups
$M\cong(\N^{r}\times\Z^{s}\times\Z/k_{1}\Z\timespkt\Z/k_{l}\Z)^{\infty}$, where $r:=\#\{i\mid f_{i}=0\}$. Thus, $K[M]=K[(\N^{r})^{\infty}]\otimes_{K}K[(\Z^{s})^{\infty}]\otimes_{K}K[T^{\infty}]$, which gives
$$K\minspec M\,\,=\,\,\A^{r}(K)\times (K\okreuz)^{s}\times K\minspec T^{\infty}\komma$$
where $T:=\Z/k_{1}\Z\timespkt\Z/k_{l}\Z$ is a finite torsion group. In particular, $M$ is not torsion-free if and only if $T\not=\{0\}$, which is equivalent to $K\minspec T^{\infty}$ is a finite set with more than one point. Therefore, $K\minspec M$ is connected if and only if $M$ is torsion-free.

(2) First let $\supp g\subseteq\supp\tilde{f}$ and choose $k\ge1$ such that $k\tilde{f_{i}}\ge g_{i}$ for all $i\in\supp\tilde{f}$. By assumption, we have $g=f=\tilde{f}+g$, which implies that $g=k\tilde{f}+g$, where $k\tilde{f}=\sum_{i\in\supp\tilde{f}}k\tilde{f_{i}}x_{i}$. Hence, $k\tilde{f}$ is an idempotent element in $M$ because
$$2k\tilde{f}=(k\tilde{f}+g)+(k\tilde{f}-g)=g+(k\tilde{f}-g)=k\tilde{f}\komma$$
in particular $X^{k\tilde{f}}$ is idempotent in $K[M]$, so $K\minspec M$ is not connected. Now let $\supp g\not\subseteq\supp\tilde{f}$, say $\supp g=\{1\kpkt r\}$ but $1\not\in\supp\tilde{f}$. Then
$$K\minspec M\,\,=\,\,\opV_{K}(X^{g}(X^{\tilde{f}}-1))\,\,=\,\,\opV_{K}(X^{g})\cup\opV_{K}(X^{\tilde{f}}-1)\komma$$
where $\opV_{K}(X^{g})=\bigcup_{i=1}^{r}\opV_{K}(X_{i}^{g_{i}})=\bigcup_{i=1}^{r}\opV_{K}(X_{i})$ is connected because the origin
$(0\kpkt 0)\in\A^{n}(K)$ lies in each $\opV_{K}(X_{i}^{g_{i}})$, $i\in\supp g$. Moreover, though $\opV_{K}(X^{\tilde{f}}-1)$ might not be connected itself, every $K\mina$point $a=(a_{1},a_{2}\kpkt a_{n})\in\opV_{K}(X^{\tilde{f}}-1)$ lies on the line $L_{a}:=\{(t,a_{2}\kpkt a_{n})\mid t\in\A^{1}(K)\}\subseteq\opV_{K}(X^{f}-1)$ since $X_{1}$ does not occur in $X^{\tilde{f}}$, on which the $K\mina$point $(0,a_{2}\kpkt a_{n})\in\opV_{K}(X_{1}^{g_{1}})$ lies as well. Therefore, every $K\mina$point of $\opV_{K}(X^{\tilde{f}}-1)$ is connected with $\opV_{K}(X^{g})$.

(3) Here,
$$K\minspec M\,\,=\,\,\opV_{K}(X^{h})\cup\opV_{K}(X^{\tilde{f}}-X^{\tilde{g}})\,\,=\,\,\Big(\bigcup_{i\in\supp h}\opV_{K}(X_{i}^{h_{i}})\Big)\cup\opV_{K}(X^{\tilde{f}}-X^{\tilde{g}})\komma$$
where $\opV_{K}(X^{\tilde{f}}-X^{\tilde{g}})$ is connected and contains the origin $(0\kpkt 0)$, cf.\ Corollary \ref{CorConectedHyperplane}, which also lies in every $\opV_{K}(X_{i}^{h_{i}})$, $i\in\supp h$.
\end {proof}

Note that the assumption on $K$ being algebraically closed of characteristic zero in the preceding proposition may make things worse, that is to say, in this case the $K\mina$spectrum of a binoid is more likely to be disconnected. For instance, in dimension $0$ or in case of one generator, the $K\mina$spectrum $\opV_{K}(X^{n}-1)$ of the (non-torsion-free) binoid $(\Z/n\Z)^{\infty}$ with $n\ge2$ such that $n\equiv1\mod2$ and $\opchar K=0$, consists only of a single point and is therefore connected if $K$ does not contain the $n$th roots of unity. Otherwise, $K\minspec(\Z/n\Z)^{\infty}$ is a set with $n\ge2$ points, hence hausdorff, cf.\ Example \ref{ExpCspecZmod8}. On the other hand, if $\opchar K=p>0$, then $K\minspec(\Z/p^{r}\Z)^{\infty}=\opV_{K}(X^{p^{r}}-1)=\opV_{K}((X-1)^{p^{r}})$ is a singleton as well. 

\begin {Example}
We have
$$K\minspec\free(x,y,z)/(x+y+2z=x+y)=\opV_{K}(XY(Z^{2}-1))=\opV_{K}(X)\cup\opV_{K}(Y)\cup\opV_{K}(Z^{2}-1)\komma$$
and these planes are connected for $K$ infinite. In this case, the connected sets $\opV_{K}(X)$ and $\opV_{K}(Y)$ intersect in the line $L=\{(0,0,a)\mid a\in\A^{1}(K)\}$. Though the closed set $\opV_{K}(Z^{2}-1)$ decomposes further into the disjoint connected sets $\opV_{K}(Z-1)$ and $\opV_{K}(Z+1)$ (i.e.\ $\opV_{K}(Z^{2}-1)$ is not connected), $K\minspec\free(x,y,z)/(x+y+2z=x+y)$ is connected since $L$ intersects with $\opV_{K}(Z-1)$ in $(0,0,1)$ and with $\opV_{K}(Z+1)$ in $(0,0,-1)$.

\begin{center}
\resizebox{0.33\linewidth}{!}{\begin {pspicture} (-5,-4)(5,4)
\qdisk (0,0){3pt}\qdisk (0,1.5){3pt}\qdisk (1.2,0){3pt}\qdisk (-0.72,0.96){3pt}
\qdisk (1.2,1.5){3pt}\qdisk (-0.72,-0.54){3pt}\qdisk (0.48,0.96){3pt}
\psline [linewidth=0.5 pt, linestyle=dashed] (-2.9,1.5)(-2.9,3.5)
\psline [linewidth=0.5 pt, linestyle=dashed] (-2.9,3.5)(2.9,3.5)            
\psline [linewidth=0.5 pt, linestyle=dashed] (2.9,3.5)(2.9,-3.5)            
\psline [linewidth=0.5 pt, linestyle=dashed] (2.9,-3.5)(0,-3.5)              
\psline [linewidth=0.5 pt, linestyle=dashed] (-1.6,-3.5)(-2.9,-3.5)         
\psline [linewidth=0.5 pt, linestyle=dashed] (-2.9,-3.5)(-2.9,-2.7)      
\psline [linewidth=0.5 pt, linestyle=dashed] (-2.9,-1.5)(-2.9,0.3)      

\psline [linewidth=0.5 pt, linestyle=dashed] (0,-3.5)(-1.6,-4.7)
\psline [linewidth=0.5 pt, linestyle=dashed] (-1.6,-4.7)(-1.6,2.3)          
\psline [linewidth=0.5 pt, linestyle=dashed] (-1.6,2.3)(1.6,4.7)             
\psline [linewidth=0.5 pt, linestyle=dashed] (1.6,4.7)(1.6,3.5)             

\psline [linewidth=0.5 pt, linestyle=dashed] (-2.9,1.5)(-4.5,0.3)
\psline [linewidth=0.5 pt, linestyle=dashed] (-4.5,0.3)(1.3,0.3)           
\psline [linewidth=0.5 pt, linestyle=dashed] (1.3,0.3)(4.5,2.7)             
\psline [linewidth=0.5 pt, linestyle=dashed] (4.5,2.7)(2.9,2.7)            
\psline [linewidth=0.5 pt, linestyle=dashed] (-2.9,-1.5)(-4.5,-2.7)
\psline [linewidth=0.5 pt, linestyle=dashed] (-4.5,-2.7)(1.3,-2.7)           
\psline [linewidth=0.5 pt, linestyle=dashed] (1.3,-2.7)(4.5,-0.3)             
\psline [linewidth=0.5 pt, linestyle=dashed] (4.5,-0.3)(2.9,-0.3)            
\psline [linewidth=0.5 pt] (0,1.5)(0,3.5)
\psline [linewidth=0.5 pt] (0,0.3)(0,-1.5)               
\psline [linewidth=0.5 pt] (0,-2.7)(0,-3.5)               
\psline [linewidth=0.5 pt, linestyle=dotted] (-1.6,0)(-2.9,0)
\psline [linewidth=0.5 pt, linestyle=dotted] (0,0)(2.9,0)
\psline [linewidth=0.5 pt] (-2.9,1.5)(-1.6,1.5)
\psline [linewidth=0.5 pt] (0,1.5)(2.9,1.5)
\psline [linewidth=0.5 pt] (-2.9,-1.5)(-1.6,-1.5)
\psline [linewidth=0.5 pt] (0,-1.5)(2.9,-1.5)
\psline [linewidth=0.5 pt, linestyle=dotted] (-1.6,-1.2)(0,0)
\psline [linewidth=0.5 pt] (-1.6,0.3)(0,1.5)
\psline [linewidth=0.5 pt] (-1.6,-2.7)(0,-1.5)
\end {pspicture}}
\end {center}
$$\text{\small{$\R\minspec\free(x,y,z)/(x+y+2z=x+y)$}}$$
\end {Example}

\begin {Example}
For $K$ infinite, we have
\begin {align*}
K\minspec\free(x,y,z)/(2x+y+z=x+z)&=\opV_{K}(XZ(XY-1))\\
&=\opV_{K}(X)\cup\opV_{K}(Z)\cup\opV_{K}(XY-1)\pkt
\end {align*}

The connected sets $\opV_{K}(X)$ and $\opV_{K}(Z)$ intersect in the line $L=\{(0,t,0)\mid t\in\A^{1}(K)\}$ but $\opV_{K}(X)$ and $\opV_{K}(XY-1)\cong K\okreuz$ are disjoint. However, every $K\mina$point $(a,b,c)\in\opV_{K}(XY-1)$ lies on the line $L_{ab}:=\{(a,b,t)\mid t\in\A^{1}(K)\}$ which intersects $\opV_{K}(Z)$ in $(a,b,0)$.

\begin{center}
\resizebox{0.5\linewidth}{!}{\begin {pspicture} (-7,-3)(7,4)
\qdisk (0,0){3pt}\qdisk (1,0.5){3pt}\qdisk (2.3,0.5){3pt}\qdisk (2.3,1.8){3pt}
\psline [linewidth=0.5 pt, linestyle=dashed] (2,-1.5)(-2,-3.5)
\psline [linewidth=0.5 pt, linestyle=dashed] (-2,-3.5)(-2,1.5)          
\psline [linewidth=0.5 pt, linestyle=dashed] (-2,1.5)(2,3.5)             
\psline [linewidth=0.5 pt, linestyle=dashed] (2,3.5)(2,1)             
\psline [linewidth=0.5 pt, linestyle=dashed] (2,-1.5)(2,-1)             
\psline [linewidth=0.5 pt, linestyle=dashed] (-6.5,-1)(-5.1,-0.3)
\psline [linewidth=0.5 pt, linestyle=dashed] (-2.3,0.85)(-2,1)    
\psline [linewidth=0.5 pt, linestyle=dashed] (-6.5,-1)(2.5,-1)       
\psline [linewidth=0.5 pt, linestyle=dashed] (2.5,-1)(6.5,1)         
\psline [linewidth=0.5 pt, linestyle=dashed] (6.5,1)(4.85,1) 
\psline [linewidth=0.5 pt, linestyle=dashed] (2.3,1)(2,1) 
\pcarc [arcangleA=30, arcangleB=10, linewidth=1 pt, linestyle=dashed] (2.3,2.5)(2.5,2.8)
\pcarc [arcangleA=5, arcangleB=0, linewidth=0.5 pt, linestyle=dashed] (2.5,2.8)(2.6,2.9)
\pcarc [arcangleA=300, arcangleB=340, linewidth=0.5 pt, linestyle=dashed] (2.3,2.5)(2.8,2.25)
\pcarc [arcangleA=355, arcangleB=355, linewidth=0.5 pt, linestyle=dashed] (2.8,2.25)(4.85,2.2)
\pcarc [arcangleA=300, arcangleB=340, linewidth=0.5 pt] (2.3,0.5)(2.8,0.25)
\pcarc [arcangleA=355, arcangleB=355, linewidth=0.5 pt] (2.8,0.25)(4.85,0.2)
\pcarc [arcangleA=300, arcangleB=340, linewidth=0.5 pt, linestyle=dashed] (2.3,-1.5)(2.8,-1.75)
\pcarc [arcangleA=355, arcangleB=355, linewidth=0.5 pt, linestyle=dashed] (2.8,-1.75)(4.85,-1.8)
\psline [linewidth=0.5 pt, linestyle=dashed] (2.6,2.9)(2.6,2.3) 
\psline [linewidth=0.5 pt, linestyle=dashed] (4.85,2.2)(4.85,-1.8) 
\psline [linewidth=0.5 pt, linestyle=dashed] (2.3,0.5)(2.3,2.5)  
\psline [linewidth=0.5 pt, linestyle=dashed] (2.3,-1)(2.3,-1.5)   
\pcarc [arcangleA=0, arcangleB=10, linewidth=0.5 pt, linestyle=dashed] (-5.1,1.7)(-2.7,1.6) 
\pcarc [arcangleA=10, arcangleB=25, linewidth=0.5 pt, linestyle=dashed] (-2.7,1.6)(-2.3,1.4)
\pcarc [arcangleA=25, arcangleB=25, linewidth=0.5 pt, linestyle=dashed] (-2.3,1.4)(-2.4,1.2)
\pcarc [arcangleA=5, arcangleB=15, linewidth=0.5 pt, linestyle=dashed] (-2.4,1.2)(-2.7,1)
\pcarc [arcangleA=0, arcangleB=10, linewidth=0.5 pt] (-5.1,-0.3)(-2.7,-0.4) 
\pcarc [arcangleA=25, arcangleB=25, linewidth=0.5 pt] (-2.3,-0.5)(-2.4,-0.8)
\pcarc [arcangleA=5, arcangleB=15, linewidth=0.5 pt] (-2.4,-0.8)(-2.7,-1)
\pcarc [arcangleA=0, arcangleB=10, linewidth=0.5 pt, linestyle=dashed] (-5.1,-2.3)(-2.7,-2.4)
\pcarc [arcangleA=25, arcangleB=25, linewidth=0.5 pt, linestyle=dashed] (-2.3,-2.5)(-2.4,-2.8) 
\pcarc [arcangleA=5, arcangleB=15, linewidth=0.5 pt, linestyle=dashed] (-2.4,-2.8)(-2.7,-3)
\psline [linewidth=0.5 pt, linestyle=dashed] (-2.7,1)(-2.7,-3)    
\psline [linewidth=0.5 pt] (-2.3,1.4)(-2.3,-0.6) 
\psline [linewidth=0.5 pt] (-2.3,-1)(-2.3,-2.6)  
\psline [linewidth=0.5 pt, linestyle=dashed] (-5.1,1.7)(-5.1,-0.3) 
\psline [linewidth=0.5 pt, linestyle=dashed] (-5.1,-2.3)(-5.1,-1) 
\psline [linewidth=0.5 pt, linestyle=dotted] (0,0)(0,2.5)
\psline [linewidth=0.5 pt, linestyle=dotted] (0,-2.5)(0,-1)
\psline [linewidth=0.5 pt, linestyle=dotted] (-2.3,0)(-2,0)
\psline [linewidth=0.5 pt, linestyle=dotted] (0,0)(4.25,0)
\psline [linewidth=0.5 pt] (-2,-1)(2,1)
\end {pspicture}}
\end {center}
$$\text{\small{$\R\minspec\free(x,y,z)/(2x+y+z=x+z)$}}$$
\end {Example}

\begin {Corollary} \label{CorIdempConnect}
Let $M$ be a torsion-free binoid such that $M\cong\free(x_{1}\kpkt x_{n})/(f=g)$. If $K$ is a field of characteristic zero, then $K\minspec M$ is connected if and only if $M$ contains no non-trivial idempotent elements.
\end {Corollary}
\begin {proof}
In general (i.e.\ regardless if $M$ is torsion-free), if $e\in M$ is a non-trivial idempotent element, then $X^{e}$ is idempotent in $K[M]$, so $K\minspec M$ is not connected. On the other hand, $K\minspec M$ is disconnected if and only if $e=\tilde{e}+g$ and $\supp g\subseteq\supp\tilde{e}$ by Proposition \ref{Prop1MonEquConnectedness}. In this case, the element  $k\tilde{e}$, where $k\ge1$ satisfies $k\tilde{f_{i}}\ge g_{i}$ for all $i\in\{1\kpkt n\}$, is an idempotent element of $M$, cf.\ the proof of \ref{Prop1MonEquConnectedness}(2).
\end {proof}

\begin {Example}
Let $K$ be an algebraically closed field. For any two elements $f,g\in\free(x_{1}\kpkt x_{n})$ with 
$i\in \supp f$ and $j\in\supp g$, but $j\not\in \supp f$ and $i\not\in\supp g$
for some $i,j\in\{1\kpkt n\}$, the binoids
$$\free(x_{1}\kpkt x_{n})/(mf=mg)\komma$$
$m\ge2$, are not torsion-free, but their $K\mina$spectra are connected if $\opchar K=0$ by Proposition \ref{Prop1MonEquConnectedness}(3).
\end {Example}

\bigskip

\chapter {Separation and gradings} \label {ChapSepGradings}
\markright{\ref{ChapKpointsGradings} Separation and gradings}

This chapter is concerned with the question when the special point is contained in every cancellative component. For this, we first study the separating ideal $\bigcap_{n\ge1}nM\Uplus$ of a commutative binoid. We are  interested in the case when this ideal is $\zero$, in which we call $M$ a separated binoid. Under certain conditions the separating ideal can easily be described, cf.\ Proposition \ref{PropSepIdeal}, which yields the notion of (un-) separated elements. We show that the existence of an unseparated element gives a necessary condition to our question, cf. Proposition \ref{PropReducedSepNoSP}. The second part of this chapter deals with gradings on a binoid (by a monoid). We are particularly interested in positive $\N^{k}\mina$gradings, whose existence is equivalent to $M$ being separated under certain assumptions, cf. Theorem \ref {ThSepGraded}, and which yield a positive result to our initial question, cf.\ Theorem \ref{ThGradedConnectedSpecialPoint}.

\bigskip

\section {Separated binoids} \label {SecSepBinoids}
\markright {\ref{SecSepBinoids} Separated binoids}

\begin {Convention}
In this section, arbitrary binoids are assumed to be \emph{commutative}.
\end {Convention}

The main focus of this section is the ideal $\bigcap_{n\ge 1}nM\Uplus$. As in ring theory, it is of interest to know when this is the zero ideal because this has topological consequences regarding the $K\mina$spectra. For a treatment of this ideal within commutative algebra of binoids, the reader may consult \cite{AndersonIT}, and \cite{Kobsa} for that of monoids.

\begin {Definition}
The ideal $\bigcap_{n=1}^{\infty}nM\Uplus$ is called the \gesperrt{separating} \index{ideal!separating --}ideal of $M$. A nonzero binoid $M$ is called \gesperrt{separated} \index{binoid!separated --}if $$\bigcap_{n\ge 1}nM\Uplus\,\,=\,\,\zero\pkt$$
The \gesperrt{separated dimension} \index{dimension!separated --}of a binoid is defined to be $\dim (M/\bigcap_{n=1}^{\infty}nM\Uplus)$ and will be denoted by $\sepdim M$.\nomenclature[Dimension3]{$\sepdim M$}{separated dimension of the binoid $M$}
\end {Definition}

If $M$ is separated, then $\dim M=\sepdim M$, but the converse need not be true, cf. Example \ref{ExpSepDim1}.

\begin{Example}
The binoid $\N^{\infty}$ is separated because $n(\N^{\infty})\Uplus=\N^{\infty}_{\ge n}$, and hence
$$\bigcap_{n\ge 1}n(\N^{\infty})\Uplus\,=\,\bigcap_{n\ge 1}(\N^{\infty}_{\ge n})\Uplus\,=\,\zero\pkt$$
\end {Example}

\begin {Remark} \label{RemCompletion}
The definition of being separated stems from the terminology in ring theory, which is based on the fact that the $\aideal\mina$adic topology defined on a commutative ring $R$, $\aideal\subseteq R$ an ideal, is separated (hausdorff) if and only if 
$$\bigcap_{n\ge 1}\aideal^{n}\,\,=\,\,\{0\}\komma$$ 
cf.\ \cite[Chapter 10]{AtiyahMacDonald} or \cite[Chapter 0 \S7.2]{EGA}. Recall that the $\aideal\mina$adic completion \index{completion!$\aideal\mina$adic -- (of a ring)}of $R$ is defined to be the projective limit of the inverse system $(R/\aideal^{n},\phi_{nm}:R/\aideal^{m}\rto R/\aideal^{n})_{m\ge n\ge0}$ and is denoted by\nomenclature[R]{$\hat{R}$}{$\aideal\mina$adic completion of the ring $R$} 
$$\varprojlim R/\aideal^{n}\,\,:=\,\hat{R}\pkt$$
Similarly, every binoid $M$ defines an inverse system $(M_{n},\varphi_{nm})_{m\ge n\ge 0}$ of binoids, where
$$M_{n}:=M/nM\Uplus\quad\text{and}\quad\varphi_{nm}:M_{m}\rto M_{n}\komma\quad\komma m\ge n\komma$$ 
are the canonical projections. The projective limit
$$\varprojlim M_{n}\,=\,\Big\{(a_{n})_{n\ge 0}\in\prod_{n\ge 0}M_{n}\,\,\Big|\,\,\varphi_{mn}(a_{n})=a_{m}\komma m\ge n\Big\}\komma$$
cf.\ Lemma \ref{LemProjectiveLimit}, may be considered as the \gesperrt{completion} \index{completion!-- of a binoid}of $M$. The inverse system $(M_{n},\varphi_{nm})_{m\ge n\ge 0}$ can be illustrated as follows
$$\cdots\Rto\,M/(k+1)M\Uplus\,\,\stackrel{\!\!\!\!\varphi_{k,k+1}}{\Rto}\,\,M/kM\Uplus\,\Rto\cdots\Rto\, M/2M\Uplus\stackrel{\!\!\!\!\varphi_{12}}{\Rto}\,M/M\Uplus\,\stackrel{\!\!\!\!\varphi_{01}}{\Rto}\,\trivial\pkt$$
Since the binoid homomorphisms $\varphi_{nm}$, $m\ge n$, are injective on $M_{m}\setminus nM\Uplus$, every $\infty\not=(a_{n})_{n\ge0}\in\varprojlim M_{n}$ is of the form 
$$(\infty\kpkt\infty,a,a,a,\ldots)$$
(by abuse of notation we write $a$ for the class $[a]_{n}$ in $M_{n}$ if $a\not\in nM\Uplus$). More precisely, $a_{n}=\infty$ for $0\le n\le k$ and $a_{n}=a\in M_{n}$ for $n\ge k+1$ and some $a\in M\opkt$, where $k\ge0$ is the index such that $a\in kM\Uplus$ but $a\not\in(k+1)M\Uplus$. In particular, if $M$ is separated, there is a one-to-one correspondence $a\leftrightarrow ([a]_{n})_{n\ge0}$, which yields 
$$M\,\,\cong\,\,\varprojlim M_{n}\pkt$$

If, in addition, $M$ is positive, then $K[M\Uplus]=:\mideal$ is a maximal ideal in $K[M]$ by Proposition \ref{PropBinoidMax=AlgebraMax}. Therefore,
$$\varprojlim K[M_{n}]\,=\,\varprojlim K[M]/\mideal^{n}\,=\, \widehat{K[M]}\pkt$$
In particular, if $K[M]$ is not complete (i.e.\ $K[M]\not=\widehat{K[M]}$), then $K[\varprojlim M_{n}]\not=\varprojlim K[M_{n}]$.
\end {Remark}

\begin {Lemma} \label{LemSubbinoidSep}
Let $M$ be a nonzero binoid.
\begin {ListeTheorem}
\item $M/\bigcap_{n\ge 1}nM\Uplus$ is separated.
\item If $M$ is integral with $\dim M=\sepdim M<\infty$, then $M$ is separated.
\item If $M$ is positive and separated, then so is every subbinoid of $M$.
\item If $M$ is separated, then so is $M/\Ical$ for every ideal $\Ical\subseteq M$.
\end {ListeTheorem}
\end {Lemma}
\begin {proof}
(1) Set $\Ical:=\bigcap_{n\ge 1}nM\Uplus$. By Corollary \ref{CorExtIdealPrime}, $(M/\Ical)\Uplus=M\Uplus/\Ical$ which implies the assertion. (2) If $\bigcap_{n=1}^{\infty}nM\Uplus\not=\zero$, then the prime ideal $\zero$ is contained in $\bigcap_{n=1}^{\infty}nM\Uplus$, and hence $\sepdim M<\dim M$ by Corollary \ref {CorExtIdealPrime}, contrary to our assumption. (3) Let $N\subseteq M$ be a subbinoid and $\iota:N\embto M$ the canonical binoid embedding. We have $\iota(N\Uplus)\subseteq M\Uplus$ by the positivity of $M$, which implies that $\iota(\bigcap_{n\ge 1}nN\Uplus)\subseteq\bigcap_{n\ge 1}nM\Uplus=\zero$. Thus, $N$ is separated as well (the positivity is trivial). (4) Suppose that $[f]\not=[\infty]$ for some $[f]\in\bigcap_{n\ge 1}n(M\Uplus/\Ical)$. For every $n\ge 1$, we find $a_{1}\kpkt a_{n}\in(M\Uplus)\setminus\Ical$ such that $f=a_{1}\pluspkt a_{n}\not=\infty$ in $M$. This means $\infty\not=f\in nM\Uplus$ for every $n\ge 1$, a contradiction to $M$ separated.
\end {proof}

\begin {Lemma}
Let $(M_{i})_{i\in I}$ be a finite family of nonzero binoids.
\begin {ListeTheorem}
\item If all $M_{i}$ are separated, then $\bigwedge_{i\in I}M_{i}$ is also separated. Conversely, if $\bigwedge_{i\in I}M_{i}$ is positive and separated, then all $M_{i}$ are so.
\item If all $M_{i}$ are positive, then $\bigcupbidot_{i\in I}M_{i}$ is separated if and only if all $M_{i}$ are so.
\end {ListeTheorem}
\end {Lemma}
\begin {proof}
(1) Set $M:=\bigwedge_{i\in I}M_{i}$. First let all $M_{i}$ be separated and assume that $f\in nM\Uplus$ for every $n\ge1$, where $\infty_{\wedge}\not=f=\wedge_{i\in I}f_{i}\in M$, $f_{i}\in M_{i}\opkt$. Since $M_{i}$ is separated, there is for every $i\in I$ an $m_{i}\ge1$ with $f_{i}\not\in l(M_{i})\Uplus$ for all $l\ge m_{i}$. Let $m:=\sum_{i\in I}m_{i}$ and $I=\{1\kpkt r\}$. By the description of the prime ideals in the smash product, cf.\ Corollary \ref{CorSpecSmash}, we have 
$$mM\Uplus\,=\,m\Big(\bigcup_{k\in I}M_{1}\wedgepkt(M_{k})\Uplus\wedgepkt M_{r}\Big)\pkt$$ 
Thus, for every $g=\wedge_{i\in I}g_{i}\in mM\Uplus$ there is at least one 
$i\in I$ such that $g_{i}\in lM_{i}$ for some $l\ge m_{i}$, which means $f$ cannot be contained in $mM\Uplus$, a contradiction. For the converse let $M$ be positive and separated. Each $M_{k}$ can be considered as a subbinoid of $M$ via the canonical binoid embedding $\iota:M_{k}\embto M$, $a\mto a\widehat{e}_{k}$. Hence, all $M_{k}$ are separated and positive by Lemma \ref{LemSubbinoidSep}(3) and Lemma \ref{LemSmashRules}(2). The second statement is easily verified.
\end {proof}

\begin {Proposition} \label{PropSepIdeal}
Let $M$ be a binoid. Then
$$\{f\in M\mid f=f+g\text{ for some }g\in M\Uplus\}\,\,\subseteq\,\,\bigcap_{n\ge 1}nM\Uplus\komma$$
and equality holds if $M$ is finitely generated. In particular, if $M$ is positive and finitely generated, then
$$\bigcap_{n\ge 1}nM\Uplus\,\,=\,\,\{f\in M\mid f=f+g\text{ for some }g\not=0\}\pkt$$
\end {Proposition}
\begin {proof}
If $f=f+g$ for some $g\in M\Uplus$, then in particular $f=f+g\in M\Uplus$, which gives $f=f+ng\in nM\Uplus$ for all $n\ge 1$. This proves the inclusion. For the equality, we may assume that $\{x_{1}\kpkt x_{l}\}$ is a generating set of $M$ with $x_{1}\kpkt x_{k}\in M\Uplus$ and $x_{k+1}\kpkt x_{l}\in M\okreuz$. If $f\in\bigcap_{n\ge 1}nM\Uplus$, there is for every $n\ge1$ at least one $k\mina$tuple $r=(r_{1}\kpkt r_{k})\in\N^{k}$ such that $f=r_{1}x_{1}\pluspkt r_{k}x_{k}+u_{r}$ with  $u_{r}\in M\okreuz$ and $r_{1}\pluspkt r_{k}\ge n$. In particular, the set 
$$K=\{(r_{1}\kpkt r_{k})\in\N^{k}\mid f=r_{1}x_{1}\pluspkt r_{k}x_{k}+u_{r}\text{ for some }u_{r}\in M\okreuz\}$$ 
is infinite. By Dicksons's Lemma\index{Dickson's Lemma}, cf.\ \cite[Theorem 5.1]{GarciaRosales}, the subset $K^{\prime}\subseteq K$ of minimal elements in $K$ with respect to the product order is finite, in particular $K^{\prime}\subsetneq K$. By the definition of a minimal element and an easy finiteness argument, see \cite[Corollary 5.4]{GarciaRosales}, there is for every $r\in K$ an $s\in K^{\prime}$ with $s\le r$. Thus, we find a pair $(s,r)\in K^{\prime}\times K$ with $s\not= r$ and $s\le r$, which yields
$$f\,=\,r_{1}x_{1}\pluspkt r_{k}x_{k}+u_{r}\,=\,s_{1}x_{1}\pluspkt s_{k}x_{k}+u_{s}$$ 
for some $u_{r},u_{s}\in M\okreuz$. This shows that $f=f+g$ with $g=u_{r}+(\minus u_{s})+\sum_{i=1}^{k}(r_{i}-s_{i})x_{i}\in M\Uplus$. The supplement is clear.
\end {proof}

\begin {Example}\label{ExpSepDim1}
The prime ideal lattice of the binoid $M=\free(x,y,z)/(x+y=x+y+z, 2z=\infty)$ is given by 
$$\xymatrix{
&\langle x,y,z\rangle&\\
\langle x,z\rangle\!\!\!\!\!\!\!\!\!\!\!\!\!\!\!\!\!\!\!\!\!\!\!\!\!\!\!\!\!\!\ar@{}[ur]|{\subsetur\!\!\!\!\!\!\!\!\!\!\!\!\!}&&\!\!\!\!\!\!\!\!\!\!\!\!\!\!\!\!\!\!\!\!\!\!\!\!\!\!\!\!\!\!\langle y,z\rangle\ar@{}[ul]|{\!\!\!\!\!\!\!\!\!\!\!\!\!\supsetul}\\
&\langle z\rangle\ar@{}[ur]|{\!\!\!\!\!\!\!\!\!\!\!\!\!\subsetur}\ar@{}[ul]|{\supsetul\!\!\!\!\!\!\!\!\!\!\!\!\!}&\\
}$$
By Proposition \ref{PropSepIdeal}, the separating ideal of $M$ is $\langle x+y\rangle$. Since there is no $\Pcal\in\spec M$ with $\Pcal\subseteq\langle x+y\rangle$, the prime ideal lattices of $M$ and $M/\langle x+y\rangle$ are one-to-one by Corollary \ref {CorExtIdealPrime}. In particular, $\dim M=\sepdim M=2$.

Consider now the integral binoid $N=\free(x,y,z)/(x+y=x+y+z)$ (i.e.\ omit the relation $2z=\infty$), whose separating ideal is $\langle x+y\rangle$ by Proposition \ref{PropSepIdeal}, and whose prime spectrum is $\spec N=\{\langle J\rangle\mid J\in\Pset(\{x,y,z\})\}$. In particular, $\zero$ is a prime ideal contained in $\langle x+y\rangle$. This gives different dimensions, namely $\dim N=3$ and $\sepdim N=2$.
\end {Example}

\begin {Example}
Here is an example, where $\dim M$, $\sepdim M$, and the dimension $\Dim R[M]$ of the ring $R[M]$ are pairwise different.\nomenclature[Dimension7]{$\Dim R$}{dimension of the ring $R$} The prime ideal lattice of the binoid $M=\free(x,y,z)/(y+x=y, z+x=z)$ is given by 
$$\xymatrix{
&\langle x,y,z\rangle=\langle x\rangle&\\
&\langle y,z\rangle\ar@{}[u]|{\subsetu}&\\
\langle y\rangle\!\!\!\!\!\!\!\!\!\!\!\!\!\!\!\!\!\!\!\!\!\!\!\!\!\!\!\!\!\!\ar@{}[ur]|{\subsetur\!\!\!\!\!\!\!\!\!\!\!\!\!}&&\!\!\!\!\!\!\!\!\!\!\!\!\!\!\!\!\!\!\!\!\!\!\!\!\!\!\!\!\!\!\langle z\rangle\ar@{}[ul]|{\!\!\!\!\!\!\!\!\!\!\!\!\!\supsetul}\\
&\{\infty\}\ar@{}[ur]|{\!\!\!\!\!\!\!\!\!\!\!\!\!\subsetur}\ar@{}[ul]|{\supsetul\!\!\!\!\!\!\!\!\!\!\!\!\!}&\\
}$$
Hence, $\dim M=3$. By Proposition \ref{PropSepIdeal}, the separating ideal is $\langle y,z\rangle$, so $\sepdim M=1$. Note that for a field $K$ the Krull dimension $\Dim K[M]=\Dim K[X,Y,Z]/(YX-Y,ZX-Z)$ is $2$. 
\end {Example}

Proposition \ref{PropSepIdeal} gives rise to the following definition.

\begin {Definition}
An element $f\in M$ is called \gesperrt{weakly separated} \index{element!weakly (un-) separated --}\index{element!separated --}\index{element!unseparated --}if $f=\infty$ or $f=f+g$ for some $g\in M$ implies $g\in M\okreuz$, otherwise $f$ is \gesperrt{weakly unseparated}. We say $f\in M$ is \gesperrt{separated} if $f=\infty$ or $f=f+g$ for some $g\in M$ implies $g=0$, otherwise $f$ is \gesperrt{unseparated}.
\end {Definition}

Of course, separated implies weakly separated and the converse holds true if $M$ is positive. When dealing with positive binoids, we will not always mention this (and omit ``weakly''). Cancellative elements are always separated and so are the absorbing and the identity element of a binoid, whereas idempotent elements $\not\in\trivial$ are not even weakly separated. A binoid in which all elements $\not=0$ are nilpotent contains only the trivial separated elements since $x=x+y$ for some $y\not=0$ implies that $ny=\infty$ for some $n\ge2$, and hence $x=x+y=x+2y=\cdots=x+ny=\infty$.

\begin {Corollary}\label{CorSeperated}
In a separated binoid $M$ all elements are weakly separated, and the converse holds true if $M$ is finitely generated. In particular, a positive and finitely generated binoid is separated if and only if all elements are separated.
\end {Corollary}
\begin {proof}
This is an immediate consequence of Proposition \ref{PropSepIdeal}.
\end {proof}

\begin {Example}
Consider the positive binoid $(\Q_{\ge 0}^{\infty},+,0,\infty)$, which is not finitely generated, cf.\ Example \ref{ExpSemifree}(5). We have $(\Q_{\ge 0}^{\infty})\Uplus=\Q_{>0}^{\infty}$ and $a/b=n(a/nb)\in n\Q_{>0}^{\infty}$ for every $a/b\in\Q_{>0}^{\infty}$ and $n\ge1$. This shows that $\Q_{\ge 0}^{\infty}$ is not separated because
$$\bigcap_{n\ge 1}n(\Q_{>0}^{\infty})\Uplus\,=\,\,\Q_{>0}^{\infty}\komma$$
but all elements are separated since $\Q_{\ge 0}^{\infty}$ is cancellative.
\end {Example}

\begin {Lemma} 
Let $M$ be a nonzero binoid and $\Ical\subsetneq M$ an ideal. If $f\in M$ is (weakly) separated, then so is $[f]\in M/\Ical$.
\end {Lemma}
\begin {proof}
This is immediate since $[f]=[f]+[g]$ in $M/\Ical$ is equivalent to $f=f+g$ or $f\in\Ical$, and since $\Ical\not=M$ implies that $\Ical\cap M\okreuz=\emptyset$, we have $(M/\Ical)\okreuz=\{[u]\mid u\in M\okreuz\}$.
\end {proof}

\begin{Lemma}
Let $(M_{i})_{i\in I}$ be a finite family of nonzero binoids. 
\begin {ListeTheorem}
\item An element $\wedge_{i\in I}f_{i}\in\bigwedge_{i\in I}M_{i}$ is (weakly) separated if and only if all $f_{i}\in M_{i}$ are so. 
\item If all $M_{i}$ are positive, then $(f;k)\in\,\,\bigcupbidot_{i\in I}M_{i}$ is separated if and only if $f\in M_{k}$ is so.
\end {ListeTheorem}
\end{Lemma}
\begin {proof}
(1) If $\wedge_{i\in I}f_{i}=\infty_{\wedge}$ or (equivalently) $f_{i}=\infty$ for some $i\in I$, the statement is trivial. So assume that all $f_{i}\not=\infty$ are (weakly) separated and 
$$\wedge_{i\in I}f_{i}=(\wedge_{i\in I}f_{i})+(\wedge_{i\in I}g_{i})=\wedge_{i\in I}(f_{i}+g_{i})$$ 
for some $g_{i}\in M_{i}$, $i\in I$. By the definition of the smash product, $\infty\not=f_{i}=f_{i}+g_{i}$ for all $i\in I$ which implies that $g_{i}=0$ ($g_{i}\in M\okreuz$) by assumption, and hence $\wedge_{i\in I}g_{i}=0_{\wedge}$ ($\wedge_{i\in I}g_{i}\in(\bigwedge_{i\in I}M_{i})\okreuz$, cf. Lemma \ref{LemSmashRules}(2)). Conversely, if $\infty_{\wedge}\not=\wedge_{i\in I}f_{i}$ is (weakly) separated and $f_{k}=f_{k}+g_{k}$ for some $k\in I$ and $g_{k}\in M_{k}$, then $f_{k}\not=\infty$ and $\infty_{\wedge}\not=\wedge_{i\in I}f_{i}=\wedge_{i\in I}f_{i}+g_{k}\widehat{e}_{k}$. This implies that $g_{k}\widehat{e}_{k}=0_{\wedge}$ ($g_{k}\widehat{e}_{k}\in(\bigwedge_{i\in I}M_{i})\okreuz$) by assumption, hence $g_{k}=0$ ($g_{k}\in M_{k}\okreuz$). (2) is trivial.
\end {proof}

\begin {Lemma} \label{LemSeparatedProp2}
Let $M$ be a binoid consisting only of (weakly) separated elements and $S\subseteq M$ a submonoid with $\infty\not\in S$.
\begin {ListeTheorem}
\item  If $M$ is positive, then so is $M_{\opcan, S}$.
\item If $S\subseteq\opint(M)$, then $M_{\opcan, S}$ consists only of (weakly) separated elements as well.
\end {ListeTheorem}
\end {Lemma}
\begin {proof}
(1) Note that in a positive binoid the notions of weakly separated and separated are equivalent. To show that $M_{\opcan, S}$ is positive suppose that $[0]=[f]+[g]=[f+g]$ in $M_{\opcan,S}$ for $f,g\in M$. This is equivalent to $f+g+s=s$ in $M$ for some $s\in S$. By assumption, $s=\infty$ or $f+g=0$, and since $\infty\not\in S$ the latter holds. Thus, $f=g=0$ by the positivity of $M$, hence $[f]=[g]=[0]$. (2) Assume that $[f]=[f]+[g]=[f+g]$ in $M_{\opcan,S}$. Then $f+s=g+f+s$ in $M$ for some $s\in S$, which implies that  $g=0$ ($g\in M\okreuz$) or $f+s=\infty$ by assumption. Since $s$ is an integral element, we obtain $g=0$ ($g\in M\okreuz$) or $f=\infty$. Thus, $[g]=0$ ($[g]\in (M_{\opcan,S})\okreuz$) or $[f]=[\infty]$.
\end {proof}

\begin {Definition}
An \gesperrt{order function} \index{order function}on a binoid $M$ is a map $\ord:M\opkt\rto \N$ such that 
$$\ord(f+g)\ge\ord(f)+\ord(g)$$ 
for all $f,g\in M$ with $f+g\not=\infty$. Such a function is called \gesperrt{positive} \index{order function!positive --}if $\ord(f)>0$ for all $f\in M\Uplus$.
\end {Definition}

Since there are no idempotent elements $\not=0$ in $\N$, every order function satisfies $\ord(0)=0$, which implies $\ord(u)=0$ for all $u\in M\okreuz$.

\begin {Proposition}
If $M$ is separated, then 
$$\delta:M\opkt\Rto \N\komma\quad f\lto\max\{k\in\N\mid f\in kM\Uplus\}\komma$$
where $0M\Uplus:=M$, is a positive order function on $M$. Conversely, a binoid is separated if it is finitely generated and admits a positive order function.
\end {Proposition}
\begin {proof}
It is easily checked that $\delta$ defines a positive order function because the convention $0M\Uplus=M$ implies that $\delta(f)=0$ if and only if $f$ is a unit. For the converse let $\ord:M\opkt\rto\N$ be a positive order function. Since $M$ is finitely generated, we have
$$\{f\in M\mid f=f+g\text{ for some }g\in M\Uplus\}\,=\,\bigcap_{n\ge 1}nM\Uplus$$
by Proposition \ref{PropSepIdeal}. Therefore, we only need to show that $\infty$ is the only weakly unseparated element. For this assume that $f=f+g$ for some $f,g\in M\opkt$. Then $\ord(f)\ge\ord(f)+\ord(g)$ in $\N$ which implies that $\ord(g)=0$, and hence $g\in M\okreuz$.
\end {proof}

\begin {Proposition}\label{PropReducedSepNoSP}
Let $M$ be a finitely generated reduced binoid and $K$ a field. If $M$ contains an unseparated element, then there exists an irreducible component of $K\minspec M$ that does not contain the special point.
\end{Proposition}
\begin {proof}
By assumption, there are elements $f,g\in M$ with $\infty\not=f=f+g$ and $g\not=0$. Since $M$ is reduced, the element $f+g$ is not nilpotent, which implies that $\opD(f+g)\not=\emptyset$ by Lemma \ref{LemBcalD}(5). Take any $\Pcal\in\opD(f+g)$ and let $Q:=\alpha_{\Pcal}:M\rto K$ be the (characteristic) $K\mina$point corresponding to $\Pcal$, which satisfies $\alpha_{\Pcal}(f+g)=\alpha_{\Pcal}(f)=1$ and $\alpha_{\Pcal}(g)=1$. Let $S:=\alpha_{M\Uplus}:M\rto K$ denote the special point. We have $T^{f}=T^{f}T^{g}$ in $K[M]$, and therefore $\opV_{K}(T^{f})\cup\opV_{K}(T^{g}-1)=K\minspec M$ by Lemma \ref{LemVKProperties}(3)\&(7). By the above considerations, the $K\mina$point $Q$ lies in $\opV_{K}(T^{g}-1)$ but is not contained in $\opV_{K}(T^{f})$. In particular, there is an irreducible component $C\subseteq\opV_{K}(T^{g}-1)$ with $Q\in C$. Then for every $K\mina$point $\varphi\in C$, one has $\varphi(T^{g})=1$, and therefore $S\not\in C$ since $g\in M\Uplus$. Thus, $C$ is the irreducible component in question.
\end {proof}

\begin {Example} 
Consider the following $\R\mina$spectra:
\begin {center}
\quad\quad
\begin {pspicture} (-1.5,-2.5)(2.5,2)
\qdisk (0,0){1.75pt}\qdisk (0,0.7){1.75pt}\qdisk (0.7,0){1.75pt}\qdisk (0.7,0.7){1.75pt}
\psline [linewidth=0.5 pt] (0,-1.5)(0,1.5)
\psline [linewidth=0.5 pt] (-1.5,0)(1.5,0)
\psline [linewidth=0.5 pt] (-1.3,-1.3)(1.3,1.3)
\uput [0] (-2.6,-2.2){\small{$\R\minspec\free( x,y)/(x+2y=2x+y)$}}
\end {pspicture}
\quad\quad\,\,
\begin {pspicture} (-1.5,-2.5)(2.5,2)
\qdisk (0,0){1.75pt}\qdisk (0,0.7){1.75pt}\qdisk (0.7,0.7){1.75pt}
\psline [linewidth=0.5 pt] (0,-1.5)(0,1.5)
\psline [linewidth=0.5 pt, linestyle=dotted] (-1.5,0)(1.5,0)
\psline [linewidth=0.5 pt] (-1.5,0.7)(1.5,0.7)
\uput [0] (-2.15,-2.2){\small{$\R\minspec\free(x,y)/(x+y=x)$}}
\end {pspicture}
\quad\quad\,\,
\begin {pspicture} (-1.5,-2.5)(2.5,2)
\qdisk (0,0){1.75pt}\qdisk (0,0.7){1.75pt}
\psline [linewidth=0.5 pt] (0,-1.5)(0,1.5)
\psline [linewidth=0.5 pt, linestyle=dotted] (-1.5,0)(1.5,0)
\uput [0] (-2.65,-2.2){\small{$\R\minspec\free(x,y)/(x+y=x,2x=\infty)$}}\end {pspicture}\end {center}
The binoid that yields the $\R\mina$spectrum on the left-hand side is separated, whereas the other binoids are not. These two only differ in being reduced or not and show that this condition is necessary in the preceding proposition. 
\end {Example}

\begin {Corollary}
Let $M$ be a finitely generated binoid that is positive and reduced. If $M$ is not separated, then there exists an irreducible component of $K\minspec M$ that does not contain the special point. 
\end {Corollary}
\begin {proof}
This is an immediate consequence of Proposition \ref{PropReducedSepNoSP} and Corollary \ref{CorSeperated}.
\end {proof}

\bigskip

\section {Graded binoids} \label{SecGraded}
\markright {\ref{SecGraded} Graded binoids}

In this section, we consider binoids that admit a grading and study their $K\mina$spectra for a field $K$. Graded binoids also appear in \cite{CortinasWeibelCDH}.

\begin {Definition}
Let $M$ be a binoid and $D$ a monoid. We call $M$ a $D\mina$\gesperrt{graded} \index{binoid!graded --}binoid if there is a map $\delta:M\opkt\rto D$ such that $\delta(0)=0$ and for all $x,y\in M\opkt$:
$$\delta(x+y)\,=\,\delta(x)+\delta(y)\quad\text{when}\quad x+y\not=\infty\pkt$$
Then $M$ is \gesperrt{graded} by $D$ with respect to $\delta$, and we call $(D,\delta)$ a \gesperrt{grading} \index{grading}of $M$. The image $\delta(x)$ is the \gesperrt{degree} \index{degree}\index{element!degree of an --}of $x$. $M$ is \gesperrt{positively} graded \index{binoid!positively graded --}\index{grading!positive --}by $D$ if in addition nonunits are mapped to nonunits; that is,  $\delta(M\opkt\Uplus)\subseteq D\setminus D\okreuz $.
\end {Definition}

By definition, every grading $\delta:M\opkt\rto D$ maps units and idempotent elements of $M$ to units and idempotent elements of $D$, respectively; that is, one always has $\delta(M\okreuz)\subseteq D\okreuz$ and $\delta((\bool M)\opkt)\subseteq\bool(D)$. 

\begin {Example}
\begin {ListeTheorem}
\item []
\item Given a binoid $M$ and a monoid $D$, the constant map $x\mto 0$ is a $D\mina$grading of $M$, which is positive if and only if $M$ is a binoid group. In particular, if $M$ is commutative, $\delta(M)$ need not necessarily be contained in the center $\{d\in D\mid d+c=c+d$ for all $c\in D\}$ of $D$.
\item An integral binoid $M$ is $(D,\delta)\mina$graded if and only if $\delta:M\opkt\rto D$ is a monoid homomorphism. 
\item The blowup binoid associated to a commutative binoid $M$ and an ideal $\Ical\subseteq M$, cf.\ Remark \ref{RemBlowupBinoid}, is an $\N\mina$graded binoid via
$$\delta:R_{\Ical}\opkt\Rto\N\komma\quad (a;n)\lto n\pkt$$
This $\N\mina$grading of $R_{\Ical}$ is positive if and only if $M$ is a binoid group.
\end {ListeTheorem}
\end {Example}

\begin {Remark}
If $\delta:M\opkt\rto D$ is a $D\mina$grading of $M$, then $M$ decomposes into the  pointed union
$$M\,=\,\bigcupdot_{d\in D}M_{d}$$
of the pointed sets $(M_{d},\infty)$, where $M_{d}=\{a\in M\mid\delta(a)=d\}\cup\zero$, $d\in D$, with $M_{d}+M_{e}\subseteq M_{d+e}$ for $d,e\in D$. By Proposition \ref{PropNsetsDecompositionDirectSum},
$$K[M]\,=\,\bigoplus_{d\in D}K[M_{d}]$$
as $K\mina$modules, that is to say, $K[M]$ is a $D\mina$graded $K\mina$algebra. In particular, if $M$ is a positive binoid with a positive $\N\mina$grading $\delta$, then
$$K[M]\,=\,\bigoplus_{n\in\N}R_{n}\quad\text{, where}\quad R_{n}\,:=\,\bigoplus_{\delta(a)=n}KT^{a}$$
is a positively graded $K\mina$algebra.
\end {Remark}

\begin {Convention}
In what follows, arbitrary binoids are assumed to be \emph{commutative}.
\end {Convention}

\begin {Lemma} \label{LemPropGraded}
Let $M$ be a binoid, $D$ a monoid, and $\delta:M\opkt\rto D$ a grading.
\begin{ListeTheorem}
\item Given an ideal $\Ical$, the quotient $M/\Ical$ is $D\mina$graded by $[a]\mto\delta(a)$, $a\in M\setminus\Ical$. In particular, if $M$ is positively graded, then so is $M/\Ical$.
\item Every binoid homomorphism $\varphi:N\rto M$ with $\ker\varphi=\zero$ yields a $D\mina$grading of $N$, namely $(\delta\varphi)_{|N\opkt}$. If $\delta$ is positive and $\varphi$ local, then $\delta\varphi_{|N\opkt}$ is also a positive grading.
\item Every monoid homomorphism $\varphi:D\rto E$ yields an $E\mina$grading of $M$, namely $\varphi\delta$. If $\varphi(D\setminus D\okreuz)\subseteq E\setminus E\okreuz$, then
 $\varphi\delta$ is positive if $\delta$ is so.
\end{ListeTheorem}
\end {Lemma}
\begin {proof}
All statements are easily verified.
\end {proof}

\begin {Lemma}\label{LemWunsepPosGrad}
Let $n\ge1$. A positively $\N^{n}\mina$graded binoid contains only weakly separated elements.
\end {Lemma}
\begin {proof}
If there are any elements $f, g\in M\opkt$ with $\infty\not=f=f+g$, then $\delta(f)=\delta(f)+\delta(g)$, which implies  that $\delta(g)=0\in(\N^{n})\okreuz$ since $\N^{n}$ is cancellative, hence $g\in M\okreuz$ by the positivity of $\delta$.
\end {proof}

\begin {Example} \label{ExpGradedSep}
The binoid $M:=\free(x,y)/(2x=x+y=3y)$ admits an unseparated element, namely $f=2x+2y$ because $f=f+y$ and $y\not\in M\okreuz=\{0\}$. Thus, $M$ cannot be positively $\N^{n}\mina$graded for every $n\ge1$ by the preceding lemma. Indeed, the only $\N^{n}\mina$grading $\delta$ is the trivial one with $x,y\mto0$ because $2\delta(x)=\delta(x)+\delta(y)$ implies  that $\delta(x)=\delta(y)$ by the cancellativity of $\N^{n}$, and therefore  $2\delta(x)=3\delta(y)=3\delta(x)$. Hence, $0=\delta(x)=\delta(y)$.
\end {Example}

\begin {Theorem} \label{ThSepGraded}
Let $M$ be a positive integral finitely generated binoid. The following conditions are equivalent.
\begin {ListeTheorem}
\item $M$ is separated.
\item $M$ is positively $\N^{k}\mina$graded for some $k\ge1$.
\item $M$ is positively $\N\mina$graded.
\end {ListeTheorem}
\end {Theorem}
\begin {proof}
To prove $(1)\Rarrow(2)$, we may assume that $M$ is regular. To see this consider the finitely generated regular binoid $M_{\opcan}=M/\sim_{\opcan,M^{\opkt}}$. By Lemma \ref{LemSeparatedProp2}, $M_{\opcan}$ is again separated and positive, and by Lemma \ref{LemPropGraded}(2) every grading of $M_{\opcan}$ extends to a grading of $M$ by the canonical projection $\pi:M\rto M_{\opcan}$.  
Thus, let $M$ be regular. By Proposition \ref{PropClassFgRegPos}(2), there exists an embedding
$$\varphi:M\embto(\N^{k}\times\Z/n_{1}\Z\timespkt\Z/n_{\ell}\Z)^{\infty}\komma\quad f\lto\varphi(f)=(\delta(f),\varphi_{1}(f)\kpkt\varphi_{\ell}(f))\komma$$
where $n_{1}\kpkt n_{\ell}\ge2$ and $k\ge1$. We claim $\delta:M\opkt\rto\N^{k}$ is a positive grading of $M$. Since $\delta$ is a monoid homomorphism, we only need to show that $\delta(f)=0\in\N^{k}$ implies $f=0$. Assume that $\varphi(f)=$ $(0,\varphi_{1}(f)\kpkt\varphi_{\ell}(f))$ for some $f\in M$. Then $\varphi(mf)=(0,0\kpkt0)=\varphi(0)$ with $m=n_{1}\cdots n_{\ell}$, and hence $mf=0$ by the injectivity of $\varphi$, which implies $f=0$ by the positivity of $M$. (2) implies (3) since the positive $\N^{k}\mina$grading of $M$ yields a positive $\N\mina$grading of $M$ via the evaluation homomorphism $\N^{k}\rto\N$, $(a_{1}\kpkt a_{k})\mto a_{1}\pluspkt a_{k}$, cf.\ Lemma \ref{LemPropGraded}(3). (1) follows from (3) by Lemma \ref{LemWunsepPosGrad}.
\end {proof}

The binoid of Example \ref{ExpGradedSep} fulfills all conditions of the preceding theorem but is not separated and admits only the trivial $\N^{n}\mina$grading, namely the constant map $a\mto0$.

\begin {Example}
The integrality condition in Theorem \ref{ThSepGraded} cannot be ommited. To see this consider first the integral binoid $M=\free(x,y,z)/(\Rcal_{1},\Rcal_{2},\Rcal_{3})$, where
\begin {align*}
\Rcal_{1}:\,\,\, a_{1}:=y+2z\,\,\,=\,\,\, 2y+z=:b_{1}\komma\\
\Rcal_{2}:\,\,\, a_{2}:=x+2z\,\,\,=\,\,\, 2x+z=:b_{2}\komma\\
\Rcal_{3}:\,\,\, a_{3}:=x+3y\,\,\,=\,\,\, 2x+y=:b_{3}\pkt
\end {align*}
The binoid $M$ fulfills all conditions of Theorem \ref{ThSepGraded} but cannot be positively $\N^{k}\mina$graded for $k\ge1$. In general, every grading $\delta:M\rto D$ by a cancellative monoid $D$ is trivial because $\Rcal_{1}$ and $\Rcal_{2}$ imply that $\delta(x)=\delta(y)=\delta(z)$ and $\Rcal_{3}$ implies that $4\delta(x)=3\delta(x)$, hence $0=\delta(x)=\delta(y)=\delta(z)$. In particular, $M$ is not separated by Theorem \ref{ThSepGraded}. An unseparated element in $M$ is given by 
\begin {align*}
f\,\,\,:=\,\,\,4x+5y+6z&\,\,\,=\,\,\,2b_{1}+2a_{2}+b_{3}\\
&\,\,\,=\,\,\,2a_{1}+2b_{2}+a_{3}\,\,\,=\,\,\,5x+5y+6z\,\,\,=\,\,\,x+f\pkt
\end {align*}
Now the binoid $M/(x+y+z=\infty)$ also fulfills all conditions of Theorem \ref{ThSepGraded} except the integrality, and for the same reason as $M$ it cannot be $\N^{k}\mina$graded, $k\ge1$, but it is separated. Elements of the form $nx+mz$ and $ny+mz$, $n,m\ge0$, are separated because
$$nx+mz=rx+sz\quad\text{or}\quad ny+mz=ry+sz$$
if and only if $r+s=n+m$ as one easily verifies. The elements $nx+my$ are separated since
\begin {align*}
x+ry&\,\,\,=\,\,\,2x+(r-2)y\\
&\,\,\,=\,\,\,3x+(r-4)y\\
&\,\,\,=\,\,\,4x+(r-6)y\,\,\,=\,\,\,\cdots\,\,\,=\,\,\,\begin {cases}
sx+2y&\text{, if }r=2s+2\komma s\ge1\komma\\
sx+y&\text{, if }r=2s+1\komma s\ge1\pkt
\end {cases}
\end {align*}
\end {Example}

Though the absorbing element has no degree by definition, every grading $\delta:M\opkt\rto D$ of an arbitrary binoid $M$ gives rise to a well-defined binoid homomorphism
$$\theta:M\Rto M\wedge D^{\infty}\quad\text{with}\quad\theta(x)=\begin {cases}
x\wedge\delta(x)&\text{, if }x\in M\opkt\!\!\komma\\
\infty_{\wedge}&\text{, otherwise,}\end {cases}$$
which in turn induces the following semigroup homomorphism
$$K\minspec D^{\infty}\times K\minspec M\Rto K\minspec M\komma\quad(\varphi_{1},\varphi_{2})\lto(x\mto\varphi_{1}(\delta(x))\cdot\varphi_{2}(x))\komma$$
cf.\ Proposition \ref{PropIndHomNspec} and Proposition \ref{PropNspecSmash}. In particular, if $M$ is $\N^{n}\mina$graded by $\delta:M\opkt\rto\N^{n}$, $a\mto(\delta_{1}(a)\kpkt\delta_{n}(a))$, and $K$ a field, there is a continuous $\A^{n}(K)\mina$action on $K\minspec M=:X$, namely
$$\A^{n}(K)\times X\Rto X\komma\quad (t,\Qcal)\lto t\Qcal\komma$$
where $t=(t_{1}\kpkt t_{n})\in\A^{n}(K)$ is given by $\varphi_{(t)}:(\N^{n})^{\infty}\rto K$ with $e_{i}\mto t_{i}$, and the $K\mina$point $\Qcal$ by $\varphi_{\Qcal}:M\rto K$. Hence, $t\Qcal$ is the $K\mina$point
$$\varphi_{t\Qcal}:a\lto t^{\delta(a)}\varphi_{\Qcal}(a)\,\,\,=\,\,\,t_{1}^{\delta_{1}(a)}\cdots t_{n}^{\delta_{n}(a)}\cdot\varphi_{\Qcal}(a)$$
in $X$. Note that $1\Qcal=\Qcal$ always holds.

\begin {Definition}
Let $K$ be a field and $M$ an $\N^{n}\mina$graded binoid. With the above notation, a $K\mina$point $\Qcal$ is called $\A^{n}(K)\mina$\gesperrt{connected}\index{connected!$\A^{n}(K)\mina$--}\index{A@$\A^{n}(K)\mina$connected} with another $K\mina$point $\Pcal$ if there is a $t\in\A^{n}(K)$ such that $t\Qcal=\Pcal$.
\end {Definition}

\begin {Theorem}\label{ThGradedConnectedSpecialPoint}
Let $K$ a field. If $M$ is a positive binoid that is positively $\N^{n}\mina$graded, then every $K\mina$point is $\A^{n}(K)\mina$connected with the special point. In particular, if $M$ is also finitely generated, then every cancellative $K\mina$component of $K\minspec M$ contains the special point.
\end {Theorem}
\begin {proof}
First, we claim that for an arbitrary $K\mina$point $\Qcal\in K\minspec M$ the $K\mina$point $0\Qcal$ given by 
$$\varphi_{0\Qcal}:a\lto 0^{\delta(a}\varphi_{\Qcal}(a)\,\,\,=\,\,\,0^{\delta_{1}(a)}\cdots 0^{\delta_{n}(a)}\cdot\varphi_{\Qcal}(a)\komma$$
is the special point; that is, $\varphi_{0\Qcal}(a)=1$ if $a=0$ and $1$ otherwise.
The first is obvious because $\delta(0)=(0\kpkt 0)\in\N^{n}$ and therefore 
$$\varphi_{t\Qcal}(0)=0^{0}\cdot\varphi_{\Qcal}(0)=\varphi_{\Qcal}(0)=1\pkt$$
To show that $\varphi_{0\Qcal}$ vanishes on $M\setminus\{0\}$, note that $\delta(a)=(\delta_{1}(a)\kpkt\delta_{n}(a))\not=0\in\N^{n}$ for every $a\in M\setminus\{0\}=M\Uplus$ since $M$ is positively graded. Thus for every $i\in\{1\kpkt r\}$ there is a $j\in\{1\kpkt n\}$ such that $\delta_{j}(a)\not=0$ and therefore $0^{\delta_{j}(a)}=0$, which implies that  $\varphi_{0\Qcal}(a)=0$. For the supplement, we may assume that $K$ is infinite, otherwise consider the algebraic closure $L$ of $K$ and the diagram
$$\xymatrix{
\quad\quad\quad\A^{n}(K)\times K\minspec M\ar[r]\ar[d]_{\iota}&K\minspec M\ar[d]^{\iota}\\
\quad\quad\quad\A^{n}(L)\times L\minspec M\ar[r]&L\minspec M\pkt}\quad\quad\quad$$
For an arbitrary $K\mina$point $\Qcal$, the orbit $O:=\{t\Qcal\mid t\in\A^{n}(K)\}$ of $\Qcal$ under the $\A^{n}(K)\mina$operation on $X:=K\minspec M$ is given by the image of the continuous map $\A^{n}(K)\rto X$, $t\mto t\Qcal$. Hence, $O$ is an irreducible subset of $X$ because $\A^{n}(K)$ is so. By Proposition \ref{PropUnionCanComp}, $X=\bigcup_{i\in I}X_{i}$ is the (finite) union of the closed cancellative $K\mina$components $X_{i}$, and the orbit $O$ of $\Qcal$ has to be contained in one of these components since otherwise $O=\bigcup_{i\in I}(O\cap X_{i})$ would be a non-trivial decomposition of $O$.
\end {proof}

\begin {Example} \label{ExpCspecZmod8}
The group homomorphism $\id:\Z/8\Z\rto\Z/8\Z$ defines a positive $\Z/8\Z\mina$grading of the binoid $(\Z/8\Z)^{\infty}$. The $\C\mina$spectrum is given by
\begin {center}
\begin {pspicture} (-1.5,-2.5)(2.5,2)
\qdisk (0.6,0.6){1.75pt}\qdisk (-0.6,0.6){1.75pt}\qdisk (-0.6,-0.6){1.75pt}\qdisk (0.6,-0.6){1.75pt}
\qdisk (0.85,0){1.75pt}\qdisk (0,0.85){1.75pt}\qdisk (-0.85,0){1.75pt}\qdisk (0,-0.85){1.75pt}
\psline [linewidth=0.5 pt, linestyle=dotted] (0,-1.5)(0,1.5)
\psline [linewidth=0.5 pt, linestyle=dotted] (-1.5,0)(1.5,0)
\uput [0] (-1.25,-2.2){\small{$\C\minspec(\Z/8\Z)^{\infty}$.}}
\end {pspicture}
\end {center}
\end {Example}

\begin {Corollary} \label{CorConectedHyperplane}
Let $K$ be an infinite field and $M=\free(x_{1}\kpkt x_{n})/(f=g)$ so that
$$K[M]=K[X]/(X^{f}-X^{g})\pkt$$
If $\supp f\cap\supp g=\emptyset$, then every cancellative component contains the special point. In particular, $K\minspec M$ is connected.
\end {Corollary}
\begin {proof}
For $f$ or/and $g=\infty$, the statement is clear. So let $f,g\not=\infty$. Then $M$ satisfies all assumptions of Theorem \ref{ThSepGraded} and is separated since $\supp f\cap\supp g=\emptyset$. Hence, $M$ is positively $\N\mina$graded, and therefore every $K\mina$point is $\A^{1}(K)\mina$connected with the special point by Theorem \ref{ThGradedConnectedSpecialPoint}. In particular, $K\minspec M$ is connected.
\end {proof}

\bigskip

\chapter {Binoids defined by a simplicial complex and simplicial binoids} \label{ChapSimplicial}
\markright{\ref {ChapSimplicial} Binoids defined by a simplicial complex and simplicial binoids}

Simplicial complexes and their associated (Stanley-Reisner) algebras play a major role in combinatorial commutative algebra. Stanley-Reisner algebras are special monomial algebras and like every monomial algebra they can be realized as binoid algebras, cf.\ Remark \ref{RemMonSemifree} and Lemma \ref{LemSRalgebraBinoid}. In this chapter, we treat these objects within the framework of binoid theory. 

We start with simplicial complexes by recalling basic definitions. With respect to the union and intersection of faces, every simplicial complex $\Delta$ defines two binoids, $\Delta_{\cup}^{\infty}$ and $\Delta_{\cap}\onull$, which are usually very different as we show by studying their generating sets. The Zariski topology on their spectra can be described in detail and related to the simplicial complex $\Delta$. For simplicial complexes $\Delta$ and $\tilde{\Delta}$ on $V$ and $\tilde{V}$, respectively, we study those morphisms $\tilde{V}\rto V$ that induce canonical homomorphisms between the binoids defined by $\Delta$ and $\tilde{\Delta}$ and between their $N\mina$spectra, which we introduce in the third section. Finally, we define the binoid $M\UDelta$ associated to a simplicial complex $\Delta$, whose binoid algebra is the Stanley-Reisner algebra of $\Delta$, and characterize those binoids that yield Stanley-Reisner algebras.

\begin {Convention}
In this chapter, arbitrary vertex sets are assumed to be \emph{finite}.
\end {Convention}

\section{Simplicial complexes} \label{SecSimplCompl}
\markright{\ref {SecSimplCompl} Simplicial complexes}

In this section, we generalize the binoid structures defined on $\Pset(V)$ to arbitrary simplicial complexes. For this we start with recalling basic definitions concerning them.

\begin {Definition}
A (\gesperrt{finite}) \gesperrt{simplicial complex} \index{simplicial complex}$\Delta$ on $V$ is a subset-closed subset of $\Pset(V)$ that contains all singletons $\{v\}$, $v\in V$. We refer to $\Pset(V)$ as the \gesperrt{full} simplicial complex on $V$. The elements of $\Delta$ are called \gesperrt{faces} \index{face}\index{simplicial complex!face of a --}and maximal faces under set inclusion are called \gesperrt{facets}\index{facet}\index{simplicial complex!facet of a --}. The \gesperrt{dimension} \index{dimension!-- of a face}\index{face!dimension of a --}\index{dimension!-- of a simplicial complex}\index{simplicial complex!dimension of a --}of a face $F\in\Delta$ and that of the simplicial complex $\Delta$ are defined by\nomenclature[Dimension4]{$\dim F$}{dimension of the face $F$} \nomenclature[Dimension5]{$\dim\Delta$}{dimension of the simplicial complex $\Delta$}
$$\dim F:=\#F-1\quad\text{and}\quad\dim\Delta:=\max\{\dim F\mid F\in\Delta\}\komma$$
respectively. Elements of $\Pset(V)\setminus\Delta$ are \gesperrt{nonfaces}\index{nonface}\index{simplicial complex!nonface of a --}. A simplicial complex is \gesperrt{pure} \index{simplicial complex!pure --}if all its facets have the same dimension. Faces of dimension $0$ and $1$ are called \gesperrt{vertices} \index{vertice!-- of a simplicial complex}\index{simplicial complex!vertice of a --}and \gesperrt{edges}\index{edge}\index{simplicial complex!edge of a --}, respectively. We denote by $f_{i}(\Delta)$\nomenclature[f1]{$f_{i}(\Delta)$}{$=\#\{F\in\Delta\mid\dim F=i\}$} the number of $i\mina$dimensional faces in $\Delta$. Then the $(d+1)\mina$tuple\nomenclature[f2]{$f(\Delta)$}{$=(f_{-1}(\Delta),f_{0}(\Delta)\kpkt f_{d}(\Delta))$, $f\mina$vector of $\Delta$}
$$f(\Delta)=(f_{-1}(\Delta), f_{0}(\Delta)\kpkt f_{d}(\Delta))$$ is called the \gesperrt{$f\mina$vector} \index{F@$f\mina$vector}\index{simplicial complex!F@$f\mina$vector of a --}of $\Delta$, where $d:=\dim\Delta$.
\end {Definition}

By definition, every simplicial complex on $V$ is nonempty and contains the face $\emptyset$, which has dimension $-1$. With respect to set inclusion, $\{\emptyset,\{v\}\mid v\in V\}$ is the smallest simplicial complex on $V$ and $\Pset(V)$ is the largest, and they coincide if and only if $V$ is the empty set or a singleton. We have $\dim\Pset(\emptyset)=-1$ and $\dim\Pset(\{v\})=0$. In general, $0\le\dim\Delta\le \#V-1$ for every simplicial complex $\Delta$ on $V$ with $\#V\ge 2$, and $\dim\Delta=0$ if and only if $\Delta=\{\emptyset,\{v\}\mid v\in V\}$. Furthermore, $f_{0}(\Delta)=\#V$ and $f_{d}(\Delta)\le\#\{$facets of $\Delta\}$, where $d=\dim\Delta$, and equality holds if and only if $\Delta$ is pure. 

A simplicial complex is uniquely determined by its facets since
$$\Delta\,\,=\!\bigcup_{F\text{ facet of }\Delta}\Pset(F)\pkt$$

The full simplicial complex $\Pset(V)$ on $V$ defines two binoids given by
$$\Pset(V)_{\cap}=(\Pset(V),\cap,V,\emptyset)\quad\text{and}\quad\Pset(V)_{\cup}=(\Pset(V),\cup,\emptyset,V)\komma$$
which have been studied in detail before. Their ideal lattices, filtra, and spectra have been described in Example \ref{ExpPsetIdeals}, Example \ref{ExpFilterPowerset}, and Example \ref{ExSpecPowerset}, respectively. These results can be generalized to arbitrary simplicial complexes $\Delta$ on $V$.

\begin {Definition}
For a simplicial complex $\Delta$ on $V$ denote by\nomenclature[ADelta]{$\Delta_{\cap}$}{semibinoid defined by $\Delta$ with respect to $\cap$} \nomenclature[ADelta]{$\Delta_{\cup}^{\infty}$}{binoid defined by $\Delta$ with respect to $\cup$}
$$\Delta_{\cap}:=(\Delta,\cap,\emptyset)$$ the boolean semibinoid with respect to the intersection of faces, and by
$$\Delta_{\cup}^{\infty}:=(\Delta\cup\{\infty\},\cup,\emptyset,\infty)$$
the boolean binoid with respect to the union of faces, where $F\cup G:=\infty$ if this is not a face of $\Delta$.
\end {Definition}

The definition of $\Delta_{\cap}$ is in line with our notation for the binoid $\Pset(V)_{\cap}$ introduced in Example \ref{ExpBinoidsPowerSet} because $\Delta_{\cap}=\Pset(V)_{\cap}$ if $\Delta=\Pset(V)$, while $\Delta_{\cup}$ is only defined for $\Delta=\Pset(V)$. However, we have the following identifications
$$\Delta_{\cap}\onull\,\,\cong\,\,(\Delta\cup\{V\},\cap,V,\emptyset)\quad\text{and}\quad\Delta_{\cup}^{\infty}\,\,\cong\,\,(\Delta\cup\{V\},\cup,\emptyset,V)\komma$$
which give rise to natural binoid embeddings $\Delta_{\cap}\onull\embto\Pset(V)_{\cap}$ and $\Delta_{\cup}\onull\embto\Pset(V)_{\cup}$ for every simplicial complex $\Delta\subseteq\Pset(V)$. In particular, if $\Delta=\Pset(V)\setminus\{V\}$, then $\Delta_{\cap}\onull\cong\Pset(V)_{\cap}$ and $\Delta_{\cup}^{\infty}\cong\Pset(V)_{\cup}$. Moreover, $\Delta_{\cap}\onull$ and $\Delta_{\cup}^{\infty}$ are isomorphic to $\zero$ if $V=\emptyset$, and to $\trivial$ if $V=\{v\}$. Let $\#V\ge2$. Then $\Delta_{\cap}$ and $\Delta_{\cup}^{\infty}$ are not cancellative since they contain non-trivial idempotent elements. The semibinoid $\Delta_{\cap}$ is not integral and so is the binoid $\Delta_{\cup}^{\infty}$ unless $\Delta\not=\Pset(V)$. In general, $\Delta_{\cap}\onull$ and $\Delta_{\cup}^{\infty}$ are very different binoids. To describe them in terms of generators, we need the following definition.

\begin {Definition}
 Let $M$ be a boolean commutative binoid. We say $M$ is \gesperrt{semifree up to idempotence} \index{semibinoid!semifree up to idempotence --}\index{binoid!semifree up to idempotence --}\index{semifree!-- up to idempotence}if there exists a generating set $\{a_{i}\mid i\in I\}$ such that every element $f\in M\opkt$ can be written uniquely as $f=\sum_{i\in J}a_{i}$ for a finite set $J\subseteq I$. By abuse of notation, cf.\ Definition \ref{DefRepFree}, $(a_{i})_{i\in I}$ is called a \gesperrt{semibasis} \index{semibasis}of $M$ and the set $\{a_{i}\mid i\in J\}$ the \gesperrt{support} \index{element!support of an --}\index{support}of $f$, denoted by $\supp(f)$.
\end {Definition}

\begin {Lemma}
The binoid $\Delta_{\cup}^{\infty}$ is semifree up to idempotence with semibasis  $\{v\}$, $v\in V$. Conversely, if $B$ is a finitely generated boolean binoid that is semifree up to idempotence, then $B\cong\Delta_{\cup}^{\infty}$ for some simplicial complex $\Delta$.
\end {Lemma}
\begin {proof}
Only the converse is not trivial. Let $V\subseteq B$ be a finite generating set that is also a semibasis. Obviously, $\Delta:=\{F\subseteq V\mid\sum_{v\in F}v\not=\infty$ in $B\}$ is a simplicial complex on $V$, and $\Delta_{\cup}^{\infty}\rto B$ with $F\mto \sum_{v\in F}v$ is a surjective binoid homomorphism, which is also injective since the elements of $V$ form a semibasis. Hence, $B\cong\Delta_{\cup}^{\infty}$.
\end {proof}

\begin {Example}\label{ExpBasisDeltaCap} 
\begin {ListeTheorem}
\item []
\item Let $\Delta=\Pset(V)\setminus\{V\}$. A semibasis of $\Delta_{\cap}\onull=\Pset(V)_{\cap}$ is given by the $n$ facets $V\setminus\{v\}$, $v\in V$, and we have $\Pset(V)_{\cup}=\Delta_{\cup}\onull\cong\Delta_{\cap}^{\infty}=\Pset(V)_{\cap}$ via
$$\Pset(V)_{\cup}\stackrel{\sim}{\Rto}\Pset(V)_{\cap}\komma\quad \{v\}\lto V\setminus \{v\}\pkt$$
\item Let $\Delta=\{\emptyset,\{v\}\mid v\in V\}$. A semibasis of $\Delta_{\cap}\onull$ is given by the $n$ facets $\{v\}$, $v\in V$, and we have $\Delta_{\cup}\onull\cong\Delta_{\cap}^{\infty}$ via
$$\Delta_{\cup}^{\infty}\stackrel{\sim}{\Rto}\Delta_{\cap}\onull\komma\quad\{v\}\lto\{v\}\pkt$$
\item Let $V=\{0\kpkt n-1\}$ and let $\Delta$ be the one-dimensional pure simplicial complex defined by the $n$ facets $F_{i}:=\{i,i+1\}$, $i\in\{0\kpkt n-2\}$, and $F_{n-1}:=\{n-1,0\}$. Then $\Delta_{\cap}\onull$ is semifree up to idempotence with a semibasis given by these facets, and we have $\Delta_{\cup}\onull\cong\Delta_{\cap}^{\infty}$ via
$$\Delta_{\cup}^{\infty}\stackrel{\sim}{\Rto}\Delta_{\cap}\onull\komma\quad\{i\}\lto F_{(k+i\mod n)}\pkt$$
\end {ListeTheorem}
In Theorem \ref{ThDeltaCapIsomCup}, we will show that $\Delta_{\cup}^{\infty}$ and $\Delta_{\cap}\onull$ are isomorphic if and only if $\Delta$ is composed of such simplicial complexes.
\end {Example}

\begin {Proposition} \label{PropDeltaBinosSemifree}
Let $\Delta$ be a simplicial complex on $V$. The binoid $\Delta_{\cap}\onull$ always admits a unique minimal generating set $X$. Moreover, if $\Delta_{\cap}\onull$ is semifree up to idempotence, then $X$ is given by the set of facets of $\Delta$.
\end {Proposition}
\begin {proof}
Obviously, the minimal generating set is given by the set consisting of all faces $F\in\Delta$ for which there are no faces $G,G^{\prime}\in\Delta\setminus\{F\}$ with $F=G\cap G^{\prime}$. For the supplement denote this set by $X$. Since every semibasis is a minimal generating set, $X$ has to be a semibasis if $\Delta_{\cap}\onull$ is semifree up to idempotence. However, $X$ cannot be a semibasis if it contains a nonfacet. Indeed, if $G\in X$ is a nonfacet and $F\in\Delta$ a facet with $G\subsetneq F$, then $G=F\cap G$ are two different expressions of $G$ in terms of elements of $X$ since the facets are always contained in $X$.
\end {proof}

\begin {Example}
The converse of the supplement of Proposition \ref{PropDeltaBinosSemifree} need not be true. If $\Delta$ is the simplicial complex on $\{1,2,3,4\}$ with facets $\{1,2\}$, $\{1,3\}$, $\{1,4\}$, $\{2,3\}$ and $\{3,4\}$, then the unique minimal generating set of $\Delta_{\cap}\onull$ is given by the set of facets but it is not a semibasis since, for instance, $\{1,2\}\cap\{1,3\}=\{1,2\}\cap\{1,4\}=\{1,3\}\cap\{1,4\}=\{1\}$.
\end {Example}

\begin {Remark}
In contrast to $\Delta_{\cup}^{\infty}$, the unique minimal generating set of $\Delta_{\cap}\onull$ need not be a semibasis. If $\Delta_{\cap}\onull$ is not semifree up to idempotence, the minimal generating set has to be constructed. For the construction note that every subset $X\subseteq\Delta_{\cap}\onull$ that generates $\Delta_{\cap}\onull$ has to contain the facets and the faces $F\setminus\{v\}$, $v\in F$, $F$ a facet, that cannot be generated by the facets (i.e.\ for which there is no facet $G\in\Delta$ with $G\cap F=F\setminus\{v\}$).

Now the unique minimal generating set is constructed in the following way: let $\{\dim F\mid F\in\Delta$ facet$\}=\{d_{1}\kpkt d_{s}\}$, where we may assume that $\dim\Delta=d_{1}>\cdots>d_{s}\ge 0$. We start with the set $Y_{1}=\{F\in\Delta$ facet$\mid\dim F=d_{1}\}$. Since $\dim(F\cap G)<\min\{\dim F,\dim G\}$ for any two facets $F,G\in\Delta$, the faces $F\setminus\{v\}$, $v\in F$, $F\in Y_{1}$, can only be generated by elements of $Y_{1}$, and those that cannot need to be added to $Y_{1}$. Denote the resulting set by $X_{1}$ and let $Y_{2}:=\{F\in\Delta$ facet$\mid\dim F=d_{2}\}$. The same argument as above shows that any face $F\setminus\{v\}$, $v\in F$, $F\in Y_{2}$, that is not contained in $\Pset(F)$ for $F\in Y_{1}$ can only be generated by the facets $F\in Y_{2}$, otherwise they need to be added to $X_{1}\cup Y_{2}$. Denote the resulting set by $X_{2}$ and proceed in the same manner till $X_{s}$ is attained. By construction, $X_{s}$ is the unique minimal generating set. 
\end {Remark}

Since the unique minimal generating set of $\Delta_{\cap}\onull$ always contains the facets, it obviously may have strictly more than $\#V$ elements, but this may also happen if there are less than $\#V$ facets, which means that $\Delta_{\cap}\onull$ and $\Delta_{\cup}^{\infty}$ cannot be isomorphic. Evidently, the binoids $\Delta_{\cap}\onull$ and $\Delta_{\cup}^{\infty}$ may be very different.

\begin {Example}
Let $\Delta\not=\Pset(V)$ be a simplicial complex on $V$ defined by pairwise disjoint facets $F_{1}\kpkt F_{r}$, $r\ge 2$, and let $n:=\#F_{1}\pluspkt\#F_{r}=\#V$. In this case, the unique minimal generating set of $\Delta_{\cap}\onull$ has $n+r$ elements, namely $F_{k}$ and $F_{k}\setminus\{v\}$ for $v\in F_{k}$, $k\in\{1\kpkt r\}$, and is not a semibasis because $F_{k}\setminus\{v\}=F_{k}\cap F_{k}\setminus\{v\}$. In particular, $\Delta_{\cap}\onull\not\cong\Delta_{\cup}^{\infty}$.
\end {Example}

\begin {Lemma} \label{LemSimplComplConstructions1}
Let  $I$ be finite, $\Delta_{i}$ a simplicial complex on $V_{i}$, $i\in I$, and $V:=\biguplus_{i\in I}V_{i}$ the disjoint union. Then the \gesperrt{disjoint union} \index{simplicial complex!disjoint union of --s}\index{disjoint union!-- of simplicial complexes}\nomenclature[AProductDisjointUnionSC]{$\biguplus_{i\in I}\Delta_{i}$}{disjoint union of a family of simplicial complexes}
$$\biguplus_{i\in I}\Delta_{i}\,\,=\,\,\{(F_{k};k)\mid F_{k}\in\Delta_{k}\setminus\{\emptyset\},k\in I\}\cup\{\emptyset\}\,\,\subseteq\,\,\Big(\biguplus_{i\in I}\Pset(V_{i})\Big)\cup\{\emptyset\}\,\,\subseteq\,\,\Pset(V)$$
of $(\Delta_{i})_{i\in I}$ is a simplicial complex on $V$, where
$$\Big(\biguplus_{i\in I}\Delta_{i}\Big)_{\cap}\onull\,=\,\,\,\bigcupbidot_{i\in I}\Delta_{i,\cap}\onull\quad\text{and}\quad\Big(\biguplus_{i\in I}\Delta_{i}\Big)_{\cup}^{\infty}\,=\,\,\,\bigcupbidot_{i\in I}\Delta_{i,\cup}^{\infty}\pkt$$
Also, the \gesperrt{product} \index{simplicial complex!product of --s}\index{product!-- of simplicial complexes}\nomenclature[AProductProductSC]{$\prod_{i\in I}\Delta_{i}$}{product of a family of simplicial complexes}
$$\prod_{i\in I}\Delta_{i}\,=\,\,\{(F_{i})_{i\in I}\mid F_{i}\in\Delta_{i},i\in I\}\,\,\subseteq\,\,\prod_{i\in I}\Pset(V_{i})\cong\Pset(V)$$
of $(\Delta_{i})_{i\in I}$ is a simplicial complex on $V$, where
$$\Big(\prod_{i\in I}\Delta_{i}\Big)_{\cap}\onull\,=\,\prod_{i\in I}\Delta_{i,\cap}\onull\quad\text{and}\quad\Big(\prod_{i\in I}\Delta_{i}\Big)_{\cup}^{\infty}\,=\,\bigwedge_{i\in I}\Delta_{i,\cup}^{\infty}\pkt$$
\end {Lemma}
\begin {proof}
This is easily verified.
\end {proof}

Recall that a \gesperrt{partition} \index{partition!-- of a set}of a set $V$ is a family $(V_{i})_{i\in I}$ of subsets such that $V=\bigcup_{i\in I}V_{i}$ and $V_{i}\cap V_{j}=\emptyset$ for $i\not=j$. Every simplicial complex $\Delta$ on $V$ admits a partition $(V_{i})_{i\in I}$ of $V$ such that $\Delta=\biguplus_{i\in I}\Delta_{i}$, where $\Delta_{i}$ is a simplicial complex on $V_{i}$, $i\in I$. Of course this partition need not be unique or might be trivial; that is when $\#I=1$. However, there exists always a (unique) \emph{finest} partition.

\begin {Definition}
A simplicial complex $\Delta$ on $V$ that admits no non-trivial partition of $V$ is said to be \gesperrt{connected}\index{simplicial complex!connected --}.
\end {Definition}

\begin {Proposition}\label{PropConnectedSimplCompl}
Let $\Delta$ be a simplicial complex on $V$.
\begin{ListeTheorem}
\item $\Delta$ is connected if and only if for any two vertices $v,w\in V$ there are facets $F_{1}\kpkt F_{s}\in\Delta$ with $v\in F_{1}$ and $w\in F_{s}$ such that $F_{i}\cap F_{i+1}\not=\emptyset$ for all $i\in\{1\kpkt s-1\}$.
\item There exists a unique partition $(V_{i})_{i\in I}$ of $V$ such that $\Delta=\biguplus_{i\in I}\Delta_{i}$, where $\Delta_{i}$ is a connected simplicial complex on $V_{i}$, $i\in I$.
\end{ListeTheorem}
\end {Proposition}
\begin {proof}
Let $A\subseteq\Delta$ be the subset of facets of $\Delta$ and consider the equivalence relation on $V$ given by 
$$v\sim w\,\,:\eq\,\,\exists F_{1}\kpkt F_{s}\in A\text{ such that }v\in F_{1}, w\in F_{s}\text{ and }F_{i}\cap F_{i+1}\not=\emptyset, \forall i\in\{1\kpkt s-1\}\pkt$$
Denote by $V_{1}\kpkt V_{l}$, $l\ge1$, the equivalence classes. Now (1) translates to:  $\Delta$ is connected if and only if $l=1$. Clearly, if $\Delta$ is not connected, then $l\ge 2$. Conversely, if $l\ge 2$ then $(V_{i})_{i\in I}$, $I=\{1\kpkt l\}$, is a non-trivial partition of $V$ such that $\Delta=\biguplus_{i\in I}\Delta_{i}$, where $\Delta_{i}$ is the simplicial complex $\{F\in\Delta\mid F\subseteq V_{i}\}$ on $V_{i}$, $i\in I$. This also proves (2) since the $\Delta_{i}$ are connected by (1). 
\end {proof}

\begin{Example}
The full simplicial complex $\Pset(V)$ and the simplicial complex of Example \ref{ExpBasisDeltaCap} (3) are connected, whereas $\Delta=\{\emptyset,\{v\}\mid v\in V\}$ is not connected for $\#V\ge2$ since $\Delta=\biguplus_{v\in V}\Pset(\{v\})$.
\end{Example}

\begin {Theorem} \label{ThDeltaCapIsomCup}
Let $\Delta$ be a simplicial complex on $V$. The binoid $\Delta_{\cap}\onull$ is semifree up to idempotence if and only if $\Delta$ decomposes into connected simplicial complexes of the form
$$\Delta^{(1)}:=\Pset(W)\setminus\{W\}\,\,\text{ or }\,\,\Delta^{(2)}:=\Pset(\{v\})\,\,\text{ or }\,\,\Delta^{(3)}:=\{\emptyset,\{i\},\{i,i+1\mod n\}\mid i\in\{0\kpkt n-1\}\}\pkt$$
In this case (and only in this case), $\Delta_{\cup}^{\infty}$ and $\Delta_{\cap}\onull$ are isomorphic.
\end {Theorem}
\begin{proof}
Of course, $\Delta_{\cap}\onull$ is semifree up to idempotence if it decomposes in such simplicial complexes because they are semifree up to idempotence by Example \ref{ExpBasisDeltaCap} and then the union of their semibases yields a semibasis of $\Delta_{\cap}\onull$. Moreover, if $V=\biguplus_{i}V_{i}$ and $\Delta=\biguplus_{i\in I}\Delta_{i}$, where $\Delta_{i}$ is one of these three simplicial complexes on $V_{i}$, $i\in I$, the isomorphism
$$\Delta_{\cup}^{\infty}\,=\,\,\,\bigcupbidot_{i\in I}\Delta_{i,\cup}^{\infty}\stackrel{\sim}{\Rto}\,\,\bigcupbidot_{i\in I}\Delta_{i,\cap}\onull\,=\,\Delta_{\cap}\onull$$
is given componentwise $\Delta_{i,\cup}^{\infty}\rto\Delta_{i,\cap}\onull$ as in Example \ref{ExpBasisDeltaCap}. For the converse, it suffices to show that if $\Delta_{\cap}\onull$ is connected and semifree up to idempotence, then $\Delta$ is of the form $\Delta^{(r)}$ for $r=1,2,$ or $3$. If $\#V\le 2$, there is nothing to show. So let $\#V\ge 3$. Note that every zero-dimensional simplicial complex is of the form $\Delta^{(2)}$. Furthermore, every one-dimensional connected simplicial complex that is semifree up to idempotence is of the form $\Delta^{(3)}$. Indeed, in such a simplicial complex every facet need to have dimension 1 by the connectedness, and every vertex has to be contained in precisely two facets because the simplicial complex is semifree up to idempotence. Thus, we may assume that $\dim\Delta=:d\ge2$. We need to show that $\Delta=\Pset(V)\setminus\{V\}$. Denote by $A$ the set of facets of $\Delta$. By Proposition \ref{PropDeltaBinosSemifree}, $A$ is a semibasis of $\Delta_{\cap}\onull$. In particular, given a facet $F$ of dimension $d$, say $F=\{v_{i}\mid i\in I\}$, $I=\{1\kpkt d+1\}$, there are facets $F_{1}\kpkt F_{d+1}\in A$ with (after a suitable renumbering) $F\cap F_{i}=F\setminus\{v_{i}\}$ and $\dim F_{i}=d$, $i\in I$. Hence, $F_{i}=F\setminus\{v_{i}\}\cup\{w_{i}\}$ for some $w_{i}\in V\setminus F$, $i\in I$. For for $i\not=j$ this gives
$$F_{i}\cap F_{j}\,=\,\begin {cases}
F\setminus\{v_{i},v_{j}\}&\text{, if }w_{i}\not=w_{j}\komma\\
(F\setminus\{v_{i},v_{j}\}\uplus\{w\}&\text{, if }w_{i}=w_{j}=:w\pkt
\end {cases}$$
In the first case, one has $F\cap F_{i}\cap F_{j}=F_{i}\cap F_{j}$, which is a contradiction to $A$ being a semibasis. Hence, $w_{i}=w$ for all $i\in I$ and therefore $\Pset(W)\setminus\{W\}\subseteq\Delta$ with $W:=\{v_{1}\kpkt v_{d+1},w\}\subseteq V$. To show that $W=V$ suppose that $v\in V\setminus W$. Since $\Delta$ is connected we find by Proposition \ref{PropConnectedSimplCompl} for any $w\in W$ facets $G_{1}\kpkt G_{s}\in A$ with $w\in G_{1}$, $v\in G_{s}$, and $G_{j}\cap G_{j+1}\not=\emptyset$, $j\in\{1\kpkt s-1\}$. We may assume that $G_{1}\subseteq W$; otherwise extend the sequence $G_{1}\kpkt G_{s}$ by a facet $G_{0}\subseteq W$ with $w\in G_{0}$. Then there is an $r\in\{1\kpkt s-1\}$ such that $G_{r}\subseteq W$ but $G_{r+1}\not\subseteq W$. This is again a contradiction to $A$ being a semibasis because $\emptyset\not=G_{r}\cap G_{r+1}\subseteq W$ is also the intersection of certain $F_{i}$ from above. Thus, $V=W$ and therefore $\Delta=\Pset(V)\setminus\{V\}$.
\end{proof}

\bigskip

\section{Simplicial morphisms} \label{SecSimplMorph}
\markright{\ref {SecSimplMorph} Simplicial morphisms}

In this section, we consider morphisms $\tilde{V}\rto V$ between vertex sets of simplicial complexes $\tilde{\Delta}$ and $\Delta$ and investigate when such morphisms induce homomorphisms between the binoids defined by these simplicial complexes. This yields the definition of simplicial, $\alpha$-simplicial, and $\beta$-simplicial morphisms for which we will give various examples that will appear later on.

\begin {Lemma} \label{LemCharacterizationCosimpMor}
Let $\Delta$ and $\tilde{\Delta}$ be simplicial complexes on $V$ and $\tilde{V}$, respectively, and $\lambda:\tilde{V}\rto V$ a map. The following statements are equivalent.
\begin {ListeTheorem}
\item Every image of a nonface is a nonface; that is, $F\not\in\tilde{\Delta}$ implies $\lambda(F)\not\in\Delta$.
\item Every preimage of a face is a face; that is, $F\in\Delta$ implies $\lambda^{-1}(F)\in\tilde{\Delta}$.
\end {ListeTheorem}
\end {Lemma}
\begin {proof}
$(1)\Rarrow(2)$ Let $F\in\Delta$. Since $\lambda(\lambda^{-1}(F))\subseteq F$ the image $\lambda(\lambda^{-1}(F))$ is also a face of $\Delta$. Hence, $\lambda^{-1}(F)\in\tilde{\Delta}$ since otherwise $\lambda(\lambda^{-1}(F))\not\in\Delta$ by assumption on $\lambda$. $(2)\Rarrow(1)$ Let $F\not\in\tilde{\Delta}$. Since $F\subseteq\lambda^{-1}(\lambda(F))$ the preimage $\lambda^{-1}(\lambda(F))$ is not a face of $\tilde{\Delta}$. Hence, $\lambda(F)\not\in\Delta$ since otherwise $\lambda^{-1}(\lambda(F))\in\tilde{\Delta}$ by assumption on $\lambda$.
\end {proof}

\begin {Definition}
Let $\Delta$ and $\tilde{\Delta}$ be simplicial complexes on $V$ and $\tilde{V}$, respectively, and $\lambda:\tilde{V}\rto V$ a map.
\begin {ListeTheorem}
\item $\lambda$ is a \gesperrt{simplicial} morphism \index{simplicial morphism}if every image of a face is a face; that is, $F\in\tilde{\Delta}$ implies $\lambda(F)\in{\Delta}$.
\item $\lambda$ is an \gesperrt{$\alpha$-simplicial} morphism \index{simplicial morphism!A@$\alpha\mina$ --}if $\lambda$ satisfies one of the equivalent conditions of Lemma \ref{LemCharacterizationCosimpMor}.
\item $\lambda$ is a \gesperrt{$\beta$-simplicial} morphism \index{simplicial morphism!B@$\beta\mina$ --}if every preimage of a nonface is a nonface; that is, $F\not\in\Delta$ implies $\lambda^{-1}(F)\not\in\tilde{\Delta}$.
\end {ListeTheorem}
Since these properties depend on the chosen simplicial complexes, we will write
$$\lambda:(\tilde{V},\tilde{\Delta})\Rto (V,\Delta)\pkt$$
The category of simplicial complexes together with simplicial, $\alpha$-simplicial and $\beta$-simplicial morphisms will be denoted by $\SCsf$, $\SCsf^{\alpha}$, and $\SCsf^{\beta}$, respectively.\nomenclature[SC]{$\SCsf$}{category of simplicial complexes together with simplicial morphisms}\nomenclature[SC]{$\SCsf^{\alpha}$}{category of simplicial complexes together with $\alpha$-simplicial morphisms}\nomenclature[SC]{$\SCsf^{\beta}$}{category of simplicial complexes together with $\beta$-simplicial morphisms}
\end {Definition}

Our definition of a simplicial morphism is consistent with the terminology in topology of abstract simplicial complexes, cf.\ \cite[Chapter II.2]{FerrarioPiccinini}.

\begin {Example}\label{ExpMorOnlyOne}
Here are examples of morphisms that satisfy only one property, simplicial or $\alpha$-simplicial or $\beta$-simplicial, but none of the other both. For this let $V=\{1\kpkt n\}$, $n\ge3$.
\begin {ListeTheorem}
\item Let $T=V\setminus\{n\}$ and $F=\{n-1,n\}$. The map
$$\lambda:(T,\Pset(T))\Rto(V,\Pset(V)\setminus\{V,F\})\quad\text{with}\quad v\lto v$$
is simplicial because $\Pset(T)\subseteq\Pset(V)\setminus\{V,F\}$. Since $\lambda^{-1}(F)=\{n-1\}\in\Pset(T)\setminus\{T\}$ but $F\not\in\Pset(V)\setminus\{V,F\}$ it is not $\alpha$-simplicial, and it is not $\beta$-simplicial because $\lambda^{-1}(V)=T\in\Pset(T)$ but $V\not\in\Pset(V)\setminus\{V,F\}$.
\item Let $\Delta=\{\emptyset,\{v\}\mid v\in V\setminus\{n\}\}$. The map 
$$\lambda:(V,\Pset(V)\setminus\{V\})\Rto(V\setminus\{n\},\Delta)\quad\text{with}\quad v\lto\begin{cases}
v&\text{, if }v\not=n\komma\\
n-1&\text{, otherwise,}
\end{cases}$$
is $\alpha$-simplicial because $V$ is the only nonface of $\Pset(V)\setminus\{V\}$ and $\lambda(V)=V\not\in\Delta$. On the other hand, it is not simplicial because  $\lambda(\{1,2\})=\{1,2\}\not\in\Delta$ but $\{1,2\}\in\Pset(V)\setminus\{V\}$, and since $\lambda^{-1}(\{1,2\})=\{1,2\}\in\Pset(V)\setminus\{V\}$ but $\{1,2\}\not\in\Delta$ it is not $\beta$-simplicial.
\item Let $n=3$, $\Delta=\Pset(V)\setminus\{V,\{1,3\}\}$, and $\tilde{\Delta}=\{\emptyset,\{1\},\{2\}\}$. The map
$$\lambda:(V,\Delta)\Rto(\{1,2\},\tilde{\Delta})\quad\text{with}\quad v\lto\begin {cases}
1\text{, if }v\in\{1,3\}\komma\\
2\text{, otherwise,}
\end {cases}$$
is $\beta$-simplicial because the only nonface in $\tilde{\Delta}$ is $\{1,2\}$ and $\lambda^{-1}(\{1,2\})=V\not\in\Delta$. On the other hand, it is not simplicial because $\lambda(\{1,2\})=\{1,2\}\not\in\tilde{\Delta}$ but $\{1,2\}\in\Delta$, and it is not $\alpha$-simplicial because $\lambda^{-1}(\{1\})=\{1,3\}\not\in\Delta$ but $\{1\}\in\tilde{\Delta}$.
\end {ListeTheorem}
\end {Example}

\begin {Example}\label{ExpMorTwoOneNot}
Here are examples of morphisms that satisfy all except one property.
Let $V=\{1\kpkt n\}$ with $n\ge2$. 
\begin {ListeTheorem}
\item Let $\Delta$ and $\tilde{\Delta}$ be a simplicial complexes on $V$ such that  $\tilde{\Delta}\subsetneq\Delta$. Then
$$\id:(V,\tilde{\Delta})\Rto(V,\Delta)$$
is simplicial since $\tilde{\Delta}\subseteq\Delta$. It is $\beta$-simplicial since nonfaces of $\Delta$ are nonfaces of $\tilde{\Delta}$, but it is not $\alpha$-simplicial because there is by assumption a face $F\in\Delta$ with  $F\not\in\tilde{\Delta}$, hence $\lambda^{-1}(F)=F\not\in\tilde{\Delta}$.
\item Let $T\subsetneq V$. The morphism 
$$\lambda:(T,\Pset(T))\Rto(V,\Pset(V)\setminus\{V\})\quad\text{with}\quad v\lto v$$
is simplicial since $\Pset(T)\subseteq\Pset(V)$, and it is $\alpha$-simplicial since $\Pset(T)$ has no nonfaces. However, $\lambda$ is not $\beta$-simplicial because $\lambda^{-1}(V)=T\in\Pset(T)$ but $V\not\in\Pset(V)\setminus\{V\}$.
\item Let $T=V\setminus\{n\}$. The morphism
$$\lambda:(V,\Pset(V)\setminus\{V\})\Rto(T,\Pset(T)\setminus\{T\})\quad\text{with}\quad v\lto\begin {cases}
v&\text{, if }v\not=n\komma\\
n-1&\text{, otherwise,}
\end{cases}$$
is $\alpha$-simplicial since $V$ is the only nonface of $\Pset(V)\setminus\{V\}$ and $\lambda(V)=T\not\in\Pset(T)\setminus\{T\}$, and since $T$ is the only nonface of $\Pset(T)\setminus\{T\}$ and $\lambda^{-1}(T)=V\not\in\Pset(V)\setminus\{V\}$ it is also $\beta$-simplicial. However, $\lambda$ is not simplicial because $\lambda(T)=T\not\in\Pset(T)\setminus\{T\}$ but $T\in\Pset(V)\setminus\{V\}$.
\end {ListeTheorem}
\end {Example}

\begin {Lemma} \label{LemSimplBetaSimpl}
Let $\Delta$ and $\tilde{\Delta}$ be simplicial complexes on $V$ and $\tilde{V}$, respectively.
\begin {ListeTheorem}
\item If $\lambda:(\tilde{V},\tilde{\Delta})\rto(V,\Delta)$ is surjective and simplicial, then it is $\beta$-simplicial.
\item If $\lambda:(\tilde{V},\tilde{\Delta})\rto(V,\Delta)$ is injective and $\beta$-simplicial, then it is simplicial.
\end {ListeTheorem}
\end {Lemma}
\begin {proof}
(1) If $\lambda$ is surjective, we have $F=\lambda(\lambda^{-1}(F))$ for all $F\subseteq V$. Thus, if $\lambda$ is simplicial and $\lambda^{-1}(F)\in\tilde{\Delta}$, then $\lambda(\lambda^{-1}(F))\in\Delta$ and therefore $F\in\Delta$. (2) If $\lambda$ is injective we have $F=\lambda^{-1}(\lambda(F))$ for all $F\subseteq\tilde{V}$.  Thus, if $\lambda$ is $\beta$-simplicial and $\lambda(F)\not\in\Delta$, then $F=\lambda^{-1}(\lambda(F))\not\in\tilde{\Delta}$.
\end {proof}

\begin {Example}\label{ExpSimplBetasimpl}
Let $V=\{1\kpkt n-1\}$, $n\ge2$. Consider $\Delta=\Pset(F_{0})\cup\Pset(F_{1})$, where $F_{0}=\{1\kpkt k\}$ and $F_{1}=\{s\kpkt n-1\}$ for some $s\in\{2\kpkt k+1\}$, and $\tilde{\Delta}=\Pset(F_{0})\cup\Pset(F_{1}\cup\{n\})$. The surjection
$$\lambda:(V\cup\{n\},\tilde{\Delta})\Rto(V,\Delta\})\quad\text{with}\quad v\lto\begin {cases}
v&\text{, if }v\not=n\komma\\
n-1&\text{, otherwise,}
\end{cases}$$
is simplicial and $\beta$-simplicial but not injective. On the other hand, for $\Delta_{0}=\{\emptyset,\{v\}\mid v\in V\}$ and $\tilde{\Delta}_{0}=\{\emptyset,\{v\},\{v,n\}\mid v\in V\cup\{n\}\}$ the injective map 
$$\lambda:(V,\Delta_{0})\Rto(V\cup\{n\},\tilde{\Delta}_{0})\quad\text{with}\quad v\lto v$$
is simplicial and $\beta$-simplicial but not surjective. Note that both morphisms are $\alpha$-simplicial.
\end {Example}

\begin {Lemma} \label{LemIndHomUnion}
Let $\Delta$ and $\tilde{\Delta}$ be simplicial complexes on $V$ and $\tilde{V}$, respectively.
\begin{ListeTheorem}
\item If $\lambda:(\tilde{V},\tilde{\Delta})\Rto (V,\Delta)$ is simplicial, then\nomenclature[Aphi1]{$\varphi_{\lambda}$}{binoid homomorphism induced by a simplicial morphism $\lambda$}
$$\varphi_{\lambda}:\Delta_{\cup}^{\infty}\Rto\tilde{\Delta}_{\cup}^{\infty}\quad\text{with}\quad F\lto\begin {cases}
\emptyset&\text{, if }F=\emptyset,\\
\infty&\text{, if }F=\infty\text{ or }\lambda^{-1}(v)=\emptyset\text{ for some }v\in F\komma\\
\lambda^{-1}(F)&\text{, otherwise,}
\end {cases}$$
is a binoid homomorphism, which is surjective if $\lambda$ is injective.
\item $\lambda:(\tilde{V},\tilde{\Delta})\Rto (V,\Delta)$ is $\alpha$-simplicial if and only if\nomenclature[Aphi2]{$\varphi_{\lambda}^{\alpha}$}{binoid homomorphism induced by an $\alpha$-simplicial morphism $\lambda$} 
$$\varphi_{\lambda}^{\alpha}:\tilde{\Delta}^{\infty}_{\cup}\Rto\Delta^{\infty}_{\cup}\quad\text{with}\quad F\lto\begin {cases}
\infty&\text{, if }F=\infty\text{ or }\lambda(F)\not\in\Delta\komma\\
\lambda(F)&\text{, otherwise,}
\end {cases}$$
is a binoid homomorphism, which is surjective if $\lambda$ is so. If $\lambda$ is $\alpha$-simplicial, simplicial, and injective, then $\varphi_{\lambda}^{\alpha}$ is also injective.
\item $\lambda:(\tilde{V},\tilde{\Delta})\Rto (V,\Delta)$ is $\beta$-simplicial if and only if\nomenclature[Aphi3]{$\varphi_{\lambda}^{\beta}$}{binoid homomorphism induced by a $\beta$-simplicial morphism $\lambda$} 
$$\varphi_{\lambda}^{\beta}:\Delta^{\infty}_{\cup}\Rto\tilde{\Delta}^{\infty}_{\cup}\quad\text{with}\quad F\lto\begin {cases}
\infty&\text{, if }F=\infty\text{ or }\lambda^{-1}(F)\not\in\tilde{\Delta}\komma\\
\lambda^{-1}(F)&\text{, otherwise,}
\end {cases}$$
is a binoid homomorphism, which is surjective if $\lambda$ is bijective.
\end{ListeTheorem}
\end {Lemma}
\begin {proof}
(1) To verify that $\varphi_{\lambda}$ is well-defined, we need to check that $\varphi_{\lambda}(F\cup G)=\infty$ (or equivalently $\lambda^{-1}(F\cup G)\not\in\tilde{\Delta}$), whenever $F\cup G\not\in\Delta$ for $F,G\in\Delta$ and $\lambda^{-1}(v)\not=\emptyset$ for all $v\in F\cup G$. Indeed, the latter assumption yields $\lambda(\lambda^{-1}(F\cup G))=F\cup G$. Hence, $\lambda^{-1}(F\cup G)\not\in\tilde{\Delta}$ since otherwise, we had $F\cup G\in\Delta$ because $\lambda$ is simplicial. Moreover, $\varphi_{\lambda}$ is a binoid homomorphism because $\lambda^{-1}(F\cup G)=\lambda^{-1}(F)\cup\lambda^{-1}(G)$ for all $F,G\in\Delta$, and $\varphi_{\lambda}(\infty)=\infty$ and $\varphi_{\lambda}(\emptyset)=\emptyset$ by definition. If $\lambda$ is injective, then $G=\lambda^{-1}(\lambda(G))$ for all $G\subseteq\tilde{V}$. In particular, if $G\in\tilde{\Delta}$, then $\lambda(G)\in\Delta$ since $\lambda$ is simplicial, and hence $\varphi_{\lambda}(\lambda(G))=\lambda^{-1}(\lambda(G))=G$, which shows that $\varphi_{\lambda}$ is surjective if $\lambda$ is injective.

(2) Clearly, $\varphi_{\lambda}^{\alpha}$ is a well-defined binoid homomorphism if $\lambda$ is $\alpha$-simplicial. For the converse, we show that every image of a nonface is a nonface. So let $F\not\in\tilde{\Delta}$. This means $F=\infty$ in $\tilde{\Delta}_{\cup}^{\infty}$, and hence 
$$\lambda(F)\,=\,\bigcup_{v\in F}\{\lambda(v)\}\,\,=\,\,\bigcup_{v\in F}\{\varphi_{\lambda}^{\alpha}(\{v\})\}\,=\,\varphi_{\lambda}^{\alpha}(F)\,=\,\varphi_{\lambda}^{\alpha}(\infty)=\infty$$
in $\Delta^{\infty}_{\cup}$ because $\varphi_{\lambda}^{\alpha}$ is a binoid homomorphism. Thus, $\lambda(F)\not\in\Delta$. If $\lambda$ is surjective and $F\in\Delta_{\cup}^{\infty}$, then $\varphi_{\lambda}^{\alpha}(\lambda^{-1}(F))=\lambda(\lambda^{-1}(F))=F$, which shows that $\varphi_{\lambda}^{\alpha}$ is also surjective. For the supplement note that since $\lambda$ is simplicial and $\alpha$-simplicial, we have $\varphi_{\lambda}^{\alpha}(F)=\infty$ in $\Delta_{\cup}^{\infty}$ if and only if $F=\infty$ in $\tilde{\Delta}_{\cup}^{\infty}$. So we only need to consider $F,G\in\tilde{\Delta}$ with $\varphi_{\lambda}^{\alpha}(F)=\varphi_{\lambda}^{\alpha}(G)$ for the injectivity. Then $\lambda(F)=\lambda(G)$, which implies $F=G$ by the injectivity of $\lambda$. Thus, $\varphi_{\lambda}^{\alpha}(F)=$ is also injective.

(3) Clearly, $\varphi_{\lambda}^{\beta}$ is a well-defined binoid homomorphism if $\lambda$ is $\beta$-simplicial. Conversely, if $\varphi_{\lambda}^{\beta}$ is a binoid homomorphism and $F\not\in\Delta$, then $F=\infty$ in $\Delta_{\cup}^{\infty}$, and hence $\varphi_{\lambda}^{\beta}(F)=\varphi_{\lambda}^{\beta}(\infty)=\infty$. This shows that $\lambda^{-1}(F)\not\in\tilde{\Delta}$, since otherwise $\varphi_{\lambda}^{\beta}(F)=\lambda^{-1}(F)\not=\infty$ in $\tilde{\Delta}_{\cup}^{\infty}$. If $\lambda$ is bijective, then $G=\lambda^{-1}(\lambda(G))$ for every $G\in\tilde{\Delta}$. This proves the surjectivity of $\varphi_{\lambda}^{\beta}$ because $\lambda(G)\in\Delta$ by Lemma \ref{LemSimplBetaSimpl}(2).
\end {proof}

\begin {Remark}
\begin {ListeTheorem}
\item[]
\item For a simplicial morphism $\lambda$ that is surjective but not injective $\varphi_{\lambda}$ need not be surjective as Example \ref{ExpSimplBetasimpl} shows. Here $F_{1}$ has no preimage under $\varphi_{\lambda}:\Delta_{\cup}^{\infty}\rto\tilde{\Delta}_{\cup}^{\infty}$ because every $G\in\Delta$ with $F_{1}=\varphi_{\lambda}(G)=\lambda^{-1}(G)$ fulfills $n-1\in G$, but then $n\in\lambda^{-1}(G)$, which is not contained in $F_{1}$. Hence, $F_{1}\subsetneq\lambda^{-1}(G)$.
\item Example \ref{ExpMorTwoOneNot}(2) shows that $\varphi_{\lambda}^{\beta}$ 
is usually no binoid homomorphism, whether $\lambda:\tilde{V}\rto V$ is simplicial or $\alpha$-simplicial.
\item If $\lambda$ is $\beta$-simplicial, then $\varphi_{\lambda}^{\beta}$ need not be surjective if $\lambda$ is surjective but not injective. An example is given by the $\beta$-simplicial morphism $\lambda:(V,\Delta)\rto(\{1,2\},\tilde{\Delta})$ of Example \ref{ExpMorOnlyOne}(3). Here the facet $\{1,2\}\in\Delta$ has no preimage under $\varphi_{\lambda}^{\beta}:\tilde{\Delta}_{\cup}^{\infty}\rto\Delta_{\cup}^{\infty}$.
\end {ListeTheorem}
\end {Remark}

\begin {Lemma}\label{LemIndHomIntersection}
Let $\Delta$ and $\tilde{\Delta}$ be simplicial complexes on $V$ and $\tilde{V}$, respectively, and $\lambda:\tilde{V}\rto V$ a map.
\begin {ListeTheorem}
\item If $\lambda:(\tilde{V},\tilde{\Delta})\rto(V,\Delta)$ is simplicial and injective, then\nomenclature[Apsi1]{$\psi_{\lambda}$}{binoid homomorphism induced by a simplicial morphism $\lambda$}  
$$\psi_{\lambda}:\tilde{\Delta}_{\cap}\onull\Rto\Delta_{\cap}\onull\quad\text{with}\quad F\lto\begin {cases}
0&\text{, if }F=0\komma\\
\lambda(F)&\text{, otherwise,}\end {cases}$$
is a binoid embedding.
\item If $\lambda:(\tilde{V},\tilde{\Delta})\rto(V,\Delta)$ is $\alpha$-simplicial, then\nomenclature[Apsi2]{$\psi_{\lambda}^{\alpha}$}{binoid homomorphism induced by an $\alpha$-simplicial morphism $\lambda$} 
$$\psi_{\lambda}^{\alpha}:\Delta_{\cap}\onull\Rto\tilde{\Delta}_{\cap}\onull\quad\text{with}\quad F\lto\begin {cases}
0&\text{, if }F=0\komma\\
\lambda^{-1}(F)&\text{, otherwise,}\end {cases}$$
is a binoid homomorphism.
\end {ListeTheorem}
\end {Lemma}
\begin {proof}
For (1) note that $\lambda(F\cap G)=\lambda(F)\cap\lambda(G)$ for all $F,G\subseteq\tilde{V}$ if and only if $\lambda$ is injective. Therefore, the injectivity is necessary. The rest of the lemma follows from the definitions of simplicial and $\alpha$-simplicial morphisms.
\end {proof}

\begin {Remark}
For $\lambda$ being only injective, the map $\psi:\tilde{\Delta}_{\cap}\onull\rto\Delta_{\cap}\onull$ with 
$$F\lto\begin {cases}
0&\text{, if }F=0\text{ or }\lambda(F)\not\in\Delta\komma\\
\lambda(F)&\text{, otherwise,}\end {cases}$$
is no binoid homomorphism if there is an $F\in\tilde{\Delta}$ with $F\not=\tilde{V}$ and $\lambda(F)\not\in\Delta$ because for $v\in \tilde{V}\setminus F$ one has $\lambda(v)\in\Delta$, and hence
$$\psi(F\cap\{v\})\,=\,\psi(\emptyset)\,=\,\emptyset\quad\text{but}\quad\psi(F)\cap\psi(\{v\})\,=\,\psi(\{v\})\,=\,\lambda(v)\pkt$$
Similarly, the map $\psi:\Delta_{\cap}\onull\rto\tilde{\Delta}_{\cap}\onull$ with 
$$F\lto\begin {cases}
0&\text{, if }F=0\text{ or }\lambda^{-1}(F)\not\in\tilde{\Delta}\komma\\
\lambda^{-1}(F)&\text{, otherwise,}\end {cases}$$
is no binoid homomorphism if $\lambda$ is not $\alpha$-simplicial. Take, for instance, $\Delta=\Pset(V)$, $\tilde{\Delta}=\{\emptyset,\{v\}\mid v\in V\}$, and $\lambda=\id_{V}$. For pairwise different elements $u,v,w\in V$ one has
$$\psi(\{u,w\}\cap\{v,w\})\,=\,\psi(\{w\})\,=\,\{w\}\quad\text{but}\quad\psi(\{u,w\})\cap\psi(\{v,w\})\,=\,0\cap0\,=\,0\pkt$$
\end {Remark}

We close this section by applying Lemma \ref{LemIndHomUnion} and Lemma \ref{LemIndHomIntersection} to three important examples that will appear later on.

\begin {Example}\label{ExpSubcomplexSameV}
Let $\Delta$ and $\tilde{\Delta}$ be simplicial complexes on $V$ such that $\tilde{\Delta}\subseteq\Delta$. The morphism $$\id:(V,\Delta)\Rto(V,\tilde{\Delta})$$
is $\alpha$-simplicial (and not simplicial or $\beta$-simplicial if $\tilde{\Delta}\subsetneq\Delta$). We thus get the binoid epimorphism
$$\varphi_{\id}^{\alpha}:\Delta_{\cup}^{\infty}\Rto\tilde{\Delta}_{\cup}^{\infty}\quad\text{with}\quad F\lto\begin {cases}
\infty&\text{, if }F\not\in\tilde{\Delta}\komma\\
F&\text{, otherwise,}
\end {cases}$$
and the binoid embedding 
$$\psi_{\id}^{\alpha}:\tilde{\Delta}_{\cap}\onull\Rto\Delta_{\cap}\onull\quad\text{with}\quad v\lto v\pkt$$
Both binoid homomorphisms can also be deduced form the identity
$$\widetilde{\id}:(V,\tilde{\Delta})\Rto(V,\Delta)\komma$$
which is simplicial and $\beta$-simplicial (but not $\alpha$-simplicial if $\tilde{\Delta}\subsetneq\Delta$), cf.\ Example \ref{ExpMorTwoOneNot}(1). Note that here $\varphi_{\widetilde{\id}}^{\beta}$ and $\varphi_{\widetilde{\id}}$ coincide with $\varphi_{\id}^{\alpha}$.
\end {Example}

\begin {Example} \label{ExpSubcomplexDifferntV}
Let $\Delta$ be a simplicial complex on $V$. If $T\subseteq V$ and $\Delta(T)=\{F\in\Delta\mid F\subseteq T\}$, then
$$\iota:(T,\Delta(T))\Rto(V,\Delta)\quad\text{with}\quad v\lto v$$
is simplicial and $\alpha$-simplicial (and not $\beta$-simplicial if $T\subsetneq V$, see below). We thus get the binoid epimorphism 
$$\varphi_{\iota}:\Delta_{\cup}^{\infty}\Rto\Delta(T)_{\cup}^{\infty}\quad\text{with}\quad F\lto\begin {cases}
F&\text{, if }F\in\Delta(T)\komma\\
\infty&\text{, otherwise,}
\end {cases}$$
and the binoid embedding
$$\quad\varphi_{\iota}^{\alpha}:\Delta(T)_{\cup}^{\infty}\Rto\Delta_{\cup}^{\infty}\quad\text{with}\quad F\lto F\pkt$$
In particular, we have 
$$\Delta(T)_{\cup}^{\infty}\stackrel{\!\!\!\varphi_{\iota}^{\alpha}}{\Rto}\Delta_{\cup}^{\infty}\stackrel{\varphi_{\iota}}{\Rto}\Delta(T)_{\cup}^{\infty}$$
with $\varphi_{\iota}\varphi_{\iota}^{\alpha}=\id$. Moreover, we have the binoid embedding
$$\psi_{\iota}:\Delta(T)_{\cap}\onull\Rto\Delta_{\cap}\onull\quad\text{with}\quad F\lto F$$
and the binoid epimorphism
$$\psi_{\iota}^{\alpha}:\Delta_{\cap}\onull\Rto\Delta(T)_{\cap}\onull\quad\text{with}\quad F\lto\begin {cases}
0&\text{, if }F=0\komma\\
F\cap T&\text{, otherwise.}
\end {cases}$$
If $T\subsetneq V$, then $\iota$ is not $\beta$-simplicial unless $\Delta\not=\Pset(V)$. To see this take a facet $F\in\Delta(T)$. Since $\Delta\not=\Pset(V)$ there are $v_{1}\kpkt v_{r}\in V\setminus T$ such that $G:=F\cup\{v_{1}\kpkt v_{r}\}\not\in\Delta$ but $\iota^{-1}(G)=F\in\Delta(T)$. If $\Delta=\Pset(V)$, the map $\iota$ is also $\beta$-simplicial, hence
$$\varphi_{\iota}^{\beta}:\Pset(V)_{\cup}^{\infty}\Rto\Delta(T)_{\cup}^{\infty}\quad\text{with}\quad F\lto\begin {cases}
\infty&\text{, if }F=\infty\komma\\
F\cap T&\text{, otherwise,}
\end {cases}$$
is a binoid homomorphism which is different from $\varphi_{\iota}:\Pset(V)_{\cup}^{\infty}\rto\Delta(T)_{\cup}^{\infty}$ because $\varphi_{\iota}(\{v\})=\{v\}=\varphi_{\iota}^{\beta}(\{v\})$ for $v\in T$, but if $v\not\in T$, then
$$\varphi_{\iota}(\{v\})=\infty\quad\text{and}\quad\varphi_{\iota}^{\beta}(\{v\})=\emptyset\pkt$$
\end {Example}

\begin {Example} \label {ExpProducts}
Let $\Delta_{i}$ be a simplicial complex on $V_{i}$, $i\in I$, where $I$ is finite and $V_{i}\cap V_{j}=\emptyset$ for $i\not=j$. For every $k\in I$, the injective morphism
$$\iota_{k}:\Big(V_{k},\Delta_{k}\Big)\Rto\Big(\biguplus_{i\in I}V_{i},\biguplus_{i\in I}\Delta_{i}\Big)\quad\text{with}\quad v\lto (v;k)$$
is simplicial and $\alpha$-simplicial. We thus get the binoid epimorphism 
$$\varphi_{\iota_{k}}:\,\,\bigcupbidot_{i\in I}\Delta_{i,\cup}^{\infty}\Rto\Delta_{k,\cup}^{\infty}\quad\text{with}\quad(F;i)\lto\begin {cases}
\infty&\text{, if }i\not=k\komma\\
F&\text{, otherwise,}
\end {cases}$$
and the binoid embedding 
$$\varphi_{\iota_{k}}^{\alpha}:\Delta_{k,\cup}^{\infty}\Rto\,\,\bigcupbidot_{i\in I}\Delta_{i,\cup}^{\infty}\quad\text{with}\quad F\lto(F;k)\pkt$$
Moreover, we have the binoid embedding
$$\psi_{\iota_{k}}:\Delta_{k,\cap}\onull\Rto\,\,\bigcupbidot_{i\in I}\Delta_{i,\cap}\onull\quad\text{with}\quad F\lto(F;k)$$
and the binoid epimorphism 
$$\psi_{\iota_{k}}^{\alpha}:\bigcupbidot_{i\in I}\Delta_{i,\cap}\onull\Rto\Delta_{k,\cap}\onull\quad\text{with}\quad(F;i)\lto\begin {cases}
0&\text{, if }(F,i)=0\komma\\
(F;k)&\text{, if }i=k\komma\\
\emptyset&\text{, otherwise.}
\end {cases}$$
The injective morphism
$$\tilde{\iota}_{k}:\Big(V_{k},\Delta_{k}\Big)\Rto\Big(\biguplus_{i\in I}V_{i},\prod_{i\in I}\Delta_{i}\Big)\quad\text{with}\quad v\lto (v;k)$$
is only $\alpha$-simplicial for $\#I\ge2$. Here we obtain the binoid embedding
$$\varphi_{\tilde{\iota}_{k}}^{\alpha}:\Delta_{k,\cup}^{\infty}\Rto\bigwedge_{i\in I}\Delta_{i,\cup}^{\infty}\quad\text{with}\quad F\lto\emptyset\wedge\cdots\wedge\emptyset\wedge F\wedge\emptyset\wedge\cdots\wedge\emptyset\komma$$
where $F$ is the $k$th entry, and the binoid epimorphism
$$\psi_{\tilde{\iota}_{k}}^{\alpha}:\bigwedge_{i\in I}\Delta_{i,\cap}\onull\Rto\Delta_{k,\cap}\onull\quad\text{with}\quad \wedge_{i\in I}F_{i}\lto F_{k}\pkt$$
\end {Example}

\bigskip

\section{$N\mina$points of a simplicial complex} \label{SecNpointsSimplCompl}
\markright{\ref {SecNpointsSimplCompl} $N\mina$points of a simplicial complex}

This section deals with $N\mina$points of a simplicial complex $\Delta$, which will later be related to the $N\mina$points of the (simplicial) binoid defined by $\Delta$

\begin {Convention}
In this section, $N$ always denotes a \emph{commutative} nonzero binoid.
\end {Convention}

\begin {Definition}
Let $\Delta$ be a simplicial complex on $V$. An \gesperrt{$N\mina$point} \index{simplicial complex!N@$N\mina$point of a --}\index{N@$N\mina$point!-- of a simplicial complex}of $\Delta$ is a map
$$\rho:V\Rto N\quad\text{such that}\quad\sum_{v\in A}\rho(v)=\infty\quad\text{if}\quad A\not\in \Delta\pkt$$
The \gesperrt{$N\mina$spectrum}\index{simplicial complex!N@$N\mina$spectrum of a --}\index{spectrum!N@$N\mina$-- of a simplicial complex}\index{N@$N\mina$spectrum!-- of a simplicial complex} of $\Delta$ is the set of all $N\mina$points of $\Delta$, denoted by $N\minspec\Delta$. With respect to the addition $N\minspec\Delta$ is a semibinoid with absorbing element $\varepsilon:v\mto\infty$, $v\in V$.\nomenclature[N]{$N\minspec\Delta$}{set of all $N\mina$points of $\Delta$ ($N\mina$spectrum of $\Delta$)}
\end {Definition}

If $N$ is integral the property $\sum_{v\in A}\rho(v)=\infty$ for all $A\not\in\Delta$ is equivalent to: for every $A\not\in\Delta$ there is a $v\in A$ such that $\rho(v)=\infty$, and if $N$ is a multiplicatively written, this condition on $\rho:V\rto N$ translates to $\prod_{v\in A}\rho(v)=0$ for every nonface $A$ of $\Delta$. In particular, if $K$ is a field, the map $\rho:V\rto K$ is a $K\mina$point if and only if for every $A\not\in\Delta$ there is a $v\in A$ such that $\rho(v)=0$. Later we will prove that $K\minspec\Delta\cong K\minSpec K[\Delta]$, where $K[\Delta]$ is the Stanley-Reisner algebra of $\Delta$, cf. Proposition \ref{PropUnivPropSimplBinoid}.

\begin {Lemma} \label{LemIndicatorPointFace}
Let $\Delta$ be a simplicial complex on $V$. A subset $F\subseteq V$ is a face of $\Delta$ if and only if $\chi_{F}:V\rto\trivial$ is an $N\mina$point (with respect to $\trivial\embto N$).
\end {Lemma}
\begin {proof}
Obviously, $F$ is a face if $\chi_{F}$ is an $N\mina$point because $\sum_{v\in F}\chi_{F}(v)=0\not=\infty$. Conversely, if $F\in\Delta$ and $A\not\in\Delta$, then $A\not\subseteq F$. Thus, we find a $w\in A\setminus(A\cap F)$, but this means $\chi_{F}(w)=\infty$ and therefore $\sum_{v\in A}\chi_{F}(v)=\infty$.
\end {proof}

\begin {Proposition}\label{PropAlphaSimplIndNSpec}
Let $\Delta$ and $\tilde{\Delta}$ be simplicial complexes on $V$ and $\tilde{V}$, respectively. The following conditions on a map $\lambda:\tilde{V}\rto V$ are equivalent.
\begin {ListeTheorem}
\item $\lambda:(\tilde{V},\tilde{\Delta})\rto(V,\Delta)$ is an $\alpha$-simplicial morphism.
\item For every $\rho\in N\minspec\Delta$, $N$ a commutative binoid, $\rho\lambda:\tilde{V}\rto N$ is an $N\mina$point of $\tilde{\Delta}$.
\item For every $\rho\in K\minspec\Delta$, $K$ a field, $\rho\lambda:\tilde{V}\rto K$ is a $K\mina$point of $\tilde{\Delta}$.
\end {ListeTheorem}
In particular, every $\alpha$-simplicial morphism $\lambda:(\tilde{V}, \tilde{\Delta})\rto(V,\Delta)$ induces a semibinoid homomorphism
$$\lambda^{\ast}:N\minspec\Delta\Rto N\minspec\tilde{\Delta}$$
with $\rho\mto\rho\lambda$, which is injective if $\lambda$ is surjective.
\end {Proposition}
\begin {proof}
$(1)\Rarrow(2)$ For an arbitrary nonface  $A\subseteq\tilde{V}$ of $\tilde{\Delta}$, we have
$$\sum_{v\in A}\rho\lambda(v)\,\,=\sum_{w\in\lambda(A)}n_{w}\rho(w)$$
for certain $n_{w}\ge1$, $w\in\lambda(A)$. By assumption, $\lambda(A)\not\in\Delta$ and $\rho$ is an $N\mina$point, hence $\sum_{w\in\lambda(A)}\rho(w)=\infty$, but this is a summand of $\sum_{v\in A}\rho\lambda(v)$. $(2)\Rarrow(3)$ is obvious. $(3)\Rarrow(1)$ Suppose that $A\subseteq\tilde{V}$ is a nonface of $\tilde{\Delta}$ but $\lambda(A)\in\Delta$. By Lemma \ref{LemIndicatorPointFace}, the map $\chi_{\lambda(A)}:V\rto \{1,0\}\subseteq K$ is a $K\mina$point of $\Delta$, which yields the $K\mina$point $\chi_{\lambda(A)}\lambda$ of $\tilde{\Delta}$ with $\prod_{v\in A}\chi_{\lambda(A)}(\lambda(v))=1$, a contradiction to $A\not\in\tilde{\Delta}$. The supplement is clear.
\end {proof}

\begin {Example}\label{ExpAlphaSimplIndNSpec}
The $\alpha$-simplicial morphism $\id:(V,\Delta)\rto(V,\tilde{\Delta})$ from Example \ref{ExpSubcomplexSameV}, where $\Delta$ and $\tilde{\Delta}$ are simplicial complexes on $V$ such that $\tilde{\Delta}\subseteq\Delta$, induces a semibinoid embedding
$$N\minspec\tilde{\Delta}\Rto N\minspec\Delta\quad\text{with}\quad\rho\lto\rho\pkt$$
The $\alpha$-simplicial morphisms $\iota:(T,\Delta(T))\rto(V,\Delta)$ from Example \ref{ExpSubcomplexDifferntV}, where $T\subseteq V$ and $\Delta(T)=\{F\in\Delta\mid F\subseteq T\}$, induces the semibinoid homomorphism
$$N\minspec\Delta\Rto N\minspec\Delta(T)\quad\text{with}\quad\rho\lto\rho_{|T}\pkt$$
Finally, the $\alpha$-simplicial morphisms $\iota_{k}:(V_{k},\Delta_{k})\rto\big(V,\biguplus_{i\in I}\Delta_{i}\big)$ and $\tilde{\iota}_{k}:(V_{k},\Delta_{k})\rto\big(V,\prod_{i\in I}\Delta_{i}\big)$, $k\in I$, from Example \ref{ExpProducts}, where $\Delta_{i}$ is a simplicial complex on $V_{i}$, $i\in I$, and $V=\biguplus_{i\in I}V_{i}$, induce the semibinoid homomorphisms
$$N\minspec\biguplus_{i\in I}\Delta_{i}\Rto N\minspec\Delta_{k}\quad\text{with}\quad\rho\lto\rho_{|(V_{k};k)}$$
and 
$$N\minspec\prod_{i\in I}\Delta_{i}\Rto N\minspec\Delta_{k}\quad\text{with}\quad\rho\lto\rho_{|(V_{k};k)}\pkt$$
\end {Example}

\begin {Proposition}\label{PropBetaSimpNSpec}
Let $\Delta$ and $\tilde{\Delta}$ be simplicial complexes on $V$ and $\tilde{V}$, respectively. The following conditions on a map $\lambda:\tilde{V}\rto V$ are equivalent.
\begin {ListeTheorem}
\item $\lambda:(\tilde{V},\tilde{\Delta})\rto(V,\Delta)$ is a $\beta$-simplicial morphism.
\item For every $\rho\in N\minspec\tilde{\Delta}$, $N$ a commutative binoid, the map 
$$\widehat{\rho}:V\Rto N\quad\text{with}\quad v\lto\sum_{w\in\lambda^{-1}(v)}p(w)$$ 
is an $N\mina$point of $\Delta$.
\item For every $\rho\in K\minspec\tilde{\Delta}$, $K$ a field, the map 
$$\widehat{\rho}:V\Rto K\quad\text{with}\quad v\lto\prod_{w\in\lambda^{-1}(v)}p(w)$$ 
is a $K\mina$point of $\Delta$.
\end {ListeTheorem}
In particular, every $\beta$-simplicial morphism $\lambda:(\tilde{V}, \tilde{\Delta})\rto(V,\Delta)$ induces a semibinoid homomorphism
$$\lambda^{\ast}:N\minspec\tilde{\Delta}\Rto N\minspec\Delta$$
with $\rho\mto\widehat{\rho}$, which is injective if $\lambda$ is so.
\end {Proposition}
\begin {proof}
$(1)\Rarrow(2)$ For an arbitrary nonface $A\subseteq V$ of $\Delta$, we have
$$\sum_{v\in A}\widehat{\rho}(v)\,\,=\,\,\sum_{v\in A}\sum_{w\in\lambda^{-1}(v)}\rho(w)\,\,=\sum_{w\in\lambda^{-1}(A)}\rho(w)\,=\,\infty$$
because $\rho$ is an $N\mina$point of $\tilde{\Delta}$ and $\lambda^{-1}(A)\not\in\tilde{\Delta}$ by assumption on $\lambda$. Hence, $\widehat{\rho}$ is an $N\mina$point of $\Delta$. $(2)\Rarrow(3)$ is obvious. To show $(3)\Rarrow(1)$ suppose that $A\subseteq V$ is a nonface of $\Delta$ but $\lambda^{-1}(A)\in\tilde{\Delta}$. By Lemma \ref{LemIndicatorPointFace}, the map $\chi_{\lambda^{-1}(A)}:\tilde{V}\rto \{1,0\}\subseteq K$ is a $K\mina$point of $\tilde{\Delta}$, so by assumption 
$$\widehat{\chi}_{\lambda(A)}:V\Rto K\quad\text{with}\quad v\lto\prod_{w\in\lambda^{-1}(v)}\chi_{\lambda^{-1}(A)}(w)\komma$$ 
is a $K\mina$point of $\Delta$ with 
$$\prod_{v\in A}\widehat{\chi}_{\lambda(A)}(v)\,\,=\,\,\prod_{v\in A}\prod_{w\in\lambda^{-1}(v)}\chi_{\lambda^{-1}(A)}(w)\,\,=\,\,\prod_{w\in\lambda^{-1}(A)}\chi_{\lambda^{-1}(A)}(w)\,=\,1\komma$$ 
which is a contradiction to $A\not\in\Delta$. For the supplement assume that $\rho_{1},\rho_{2}\in N\minspec\tilde{\Delta}$ such that $\widehat{\rho_{1}}=\widehat{\rho_{2}}$. Then $\widehat{\rho_{1}}(v)=\widehat{\rho_{2}}(v)$ for all $v\in V$, but by the injectivity of $\lambda$ we have
$$\widehat{\rho_{i}}(v)\,=\,\sum_{w\in\lambda^{-1}(v)}\rho_{i}(w)\,=\,\rho_{i}(\lambda^{-1}(v))$$
for all $v\in V$, $i\in\{1,2\}$. Thus, $\rho_{1}=\rho_{2}$.
\end {proof}

\begin {Example}
The $\beta$-simplicial morphism $\widetilde{\id}:(V,\tilde{\Delta})\Rto(V,\Delta)$ from Example \ref{ExpSubcomplexSameV}, where $\Delta$ and $\tilde{\Delta}$ are simplicial complexes on $V$ such that $\tilde{\Delta}\subseteq\Delta$, induces a semibinoid embedding
$$N\minspec\tilde{\Delta}\Rto N\minspec\Delta\quad\text{with}\quad\rho\lto\rho\komma$$
which is the semibinoid embedding of Example \ref{ExpAlphaSimplIndNSpec}. In particular, we always have an embedding
$$N\minspec\tilde{\Delta}\Rto N\minspec\Pset(V)=N^{\#V}\pkt$$
\end {Example}

\begin {Remark}
By Lemma \ref{LemSimplBetaSimpl}(1), a surjective simplicial morphism $\lambda:(\tilde{V},\tilde{\Delta})\rto(V,\Delta)$ is also $\beta$-simplicial. Hence, Proposition \ref{PropBetaSimpNSpec} applies to $\lambda$ in this case and gives a semibinoid embedding
$$\lambda^{\ast}:N\minspec\tilde{\Delta}\Rto N\minspec\Delta\quad\text{with}\quad\rho\lto\widehat{\rho}\pkt$$
\end {Remark}


\bigskip

\section{Topology of simplicial complexes} \label{SecSimplComplTop}
\markright{\ref {SecSimplComplTop} Topology of simplicial complexes}

In this section, we study the ideal lattices of $\Delta_{\cup}^{\infty}$ and $\Delta_{\cap}\onull$ and describe the spectra of these binoids.

\medskip

If $\Delta$ is not the full simplicial complex on $V$ (i.e.\ $\Delta\not=\Pset(V)$), we will use the identification
$$\Delta_{\cap}\onull\,\,\,\cong\,\,\,(\Delta\cup\{V\},\cap,V,\emptyset)\quad\text{\large{and}}\quad\Delta_{\cup}^{\infty}\,\,\,\cong\,\,\,(\Delta\cup\{V\},\cup,\emptyset,V)$$
throughout this section. For a simplicial complex $\Delta\not=\Pset(V)$ on $V$, the canonical binoid embedding and epimorphism induced by the $\alpha$-simplicial morphism $\id:(V,\Pset(V)\setminus\{V\})\rto(V,\Delta)$, cf.\ Example \ref{ExpSubcomplexSameV}, read as follows
$$\begin {array} {rclcrcl}
\iota\UDelta:\Delta_{\cap}\onull\!\!\!&\Rto&\!\!\!\Pset(V)_{\cap}\!\!\!&\quad\text{and}\quad&\!\!\!\pi\UDelta:\Pset(V)_{\cup}\!\!\!&\Rto&\!\!\!\Delta_{\cup}^{\infty}\\
F\!\!\!&\lto&F\!\!\!&&\!\!\!F\!\!\!&\lto&\!\!\!
\begin {cases}
F&\text{, if }F\in\Delta,\\
V&\text{, otherwise.}
\end {cases}
\end {array}$$

\begin {Proposition}
Let $\Delta\not=\Pset(V)$ be a simplicial complex on $V$.
\begin {ListeTheorem}
\item  $\Delta$ is an ideal in $\Pset(V)_{\cap}$. 
\item A subset $A\subseteq\Pset(V)$ is a prime ideal in $\Pset(V)_{\cap}$ if and only if $A\not=\Pset(V)$ defines a pure simplicial complex on $V$ of dimension $\#V-2$; that is, $A=\bigcup_{v\in G}\Pset(V\setminus\{v\})$ for some $\emptyset\not=G\subseteq V$.
\item $\Pset(V)\setminus\Delta=\Delta^{\opc}$ is an ideal in $\Pset(V)_{\cup}$ with $\Pset(V)_{\cup}/\Delta^{\opc}\,\,\cong\,\,\Delta_{\cup}^{\infty}$ via $\pi\UDelta$. In particular, $\spec\Delta_{\cup}^{\infty}\cong\opV(\Delta^{\opc})=\{\Pcal\in\spec\Pset(V)_{\cup}\mid\Delta^{\opc}\subseteq\Pcal\}$.
\end {ListeTheorem}
\end {Proposition}
\begin {proof}
(1) and (2) follow immediately from Example \ref{ExSpecPowerset}, and (3) from Example \ref{ExpPsetIdeals} and Corollary \ref{CorInducedEmb}.
\end {proof}

\begin {Proposition} \label{PropFilterSimplCompl}
Let $\Delta$ be a simplicial complex on $V$. The one-to-one correspondences
$$\begin {array} {rcccl}
\Fcal(\Delta_{\cap}\onull)\!\!\!&\longleftrightarrow&\!\!\!\Delta\cup\{V\}\!\!\!&\longleftrightarrow&\!\!\!\Fcal(\Delta_{\cup}^{\infty})\\
S_{F,\cap}:=\{G\in\Delta_{\cap}\onull\mid F\subseteq G\}\!\!\!&\longleftrightarrow&\!\!\! F\!\!\!&\longleftrightarrow&\!\!\!\{G\in\Delta_{\cup}^{\infty}\mid G\subseteq F\}=:S_{F,\cup}
\end {array}$$
yield binoid isomorphisms $\Fcal(\Delta_{\cap}\onull)_{\cap}\cong\Delta_{\cup}^{\infty}$ and $\Fcal(\Delta_{\cup}^{\infty})_{\cap}\cong\Delta_{\cap}\onull$. In particular, 
$$\begin {array} {rclcrcl}
\spec\Delta_{\cap}\onull\!\!\!&\stackrel{\sim}{\longleftrightarrow}&\!\!\!\Delta_{\cup}^{\infty}\setminus\{\emptyset\}&\quad\text{and}\quad&\spec\Delta_{\cup}^{\infty}\!\!\!&\stackrel{\sim}{\longleftrightarrow}&\!\!\!\Delta_{\cap}\\
\Pcal_{F,\cap}:=\{G\in\Delta_{\cap}\onull\mid F\not\subseteq G\}\!\!\!&\longleftrightarrow&\!\!\!F&&\Pcal_{F,\cup}:=\Delta_{\cup}^{\infty}\setminus\Pset(F)\!\!\!&\longleftrightarrow&\!\!\!F
\end {array}$$
are semibinoid isomorphisms, the latter inclusion preserving and the first reversing. In particular, $\specE\Delta_{\cap}\onull\cong\Delta_{\cup}^{\infty}$ and $\specE\Delta_{\cup}^{\infty}\cong\Delta_{\cap}\onull$ as binoids.
\end {Proposition}
\begin {proof}
By definition, every filter $S$ of $\Delta_{\cap}\onull$ has to contain the identity element $V$ and since $A,B\in S$ is equivalent to $A\cap B\in S$, there is a unique minimal element $F\in S$ with respect to $\subseteq$ and all supersets of $F$ in $\Delta\cup\{V\}$ lie in $S$. This proves the correspondence $S_{F,\cap}\leftrightarrow F$. Similarly, every filter $S$ of $\Delta_{\cup}^{\infty}$ admits a maximal element $F\in S$ with respect to $\subseteq$ and all subsets of $F$ in $\Delta\cup\{V\}$ lie in $S$, which proves the correspondence $F\leftrightarrow S_{F,\cup}$. By taking complements, we obtain the results for the spectra of $\Delta_{\cap}\onull$ and $\Delta_{\cup}^{\infty}$. Moreover, for two faces $F,F^{\prime}\in\Delta$ one has 
$$S_{F,\cap}\cap S_{F^{\prime},\cap}=S_{F\cup F^{\prime},\cap}\quad\text{and}\quad S_{F,\cup}\cap S_{F^{\prime},\cup}=S_{F\cap F^{\prime},\cup}\komma$$
and (with respect to the union on $\Delta_{\cup}^{\infty}$)
$$\Pcal_{F,\cap}\cup\Pcal_{F^{\prime},\cap}=\Pcal_{F\cup F^{\prime},\cap}\quad\text{and}\quad \Pcal_{F,\cup}\cup\Pcal_{F^{\prime},\cup}=\Pcal_{F\cap F^{\prime},\cup}\pkt$$
In particular, $S_{\emptyset,\cap}=\Delta\cup{V}=S_{V,\cup}$, $S_{V,\cap}=\{V\}$, and $S_{\emptyset,\cup}=\{\emptyset\}$, which gives $\Pcal_{V,\cap}=\Delta=(\Delta_{\cap}\onull)\Uplus$ and $\Pcal_{\emptyset,\cup}=\Delta_{\cup}^{\infty}\setminus\{\emptyset\}=(\Delta_{\cup}^{\infty})\Uplus$. This implies the binoid and semibinoid isomorphisms. Clearly, if $F\subseteq F^{\prime}$, then $\Pcal_{F^{\prime},\cup}\subseteq\Pcal_{F,\cup}$. On the other hand, if $\Pcal_{F^{\prime},\cup}\subseteq\Pcal_{F,\cup}$, then $F\subseteq F^{\prime}$ since $F^{\prime}\in\Pcal_{F^{\prime},\cup}$. Similarly, $F\subseteq F^{\prime}$ is equivalent to $\Pcal_{F,\cap}\subseteq\Pcal_{F^{\prime},\cap}$.
\end {proof}

Note that for $\Delta=\Pset(V)\setminus\{V\}$ the preceding lemma includes all results for $\Pset(V)_{\cap}$ and $\Pset(V)_{\cup}$, cf.\ Example \ref{ExpFilterPowerset} and Example \ref{ExSpecPowerset}. The semibinoid homeomorphisms induced by $\iota\UDelta$ and $\pi\UDelta$ are given by the epimorphism
$$\iota\UDelta^{\ast}:\spec\Pset(V)_{\cap}\Rto\spec\Delta_{\cap}\onull\quad\text{with}\quad\{J\subseteq V\mid F\not\subseteq J\}\lto\Pcal_{F,\cap}$$
for $\emptyset\not= F\subseteq V$, and by the embedding 
$$\pi\UDelta^{\ast}:\spec\Delta_{\cup}^{\infty}\Rto\spec\Pset(V)_{\cup}\quad\text{with}\quad\Pcal_{F,\cup}\lto\Pset(V)\setminus\Pset(F)$$
for $F\in\Delta$.

\begin{Corollary}
$\dim\Delta+1=\dim\Delta_{\cap}\onull=\dim\Delta_{\cup}^{\infty}$ for every simplicial complex $\Delta\not=\Pset(V)$ on $V.$
\end{Corollary}
\begin{proof}
This is an immediate consequence of Proposition \ref {PropFilterSimplCompl}.
\end{proof}

The preceding results show that $\spec\Delta_{\cap}\onull$ and $\spec\Delta_{\cup}^{\infty}$ always have the same cardinality and the same dimension as topological spaces endowed with the Zariski topology. However, if $\Delta_{\cap}\onull$ and $\Delta_{\cup}^{\infty}$ are not isomorphic as binoids, the Zariski topologies on their spectra are dual to each other as the following corollary shows.

\begin {Corollary}
Consider $\spec\Delta_{\cup}^{\infty}$ and $\spec\Delta_{\cap}\onull$ as topological spaces with respect to the Zariski topology. The bijection 
$$h:\Delta\Rto\spec\Delta_{\cup}^{\infty}$$
with $F\mto\Pcal_{F,\cup}$, is a homeomorphism if the poset $(\Delta,\subseteq)$ is endowed with the upper topology, and the bijection 
$$g:\Delta\Rto\spec\Delta_{\cap}\onull$$
with $F\mto\Pcal_{F,\cap}$, is a homeomorphism if the poset $(\Delta,\subseteq)$ is endowed with the lower topology
\end {Corollary}
\begin {proof}
Since $\spec\Delta_{\cup}^{\infty}$ and $\spec\Delta_{\cap}\onull$ are finite sets, the Zariski topology coincides with the lower topology (with respect to $\subseteq$) on them, cf.\ Remark \ref{RemLowUpperTop}. Hence, the basic open sets of $\spec\Delta_{\cup}^{\infty}$ are given by $\emptyset$ and 
$$D_{\Pcal}=\{\Qcal\in\spec\Delta_{\cup}^{\infty}\mid\Qcal\subseteq\Pcal\}$$
for $\Pcal\in\spec\Delta_{\cup}^{\infty}$. The basic open sets in the upper topology of the poset $(\Delta,\subseteq)$ are given by $\emptyset$ and 
$$D_{F}=\{G\in\Delta\mid F\subseteq G\}$$
for $F\in\Delta$. This shows that $h^{-1}(D_{\Pcal_{F,\cup}})=D_{F}$ because $G\subseteq F$ is equivalent to $\Pcal_{F,\cup}\subseteq\Pcal_{G,\cup}$ by Proposition \ref{PropFilterSimplCompl}. Similarly, $g$ is a homeomorphism if $(\Delta,\subseteq)$ is endowed with the lower topology, where the basic open subsets are given by
$$D_{F}=\{G\in\Delta\mid G\subseteq F\}$$
for $F\in\Delta$. By Proposition \ref{PropFilterSimplCompl}, $G\subseteq F$ is equivalent to $\Pcal_{G,\cap}\subseteq\Pcal_{F,\cap}$, and therefore $g^{-1}(D_{\Pcal_{F,\cap}})=D_{F}$, where $D_{\Pcal}=\{\Qcal\in\spec\Delta_{\cap}\onull\mid\Qcal\subseteq\Pcal\}$ for $\Pcal\in\spec\Delta_{\cap}\onull$ are the basic open sets of $\spec\Delta_{\cap}\onull$.
\end {proof}

\bigskip

\section{Simplicial binoids} \label{SecSimplBinos}
\markright{\ref {SecSimplBinos} Simplicial binoids}

Now we define the binoid associated to a simplicial complex $\Delta$ such that its binoid algebra is the Stanley-Reisner algebra of $\Delta$. Conversely, we characterize those binoids that yield Stanley-Reisner algebras, cf.\ Theorem \ref {ThClassSimplB}. Then the tools and notations developed in the last sections will be applied to these binoids.

\begin {Definition}
Let $\Delta$ be a simplicial complex on $V$. The binoid associated to $\Delta$ is defined to be 
$$M\UDelta\,:=\,\free(V)/\Ical\UDelta\komma$$
\nomenclature[M]{$M\UDelta$}{binoid associated to $\Delta$}where $\Ical\UDelta$ denotes the ideal $\{f\in\free(V)\mid\supp(f)\not\subseteq\Delta\}$ in the free commutative binoid $\free(V)$. A binoid $M$ is called a \gesperrt{simplicial binoid} \index{binoid!simplicial --}\index{simplicial! binoid}if $M\cong M\UDelta$ for a simplicial complex $\Delta$. In this case, the simplicial complex $\Delta$ is unique and called the \gesperrt{underlying simplicial complex} \index{simplicial complex!underlying --}of $M$. Simplicial binoids together with binoid homomorphisms form a full subcategory of the category $\cBsf$ of commutative binoids, which will be denoted by $\sBsf$.\nomenclature[SB]{$\sBsf$}{category of simplicial binoids}
\end {Definition}

Equivalently, one may define the binoid $M\UDelta$ by
$$M\UDelta:=\{\varphi:V\rto\N\text{ map}\mid\supp\varphi\in\Delta\}\cup\{\infty\}\komma$$
where $\supp\varphi=\{v\in V\mid\varphi(v)\not=0\}$, with addition defined by
$$\varphi+\psi:=
\begin{cases}
\varphi+\psi&\text{, if }\supp(\varphi+\psi)\in\Delta\komma\\
\infty&\text{, otherwise,}
\end{cases}$$
or by 
$$M\UDelta:=\Big\{\sum_{v\in F}n_{v}v\,\,\Big|\,\, F\in\Delta,n_{v}\in\N\Big\}\,\cup\,\{\infty\}$$ 
with addition defined by
$$\sum_{v\in F}n_{v}v+\sum_{v\in G}m_{v}v\,:=\,
\begin{cases}
\sum_{v\in F\cup G}(n_{v}+m_{v})v&\text{, if }F\cup G\in\Delta\komma\\
\infty&\text{, otherwise,}
\end{cases}$$
where $n_{v}=0$ if $v\not\in F$ and $m_{v}=0$ if $v\not\in G$. In particular, for every $f\in M\UDelta\opkt$ there is a unique face $F\in\Delta$ such that $f=\sum_{v\in F}n_{v}v$ with $n_{v}\ge1$ for all $v\in F$.

\begin {Example}
$M_{^{_{\Pset(V)}}}$ coincides with the free binoid $\free(V)\cong(\N^{n})^{\infty}$.
\end {Example}

\begin {Example}
If $\Delta=\{\emptyset,\{v\}\mid v\in V\}$, then $M\UDelta=\,\,\bigcupbidot_{v\in V}\free(\{v\})$. In general, if the facets of a simplicial complex $\Delta$ are pairwise disjoint, then $M\UDelta=\,\,\bigcupbidot_{ F\text{ facet of }\Delta}\free(F)$. See also Corollary \ref {CorSimplCompositions} below.
\end {Example}

Now we recall the definition of the Stanley-Reisner algebra of a simplicial complex, cf.\ \cite[Chapter 5.1]{BrunsHerzog} or \cite[Chapter 5.3]{Villarreal}.

\begin{Definition}
Let $\Delta$ be a simplicial complex on $V=\{v_{1}\kpkt v_{n}\}$ and $R[X_{1}\kpkt X_{n}]$ the polynomial algebra over the ring $R$. The \gesperrt{Stanley-Reisner ideal} \index{Stanley-Reisner ideal}\index{ideal!Stanley-Reisner --}$\Iideal\UDelta$\nomenclature[I]{$\Iideal\UDelta$}{Stanley-Reisner ideal of $\Delta$} of $\Delta$ is the ideal in $R[X_{1}\kpkt X_{n}]$ generated by all monomials $X_{j_{1}}\cdots X_{j_{s}}$ such that $\{v_{j_{1}}\kpkt v_{j_{s}}\}\not\in\Delta$. Then the \gesperrt{Stanley-Reisner algebra} \index{Stanley-Reisner algebra} (also called \gesperrt{face ring}\index{face!-- ring}) of $\Delta$ over $R$ is given by the graded $R\mina$algebra $R[X_{1}\kpkt X_{n}]/\Iideal\UDelta=:R[\Delta]$\nomenclature[R1]{$R[\Delta]$}{Stanley-Reisner algebra of $\Delta$}.
\end{Definition}

\begin {Example}
If $\Delta=\Pset(V)$, then $\Iideal_{\Delta}=(0)$, hence $R[\Delta]$ is the polynomial algebra $R[X_{v}\mid v\in V]$.
\end {Example}

\begin {Remark}
By definition, $\Iideal\UDelta$ is generated by squarefree monomials. On the other hand, for every ideal $\Iideal$ in $R[X_{1}\kpkt X_{n}]$ that is generated by squarefree monomials, there is a simplicial complex on $V$, $\# V=n$, such that $R[X_{1}\kpkt X_{n}]/\Iideal\cong R[\Delta]$. Moreover, if $\Delta$ and $\Delta^{\prime}$ are simplicial complexes on the same vertex set $V$, then $\Delta\subseteq\Delta^{\prime}$ if and only if $\Iideal_{\Delta^{\prime}}\subseteq \Iideal_{\Delta}$.
\end {Remark}

As announced, every Stanley-Reisner algebra can be realized as a binoid algebra.

\begin {Lemma} \label {LemSRalgebraBinoid}
If $\Delta$ is a simplicial complex and $R$ a ring, then $R[\Delta]\cong R[M\UDelta]$. 
\end {Lemma}
\begin {proof}
If $\Delta$ be a simplicial complex on $V=\{1\kpkt n\}$, then $R[M\UDelta]\,=\,R[\free(V)]/R[\Ical_{\Delta}]$ with
$$R[\Ical_{\Delta}]\,=\,(T^{a}\mid a\in\Ical_{\Delta})\,=\,(X^{a}\mid a\in\free(V),\supp a\not\subseteq\Delta)=:\Iideal\komma$$
where $X^{a}=X_{1}^{a_{1}}\cdots X_{n}^{a_{n}}$ for $a=\sum_{i\in V}a_{i}i$. Since $X^{a}\in\Iideal$ if and only if $X^{\tilde{a}}\in\Iideal$ for $\tilde{a}=\sum_{i\in \supp(a)}i$, we get 
$$\Iideal=\Big(X^{a}\,\,\Big|\,\, a=\sum_{i\in J\subseteq V}i,J\not\subseteq\Delta\Big)=(X_{j_{1}}\cdots X_{j_{s}}\mid \{j_{1}\kpkt j_{s}\}\not\subseteq\Delta)=\Iideal\UDelta\pkt$$
Thus, $R[M\UDelta]$ is the Stanley-Reisner algebra of $\Delta$.
\end {proof}

Simplicial binoids can be characterized as follows.

\begin {Theorem} \label{ThClassSimplB}
Let $\Delta$ be a simplicial complex on $V$. The binoid $M\UDelta$ is finitely generated by $\# V=n$ elements, semifree, and reduced. Conversely, every commutative binoid $M$ satisfying these properties is a simplicial binoid. More precisely, $M$ is isomorphic to $M\UDelta$, where $\Delta:=\{F\subseteq W\mid\sum_{w\in F}w\not=\infty\}$ for a minimal generating set $W$ of $M$.
\end {Theorem}
\begin {proof}
By Corollary \ref {CorQuotientRepFree} and Lemma \ref{LemStrongRed=Red}, we only need to show that $M\UDelta$ is reduced. This follows immediately from the fact that $\supp(f+f)=\supp(f)\cup\supp(f)=\supp(f)$. For the converse note that such a generating set exists by Proposition \ref{PropUniqueMinSyst}. So take a semibasis $W$ of $M$ and consider the canonical binoid epimorphism $\varphi:\free(W)\Rto M$ with $w\mto w$. By Lemma \ref{LemPropSemifree}, $M$ is cancellative. Therefore, if $f,g\in\free(W)$ with $f=\sum_{w\in W}n_{w}w$ and $g=\sum_{w\in W}m_{w}w$, then  $\varphi(f)=\varphi(g)$ is equivalent to  $\varphi(f)=\varphi(g)=\infty$ or $n_{w}=m_{w}$ for all $w\in W$, and hence $f=g$. Furthermore, we have by Lemma \ref{LemStrongRed=Red} that $f+f+g=\infty$ implies that $ f+g=\infty$ for all $f,g\in M$. Applying this successively to $\sum_{w\in T}n_{w}w=\infty$, where $T\subseteq W$, $n_{w}\ge1$, $w\in T$, yields $\sum_{w\in T}w=\infty$ because if $u\in T$ with $n_{u}\ge2$, then
\begin {align*}
\infty\,=\,\sum_{w\in T}n_{w}w&\,\,=\,\,u+u+\Big((n_{u}-2)u+\sum_{T\setminus\{u\}}n_{w}w\Big)\\
&\,\,=\,\,u+\Big((n_{u}-2)u+\sum_{T\setminus\{u\}}n_{w}w\Big)\\
&\,\,=\,\,\ldots\\
&\,\,=\,\,u+\sum_{T\setminus\{u\}}n_{w}w\pkt
\end {align*}
Thus, we have shown that $\sim_{\varphi}\,=\,\sim_{\Ical}$, where $\Ical=\{\sum_{w\in T}w\in\free(W)\mid\sum_{w\in T}w=\infty$ in $M\}$, which gives $\free(W)/\Ical\cong M$ by Proposition \ref{PropHomCong}. Since 
$$\Ical=\Big\{\sum_{w\in F}w\in\free(W)\,\Big| F\not\subseteq\Delta\Big\}\komma$$
where $\Delta:=\{F\subseteq W\mid\sum_{w\in F}w\not=\infty\}$ is a simplicial complex on $W$, we obtain $\Ical=\Ical_{\Delta}$. Thus, $M\cong\free(W)/\Ical\UDelta$.
\end {proof}

\begin {Example}
The binoid $\free(x,y)/(2x+y=\infty)$ is not a simplicial binoid since it is not (strongly) reduced, cf.\ Lemma \ref{LemStrongRed=Red}(2).
\end {Example}

\begin {Corollary}
Let $\Ical$ be a radical ideal in $M$. If $M$ is a simplicial binoid, then so is $M/\Ical$.
\end {Corollary}
\begin {proof}
By the characterization of simplicial binoids given in Theorem \ref{ThClassSimplB} above, the statement follows from Lemma \ref {LemQuotientCancPosRepF} and Lemma \ref {LemModuloRadical}, and the simple fact that $f+f+g=\infty$ implies $f+g=\infty$ also holds in $M/\Ical$.
\end {proof}

\begin {Example}
Let $\Delta$ be a simplicial complex on $V$ and $\Pcal\in\spec M\UDelta$. By Corollary \ref{CorFacesPrimes} below, $\Pcal=\Pcal_{\tilde{F}}=\langle v\in V\mid v\not\in\tilde{F}\rangle$ for some face $\tilde{F}\in\Delta$. If $A=\{F\in\Delta\mid F$ facet$\}$ so that $\Delta=\bigcup_{F\in A}\Pset(F)$, then 
$M\UDelta/\Pcal_{F}\,=\,M_{^{_{\tilde{\Delta}}}}$, where $\tilde{\Delta}=\bigcup_{F\in A}\Pset(F\cap\tilde{F})$.
\end {Example}

\begin {Lemma} \label{LemSpecBool=SimplCompl}
The booleanization of $M\UDelta$ is isomorphic to $\Delta_{\cup}^{\infty}$.
\end {Lemma}
\begin {proof}
By Proposition \ref{PropMboolUniversalProp}, the binoid homomorphism $\varphi:M\UDelta\rto\Delta_{\cup}^{\infty}$ with $\sum_{v\in F}n_{v}v\mto F$, where $F\in\Delta$ and $n_{v}\ge1$, induces a binoid homomorphism $\bar{\varphi}:(M\UDelta)_{\bool}\rto\Delta_{\cup}^{\infty}$ with $\sum_{v\in F}v\mto F$, where $F\in\Delta$,
which is obviously bijective.
\end {proof}

\begin {Corollary}  \label{CorFacesPrimes}
Let $\Delta$ be a simplicial complex on $V$. Every face $F\in\Delta$ defines a filter in $M\UDelta$, namely $\opFilt(F)=\{f\in M\UDelta\mid\supp(f)\subseteq F\}$, and hence a prime ideal $M\UDelta\setminus\opFilt(F)$ such that the canonical maps
$$\begin {array}{rcccl}
\Fcal(M\UDelta)\setminus\{M\UDelta\}\!\!\!&\stackrel{\sim}{\longleftrightarrow}&\!\!\! \Delta_{\cap}\!\!\!&\stackrel{\sim}{\longleftrightarrow}&\!\!\!\spec M\UDelta\\
S\!\!\!&\lto&\!\!\! \{v\in V\mid v\in S\}&&\\
\opFilt(F)\!\!\!&\longmapsfrom&\!\!\! F\!\!\!&\lto&\!\!\!\langle v\in V\mid v\not\in F\rangle\\
&&\{v\in V\mid v\not\in\Pcal\}\!\!\!&\longmapsfrom&\!\!\!\Pcal
\end {array}$$
are semibinoid isomorphism. In particular, the binoids $\Delta_{\cap}\onull$, $\Fcal(M\UDelta)_{\cap}$, and $\specE M\UDelta$ are isomorphic.
\end {Corollary}
\begin {proof}
Since $\opFilt(G\cap G^{\prime})=\opFilt(G)\cap\opFilt(G^{\prime})$ and $\opFilt(\emptyset)=M\UDelta\okreuz$, the canonical maps on the left-hand side are semibinoid homomorphism, which are inverse to each other. The rest follows from the first part and the following identifications
$$\spec M\UDelta\,\,\cong\,\,\spec\Bcal(M\UDelta)\,\,\cong\,\,\spec\Delta_{\cup}^{\infty}\,\,\cong\,\,\Delta_{\cap}\komma$$
which are due to Corollary \ref{CorSpecBool=Spec}, Lemma \ref{LemSpecBool=SimplCompl}, and Proposition \ref{PropFilterSimplCompl}.
\end {proof}

\begin {Corollary}\label{CorSpecDelta}
Let $\Delta$ be a simplicial complex on $V$ and $\Pcal_{F}\in\spec M\UDelta$, $F\in\Delta$. Then
\begin {ListeTheorem}
\item $\dim\Pcal_{F}=\dim F+1$. In particular, $\dim M\UDelta=\dim\Delta+1$.
\item $F_{i}(M\UDelta)=f_{i-1}(\Delta)$ for all $i\in\{0,1\kpkt d\}$, $d:=\dim M\UDelta$; that is, $ F(M_{\Delta})=f(\Delta)$.
\item The minimal prime ideals of $M\UDelta$ are of the form $\Pcal_{F}=\{f\in M\UDelta\mid\supp(f)\subseteq V\setminus F\}$ for some facet $F\in\Delta$; that is, the irreducible components of $\spec M\UDelta$ correspond to the facets of $\Delta$.
\end {ListeTheorem}
\end {Corollary}
\begin {proof}
By Corollary \ref{CorFacesPrimes}, every prime ideal in $\spec M\UDelta$ is of the form $\Pcal_{F}:=\langle v\mid v\in V\setminus F\rangle$ for some $F\in\Delta$. In particular, if $\dim F=r-1$, say $F=\{v_{1}\kpkt v_{r}\}$, then 
$$\Pcal_{F}=\Pcal_{r}\subsetneq\Pcal_{r-1}\subsetneq\cdots\subsetneq\Pcal_{1}\subsetneq\Pcal_{0}\komma$$
where $\Pcal_{i}=\langle v\in V\mid v\not\in\{v_{1}\kpkt v_{i}\}\rangle$, $i\in\{1\kpkt r\}$, and $\Pcal_{0}=\langle v\mid v\in V\rangle=(M\UDelta)\Uplus$. Hence, every strictly increasing chain of faces is equivalent to a strictly decreasing chain of prime ideals and vice versa. This implies all statements, where the supplement of (3) is due to Corollary \ref {CorIrredCompMinPrime}.
\end {proof}

\begin {Remark}
Let $\Delta$ be a simplicial complex on $V=\{1\kpkt n\}$ and $K$ a field. Now we can easily visualize the $K\mina$spectrum of a simplicial binoid $M\UDelta$ for an arbitrary simplicial complex $\Delta$ on $V$. By Proposition \ref{PropUnionCanComp}, we have
$$K\minspec M\UDelta\,\,\cong\bigcup_{F\in\Delta\text{ facet}} \opV_{K}(K[\Pcal_{F}])\komma$$
where $\Pcal_{F}=\langle v\in V\mid v\not\in F\rangle$ by Corollary \ref{CorSpecDelta}(3), and
$$\opV_{K}(K[\Pcal_{F}])\,\,\cong\,\,\{(a_{1}\kpkt a_{n})\in\A^{n}(K)\mid a_{i}=0\text{ if }i\not\in F\text{ and }a_{i}\in K\text{ otherwise}\}$$
by Remark \ref{RemAffineEmbeddingKspec}. Hence,
$$K\minspec M\UDelta\,\,\cong\bigcup_{F\in\Delta\text{ facet}}A(F)\quad\,\,\,\,\,\big(\subseteq\,\A^{n}(K)\,\,\cong\,\, K\minspec M_{^{_{\Pset(V)}}}\big)\komma$$
where $A(F)=A_{1}\timespkt A_{n}$ with
$$A_{i}:=\begin {cases}
\A^{1}(K)&\text{, if }i\in F\komma\\
\{0\}&\text{, otherwise,}
\end {cases}$$
In particular, every union of coordinate planes is realizable as the $K\mina$spectrum of a simplicial binoid, and therefore (see below) as the $K\mina$spectrum of its underlying simplicial complex. For instance, if $n=3$, then

\begin{center}
\resizebox{0.33\linewidth}{!}{\begin {pspicture} (-5,-6)(5,4.25)
\qdisk (0,0){3pt}\qdisk (0,1.2){3pt}\qdisk (1.2,0){3pt}\qdisk (-0.72,0.76){3pt}
\qdisk (1.2,1.2){3pt}\qdisk (-0.72,-0.54){3pt}\qdisk (0.48,-0.54){3pt}
\psline [linewidth=0.5 pt, linestyle=dashed] (-2.9,2.9)(2.9,2.9)
\psline [linewidth=0.5 pt, linestyle=dashed] (2.9,2.9)(2.9,-2.9)            
\psline [linewidth=0.5 pt, linestyle=dashed] (2.9,-2.9)(0,-2.9)              
\psline [linewidth=0.5 pt, linestyle=dashed] (-1.6,-2.9)(-2.9,-2.9)         
\psline [linewidth=0.5 pt, linestyle=dashed] (-2.9,-2.9)(-2.9,-1.2)      
\psline [linewidth=0.5 pt, linestyle=dashed] (-2.9,0)(-2.9,2.9)      

\psline [linewidth=0.5 pt, linestyle=dashed] (0,-2.9)(-1.6,-4.1)
\psline [linewidth=0.5 pt, linestyle=dashed] (-1.6,-4.1)(-1.6,1.7)          
\psline [linewidth=0.5 pt, linestyle=dashed] (-1.6,1.7)(1.6,4.1)             
\psline [linewidth=0.5 pt, linestyle=dashed] (1.6,4.1)(1.6,2.9)             

\psline [linewidth=0.5 pt, linestyle=dashed] (-2.9,0)(-4.5,-1.2)
\psline [linewidth=0.5 pt, linestyle=dashed] (-4.5,-1.2)(1.3,-1.2)           
\psline [linewidth=0.5 pt, linestyle=dashed] (1.3,-1.2)(4.5,1.2)             
\psline [linewidth=0.5 pt, linestyle=dashed] (4.5,1.2)(2.9,1.2)            
\psline [linewidth=0.5 pt] (0,0)(0,2.9)
\psline [linewidth=0.5 pt] (0,-1.2)(0,-2.9)    
\psline [linewidth=0.5 pt] (-1.6,0)(-2.9,0)
\psline [linewidth=0.5 pt] (0,0)(2.9,0)
\psline [linewidth=0.5 pt] (0,0)(-1.6,-1.2)
\uput [0] (-2,-5.25){\LARGE{$\R\minspec M_{^{_{\Pset(V)\setminus{V}}}}$}}
\end {pspicture}}
\end {center}

In topology, the geometric realization of a simplicial complex $\Delta$ on $V=\{1\kpkt n\}$ is given by
$$\{\varphi\in\R\minspec\Delta\mid\varphi\,\,\widehat{=}\,\,(a_{1}\kpkt a_{n})\in\A^{n}(\R),a_{i}\ge0,a_{1}\pluspkt a_{n}=1\}\komma$$
cf.\ \cite[Chapter II.2.1]{FerrarioPiccinini}.
\end {Remark}

\begin {Proposition} \label{PropUnivPropSimplBinoid}
Let $\Delta$ be a simplicial complex on $V$ and $N$ a commutative binoid. Every $N\mina$point $\rho:V\rto N$ of $\Delta$ factors uniquely through $M\UDelta$; that is, there is a unique binoid homomorphism $\tilde{\rho}:M\UDelta\rto N$ such that $\tilde{\rho}\iota=\rho$, where $\iota:V\rto M\UDelta$ denotes the canonical map $v\mto v$. In particular, 
$$N\minspec M\UDelta\,\cong\,\ N\minspec\Delta$$
as semibinoids and 
$$R\minSpec R[\Delta]\,\,\cong\,\,R\minspec\Delta$$
for a ring $R$.
\end {Proposition}
\begin {proof}
Since every element $f\in M\UDelta\opkt$ can be written uniquely as $f=\sum_{v\in F}n_{v}v$ with $n_{v}\ge1$, $v\in F$, for some $F\in\Delta$, the map $\tilde{\rho}:M\UDelta\rto N$ with $\tilde{\rho}(f):=\sum_{v\in F}n_{v}\rho(v)$ and $\tilde{\rho}(\infty):=\infty$ is well-defined and satisfies $\tilde{\rho}\iota=\rho$ by definition. To show that $\tilde{\rho}$ is a binoid homomorphism let $f=\sum_{v\in F}n_{v}v$ and $g=\sum_{v\in g}m_{v}v$ with $F,G\in\Delta$ such that $n_{v},m_{v}\ge1$. Then 
$$
\tilde{\rho}(f)+\tilde{\rho}(g)\,=\,\sum_{v\in F}n_{v}\rho(v)+\sum_{v\in F}m_{v}\rho(v)\,\,=\sum_{v\in F\cup G}(n_{v}+m_{v})\rho(v)$$
in $N$, where $n_{v}:=0$ if $v\not\in F$ and $m_{v}:=0$ if $v\not\in G$. In particular, $n_{v}+m_{v}-1\ge0$ for all $v\in F\cup G$. Thus, if $F\cup G$ is not a face in $\Delta$, then $f+g=\infty$ in $M\UDelta$ and 
\begin {align*}
\tilde{\rho}(f)+\tilde{\rho}(g)&\,=\,\sum_{v\in F\cup G}\rho(v)\,\,+\sum_{v\in F\cup G}(n_{v}+m_{v}-1)\rho(v)\\
&\,=\,\,\infty\,\,+\sum_{v\in F\cup G}(n_{v}+m_{v}-1)\rho(v)\\
&\,=\,\,\infty
\end {align*}
since $\rho$ is an $N\mina$point. Hence, $\tilde{\rho}(f)+\tilde{\rho}(g)=\infty=\tilde{\rho}(\infty)=\tilde{\rho}(f+g)$. If, on the other hand, $F\cup G\in\Delta$, then $f+g=\sum_{v\in F\cup G}(n_{n}+m_{v})v\in M\UDelta$, and therefore $\tilde{\rho}(f+g)=\sum_{v\in F\cup G}(n_{n}+m_{v})\rho(v)=\tilde{\rho}(f)+\tilde{\rho}(g)$ as above. For the supplement consider the canonical map 
$$\psi:N\minspec M\UDelta\Rto N\minspec\Delta\quad\text{with}\quad\varphi\lto\varphi\iota=:\varphi_{|V}\komma$$
where $\iota:V\rto M\UDelta$, $v\mto v$, as in the proposition. The bijectivity of $\psi$ follows from the first part. Obviously, $\psi$ is a semigroup homomorphism which satisfies $\chi_{\{0\}}\mto\varepsilon$, where $\chi_{\{0\}}$ is the absorbing element of $N\minspec M\UDelta$ (because $M\UDelta$ is positive) and $\varepsilon:v\mto\infty$, $v\in V$, that of $N\minspec\Delta$. The statement on the $R\mina$points of the Stanley-Reisner algebra $R[\Delta]$ and $\Delta$ follows from Proposition \ref{PropUnivPropBinoidA} and Lemma \ref{LemSRalgebraBinoid}.
\end {proof}

\begin {Proposition} \label{PropIndBinoidHomDeltaSimpl}
Every binoid homomorphism $\varphi:\tilde{\Delta}_{\cup}^{\infty}\rto\Delta_{\cup}^{\infty}$ gives rise to a homomorphism $\phi:M_{^{_{\tilde{\Delta}}}}\rto M_{^{_{\Delta}}}$ of the simplicial binoids defined by $\Delta$ and $\tilde{\Delta}$. This binoid homomorphism commutes via $(M\UDelta)_{\bool}\cong\Delta_{\cup}^{\infty}$ with $\varphi$. Hence, $\varphi=\Bcal(\phi)$ and there is a commutative diagram
$$\xymatrix{
M_{^{_{\tilde{\Delta}}}}\ar[r]^{\phi}\ar[d]_{\pi_{\bool}}&M\UDelta\ar[d]^{\pi_{\bool}}\komma\\
\tilde{\Delta}_{\cup}^{\infty}\ar[r]^{\Bcal(\phi)}&\Delta_{\cup}^{\infty}}$$
where $\pi_{\bool}$ is the canonical binoid homomorphism $\sum_{v\in F}n_{v}v\mto F$. Moreover, $\Bcal(\phi)$ is surjective if and only if $\phi$ is so.
\end {Proposition}
\begin {proof}
Let $\varphi:\tilde{\Delta}_{\cup}^{\infty}\rto\Delta_{\cup}^{\infty}$ be a binoid homomorphism. For $f=\sum_{v\in F}n_{v}v\in M_{^{_{\tilde{\Delta}}}}\opkt$ with $n_{v}\ge1$ and $F\in\tilde{\Delta}$ define 
$$\phi(f):=\begin {cases}
\sum_{v\in F}n_{v}\big(\sum_{w\in\varphi(\{v\})}w\big)&\text{, if }\varphi(F)\in\Delta\komma\\
\infty&\text{, otherwise (i.e.\ when $\varphi(F)=\infty$ in $\Delta_{\cup}^{\infty}$)}\pkt
\end {cases}$$ 
If $g\in M_{^{_{\tilde{\Delta}}}}\opkt$ with $g=\sum_{v\in G}m_{v}v$, $m_{v}\ge1$, $G\in\tilde{\Delta}$, then
$$f+g\,=\,\sum_{v\in F}n_{v}v+\sum_{v\in G}m_{v}v\,\,=\sum_{v\in F\cup G}(n_{v}+m_{v})v\komma$$
where $n_{v}:=0$ if $v\not\in F$ and $m_{v}:=0$ if $v\not\in G$. Thus, for $f+g\not=\infty$
\begin {align*}
\phi(f+g)&\,=\sum_{v\in F\cup G}(n_{v}+m_{v})\Big(\sum_{w\in\varphi(v)}w\Big)\\
&\,\,=\sum_{v\in F\cup G}n_{v}\Big(\sum_{w\in\varphi(v)}w\Big)\,\,+\sum_{v\in F\cup G}m_{v}\Big(\sum_{w\in\varphi(v)}w\Big)\\
&\,\,=\,\,\,\sum_{v\in F}n_{v}\Big(\sum_{w\in\varphi(v)}w\Big)\,+\,\sum_{v\in G}m_{v}\Big(\sum_{w\in\varphi(v)}w\Big)\,\,=\,\phi(f)+\phi(g)\pkt
\end {align*}
On the other hand, $f+g=\infty$ means $F\cup G\not\in\tilde{\Delta}$ (or equivalently $F\cup G=\infty$ in $\tilde{\Delta}_{\cup}^{\infty}$), which implies that $\varphi(F\cup G)=\infty$ since $\varphi$ is a binoid homomorphism. Hence, $\phi(f+g)=\infty$. By Corollary \ref{CorBoolFunctor} and Lemma \ref{LemSpecBool=SimplCompl}, every binoid homomorphism $\phi:M_{^{_{\tilde{\Delta}}}}\rto M_{^{_{\Delta}}}$ induces such a commutative diagram with $\Bcal(\phi)(\{v\})=\pi_{\bool}\phi(v)$. Thus, $\varphi=\Bcal(\phi)$ by the first part. The supplement is clear.
\end {proof}

Recall from Section \ref{SecSimplMorph}, cf.\ Lemma \ref{LemIndHomUnion}, that we have the following binoid homomorphisms:
$$\varphi_{\lambda}:\Delta_{\cup}^{\infty}\Rto\tilde{\Delta}_{\cup}^{\infty}\quad\text{with}\quad F\lto\begin {cases}
\emptyset&\text{, if }F=\emptyset\komma\\
\infty&\text{, if }F=\infty\text{ or }\lambda^{-1}(v)=\emptyset\text{ for some }v\in F\komma\\
\lambda^{-1}(F)&\text{, otherwise,}
\end {cases}$$
if $\lambda:(\tilde{V},\tilde{\Delta})\rto(V,\Delta)$ is simplicial.

$$\varphi_{\lambda}^{\alpha}:\tilde{\Delta}_{\cup}^{\infty}\Rto\Delta_{\cup}^{\infty}\quad\text{with}\quad F\lto\begin {cases}
\infty&\text{, if }F=\infty\text{ or }\lambda(F)\not\in\Delta\komma\\
\lambda(F)&\text{, otherwise,}
\end {cases}$$
if $\lambda:(\tilde{V},\tilde{\Delta})\rto(V,\Delta)$ is $\alpha$-simplicial.

$$\varphi_{\lambda}^{\beta}:\Delta_{\cup}^{\infty}\Rto\tilde{\Delta}_{\cup}^{\infty}\quad\text{with}\quad
F\lto\begin {cases}
\infty&\text{, if }F=\infty\text{ or }\lambda^{-1}(F)\not\in\tilde{\Delta}\komma\\
\lambda^{-1}(F)&\text{, otherwise,}
\end {cases}$$
if $\lambda:(\tilde{V},\tilde{\Delta})\rto(V,\Delta)$ is $\beta$-simplicial.

\begin {Corollary} \label{CorCosimpIndBinoidHom}
\begin {ListeTheorem}
\item[]
\item Every simplicial morphism $\lambda:(\tilde{V},\tilde{\Delta})\rto(V,\Delta)$ induces a binoid homomorphism 
$$\phi_{\lambda}:M\UDelta\Rto M_{^{_{\tilde{\Delta}}}}\komma\quad\sum_{v\in F}n_{v}v\lto\begin {cases}
\infty&\text{, if }\lambda^{-1}(v)=\emptyset\text{ for some }v\in F\komma\\
\sum_{w\in\lambda^{-1}(F)}n_{\lambda(w)}w&\text{, otherwise,}
\end {cases}$$ 
that is, $\Delta\mto M\UDelta$ is a contravariant functor $\SCsf\rto\sBsf$. If $\lambda$ is injective, then $\phi_{\lambda}$ is surjective.
\item Every $\alpha$-simplicial morphism $\lambda:(\tilde{V},\tilde{\Delta})\rto(V,\Delta)$ induces a binoid homomorphism 
$$\phi_{\lambda}^{\alpha}:M_{^{_{\tilde{\Delta}}}}\Rto M\UDelta\komma\quad\sum_{v\in F}n_{v}v\lto
\sum_{v\in F}n_{v}\lambda(v)\komma$$ 
that is, $\Delta\mto M\UDelta$ is a covariant functor $\SCsf^{\alpha}\rto\sBsf$. If $\lambda$ is surjective, then so is $\phi_{\lambda}^{\alpha}$.
\item Every $\beta$-simplicial morphism $\lambda:(\tilde{V},\tilde{\Delta})\rto(V,\Delta)$ induces a binoid homomorphism 
$$\phi_{\lambda}^{\beta}:M\UDelta\Rto M_{^{_{\tilde{\Delta}}}}\komma\quad\sum_{v\in F}n_{v}v\lto
\sum_{w\in\lambda^{-1}(F)}n_{\lambda(w)}w\komma$$ 
that is, $\Delta\mto M\UDelta$ is a contravariant functor $\SCsf^{\beta}\rto\sBsf$. If $\lambda$ is bijective, then $\phi_{\lambda}^{\beta}$ is surjective.
\end {ListeTheorem}

In particular, there are three commuting diagrams
$$\xymatrix{
M\UDelta\ar[r]^{\phi_{\lambda}}\ar[d]_{\pi_{\bool}}&M_{^{_{\tilde{\Delta}}}}\ar[d]^{\pi_{\bool}}\\
\Delta_{\cup}^{\infty}\ar[r]^{\varphi_{\lambda}}&\tilde{\Delta}_{\cup}^{\infty}}\quad\quad\quad
\xymatrix{
M_{^{_{\tilde{\Delta}}}}\ar[r]^{\phi_{\lambda}^{\alpha}}\ar[d]_{\pi_{\bool}}&M\UDelta\ar[d]^{\pi_{\bool}}\\
\tilde{\Delta}_{\cup}^{\infty}\ar[r]^{\varphi_{\lambda}^{\alpha}}&\Delta_{\cup}^{\infty}}\quad\quad\quad
\xymatrix{
M\UDelta\ar[r]^{\phi_{\lambda}^{\beta}}\ar[d]_{\pi_{\bool}}&M_{^{_{\tilde{\Delta}}}}\ar[d]^{\pi_{\bool}}\komma\\
\Delta_{\cup}^{\infty}\ar[r]^{\varphi_{\lambda}^{\beta}}&\tilde{\Delta}_{\cup}^{\infty}}
$$
where $\varphi_{\lambda}$, $\varphi_{\lambda}^{\alpha}$, and $\varphi_{\lambda}^{\beta}$ are the binoid homomorphisms from above and $\pi_{\bool}$ the canonical binoid homomorphism $\sum_{v\in F}n_{v}v\mto F$. 
\end {Corollary}
\begin {proof}
All results follow from Proposition \ref{PropIndBinoidHomDeltaSimpl} using the binoid homomorphisms $\varphi_{\lambda}$, $\varphi_{\lambda}^{\alpha}$, and $\varphi_{\lambda}^{\beta}$ from above to obtain $\phi_{\lambda}$, $\phi_{\lambda}^{\alpha}$, and $\phi_{\lambda}^{\beta}$, respectively.
\end {proof}

\begin {Remark}
The preceding corollary combined with Proposition \ref{PropIndHomeoKspec} yields the following homeomorphisms between $K\mina$spectra.
\begin {ListeTheorem}
\item 
If $\lambda:(\tilde{V},\tilde{\Delta})\rto(V,\Delta)$ is simplicial, then 
$$K\minspec M\UDelta\stackrel{\phi_{\lambda}^{\ast}}{\Rto} K\minspec M_{^{_{\tilde{\Delta}}}}\quad\text{with}\quad\varphi\lto\tilde{\varphi}\komma$$
where
$$\tilde{\varphi}\Big(\sum_{v\in F}n_{v}v\Big)\,=\,\begin {cases}
\infty&\text{, if }\lambda^{-1}(v)=\emptyset\text{ for some }v\in F\komma\\
\prod_{w\in\lambda^{-1}(F)}\varphi(w)^{n_{\lambda(w)}}&\text{, otherwise.}
\end {cases}$$
\item 
If $\lambda:(\tilde{V},\tilde{\Delta})\rto(V,\Delta)$ is $\alpha$-simplicial, then
$$K\minspec M_{^{_{\tilde{\Delta}}}}\stackrel{(\phi_{\lambda}^{\alpha})^{\ast}}{\Rto} K\minspec M\UDelta\quad\text{with}\quad\varphi\lto\Big(\tilde{\varphi}:\sum_{v\in F}n_{v}v\lto\prod_{v\in F}\varphi(\lambda(v))^{n_{v}}\Big)\pkt$$
\item 
If $\lambda:(\tilde{V},\tilde{\Delta})\rto(V,\Delta)$ is $\beta$-simplicial, then 
$$K\minspec M\UDelta\stackrel{(\phi_{\lambda}^{\beta})^{\ast}}{\Rto} K\minspec M_{^{_{\tilde{\Delta}}}}\quad\text{with}\quad\varphi\lto\Big(\tilde{\varphi}:\sum_{v\in F}n_{v}v\lto\prod_{w\in\lambda^{-1}(F)}\varphi(w)^{n_{\lambda(w)}}\Big)\pkt$$
\end {ListeTheorem}
\end {Remark}

\begin {Corollary} \label {CorSimplCompositions}
If $\Delta_{i}$ is a simplicial complex on $V_{i}$, $i\in I$, $I$ finite, then
$$M_{^{_{\prod_{i\in I}\Delta_{i}}}}\,\cong\,\,\bigwedge_{i\in I}M_{^{_{\Delta_{i}}}}\quad\text{and}\quad M_{^{_{\biguplus_{i\in I}\Delta_{i}}}}\,\cong\,\,\,\bigcupbidot_{i\in I}M_{^{_{\Delta_{i}}}}\pkt$$
In particular, if $N$ is a commutative binoid, then 
$$N\minspec M_{^{_{\prod_{i\in I}\Delta_{i}}}}\,\cong\,\,\prod_{i\in I}N\minspec M_{^{_{\Delta_{i}}}}\komma$$
and if $N$ an integral binoid, then
$$N\minspec M_{^{_{\biguplus_{i\in I}\Delta_{i}}}}\,\cong\,\,\bigcupdot_{i\in I}N\minspec M_{^{_{\Delta_{i}}}}\pkt$$
\end {Corollary}
\begin {proof}
The $\alpha$-simplicial morphisms $(V_{k},\Delta_{k})\embto(\biguplus_{i\in I}V_{i},\prod_{i\in I}\Delta_{i})$ with $v\mto(v;k)$, $k\in I$, cf.\
Example \ref{ExpProducts}, induce by Corollary \ref{CorCosimpIndBinoidHom}(2) binoid homomorphisms
$\iota_{k}:M_{^{_{\Delta_{k}}}}\rto M_{^{_{\prod_{i\in I}\Delta_{i}}}}$, $k\in I$,
which give rise to the binoid homomorphism $\bigwedge_{i\in I}M_{^{_{\Delta_{i}}}}\rto M_{^{_{\prod_{i\in I}\Delta_{i}}}}$ with $\wedge_{i\in I}f_{i}\mto\sum_{i\in I}\iota_{i}(f_{i})$
by Proposition \ref{PropCoproductComBinoid}. Since every $h\in  M_{^{_{\prod_{i\in I}\Delta_{i}}}}$ can be written as $h=\sum_{i\in I}\sum_{v\in F_{i}}n_{v}\iota_{i}(v)$ for unique faces $F_{i}\in\Delta_{i}$ such that $n_{v}\ge1$ for all $v\in F_{i}$, $i\in I$, this binoid homomorphism is an isomorphism. Similarly, the $\alpha$-simplicial morphisms $(V_{k},\Delta_{k})\embto(\biguplus_{i\in I}V_{i},\biguplus_{i\in I}\Delta_{i})$ with $v\mto(v;k)$, $k\in I$, cf.\ Example \ref{ExpProducts}, induce by Corollary \ref{CorCosimpIndBinoidHom}(2) binoid homomorphisms $M_{^{_{\Delta_{k}}}}\rto M_{^{_{\biguplus_{i\in I}\Delta_{i}}}}$, $k\in I$, which give rise to the binoid homomorphism 
$\bigcupbidot_{i\in I}M_{^{_{\Delta_{i}}}}\rto M_{^{_{\biguplus_{i\in I}\Delta_{i}}}}$ with $(f;k)\mto f$
by Proposition \ref{PropUnivPropPUnion}. Since every $h\in  M_{^{_{\biguplus_{i\in I}\Delta_{i}}}}$ can be written as $h=\sum_{v\in F}n_{v}v$ for a unique face $F\in\Delta_{i}$, $i\in I$, such that $n_{v}\ge1$ for all $v\in F$, this binoid homomorphism is also an isomorphism. The isomorphisms of the supplement follow from Proposition \ref{PropNspecSmash} and Corollary \ref{CorNSpecPUnion}.
\end {proof}

\begin {Remark}
In combinatorial commutative algebra, the product of two simplicial complexes $\Delta_{1}$ and $\Delta_{2}$ is called the \gesperrt{join} \index{simplicial complex!join of --s} and is usually denoted by $\Delta_{1}\ast\Delta_{2}$, cf.\ \cite[Definition 5.2.4]{Villarreal}. It is known that the Stanley-Reisner algebra $K[\Delta_{1}\ast\Delta_{2}]$ of the join is given by the tensor product $K[\Delta_{1}]\otimes_{K}K[\Delta_{2}]$, with which one proves, for instance, that $\Delta_{1}\ast\Delta_{2}$ is Cohen-Macaulay if and only if $\Delta_{1}$ and $\Delta_{2}$ are Cohen-Macaulay, cf.\ \cite[Proposition 5.3.16]{Villarreal}. The result for the tensor product can be deduced elegantly from the theory developed in this chapter. By Lemma \ref{LemSRalgebraBinoid} and Corollary \ref{CorSimplCompositions}, we have $K[\Delta]=K[M_{^{_{\Delta}}}]$ and $M_{^{_{\prod_{i\in I}\Delta_{i}}}}=\bigwedge_{i\in I}M_{^{_{\Delta}}}$. Hence,
$$K[\Delta_{1}\times\Delta_{2}]\,=\,K[M_{^{_{\Delta_{1}\times\Delta_{2}}}}]\,=\,K[M_{^{_{\Delta_{1}}}}\wedge M_{^{_{\Delta_{2}}}}]\,=\,K[M_{^{_{\Delta_{1}}}}]\otimes_{K}K[M_{^{_{\Delta_{2}}}}]\,=\, K[\Delta_{1}]\otimes_{K}K[\Delta_{2}]\komma$$
where the third equality is due to Corollary \ref{CorSmash=Tensor}.
\end {Remark}

\begin {Example}
Let $\Delta$ and $\tilde{\Delta}$ be simplicial complexes on $V$ such that $\tilde{\Delta}\subseteq\Delta$. The morphisms 
$$(V,\Delta)\stackrel{\id}{\Rto}(\tilde{V},\tilde{\Delta})\stackrel{\id}{\Rto}(V,\Delta)\komma$$
where the first morphism is $\alpha$-simplicial and the latter simplicial and $\beta$-simplicial, cf.\ Example \ref{ExpSubcomplexSameV}, induce the same binoid epimorphism
$$M\UDelta\Rto M_{^{_{\tilde{\Delta}}}}\quad\text{with}\quad\sum_{v\in F}n_{v}v\lto\sum_{v\in F}n_{v}v\pkt$$
In particular, for every simplicial complex $\Delta$ on $V$ there is a binoid epimorphism $M_{^{_{\Pset(V)}}}\rto M\UDelta\komma$
which yields the canonical $K\mina$algebra epimomorphism
$$K[\Pset(V)]\,\,=\,\,K[X_{v}\mid v\in V]\Rto K[X_{v}\mid v\in V]/\Iideal\UDelta\,\,=\,\,K[M_{\Delta}]\,\,=\,\,K[\Delta]\komma$$
$K$ a ring. Note that if $\Delta\subsetneq\Pset(V)$ there is no binoid embedding $M\UDelta\embto M_{^{_{\Pset(V)}}}$.
\end {Example}

\begin{Example}
Let $\Delta$ be a simplicial complex on $V$, $T\subseteq V$, and $\Delta(T)=\{F\in\Delta\mid F\subseteq T\}$. The simplicial and $\alpha$-simplicial morphism, cf.\ Example \ref{ExpSubcomplexDifferntV},
$$\lambda:(T,\Delta(T))\rto(V,\Delta)\quad\text{with}\quad v\lto v$$
induces the binoid homomorphisms
$$M_{^{_{\Delta(T)}}}\stackrel{\phi_{\lambda}^{\alpha}}{\Rto}M\UDelta\stackrel{\phi_{\lambda}}{\Rto}M_{^{_{\Delta(T)}}}$$
with
$$\sum_{v\in F}n_{v}v\lto\sum_{v\in F}n_{v}v\lto\begin {cases}
\sum_{v\in F}n_{v}v&\text{, if }F\in\Delta(T)\komma\\
\infty&\text{, otherwise.}
\end {cases}$$
$\phi_{\lambda}^{\alpha}$ is a binoid epimorphism and $\phi_{\lambda}$ a binoid embedding, which satisfy $\phi_{\lambda}\circ\phi_{\lambda}^{\alpha}=\id_{M_{^{_{\Delta(T)}}}}$. We consider two special cases. If $T=F\in\Delta$ with $\#F=r$, then $\Delta(T)=\Pset(F)$ and $M_{^{_{\Pset(F)}}}\cong(\N^{r})^{\infty}$, which gives
$$(\N^{r})^{\infty}\stackrel{\phi_{\lambda}^{\alpha}}{\Rto}M\UDelta\stackrel{\phi_{\lambda}}{\Rto}(\N^{r})^{\infty}\komma$$
where $\phi_{\lambda}$ is the canonical projection $M\UDelta\rto M\UDelta/\Pcal_{F}$ with $\Pcal_{F}=\langle v\in V\mid v\not\in F\rangle\in\spec M\UDelta$. By Proposition \ref {PropIndHomNspec} and Remark \ref{RemAffineEmbeddingKspec}, this yields for a field $K$ 
$$\A^{r}(K)\stackrel{\phi_{\lambda}^{\ast}}{\Rto}K\minspec M\UDelta\stackrel{(\phi_{\lambda}^{\alpha})^{\ast}}{\Rto}\A^{r}(K)\pkt$$
The other special case is when $\Delta=\Pset(V)$. Then $M\UDelta\cong(\N^{n})^{\infty}$, where $n=\# V$. As observed in Example \ref{ExpSubcomplexDifferntV}, the morphism $\lambda:(T,\Delta(T))\rto(V,\Pset(V))$ is also $\beta$-simplicial in this case. Hence,
$$\phi_{\lambda}, \phi_{\lambda}^{\beta}:(\N^{n})^{\infty}\Rto M_{^{_{\Delta(T)}}}$$
are the two binoid epimorphisms generated by
$$\phi_{\lambda}(e_{i})=\begin {cases}
i&\text{, if }i\in T\komma\\
\infty&\text{, otherwise,}
\end {cases}\quad\quad\text{and}\quad\quad\phi_{\lambda}^{\beta}(e_{i})=\begin {cases}
i&\text{, if }i\in T\komma\\
0&\text{, otherwise.}
\end {cases}$$
These yield two different embeddings 
$$\phi_{\lambda}^{\ast}, (\phi_{\lambda}^{\beta})^{\ast}:K\minspec M_{^{_{\Delta(T)}}}\Rto\A^{n}(K)\pkt$$
If $V=\{1,2\}$, $T=\{1\}$, and $K=\R$, these two embeddings $\A^{1}(\R)\embto\A^{2}(\R)$ are given by
\begin {center}
\begin {pspicture}  (-2,-2.25)(2,2)
\qdisk (0,0){1.75pt}\qdisk (0.7,0){1.75pt}
\psline [linewidth=0.5 pt, linestyle=dotted] (0,-1.5)(0,1.5)
\psline [linewidth=0.5 pt] (-1.5,0)(1.5,0)
\uput [0] (-0.75,-2.2){\small{via $\phi_{\lambda}^{\ast}$}}
\end {pspicture}
\quad\quad\quad
\begin {pspicture} (-2,-2.25)(2,2)
\qdisk (0,0.7){1.75pt}\qdisk (0.7,0.7){1.75pt}
\psline [linewidth=0.5 pt, linestyle=dotted] (0,-1.5)(0,1.5)
\psline [linewidth=0.5 pt, linestyle=dotted]  (-1.5,0)(1.5,0)
\psline [linewidth=0.5 pt]  (-1.5,0.7)(1.5,0.7)
\uput [0] (-0.75,-2.2){\small{via $(\phi_{\beta})^{\ast}$}}
\end {pspicture}
\end {center}
\end {Example}

\begin {Example}
Let $I$ be finite and $\Delta_{i}$ a simplicial complex on $V_{i}$, $i\in I$. For every $k\in I$, the simplicial and $\alpha$-simplicial morphism $\iota_{k}:(V_{k},\Delta_{k})\rto(\biguplus_{i\in I}V_{i},\biguplus_{i\in I}\Delta_{i})$, $(v;k)\mto(v;k)$, cf.\ Example \ref{ExpProducts}, induces the binoid homomorphisms
$$M_{\Delta_{k}}\stackrel{(\phi_{\iota_{k}}^{\alpha})^{\ast}}{\Rto}\bigcupbidot_{i\in I}M_{^{_{\Delta_{i}}}}\stackrel{\phi_{\iota_{k}}^{\ast}}{\Rto}M_{\Delta_{k}}$$
with
$$f\lto(f;k)\quad\text{and}\quad(f;i)\lto\begin {cases}
f&\text{, if }i=k\komma\\
\infty&\text{, otherwise.}
\end {cases}$$
$\phi_{\iota_{k}}^{\alpha}$ is a binoid epimorphism and $\phi_{\iota_{k}}$ a binoid embedding. Moreover, the $\alpha$-simplicial morphism $\tilde{\iota}_{k}(V_{k},\Delta_{k})\rto(\biguplus_{i\in I}V_{i},\prod_{i\in I}\Delta_{i})$ with $(v;k)\mto(v;k)$, cf.\ Example \ref{ExpProducts}, induces a binoid embedding 
$$\phi_{\iota_{k}}^{\alpha}:M_{\Delta_{k}}\Rto\bigwedge_{i\in I}M_{^{_{\Delta_{i}}}}\quad\text{with}\quad f\lto0\wedge\cdots\wedge0\wedge f\wedge0\wedge\cdots\wedge0\komma$$
where $f$ is the $k$th entry.
\end {Example}

\newpage

\markright{Index}
\addcontentsline{toc}{chapter}{Index}
\printindex

\markright{Symbols}
\printnomenclature [3cm]

\markright{Bibliography}

\end{document}